\documentclass[reqno]{amsart}

\makeatletter
\renewcommand\part{%
	\if@noskipsec \leavevmode \fi
	\par
	\addvspace{4ex}%
	\@afterindentfalse
	\secdef\@part\@spart}

\def\@part[#1]#2{%
	\ifnum \c@secnumdepth >\m@ne
	\refstepcounter{part}%
	\addcontentsline{toc}{part}{\thepart\hspace{1em}#1}%
	\else
	\addcontentsline{toc}{part}{#1}%
	\fi :
	{\parindent \z@ \raggedright
		\interlinepenalty \@M
		\normalfont
		\ifnum \c@secnumdepth >\m@ne
		\Large\bfseries \partname\nobreakspace\thepart
		\par\nobreak
		\fi
		\huge \bfseries #2%
		\par}%
	\nobreak
	\vskip 3ex
	\@afterheading}
\def\@spart#1{%
	{\parindent \z@ \raggedright
		\interlinepenalty \@M
		\normalfont
		\huge \bfseries #1\par}%
	\nobreak
	\vskip 3ex
	\@afterheading}
\makeatother
\usepackage{array}
\usepackage{color}
\usepackage{float}
\usepackage{esint}
\usepackage[dvipsnames]{xcolor}
\usepackage{ifpdf}
\ifpdf 
    \usepackage[pdftex]{graphicx}   
    \usepackage[pdftex,     
            plainpages=false,   
            breaklinks=true,    
            colorlinks=true,
            linkcolor=red,
            citecolor=green,
            pdftitle=My Document
            pdfauthor=My Good Self
           ]{hyperref} 
\else 
    \usepackage{graphics}       
    \usepackage{hyperref}       
\fi 

\usepackage{siunitx} 
\usepackage{glossaries}
\usepackage{glossary-mcols}

\usepackage{aurical}
\usepackage{amsfonts,amsmath}
\usepackage{amssymb}
\usepackage{verbatim}
\usepackage{amsopn}
\usepackage[english]{babel}
\usepackage{amsthm}
\usepackage{enumerate}
\usepackage{mathrsfs}
\usepackage{enumitem}
\usepackage{mathtools}
\usepackage{esint}
\usepackage{bbm}
\usepackage{caption}
\usepackage{marginnote}
\usepackage{float} 
\usepackage[marginparwidth=2cm]{geometry}
\usepackage{fourier}
\usepackage{listings}
\usepackage{subcaption}
\usepackage{tcolorbox}
\usepackage{anyfontsize}

\date{\today}

\theoremstyle{definition} \newtheorem{definition}{Definition}[section]
\theoremstyle{definition} \newtheorem{remark}[definition]{Remark}
\theoremstyle{plain} \newtheorem{lemma}[definition]{Lemma}
\theoremstyle{plain} \newtheorem{proposition}[definition]{Proposition}
\theoremstyle{plain} \newtheorem{theorem}[definition]{Theorem}
\theoremstyle{plain} \newtheorem{corollary}[definition]{Corollary}
\theoremstyle{definition} 
\theoremstyle{plain} 
\theoremstyle{definition} 
\theoremstyle{definition}

\DeclareMathOperator{\conv}{conv}

\DeclareMathOperator{\BV}{BV}

\DeclareMathOperator{\clos}{clos}

\newcommand{\R}{\mathbb{R}}

\newcommand{\C}{\mathbb{C}}
\newcommand{\TV}{\text{\rm Tot.Var.}}

\newcommand{\loc}{\text{\rm loc}}

\newcommand{\ind}{1\!\!\mathrm{I}}

\newcommand{\sing}{\mathrm{sing}}

\newcommand{\jump}{\mathrm{jump}}
\newcommand{\cont}{\mathrm{cont}}

\newcommand{\n}{\mathbf n}

\newcommand{\Diracd}{\mathfrak{d}}
\newcommand{\smed}{\medskip}
\newcommand{\sbig}{\bigskip}

\renewcommand{\H}{\mathscr H}
\newcommand{\rest}{\llcorner}
\renewcommand{\i}{\imath}

\numberwithin{equation}{section} 

\theoremstyle{plain} \newtheorem*{theorem*}{Theorem}
\theoremstyle{plain} 
\theoremstyle{plain} \newtheorem*{mthm*}{Main Theorem}
\theoremstyle{plain} \newtheorem*{conjecture*}{Conjecture}
\theoremstyle{plain} \newtheorem{conjecture}[definition]{Conjecture}
\theoremstyle{plain} \newtheorem*{problem*}{Problem}

\title[Spiral strategies for blocking fire]{Existence of spiral strategies for blocking fire spreading}

\author{Stefano Bianchini}
\address{S. Bianchini: S.I.S.S.A., via Bonomea 265, 34136 Trieste, Italy}
\email{bianchin@sissa.it}

\author{Martina Zizza}
\address{M. Zizza: Max Planck Institute for Mathematics in the Sciences, Leipzig}
\email{martina.zizza@mis.mpg.de}

\makeglossaries

\begin{document}

\begin{abstract}
In this paper we address the problem for blocking fire by constructing a wall $\zeta$ whose shape is spiral-like. This is supposed to be the best strategy when a single firefighter is constructing the wall with a finite construction speed $\sigma$: the barriers which satisfy this bound on the construction speed are called admissible.

We prove a sharp version of Bressan's Fire Conjecture \cite{Bressan_conj} in this case, i.e. when admissible barriers are spiral-like curves: namely, there exists a spiral-like barrier confining the fire in a bounded region of $\R^2$ if and only if the speed of construction of the barrier $\sigma$ is strictly larger than a critical speed $\bar \sigma = 2.614...$.

The existence of confining spiral barriers for $\sigma > \bar \sigma$ is already known \cite{Bressan_friends,firefighter}, while we concentrate on the negative side, i.e. if $\sigma \leq \bar \sigma$ no admissible spiral blocks the fire.

The proof of these results relies on:
\begin{enumerate}
\item the precise definition of spiral barrier and its representation;
\item the analysis of saturated spiral barriers as a Retarded Differential Equation (RDE) in the spirit of \cite{firefighter};
\item the equivalent reformulation of the conjecture as a minimum problem of a functional  for a prescribed functional;
\item the construction of the optimal closing spiral;
\item \label{Point_5:abstract} the analysis of a differentiable path of admissible spirals along which the functional is differentiable, and in particular increasing when moving from the optimal spiral to any other one (homotopy argument).
\end{enumerate}
Due to the complexity of the solution, the evaluation of the quantities needed to prove that the functional is increasing is performed numerically.
\end{abstract}

\bigskip

\maketitle

\begin{center}
	Preprint SISSA: 09/2025/MATE
\end{center}

\tableofcontents

\section{Introduction}
\label{S:intro}

We study a dynamic blocking problem first proposed by A. Bressan in \cite{Bressan_2007}. The problem is concerned with the model of wild fire spreading in a region of the plane $\R^2$ and the possibility to block it by constructing \emph{barriers}, i.e. rectifiable curves which cannot be crossed by the fire. The local fire spreading is described by giving a set valued function
$$
F : \R^2 \to \big\{ K \subset \R^2 \ \text{compact, closed} \big\}.
$$
such that
\begin{enumerate}
\item there exists $r>0$ such that $B_r(0)\subset F(x)\quad\forall x\in\R^2$;
\item $x\rightarrow F(x)$ is Lipschitz-continuous in the Hausdorff topology.
\end{enumerate}
This function represents the infinitesimal fire evolution when barriers are not present, i.e. $x + \delta t F(x)$ are the points reached by the fire in the infinitesimal time $\delta t$. In other words, if barriers are not present, the \emph{reachable set $R(t)$} is the solution to the differential inclusion
\begin{equation*}
R(t) = \Big\lbrace x(t), \quad x(\cdot) \text{ abs. cont.},\ x(0) \in R_0,\ \dot x(\tau)\in F(x(\tau))\text{ for a.e. }\tau\in[0,t] \Big\rbrace,
\end{equation*}
where the set $R_0\subset\R^2$ is the open region burnt by the fire at the initial time $t=0$.

%
%
%

When the fire starts spreading, a fireman can construct some barriers, modeled by a one-dimensional rectifiable set $\zeta\subset\R^2$, in order to block the fire. More in detail, we consider a continuous function $\psi: \R^2 \rightarrow \R^+$ together with a positive constant $\psi_0>0$ such that $\psi\geq\psi_0$. If we denote by $\zeta(t)\subset\R^2$ the portion of the barrier constructed within the time $t\geq 0$, we say that  $\zeta$ is an admissible barrier (or admissible strategy) if
\begin{enumerate}
\item{(H1)}
$\zeta(t_1)\subset \zeta(t_2)$, $\forall  t_1\leq t_2$;
\item{(H2)}
$\int_{\zeta(t)}\psi d\H^1\leq t,\quad \forall t\geq 0$,
\end{enumerate}
where $\H^1$ denotes the one-dimensional Hausdorff measure. Once we have an admissible strategy $\zeta$, then we define the reachable set for $\zeta$ at time $t$ the set
\begin{equation*}
R^\zeta(t) = \Big\lbrace x(t): x\text{ abs. cont.},\ x(0) \in R_0, \ \dot x(\tau)\in F(x(\tau))\text{ for a.e. } \tau\in[0,t],\ x(\tau)\not\in\zeta(\tau) \ \forall\tau\in[0,t] \Big\rbrace.
\end{equation*}

\begin{definition}
Let $t\rightarrow\zeta(t)$ be an admissible strategy. We say that it is a \emph{blocking} strategy if
\begin{equation*}
R^\zeta_\infty\doteq\bigcup_{t\geq 0} R^\zeta(t)
\end{equation*}
is a bounded set.
\end{definition}

We call \emph{isotropic} the case in which the fire is assumed to propagate with unit speed in all directions, while the barrier is constructed at a constant speed $\sigma>0$, namely
\begin{equation}
\label{isotropic case}
F \equiv \overline{B_1(0)},\quad R_0=B_1(0), \quad \psi\equiv\frac{1}{\sigma},
\end{equation}
where $\overline{B_1(0)}$ denotes the closure of the unit ball of the plane centered at the origin. We remark that in \cite{Bressan_friends} there are comparison results between more general choices of the data $R_0$ and $F$ and the isotropic problem for the study of the fire problem in a more general setting.

The existence of admissible blocking (or winning) strategies for the isotropic blocking problem is a very challenging open problem and it has been addressed mainly in \cite{Bressan_2007},\cite{Bressan_friends}.\footnote{One can prove that the existence of blocking strategy does not depend on the starting set $R_0$ but only on the speed $\sigma$ \cite{Bressan_survey}, see Proposition \ref{Prop:limit_speed}.} This is one of the main motivations for which the problem has been studied in the isotropic case. In particular, the following theorems hold:

\begin{theorem}
\label{thm:existence:2}
Assume that \eqref{isotropic case} hold. Then if $\sigma>2$ there exists an admissible blocking strategy.
\end{theorem}

\begin{theorem}
\label{thm:nonexistence:1}
Assume that \eqref{isotropic case} hold. Then if $\sigma\leq 1$ no admissible blocking strategy exists.
\end{theorem}

 The two theorems are proved in  \cite{Bressan_2007} and they motivate the following Fire Conjecture \cite{Bressan_conj}:
\begin{conjecture}
\label{conj:bressan}
Let \eqref{isotropic case} hold. Then if $\sigma\leq 2$ no admissible blocking strategy exists.
\end{conjecture}
For a survey of results related to the previous conjecture, see \cite{Bressan_survey}.

\medskip

\textbf{Equivalent formulation.} Throughout the paper we will use the following equivalent formulation of the dynamic blocking problem for the isotropic case (for a proof of the equivalence, see \cite{Bressan_equivalent}): let $Z\subset\R^2$ be a rectifiable set. We denote by
\begin{equation}
\label{eq:burned:region_intro}
R^Z(t) = \Big\lbrace  x(t): x\text{ abs. cont.},\ \dot x(\tau) \in \overline{B_1(0)} \text{ for a.e. } \tau \in [0,t],\ x(\tau)\not\in Z \ \forall \tau \in[0,t] \Big\rbrace.
\end{equation}
We say that the strategy $Z$ is \emph{admissible} if $\forall t\geq 0$
\begin{equation*}
\H^1(Z\cap\overline{R^Z(t)})\leq\sigma t.
\end{equation*}
Similarly to the previous formulation we denote by
\begin{equation}
\label{eq:burned:region}
R^Z_\infty = \bigcup_{t\geq 0} R^Z(t)
\end{equation}
the burned region and we say that $Z$ is an admissible blocking strategy if $R^Z_\infty$ is bounded. The advantage of this description is that the barrier is fixed and it does not grow while the time evolves. Here we can interpret $Z \cap \overline{R^Z(t)}$ as the amount of barrier $Z$ burned by the fire within the time $t$.

\medskip

 \textbf{Minimum time function and Hamilton-Jacobi formulation.} The propagation of fire can be described also in terms of the minimum time function
\begin{equation}\label{eq:min:time:function}
u(x) \doteq \inf \big\lbrace t \geq 0:\ x \in \overline{R^Z(t)} \big\rbrace.
\end{equation}
The function $u$ is the time needed for the fire to reach the point $x$ in the burned region, without crossing the barrier. The minimum time function can be computed by solving an Hamilton-Jacobi equation with obstacles, namely
\begin{equation}
\label{eq:hamilton:jacobi}
\begin{cases}
|\nabla u(x)|\leq 1 & x\not\in Z, \\
u(x)=0 & x\in R_0.
\end{cases}
\end{equation}
For the properties of the solution of \eqref{eq:hamilton:jacobi} we refer to \cite{delellisrobyr}. We conjecture that $|\nabla u(x)|=1$ for every $x\not\in Z$ as in the classic case of Eikonal equation without obstacles. Given the minimum time function $u$, and an admissible barrier $Z$, one defines the following \emph{Admissibility Functional}
\begin{equation*}
\mathcal A(x) = u(x) - \frac{1}{\sigma} \H^1 \big( Z\cap\overline{R^Z(u(x))} \big) \geq 0, \quad \forall x \in Z.
\end{equation*}
The inequality follows by the admissibility of the barrier $Z$.
In particular, one defines the \emph{saturated set} as
\begin{equation*}
\mathcal{S} = \big\lbrace x \in Z:\ \mathcal A(x) = 0 \big\rbrace,
\end{equation*}
corresponding to those points on the barrier where the firefighters are using all the barrier available in order to block the fire.

One can easily prove that if the strategy $Z$ consists of a simple closed curve, then it is not admissible for $\sigma\leq 2$ \cite{Bressan_2007} (the key idea is taking the last point $P$ on the curve reached by the fire and observing that $\mathcal{A}(P)<0$), but there are only partial results in the literature if the strategy has more complicated structures, as for example the presence of internal barriers that slow down the fire.
\medskip

 \textbf{Optimization problem}. A possible approach to attack Bressan's Fire Conjecture \eqref{conj:bressan} would be to look at optimal strategies for a given minimization problem and prove that they can not be admissible blocking strategies, reducing the analysis to concrete cases. In this spirit, one  can give an optimization problem among blocking barriers as follows: one introduces the following cost functional
\begin{equation}\label{eq:opt:problem}
    J(Z)=\int_{R^Z_\infty} \kappa_1 d\mathcal{L}^2 +\int_{Z}\kappa_2 d\H^1,
\end{equation}
among all possible admissible blocking strategies, where $\kappa_1,\kappa_2:\R^2\rightarrow \R^+$ are two non-negative continuous functions. In \cite{Bressan_deLellis} it is proved that, if the class of admissible blocking strategies is non-empty, then there exists an optimal strategy. Moreover, there exists an optimal strategy $Z^*$ which is \emph{complete}:
\begin{definition}
Let $Z\subset\R^2$ be a rectifiable set. $Z$ is complete if it contains all its points of positive upper density,  that is, let $x\in\R^2$ such that
\begin{equation*}
\limsup_{r\to 0^+} \frac{\H^1(Z\cap B_r(x))}{2r}>0 \Longrightarrow x\in Z.
\end{equation*}
\end{definition}
 The following corollary also holds \cite{Bressan_equivalent}:
\begin{corollary}
If there exists an admissible blocking strategy $Z$ with $\H^1(Z)<\infty$, then there exists an optimal blocking strategy $Z^*$ such that
\begin{equation*}
Z^* = \bigg( \bigcup_i Z_i \bigg) \cup \mathcal{N},
\end{equation*}
where $Z_i$ are countably many compact rectifiable connected components and $\H^1(\mathcal{N})=0$.
\end{corollary}
We mention also the recent result in \cite{Bressan_Chiri} where it is proved that the optimal strategy is nowhere dense.  
What really makes the optimization problem hard is that the class of admissible strategy could be empty (and it is conjectured to be empty if $\sigma\leq 2$, Conjecture \ref{conj:bressan}).

Giving necessary conditions for optimality is a hard open problem. Some necessary optimality conditions are derived in \cite{Bressan_optnec} assuming further regularity on the optimal strategy. In \cite{FireSegment} some necessary optimality condition are given to characterize the boundary $\partial R^Z_\infty$ in the case $Z$ is an optimal strategy (for the problem $\kappa_1=0$, $\kappa_2=1$) without assuming any regularity assumption on the barrier. There, some properties of the optimal barrier are derived, as the connectedness in the case the $Z=Z^2\cup Z^1$ where $Z^2=\partial R^Z_\infty$ and $Z^1$ is a single internal curve, and some concrete cases are analyzed. The major issue, however, remains deriving necessary conditions for internal barriers.

\subsection{Non admissibility of spiral-like strategies}
\label{Ss:spiral_intro}

In this paper we study spiral-like strategies: namely, admissible barriers that are constructed putting all the effort on a single curve. The study of spiral strategies is of key importance in the  complete solution of Bressan's Fire conjecture, indeed it is conjectured that these strategies are the "best" possible barriers that can be constructed when $\sigma\leq 2$. The heuristic reason is that, if $\sigma\leq 2$, then we can not have two firefighters (constructing simultaneously two branches of the strategy) with speed strictly greater than the speed of the fire, so that they will be engulfed by the fire: thus the best strategy should be a single curve. The case $\sigma>2$ is different, indeed we expect that the optimal strategies in this case are simple closed curves. In \cite{Bressan_isotropic} the authors construct the best strategy for the optimization problem \ref{eq:opt:problem} among simple closed curves.

\medskip

We start giving the definition of spiral strategies.

\begin{definition}[Spiral strategies]
\label{Def:spiral_barrier_intro}
Let $Z = \zeta([0,S]) \subset \R^2$ be a strategy, where $\zeta$ is a Lipschitz curve parametrized by arc-length. We say that it is a \emph{spiral strategy} if it satisfies:
\begin{enumerate}
\item $\zeta(0) = (1,0)$ and $\zeta \rest_{[0,S)}$ is simple,
\item $s \mapsto u \circ \zeta(s)$ is increasing,
\item $\zeta$ is locally convex, i.e. locally it is the graph of a convex function,
\item \label{Point:4_intro_def_spiral} 
$\zeta$ is rotating counterclockwise about $0$. For a precise definition, see Section \ref{Ss:spiral_barrier_intro}.
\end{enumerate}
The spiral $\zeta$ is \emph{admissible} if it is also an admissible barrier.
\end{definition}

The assumptions that $\zeta$ is locally convex yields a direction of rotation to the curve, and then Point \eqref{Point:4_intro_def_spiral} selects a direction of rotation (Proposition \ref{thm:param_1}).

%
%

A deeper discussion is needed for the assumption 
\begin{equation*}
s \rightarrow u\circ\zeta(s) \quad \text{ is increasing},
\end{equation*}
which corresponds to the fact that either a portion of the barrier lies on the level set $\lbrace u=t\rbrace$ (so that the previous function is constant), or the fire can not burn two portions of the barrier simultaneously. 
Here we just observe that, given the barrier up to time $\bar t$ (or equivalently length $\bar s$), we can ask which is the optimal barrier for a given optimization problem with the same freedom as for the initial problem. This final observation will be of fundamental importance in our homotopy argument (Point \eqref{Point_5:abstract} in the abstract).

%

\medskip

The only results known on this family of barriers can be found in \cite{Bressan_friends} and \cite{firefighter}. In the two papers it is proved independently and with different techniques the following result:

\begin{theorem}
\label{bressan,klein}
Let $\sigma > 2.6144..$ (critical speed). Then there exists a spiral-like strategy which confines the fire.
\end{theorem}

The proof of this result is obtained by the study of a particular spiral, that we will call \emph{saturated spiral}. Let $Z$ be an admissible spiral. We say that $Z=Z_{\text{sat}}$ is a saturated spiral if
\begin{equation*}
\mathcal{S}(Z_{\text{sat}}) = \big\lbrace t \in[0,T]:\ \H^1 \big( Z_{\text{sat}} \cap \overline{R^{Z_{\text{sat}}}(t)}\big) = \sigma t \big\rbrace = [0,T].
\end{equation*}
In terms of the admissibility functional previously defined, every point $P\in Z$ on a saturated spiral is \emph{saturated}, that is $\mathcal{A}(P)=0$.

Saturated spirals (or compact and connected subsets of saturated spirals) are thought to be building blocks for optimal strategies. Indeed, it has been proved in \cite{Bressan_isotropic} that, if $\sigma>2$, the optimal strategy (that minimizes the cost functional \eqref{eq:opt:problem}) among simple closed curves is the union of an arc of circle and two branches of saturated logarithmic spirals in polar coordinates. As we will see, there are intimate relations between logarithmic spirals and saturated spirals. Also, in \cite{FireSegment}, it is proved that for $\sigma<2$, if a strategy $Z$ is optimal ($\kappa_1=0,\kappa_2=1$) and its internal barrier is a segment, then the boundary $\partial R^Z_\infty$ is the union of a segment and a branch of logarithmic spiral.

The proof of Theorem \ref{bressan,klein} relies on the following fact: a saturated spiral $Z_{\text{sat}}$ can be represented by means of polar like coordinates $(\phi,r_{\text{sat}}(\phi)) \in [0,+\infty) \times \R^+$ (see Theorem \ref{Cor:angle_preprest}), and moreover $r_{\text{sat}}(\phi)$ solves the following \emph{Retarded Differential Equation (RDE)} 
\begin{equation}
\label{eq:ODE:saturated:1}
\frac{d}{d\phi} r_{\text{sat}}(\phi) = \cot(\alpha) r_{\text{sat}}(\phi) - \frac{r_{\text{sat}}(\phi - (2\pi+\alpha))}{\sin(\alpha)}
\end{equation}
with initial data
\begin{equation}
\label{Equa:init_satur}
r_{\text{sat}}(\phi) = \begin{cases}
e^{\cot\alpha\phi} & \forall\phi\in[0,2\pi), \\
(e^{2\pi\cot\alpha}-1) e^{\cot\alpha(\phi-2\pi)} & \forall \phi \in [2\pi,2\pi+\alpha].
\end{cases}
\end{equation}
The angle $\alpha \in [0,\frac{\pi}{2})$ is given by
\begin{equation*}
\alpha = \arccos \left( \frac{1}{\sigma} \right),
\end{equation*}
(this is precisely the approach followed by \cite{firefighter}). The angle $\alpha$ is the constant angle made by the fire ray with the spiral. Here the initial fire is the set $R_0 = B_1(0)$ and the spiral starts in a point $P=(1,0)$. Observe that for $\phi<2\pi$ the solution is an arc of the logarithmic spiral in polar coordinates $r_{\text{sat}}(\phi) = e^{\cot\alpha\phi}$.

A study of the eigenvalues of the operator associated to the previous RDE gives the proof of Theorem \ref{bressan,klein}: there is a critical value ($\bar \sigma = 2.6144...$, solution of a trascendental equation) such that if $\sigma > \bar \sigma$ then all eigenvalues of the associated linear operator are complex conjugate: in particular the solution oscillates, and thus there is some first angle $\bar \phi$ such that $r_{\text{sat}}(\bar \phi) = 0$. This corresponds exactly to the fact that $Z_{\text{sat}}$ blocks the fire.


However, in \cite{Bressan_friends} it is not proved that, instead, for $\sigma \leq \bar \sigma$ the saturated spiral does not confine the fire, that is $r_{\text{sat}}(\phi) > 0$ for every angle $\phi$: this fact can be deduced from \cite{firefighter}, Theorem $3$, which gives the estimate
\begin{equation*}
\text{minimal closing angle}\sim(\text{imaginary part of the largest eigenvalue})^{-1}.
\end{equation*}
Their result motivates the following conjecture for general spirals.

\begin{conjecture}
\label{conj:spiral}
If $\sigma \leq \bar \sigma$ then no spiral-like strategy blocks the fire.
\end{conjecture}

We remark here that due to the constraints on the barrier, the critical speed $\bar\sigma$ is larger than the speed $\sigma=2$ for the original conjecture \ref{conj:bressan}.

%
%
%
%
%

A partial answer to this conjecture has been given in \cite{firefighter} where the authors use a geometric argument to  prove the following result:

\begin{theorem}
\label{thm:klein}
If $\sigma\leq\frac{1+\sqrt{5}}{2}$ then no spiral-like strategy is admissible. 
\end{theorem}

This bound is obtained in the assumption that only some points on the spiral are admissible, while assuming the admissibility of the whole barrier requires more effort.

The aim of this paper is to prove the previous conjecture: 

\begin{theorem}
\label{thm:main}
No admissible spiral-like strategy confines the fire if $\sigma \leq \bar \sigma$.
\end{theorem}

\medskip

We conclude this first part of the introduction by giving a very much simplified idea of our approach. We will neglect some technical issue, which will be further addressed in the section below.

\begin{figure}
\resizebox{.75\textwidth}{!}{\includegraphics{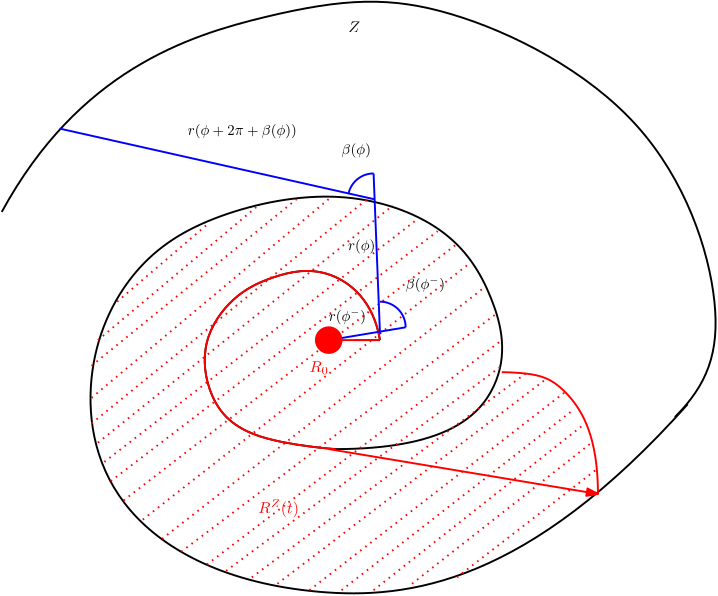}}
\caption{The angle-ray parametrization $(\phi,r(\phi))$ of a spiral barrier: the optimal fire ray is red, and the red curve is the level set of the minimum time function $u$. In blue, the angle-ray coordinates.}
\label{Fig:angle_spiral_intro}
\end{figure}

First, a spiral strategy admits a polar-like representation $(\phi,r(\phi)) \in [0,+\infty) \times \R^+$ as in Fig. \ref{Fig:angle_spiral_intro} (we will call this representation the \emph{angle-ray coordinates}):
\begin{itemize}
\item $r$ is the length of the segment starting from one point of $\zeta$ and with direction $\dot \zeta$ and ending in the next round of $\zeta$;
\item $\phi$ is the rotation angle of $r$;
\item $\beta(\phi)$ is the angle between $\dot \zeta$ and the direction of $r(\phi)$.
\end{itemize}
We can also think of $r(\phi)$ as the last part of the optimal ray needed to compute the minimum time function $u$ in the point $\zeta(\phi)$, where we used $\phi$ instead of $s$.

The parameter $\beta(\phi)$ is our control parameter: once it is chosen, the RDE satisfied by $r(\phi)$ is
\begin{equation*}
\frac{d}{d\phi} r(\phi) = \cot(\beta(\phi)) r(\phi) - \frac{r(\phi^-)}{\sin(\beta(\phi^-){\color{red})}},
\end{equation*}
where the angle $\phi^-$ is given by (see Fig. \ref{Fig:angle_spiral_intro})
\begin{equation*}
\phi = \phi^- + 2\pi + \beta(\phi^-).
\end{equation*}
Using this representation, and observing that $r(\bar \phi) = 0$ means that the spiral closes into itself and thus blocks the fire, we replace the closure problem with the following minimum problem:
\begin{equation}
\label{Equa:minimum_rphi_intro}
\text{{\bf Optimization Problem}: given a rotation angle $\bar \phi$, choose $\phi \mapsto \beta(\phi)$ in order to minimize $r(\bar \phi)$.}
\end{equation}
Differently from the original problem (where we have to show that the set of admissible strategies is empty), we prove that the one above has an optimal solution.

We observe that while the RDE for $r$ is linear, the choice of the control parameter affects the RDE by changing the coefficients at the next round and also by varying the delay interval $2\pi + \beta(\phi^-)$: hence it seems quite difficult to find a positive lower bound for the optimization problem \eqref{Equa:minimum_rphi_intro}, or even more to prove that a choice of the control $\beta(\phi)$ is optimal. Notice also that $\beta$ depends on the angle $\bar \phi$, i.e. the optimal spiral is not the same for all angles (and in particular we will prove it is not the saturated spiral).

In order to overcome these difficulties, we use a homotopy approach as follows. We will refer to Figure \ref{Fig:fire:intro:2}.

\begin{figure}
\resizebox{.60\textwidth}{!}{\includegraphics{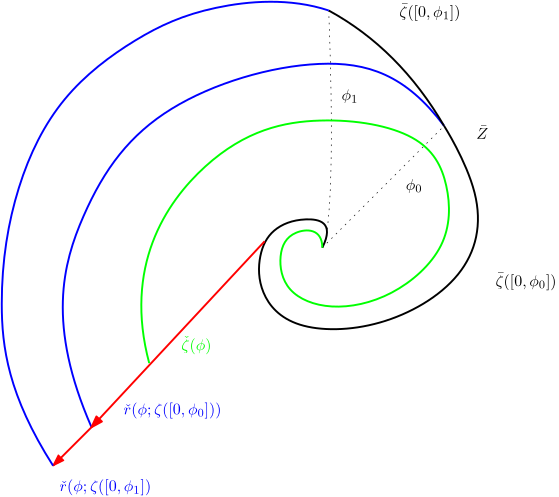}}
\caption{The barrier $\zeta([0,\phi_0])$ is fixed. Given any angle $\phi\geq \phi_0$, we construct a new spiral $\check\zeta(\phi;\bar\zeta([0,\phi_0])$ (blue). The family of continuations of $\bar\zeta$ (black spiral) is represented in blue. The base spiral (green) is $\check\zeta(\phi;\bar\zeta(\lbrace 0\rbrace))$.}
\label{Fig:fire:intro:2}
\end{figure}

 Let $\bar \zeta(\phi)$ be a given spiral strategy, parametrized by the rotation angle: for every $\phi_0$, we construct a new spiral $\check\zeta(\phi;\bar \zeta([0,\phi_0]))$ with the following properties:
\begin{enumerate}
\item $\check\zeta(\phi;\bar \zeta([0,\phi_0])) = \bar \zeta(\phi)$ for all $\phi \in [0,\phi_0]$;
\item \label{Point:hat_spiral_2_intro} if $\check r(\phi;\bar \zeta([0,\phi_0]))$ is the representation of the spiral in the angle-ray coordinates, the final value $\phi_0 \mapsto \check r(\bar \phi;\bar \zeta([0,\phi_0]))$ is differentiable and has positive derivative;
\item \label{Point:hat_spiral_3_intro} the spiral $\check \zeta (\phi)= \check \zeta(\phi;\bar \zeta(\{0\}))$ is such that $\check r(\bar \phi) > 0$.
\end{enumerate}
Since the spiral $\bar \zeta$ is fixed during the homotopy argument, we shorten the notation to $\check r(\cdot,\phi_0)$ or also $\check r(\cdot;s_0)$ when using the length parametrization for $\bar \zeta$. 

The existence of such a family for every spiral $\bar\zeta$ gives thus the expanded version of the main Theorem \ref{thm:main} (which now becomes a corollary):

\begin{theorem}
\label{Theo:main_extended}
For each spiral strategy $\bar \zeta$ and rotation angle $\bar \phi$, the spiral $\check \zeta(\cdot;\phi_0)$ is the optimal solution to the minimization problem \eqref{Equa:minimum_rphi_intro} given the initial spiral arc $\bar \zeta([0,\phi_0])$. In particular, $\check \zeta(\phi;0)$ is optimal among all spirals.
\end{theorem}

Given the optimality of $\check \zeta(\phi;0)$, we will refer to this spiral as $Z_\mathrm{opt} = \zeta_{\text{opt}}([0,\bar \phi];0)$.
Point \eqref{Point:hat_spiral_3_intro} and Point \eqref{Point:hat_spiral_2_intro} above implies that
\begin{equation*}
\check r(\bar\phi;s_0)=\check r(\bar\phi;0)+\int_0^{\phi_0} \frac{d}{d\phi} \check r(\bar \phi;\phi_0) d\phi \geq \check r(\bar\phi;0)
\end{equation*}
but since this holds for every $s_0$, the previous equation is also valid for the last parameter $\bar s$, so that $\bar r(\bar\phi)\geq\check r(\bar\phi;0)$,
where we used $(\phi,\bar r(\phi))$ as the angle-ray coordinated of the fixed spiral $\bar Z$. In particular Theorem \ref{thm:main} holds.

We will also prove that the only case when the derivative $\frac{d}{d\phi_0} \check r(\bar \phi;s_0)$ is $0$ is by following exactly the spiral $\check r(\cdot;\phi_0)$, so that the above chain of inequalities become strict as soon as we perturb the spiral $Z_{\text{opt}}$. Hence we have the following

\begin{corollary}
\label{Cor:uniqueness}
The optimal spiral strategy for the angle $\bar \phi$ is $Z_{\text{opt}}$, parametrized by $\check r(\phi;0) = \check r(\phi)$.
\end{corollary}

We remark that by changing the angle $\bar \phi$, the optimal spiral changes. Therefore we will often refer to the optimal spiral at a fixed angle $\bar\phi$. We also remark that the optimal spiral is not regular, but only Lipschitz at the last round. This can be heuristically explained by observing that the last round is not having influence on the next evolution (being the last one!), so that the regularizing effect of the RDE is not acting.

Being able to describe the optimal spiral, it is possible to give estimates, in particular one can prove the asymptotic divergence of the optimal spiral:

\begin{corollary}
\label{Cor:divergence}
The spiral $Z_{\text{opt}}$, optimal at the angle $\bar \phi$, satisfies
\begin{equation*}
\check r(\bar \phi) \simeq \bar \phi e^{\bar c \bar \phi},
\end{equation*}
where $\bar c$ is the only positive eigenvalue of the RDE.
\end{corollary}

We are also able to give a precise value to the constant in front of $\phi e^{\bar c \phi}$, as we will point out in the extended introduction below.


Summarizing, we build up a new differential tool to treat spiral-like strategies. The choice of the angle-ray coordinates for a given spiral is one of the key ideas of this paper. In particular, in these coordinates the spiral solves a retarded differential equation (RDE) which allows to set up a new optimization problem, admitting a differential structure. Given any spiral $\bar Z$, it is indeed possible to find a differential family of spirals and to use an homotopy-type approach, in order to prove that any spiral-like strategy does not confine the fire for $\sigma\leq\bar\sigma=2.6144...$. Here, the main point is to guess what is the correct expression of the family of spirals, to prove that it is differentiable and compute the derivative, and finally to show that this derivative is positive. This will require the detailed study of Sections \ref{S:family}, \ref{S:optimal_solut}, \ref{S:optimal_sol_candidate}, \ref{S:segment_case}, \ref{S:arc_1}, together with the use of the Software Mathematica for the evaluation of trascendental functions. Each single building-block is fundamental in the proof of Bressan Fire Conjecture for spirals in its sharp form. We point out that the importance of this work is not only the proof of the conjecture in the setting of spiral-like strategies, but the implementation of a framework that could be applied in order to study more general situations, hopefully leading to the full solution of the conjecture. 

\medskip

We remind that all the notations are contained in the Glossary at the end of the paper.

\subsection{Extended introduction and collection of all principal results}
\label{Ss:extend_more_intro}

In this section we present an extended introduction to the proof of Theorem \ref{Theo:main_extended}, and all the steps required to construct the family $\check r(\cdot;\phi_0)$. Many of the results obtained along the proof are interesting on their own, and we will report their statement.

\subsubsection{Section \ref{s:admissible:barriers}}
\label{Sss:admiss_barr_intro}

This section introduces the basic objects considered in this paper: the most important ones are the set of admissible curves \gls{Gamma}, Equation \eqref{eq:admissible:curves} (recall also \eqref{eq:burned:region_intro}), the minimum time function \gls{utimefunct}, Equation \ref{eq:solution:HJ} (recall \eqref{eq:min:time:function}), the admissibility functional \gls{Acaladmisx}, Equation \ref{eq:admiss:funct}, and the saturated set \gls{Scal}, Equation \eqref{eq:sat:set}. The limit of a minimizing sequence of admissible curves in $\Gamma$ are the \emph{optimal rays}, Definition \ref{def:optimal:ray}. These Lipschitz curves are not admissible in general, but are the ones needed to compute the value of the minimum time function $u(x)$.

The first result is in Section \ref{Ss:prelimnaries}, where the existence of a closing barrier for every initial burning set $R_0$ is reduced to the study when $R_0 = \{0\}$ and the barrier is constructed outside the unit ball $B_1(0)$, Proposition \ref{Prop:limit_speed}. The equivalence is not exact as stated in the proposition, but since for $\sigma \leq \bar \sigma$ we show that it is not possible to block the fire even in this situation, we conclude that in our case the equivalence is complete.

Section \ref{Ss:spiral_barrier_intro} can be thought as a detailed description of the notion of spiral barrier of Definition \ref{Def:spiral_barrier_intro}: in particular Remark \ref{Rem:comment_single_curve} explains the importance of the monotonicity of $u(\zeta(s))$. Some notation for convex barriers are also introduced, like the left and right tangent vectors \gls{t-t+} , the local radius of curvature \gls{Rcurva}, the subdifferential \gls{subdiffzeta} (Definition \ref{Def:supporting_direction}). \\
The most important result of this section is that for spiral barrier one can use the angle-ray representation as considered below Theorem \ref{thm:main}. This result is obtained in two steps. First, Proposition \ref{thm:param_1} constructs this representation round-by-round, starting with the initial round about the origin $(0,0) \in \R^2$ (where it corresponds to polar coordinates), and then recursively constructing the angle at the next round. The idea is that the direction of $r(\phi)$ is determined by the previous round, and then the angle $\beta(\phi)$ is deduced by knowing the tangent $\mathbf t = \dot \zeta$ at that point: a look at Fig. \ref{fig:param:spiral} is probably better than any explanation. \\
A version of Proposition \ref{thm:param_1} avoiding the explicit reference to rounds is Theorem \ref{Cor:angle_preprest}. Together with the function \gls{rphi}, we also introduce the functions \gls{s-phi}, \gls{s+phi}, which are the length parametrization of the starting/ending point of the segment $r(\phi)$. Their regularity properties ($s^-(\phi)$ BV, $s^+(\phi)$ Lipschitz) are also proved. We conclude this section by using the angle-ray representation to rewrite the convexity of $\zeta$ in terms of the angle \gls{beta}, the admissibility functional $\mathcal A = \mathcal A(\phi)$ and the burning rate \gls{b(t)} of the barrier $\zeta$, Proposition \ref{Lem:burnign_rate}.

As a warning, because of Theorem \ref{Cor:angle_preprest} we can use indifferently the length $s$ or the angle $\phi$ as parameter for a spiral barrier: in the following we will often choose the most convenient one, without underlining it, in order to simplify notation and shorten the analysis.

\subsubsection{Section \ref{S:ODE}}
\label{Sss:ODE_intro}

The first part of the section is devoted to obtain the RDE satisfied by a spiral barrier: Lemma \ref{lem:ODE:spiral} offers some differential relations among the various quantities introduced in the angle-ray description of a spiral. Equation \eqref{eq:spiral:angle} is the distributional RDE for the spiral barrier: it is nothing else than the relation
\begin{equation*}
\big\{ \text{variation of $r$} \big\} \ = \ \big\{ \text{increase due to $\beta$ at the end point} \big\} \ - \ \big\{ \text{decrease due to curvature at the base point} \big\}.
\end{equation*}
Proposition \ref{Prop:construct_spiral} is the inverse statement: assume that $r(\phi),\beta(\phi)$ satisfies the differential relations of Lemma \ref{lem:ODE:spiral}, then, up to fixing the initial point, there is a unique spiral barrier such that its angle-ray coordinates are the given ones. The proof is based on the round decomposition.

In Section \ref{Sss:equation_satur} we prove that the equation of the saturated spiral is \eqref{eq:ODE:saturated:1} with initial data \eqref{Equa:init_satur}.

The next section contains the analysis of the linear RDE
\begin{equation}
\label{Equa:satur_spiral_intro}
\frac{d}{d\phi} r(\phi) = \cot(\alpha) r(\phi) - \frac{r(\phi - 2\pi - \alpha)}{\sin(\alpha)}
\end{equation}
with $\alpha$ constant. This section contains many results which have already been proved in \cite{firefighter}: we present them for completeness. The main one is that there is a critical angle $\bar \alpha$ such that for $\alpha < \bar \alpha$ the characteristic equation of the above RDE has two distinct positive real eigenvalues, for $\alpha = \bar \alpha$ the two eigenvalues coincide and the geometric multiplicity is $1$, for $\alpha > \bar \alpha$ all eigenvalues are complex conjugate, Corollary \ref{Cor:eigen_angle}. This gives in particular that for $\alpha > \bar \alpha$ the saturated spiral confines the fire, Proposition \ref{prop:at:crit:speed:fire:dies}.

Section \ref{ss:subset:saturated} is concerned with the converse result for the saturated spiral. Observe that even if there are positive real eigenvalues, if the initial data is $0$ when projected on the corresponding eigenspace then the solution oscillates and thus the fire is blocked: thus the idea is to give a condition ensuring the exponential blow-up of the solution and checking that the saturated spiral satisfies this condition. \\
The starting point is to rescale the variables 
\begin{equation*}
\gls{rhogrowt} = r((2\pi+\bar \alpha) \tau) e^{-\bar c (2\pi + \bar \alpha) \tau},
\end{equation*}
with $\bar c$ equal to the only real eigenvalue (see Corollary \ref{Cor:divergence}), and rewrite the RDE in the critical case as
\begin{equation}
\label{Equa:retarde_intro}
\dot \rho(\tau) = \rho(\tau) - \rho(\tau-1).
\end{equation}
We remark that in this parametrization the integer part of $\tau$ represents the number of rounds of the spiral.
Lemma \ref{lem:key} gives a simple-to-check criterion to verify if the solution diverges:

\begin{lemma}[Second part of Lemma \ref{lem:key}]
\label{lem:key_intro}
%
In the critical case $\sigma=\bar\sigma$, if $\rho(1) > \rho(\tau) > 0$, $\tau \in [0,1)$, then the solution $\rho(\tau)$ diverges linearly.
%
\end{lemma} 

Even if it seems trivial, it gives immediately that the saturated spiral is diverging:

\begin{proposition}[Proposition \ref{prop:saturated:not:closed_1}]
\label{prop:not_closed_intro}
Let $Z_{\text{sat}}$ be the saturated spiral. Then it does not confine the fire for $\sigma\leq \bar \sigma$. 
\end{proposition}

The optimal solution $\check r(\cdot;\phi_0)$ and its derivatives have a more complicated structure than the saturated spiral: hence Section \ref{Sss:green_kernels} introduces the Green kernel \gls{gkernel} for the RDE \eqref{Equa:retarde_intro} (Lemma \ref{Lem:explic_kernel_RDE}), and there it is studied its asymptotic behavior (Lemma \ref{Lem:aympt_g}). It is essentially standard analysis for a linear functional equation, and for completeness we write also the explicit forms of the kernel \gls{Gkernel} for the RDE of the saturated spiral \eqref{Equa:satur_spiral_intro}.

Since it will be useful to study also the evolution of the length of a spiral barrier, this section ends with some useful estimates for the primitives of $G(\phi)$ (Section \ref{Sss:length_saturated}).

\subsubsection{Section \ref{S:exist_opti_traj}}
\label{Sss:exists_opti_intro}

In this section we prove that the optimization problem \eqref{Equa:minimum_rphi_intro} has a solution in the domain of admissible spiral barriers with bounded length $\bar L$, Theorem \ref{Theo:exists_min}:

\begin{theorem}[Theorem \ref{Theo:exists_min}]
\label{Theo:exists_min_intro}
Fix a rotation angle ${\bar \phi} \geq {\phi_0}$. Then either there exists an admissible spiral in $\mathcal A_S(\bar Z,{\phi_0})$ with $\bar L \gg 1$ blocking the fire before or at ${\bar \phi}$ or there exists an admissible spiral $\zeta$ such that $r({\bar \phi})$ is minimal among all admissible spirals of length bounded by $\bar L$.
\end{theorem}

Here $\mathcal{A}_S(\bar Z,\phi_0)$ is the set of admissible continuations (Definition \ref{def:adm:spirals}), that is the set of admissible spirals that coincide with a given spiral $\bar Z$ up to the rotation angle $\phi_0$.
Note that in the above statement we consider the more general case where a first arc of $r(\phi)$, $\phi \leq \phi_0$, is fixed, and we are allowed to choose the remaining part. This is perfectly in line with the statement of Theorem \ref{Theo:main_extended}, where the optimality is proved given the initial arc $\bar \zeta([0,\phi_0])$.

A solution means that either there is an admissible spiral blocking the fire before the angle $\bar \phi$, or there is a minimizing one, and the length $\bar L$ is sufficiently large to guarantee that the set of admissible spirals arriving at $\bar\phi$ with length shorter than $\bar L$ is not empty. The approach is the standard one: first one proves the compactness of the set of admissible spirals, Proposition \ref{Prop:compact_curves} and Corollary \ref{Cor:convergs}, and then passes to the limit to a minimizing sequence, being the functional $\zeta \mapsto r(\bar \phi)$ continuous. We just remark that as a corollary we have that the admissibility functional is u.s.c. for the Hausdorff convergence, hence giving the closure of the family of admissible spirals, and also observe that a key ingredient is the decomposition by rounds of Proposition \ref{thm:param_1}. \\
The parameter $\bar L$ will not play any role in the future: indeed once we show that the spiral $\check r$ is optimal, we have an estimate for all spirals. It is introduced just to ensure that the curves $\zeta$ live in a compact set of $\R^2$.

When studying the derivative $\frac{d}{ds_0} \check r(\cdot;s_0)$ of $\check r(\phi;s_0)$ (see the discussion at the end of Section \ref{Sss:admiss_barr_intro} for the equivalence of $\phi$ and $s$), it is important to know the structure of the minimal time function $u$ generated by $\bar \zeta([0,s_0])$. Lemma \ref{Lem:convexity} and Corollary \ref{lem:C:1:1} give explicit computations of the first and second derivatives of $u$: roughly speaking, we are solving the Eikonal equation, because of the assumptions on the fire spreading speed $F(x) = \bar B_1(0)$, and thus $\nabla u$ is BV and the level sets of $u$ are $C^{1,1}$-curves.

\begin{figure}
\resizebox{.75\textwidth}{!}{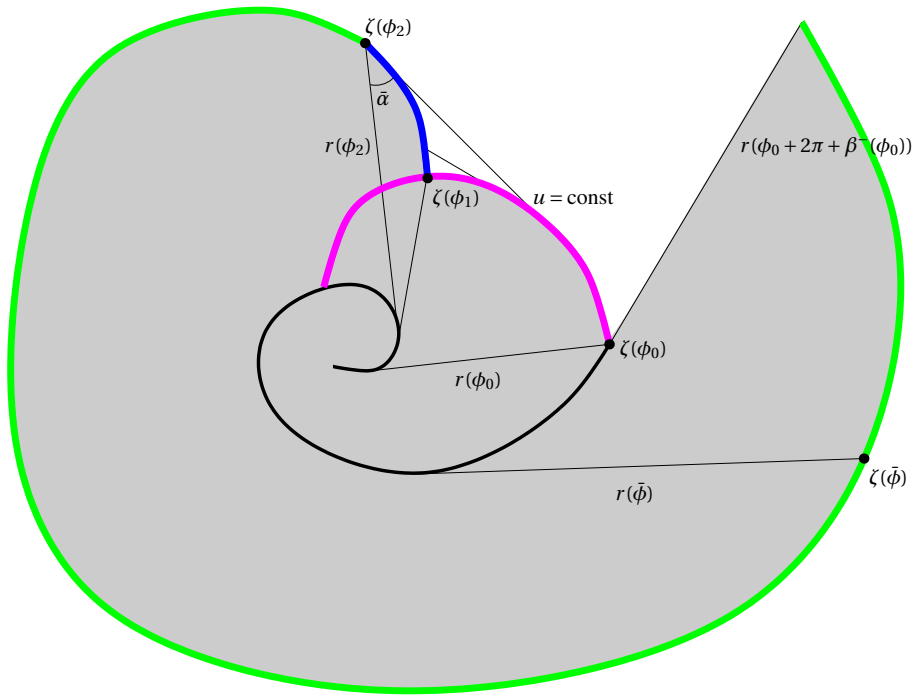}
\caption{The solution to the optimization problem at the last round. The gray region is the region which cannot be reached by any admissible spiral.}
\label{Fig:optimal_last_intro}
\end{figure}

The first glimpse on the structure of the optimal solution $\check r$ is given in Section \ref{Ss:minimal_reach}: if we restrict the optimality to just one round after $\phi_0$, then $\check r$ can be explicitly constructed without further analysis. \\
To describe the structure of such a solution, we refer to Fig. \ref{Fig:optimal_last_intro}. Starting from the angle $\phi_0$, the solution is made by three components. In the following, we will mean by \emph{arc of the level set}, or simply \emph{arc}, a compact and connected subset of $\lbrace u^{-1}(t)\rbrace$.
\begin{enumerate}
\item We follow the level set of the minimal time function $u$ until the point $\zeta(\phi_1)$, which is saturated. At this point it is not possible to proceed along this level set, because the spiral would not be admissible. This arc (the magenta one between the points $\zeta(\phi_0)$ and $\zeta(\phi_1)$ in the picture) is the optimal solution for angles $\bar \phi \in [\phi_0,\phi_1]$, since the monotonicity of $s \mapsto u(\zeta(s))$ gives that we cannot cross the level set of $u(\zeta(\phi_0))$. In the case $\zeta(\bar \phi_0)$ is saturated, then this regions reduces to a point.

The optimal spiral barrier coincides here with the arc of the level set.

\item If $\bar \phi > \phi_1$, then we consider the spirals made by a portion of the magenta arc and a segment: the segment ends in the first point that is saturated. We show that such a construction defines a convex curve (the blue one in Fig. \ref{Fig:optimal_last_intro}) starting orthogonally to the level set of $u$ and ending in a saturated point $\zeta(\phi_2)$ such that the angle formed with the direction of the optimal ray is exactly $\bar \alpha$ (this argument in reality holds for any $\sigma\leq\bar\sigma$, but here we prove it for the critical case). Also in this case, if the initial point is saturated, it holds $\phi_2 = \phi_0$.

The optimal spiral barrier is an arc of the level set and a segment ending in the blue curve.

\item For angles $\bar \phi > \phi_2$ one follows the saturated spiral starting from $\zeta(\phi_2)$ (green spiral in the picture). The optimal spiral barrier is an arc of the level set, the segment arriving in $\zeta(\phi_2)$ and an arc of a saturated spiral.
\end{enumerate}

\begin{figure}
\resizebox{\textwidth}{!}{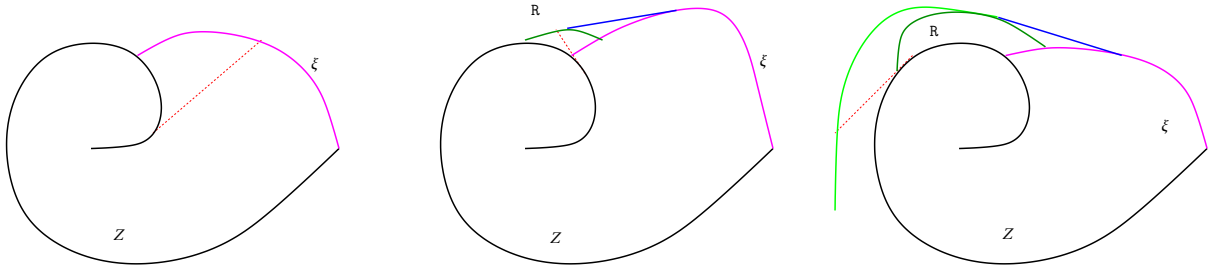}
\caption{The three cases of Theorem \ref{Cor:curve_cal_R_sat_spiral_intro}: in the first picture all level set $\xi$ is admissible, and then for every angle $\geq {\phi_0}$ the optimal solution is an arc of the level set; in the second case all curve $\mathtt R$ is admissible, and then the optimal solutions for point on $\mathtt R$ is an arc of $\xi$ and a segment; in the third case for large angles the minimal point lies on a saturated spiral, and the optimal solution is an arc of $\xi$, a segment and an arc of saturated spiral.}
\label{Fi:three_cases_intro}
\end{figure}

It is important to observe that the curve formed by the magenta arc up to $\phi_1$, the blue convex curve and the green saturated spiral is not a spiral barrier: indeed it is not convex. It gives the optimal solution for $\phi_0 \leq \bar \phi \leq \phi_0 + 2\pi + \beta^-(\phi_0)$, where $\beta^-(\phi_0) = \lim_{\phi \nearrow \phi_0} \mathbf t(\phi)$, i.e. the left tangent vector: for every angle $\bar \phi$ there exists an explicit spiral barrier ending in the corresponding point of the curve constructed above. The three situations above are drawn in Fig. \ref{Fi:three_cases_intro}. The main result is then the following

\begin{theorem}(Theorem \ref{Cor:curve_cal_R_sat_spiral})
\label{Cor:curve_cal_R_sat_spiral_intro}
The curves constructed above are the unique solutions to the minimum problem considered in Theorem \ref{Theo:main_extended} for ${\bar \phi} \in [{\phi_0},{\phi_0} + 2\pi + \beta^-({\phi_0}))$.
\end{theorem}

A question which arises naturally is the following: if we continue the saturated spiral also for angles larger that $\phi_0 + 2\pi + \beta^-(\phi_0)$, is this spiral barrier optimal? The short answer is no, but the proof of this fact and the construction of the correct optimal spiral require much more analysis and it is the central part of this work. Due to its importance, we will refer to the spiral constructed by the procedure above and obtained by prolonging the saturated spiral after the angle $\phi_ 0 + 2\pi + \beta^-(\phi_0)$ as the \emph{fastest saturated spiral}.

\subsubsection{Section \ref{S:case:study}}
\label{Sss:case_study_intro}

The construction above when $\phi_0 = 0$ is a particular spiral which is important for the next analysis, since the optimal one will be obtained as a perturbation of this spiral: the study of its properties is in Section \ref{S:case:study}. For the seek of generality, we consider the case where $R_0 = B_{\mathtt a}(0)$, and show explicitly the form of the fastest saturated spiral: depending on the radius $\mathtt a$, it is either a an arc + segment + saturated spiral (if $\mathtt a \in (0,\sin(\bar \alpha))$), a segment + saturated spiral (if $\mathtt a \in [\sin(\bar \alpha),1)$) or just the saturated spiral when $\mathtt a = 1$. We will use the denomination
\begin{description}
\item[Arc case] the fastest saturated spiral starts with a non-trivial arc of the level set, i.e. this arc is not just a point;
\item[Segment case] The fastest saturated spiral starts with a segment, possibly made of just a point, i.e. the initial point is saturated and the spiral coincides with the saturated one.
\end{description}
These various cases will be also encountered when considering as a starting point the spiral arc $\zeta([0,s_0])$, so that they are worth studying in this setting where the functions are explicit.

The study of the behavior of this spiral gives the following

\begin{theorem}[Theorem \ref{thm:case:a}]
\label{thm:case:a_intro}
For any value of $\mathtt a \in (0,1]$, the fastest saturated spiral $S_{\mathtt a}$ with initial burning region $B_\mathtt a(0)$ does not confine the fire: more precisely, the spiral diverges exponentially as $r(\phi) \sim \phi e^{\bar c \phi}$,
being $\bar c = \ln(\frac{2\pi+\bar \alpha}{\sin(\bar \alpha)})$ the only real eigenvalue for the characteristic equation of the RDE for saturated spirals.
\end{theorem}

The proof is also interesting, because it is the prototype of the line of proof of many of the subsequent statements. The explicit form of the spiral barrier in the angle-ray coordinates can be given explicitly using the kernel \gls{Gkernel}: in the segment case, where the explicit forms are simpler, the solution takes the form (see Equation \ref{Equa:expl_sol_sat})
\begin{equation*}
r(\phi) = \frac{\sin(\bar \theta_{\mathtt a})}{\sin(\bar \alpha)} G \big( \phi - \bar \phi_0(\mathtt a) \big) - G(\phi - 2\pi) - \frac{\sin(\bar \theta_{\mathtt a} - \bar \alpha)}{\sin(\bar \alpha)} G \big( \phi - \bar \phi_0(\mathtt a) - 2\pi - \bar \alpha \big),
\end{equation*}
with
\begin{equation*}
\bar \phi_0 = \bar \theta_{\mathtt a} - \bar \alpha, \quad \bar \theta_a = \arccos(\mathtt a) + \bar \alpha.
\end{equation*}
By studying the derivative w.r.t. $\mathtt a$ (Step 2 of the proof of Theorem \ref{thm:case:a}, pag. \pageref{Page:step_2_case_study_segment}), one proves that $r(\phi) = r(\phi,\mathtt a)$ is monotonically increasing w.r.t. $\mathtt a$ increases, and then that the worst case happens for $\mathtt a = \sin(\bar \alpha)$, which means $\bar \theta_a = \frac{\pi}{2}$, $\bar \phi_0 = \frac{\pi}{2} - \bar \alpha$. Then one plots numerically the rescaled function
$$
\rho(\tau) = r(\bar \phi_0 + (2\pi + \bar \alpha) \tau) e^{-\bar c (2\pi + \bar \alpha) \tau}
$$
which is explicitly given by
\begin{equation}
\label{Equa:base_sol_intro}
\rho(\tau) = \frac{1}{\sin(\bar \alpha)} g(\tau) - e^{-\bar c(2\pi + \bar \alpha - \frac{\pi}{2})} g \bigg( \tau - \frac{\frac{\pi}{2}}{2\pi + \bar \alpha} \bigg) - \cot(\bar \alpha) e^{-\bar c (2\pi + \bar \alpha)} g(\tau - 1).
\end{equation}

\begin{figure}
\resizebox{.75\textwidth}{!}{\includegraphics{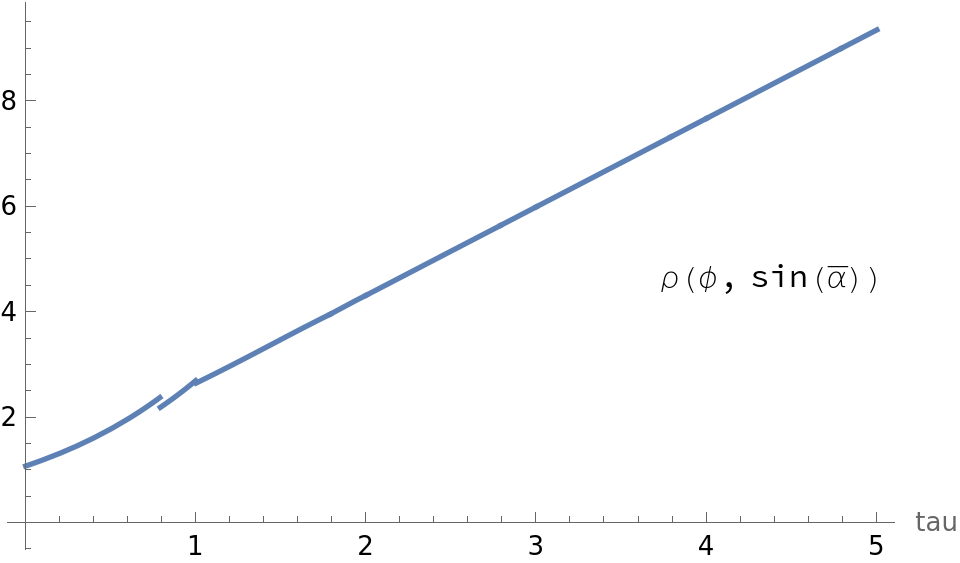}}
\caption{Plot of the first 5 rounds of the function \eqref{Equa:base_sol_intro}.}
\label{Fig:base_sol_intro}
\end{figure}

A numerical plot of the function above is in Fig. \ref{Fig:base_sol_intro}: one observes that the function is positive, and its derivative can be numerically computed to be strictly positive for $\tau \in [4,5]$ (which corresponds to the fifth round of the spiral). Then one can apply Lemma \ref{lem:key_intro} to deduce that $r(\phi)$ is positive, and actually exponentially increasing as $\phi \to \infty$. The asymptotic behavior is a consequence of the asymptotic expansion of the kernel $G(\tau)$ as $\tau \to \infty$.

The last result in this section is an estimate on the ratio between $r(\phi)$ and the length $L(\phi)$ of the spiral:

\begin{proposition}[Proposition \ref{Prop:length_compar}]
\label{Prop:length_compar_intro}
It holds for $\phi \geq \tan(\bar \alpha) + \frac{\pi}{2} - \bar \alpha$
\begin{equation*}
r(\phi) - 2.08 L(\phi - 2\pi - \bar \alpha) \geq 0.67 e^{\bar c(\phi - \tan(\bar \alpha) - \frac{\pi}{2} + \bar \alpha)}.
\end{equation*}
\end{proposition}

Besides having an interest on its own, this estimate is fundamental to show that the optimal spiral $\check r(\phi;\phi_0)$ may have an arc only for angles $\phi$ close to $\phi_0$: this means that there is a regularizing effect in the evolution of $\check r$, because it can be proved that one can choose the control parameter $\beta$ in order to have an arbitrarily number of arcs at arbitrarily large angles, see Remark \ref{Rem:inital_tent}. The constant $2.08$ is almost optimal, see Remark \ref{Rem:asymtp_consta_for_comp}.

\subsubsection{Section \ref{S:family}}
\label{Sss:family_intro}

This section is the core of the paper: we construct a differentiable path from the fastest saturated spiral of Section \ref{S:case:study} (Section \ref{Sss:case_study_intro}) to any admissible spiral. The construction is based on the results of Theorem \ref{Cor:curve_cal_R_sat_spiral_intro} which gives an explicit construction for the first round starting from $\zeta([0,s_0])$, and this construction is then extended to every angle by prolonging as a saturated spiral. We denote this barrier as $\tilde r(\phi) = \tilde r(\phi;s_0)$.

As we said the fastest saturated spiral $\tilde r$ is not the optimal spiral after the first round starting in $\phi_0$, and indeed the derivative w.r.t. $s_0$ will be a function with some negative regions. There re however two main advantages in studying this one-parameter family of spirals.
\begin{enumerate}
\item There are only two main cases to be considered (for the optimal case we need to work on about 20 cases): we can start with a segment or with an arc+segment. It is actually possible to reduce the analysis of the derivative $\delta \tilde r = \frac{d\tilde r}{ds_0}$ to a single case (the segment is the arc when the latter reduces to a single point, Remark \ref{Rem:same_as_segment}, yielding as a consequence that the family is actually $C^1$ in a suitable topology), but since the segment case is much easier we prefer to keep it. In Fig. \ref{Fig:segment_intro_1} the geometric situation to compute the derivative $\delta \tilde r(\phi;s_0)$ is represented.

\item The optimal solution $\check r$ will be a perturbation of the spiral $\tilde r$: this perturbation will be added to the solution computed in this section, and it will be shown that the analysis is fundamentally different only in the first round, more precisely in the regions where the derivative $\delta \tilde r$ is negative. In the other region a simple adaptation of the estimates obtained in this section is sufficient to prove the positivity of the derivative $\delta \check r = \frac{d\check r}{ds_0}$.
\end{enumerate}

The main results of this section are:
\begin{itemize}
\item Propositions \ref{Prop:equa_delta_tilde_r_segm} and \ref{Prop:equa_delta_tilde_r_arc}, which compute the RDE satisfied be the derivative $\delta \tilde r(\phi)$: for example, in the segment case this derivative satisfies

\begin{proposition}(Proposition \ref{Prop:equa_delta_tilde_r_segm})
\label{Prop:equa_delta_tilde_r_segm_intro}
For $\phi > \phi_0 + \bar \theta - \bar \alpha$, the derivative 
$$
\delta \tilde r(\phi;s_0) = \lim_{\delta s_0 \searrow 0} \frac{\tilde r(\phi;s_0 + \delta s_0) - \tilde r(\phi;s_0)}{\delta s_0}
$$
satisfies the RDE on $\R$
\begin{equation*}
\frac{d}{d\phi} \delta \tilde r(\phi;s_0) = \cot(\bar \alpha) \delta \tilde r(\phi;s_0) - \frac{\delta \tilde r(\phi - 2\pi - \bar \alpha)}{\sin(\bar \alpha)} + S(\phi;s_0),
\end{equation*}
with source
\begin{align*}
S(\phi;s_0) &= \cot(\bar \alpha) \frac{1 - \cos(\theta)}{\sin(\bar \alpha)} \Diracd_{\phi_0 + \bar \theta - \bar \alpha} - \Diracd_{\phi_0 + 2\pi + \theta_0} \\
& \quad + \frac{\cos(\theta) - \cos(\bar \alpha)^2}{\sin(\bar \alpha)^2} \Diracd_{\phi_0 + 2\pi + \bar \theta} + \frac{\cos(\bar \alpha - \theta) - \cos(\bar \alpha)}{\sin(\bar \alpha)} \Diracd'_{\phi_0 + 2\pi + \bar \theta},
\end{align*}
where $\Diracd_0$ is the Dirac-delta measure in $0$, and $\Diracd_0'$ its distributional derivative.
\end{proposition}
The angle $\theta$ is the angle formed by the tangent to the spiral $\zeta$ and the initial segment of $\tilde r$ at the angle $\phi_0$. Note the similarity with the fastest saturated spiral studied in the previous section.

\item Corollaries \ref{Cor:evolv_r_after_phi_2} and \ref{Cor:evolv_r_after_phi_2:arc}, which gives the explicit solution in terms of the kernel $G$;

\item Propositions \ref{Prop:regions_pos_neg_segm} and \ref{Prop:neg_region}, where the positivity properties and asymptotic behavior of $\delta \tilde r$ is analyzed. We present here the segment case, where we use the change of variable $\phi = \phi_0 + \bar \theta - \bar \alpha + (2\pi + \bar \alpha) \tau$, $\bar \theta$ being the angle of the ray $r(\phi_0)$ and the initial segment starting at $\zeta(\phi_0)$: using the rescaled function
\begin{equation*}
\rho(\tau) = \delta \tilde r \big( \phi_0 + \bar \theta - \bar \alpha + (2\pi + \bar \alpha) \tau \big) e^{-\bar c (2 \pi + \bar \alpha) \tau}
\end{equation*}

\begin{proposition}(Proposition \ref{Prop:regions_pos_neg_segm})
\label{Prop:regions_pos_neg_segm_intro}
The function $\rho(\tau) = \rho(\tau,\theta)$ is strictly positive outside the region
\begin{equation*}
\gls{Nsegm} := \bigg\{ 1 - \frac{\theta}{2\pi + \bar \alpha} \leq \tau \leq 1, 0 \leq \theta \leq \hat \theta \bigg\} \cup \big\{ \theta = 0 \} \cup \{ \tau = 1, \theta \in [2\bar \alpha,\pi] \big\} \subset \R \times [0,\pi],
\end{equation*}
where the angle \gls{thetahat} is determined by the unique solution to
\begin{equation*}
\cot(\bar \alpha) \frac{1 - \cos(\hat \theta)}{\sin(\bar \alpha)} -  e^{- \cot(\bar \alpha) (2\pi + \bar \alpha - \hat \theta)} = 0, \qquad \hat \theta \in [0,\pi].
\end{equation*}
Moreover, as $\tau \to \infty$, the function $\rho(\tau)$ diverges like
\begin{equation*}
\lim_{\tau \to \infty} \frac{\rho(\tau) e^{-\bar c \tau}}{\tau} = 2 (S_0 + S_1 + S_2 + S_3), \quad 2 (S_0 + S_1 + S_2 + S_3) \in \theta^2 (0.08,028).
\end{equation*}
\end{proposition}

The arc case is more complicated, we refer the reader to Proposition \ref{Prop:neg_region}. In Fig. \ref{Fig:segment_intro_2} we plot numerically the solution of Proposition \ref{Prop:equa_delta_tilde_r_segm_intro} with $\theta = .3$ for the first 3 rounds. Observe that there is a negative region $\tau \in [1 - \frac{\theta}{2\pi + \bar \alpha},1]$ in the first round, then the solution becomes positive. Note that the derivative is not smooth, and this forces us to study numerically the first 5 rounds in order to apply Lemma \ref{lem:key_intro}.
\end{itemize}

\begin{figure}
\begin{subfigure}{.475\textwidth}
\resizebox{\textwidth}{!}{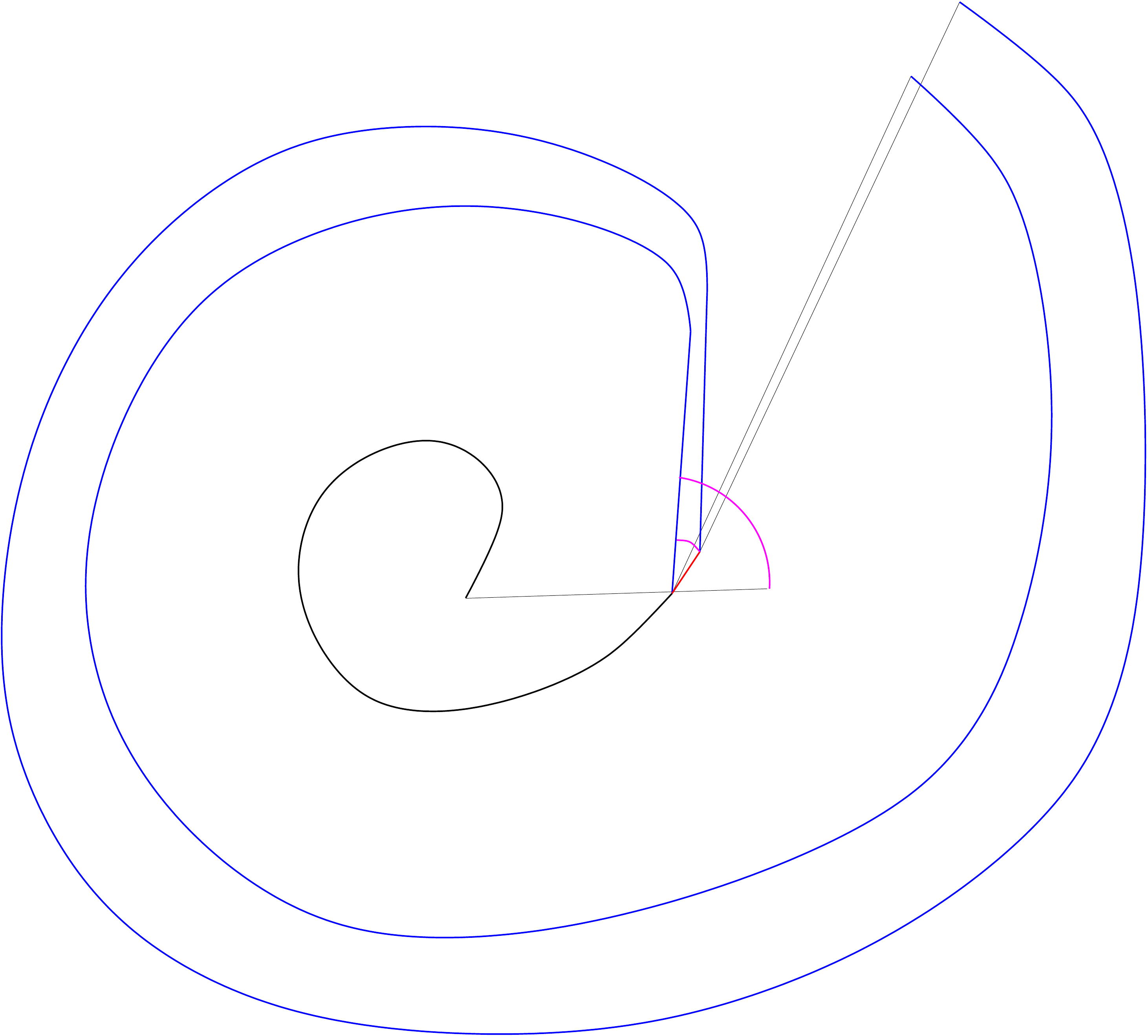}
\caption{Geometric structure of the spiral $\tilde r$ and its perturbation.}
\label{Fig:segment_intro_1} \hfill
\end{subfigure}
\begin{subfigure}{.475\textwidth}
\resizebox{\textwidth}{!}{\includegraphics{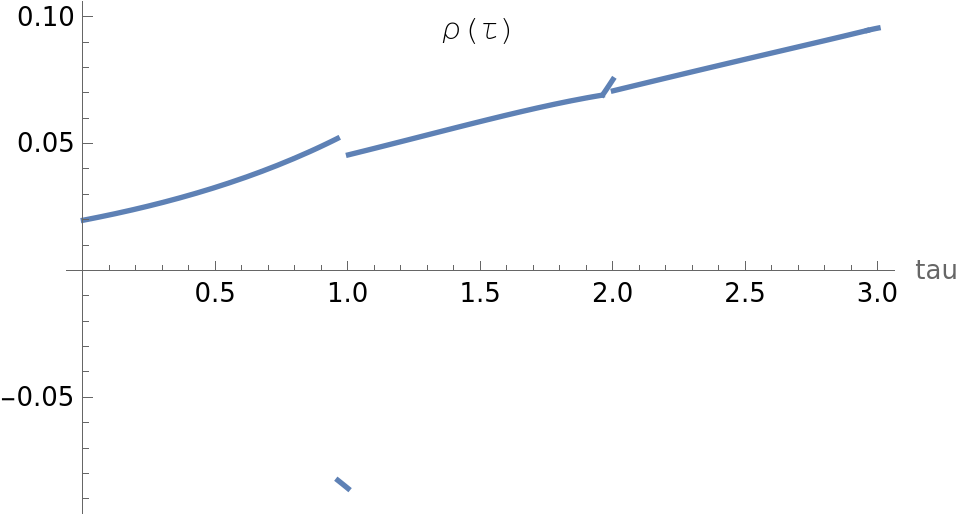}}
\caption{Numerical plot of the function $\rho(\tau)$ for the first 3 rounds. Observe the negativity region for $\tau\in N_{\text{segm}}$.}
\label{Fig:segment_intro_2}
\end{subfigure}
\caption{The segment case of Section \ref{Sss:family_intro}}
\label{Fig:segment_intro}
\end{figure}

The fact that there is a negative region at about the end of the first round for some angles $\theta$ means in particular that the fastest saturated spiral is \emph{not optimal}: indeed, at the angle $\bar \phi - 2\pi - \bar \alpha$ by moving the barrier $\zeta$ along the direction $\bar \phi$ instead of following the spiral one obtains that $r(\bar \phi)$ is decreasing (it has negative derivative by Proposition \ref{Prop:regions_pos_neg_segm_intro}. However, since this negative region is for a relatively short interval of angles, the correction to the fastest saturated spiral $\tilde r$ will be localized only at the last round. Thus we have another argument to affirm that the fastest saturated spiral is in some sense the building block for the optimal spiral $\check r$.

For the proof of the above propositions, the approach is similar to the one used for Theorem \ref{thm:case:a_intro}: the only difference is that since the structure of the derivative is more complicate, we split the first rounds into pieces in order to avoid discontinuities. After $3$ rounds the regularizing effect of the RDE has smoothen the solution, so that we can prove the monotonicity of the solution. Numerically, we plot the function
\begin{equation*}
\frac{\delta \tilde r(\phi)}{\theta (\theta + \Delta \phi)}, \quad \Delta \phi \ \text{being the opening of the initial arc and $\theta$ as above}, 
\end{equation*}
and observe the uniform positivity or monotonicity. As observed before, we require the computer only computations on explicit functions (polynomials, exponential, trigonometric functions and their combinations), we never ask to solve a differential equation. We have used the Software Mathematica for the explicit evaluation of functions.


The fact that the negativity region is limited to the first round gives that we can extend the length estimate of Proposition \ref{Prop:length_compar_intro} to a general spiral:

\begin{proposition}(Corollary \ref{Cor:bound_length_gen})
\label{Cor:bound_length_gen_intro}
Consider the spiral $\tilde r(\phi;\phi_0)$: if
$$
\phi \geq \begin{cases}
\phi_0 + \bar \theta + 2\pi & \text{segment case}, \\
\phi_0 + \Delta \phi + \frac{\pi}{2} + 2\pi & \text{arc case},
\end{cases}
$$
then
\begin{equation*}
\tilde r(\phi) - 2.08 \tilde L(\phi - 2\pi - \bar \alpha) \geq r(\phi) - 2.08 L(\phi - 2\pi - \bar \alpha) > 0.67 e^{\bar c (\phi - \tan(\bar \alpha) - \frac{\pi}{2} + \bar \alpha)},
\end{equation*}
\end{proposition}

This is a regularity result: after one round the spirals $\tilde r$ starts to be similar to the fastest saturated spiral of the previous section. It will be used to simplify the expression of the optimal spiral $\check r(\cdot;s_0)$.

\subsubsection{Section \ref{S:optimal_solut}}
\label{Sss:optimal_solut_intro}

In this section we construct the optimal solution candidate $\check r(\cdot;s_0)$, when the spiral $\zeta([0,s_0])$ is fixed. We recall that we are given an angle $\bar\phi$ and we ask if, by changing the spiral at an angle $\phi_0$, corresponding to the point $\zeta(s_0)$, the ray $r(\bar\phi)$ becomes smaller (see the optimization problem \eqref{Equa:minimum_rphi_intro}). In order to avoid additional technicalities, we explain here the construction in the simple case $\zeta([0,s_0])$ is the saturated spiral, i.e. $\mathcal{A}(\zeta(\phi))=0$ and $\beta(\phi)=\bar\alpha$ for every $\phi\in[0,\phi_0]$ (see Figure \ref{Fig:optimal_intro}). By the computations of the previous section, in the case of the family of fastest saturated spirals, we proved that there is a negativity region for the derivative $\delta\tilde r(\cdot;s_0)$. We also proved that this region is included in the first round (see Proposition \ref{Prop:regions_pos_neg_segm_intro} and Figure \ref{Fig:segment_intro}).

\medskip
Observe that
\begin{itemize}
\item If $\phi_0 + 2\pi + \bar \alpha < \bar \phi$, corresponding to the parameter $\bar\tau>1$ of the previous parametrization, then we are  outside the negativity region $N_\mathrm{segm}$ and Proposition \ref{Prop:regions_pos_neg_segm_intro} yields that any perturbation to $\zeta$ at $\phi_0$ has a positive derivative, and then the ray $r$ increases.
\item  If $\phi_0+2\pi+\bar\alpha=\bar\phi$, corresponding to the case of Figure \ref{Fig:optimal_intro}, then we are entering the negativity region: indeed $\bar\tau = 1$ is the boundary of $N_\mathrm{segm}$ for $\theta = 0$: this last condition on the angle $\theta$ means that at the angle $\phi_0$ the perturbation is tangent to the saturated spiral. Therefore, if we continue our spiral along the segment tangent to the saturated spiral at $P_0$, it follows from the computations of Section \ref{S:family} that $\theta$ is increasing ($\theta$ is the angle made by the perturbation and the segment), and then it is convenient to construct this segment up to the point $P_1$ where $\theta = \hat \theta$ (the angle constructed in Proposition \ref{Prop:regions_pos_neg_segm_intro}), i.e. we are exiting the negativity region.   In order to minimize the quantity $r(\bar\phi)$ according to the optimization problem, we need to prolong the barrier precisely along the angle $\bar\phi$. The length of barrier we construct along the direction $\bar\phi$ is prescribed by the negativity region we found from the former computations.

From this point onward we can apply Theorem \ref{Cor:curve_cal_R_sat_spiral_intro} to prescribe the remaining part of the last round: it is the fastest saturated spiral starting from the point $P_1$ such that the segment originating in $P_1$ form the critical angle $\hat \theta$ with the direction $\bar \phi$.
\end{itemize}

\begin{figure}
\resizebox{.75\textwidth}{!}{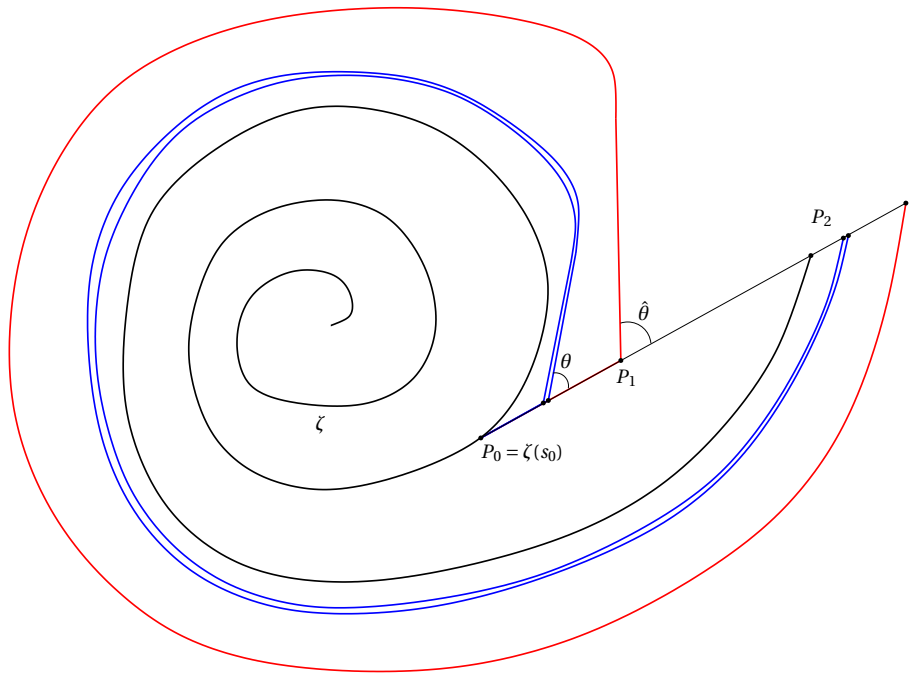}
\caption{The construction of the optimal solution for the saturated spiral.}
\label{Fig:optimal_intro}
\end{figure}

The first main result of this section is the definition of the optimal closing spiral $\check r(\phi;s_0)$ given the spiral arc $\zeta([0,s_0])$ and the final angle $\bar \phi$. This is the done in the first part of Section \ref{Ss:optimal_sol_def}, which contains also the proof that such a definition yields a unique spiral $\check r(\cdot;s_0)$ (Proposition \ref{Prop:unique_optimal_guess}{\color{red})}:

\begin{proposition}(Proposition \ref{Prop:unique_optimal_guess})
\label{Prop:unique_optimal_guess_intro}
Given $s_0 = s(\phi_0) \geq 0$ and an angle $\bar \phi \geq \phi_0$, there exists a unique optimal spiral candidate $\check r(\phi;s_0)$.
\end{proposition}

Such a statement is not elementary, indeed one has to rule out the possibility that when $P_1$ moves along the direction $\bar \phi$ it may enter again the negativity region. In the segment case this is easy, but in the arc case (i.e. when in the point $P_1$ the fastest saturated spiral starts with an arc) it is not, due to the more complicated structure of the negativity region \gls{Narc} for the arc perturbation.

\medskip 

The most common geometric situation for the perturbation is as in Fig. \ref{Fig:optimal_intro}: the last round starts with a segment in the direction $\bar \phi$ followed by the fastest saturated spiral starting again with a segment. We say it is the most common because the other case (i.e. the fastest saturated spiral starts with an arc in $P_1$) requires that the point $P_0$ is close to the last point of the given spiral arc $\zeta(s_0)$: roughly speaking, this is a regularity property of the fastest saturated spiral $\tilde r(\cdot;s_0)$, i.e. after less than 1 round after $s_0$ the spiral is already close to the saturated spiral. \\
We call \emph{tent} these two consecutive segments configuration. In Section \ref{Ss:comput_tent_new} we study the equations describing a spiral with a tent, obtaining the formulas which allow to compute how much the final tent in $\check r(\cdot;s_0)$ is going to perturb the fastest saturated spiral $\tilde r(\cdot;s_0)$ in the last round. Lemma \ref{Lem:relation_admissibility_tent} explicitly computes these formulas, and being linear relations the same formulas hold for the perturbation (Corollary \ref{Cor:tent_corol_for_pert}). \\
Next, with these relations, we can check which conditions are required in order  not to be the admissible solution: in other words, when the tent is not an admissible spiral, which in this particular case means that $s \mapsto u(\zeta(s))$ is not monotone (and then we are forced to follow an arc of the level set of $u$). It turns out that, if the the initial point $P_0$ of the tent is saturated, a necessary condition is that the length of the spiral at the previous round is much larger than the estimate of Proposition \ref{Cor:bound_length_gen_intro} (Lemma \ref{Lem:when_tent}). One thus concludes that if $P_0$ is saturated, then the optimal spiral has a tent:

\begin{proposition}(Proposition \ref{Prop:tent_admissible})
\label{Prop:tent_admissible_intro}
The tent is admissible (and then is the optimal candidate) in the saturated part of $\tilde r(\cdot;s_0)$.
\end{proposition}

We remark that in the non saturated part of $\tilde r(\cdot;s_0)$, i.e. when $\phi$ is close to $\phi_0$, the tent is in general not the admissible spiral solution, see the example of Remark \ref{Rem:inital_tent}.

The last result of this section is a formula based on a simple observation, Lemma \ref{Lem:final_value_satu}, which allows to prove that the derivative $\delta \check r(\bar \phi;s_0) = \frac{d}{ds_0} \check r(\bar \phi;s_0)$ is positive if $\delta \tilde r(\bar \phi;s_0) = \frac{d}{ds_0} \tilde r(\bar \phi;s_0)$ is. This formula can be applied only after 2 rounds, so that in the following we have to address what happens in these initial part of $\check r$ when $\bar \phi$ is close to $\phi_0$. One can rightly suspect that the non smoothness of $\delta \tilde r$ is the problem.

\bigskip

The last three sections address the homotopy argument starting from the optimal solution candidate and proving that the derivative $\delta \check r(\bar \phi;s_0)$ is increasing.

\subsubsection{Section \ref{S:optimal_sol_candidate}}
\label{Sss:optimal_sol_cand_intro}

\begin{figure}
\resizebox{.75\textwidth}{!}{\includegraphics{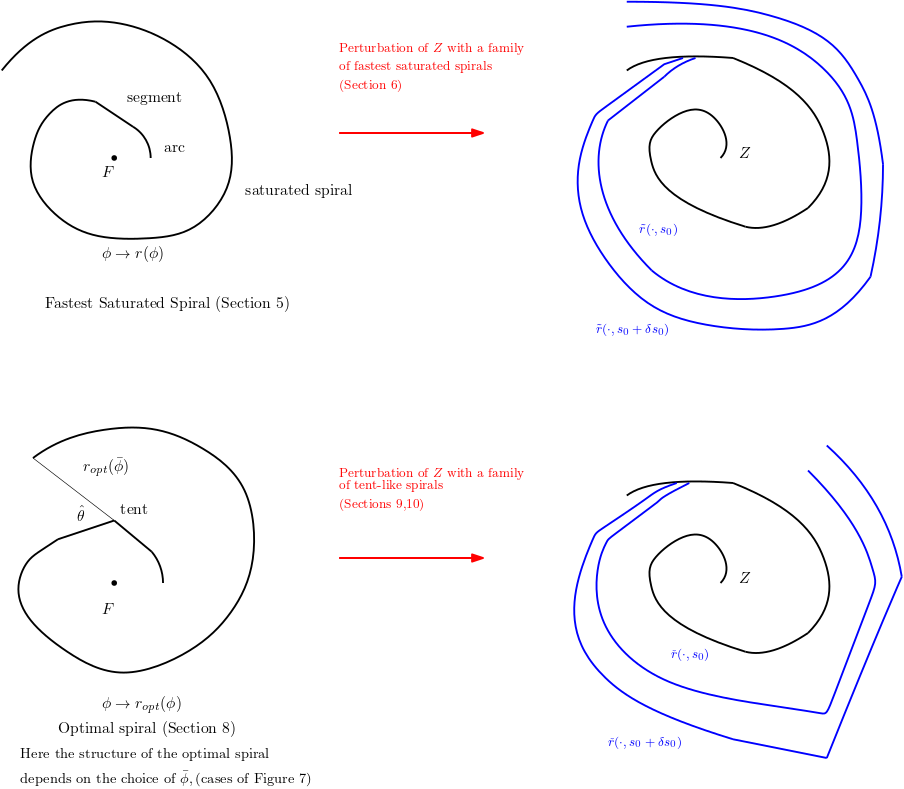}}
\caption{In this figure, we explain the general idea of our proof and in which section the results are contained.}
\label{Fig:optimal_intro:2}
\end{figure}

This section is devoted to the construction of the optimal solution candidate for any given $\bar \phi$.  The analysis carried out in this section corresponds to the one performed in Section \ref{Sss:case_study_intro} where it was proposed as a solution the fastest saturated spiral $S_a$ (Theorem \ref{thm:case:a_intro}), which is  the base element of the family $\tilde r(\phi;s_0)$ for $s_0=0$, or $(r(\phi),\phi)$ in angle-ray coordinates. We recall that the analysis of the positivity of $\delta \tilde r(\phi;s_0)$ was the content of Section \ref{S:family}. In Figure \ref{Fig:optimal_intro:2} a summary of the results and the corresponding sections where the analysis is carried out.  The solution \gls{roptim} will depend on the angle $\bar \phi$, therefore the analysis on the structure of $r_{\text{opt}}(\phi)$ will be more articulated than the fastest saturated spiral of Section \ref{Sss:case_study_intro}. We refer to Fig. \ref{Fig:opt_spiral_intro} for the description of the solution $r_\mathrm{opt}$.

\begin{figure}
\begin{subfigure}{.475\textwidth}
\resizebox{\textwidth}{!}{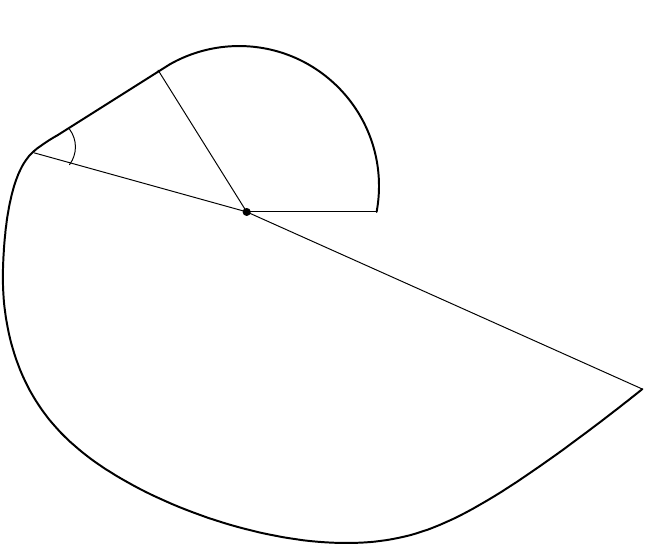}
\caption{The optimal solution for $\bar \phi \in [0,2\pi + \frac{\pi}{2} - h_1]$.}
\label{Fig:spiral_opt_first}
\end{subfigure} \hfill
\begin{subfigure}{.475\textwidth}
\resizebox{\textwidth}{!}{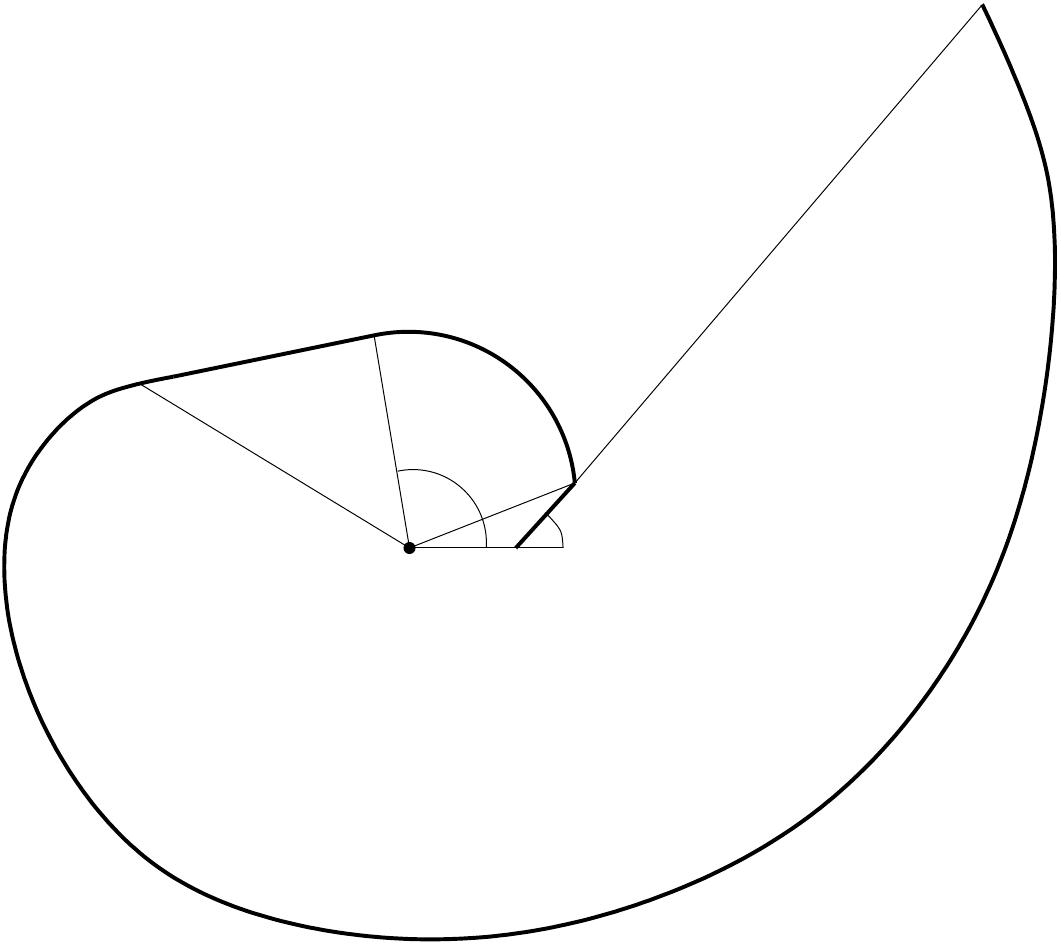}
\caption{The optimal solution for $\bar \phi \in 2\pi + \frac{\pi}{2} + [- h_1,0]$.}
\label{Fig:spiral_opt_first_2}
\end{subfigure}
\begin{subfigure}{.475\textwidth}
\resizebox{\textwidth}{!}{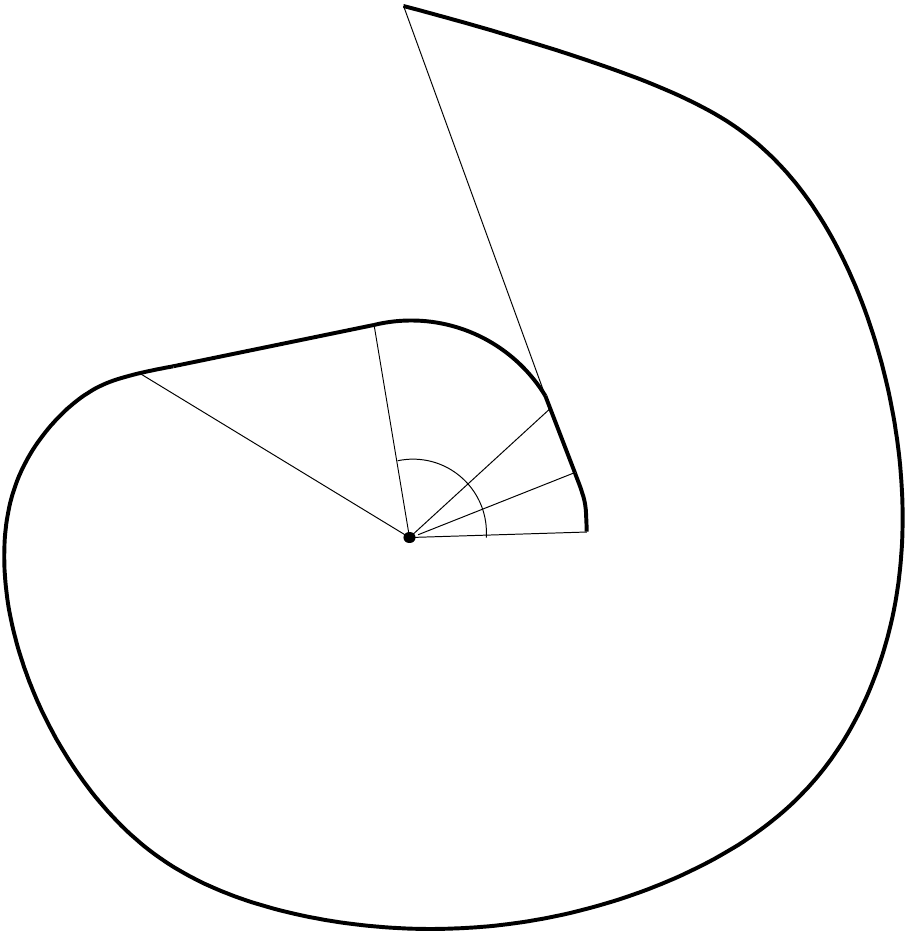}
\caption{The optimal solution for $\bar \phi \in 2\pi + \frac{\pi}{2} + [0,h_2]$.}
\label{Fig:spiral_opt_first_3}
\end{subfigure} \hfill
\begin{subfigure}{.475\textwidth}
\resizebox{\textwidth}{!}{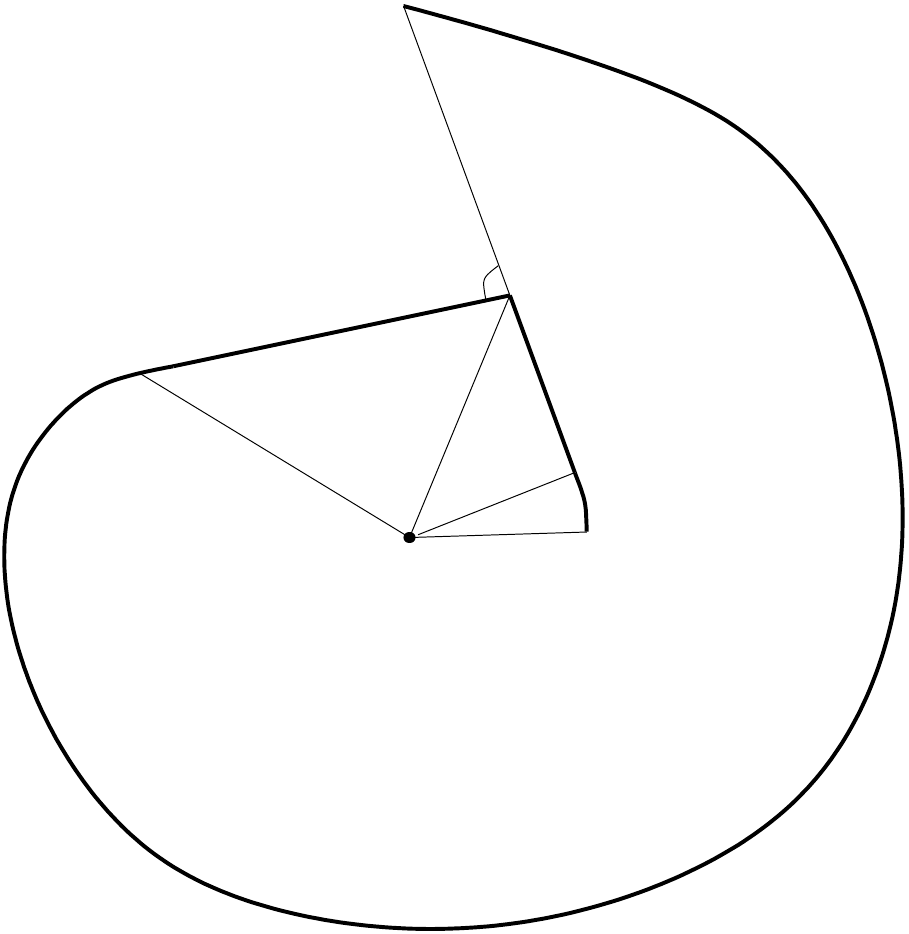}
\caption{The optimal solution for $\bar \phi \in 2\pi + \frac{\pi}{2} + [h_2,\tan(\bar \alpha)]$.}
\label{Fig:spiral_opt_first_4}
\end{subfigure}
\caption{The optimal solution depending on the final angle $\bar \phi$ for the first round.}
\label{Fig:opt_spiral_intro}
\end{figure}

\begin{itemize}
\item The first situation is the $\bar \phi$ is the first round, more precisely $\bar \phi \in [0,2\pi + \frac{\pi}{2} - h_1]$, where the constant $h_1$ is computed explicitly by the results of Section \ref{S:family}: it is $\theta$ coordinates of the unique point on the border of the negativity region \gls{Narc} for $\Delta \phi = \tan(\bar \alpha)$. In this case the optimal solution coincides with the fastest saturated spiral of Section \ref{Sss:case_study_intro}: see Fig. \ref{Fig:spiral_opt_first}. 
\item After the angle $2\pi + \frac{\pi}{2} - h_1$, the derivative $\delta \tilde r(\bar \phi;0)$ is negative: hence the optimal solution has a initial segment of length $h$ in the direction $\bar \phi$, then it is the fastest saturated solution starting with an arc of the level set, Fig. \ref{Fig:spiral_opt_first_2}.
\item For $\bar \phi \in 2\pi + \frac{\pi}{2} + [0,h_2]$, the optimal solution candidate starts with an arc of the level set $\partial B_1(0)$, then a segment in the direction $\bar \phi$ and next the last round is the fastest saturated spiral, Fig. \ref{Fig:spiral_opt_first_3}. The value $h_2$ is computed explicitly in Section \ref{Ss:optimal_cand_3}, and corresponds to the transition from the arc case to the segment case for the fastest saturated spiral of the last round, Fig. \ref{Fig:spiral_opt_first_4}.
\item Finally, for $\bar \tau > 2\pi + \tan(\bar \alpha) + \frac{\pi}{2}$, the solution is the fastest saturated spiral up to the angle $\bar \phi - 2\pi - \bar \alpha$, followed by the tent case as described in Section \ref{Sss:optimal_solut_intro}, see Fig. \ref{Fig:opt_sol_3_tent_intro}.
\end{itemize}

\begin{figure}
\resizebox{.75\textwidth}{!}{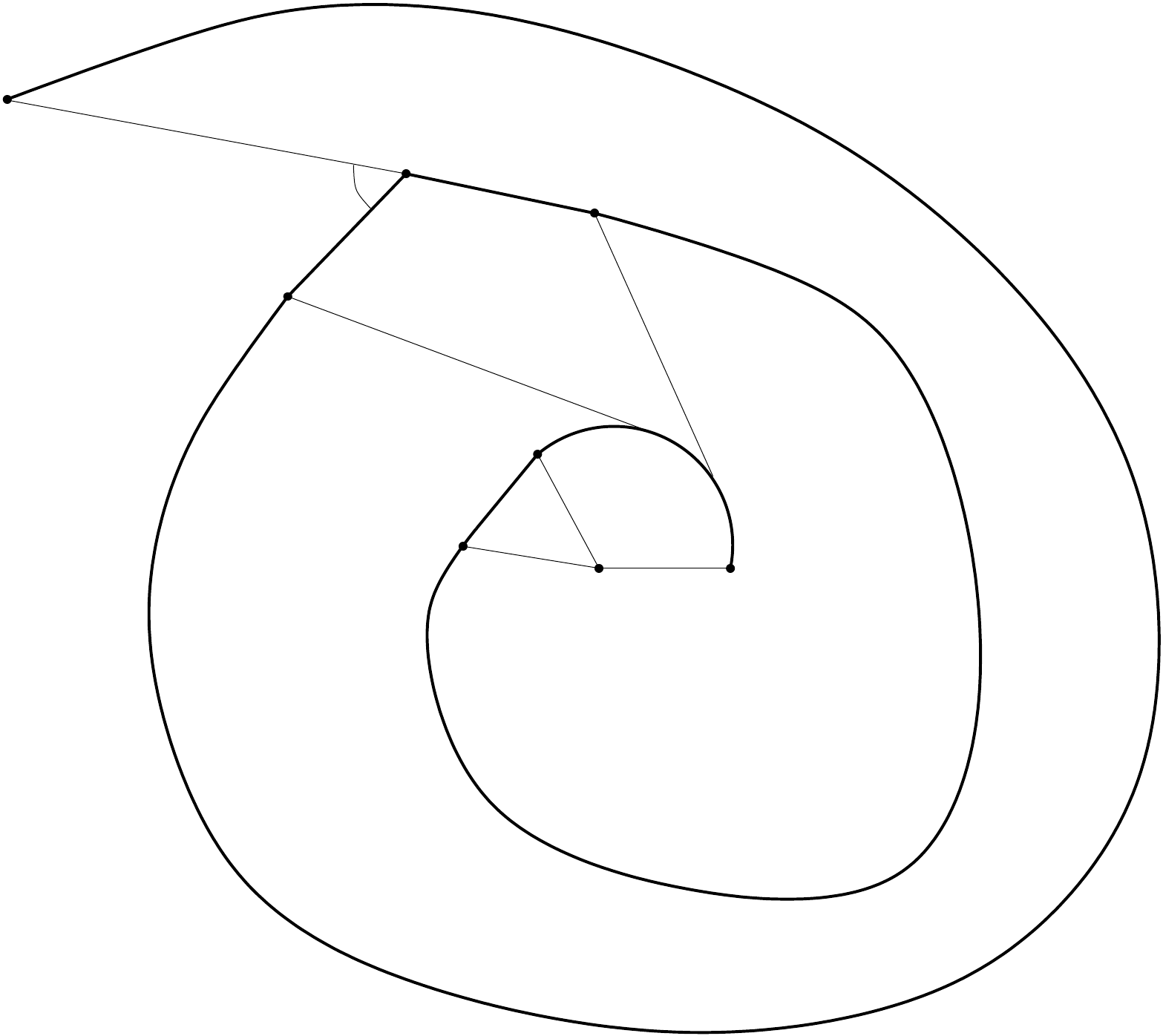}
\caption{The optimal candidate $r_\mathrm{opt}(\phi)$ for $\bar \phi \geq 2\pi + \tan(\bar \alpha) + \frac{\pi}{2}$.}
\label{Fig:opt_sol_3_tent_intro}
\end{figure}

The main result of this section is that this construction yields a unique spiral barrier which is positive:

\begin{proposition}[Propositions \ref{Prop:e_opt_large_phi} and \ref{Prop:postiv_aympt}]
\label{Prop:e_opt_large_phi_intro}
The optimal solution $r_\mathrm{opt}$ is positive for all $\bar \phi$ and for large angles can be estimated as 
\begin{equation*}
r_\mathrm{opt}(\bar \phi) = 0.976359 r(\bar \phi) + \mathcal O(1)
\end{equation*}
where $r(\bar \phi)$ is the fastest saturated spiral
\end{proposition}

A remark is natural here: the gain of optimality is about 2.25\% w.r.t. the fastest saturated spiral, so one may suspect that there are some simpler proofs of this result. Besides the fact that knowing the optimal solution is not a trivial fact, the difficulty in obtaining estimates is due to the RDE satisfied by $r(\phi)$, which makes comparisons difficult. One the other hand this method  can be a powerful tool, since it can be applied also for speeds which are not the critical one, and also to other geometric situations.

\subsubsection{Sections \ref{S:segment_case} and \ref{S:arc_1}}
\label{Sss:final_case_opt_intro}

At this point, having proved that for any given spiral barrier $\zeta$ then one can construct the one parameter family of spirals $\check r(\cdot;s_0)$ starting from $r_\mathrm{opt}$ and ending in $\zeta$, we are left to prove that $s_0 \mapsto \check r(\bar \phi;s_0)$ is increasing for all $\bar \phi$, thus effectively removing the negative regions of Proposition \ref{Prop:regions_pos_neg_segm_intro} (and Proposition \ref{Prop:neg_region} in the arc case) because of the additional perturbation to $\tilde r(\cdot;s_0)$. Since the structure of $\check r(\cdot;s_0)$ is more complicated, we have to consider several geometric situations, depending on the relative distance of $\bar \phi$ from $\phi_0$ and the structure of $\tilde r(\cdot;s_0)$.

\begin{figure}
\begin{subfigure}{.475\textwidth}
\resizebox{\textwidth}{!}{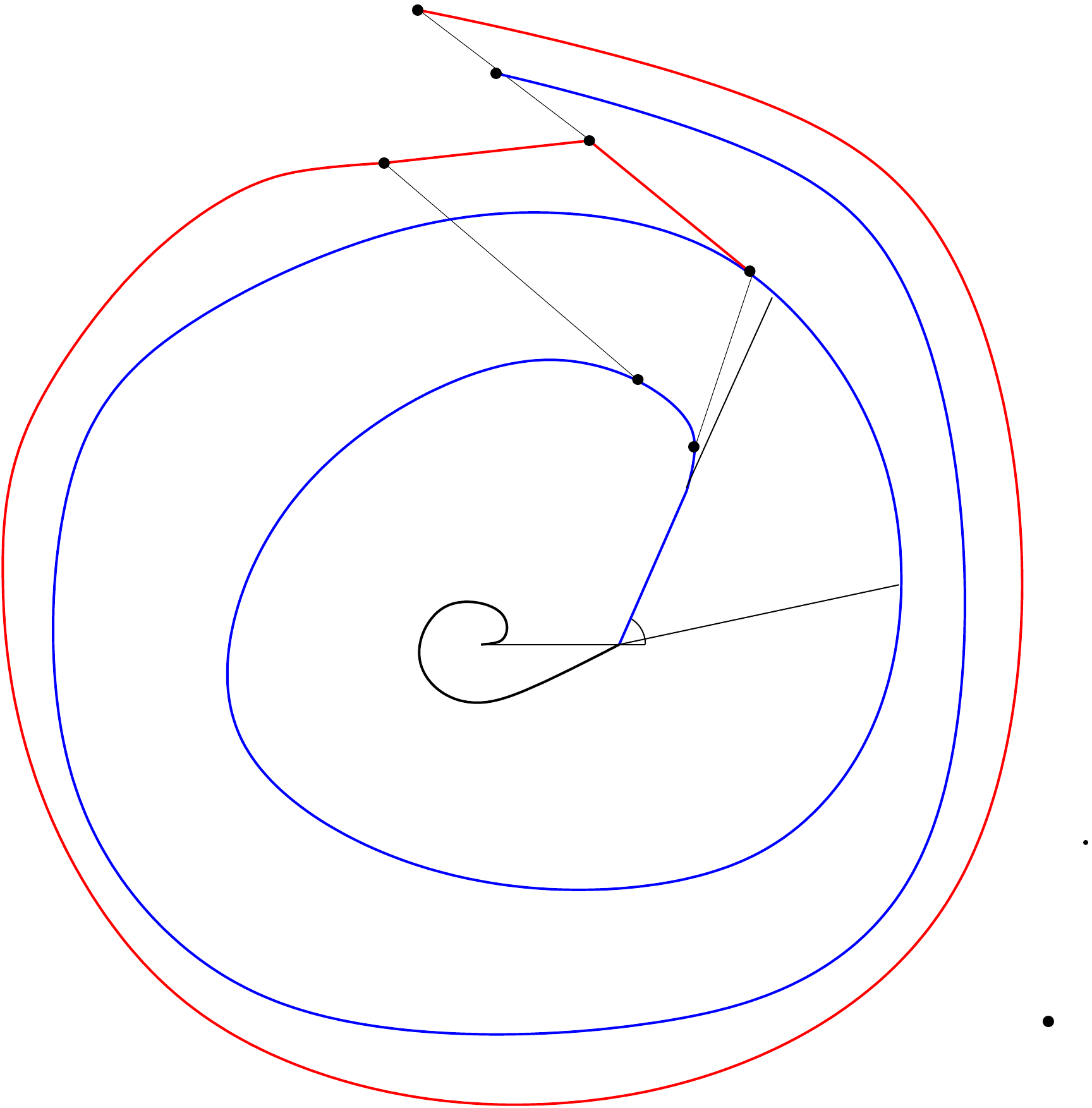}
\caption{Segment-segment case.}
\label{Fig:deltacheckr_1}
\end{subfigure} \hfill
\begin{subfigure}{.475\textwidth}
\resizebox{\textwidth}{!}{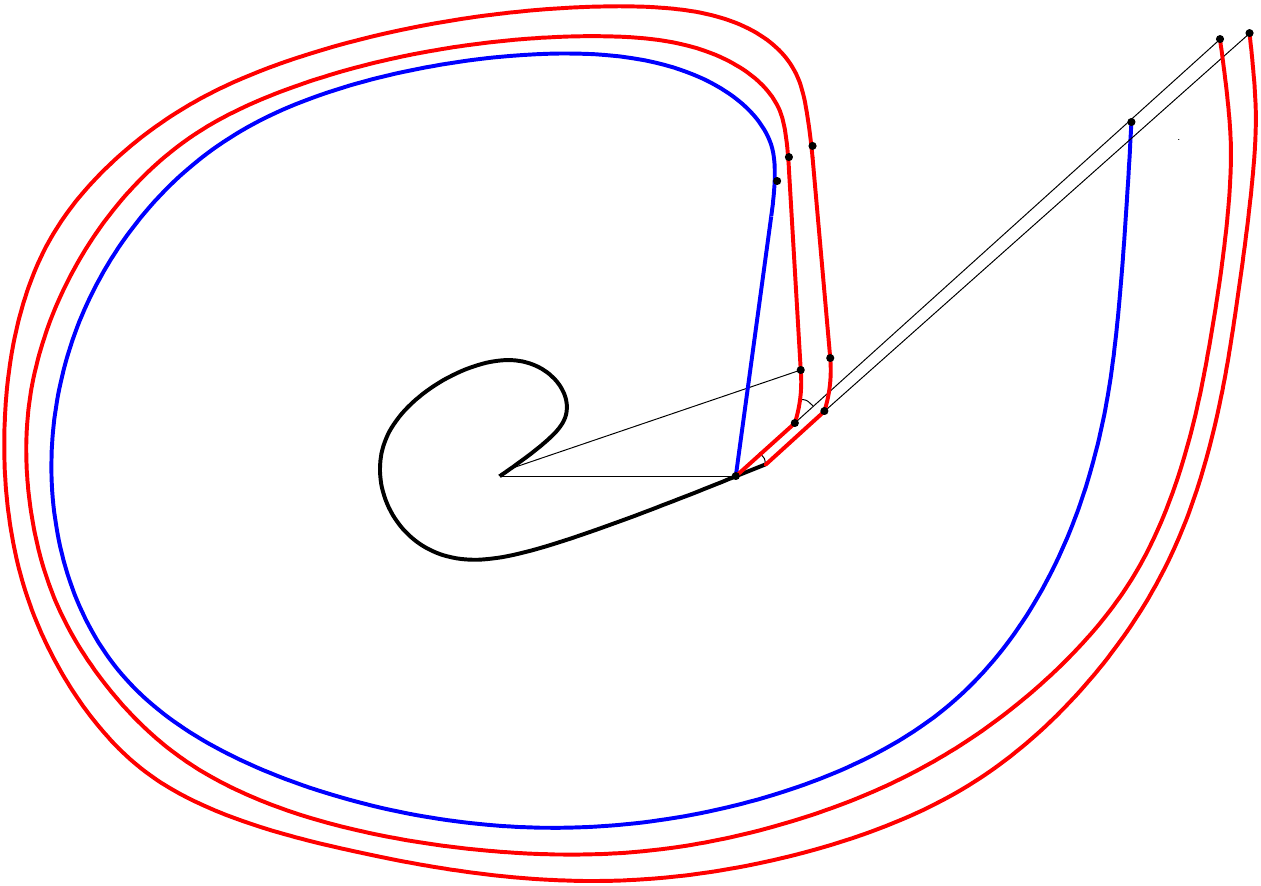}
\caption{Segment-arc case.}
\label{Fig:deltacheckr_2}
\end{subfigure}
\begin{subfigure}{.475\textwidth}
\resizebox{\textwidth}{!}{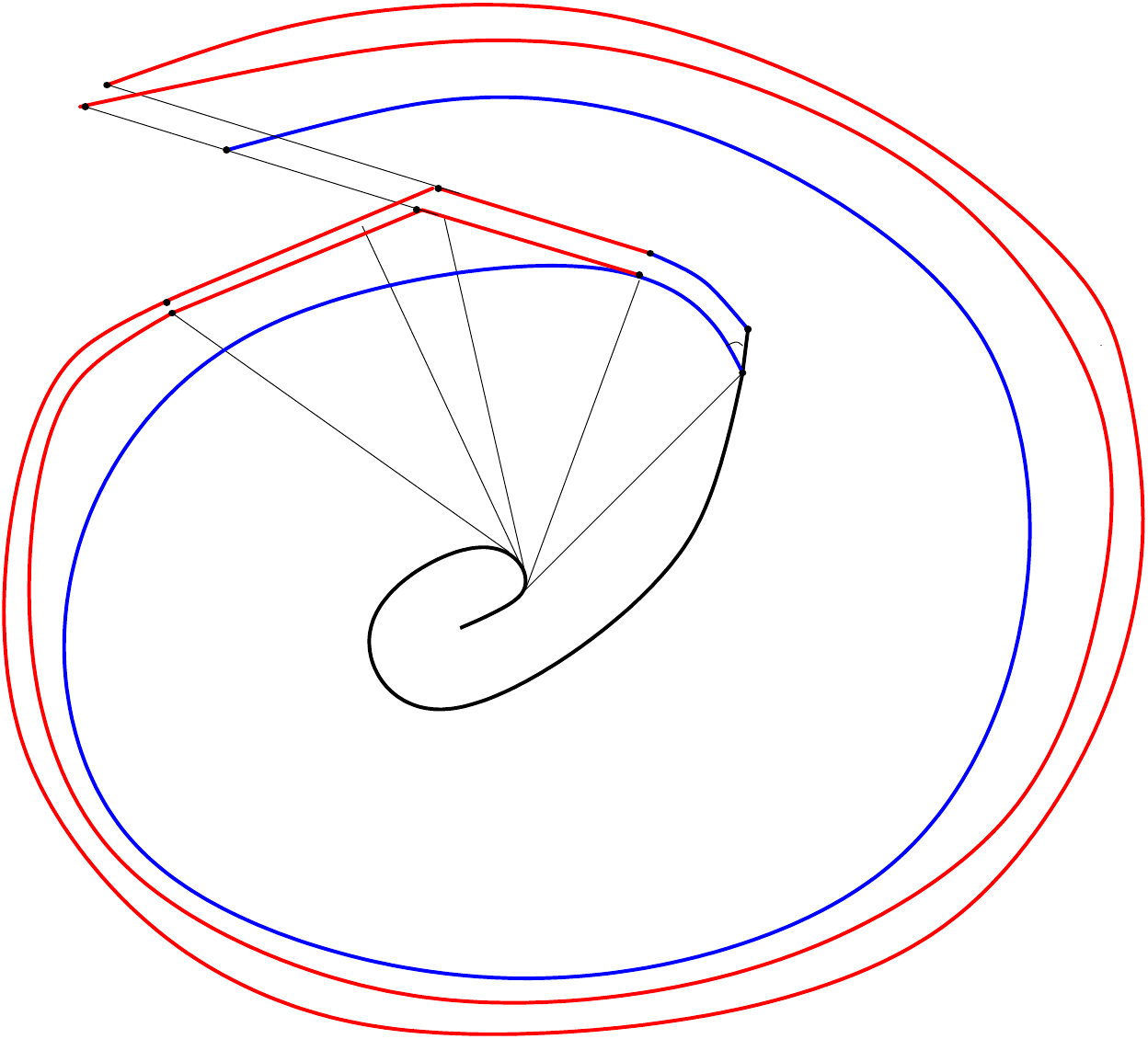}
\caption{Arc-segment case.}
\label{Fig:deltacheckr_3}
\end{subfigure} \hfill
\begin{subfigure}{.475\textwidth}
\resizebox{\textwidth}{!}{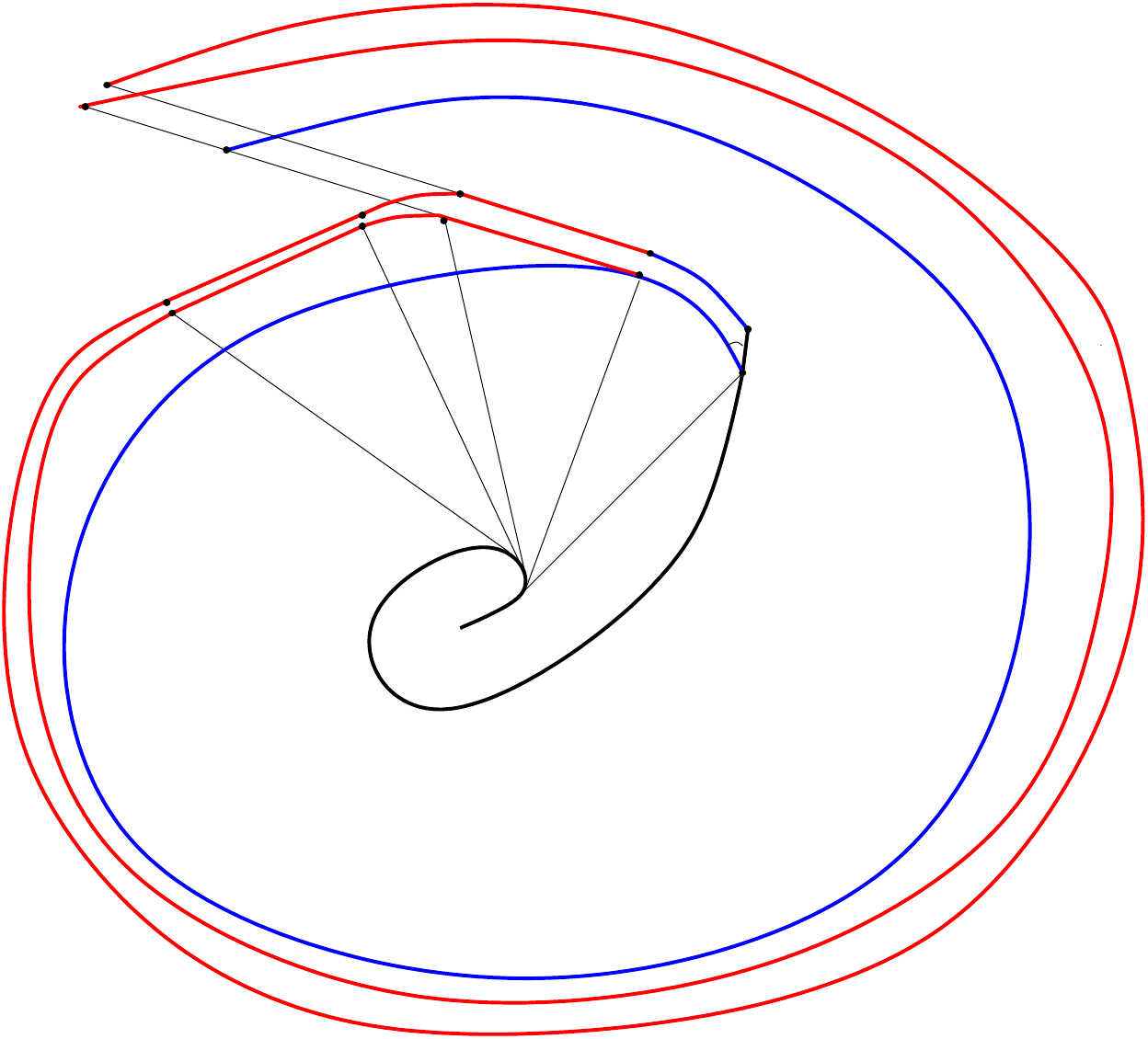}
\caption{Arc-arc case..}
\label{Fig:deltacheckr_4}
\end{subfigure}
\caption{The 4 main geometric situations for the optimal solution $\check r$.}
\label{Fig:deltacheckr}
\end{figure}

There are 4 main cases, each one with several subcases:
\begin{enumerate}
\item $\tilde r(\cdot;s_0)$ starts with a segment and its perturbation $\check r(\cdot;s_0)$ has an additional tent, Fig. \ref{Fig:deltacheckr_1}: we call it the \emph{segment-segment case};
\item $\tilde r(\cdot;s_0)$ starts with a segment and its perturbation $\check r(\cdot;s_0)$ end with an arc, Fig. \ref{Fig:deltacheckr_2}: we call it the \emph{segment-arc case};
\item $\tilde r(\cdot;s_0)$ starts with an arc and its perturbation $\check r(\cdot;s_0)$ has an additional tent, Fig. \ref{Fig:deltacheckr_3}: we call it the \emph{arc-segment case};
\item $\tilde r(\cdot;s_0)$ starts with an arc and its perturbation $\check r(\cdot;s_0)$ ends with an arc, Fig. \ref{Fig:deltacheckr_4}: we call it the \emph{arc-arc case}.
\end{enumerate}

The almost 20 subscases are due to the relative position of the final segment-arc and the initial segment-arc: referring for example to Fig. \ref{Fig:deltacheckr_1}, the tent may start inside the cone of opening $\theta$, instead that afterwards as depicted. Since this region is exactly when $\tilde r$ is discontinuous (it has also a Dirac-delta!), we did not find a general way of reducing these cases to a general situation, and we analyzed them one by one.

It is important to notice that, because of Proposition \ref{Prop:tent_admissible_intro}, the segment-segment or arc-segment case are the only possibilities after some small initial angle of rotation, and this fact greatly simplifies the analysis: for example, we do not need to study the arc-arc case for large angles.

The main results are a series of lemmata, which state that the derivative $\delta \check r(\bar \phi;s_0)$ is positive: we collect their statements into the following:

\begin{proposition}
\label{Prop:deltacheckr_intro}
For a.e. every choice of $0 \leq \phi_0 < \bar \phi$ it holds $\delta \check r(\bar \phi;\phi_0) \geq 0$, with the equality only if $\dot \zeta(\phi_0+)$ is tangent to the optimal spiral $\check \zeta(\phi_0+;\phi_0)$.
\end{proposition}

This gives the proof of Theorem \ref{Theo:main_extended}, and concludes the paper.

\subsubsection{Appendix \ref{S:num_code}}
\label{Sss:num_code_intro}

The Mathematica numerical code is collected in Appendix \ref{S:num_code}.

\section{Notation and definition of admissible spirals}
\label{s:admissible:barriers}

In this section we give the definition of admissible spiral barriers (Definition \ref{def:adm:spirals}), and we construct the angle-ray representation (Theorem \ref{Cor:angle_preprest}). Other preliminary results are the justification of the simplified setting $\Omega_0 = \{(0,0)\}$ and that the barrier lies outside $B_1(0)$ (Proposition \ref{Prop:limit_speed}) and the round-to-round decomposition of spirals of Proposition \ref{thm:param_1}.

\subsection{Notation}
\label{Ss:notation_general}

We will use the following notations:
\begin{itemize}
\item \newglossaryentry{Rreal}{name=\ensuremath{\R},description={real numbers}} \newglossaryentry{Ccomplex}{name=\ensuremath{\C},description={complex numbers}} the field of real/complex numbers are denoted by \gls{Rreal},\gls{Ccomplex} respectively, the imaginary unit is \gls{iimagi};

\item \newglossaryentry{Rscal}{name=\ensuremath{rR},description={scales set $\{rx,x\in R\}$}} \newglossaryentry{S1}{name=\ensuremath{\mathbb S^1},description={unit cicrle in $\R^2$}} if $R \subset \R^2$ and $r \in \R$, we write
\begin{equation*}
rR = \{ rx, x \in R\}.
\end{equation*}
The unit circle in $\R^2$ is denoted by \gls{S1}.

\item \newglossaryentry{Rez}{name=\ensuremath{\Re(z)},description={real part of the complex number $z$}}\newglossaryentry{Imz}{name=\ensuremath{\Im(z)},description={imaginary part of the complex number $z$}} We will often make use of the identification of $\R^2$ with $\C$ for shortening notation: \gls{Rez}, \gls{Imz} are the real, imaginary part of the complex number $z \in \C$; we will also use $e^{\i \phi} \in \C$ to denote the unit vector $(\cos(\phi),\sin(\phi)) \in \mathbb S^1$;

\item \newglossaryentry{Z}{name = \ensuremath{Z}, description={admissible strategy (i.e. a rectifiable subset of the plane $\R^2$ with finite length)}} \gls{Z} is an admissible strategy (i.e. a rectifiable subset of the plane $\R^2$ with finite length);

\item \newglossaryentry{RZt}{name=\ensuremath{R^Z(t)},description={burned region at time $t$ by the fire constrained by the barrier $Z$}} \gls{RZt} is the burned region at time $t$ by the fire constrained by the barrier $Z$, see Equation \eqref{eq:burned:region};

\item \newglossaryentry{R0}{name=\ensuremath{R_0},description={initial burning region}} $\gls{R0} = R^Z(0)$ is an open set such that its boundary has empty interior: during the analysis we will assume that $R_0 = B_r(0)$, $r \in (0,1)$ (the \emph{starting set}), and we will let $r \searrow 0$ to make the results independent on $r$, i.e. $R(0) = \{0\}$ (see Proposition \ref{Prop:limit_speed});

\item \newglossaryentry{sigma}{name=\ensuremath{\sigma},description={speed of construction of the barrier}} \newglossaryentry{barsigma}{name=\ensuremath{\bar \sigma},description={critical speed, Corollary \ref{Cor:eigen_angle}}} \newglossaryentry{alphabar}{name=\ensuremath{\bar \alpha},description={critical angle, corresponding to the critical speed $\bar \sigma = \frac{1}{\cos(\bar \alpha)}$}} \gls{sigma} is the construction speed of the barrier; we will indicate as \gls{barsigma} the critical speed $2.6144..$ and the corresponding critical angle $\gls{alphabar} =\arccos\left(\frac{1}{\bar\sigma}\right)$ (see Corollary \ref{Cor:eigen_angle});

\item \newglossaryentry{Gamma}{name=\ensuremath{\Gamma},description = {set of admissible curves for the fire spreading}} \newglossaryentry{gamma}{name=\ensuremath{\gamma},description={Lipschitz curves in $\R^2$}} the set \gls{Gamma} of admissible curves (admissible trajectories of the fire) is given by
\begin{equation}
\label{eq:admissible:curves}
\Gamma = \Big\lbrace \gls{gamma} \in \mathrm{Lip}([0,1],\R^2), \gamma((0,1)) \cap Z = \emptyset 
\Big\rbrace;
\end{equation}

\item \newglossaryentry{distxy}{name=\ensuremath{d(x,y)},description={distance between points by admissible curves}} \newglossaryentry{Lgamma}{name=\ensuremath{L(\gamma)},description={length of the curve $\gamma$}} the distance of two points is
\begin{equation*}
\gls{distxy} = \inf \Big\lbrace \gls{Lgamma}, \gamma \in \Gamma, \gamma(0) = x,\gamma(1)=y \Big\rbrace,
\end{equation*}
where we indicated by \gls{Lgamma} the length of the curve;

\item \newglossaryentry{vwvectors}{name=\ensuremath{\mathbf v,\mathbf w},description={generic vectors in $\R^2$}} \newglossaryentry{vperp}{name=\ensuremath{\mathbf v^\perp},description={vector rotated by $\frac{\pi}{2}$ counterclockwise}} if $\mathbf v$ is any vector in the plane, \gls{vperp} is its counterclockwise rotation of $\frac{\pi}{2}$ such that $\lbrace \textbf v,\textbf v^{\perp}\rbrace$ is oriented as the canonical base;

\item \newglossaryentry{anglevw}{name=\ensuremath{\angle( \mathbf v,\mathbf w)},description={angle between the vectors $\mathbf v,\mathbf w$}} \newglossaryentry{vcdotwscalarprod}{name=\ensuremath{\mathbf v \cdot \mathbf w},description={scalar product in $\R^2$}} \newglossaryentry{vcdotwexternalprod}{name=\ensuremath{\mathbf v \wedge \mathbf w},description={vector product in $\R^2$}} given two vectors \gls{vwvectors} of the plane, \gls{anglevw} will indicate the angle between the two vectors and \gls{vcdotwscalarprod} their scalar product, \gls{vcdotwexternalprod} their vector product;

\item \newglossaryentry{A:B}{name=\ensuremath{A:B},description={scalar product of matrices}} \newglossaryentry{votimesv}{name=\ensuremath{\mathbf v \otimes \mathbf{v'}},description={tensor product of two vectors}} the scalar product of real matrices is
\begin{equation*}
\gls{A:B} = \sum_{ij} a_{ij} b_{ij},
\end{equation*}
the tensor product of two vectors is \gls{votimesv}.

\item \newglossaryentry{Admspiral}{name=\ensuremath{\mathcal{A}_S},description={set of admissible spirals, Definition \ref{def:adm:spirals}}} \gls{Admspiral} denotes the admissible spirals, see Definition \ref{def:adm:spirals};

\item \newglossaryentry{pointPQ}{name=\ensuremath{P,Q},description={generic points in $\R^2$}} if $Z$ is a simple rectifiable Lipschitz curve and $P \in Z$, i.e. $P = \zeta(s)$ for some $s \in [0,\infty)$, we will often use the notation $\zeta(P),\dot\zeta(P),\ddot\zeta(P),\mathbf t(P),\mathbf t^{\pm}(P)$ for the corresponding $\zeta(s),\dot\zeta(s),\ddot \zeta(s)$, $\mathbf t(s),\mathbf t^{\pm}(s)$, where the latest are the (left/right) tangent vectors to $\zeta$; 

\item \newglossaryentry{PQopen}{name=\ensuremath{(P,Q)},description={open segment with endpoints $P,Q$}} \newglossaryentry{PQclos}{name=\ensuremath{[P,Q]},description={close segment with endpoints $P,Q$}} \newglossaryentry{PQsemi}{name=\ensuremath{[P,Q)},description={segment with endpoints $P,Q$ open in $Q$}} given $P,Q \in \R^2$, we write 
\begin{equation*}
\gls{PQopen} = \big\{ (1 - \ell) P + \ell Q, \ell \in (0,1) \big\}
\end{equation*}
to be the open segment with endpoints $P,Q$, while \gls{PQclos} will denote the closed one. We will also use the notation $\gls{PQsemi}=[P,Q]\setminus\lbrace Q\rbrace$ or $(P,Q] = [P,Q] \setminus \{P\}$ with obvious meaning.
\end{itemize}

\newglossaryentry{utimefunct}{name=\ensuremath{u},description={minimum time function}} The \emph{minimum time function} \gls{utimefunct} is defined as
\begin{equation}\label{eq:solution:HJ}
\gls{utimefunct}(x) = \inf \Big\{ L(\gamma) : \gamma \in \Gamma, \gamma(0) \in \gls{R0}, \gamma(1) = x \Big\} = d(x,R_0).
\end{equation}
It is the the time the fire needs to reach a point $x\in\R^2$ from the starting set $R_0$; it is also the solution to the Eikonal equation \eqref{eq:hamilton:jacobi}. In the case $R_0 = \{0\}$ then $u(x) = d(x,0)$.

\begin{definition}[Optimal ray]
\label{def:optimal:ray}
Let us fix a point $x \in \R^2$. An \emph{optimal ray} \newglossaryentry{bargammax}{name=\ensuremath{\bar{\gamma}_x},description={optimal ray for the point $x$}} \gls{bargammax} from $R_0$ to the point $x$ is a Lipschitz path $\bar\gamma : [0,1] \rightarrow \R^2$, $\gamma \in \text{Lip}([0,1];\R^2)$ with the following property: there exists a sequence $\lbrace\gamma_n \rbrace \subset \Gamma$ minimizing \eqref{eq:solution:HJ}, with $\gamma(0) \in R_0$ and $\gamma(1) = x$ such that $\gamma_n\to\bar{\gamma}_x$ uniformly. \\ 
We call \newglossaryentry{barGamma}{name=\ensuremath{\bar{\Gamma}},description={set of optimal rays}} \gls{barGamma} the \emph{set of optimal rays}.
\end{definition}

We say that, given a strategy $Z\subset\R^2$ and given the starting set $R_0\subset \R^2$, the barrier $Z$ is \emph{admissible} if
\begin{equation}
\label{Equa:admissbbd}
\H^1(Z \cap \overline{R^Z(t)}) \leq \sigma t,\quad\forall t\geq 0.
\end{equation}
Moreover, it is an \emph{admissible blocking strategy} if it is admissible and \newglossaryentry{RZinfty}{name=\ensuremath{R^Z_\infty},description={burned set for a given strategy $Z$}}
\begin{equation*}
\gls{RZinfty} = \bigcup_{t\geq 0} R^Z(t)
\end{equation*}
is bounded. We define the \emph{admissibility functional} as \newglossaryentry{Acaladmisx}{name=\ensuremath{\mathcal A(x)},description={admissibility functional of the point $x$}}
\begin{equation}
\label{eq:admiss:funct}
\gls{Acaladmisx} = u(x)-\frac{1}{\sigma} \H^1 \big( Z\cap\overline{R^Z(u(x))} \big), \quad\forall x\in Z,
\end{equation}
and we observe that it is positive iff $Z$ is admissible. We will denote by \newglossaryentry{Scal}{name=\ensuremath{\mathcal S},description={set of saturated points}}
\begin{equation}
\label{eq:sat:set}
    \gls{Scal} = \big\lbrace x\in Z: \mathcal{A}(x)=0 \big\rbrace,
\end{equation}
the \emph{saturated set of $Z$}, and $x \in \mathcal{S}$ will be called \emph{saturated points}.

\subsection{Some preliminaries}
\label{Ss:prelimnaries}

We collect here some simple observation which allows to simplify the problem and just consider a precise initial geometric situation: $R_0 = \{0\}$ and $Z \cap B_1(0) = \emptyset$.

We first observe that, if the strategy $Z$ is admissible for the initial set $R_0$, then for all $r > 0$ the strategy $\check Z = rZ$ is admissible for $\check R_0 = r R_0$: indeed, it holds
\begin{equation*}
\check R_{rt} = r R_t, 
\end{equation*}
and then
\begin{equation*}
\mathcal H^1 \big( \check Z \cap \overline{\check R^{\check Z}(rt)} \big) = r \mathcal H^1 \big( Z \cap \overline{R^{Z}(t)} \big) \leq \sigma (r t).
\end{equation*}
We deduce that to prove existence of blocking strategies it is sufficient to construct a confining admissible barriers when $R_0$ is contained in a ball.

Clearly if $\check R \subset R$, then $Z$ admissible for $R$ is also admissible for $\check R$: hence it is enough to consider \newglossaryentry{B1(0)}{name=\ensuremath{B_1(0)},description={unit ball centered in the origin}} $R = \gls{B1(0)}$.

The previous observation gives that if $Z$ is admissible for $B_1(0)$, then $Z$ is also admissible for $R_0 = \{(0,0)\}$. Since  constructing barriers inside $R_0$ does not change $u(x)$ (they should have $\H^1$-measure $0$ for admissibility), we conclude that if there are no blocking admissible barriers outside $B_1(0)$ for $R_0 = \{0\}$, then there are no blocking admissible barriers for all $R_0 \subset \R^2$ open bounded.

We can thus pass to the limit and consider $R_0 = \{(0,0)\}$ with $Z \cap B_1(0) = \emptyset$: if no blocking barriers $Z$ are admissible in this configuration, then the same holds for all $R_0$. Note that in this setting the problem is invariant for scaling and rotations of the barrier, so that we can assume that $(1,0) \in Z$.

On the other hand, if a barrier is admissibile for $R_0 = \{(0,0)\}$ with the speed $\sigma$, then it is admissible for $B_\epsilon(0)$ with speed $\frac{\sigma}{1-\epsilon}$: indeed, denoting by $R^Z(t)$ the burned region for $R_0 = \{(0,0)\}$ and $R^Z_\epsilon(t)$ the one with $R_0 = B_\epsilon(0)$, it holds
$$
R_{\epsilon}^Z(t-\epsilon) = R^Z(t)
$$
and then
\begin{equation*}
\mathcal H^1(Z \cap \overline{R^Z_\epsilon
(t-\epsilon)}) = \mathcal H^1(Z \cap R^Z(t)) \leq \sigma t \leq \sigma \frac{t}{t-\epsilon} (t - \epsilon) \leq \frac{\sigma}{1 - \epsilon} (t-\epsilon),
\end{equation*}
being $t \geq 1$ because $Z \cap B_1(0) = \emptyset$. We thus conclude that if the fire can be blocked when $R_0 = \{(0,0)\}$ with velocity $\sigma$, then it can be blocked for any $R_0$ open bounded with any speed $> \sigma$.

Since the same scaling argument shows that the set of speeds $\sigma$ for which there is a blocking barrier is an open unbounded interval $(\bar \sigma,+\infty)$, we have proved the following

\begin{proposition}
\label{Prop:limit_speed}
There exists a value $\bar \sigma$ such that:
\begin{enumerate}
\item if $\sigma < \bar \sigma$ it is not possible to block the fire $R_0 = \{(0,0)\}$, and then also for all open bounded initial $R_0$;
\item if $\sigma > \bar \sigma$ then it is possible to block the fire for all $R_0$ and then for $R_0 = \{(0,0)\}$;
\item in the critical value it is maybe possible to block the fire with $R_0 = \{(0,0)\}$, but not for any other $R_0$ open bounded.
\end{enumerate}
\end{proposition}

Therefore we will assume the following geometric setting:
\begin{description}
\item[Initial configuration] the fire starts spreading in the point $(0,0)$, the barrier contains the point $(1,0)$ and it lies outside $B_1(0)$.
\end{description}

\begin{remark}
\label{Rem:critical_value_origin}
We will show that for spiral barriers also in the critical case $R_0 = \{(0,0)\}$ the fire cannot be blocked with the critical speed $\bar \sigma$. We believe this is true in general.
\end{remark}

\subsection{Spiral barriers}
\label{Ss:spiral_barrier_intro}

Among admissible strategies, we will consider spiral-like strategies: namely admissible barriers where the effort of construction is put on a single wall that rotates on one direction around $R_0$.

We start with the definition of single wall strategy.

\begin{definition}[Single-barrier strategy]\label{def:spiral:strategy}
Let $Z=\zeta([0,S])\subset\R^2$ be a strategy, where \newglossaryentry{zeta}{name={\ensuremath{\zeta}},description={parametrization of the barrier $Z$}} $\gls{zeta} : [0,S] \to \R^2$ is a Lipschitz curve parametrized by length. We say that it is a \emph{single barrier strategy} if it satisfies:
\begin{itemize}
\item $\zeta(0) = (1,0)$ and $\zeta \rest_{[0,S)}$ is simple;

\item $s \mapsto u(\zeta(s))$ is increasing.
\end{itemize}
\end{definition}

\begin{remark}
\label{Rem:comment_single_curve}
The first condition implies that $\zeta$ does not have internal loops: we believe that these loops are not optimal, but allowing them would alter the spiral structure of the barriers considered here (see Definition \ref{def:convexity:spiral} below). Also notice that every continuum can be written as the graph of a simple connected curve, so without this assumption we would allow all connected barriers.

The second condition has a deeper meaning: indeed, it allows some independence of the strategy $\zeta((s,S])$ from $\zeta([0,s])$. To explain better this fact, assume that there is an arc $\zeta((s_1,s_2))$ such that
\begin{equation*}
u(\zeta(s_1)) = u(\zeta(s_2)) < u(\zeta(s)) \qquad \forall s \in (s_1,s_2).
\end{equation*}
Then, the strategy $\zeta(s)$, $s > s_2$ needs to consider that the fire is burning part of the previous built barrier for $t \in u(\zeta(s_1,s_2))$, and it may even be not possible to continue the curve $\zeta(s)$, $s \geq s_2$ (for example if $\sigma \leq 2$ and $u(\zeta(s))$, $s \in (s_1,s_2)$, is saturated, then no additional barriers can be added on the level set $(u \circ \zeta)^{-1}(s)$). Instead with our assumption it is always possible to continue the curve in some way, respecting the bound \eqref{Equa:admissbbd}. These considerations are in line with the idea that at every time $t$ the firefighter chooses the best strategy by evaluating the configuration at time $t$, where the only unmovable barriers are $\zeta \cap u^{-1}((0,t])$, i.e. the burned ones (if we alter these, then $R^Z(t)$ would change!).

Notice also that this definition implies that if $u(\zeta(s_1)) = u(\zeta(s_2))$, then $u(\zeta(s))$ is constant for $s \in [s_1,s_2]$, excluding thus from our analysis those strategies that are touched simultaneously in more than one point by the fire and the spiral arc $\zeta \rest_{[s_1,s_2]}$ is not contained in a level set of $u$. We also exclude curves $\zeta$ which are not simple: if our strategies blocks the fire, only the last point $\zeta(S)$ belongs to $\zeta([0,S))$.
%
\end{remark}

We give a definition of (local) convexity for strategies: since this local convexity implies that the tangent $\dot \zeta(s)$ rotates either clockwise or counterclockwise, w.l.o.g. we assume the second situation (in the first one we assume that $(\mathbf v,\mathbf v^\perp)$ is oriented clockwise).

\begin{definition}
\label{def:convexity:spiral}
We say that a single-barrier strategy $Z=\zeta([0,S])$ is \emph{locally convex} if for every $x\in Z\setminus \zeta(S)$ there exists a line 
$H = x + \lambda \mathbf{v}$, $\lambda \in \R$, constant $\epsilon>0,\delta>0$ and a function $f:[-\delta,\delta]\rightarrow [0,+\infty)$, $f(0) = 0$, convex Lipschitz such that, in the framework $\lbrace \mathbf v,\mathbf v^\perp\rbrace$ 
oriented as the canonical base and centered in $x$, \newglossaryentry{Br(x)}{name={\ensuremath{B_r(x)}},description={open ball of radius $r$ centered in $x$}}
\begin{equation*}
Z \cap \overline{B}_\epsilon(x)= \big\lbrace (z,y): y= f(z), z\in[-\delta,\delta] \big\rbrace.
\end{equation*}
and $s \mapsto z(s) = \mathbf v \cdot (\zeta(s) - x)$ is locally Lipschitz strictly increasing.
\end{definition}

From this definition it follows immediately that the function $s \rightarrow \dot\zeta(s)$ defined above is $\BV_{\loc}$: in particular there exist the left and right limits \newglossaryentry{t-t+}{name={\ensuremath{\mathbf t^-(s),\mathbf t^+(s)}},description={left, rigth tangent to $\zeta(s)$}} \gls{t-t+}:
\begin{equation*}
\mathbf {t}^-(s)=\dot \zeta(s-) = \lim_{s' \nearrow s} \dot \zeta(s'), \quad \mathbf {t}^+(s)=\dot \zeta(s+) = \lim_{s' \searrow s} \dot \zeta(s').
\end{equation*}
Let \newglossaryentry{Dsing}{name={\ensuremath{D^\mathrm{sing}}},description={singular part of the derivative of a BV function}}
\begin{equation*}
D \dot \zeta = \gls{Dsing} \dot \zeta + \ddot \zeta \mathscr L^1
\end{equation*}
be its measure derivative. By Radon-Nikodym Theorem and the Lebesgue Decomposition Theorem \cite{Royden:real} and the local convexity assumption, it follows that \newglossaryentry{Dcantor}{name={\ensuremath{D^\mathrm{cantor}}},description={Cantor part of the derivative of a BV function}}
\begin{align*}
D \dot \zeta = \mathbf w(s) |D^{\mathrm{sing}} \dot \zeta| + \frac{\dot \zeta(s)^\perp}{R(s)} \mathscr L^1 &= \sum_i (\dot \zeta(s_i+) - \dot \zeta(s_i-)) \Diracd_{s_i}(ds) + \dot \zeta(s)^\perp |\gls{Dcantor} \dot \zeta| + \frac{\dot \zeta(s)^\perp}{R(\zeta(s))} \mathscr L^1 \\
&= \sum_i (\mathbf t^+(s_i) - \mathbf t^-(s_i)) \Diracd_{s_i}(ds) + \dot \zeta(s)^\perp |\gls{Dcantor} \dot \zeta| + \frac{\dot \zeta(s)^\perp}{R(\zeta(s))} \mathscr L^1,
\end{align*}
where \newglossaryentry{Rcurva}{name=\ensuremath{R(s) = R(\zeta(s)) = R(\zeta;s)},description={curvature of the line $\zeta$ in the point $\zeta(s)$}} \gls{Rcurva} is the local radius of curvature and \newglossaryentry{Diracdeltas}{name=\ensuremath{\Diracd_s},description={Dirac's delta measure centered in $s$}} \gls{Diracdeltas} is the Dirac delta measure. Hence there exists a unit vector $\mathbf v$ in the cone
\begin{equation}\label{Equa:cone_subcone}
\mathbf v \in \R^+ \conv \big\{ \dot \zeta(s-),\dot \zeta(s+) \big\}
\end{equation}
so that
\begin{equation*}
\mathbf v \wedge \mathbf w(s) = 1,
\end{equation*}
i.e. the vectors $(\mathbf v,\frac{\mathbf w}{|\mathbf w|})$ are a possible base $(\mathbf v,\mathbf v^\perp)$ for representing $\zeta$ locally as a convex function as in Definition \ref{def:convexity:spiral}: the only case to consider is the singular part, and we only observe that the vectors
\begin{equation*}
\mathbf w(s_i) = \dot \zeta(s_i+) - \dot \zeta(s_i-), \quad \mathbf v = \frac{\dot \zeta(s_i+) + \dot \zeta(s_i-)}{|\dot \zeta(s_i+) + \dot \zeta(s_i-)|}
\end{equation*}
do the job.

\begin{definition}
\label{Def:supporting_direction}
A unit vector $\mathbf v$ is called \emph{supporting direction} in $\zeta(s)$ if $\mathbf v$ satisfies \eqref{Equa:cone_subcone}. The arc of the unit ball made of vectors satisfying \eqref{Equa:cone_subcone} will be called \emph{subdifferential of $\dot \zeta$}, and denoted by \newglossaryentry{subdiffzeta}{name={\ensuremath{\partial^- \zeta(s)}},description={subdifferential of the curve $\zeta$}} \gls{subdiffzeta}.
\end{definition}

\begin{remark}
\label{Rem:equiv_convex}
The considerations above show that an equivalent definition is that $s \mapsto \dot \zeta(s)$ is locally BV, and that (assuming that $\dot \zeta(s)$ is right continuous)
\begin{equation*}
D\dot \zeta(s) \wedge \dot \zeta(s) \geq 0.
\end{equation*}
If the above product is negative, then the rotation is clockwise.
\end{remark}

We introduce now the rotation angle of the vector $\dot \zeta(s)$ as the monotone increasing right continuous function \newglossaryentry{varphitan}{name=\ensuremath{\varphi(s)},description={rotation angle of the tangent vector $\dot \xi(s) = \mathbf t(s) = e^{i\varphi(s)}$}} $s \mapsto \gls{varphitan}$ such that \newglossaryentry{iimagi}{name=\ensuremath{\i},description={imaginary unit}}
\begin{equation*}
\mathbf{t}^+(s) = e^{\gls{iimagi} \varphi(s)}.
\end{equation*}
This function is uniquely defined once we fix the initial angle, and in this paper we set
\begin{equation*}
\varphi(0) = 2\pi + \angle (\textbf{t}^+(0),\textbf{e}_1).
\end{equation*}
The interval $[0,2\pi)$ will be used as the initial parametrization for the first rotation, see Proposition \ref{thm:param_1}.

We make the following additional assumption:
\begin{equation}
\label{eq:assumption:spiral}
0 \leq \angle (\textbf{t}^+(0),\textbf{e}_1) \leq \frac{\pi}{2},
\end{equation}
where $\textbf{e}_1$ is the horizontal vector of the canonical base. This assumption is used to parametrize the spiral in polar coordinate for a first interval. By convention we assume that
\begin{equation*}
\mathbf t^-(0) = \mathbf e_1 = (1,0).
\end{equation*}

\begin{remark}
\label{Rem:another_def}
It is fairly easy to see that another equivalent definition to Definition \ref{def:convexity:spiral} is that there exists such a monotone increasing function $s \mapsto \varphi(s)$ such that $\varphi(s+) - \varphi(s-) < \pi$ and $\dot \zeta(s) = e^{i\varphi(s)}$: in this case
\begin{equation*}
\partial^- \zeta = \big\{ e^{\i \phi}, \phi \in [\varphi(s-),\varphi(s+)] \big\}.
\end{equation*}
\end{remark}

We can now give the definition of the class of barriers which we consider in this paper.

\begin{definition}
\label{def:adm:spirals}
A single-barrier (not necessarily blocking) convex strategy satisfying Assumption \eqref{eq:assumption:spiral} will be called a \emph{spiral}. The set of \emph{admissible spirals} (i.e. is the set of admissible single-barriers (not necessarily blocking) convex strategies satisfying assumption \eqref{eq:assumption:spiral}) is denoted by \gls{Admspiral}.
\end{definition}

\subsubsection{Angle-ray parametrization of spirals}
\label{Sss:alternative_angles}

From now on we will fix an admissible spiral $Z \in \mathcal{A}_S$. We will extend the alternative parametrization introduced first in \cite{firefighter} for the self-similar saturated spiral (see Section \ref{Sss:equation_satur} for definitions) to a general spiral strategy $Z$. We will call such a parametrization \emph{angle-ray parametrization} or \emph{angle-ray representation}.

The next proposition contains the core results of this section. In the statement the choice of the sequence of rounds is fixed (see Equation \ref{Equa:next_angle_ell+2}): a simple adaptation of the proof allows some flexibility in the choice of the sequence of angles $\bar \phi_\ell$, see Remark \ref{Rem:other_choice_rounds}.

\begin{proposition}
\label{thm:param_1}
Let $\zeta$ be a spiral strategy. Then there exist two sequences of real positive numbers \newglossaryentry{skbar}{name=\ensuremath{\bar s_k},description={partition of the spirals into rounds, Proposition \ref{thm:param_1}}} \newglossaryentry{phikbar}{name=\ensuremath{\bar \phi_k},description={rotation angles for the partition of a spiral in rounds, Proposition \ref{thm:param_1}}}
$$
0 \leq \bar s_1 < \bar s_2 < \dots < \gls{skbar} < \bar s_{k+1} \leq S, \quad 0 < \bar \phi_1 = 2\pi < \bar \phi_2 < \dots < \gls{phikbar} \leq \bar \phi,
$$
for some angle $\bar \phi$, with the following properties.

\begin{enumerate}
\item For $0 \leq s < \bar s_1$ the spiral $\zeta$ coincides with the oriented segment
\begin{equation*}
\zeta(s) = (1,0) + s \mathbf e_1,
\end{equation*}
and the value $\bar s_1$ is given by
\begin{equation*}
\bar s_1 = \inf \big\{ \zeta^{-1}(\R^2 \setminus \R \mathbf e_1) \big\} = \max \big\{ s : \zeta([0,s]) \subset \R \mathbf e_1 \big\}.
\end{equation*}

\item For $0 \leq \phi < \bar \phi_1 = 2\pi$ the spiral can be represented in polar coordinates as
\begin{equation}
\label{Equa:maps_1}
\zeta(s^+_0(\phi)) = r_0(\phi) e^{\i \phi}, 
\end{equation}
with $r_0(\phi)$ a Lipschitz function in $[0,2\pi]$. Defining the value
\begin{equation*}
\bar s_2 = \lim_{\phi \nearrow 2\pi} s^+_0(\phi) = \min \big\{ s > \bar s_1: \zeta(s) \in \R^+ \mathbf e_1 \big\} 
\end{equation*}
as the value $s$ of the next intersection with the line $\zeta(\bar s_1) + \mathbf e_1 \R^+ = \zeta(\bar s_1) + \mathbf t^-(\bar s_1) \R^+$, the map \eqref{Equa:maps_1} defines a continuous bijection
$$
[0,2\pi] \ni \phi \mapsto s^+_0(\phi) \in [\bar s_1,\bar s_2],
$$
Lipschitz in $(\epsilon,2\pi]$ for every $\epsilon > 0$ and with Lipschitz inverse, where the final value is $\bar s_2 = s_0^+(2\pi)$. 
The optimal rays are the segments $[(0,0),\zeta(\phi)] = r_0(\phi) e^{\i \phi}$, where $\zeta(\phi)=\zeta(s_0^+(\phi))$ and the region $S_1$ bounded by the Jordan curve $\zeta([\bar s_1,\bar s_2)) \cup [\zeta(\bar s_2),\zeta(\bar s_1))$ has convex boundary apart from the point $\zeta(\bar s_1)$ and it is star-shaped w.r.t. the point $\zeta(\bar s_1)$.

\item \label{Point_3:round_dec} For $\ell = 1,\dots,k-1$ the angles $\bar\phi_\ell,\bar\phi_{\ell+1}$ satisfy \newglossaryentry{deltaphiell}{name=\ensuremath{\delta \phi_\ell},description={variation of angle between round, Proposition \ref{thm:param_1}}}
\begin{equation}
\label{Equa:next_angle_ell+2}
\bar \phi_{\ell+1} = 2\pi + \bar \phi_\ell + \gls{deltaphiell}, \quad \delta \phi_\ell = \angle (\mathbf t^-(\bar s_{\ell+1}), e^{\i \bar \phi_\ell}) 
\in \bigg( 0,\frac{\pi}{2} \bigg),
\end{equation}
and there exist three functions \newglossaryentry{rellphi}{name=\ensuremath{r_\ell(\phi)},description={length of the last segment of an optimal ray at the $\ell$-th round, round decomposition of Proposition \ref{thm:param_1}}} \newglossaryentry{s-ellphi}{name=\ensuremath{s^-_\ell(\phi)},description={initial point of the last segment of an optimal ray at the $\ell$-th round, round decomposition of Proposition \ref{thm:param_1}}} \newglossaryentry{s+ellphi}{name=\ensuremath{s^+_\ell(\phi)},description={final point of the last segment of an optimal ray at the $\ell$-th round, round decomposition of Proposition \ref{thm:param_1}}}
\begin{equation*}
[\bar \phi_{\ell},\bar \phi_{\ell+1}] \ni \phi \mapsto (\gls{rellphi},\gls{s-ellphi},\gls{s+ellphi}) \in (0,\infty) \times [\bar s_{\ell},\bar s_{\ell+1}] \times [\bar s_{\ell+1},\bar s_{\ell+2}],
\end{equation*}
with $s^-_\ell$ increasing right continuous, $r_\ell(\phi)$ right continuous one-sided Lipschitz with negative singular part of the derivative (i.e. $D^\sing r \leq 0$), and $s^+_\ell$ Lipschitz increasing, bijective with Lipschitz inverse, such that
\begin{equation*}
\zeta(s^+_\ell(\phi)) = \zeta(s^-_\ell(\phi)) + r_\ell(\phi) e^{\i \phi}.
\end{equation*}
The final point $\bar s_{\ell+2}$ is given by
\begin{equation}
\label{Equa:final_bar_s_ell+2}
\bar s_{\ell+2} = \lim_{\phi \nearrow \bar \phi_{\ell+1}} s^+_\ell(\phi) = \min \big\{ s > \bar s_{\ell+1}: \zeta(s) \in \zeta(\bar s_{\ell+1}) + \mathbf t^-(\bar s_{\ell+1}) \R^+  \big\},
\end{equation}
Moreover the optimal ray ending in the point $\zeta(s^+_\ell(\phi))$ is the union of the initial segment $[(0,0),(1,0))$, the spiral arc $\zeta([0,s^-_\ell(\phi)))$ and the final segment $[\zeta(s^-_\ell(\phi)),\zeta(s^+_\ell(\phi))] = r_\ell(\phi) e^{\i \phi}$. Finally the boundary of the open bounded region \newglossaryentry{Sell+1}{name=\ensuremath{S_{\ell+1}},description={star shaped region for the round decomposition}} \gls{Sell+1} is the Jordan curve 
\begin{equation*}
\zeta([\bar s_{\ell+1},\bar s_{\ell+2})) \cup [\zeta(\bar s_{\ell+2}),\zeta(\bar s_{\ell+1}))
\end{equation*}
which is convex outside the point $\zeta(\bar s_{\ell+1})$ and it is star shaped w.r.t. the point $\zeta(\bar s_{\ell+1})$.

\item For $\phi \geq \bar \phi_k$ there are three functions
\begin{equation*}
[\bar \phi_k,\bar \phi] \ni \phi \mapsto (r_k(\phi),s^-_k(\phi),s^+_k(\phi)) \in [0,\infty) \times [\bar s_k,\bar s_{k+1}] \times [\bar s_{k+1},S],
\end{equation*}
with $s^-_k$ increasing right continuous, $r_k(\phi)$ right continuous one-sided Lipschitz with negative singular part of the derivative and $s^+_k$ Lipschitz increasing, bijective and with Lipschitz inverse, such that
\begin{equation*}
\zeta(s^+_k(\phi)) = \zeta(s^-_k(\phi)) + r_k(\phi) e^{\i \phi}.
\end{equation*}
Moreover $r_k(\phi) = 0$ only if $\phi = \bar \phi$ (and in this case $\zeta(S) \in \zeta([0,S))$,
and the optimal ray ending in the point $\zeta(s^+_k(\phi))$ is the union of the initial segment $[(0,0),(1,0))$, the spiral arc $\zeta([0,s^-_k(\phi)))$ and the final segment $[\zeta(s^-_k(\phi)),\zeta(s^+_k(\phi))] = r_k(\phi) e^{\i \phi}$.
\end{enumerate}
\end{proposition}

We will refer to the decomposition of $[0,S] = \cup_\ell [\bar s_\ell,\bar s_{\ell+1}), [0,\bar \phi] = \cup_\ell [\bar \phi_\ell,\bar phi_{\ell+1})$ of the above proposition as the \emph{decomposition in rounds}: indeed, Point (2) describes the first round about ${(0,0)}$, which in polar coordinates corresponds to a rotation of $2\pi$, Point (3) describes the next rotations of angles $>2\pi$ about $\zeta(\bar s_\ell)$, and the final Point (4) is the last part of the curve, which may close or remain open. For a picture of this representation, see Figure \ref{fig:decomposition:rounds}.

\begin{figure}
	\includegraphics[scale=0.45]{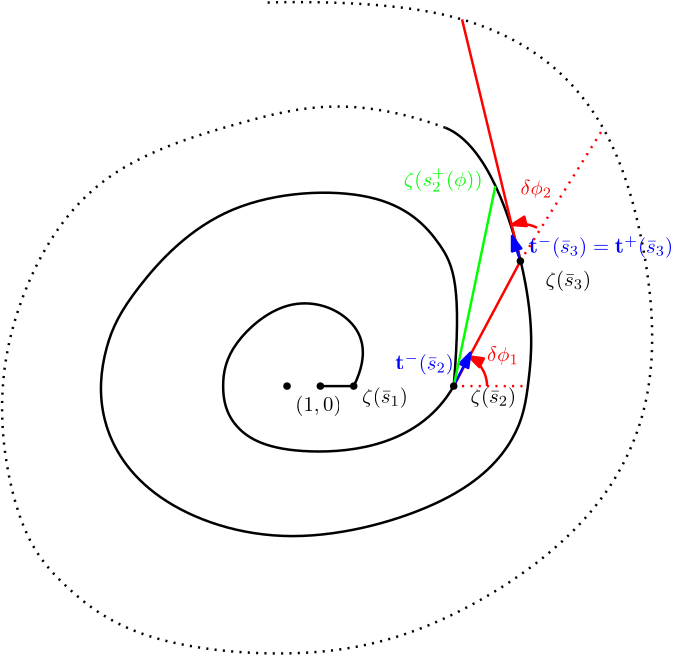}
	\caption{ Decomposition in rounds of any admissible spiral $Z$. The notation is the one of Proposition \ref{thm:param_1}. Observe that in the points of non-differentiability the angle is computed following the direction $\mathbf{t^{-}(\bar s_\ell)}$}
	\label{fig:decomposition:rounds}
	\end{figure}

\begin{proof}
The first two and the last cases can be seen as a particular case of Point (\ref{Point_3:round_dec}) in the statement, hence we will only address it.

The iterative assumptions are that
\begin{enumerate}
\item the open region $S_\ell$, bounded by the Jordan curve $\zeta([\bar s_\ell,\bar s_{\ell+1})) \cup [\zeta(\bar s_{\ell+1}),\zeta(\bar s_\ell))$, is a star shaped region about $\zeta(\bar s_\ell)$, with convex boundary outside the point $\zeta(\bar s_\ell)$;
\item $\zeta(\bar s_{\ell+1}) \in \zeta(\bar s_\ell) + \mathbf t^-(\bar s_\ell) \R^+$.
\end{enumerate}
These assumptions are clearly satisfied for $\ell = 0$.

\medskip

\noindent{\it Step 1.} We first consider the barrier
\begin{equation*}
\hat \zeta_\ell(s) = \begin{cases}
\zeta(s) & s \in [\bar s_\ell,\bar s_{\ell+1}), \\
\zeta(\bar s_{\ell+1}) + \mathbf t^-(\bar s_{\ell+1}) (s - \bar s_{\ell+1}) & s \geq \bar s_{\ell+1},
\end{cases}
\end{equation*}
obtained prolonging the tangent $\mathbf t^-(\bar s_{\ell+1})$ indefinitely. A consequence of the analysis below is that if the half-line intersects the curve $\zeta([\bar s_\ell,\bar s_{\ell+1}))$, then we are in the situation of Point (4) of the statement, so it is not the case considered here: it is nevertheless easy to adapt the next steps to this last case.

We consider next the minimum time function for the new barrier $\hat Z_\ell = \hat \zeta([\bar s_\ell,+\infty))$ (see Fig. \ref{Fig:S_ell}):
\begin{equation}
\label{Equa:u_ell_constr}
\R^2 \setminus S_\ell \ni x \mapsto u(x) = u(\zeta(\bar s_\ell)) + \inf \Big\{ L(\gamma), \gamma(0) = \zeta(\bar s_\ell), \gamma(1) = x \ \text{and} \ \gamma((0,1)) \cap \hat Z_\ell = \emptyset \Big\}.
\end{equation}
In the above representation we have observed that every optimal ray has to cross the segment $[\zeta(\bar s_\ell),\zeta(\bar s_{\ell+1})]$, \eqref{Equa:u_ell_constr} gives the correct time function on the segment $[\zeta(\bar s_\ell),\zeta(\bar s_{\ell+1})]$, and an optimal ray is a segment when not touching the barrier (because it must be of minimal length): hence it must pass through $\zeta(\bar s_\ell)$.

\begin{figure}
\begin{subfigure}{.475\textwidth}
\resizebox{\textwidth}{!}{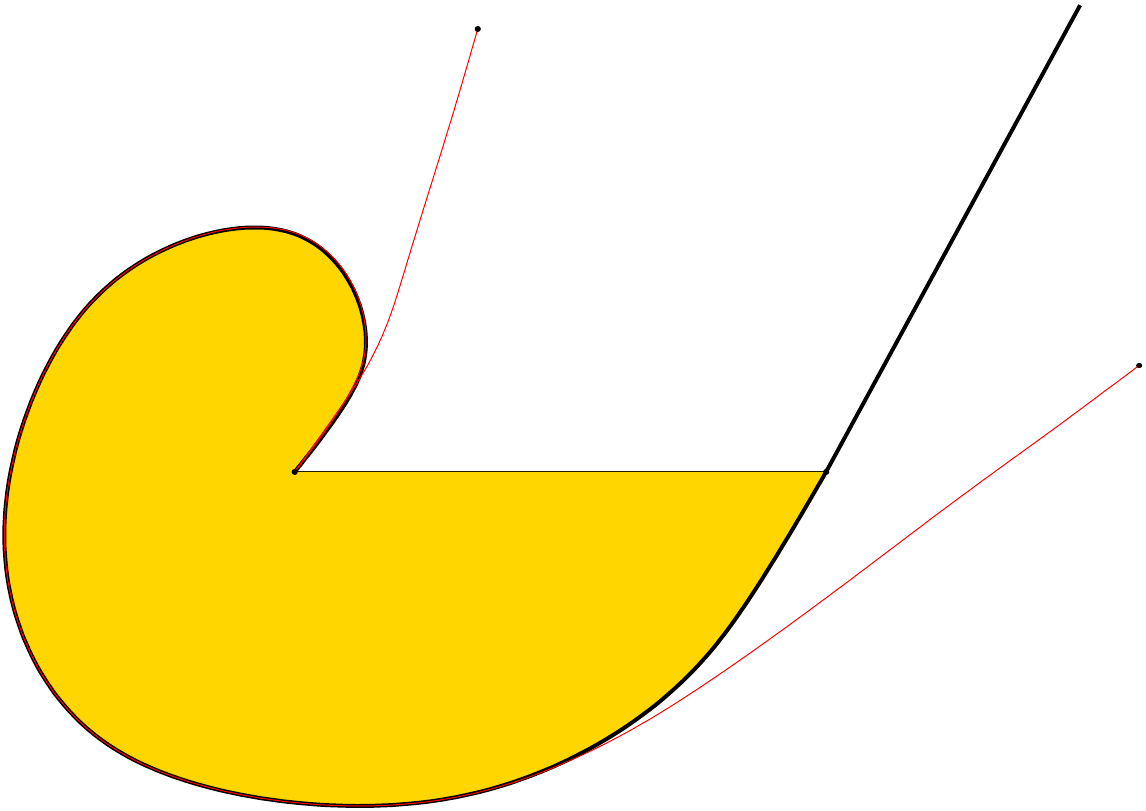}
\caption{The barrier for constructing $u(x)$, Equation \ref{Equa:u_ell_constr}.}
\label{Fig:S_ell}
\end{subfigure} \hfill
\begin{subfigure}{.475\textwidth}
\resizebox{\textwidth}{!}{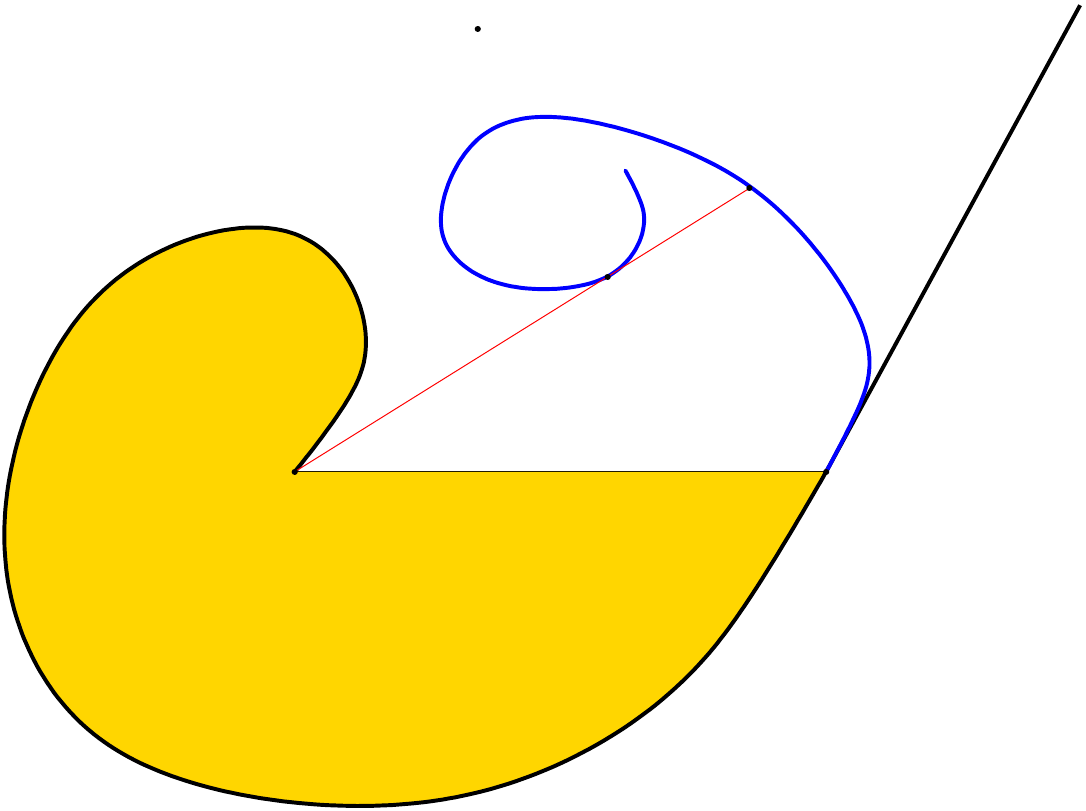}
\caption{Step 2 of the proof of Proposition \ref{thm:param_1}.}
\label{Fig:S_ell_2}
\end{subfigure}
\caption{Proof of Proposition \ref{thm:param_1}.}
\label{Fig:S_ell_0}
\end{figure}

Elementary geometric considerations show that for every point $x \in \R^2 \setminus \clos\, S_\ell$ there are at most $2$ points $\zeta(s)$ with $s \in [\bar s_\ell,\bar s_{\ell+1})$ such that
\begin{equation*}
\frac{x - \zeta(s)}{|x - \zeta(s)|} \in \partial^- \zeta(s), \quad \zeta(s) + \frac{x - \zeta(s)}{|x - \zeta(s)|} \epsilon \notin \zeta([\bar s_\ell,\bar s_{\ell+1})) \ \text{for $\epsilon \ll 1$},
\end{equation*}
and in case there are effectively two points, the only one for which the segment $(\zeta(s),x)$ is not crossing the half line $\zeta(\bar s_{\ell+1}) + \mathbf t^-(\bar s_{\ell+1}) \R^+$ is the larger one: we denote this value by the function
$$
\R^2 \setminus \big( S_\ell \cup \big\{ \zeta(\bar s_{\ell+1}) + \mathbf t^-(\bar s_{\ell+1}) \R^+ \big\} \big) \ni x \mapsto \upsilon(x).
$$
Since the optimal ray must be in the supporting direction when touching the barrier, we deduce that the structure of the optimal ray is given by
\begin{equation}
\label{Equa:struct_optimal_ray_ell}
\gamma_x = \zeta([\bar s_\ell,\upsilon(x))) \cup [\zeta(\upsilon(x)),x],
\end{equation}
where we used the convexity of $\zeta([\bar s_\ell,\bar s_{\ell+1}))$ to deduce that if the optimal ray touches $Z$ in two different points $\zeta(s_1),\zeta(s_2)$, it coincides with $\zeta((s_1,s_2))$ in between.

\medskip

\noindent{\it Step 2.} In this step we prove that the structure of the optimal rays \eqref{Equa:struct_optimal_ray_ell} remains unchanged for the points $\zeta(s)$, $\bar s_{\ell+1} \leq s \leq \bar s_{\ell+2}$, where the last point will be defined later.
Let $\bar s > \bar s_{\ell+1}$ be the first value such that
$$
\text{there is} \ \zeta(s') \in \big( \zeta(\upsilon(\zeta(\bar s))), \zeta(\bar s) \big), \ \text{i.e. the barrier touches the optimal ray}.
$$
This cannot happen before $\bar s_{\ell+1}$ by the iterative assumption, see Fig: \ref{Fig:S_ell_2}. In particular, the point
\begin{equation*}
\zeta(s') = \text{arg-min} \Big\{ \big| \zeta(s) - \zeta(\upsilon(\zeta(\bar s))) \big|, \zeta(s) \in \big( \zeta(\upsilon(\zeta(\bar s))),\zeta(\bar s) \big) \Big\}
\end{equation*}
satisfies $s' > \bar s$ (the latter being the first point), and hence by the structure of the optimal ray one has
\begin{equation*}
u(\zeta(s')) - u(\zeta(\bar s)) = - |\zeta(\bar s) - \zeta(s')| < 0,
\end{equation*}
contradicting the monotonicity of the function $u \circ \zeta$. Hence $u(x)$ defined in \eqref{Equa:u_ell_constr} is the time function on $\zeta([\bar s_{\ell+1},\bar s_{\ell+2}])$ because for $s \in (\bar s_{\ell+1},\bar s_{\ell+2})$ the optimal rays are touching $\zeta(s)$ only in the last point In particular the optimal rays are given by the construction of the previous point.

Using the iterative assumption on the structure of optimal rays (i.e. the optimal ray arriving in $\zeta(\bar s_{\ell+1})$ is the initial segment plus the spiral arc), we have proved the statement on the structure of the optimal rays: the initial segment, the spiral arc $\zeta([0,\upsilon(s)])$ and the final segment $[\zeta(\upsilon(s)),\zeta(s)]$.

\medskip

\noindent{\it Step 3.} We then obtain that \eqref{Equa:struct_optimal_ray_ell} is valid for all points $\zeta(s)$, and elementary geometric considerations give that
\begin{equation}
\label{Equa:angle_ray}
\angle \big( \zeta(s) - \zeta(\upsilon(\zeta(s))), p \big) \in \bigg( 0,\frac{\pi}{2} \bigg], \quad p \in \partial^- \zeta(s).
\end{equation}
The upper bound is the monotonicity, while the lower bound is because
\begin{itemize}
\item the angle is $>0$ at $\zeta(\bar s_{\ell+1})$ by convexity and the iterative assumption on $S_\ell$,
\item the derivative of the angle $\phi$ (with $\mathbf t(\upsilon) = e^{\i \phi(\upsilon)}$) w.r.t. $s$ satisfies the distributional bound
\begin{equation}\label{Equa:convex_phi_lower}
\frac{d^2\phi}{ds^2} + 2 \frac{d\phi}{ds} \geq 0,
\end{equation}
obtained by comparing $\zeta$ with a supporting line in the point $\zeta(s)$.
\end{itemize}

Using the local convexity about $\upsilon(x)$, again by elementary geometry one can show that
\begin{equation*}
s \mapsto \phi(s) \quad \text{such that} \quad e^{\i \phi(s)} = \frac{\zeta(s) - \zeta(\upsilon(\zeta(s)))}{|\zeta(s) - \zeta(\upsilon(\zeta(s)))|}
\end{equation*}
is locally Lipschitz, with right derivative
\begin{equation}
\label{Equa:inverse_der_r_phi}
\frac{d\phi}{ds} = \frac{(\zeta(s) - \zeta(\upsilon(\zeta(s)))) \wedge \mathbf t^+(s)}{|\zeta(s) - \zeta(\upsilon(\zeta(s)))|^2}.
\end{equation}
Hence, if \eqref{Equa:angle_ray} holds, $\phi(s)$ is locally invertible and monotone strictly increasing.

Therefore the next angle $\bar \phi_{\ell+2}$ is given by \eqref{Equa:next_angle_ell+2} and the last point $\bar s_{\ell+2}$ is given by \eqref{Equa:final_bar_s_ell+2}.

\medskip

{\it Step 4.} The function $s^+_\ell(\phi)$ is defined as the inverse of the function $s \mapsto \phi(s)$ above, and the previous point shows that it is Lipschitz and strictly increasing, bijective with Lipschitz inverse. Here we notice that because of \eqref{Equa:angle_ray}, we have that it has Lipschitz inverse in the whole interval $[\bar \phi_{\ell+1},\bar \phi_{\ell+2}]$, which would correspond to Point (2) of the statement.

The function $s^-_\ell(\phi)$ is given by
\begin{equation}
\label{Equa:s_-_def_prof}
s^-_\ell(\phi) = \upsilon(\zeta(s^+(\phi))) = \max \big\{ s : e^{\i \phi} \in \partial^- \zeta(s) \big\}.
\end{equation}
The latter representation shows that $s^-_\ell(\phi)$ is monotone increasing, right continuous, coinciding locally with the right inverse of $s \mapsto \{\phi, e^{\i \phi} \in \partial^- \zeta(s)\}$, and then BV.

Finally
\begin{equation*}
r_\ell(\phi) = \big| \zeta(s^+_\ell(\phi)) - \zeta(s^-_\ell(\phi)) \big|
\end{equation*}
is clearly BV, right continuous and
\begin{equation*}
\begin{split}
D^\sing r_\ell(\phi) &= -\frac{\zeta(s^+_\ell(\phi)) - \zeta(s^-_\ell(\phi))}{|\zeta(s^+_\ell(\phi)) - \zeta(s^-_\ell(\phi))|} \cdot D^{\mathrm{cantor}} \zeta(s^-(\phi)) \\
& \quad + \sum_{s^-_\ell(\phi) \ \text{jumps}} \Big( \big| \zeta(s^+_\ell(\phi)) - \zeta(s^-_\ell(\phi+)) \big| - \big| \zeta(s^+_\ell(\phi)) - \zeta(s^-_\ell(\phi-)) \big| \Big) \Diracd_{\phi}.
\end{split}
\end{equation*}
Formula \eqref{Equa:s_-_def_prof} gives that the singular parts of $\zeta(s^-_\ell(\phi))$ have direction $e^{\i \phi}$ and are positive, so that $D^\sing r_\ell \leq 0$.

\medskip

{\it Step 5.} It remains to prove the structure of the set $S_{\ell+1}$. The previous points shows that the convex curve $\zeta \rest_{[\bar s_{\ell+1}, \bar s_{\ell+2}]}$ rotates about $\zeta(\bar s_{\ell+1})$ until it crosses the half-line $\zeta(\bar s_{\ell+1}) + \mathbf t^-(\bar s_{\ell+1}) \R^+$ in the point $\zeta(\bar s_{\ell+2})$: its star shaped structure follows by elementary considerations.
\end{proof}

\begin{remark}
\label{Rem:other_choice_rounds}
We observe that the choice of the intervals $\bar \phi_{\ell+1} = \bar{\phi}_\ell + 2\pi + \angle(\mathbf t^-(\bar s_\ell),e^{\i \bar{\phi}_\ell})$ can be changed into
\begin{equation}
\label{Equa:round_freedom}
\bar \phi_{\ell+1} \in \bar \phi_\ell + 2\pi + \big( \angle(\mathbf t^-(\bar s_\ell),e^{\i \bar\phi_\ell}),\angle(\mathbf t^+(\bar s_\ell),e^{\i \bar{\phi}_\ell} )\big)
\end{equation}
without any variation in the proof, the idea being that in any case we know the evolution of the spiral for
\begin{equation*}
\phi \in \bar \phi_\ell + 2\pi + \big( \angle(\mathbf t^-(\bar s_\ell),e^{\i \phi_\ell},\angle(\mathbf t^+(\bar s_\ell),e^{\i \phi_\ell})\big).
\end{equation*}
More precisely, one has used in Step 1 of the previous proof the barrier
\begin{equation*}
\hat \zeta_\ell(s) = \begin{cases}
\zeta(s) & s \in [\bar s_\ell,\bar s_{\ell+1}), \\
\zeta(\bar s_{\ell+1}) + \mathbf v (s - \bar s_{\ell+1}) & s \geq \bar s_{\ell+1},
\end{cases} \mathbf v \in \partial^- \zeta(\bar s_{\ell+1}).
\end{equation*}
This observation is useful when taking limits of sequences of spirals, in which cases every round converges to a round satisfying \eqref{Equa:round_freedom}.
\end{remark}

\begin{figure}
\centering
\includegraphics[scale=0.6]{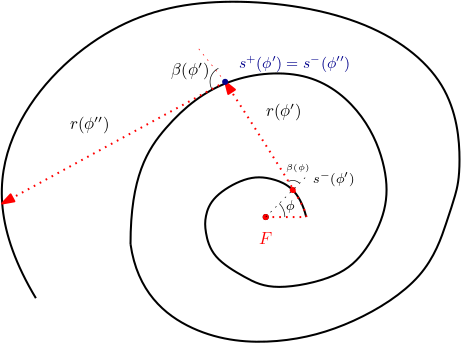}
\caption{Parametrization of a spiral-like strategy. Here the parametrization is given by the geometry of the fire rays. The angle $\phi'$ 
is obtained by the angle $\phi$ through Equation \eqref{eq:subs:angle}, while $r(\phi')$ represents the length of the last segment of the fire ray.}
\label{fig:param:spiral}
\end{figure}

By piecing together the function $s^-_\ell,s^+_\ell, r_\ell$ of the previous proposition we obtain the following

\begin{theorem}
\label{Cor:angle_preprest}
There exist \newglossaryentry{s1bar}{name=\ensuremath{\bar s_1},description={first point where an optimal ray leaves the spiral, Theorem \ref{Cor:angle_preprest}}} \newglossaryentry{sbar}{name=\ensuremath{\bar s},description={last point where an optimal ray leaves the spiral, Theorem \ref{Cor:angle_preprest}}} \newglossaryentry{s2bar}{name=\ensuremath{\bar s_2},description={first point where an optimal ray arrives after leaving the spiral, Theorem \ref{Cor:angle_preprest}}} $0 \leq \gls{s1bar} < \gls{s2bar},\gls{sbar} < S$, \newglossaryentry{phibar}{name=\ensuremath{\bar \phi},description={last rotation angle, Theorem \ref{Cor:angle_preprest}}} $\gls{phibar} \geq 2\pi$ and $3$ functions \newglossaryentry{rphi}{name=\ensuremath{r(\phi)},description={length of the last segment of an optimal ray, Theorem \ref{Cor:angle_preprest}}} \newglossaryentry{s-phi}{name=\ensuremath{s^-(\phi)},description={initial point of the last segment of an optimal ray, Theorem \ref{Cor:angle_preprest}}} \newglossaryentry{s+phi}{name=\ensuremath{s^+(\phi)},description={final point of the last segment of an optimal ray, Theorem \ref{Cor:angle_preprest}}}
\begin{equation*}
[2\pi,\bar \phi] \ni \phi \mapsto \big( \gls{rphi}, \gls{s-phi}, \gls{s+phi} \big) \in [0,+\infty) \times [\bar s_1, \bar s] \times [\bar s_2, S]
\end{equation*}
such that the following hold:
\begin{enumerate}
\item $r(\phi)$ is one-sided Lipschitz and right continuous, positive for all $\phi < \bar 
\phi$ and the spiral $\zeta$ is closed iff $r(\bar \phi) = 0$;
\item $s^-(\phi)$ is monotone increasing, right continuous, and $s^-(2\pi) = \bar s_1$, $s^-(\bar \phi) = \bar s$;
\item $s^+(\phi)$ is uniformly Lipschitz, bijective with Lipschitz inverse and increasing, and $s^+(2\pi) = s_2$, $s^+(\bar \phi) = S$;
\item it holds
\begin{equation*}
e^{\i \phi} \in \partial^- \zeta(s^-(\phi)), \quad \zeta(s^+(\phi)) = \zeta(s^-(\phi)) + r(\phi) e^{\i \phi};
\end{equation*}
\item if $s^+(\phi) = s^-(\phi')$, then
\begin{equation*}
\phi' - \phi - 2\pi \in \bigg( 0, \frac{\pi}{2} \bigg]. 
\end{equation*}
\end{enumerate}
\end{theorem}

The structure of optimal rays given in the previous theorem motivates the following

\begin{definition}
\label{def:starting:point}
The point $\zeta(s^-(\phi))$ will be called \emph{base point} of the optimal ray arriving in $\zeta(s^+(\phi))$. 
\end{definition}

We define the two following functions: \newglossaryentry{beta-}{name=\ensuremath{\beta^-(\phi)},description={minimal angle between the optimal ray and the subdifferential of the spiral at its end point}} \newglossaryentry{beta+}{name=\ensuremath{\beta^+(\phi)},description={maximal angle between the optimal ray and the subdifferential of the spiral at its end point}}
\begin{equation}
\label{Equa:beta_pm_phi_def}
\gls{beta-} = \angle \big( \mathbf t^-(s^+(\phi)),e^{\i \phi} \big), \quad \gls{beta+} = \angle \big( \mathbf t^+(s^+(\phi)),e^{\i \phi} \big).
\end{equation}
Being $s^+(\phi)$ Lipschitz and $\mathbf t(s)$ a BV function, it follows that $\beta^\pm(\phi)$ are the right/left representative of the BV function \newglossaryentry{beta}{name=\ensuremath{\beta(\phi)},description={angle between the optimal ray and the tangent to the spiral at its end point}}
\begin{equation}\label{Equa:beta_phi_def}
\gls{beta} = \angle \big( \dot \zeta(s^+(\phi)), e^{\i \phi} \big). 
\end{equation}
The range of these functions is $( 0,\frac{\pi}{2}]$ in the interval $\phi > 0$ by \eqref{Equa:angle_ray}, and clearly $D^\sing \beta \geq 0$, using the same estimate as \eqref{Equa:convex_phi_lower}: more precisely we have the one-sided Lipschitz estimate
\begin{equation}
\label{Equa:convex_beta}
D \beta + \mathscr L^1 \geq 0,
\end{equation}
obtained again by comparison with a supporting line at $\zeta(s^+(\phi))$. We will equivalently use the notation
\begin{equation*}
\beta^\pm(s) = \beta^\pm(\phi) \quad \text{where} \quad s = s^+(\phi).
\end{equation*}

Define the \emph{instantaneous burning rate} as \newglossaryentry{b(t)}{name=\ensuremath{b(t)},description={istantaneous burning rate}}
\begin{equation}
\label{def:burning:rate}
\gls{b(t)} = D_t \H^1(Z\cap\overline{R^Z(t)}) = D_t (u \circ \zeta)^{-1}([0,t]).
\end{equation}
The following proposition provides a formula for the admissibility functional \eqref{eq:admiss:funct} in the case of spiral strategies, due to the previous geometric properties of these barriers. Since it is a fairly easy consequence of the previous results, we will not provide a proof.

\begin{proposition}
\label{Lem:burnign_rate}
The admissibility functional (see Equation \eqref{eq:admiss:funct}) $\mathcal A$ as a function of $\phi$ can be written as
\begin{equation*}
\mathcal A(\phi) = 1 + s^-(\phi) + r(\phi) - \frac{1}{\sigma} s^+(\phi),
\end{equation*}
and the right continuous representative of the a.c. part of the burning rate function $b(s)$ is
\begin{equation}
\label{Equa:burning_rate}
b^+(t) = \frac{1}{\cos(\beta^+(s))}, \quad u(\zeta(s)) = t.
\end{equation}
\end{proposition}

As a simple application, we have:

\begin{corollary}
\label{Cor:sat_burning}
If $\zeta(s)$ is saturated, then $\beta^+(s) \leq \arccos(\frac{1}{\sigma})$.
\end{corollary}

\begin{proof}
By \eqref{Equa:burning_rate} we must have for saturated points $b^+ \leq \sigma$, which is the statement.
\end{proof}

\section{RDE description of a spiral}
\label{S:ODE}

In this section we will give a representation of a spiral strategy $Z$ by means of a Retarded Differential Equation (RDE), based on the angle-ray representation of Theorem \ref{Cor:angle_preprest}. This is a generalization of the RDE obtained in \cite{firefighter}, where the angle $\beta(\phi)$ is equal to $\alpha = \arccos(\frac{1}{\sigma})$. If one considers $\beta(\phi)$ as a control parameter, the spiral is described by the linear RDE \eqref{eq:spiral:angle} with control $\beta$, which in the case $\beta(\phi)$ is continuous becomes 
\begin{equation*}
\frac{d}{d\phi} r(\phi') = \cot(\beta(\phi') r(\phi') - \frac{r(\phi)}{\sin(\beta(\phi))}, \quad \phi' = \phi + 2\pi + \beta(\phi).
\end{equation*}
The term $\frac{r(\phi)}{\sin(\beta(\phi))}$ is exactly the curvature of the spiral at the base point of $r(\phi')$ (see Definition \ref{def:starting:point}). Observe moreover that the delay interval
\begin{equation*}
\phi' - \phi = 2\pi + \beta(\phi)
\end{equation*}
depends on the control parameter at the previous round, unless $\beta$ is constant.

The main result of this section is the equivalence between the solution to the RDE above and convex spiral barriers (Proposition \ref{Prop:construct_spiral}) and the analysis of the solution for $\beta(\phi) = \alpha$ constant: in particular the key criterion (Lemma \ref{lem:key}) to check if the solution $r(\phi)$ is diverging and its application to the saturated spiral (Proposition \ref{prop:saturated:not:closed_1}). In addition, we compute some Green kernels for the linear RDE and study some of their properties: they will be useful in the next chapters.

We start now by writing the RDE for spiral barriers.

\begin{lemma}
\label{lem:ODE:spiral}
Let $Z$ be a spiral (Definition \ref{def:adm:spirals}) and let $\phi \rightarrow \beta(\phi)$ (alternatively $s \rightarrow \beta(s)$) as in \eqref{Equa:beta_phi_def}. Then the following formulas hold:
\begin{enumerate}
\item The relation $s^-(\phi') = s^+(\phi)$ has solutions (see Figures \ref{fig:param:spiral}) 
\begin{equation}
\label{eq:subs:angle}
\phi' \in \phi + 2\pi + [\beta^-(\phi),\beta^+(\phi)].
\end{equation}
\item \label{Point_2:lem:ODE:spiral} It holds
\begin{equation}
\label{eq:var:length:1}
\frac{d}{d\phi} s^+(\phi) = \frac{r(\phi)}{\sin(\beta(\phi))}.
\end{equation}
\item \label{Point_3:lem:ODE:spiral} If $\zeta(s^-(\phi))$ is a differentiability points of $Z$ with $\ddot\zeta(s^-(\phi)) \not= 0$, the curvature \newglossaryentry{kappacurv}{name=\ensuremath{\kappa},description={curvature of a curve in the plane, $\kappa = \frac{1}{R}$ being $R$ the radius of curvature}} \gls{kappacurv} is given by 
\begin{equation*}
\kappa = \frac{1}{R} = \frac{d\phi}{ds^-(\phi)},
\end{equation*}
and, for $\phi > 2\pi$, the spiral satisfies the equation
\begin{equation}
\label{eq:spiral:angle}
\frac{d}{d\phi} r(\phi) = \cot(\beta(\phi)) r(\phi) - \frac{d}{d\phi} s^-(\phi)
\end{equation}
in the sense of distributions.
\end{enumerate}
\end{lemma}

Since the functions involved are BV by Theorem \ref{Cor:angle_preprest}, we will use also the notation
\begin{equation*}
Dr = \cot(\beta) r \mathscr L^1 + Ds^-
\end{equation*}
in the sense of measures.

The next definition corresponds to a generalization of the linkage introduced in \cite{firefighter}.

\begin{definition}[Subsequent angles]
\label{Def:subsequent_angles}
Given the rotation angle $\phi$, we will refer to the angles $\phi'$ given by \eqref{eq:subs:angle} as \emph{subsequent angles}.
\end{definition}

\begin{remark}
In the case of the first round  the equation for spirals reads as follows
\begin{equation}\label{eq:spiral:first:round}
\frac{dr}{ds}=\cos(\beta(s)),
\end{equation}
or, in the variable $\phi$,
\begin{equation*}
\frac{d}{d\phi}r(\phi)=\cot(\beta(\phi))r(\phi).
\end{equation*}
Note that \eqref{eq:spiral:angle} holds for all $\phi$ by extending $s^-([0,2\pi)) = 0$.
\end{remark}


\begin{proof}[Proof of Lemma \ref{lem:ODE:spiral}]
The first point is a consequence of Points (4),(5) of Theorem \ref{Cor:angle_preprest} and the definition of $\beta^\pm$ in \eqref{Equa:beta_pm_phi_def} (see Figure  \ref{fig:param:spiral}).

Point \eqref{Point_2:lem:ODE:spiral} is the inverse of \eqref{Equa:inverse_der_r_phi}.

Finally Point \eqref{Point_3:lem:ODE:spiral} follows by projecting the derivative of
\begin{equation*}
\zeta(s^+(\phi)) = \zeta(s^-(\phi)) + e^{\i \phi} r(\phi)
\end{equation*}
on the direction $e^{\i \phi}$ and using Point \eqref{Point_2:lem:ODE:spiral}.
\end{proof}

\begin{figure}
\centering
\includegraphics[scale=0.5]{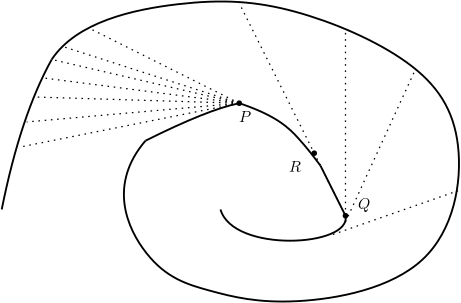}
\caption{In a point of non-differentiability $P$ the variation of the radius is computed as $\frac{d}{d\phi}r(\phi)=\cot\beta(\phi)r(\phi)$. The opposite situation occurs at the poin $R$, where the $r$ has a downward jump of size equal to the length of the segment $[Q,R]$.}
\label{fig:point:non:diff}
\end{figure}

\begin{remark}
\label{rmk:constant:angle}
In the case of $\frac{ds^-}{d\phi} = R(\phi) \mathscr L^1$, $R(\phi)$ being the curvature of $\zeta(s^-(\phi))$, we obtain the RDE
\begin{equation}
\label{eq:spiral:angle_2}
\frac{d}{d\phi} r(\phi) = \cot \beta(\phi) r(\phi) - R(\phi).
\end{equation}

When $s^-$ jumps (or it has a singular part), then $r(\phi)$ has the same singular part with opposite sign, in particular it jumps down, as observed in the proof of Proposition \ref{thm:param_1}.

If instead $s^-$ is constant, i.e. there is a corner point in $\zeta(s^-(\phi))$, then the ODE becomes
\begin{equation*}
\frac{d r(\phi)}{d\phi} = \cot\beta(\phi) r(\phi).
\end{equation*}
See Figure \ref{fig:point:non:diff}.
\end{remark}

In the next proposition we show that the one-sided Lipschitz bound on $\beta(\phi)$ given by \eqref{Equa:convex_beta} suffices to construct the spiral $\zeta$.

\begin{proposition}
\label{Prop:construct_spiral}
Assume that the function $\beta : (0,\infty) \to (0,\frac{\pi}{2}]$ satisfies
\begin{equation}
\label{Equa:beta_convex_1}
D\beta + \mathscr L^1 \geq 0, \quad \text{and} \quad \beta((0,\epsilon)) > 0, \int_0^\epsilon \frac{d\phi}{\beta(\phi)} < \infty \ \text{for some } \ \epsilon \ll 1.
\end{equation}
Then, there exists an angle $\bar \phi \in (0,+\infty]$ depending only on $\beta$ with the following properties: for every given $r(0)$, there is a unique spiral $\zeta$ such that the functions $r(\phi), s^-(\phi),s^+(\phi)$, $\phi \in [0,\bar \phi)$ constructed in Theorem \ref{thm:param_1} satisfy the equations
\begin{equation}
\label{Equa:eq_for_contr_spir}
\zeta(s^+(\phi)) = \zeta(s^-(\phi)) + r(\phi) e^{\i \phi}, \quad \frac{dr}{d\phi} = \cot(\beta) r + \frac{ds^-}{d\phi}, \quad r(0) = \lim_{\phi \searrow 0} r(\phi).
\end{equation}
If $\bar \phi < +\infty$, then as $\phi \nearrow \bar \phi$ either $r(\phi) \searrow 0$ or $r(\phi) \nearrow \infty$ with asymptotic direction $e^{\i \bar \phi}$.
\end{proposition}

The integrability of $\beta^{-1}$ near $\phi = 0$ is needed to construct an initial arc of the spiral, since the one-sided Lipschitz assumptions on $\beta$ will take care of the integrability at the next angles. The positivity assumption is just the requirement that we are rotating counterclockwise. We remark that every $\beta$ which is locally 1-sided Lipschitz in $(0,\epsilon)$ can be extended to a 1-sided Lipschitz function on $(0,\infty)$, so that $\beta $ defined on $(0,+\infty)$ is not a requirement.

\begin{proof}
We start with the first round, where the angle-ray parametrization corresponds to the polar coordinates: the ODE is then \eqref{eq:spiral:first:round}, whose solution is
\begin{equation}
\label{Equa:polar_spiral}
r(\phi) = r(0) e^{\int_0^\phi \cot(\beta) \mathscr L^1}.
\end{equation}
Due to the assumption that $\int_0^\epsilon \frac{d\phi}{\beta(\phi)} < \infty$, the above formula is locally meaningful.

The curve is given by
\begin{equation*}
\zeta(\phi) = r(\phi) e^{\i \phi},
\end{equation*}
and it is a fairly easy computation to show that \eqref{Equa:beta_convex_1} implies the convexity of $\zeta$. Note that if $\beta$ converges to $0$ for $\phi \nearrow \bar \phi$, then by the 1-sided Lipschitz condition \eqref{Equa:beta_convex_1} we get $\beta(\phi) \leq (\bar \phi - \phi)$, so that $r(\phi)$ given by \eqref{Equa:polar_spiral} diverges as $\phi \to \bar \phi$ with asymptotic direction $e^{\i \bar \phi}$, and in this case the construction procedure stops. Note that for this case $r(\phi) \geq r(0)$, so that if $r$ is not blowing up the polar representation stops when $\phi = 2\pi$. For coherence with the notation of Proposition \ref{thm:param_1}, we will set
\begin{equation*}
\bar s_1 = 0, \quad \bar \phi_1 = 2\pi, \quad \bar s_2 = \int_0^{2\pi} \frac{r(\phi)}{\sin(\beta(\phi))} d\phi,
\end{equation*}
the latter being the length of the spiral in the first round.

\medskip

Assume that the spiral has been constructed up to the rotation angle $\bar \phi_\ell$, with length $\bar s_{\ell+1}$: more precisely,
\begin{itemize}
\item the curve $\zeta(\phi)$, $\phi \in [\bar \phi_{\ell-1},\bar \phi_\ell]$, is locally convex, and it can be represented as in Proposition \ref{thm:param_1}: we will write $\zeta(s)$, $s \in [\bar s_{\ell},\bar s_{\ell+1}]$, for its length parametrization, and we will use the notation of convex spirals;
\item $\zeta(\bar \phi_\ell) \in \zeta(\bar \phi_{\ell-1}) + \mathbf t^-(\phi_{\ell-1}) \R^+$.
\end{itemize}
In particular, by defining the next angle
\begin{equation*}
e^{\i \bar \phi_{\ell+1}} = \mathbf t^-(\phi_\ell), \quad \bar \phi_{\ell+1} \in \bar \phi_\ell + 2\pi + \bigg(0,\frac{\pi}{2} \bigg],
\end{equation*}
the function
$$
[\bar \phi_\ell,\bar \phi_{\ell+1}) \ni \phi \mapsto s^-(\phi) \in [\bar s_{\ell},\bar s_{\ell+1}]
$$
is defined as in \eqref{Equa:s_-_def_prof}: it is the right continuous inverse of the monotone function $[\bar s_\ell,\bar s_{\ell+1}] \ni s \mapsto \mathbf t(s) = e^{\i \phi}$, $\phi \in [\bar \phi_\ell,\bar \phi_{\ell+1}]$. 
We can also apply Step 1 of the proof of Proposition \ref{thm:param_1} to construct the time function $u$ for the next round, see \eqref{Equa:u_ell_constr}.

Given the function $\beta(\phi)$, $\phi \in [\bar \phi_\ell,\bar \phi_{\ell+1}]$, we now solve the ODE \eqref{eq:spiral:angle}, namely
\begin{equation*}
D r = \cot(\beta) r \mathscr L^1 - D s^-,
\end{equation*}
yielding a unique right continuous BV function $r(\phi)$, $\phi \in [\bar \phi_\ell,\bar \phi_{\ell+1}]$. The domain of $r(\phi)$ will be the set $[\bar \phi_\ell,\bar \phi) = \big\{ r(\phi) \in (0,\infty) \big\}$. 

We can prolong the spiral $\zeta$ for $\phi \in [\bar \phi_\ell,\bar \phi)$ by defining the curve
\begin{equation*}
\zeta(\phi) = \zeta(s^-(\phi)) + r(\phi) e^{\i \phi}.
\end{equation*}
We need to show that this is actually a curve and it is Lipschitz, locally convex, and $\phi \mapsto u(\zeta(\phi))$ is monotone, where $u$ is well defined as in the proof of Proposition \ref{thm:param_1}, see Formula \eqref{Equa:struct_optimal_ray_ell} for the explicit computation of the time function at every round.

Taking the derivative and observing that $D\zeta(s^-(\phi)) = e^{\i \phi} |D\zeta(s^-(\phi))|$ by definition, we have
\begin{equation*}
D \zeta = e^{\i \phi} \big( D s^- + D r \big) + \i r e^{\i \phi} = (\cot(\beta) + \i) r e^{\i \phi},
\end{equation*}
where we recall that \gls{iimagi} is the imaginary unit. This shows that $\zeta$ is Lipschitz whenever $r < \infty$, and moreover that
\begin{equation*}
\frac{d}{d\phi} u \circ \zeta = \cot(\beta) r(\phi) > 0,
\end{equation*}
giving the monotonicity.

Being $\phi \mapsto \beta(\phi)$ a BV function by \eqref{Equa:beta_convex_1}, we can further differentiate $D\zeta$ obtaining
\begin{equation*}
\begin{split}
D \frac{D \zeta}{|D \zeta|} &= D [ (\cos(\beta) + \i \sin(\beta)) e^{\i \phi} ] = (D \cos(\beta) + \i D\sin(\beta)) e^{\i \phi} + (\i \cos(\beta) - \sin(\beta)) e^{\i \phi}.
\end{split}
\end{equation*}
Taking the scalar product with
\begin{equation*}
\frac{D\zeta^\perp}{|D\zeta|} = \bigg( - \sin \bigg( \frac{\beta^+(\phi) + \beta^-(\phi)}{2} \bigg) + \i \cos \bigg( \frac{\beta^+(\phi) + \beta^-(\phi)}{2} \bigg) \bigg) e^{\i \phi},
\end{equation*}
we get
\begin{equation*}
\begin{split}
\frac{D\zeta^\perp}{|D\zeta|} \cdot D \frac{D\zeta}{|D\zeta|} &= 1 + \cos \bigg( \frac{\beta^+ + \beta^-}{2} \bigg) D \sin(\beta) - \sin \bigg( \frac{\beta^+ + \beta^-}{2} \bigg) D \cos(\beta) \\
&= \frac{\sin( \frac{\beta^+ - \beta^-}{2})}{\frac{\beta^+ - \beta^-}{2}} D^\jump \beta + \mathscr L^1 + D^{\cont} \beta.
\end{split}
\end{equation*}
Hence \eqref{Equa:beta_convex_1} implies $\zeta$ is convex. We thus conclude that it is a spiral according to Definition \ref{def:adm:spirals}, once we reparametrized it by length.

If $\beta(\phi) \searrow 0$ as $\phi \nearrow \bar \phi$, the same conclusion as for the first round applies, i.e. the spirals blow up with asymptotic direction $e^{\i \check \phi}$. If $r(\phi) \searrow 0$, then the curve $\zeta(\phi)$ closes at the angle $\bar \phi$. Otherwise, we can prolong the spiral to the next round $\phi \in [\bar \phi_\ell,\bar \phi_{\ell+1}]$.

It is clear that such a construction gives a unique solution $r(\phi)$ and curve $\zeta(\phi)$, and by construction at each round the angle-ray representation coincides with the one of Proposition \ref{thm:param_1}.

\medskip

The last steps in the proof is to prove that $\bar \phi$ does not depend on $r(0)$: for this, we just observe that the equations \eqref{Equa:eq_for_contr_spir} are linear w.r.t. $r,\zeta$, so that $r(\phi) = r(0) r_{r(0)=1}(\phi)$ and similarly $\zeta(\phi) = r(0) \zeta_{r(0) = 1}(\phi)$. 
%
\end{proof}

We conclude this section by observing that in case of constant angle $\beta(\phi) \equiv \beta$, the radius of curvature is
\begin{equation*}
R(\phi')=\frac{r(\phi)}{\sin(\beta)},
\end{equation*}
and then the equation \eqref{eq:spiral:angle_2} reads as
\begin{equation*}
\frac{dr}{d\phi'}(\phi') = \cot(\beta) r(\phi') - \frac{r(\phi)}{\sin(\beta)}, \quad \phi' = 2\pi + \phi + \beta.
\end{equation*}
This is a linear RDE with constant coefficients introduced first by \cite{firefighter}, written as
\begin{equation}
\label{Equa:linear_retarded_1}
\frac{d}{d\phi}r(\phi) = \cot(\beta) r(\phi) - \frac{r(\phi - (2\pi + \beta))}{\sin(\beta)},
\end{equation}
which we are going to study in the next section.

\subsection{Analysis of the Saturated Spiral}
\label{Sss:equation_satur}

We start with the definition of saturated spiral, with the initial set of the fire $R_0 = B_1(0)$: even if it is not the initial fire set we consider, it will serve as a simpler situation where we develop one of the building blocks of our analysis. The analysis for the "saturated spiral" when $R_0 = B_r(0)$, $r \in (0,1)$, or $R_0 = \{(0,0)\}$ will be done in Section \ref{S:case:study}. In particular, it is a simple observation that given an admissible spiral $\bar \zeta$ up to the angle $\phi_0$ where $\zeta(\phi_0)$ is saturated, continuing $\bar \zeta$ as a saturated spiral (i.e. solving \eqref{eq:spiral:angle} with $\beta = \arccos(\frac{1}{\sigma})$ constant for $r(\phi)$ is the angle-ray coordinates) gives an admissible spiral up to the angle where $r(\phi)$ may become $0$: indeed by definition the rate of construction is equal to $\sigma$.

\begin{definition}
\label{Def:satur_spiral}
We define the \emph{saturated spiral} as the admissible spiral $Z_\mathrm{sat}$ with the following property:
\begin{equation*}
\mathcal{S}(Z_\mathrm{sat}) = \big\{ t \in [0,T] : \H^1(Z_\mathrm{sat} \cap \overline{R^{Z_\mathrm{sat}}(t)}) = \sigma t \big\} = [0,T].
\end{equation*}
\end{definition}

Clearly the above definition implies that $\Omega_0 = B_1(0)$, otherwise $(1,0) \in Z_\mathrm{sat}$ is not saturated.

Saturated spirals are strategies built assuming the instantaneous speed of construction is constant and takes the maximum value $\sigma$, that is the burning rate function $b(t)$ is constantly equal to $\sigma$.

The formula for burning rates \eqref{Equa:burning_rate} tells us that, in case of saturated spirals, the angle between the fire ray and the spiral is constantly equal to
$$
\beta^+(\phi) = \beta^-(\phi) = \arccos \bigg( \frac{1}{\sigma} \bigg),
$$
which will be denoted by \newglossaryentry{alphaspeed}{name=\ensuremath{\alpha},description={angle between the optimal ray and the tangent of a saturated spiral}} $\gls{alphaspeed} = \arccos(\frac{1}{\sigma})$. 

The first lemma addresses the first round: it is just the integration of the linear ODE \eqref{eq:spiral:first:round}, which is now with constant coefficient $\cot(\beta(\phi)) = \cot(\alpha)$, and it has been obtained already in \cite{firefighter}.

%

\begin{lemma}
\label{lem:initial:data:saturated}
Let $Z_\mathrm{sat}$ be a saturated spiral. Then, in the angle-ray representation $\zeta(\phi) = \zeta(\phi - 2\pi - \bar \alpha) + r_\mathrm{sat}(\phi) e^{\i \phi}$, the ray \newglossaryentry{rsatur}{name=\ensuremath{r_\mathrm{sat}(\phi)},description={saturated spiral barrier}} $\gls{rsatur}$ satisfies the RDE
\begin{equation}
\label{eq:ODE:saturated}
\frac{d}{d\phi} r_\mathrm{sat}(\phi) = r_\mathrm{sat}(\phi) \cot(\alpha) - \frac{r_\mathrm{sat}(\phi - (2\pi + \alpha))}{\sin(\alpha)},
\end{equation}
with initial data
\begin{equation*}
r_{\mathrm{sat},0}(\phi) = \begin{cases}
e^{\cot(\alpha) \phi} & \forall \phi \in [0,2 \pi), \\
(e^{2 \pi \cot(\alpha)} - 1) e^{\cot(\alpha) (\phi - 2 \pi)} & \forall \phi \in [2\pi,2\pi+\alpha].
\end{cases}
\end{equation*}
\end{lemma}


\begin{proof}
For $\phi\in[0,2\pi]$ it is a consequence of the RDE which reduces to the ODE $\dot r = \cot(\alpha) r$. Assume $\phi \in [2\pi,2\pi + \alpha]$, one easily computes that, calling $P=(1,0)$, the saturated spiral is a logarithmic spiral centered at $P$ with initial radius $r(2\pi)-1$. Finally, \eqref{eq:ODE:saturated} follows by \eqref{Equa:linear_retarded_1} for which the angle of the spiral is constantly equal to $\alpha$.
\end{proof}


In the next section we study the RDE of the saturated spiral.

\subsection{Formulation as a retarded ODE}
\label{ss:dde}

In this subsection we study in detail the RDE \eqref{eq:ODE:saturated}, and we find for which angle $\alpha$ (or equivalently for which burning rate) the associated spiral confines the fire. Part of these results have been already obtained in \cite{firefighter} (we repeat them for completeness using a different approach based on RDE analysis), with also a new direct proof that the saturated spiral constructed with speed $\sigma \leq \bar \sigma = 2.6144..$ does not confine the fire (Proposition \ref{prop:saturated:not:closed_1}): the speed \gls{barsigma} is called \emph{critical speed}, and it is computed as the solution to an eigen-equation.

The proof of this result exploits a change of variable in the equation \eqref{eq:ODE:saturated}, which will be of key importance also in the next sections (Lemma \ref{lem:key}). we refer mainly to \cite{RDE} for the theory of RDEs.

\subsubsection{Change of variables}
\label{Sss:change_variables}

Define the constant \newglossaryentry{cexp}{name=\ensuremath{c},description={exponent for the change of variable \eqref{Equa:c_gene_defi}}}
\begin{equation}
\label{Equa:c_gene_defi}
\gls{cexp} = \frac{\ln(\frac{2\pi+\alpha}{\sin(\alpha)})}{2\pi+\alpha},
\end{equation}
and rescale the solution $r(\phi)$ to \eqref{eq:ODE:saturated} as \newglossaryentry{rhogrowt}{name=\ensuremath{\rho(\tau)},description={rescaled variable $r$, Equation \eqref{Equa:rho_deff}}}
\begin{equation}
\label{Equa:rho_deff}
\gls{rhogrowt} = r((2\pi + \alpha) \tau) e^{-c (2\pi+\alpha) \tau}.
\end{equation}
By direct computations \newglossaryentry{acost}{name=\ensuremath{a(\alpha)},description={constant for the rescaled spiral equation \eqref{Equa:rescal_RDE}}}
\begin{align*}
\dot \rho &= (2\pi + \alpha) e^{-c (2\pi + \alpha) \tau} \bigg( \frac{d}{d\phi} r((2\pi + \alpha) \tau) - c r(2\pi + \alpha) \bigg) \\
&= (2\pi + \alpha) (\cot(\alpha) - c) \rho(\tau) - \frac{(2\pi + \alpha) e^{- c(2\pi + \alpha)}}{\sin(\alpha)} \rho(\tau-1) \\
&= \gls{acost} \rho(\tau) - \rho(\tau - 1),
\end{align*}
with
\begin{equation}
\label{Equa:a_const}
\gls{acost} = (2\pi + \alpha) (\cot(\alpha) - c(\alpha)) = (2\pi + \alpha) \cot(\alpha) - \ln \bigg( \frac{2\pi + \alpha}{\sin(\alpha)} \bigg).
\end{equation}
We will now study the spectrum of the RDE
\begin{equation}
\label{Equa:rescal_RDE}
\dot \rho(\tau) = a \rho(\tau) - \rho(\tau-1).
\end{equation}
The characteristic equation for the eigenvalues \newglossaryentry{lambda}{name=\ensuremath{\lambda},description={eigenvalues of the RDE}} \gls{lambda} is
\begin{equation}
\label{eq:eigenvalues}
\lambda + e^{-\lambda} - a = 0.
\end{equation}

We start by studying the real eigenvalues: the following result has been obtained also by \cite{Bressan_friends,firefighter}.

\begin{lemma}[Real Eigenvalues]
\label{lem:real:Eigenvalues}
The following hold.
\begin{enumerate}
\item If $a > 1$, there are two real eigenvalues $\lambda_0^-,\lambda_0^+$ with $\lambda_0^- < 0 < \lambda_0^+$ with algebraic multiplicity $1$.
\item If $a = 1$, there exists a unique real eigenvalue $\lambda_0 = 0$ with algebraic multiplicity $2$.
\item If $a < 1$ there are no real eigenvalues.
\end{enumerate}
\end{lemma}

\begin{proof}
It is elementary to see that the function $\lambda \mapsto \lambda + e^{-\lambda}$ has a global minimum at $\lambda = 0$ with value $1$ and it is strictly convex.
\end{proof}

From the point of view of solutions we thus have \cite[Theorem 4.1]{RDE} the following

\begin{corollary}
\label{Cor:rho_fund_sol_exp}
The following holds:
\begin{enumerate}
\item If $a > 1$, there are two solutions to \eqref{Equa:rescal_RDE} of the form
\begin{equation*}
\rho^-(\tau) = e^{\lambda_0^- \tau}, \quad \rho^+(\tau) = e^{\lambda_0^+ \tau},
\end{equation*}
the first exponentially decreasing while the second exponentially increasing. The two 
real eigenvalues $\lambda_0^-,\lambda_0^+$ with $\lambda_0^- < 0 < \lambda_0^+$ have algebraic multiplicity $1$.
\item If $a = 1$, there exists two solutions of the form
\begin{equation*}
\rho(\tau) = 1, \quad \rho(\tau) = \tau.
\end{equation*} 
\end{enumerate}
\end{corollary}

Since the function equation
\begin{equation*}
\alpha \mapsto a(\alpha) = (2\pi +\alpha) \bigg( \cot(\alpha) + \frac{\ln(\frac{2\pi + \alpha}{\sin(\alpha)})}{2\pi + \alpha} \bigg)
\end{equation*}
has derivative
\begin{equation*}
- \frac{2\pi + \alpha}{\sin^2(\alpha)} + (2\pi + \alpha) \cot(\alpha) = \frac{2\pi + \alpha}{\sin^2(\alpha)} \bigg( \frac{1}{2} \sin(2\alpha) - 1 \bigg) < 0,
\end{equation*}
by inverting it we obtain the following (see \cite[Lemma 6]{firefighter}):

\begin{corollary}
\label{Cor:eigen_angle}
Let \gls{alphabar} be such that
\begin{equation*}
(2\pi +\alpha) \bigg( \cot(\alpha) + \frac{\ln(\frac{2\pi + \alpha}{\sin(\alpha)})}{2\pi + \alpha} \bigg) = 1, \quad \bar \alpha = 1.1783..,
\end{equation*}
and
\begin{equation*}
\gls{barsigma} = \frac{1}{\cos(\bar \alpha)} = 2.61443....
\end{equation*}
Then the following hold.
\begin{itemize}
\item 
If $ \alpha < \bar \alpha$ ($\sigma > \bar \sigma)$, there are two exponential solutions to \eqref{eq:ODE:saturated}
\begin{equation*}
r^-(\phi) = e^{(c(\alpha) + \frac{\lambda^-_0}{2\pi + \alpha}) \phi}, \quad r^-(\phi) = e^{(c(\alpha) + \frac{\lambda^+_0}{2\pi + \alpha}) \phi}
\end{equation*}
with
$$
c(\alpha) + \frac{\lambda^-_0}{2\pi + \alpha} < \bar c = \cot(\bar \alpha) - \frac{1}{2\pi + \bar \alpha} < c(\alpha) + \frac{\lambda^+_0}{2\pi + \alpha}.
$$
\item If $\alpha = \bar \alpha$ ($\sigma = \bar \sigma$), there exist two solutions to \eqref{eq:ODE:saturated}
\begin{equation*}
r(\phi) = e^{\bar c \phi}, \quad r(\phi) = \phi e^{\bar c \phi},
\end{equation*}
with \newglossaryentry{cbar}{name=\ensuremath{\bar c},description={real eigenvalue in the critical case}}
$$
\gls{cbar} = \cot(\bar \alpha) - \frac{1}{2\pi + \bar \alpha} = 0.27995...
$$
\item If $\alpha > \bar \alpha$ ($\sigma > \bar \sigma$), there are no purely exponential solutions to  \eqref{eq:ODE:saturated}, i.e. every solution oscillates.
\end{itemize}
\end{corollary}

We now turn the attention to the complex eigenvalues $\lambda = x + \i y$: 
\begin{equation*}
x + \i y + e^{-x} \big( \cos(y) - \i \sin(y) \big) = a,
\end{equation*}
which gives
\begin{equation}
\label{Equa:real_complex_eig}
x = \ln \bigg( \frac{\sin(y)}{y} \bigg), \quad |y| \in (-\pi,\pi) + 2k\pi,
\end{equation}
and
\begin{equation}
\label{Equa:eigenvl_ima_1}
f(y) = \frac{e^{y \cot y} \sin y}{y} = e^a > 0.
\end{equation}

The function $y \mapsto f(y) = e^{y \cot y} \frac{\sin y}{y}$ has the properties:

\begin{figure}
\centering
\includegraphics[width=.7\textwidth]{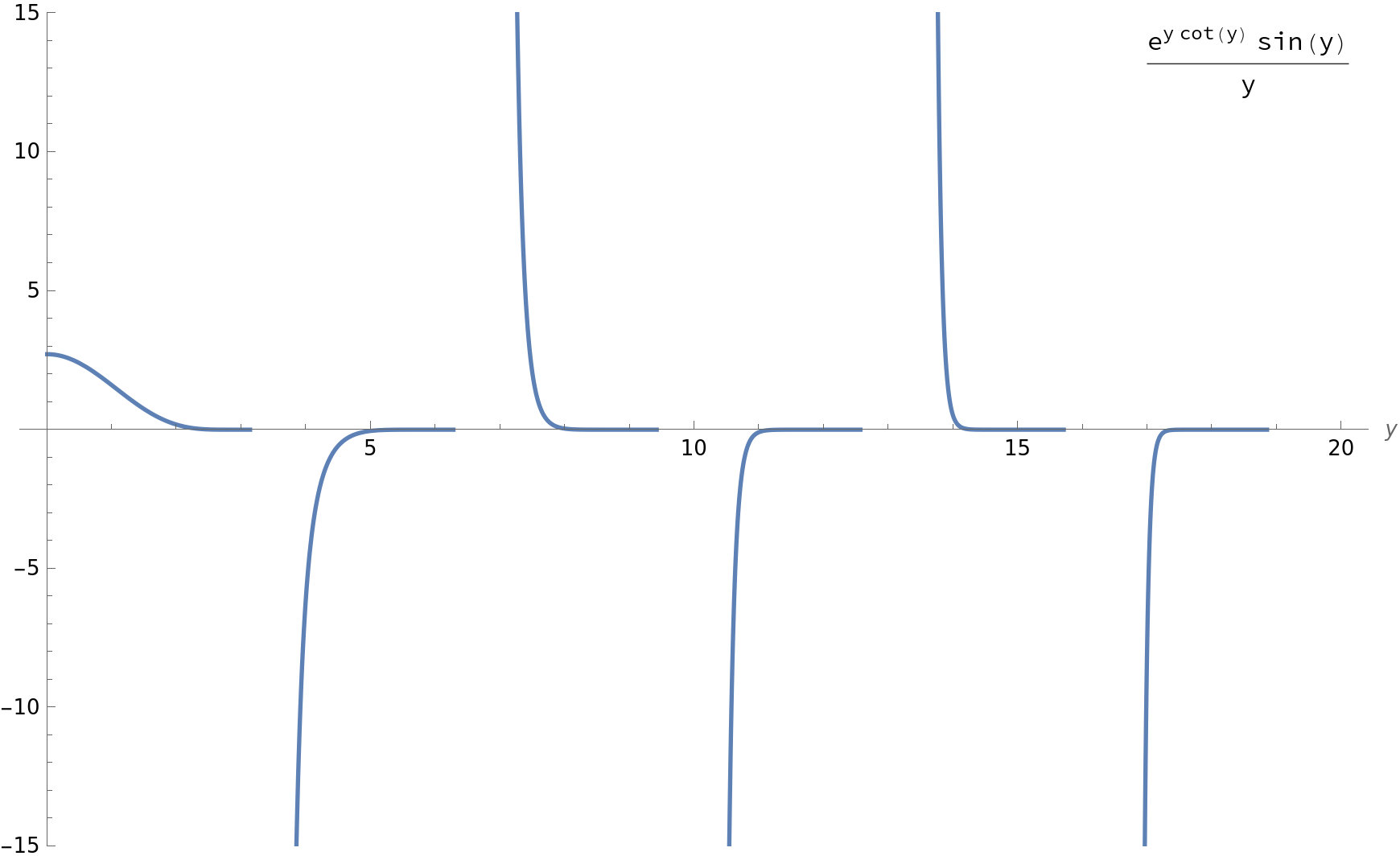}
\caption{Plot of the function $f(y)$, Equation \eqref{Equa:eigenvl_ima_1}.}
\label{Fig:complex_eigen_f}
\end{figure}

\begin{itemize}
\item it is symmetric, so we need to study only the case $y > 0$ (positive real part), and $f(0) = e$ (which is the critical case for the real eigenvalues);

\item it is positive when $\sin(y) > 0$, i.e. for
$$
y \in (0,\pi) + 2 k \pi, \quad k = 0,1,\dots,
$$
negative in
$$
y \in (\pi,2\pi) + 2k \pi, \quad k = 0,1,\dots,
$$
and
\begin{equation*}
\lim_{y \to k \pi-} f(y) = 0, \quad \lim_{y \to k \pi+} f(y) = \begin{cases}
- \infty & k \ \text{odd}, \\
+ \infty & k \ \text{even};
\end{cases}
\end{equation*}
moreover it is strictly monotone in each interval.

\item For each $k = 1,2,\dots$, let \newglossaryentry{ykappa}{name=\ensuremath{y_k},description={solutions to the eigenvalue equation \eqref{Equa:eigenvl_ima_1}}} \gls{ykappa} be the unique solution to
\begin{equation*}
\frac{e^{y \cot y} \sin y}{y} = e^a, \quad y \in (2k\pi,(2k+1)\pi].
\end{equation*}
Writing \newglossaryentry{deltakappa}{name=\ensuremath{\delta_k},description={reparametrization of $\lambda_k$}} $y_k = 2k \pi + \frac{\pi}{2} - \gls{deltakappa}$, we obtain the estimate
\begin{equation*}
e^a = \frac{e^{(2k+\frac{1}{2}) \pi \delta_k} \cos(\delta_k)}{(2k + \frac{\pi}{2}) - \delta_k} \sim \frac{e^{(2k+\frac{1}{2}) \pi \delta_k}}{(2k + \frac{\pi}{2})},
\end{equation*}
which shows that $\delta_k \sim \frac{\ln(k)}{2 \pi k}$ as $k \to \infty$.

\item The corresponding eigenvalues are \newglossaryentry{lambdak}{name=\ensuremath{\lambda_k},description={complex eigenvalues of the RDE}}
\begin{equation*}
\gls{lambdak} = \ln \bigg( \frac{\sin(y_k)}{y_k} \bigg) + \i y_k = a - \frac{y_k}{\sin(y_k)} \big( \cos(y_k) - \i \sin(y_k) \big), \quad |\lambda_k - a| = \frac{y_k}{\sin(y_k)}.
\end{equation*}
By means of the asymptotic estimate of $y_k$ we get
\begin{equation*}
\lambda_k - a = - \frac{(2 k + \frac{1}{2}) \pi}{\cos(\delta_k)} (\sin(\delta_k) - \i \cos(\delta_k)) \simeq - \ln(k) + \i k.
\end{equation*}
Hence they asymptotically lie on the curve
\begin{equation*}
x \simeq - \ln(y).
\end{equation*}

\item Since every $0$ of \eqref{Equa:eigenvl_ima_1} is simple, the complex eigenvalues are simple. 
\end{itemize}

Reconstructing the eigenexponent for the original RDE \eqref{eq:ODE:saturated:1}, we have proved the following lemma (see \cite[Lemma 6]{firefighter}).

\begin{lemma}[Complex eigenvalues]
\label{lem:complex:eigenvalues}
The following hold.
\begin{enumerate}
\item There is a sequence of complex conjugate exponents indexed by $k = 1,2,\dots$ \newglossaryentry{omegak}{name=\ensuremath{\omega_k},description={complex eigenvalues of the RDE \eqref{eq:ODE:saturated}}}
\begin{equation}
\label{Equa:omega_exp}
\gls{omegak} = c(\alpha) + \frac{\lambda_k}{2\pi + \alpha} = \cot(\alpha) - \frac{y_k}{2\pi + \alpha} (\cot(y_k) \pm \i) , \quad y_k = 2k \pi + \frac{\pi}{2} - \delta_k, \ \delta_k \simeq \frac{\ln(k)}{2k\pi},
\end{equation}
such that the corresponding eigensolution to \eqref{eq:ODE:saturated} is
\begin{equation*}
r_k(\phi) = e^{\omega^\pm_k \phi};
\end{equation*}
\item if $\alpha > \bar \alpha$, there is a couple of complex conjugate exponents $\omega_0$ computed according to \eqref{Equa:omega_exp} with $y_0 \in (0,\pi)$;
\item for $\alpha = \bar \alpha$ the couple at $|y| < \pi$ of the previous point merges to the exponent
$$
\bar c = \cot(\bar \alpha) - \frac{1}{2\pi + \bar \alpha}.
$$
\end{enumerate}
\end{lemma}

\subsubsection{Application to saturated spiral strategies}
\label{ss:subset:saturated}

Aim of this section is to prove that the saturated spiral confines the fire for $\sigma > \bar \sigma$ and does not confine the fire for $\sigma \leq \bar \sigma$. The first statement is quite easy and it has been already proved in \cite{Bressan_friends} and \cite[Theorem 1]{firefighter}, while the second is a consequence of the estimate on the number $n$ of rounds needed to close the saturated spiral, i.e. $n \sim \Im(\omega_0)^{-1}$ given in \cite[Theorem 4]{firefighter}. In this section we present a different approach for this second result which can be adapted to general spirals, and obtain the asymptotic behavior. 

We assume first that $\sigma > \bar\sigma$ and consider as initial data for the RDE \eqref{eq:ODE:saturated:1} any admissible spiral. This is the case where a firefighter constructs any admissible spiral with construction speed $\sigma$ up to some point $Q$, with the property that in the point $Q$ it holds $\beta^-(Q) \leq \bar \alpha$, and then it starts constructing a spiral with constant angle $\bar \alpha$. We will call \emph{saturated branches} the arcwise connected subsets of spiral-like strategies constructed with constant angle $\bar \alpha$.

The new spiraling strategy obtained in this way is still convex and it is admissible, being the point $Q$ admissible: indeed, if $Z$ is this particular spiral strategy, then
\begin{equation*}
\H^1(S\cap\overline{R^Z(t)})= \H^1(Z\cap\overline{R^Z(s_Q)})+\sigma (t-s_Q)\leq \sigma t,\quad\forall t\geq s_Q,
\end{equation*}
with $s_Q=u(Q)$ ($u$ is the minimum time function), where we have used that the burning rate \eqref{def:burning:rate} is constantly equal to $\sigma$ and the fact that $Q$ is admissible.

The previous discussion on the complex spectrum of the operator justifies the following statement.

\begin{proposition}[\cite{Bressan_friends,firefighter}]
\label{prop:at:crit:speed:fire:dies}
If $\sigma>\bar\sigma$, then for any initial data the saturated spiral confines the fire.
\end{proposition}

\begin{proof}
We consider the change of variables
\begin{equation*}
\rho(\tau) = r((2\pi+\alpha) \tau) e^{-(2\pi + \alpha) \cot(\alpha) \tau},
\end{equation*}
then the function $\rho$ satisfies the following RDE:
\begin{equation*}
\dot \rho(\tau) + \frac{(2\pi + \alpha)}{\sin(\alpha)} e^{ -(2\pi+\alpha) \cot(\alpha)} \rho(\tau-1) = 0.
\end{equation*}
It is a well known fact that every solution of a RDE of the form $\dot f(\tau) + c f(\tau - r) = 0$ with $c > 0$ is oscillating iff the characteristic equation $\lambda + c e^{- \lambda r} = 0$ has only complex eigenvalues (see for example \cite{RDErussi}).
\end{proof}

\begin{remark}
\label{Rem:numer_rounds}
Observing that by \eqref{Equa:omega_exp}
\begin{equation*}
\Re(\omega_k) \leq \Re(\omega_1) < \cot(\alpha) - \frac{2\pi}{2\pi + \alpha} \leq \cot(\bar \alpha) - \frac{2\pi}{2\pi + \bar \alpha} = - 0.428111,
\end{equation*}
so that for $\alpha \searrow \bar \alpha$ the only surviving eigenvectors are the ones corresponding to $\omega_0$, which takes about $\frac{\pi}{y_0}$ to close. A more precise argument using the projection of the initial data gives the estimate of \cite[Theorem 4]{firefighter}.
\end{remark}

For the analysis of the second case, we consider \eqref{Equa:rescal_RDE} and use the following simple lemma.

\begin{lemma}
\label{lem:key}
The following holds:
\begin{enumerate}
\item If $a > 1$ and $\rho(1) \geq \rho(\tau) > 0$ for every $\tau \in [0,1]$, then $\rho(\tau)$ diverges exponentially.

\item If $a = 1$ and $\rho(1) > \rho(\tau) > 0$, $\tau \in [0,1)$, then the solution $\rho(\tau)$ diverges linearly.

\item If $a < 1$, the solution $\rho(\tau)$ oscillates.
\end{enumerate}
\end{lemma}

In particular, the solution $r(\phi)$ is diverging when $\rho(\tau)$ satisfies the first two conditions.

\begin{proof}
It is immediate to see that in the case $a > 1$ the solution $\rho(\tau)$ is non decreasing  for $\tau \geq 2$. Similarly when $a = 1$ the solution $\rho(\tau)$ is strictly increasing for $\tau > 2$ under the assumptions of the statement. Moreover it is fairly easy to see from \eqref{Equa:real_complex_eig} that all complex eigenvalues have negative real part.

Thus if $a > 1$ the unique non decreasing eigensolution is $e^{\lambda_0^+ \tau}$, which is the asymptotic limit of $\rho(\tau)$. The same for $a = 1$, where the unique increasing solution is $\tau$.

The third point is due to the fact that there are only complex eigenvalues.
%
%
\end{proof}

\begin{remark}
\label{rmk:exponentially:explod}
Observe that the first two cases clearly hold also if
\begin{equation*}
\dot \rho(\tau) \geq a \rho(\tau) - \rho(\tau - 1), \quad a \geq 1.
\end{equation*}
\end{remark}

We now show how this lemma can be applied to a concrete case, namely the saturated spiral \gls{rsatur} of Lemma \ref{lem:initial:data:saturated}, to prove that the strategy cannot confine the fire.

\begin{proposition}
\label{prop:saturated:not:closed_1}
Let $Z$ be the saturated spiral. Then it does not confine the fire for $\sigma\leq \bar \sigma$.
\end{proposition}

\begin{figure}
\centering
\includegraphics[width=.7\textwidth]{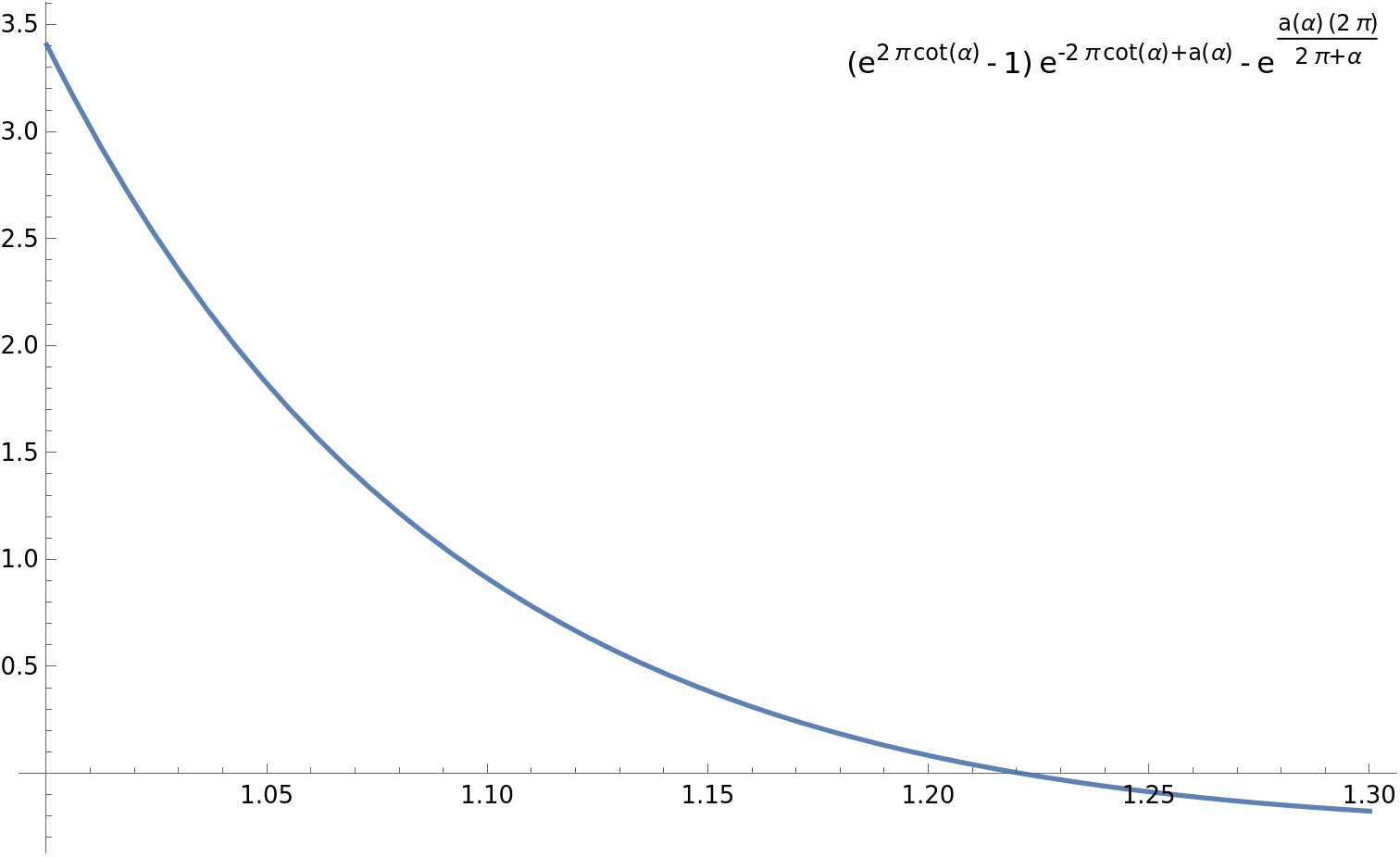}
\caption{Plot of the function \eqref{Equa:intern_prop_3.4}.}
\label{Fig:proposition_3.4_proof}
\end{figure}

\begin{proof}
Recall the initial data for the saturated spiral is given in Lemma \ref{lem:initial:data:saturated}: it is precisely
\begin{equation*}
r_\mathrm{sat}(\phi) = \begin{cases}
e^{\cot(\alpha) \phi} & \phi \in [0,2\pi), \\
(e^{2\pi \cot(\alpha)} - 1) e^{\cot(\alpha) (\phi-2\pi)} & \phi \in [2\pi,2\pi + \alpha].
\end{cases}
\end{equation*}

According to Lemma \ref{lem:key}, we have to prove that for all $\phi \in [0,2\pi)$
\begin{equation*}
r_\mathrm{sat}(2\pi + \alpha) e^{-c (2\pi + \alpha)} > r_\mathrm{sat}(\phi) e^{-c \phi} = \begin{cases}
e^{a(\alpha) \frac{\phi}{2\pi+\alpha}} & \phi \in [0,2\pi), \\
(e^{2\pi \cot(\alpha)} - 1) e^{- 2\pi \cot(\alpha) + a(\alpha) \frac{\phi}{2\pi+\alpha}} & \phi \in [2\pi,2\pi + \alpha);
\end{cases}
\end{equation*}
it is clear that the assumption of the lemma is verified for $\phi \in [2\pi,2\pi+\bar \alpha)$, being $r_\mathrm{sat}(\phi) e^{a(\alpha) \frac{\phi}{2\pi + \alpha}}$ increasing in that interval since $a(\alpha) \geq 1$. It is also clear from the formula that we need only to check that
\begin{align*}
&(e^{2\pi \cot(\alpha)} - 1) e^{- 2\pi \cot(\alpha) + a(\alpha)} - e^{a(\alpha) \frac{2\pi}{2\pi + \alpha}} > 0.
\end{align*}
Using \eqref{Equa:a_const} one can numerically verify that the function (see Fig. \ref{Fig:proposition_3.4_proof})
\begin{equation}
\label{Equa:intern_prop_3.4}
\alpha \mapsto (e^{2\pi \cot(\alpha)} - 1) e^{- 2\pi \cot(\alpha) + a(\alpha)} - e^{a(\alpha) \frac{2\pi}{2\pi + \alpha}}
\end{equation}
is decreasing w.r.t. $\alpha$, and it is positive for $\alpha < 1.2$. Being the critical angle $\bar \alpha = 1.17.. < 1.2$, then we can apply Lemma \ref{lem:key} and deduce that the saturated spirals does not confine the fire.
\end{proof}

\subsubsection{Green kernels of the RDE}
\label{Sss:green_kernels}

We write now explicitly the kernel for the RDE \eqref{Equa:rescal_RDE}, i.e. a functions $g : \R \to \R$ satisfying
\begin{equation*}
\dot g(\tau) = a g(\tau)-g(\tau-1) + \Diracd_0
\end{equation*}
in the sense of distributions, where \newglossaryentry{Diracdelta}{name=\ensuremath{\Diracd_0},description={Dirac's delta measure in $0$}} \gls{Diracdelta} is the Dirac's delta measure.

\begin{lemma}
\label{Lem:explic_kernel_RDE}
The kernel for the RDE \eqref{Equa:rescal_RDE} is the function \newglossaryentry{gkernel}{name=\ensuremath{g(\tau)},description={kernel for the linear Delay Differential Equation \eqref{Equa:rescal_RDE}}}
\begin{equation*}
\gls{gkernel} = \sum_{k=0}^\infty (-1)^k e^{a (\tau - k)} \frac{(\tau - k)^k}{k!} \ind_{[k,\infty)}(\tau).
\end{equation*}
\end{lemma}

\begin{proof}
Indeed
\begin{equation*}
\lim_{\tau \nearrow 0} g(\tau) = 0, \quad \lim_{\tau \searrow 0} g(\tau) = 1
\end{equation*}
and by differentiation for $\tau > 0$ we obtain
\begin{align*}
\dot g(\tau) &= \sum_{k=0}^\infty (-1)^k a e^{a (\tau - k)} \frac{(\tau - k)^k}{k!} \ind_{[k,\infty)}(\tau) + \sum_{k=1}^\infty (-1)^k e^{a (\tau - k)} \frac{(\tau - k)^{k-1}}{(k-1)!} \ind_{[k,\infty)}(\tau) \\
&= a g(\tau) - g(\tau - 1). \qedhere
\end{align*}
\end{proof}

Next we give the asymptotic behavior of $g$.

\begin{lemma}
\label{Lem:aympt_g}
It holds:
\begin{enumerate}
\item if $a > 1$, then the kernel $g(\tau)$ grows exponentially as
$$
\frac{e^{\lambda_0^+ \tau}}{1 - e^{-\lambda_0^+}}, 
$$
where $\tilde \lambda_0^-$ is the unique positive real eigenvalue of \eqref{eq:eigenvalues}, and moreover $\tau \mapsto g(\tau) e^{-  \lambda_0+ \tau} - \frac{1}{1 - e^{-\lambda_0^+}}$ is exponentially decaying for $\tau \to \infty$;

\item \label{Point_2:expli_exp_kern} if $a = 1$, then
\begin{equation*}
\tau \mapsto g(\tau) - 2 \tau + \frac{2}{3}
\end{equation*}
is exponentially decaying as $\tau \to \infty$;

\item if $a < 1$, defining the quantity
$$
U_0^+ = \frac{1}{1 - \tilde \lambda_0^+} = |U_0^+| e^{i \ln(\Im(U_0^+))},
$$
where $\tilde \lambda_0^+$ is the eigenvalue of \eqref{eq:eigenvalues} with minimal real part and positive imaginary part, we obtain that
\begin{equation*}
\tau \mapsto \big[ g(\tau) e^{- \Re(\lambda_0^+) \tau} - 2 |U_0^+| \cos(\Im(s_0^+) \tau + \ln(\Im(U_0^+)) \big) \Big]
\end{equation*}
is exponentially decaying as $\tau \to \infty$.
\end{enumerate}
\end{lemma}

\begin{figure}
\centering
\includegraphics[width=.45\textwidth]{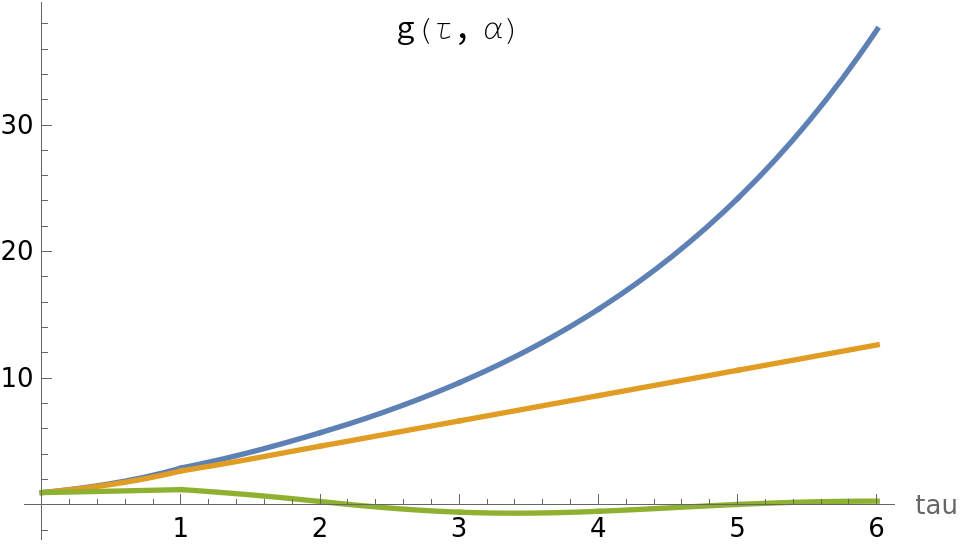}\hfill\includegraphics[width=.45\textwidth]{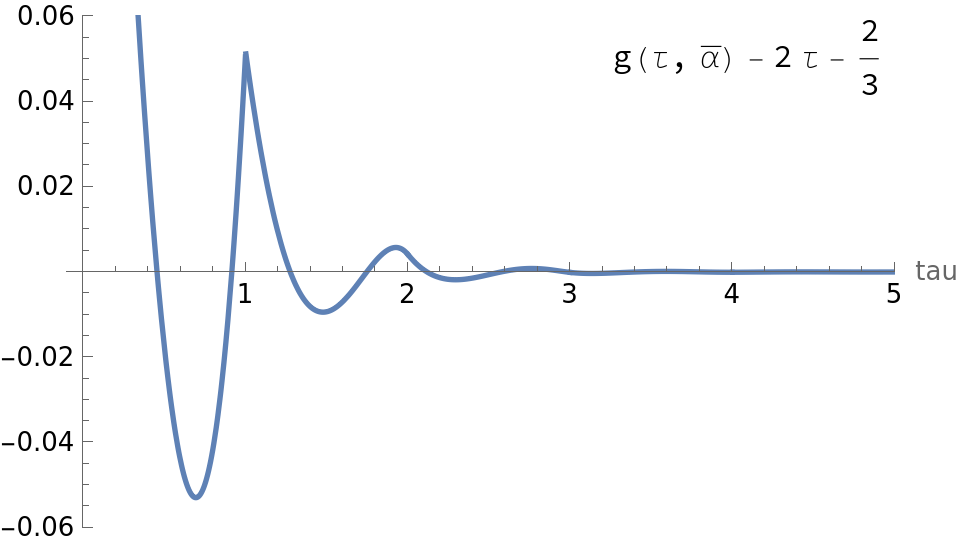}
\caption{Plot of $g(\tau)$ with $\alpha = \bar \alpha- 0.01$ (green), $\alpha = \bar \alpha$ (orange), $\alpha = \bar \alpha + 0.1$ (blue), and the asymptotic behavior of Point (2) of Lemma \ref{Lem:aympt_g}.}
\label{Fig:g_kernel_sympt}
\end{figure}

\begin{proof}
Take the Laplace transform \newglossaryentry{gkernelhat}{name=\ensuremath{\hat g(s)},description={Laplace transform of the kernel $g$}} \gls{gkernelhat} of $g(\tau)$: by using \eqref{Equa:rescal_RDE} we obtain
\begin{equation*}
- g(0) + s \hat g(s) = a \hat g(s) - e^{-s} \hat g(s) \quad \Longrightarrow \quad \hat g(s) = \frac{1}{s + e^{-s} - a}.
\end{equation*}
If $a \not= 1$, then $\hat g$ has a simple pole when 
\begin{equation*}
s + e^{-s} = a \quad \Longrightarrow \quad s = \lambda_k,
\end{equation*}
where $\lambda_k$ are the eigenvalues computed in the previous section. 

For $a < 1$, there is a couple of complex conjugate poles which have the least negative real part $\lambda_0^\pm = - x_0 \pm i y_0$, $x_0,y_0 > 0$: the residue about them is \newglossaryentry{U0pm}{name=\ensuremath{U_0^\pm},description={residue about the eigenvalues $s_0^\pm$ with largest real part}}
\begin{equation*}
\gls{U0pm} = \frac{1}{\frac{d}{ds} \big( s + e^{-s} - a \big) \rest_{\lambda_0^\pm}} = \frac{1}{1 - e^{-\lambda_0^\pm}}. 
\end{equation*}
Defining \newglossaryentry{V0tau}{name=\ensuremath{V_0(\tau)},description={inverse Laplace transform with poles in $s_0^\pm$}}
\begin{align*}
\gls{V0tau} &= U_0^- e^{\lambda_0^- \tau} + U_0^+ e^{\lambda_0^+ \tau} \\
&= 2 |U_0^+| e^{\Re(\lambda_0^+) \tau} \cos(\Im(\lambda_0^+) \tau + \Im(\ln(U_0^+))),
\end{align*}
where
\begin{equation*}
U_0^+ = |U_0^+| e^{i\Im(\ln(U_0^+))},
\end{equation*}
we conclude that
\begin{equation*}
\tau \mapsto (g(\tau) - V_0(\tau)) e^{- \Re(\lambda_0^+) \tau}
\end{equation*}
is exponentially decaying as $\tau \to +\infty$.

Same reasoning for $a > 1$: in this case there is only one pole $\lambda_0^+$ with positive real part, and then the residue is
\begin{equation*}
U_0^+ = \frac{1}{1 - e^{-\lambda_0^+}}.
\end{equation*}
We thus conclude as before that
\begin{equation*}
\tau \mapsto g(\tau) e^{-\lambda_0^+ \tau} - U_0^+
\end{equation*}
is exponentially decaying as $\tau \to \infty$.

In the critical case, the Laplace transform has a pole in $0$ of order $2$ and can be expanded as
\begin{equation*}
\hat g(s) = \frac{2}{s^2} \bigg( 1 + \frac{s}{3} + \mathcal O(s^2) \bigg) = \frac{2}{s^2} + \frac{2}{3s} + \mathcal O(1).
\end{equation*}
Hence
$$
g(\tau) - 2 \tau - \frac{2}{3}
$$
decays exponentially as the first eigenvalue with negative real part for $\tau \to \infty$.
\end{proof}

In the following we need also to compute the solution $m(\tau,\tau_1,\tau_2)$ for a Cauchy problem with a diffuse source: it will be useful when studying the fastest saturated spiral and its perturbations.

\begin{lemma}
\label{Lem:expli_source_RDE}
The unique solution to the problem
\begin{equation}
\label{Equa:kernel_source}
\dot \rho(\tau) = a \rho(\tau)-\rho(\tau-1) + (2\pi+\alpha) \ind_{[\tau_1,\tau_2)} e^{-c(2\pi+\alpha)\tau}, \quad \rho(\tau) = 0 \ \text{for $\tau < \tau_1$},
\end{equation}
where $c = c(\alpha)$ is given by \eqref{Equa:c_gene_defi}, is \newglossaryentry{mkernel}{name=\ensuremath{m(\tau,\tau_1,\tau_2)},description={kernel for the diffuse source of Lemma \ref{Lem:expli_source_RDE}}}
\begin{align*}
\gls{mkernel} &= \sum_{k=0}^\infty \Bigg[ \frac{e^{-c (2\pi + \alpha) \tau_1}}{\cot(\alpha)} \sum_{\ell = 1}^k (-1)^\ell e^{a (\tau - \tau_1 - k)} \frac{(\tau - \tau_1 - k)^\ell}{\ell! ((2\pi + \alpha) \cot(\alpha))^{k-\ell}} \ind_{[0,\infty)}(\tau - \tau_1 - k) \\
& \qquad \qquad - \frac{e^{-c(2\pi + \alpha) \tau_2}}{\cot(\alpha)} \sum_{\ell = 1}^k (-1)^\ell e^{a (\tau - \tau_2 - k)} \frac{(\tau - \tau_2 - k)^\ell}{\ell! (\cot(\alpha) (2\pi + \alpha))^{k-\ell}} \ind_{[0,\infty)}(\tau - \tau_2 - k) \\
& \qquad \qquad + \frac{\frac{e^{- c (2\pi+\alpha) \tau_1}}{\cot(\alpha)} [ e^{a (\tau - \tau_1 - k)} - 1 ]^+ - \frac{e^{-c(2\pi+\alpha) \tau_2}}{\cot(\alpha)} [ e^{a (\tau - \tau_2 - k)} - 1 ]^+}{(\cot(\alpha) (2\pi + \alpha))^{k}} \Bigg].
\end{align*}
\end{lemma}

\begin{proof}
Using Duhamel's formula we get that the solution is
$$
m(\tau,\tau_1,\tau_2) = g \ast \rho_0(\tau), \quad \rho_0(\tau) = (2\pi+\alpha)e^{-c(2\pi+\alpha)\tau}\ind_{[\tau_1,\tau_2)}(\tau),
$$
that is
\begin{align*}
m(\tau,&\tau_1,\tau_2) = (2\pi + \alpha) \int_{\tau_1}^{\tau_2} g(\tau-\tau') e^{-c (2\pi+\alpha) \tau'} d\tau' \\
&= (2\pi + \alpha) \int_{\tau_1}^{\tau_2} \sum_{k=0}^\infty (-1)^k e^{a (\tau - \tau' - k) - (\cot(\alpha) (2\pi + \alpha) - a) \tau'} \frac{(\tau - \tau' - k)^k}{k!} \ind_{[k,\infty)}(\tau - \tau') d\tau' \\
&= (2\pi + \alpha) \sum_{k=0}^\infty e^{-c (2\pi + \alpha) (\tau-k)} \int_{\tau_1}^{\tau_2} (-1)^k e^{\cot(\alpha) (2\pi + \alpha) (\tau - \tau' - k)} \frac{(\tau - \tau' - k)^k}{k!} \ind_{[k,\infty)}(\tau - \tau') d\tau' \\
&= \sum_{k=0}^\infty \frac{(2\pi + \alpha) e^{-c (2\pi + \alpha) (\tau-k)}}{(\cot(\alpha) (2\pi + \alpha))^{1+k}} \int_{\cot(\alpha) (2\pi + \alpha) (\tau - \tau_2 - k)}^{\cot(\alpha) (2\pi + \alpha) (\tau - \tau_1 - k)} (-1)^k e^{\sigma} \frac{\sigma^k}{k!} \ind_{[0,\infty)}(\sigma) d\sigma \\
&= \sum_{k=0}^\infty \frac{(2\pi + \alpha) e^{-c (2\pi + \alpha) (\tau - k)}}{(\cot(\alpha) (2\pi + \alpha))^{1+k}} \bigg[ \sum_{\ell=1}^k (-1)^\ell e^{\cot(\alpha) (2\pi + \alpha) (\tau - \tau_1 - k)} \frac{(\cot(\alpha) (2\pi + \alpha) (\tau - \tau_1 - k))^\ell}{\ell!} \ind_{[0,\infty)}(\tau - \tau_1 - k) \\
& \qquad \qquad \qquad \qquad \qquad \qquad \qquad + \big[ e^{\cot(\alpha) (2\pi + \alpha) (\tau - \tau_1 - k)} - 1 \big]^+ \\
& \qquad \qquad \qquad \qquad \qquad - \sum_{\ell=1}^k (-1)^\ell e^{\cot(\alpha) (2\pi + \alpha) (\tau - \tau_2 - k)} \frac{(\cot(\alpha) (2\pi + \alpha) (\tau - \tau_2 - k))^\ell}{\ell!} \ind_{[0,\infty)}(\tau - \tau_2 - k) \\
& \qquad \qquad \qquad \qquad \qquad \qquad \qquad - \big[ e^{\cot(\alpha) (2\pi + \alpha) (\tau - \tau_2 - k)} - 1 \big]^+ \bigg] \\
&= \sum_{k=0}^\infty \bigg[ \frac{e^{-c(2\pi + \alpha) \tau_1}}{\cot(\alpha)} \sum_{\ell = 1}^k (-1)^\ell e^{a (\tau - \tau_1 - k)} \frac{(\tau - \tau_1 - k)^\ell}{\ell! ((2\pi + \alpha) \cot(\alpha))^{k-\ell}} \ind_{[0,\infty)}(\tau - \tau_1 - k) \\
& \qquad \qquad - \frac{e^{-c(2\pi + \alpha) \tau_2}}{\cot(\alpha)} \sum_{\ell = 1}^k (-1)^\ell e^{a (\tau - \tau_2 - k)} \frac{(\tau - \tau_2 - k)^\ell}{\ell! (\cot(\alpha) (2\pi + \alpha))^{k-\ell}} \ind_{[0,\infty)}(\tau - \tau_2 - k) \\
& \qquad \qquad + \frac{\frac{e^{- c (2\pi+\alpha) \tau_1}}{\cot(\alpha)} [ e^{a (\tau - \tau_1 - k)} - 1 ]^+ - \frac{e^{-c(2\pi+\alpha) \tau_2}}{\cot(\alpha)} [ e^{a (\tau - \tau_2 - k)} - 1 ]^+}{(\cot(\alpha) (2\pi + \alpha))^{k}} \bigg],
\end{align*}
which is the formula in the statement.
\end{proof}

We compute the asymptotic behavior of $m$, which is the integral of the convolution of the asymptotic formula of $g$ with the source $\rho_0$. 

\begin{figure}
\centering
\includegraphics[width=.45\textwidth]{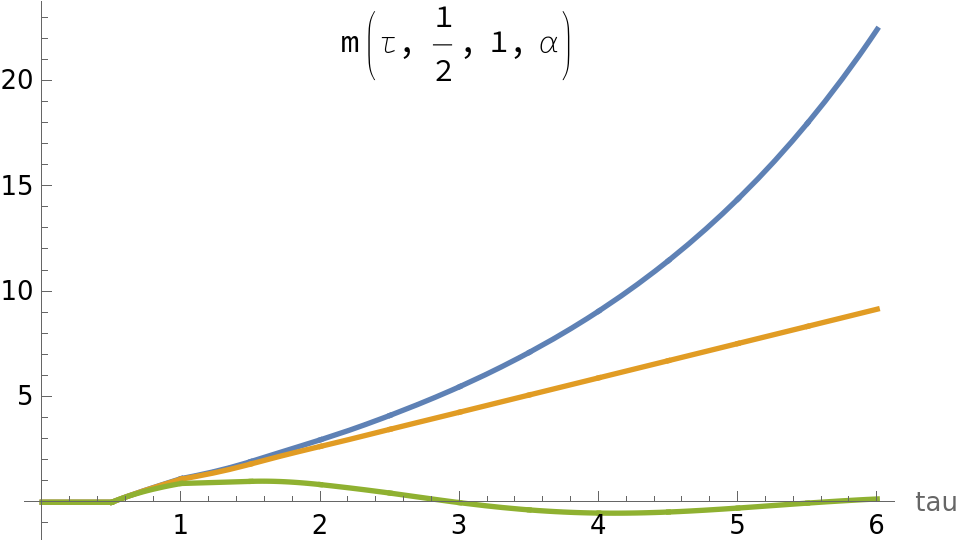}\hfill\includegraphics[width=.45\textwidth]{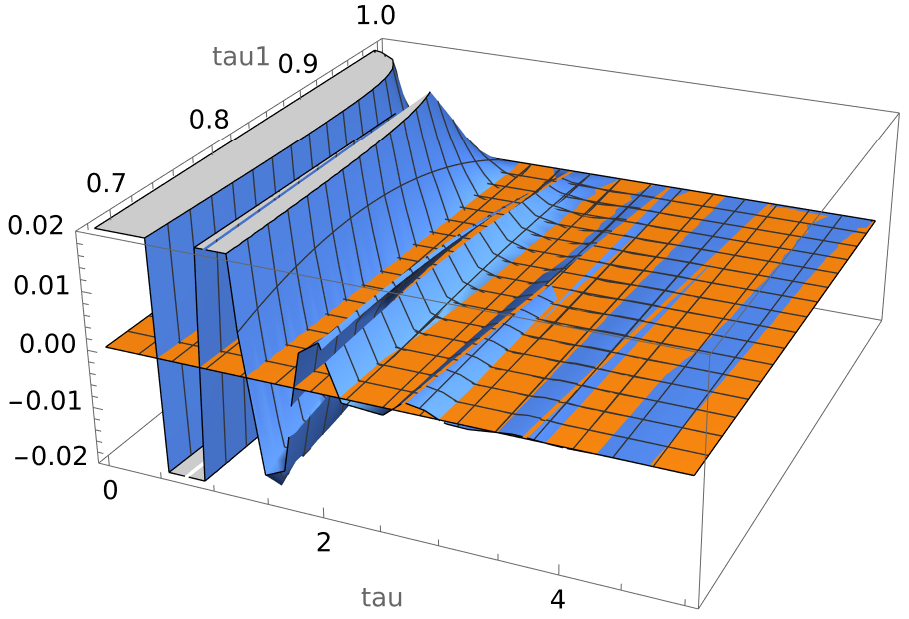}
\caption{The kernel $m(\tau,\frac{1}{2},1)$ with $\alpha = \bar \alpha- 0.1$ (green), $\alpha = \bar \alpha$ (orange), $\alpha = \bar \alpha + 0.1$ (blue), and its oscillations about the asymptotic expansion of $m(\tau,\tau_1,1)$ for $\tau_1 \in [\frac{1}{2},1]$.}
\label{Fig:m_m_sympt}
\end{figure}

\begin{lemma}
\label{Lem:m_aymptotic}
Denoting with \newglossaryentry{gkernelasympt}{name=\ensuremath{g^\mathrm{asympt}(\tau)},description={asymptotic expansion as $\tau \to \infty$ of the kernel $g$}} \gls{gkernelasympt} the asymptotic expansions of $g(\tau)$ in Lemma \ref{Lem:aympt_g}, it holds \newglossaryentry{mkernelasympt}{name=\ensuremath{m^\mathrm{asympt}(\tau,\tau_1,\tau_2)},description={asymptotic expansion as $\tau \to \infty$ of the kernel $m(\tau,\tau_1,\tau_2)$}}
\begin{equation*}
\gls{mkernelasympt} = \int_{\tau_1}^{\tau_2} \rho_0(\tau') \ast g^\mathrm{asympt}(\tau - \tau') d\tau'.
\end{equation*}
In particular, it is blowing up exponentially for $a > 1$, decaying exponentially and oscillating for $a < 1$ and in the critical case $a = 1$ we have the asymptotic linear growth
\begin{align*}
m^{\mathrm{asympt}}(\tau,\tau_1,\tau_2) &\sim \bigg( 2 \tau + \frac{2}{3} \bigg) \frac{e^{-\bar c(2\pi + \bar \alpha) \tau_1} - e^{-\bar c(2\pi + \bar \alpha) \tau_2}}{\bar c} \\
& \quad - 2 \frac{[\bar c(2\pi + \alpha) \tau_1 + 1] e^{-\bar s(2\pi + \bar \alpha) \tau_1} - [\bar c(2\pi + \alpha) \tau_2 + 1] e^{-\bar c(2\pi + \bar \alpha) \tau_2}}{\bar c^2} + \mathcal O(e^{- \Re(s_1^\pm) \tau}),
\end{align*}
with $s_1^\pm$ the first non $0$ eigenvalues.
\end{lemma}

\begin{proof}
Integrating
\begin{align*}
m(\tau,\tau_1,\tau_2) &\sim \int_{\tau_1}^{\tau_2} (2\pi + \bar \alpha) e^{-\bar c (2\pi + \bar \alpha) \tau'} \bigg( 2 (\tau - \tau') + \frac{2}{3} + \mathcal O(e^{- \Re(s_1^\pm) (\tau - \tau')}) \bigg) d\tau' \\
&= \bigg( 2 \tau + \frac{2}{3} \bigg) \int_{(2\pi + \bar \alpha) \tau_1}^{(2\pi + \bar \alpha) \tau_2} e^{-\bar c \sigma} d\sigma - 2 \int_{(2\pi + \bar \alpha) \tau_1}^{(2\pi + \bar \alpha) \tau_2} \sigma e^{-\bar c \sigma} d\sigma + \mathcal O(e^{- \Re(s_1^\pm) \tau}) \\
&= \bigg( 2 \tau + \frac{2}{3} \bigg) \frac{e^{-\bar s(2\pi + \bar \alpha) \tau_1} - e^{-\bar c(2\pi + \bar \alpha) \tau_2}}{\bar c} \\
& \quad - 2 \frac{[\bar c(2\pi + \alpha) \tau_1 + 1] e^{-\bar s(2\pi + \bar \alpha) \tau_1} - [\bar c(2\pi + \alpha) \tau_2 + 1] e^{-\bar c(2\pi + \bar \alpha) \tau_2}}{\bar c^2} + \mathcal O(e^{- \Re(s_1^\pm) \tau}),
\end{align*}
which is the statement.
\end{proof}

\begin{remark}
\label{Rem:oringal_phi_kernel}
In the original variables $r,\phi$, by \eqref{Equa:rho_deff} the kernel \newglossaryentry{Gkernel}{name=\ensuremath{G(\phi)},description={Green kernel for the equation \eqref{eq:ODE:saturated}}} \gls{Gkernel} is computed as
\begin{align*}
\gls{Gkernel} &= g \bigg( \frac{\phi}{2\pi + \alpha} \bigg) e^{c \phi} \\
&= \sum_{k=0}^\infty (-1)^k e^{\cot(\alpha) \phi - a k} \frac{(\frac{\phi}{2\pi+\alpha} - k)^k}{k!} \ind_{(0,\infty]}\bigg( \frac{\phi}{2\pi+\alpha} - k \bigg) \\
&= \sum_{k=0}^\infty (-1)^k e^{\cot(\alpha)(\phi - (2\pi + \alpha) k)} \frac{(\phi - (2\pi + \alpha) k)^k}{k! \sin(\alpha)^k} \ind_{[0,\infty)}(\phi - (2\pi + \alpha)k),
\end{align*}
where we used \eqref{Equa:c_gene_defi}. We recall that this means in distributions
\begin{equation*}
\dot G(\phi) = \cot(\alpha) G(\phi) - \frac{G(\phi - (2\pi + \alpha))}{\sin(\alpha)} + \Diracd_0.
\end{equation*}

In the original variables $r,\phi$ the Cauchy problem with source \eqref{Equa:kernel_source} corresponds to
\begin{equation}
\label{Equa:diffuse_source}
\dot r(\phi) = \cot(\alpha) r(\phi) - \frac{r(\phi-2\pi - \alpha)}{\sin(\alpha)} + \ind_{[\phi_1,\phi_2)},
\end{equation}
and the solution is \newglossaryentry{Mkernel}{name=\ensuremath{M(\phi,\phi_1,\phi_2)},description={kernel for the diffuse source \eqref{Equa:diffuse_source}}}
\begin{align*}
\gls{Mkernel} &= e^{c \phi} m \bigg( \frac{\phi}{2\pi + \alpha}, \frac{\phi_1}{2\pi + \alpha}, \frac{\phi_2}{2\pi + \alpha} \bigg) \\
&= \sum_{k=0}^\infty \bigg[ \sum_{\ell = 1}^k (-1)^\ell e^{\cot(\alpha) (\phi - \phi_1 - (2\pi + \alpha) k)} \frac{(\phi - \phi_1 - (2\pi + \alpha) k)^\ell}{\ell! \sin(\alpha)^k \cot(\alpha)^{k+1-\ell}} \ind_{[0,\infty)}(\phi - \phi_1 - (2\pi + \alpha) k) \\
& \qquad \qquad - \sum_{\ell = 1}^k (-1)^\ell e^{\cot(\alpha) (\phi - \phi_2 - (2\pi + \alpha) k)} \frac{(\phi - \phi_2 - (2\pi + \alpha) k)^\ell}{\ell! \sin(\alpha)^k \cot(\alpha)^{k+1-\ell}} \ind_{[0,\infty)}(\phi - \phi_2 - (2\pi + \alpha) k) \\
& \qquad \qquad + \frac{[ e^{\cos(\alpha) (\phi - \phi_1 - (2\pi + \alpha) k)} - 1 ]^+ - [ e^{\cot(\alpha) (\phi - \phi_2 - (2\pi + \alpha) k)} - 1 ]^+}{\cot(\alpha)} \bigg].
\end{align*}
\end{remark}

\begin{remark}
\label{Rem:diff_relations}
The following differential relations are elementary, but used very often in the next sections: we will thus collect them here as reference.

\begin{description}
\item[Derivative of $g,G$] it holds
\begin{equation*}
\frac{d}{d\tau} \big[ \partial_{\phi_0} g(\tau - \tau_0) \big] = a \big[ \partial_{\tau_0} g(\tau-\tau_0) \big] - \big[ \partial_{\tau_0} g(\tau - \tau_0 - 1) \big] - \frac{d}{d\tau_0} \Diracd_{\tau_0},
\end{equation*}
\begin{equation*}
\partial_{\tau_0} g(\tau-\tau_0) = - \Diracd_{\tau_0} - a g(\tau - \tau_0) + g(\tau - \tau_0 - 1).
\end{equation*}

Similarly for $G$
\begin{equation*}
\frac{d}{d\phi} \big[ \partial_{\phi_0} G(\phi - \phi_0) \big] = \cot(\alpha) \big[ \partial_{\phi_0} G(\phi - \phi_0) \big] - \frac{\big[ \partial G(\phi - \phi_0 - 2\pi - \alpha) \big]}{\sin(\alpha)} + \frac{d}{d\phi_0} \Diracd_{\phi_0},
\end{equation*}
\begin{equation}
\label{Equa:ker_der_G}
\partial_{\phi_0} G(\phi-\phi_0) = - \Diracd_{\phi_0} - \cot(\alpha) G(\phi - \phi_0) + \frac{G(\phi - \phi_0 - 2\pi - \alpha)}{\sin(\alpha)}.
\end{equation}

\item[Derivative of $m,M$] it holds
\begin{equation*}
\frac{d}{d\tau} \big[ \partial_{\tau_1} m(\tau,\tau_1,\tau_2) \big] = a \big[ \partial_{\tau_1} m(\tau,\tau_1,\tau_2) \big] - \big[ \partial_{\tau_1} m(\tau,\tau_1,\tau_2) \big] - (2\pi + \alpha) e^{-c(2\pi + \alpha) \tau_1} \Diracd_{\tau_1},
\end{equation*}
\begin{equation*}
\partial_{\tau_1} m(\tau,\tau_1,\tau_2) = - (2\pi + \alpha) e^{-c(2\pi + \alpha) \tau_1} g(\tau - \tau_1),
\end{equation*}
\begin{equation*}
\frac{d}{d\tau} \big[ \partial_{\tau_2} m(\tau,\tau_1,\tau_2) \big] = a \big[ \partial_{\tau_2} m(\tau,\tau_1,\tau_2) \big] - \big[ \partial_{\tau_2} m(\tau,\tau_1,\tau_2) \big] + (2\pi + \alpha) e^{-c(2\pi + \alpha) \tau_2} \Diracd_{\tau_2},
\end{equation*}
\begin{equation*}
\partial_{\tau_2} m(\tau,\tau_1,\tau_2) = (2\pi + \alpha) e^{-c(2\pi + \alpha) \tau_2} g(\tau - \tau_2).
\end{equation*}

Similarly for $M$
\begin{equation*}
\frac{d}{d\phi} \big[ \partial_{\phi_1} M(\phi,\phi_1,\phi_2) \big] = \cot(\alpha) \big[ \partial_{\phi_1} M(\phi,\phi_1,\phi_2) \big] - \frac{\big[ \partial_{\phi_1} M(\phi,\phi_1,\phi_2) \big]}{\sin(\alpha)} - \Diracd_{\phi_1},
\end{equation*}
\begin{equation}
\label{Equa:der_M_phi1}
\big[ \partial_{\phi_1} M(\phi,\phi_1,\phi_2) \big] = - G(\phi - \phi_1),
\end{equation}
\begin{equation*}
\frac{d}{d\phi} \big[ \partial_{\phi_2} M(\phi,\phi_1,\phi_2) \big] = \cot(\alpha) \big[ \partial_{\phi_2} M(\phi,\phi_1,\phi_2) \big] - \frac{\big[ \partial_{\phi_2} M(\phi,\phi_1,\phi_2) \big]}{\sin(\alpha)} + \Diracd_{\phi_2},
\end{equation*}
\begin{equation*}
\big[ \partial_{\phi_2} M(\phi,\phi_1,\phi_2) \big] = G(\phi - \phi_2).
\end{equation*}
\end{description}

\end{remark}

For the sake of completeness, in the critical case $\alpha = \bar \alpha$ we give the asymptotic expansions of $G(\phi - \phi_0)$ and $M(\phi,\phi_1,\phi_2)$.

\begin{corollary}
\label{Cor:crit_G_M}
When $\alpha = \bar \alpha$ the following expansions hold
\begin{equation*}
G(\phi - \phi_0) = (\phi - \bar \phi_0) e^{\bar c (\phi - \bar \phi_0)} \bigg( 1 + \mathcal O \bigg( \frac{1}{\phi - \bar \phi_0} \bigg) \bigg),
\end{equation*}
\begin{equation*}
M(\phi,\phi_1,\phi_2) = 2 \frac{1 - e^{\bar c(\phi_2 = \phi_1)}}{\bar c} (\phi - \phi_1) e^{\bar c (\phi - \phi_1)} \bigg( 1 + \mathcal O \bigg( \frac{1}{\phi - \phi_1} \bigg) \bigg).
\end{equation*}
\end{corollary}

\begin{proof}
It is an application of the asymptotic expansions for $g(\tau)$ and $m(\tau,\tau_1,\tau_2)$.
\end{proof}

\subsubsection{Length of saturated spirals}
\label{Sss:length_saturated}

For later use, we now compute the length of a saturated spiral, namely the quantity 
\newglossaryentry{Gcal}{name=\ensuremath{\mathcal G(\phi)},description={primitive of the kernel G}} \newglossaryentry{Heavyside}{name=\ensuremath{H(\phi)},description={Heaviside function}}
\begin{equation*}
\gls{Gcal} = \int_{-\infty}^\phi \frac{G(\phi')}{\sin(\bar \alpha)} d\phi' = \int_{0}^{\tau} g(\tau') e^{\bar c(2\pi + \bar \alpha) (\tau' + 1)} d\tau', \quad \phi = (2\pi + \bar \alpha) \tau,
\end{equation*}
satisfying the RDE 
\begin{equation}
\label{Equa:eq_for_calG}
\frac{d}{d\phi} \mathcal G(\phi) = \cot(\bar \alpha) \mathcal G(\phi) - \frac{\mathcal G(\phi - 2\pi - \bar \alpha)}{\sin(\bar \alpha)} + \frac{H(\phi)}{\sin(\bar \alpha)},
\end{equation}
obtained by computing the expression $\frac{d^2}{d\phi^2}\mathcal{G}(\phi)$, 
where \gls{Heavyside} is the Heaviside function
\begin{equation*}
H(\phi) = \begin{cases}
0 & x < 0, \\
1 & x \geq 0.
\end{cases}
\end{equation*}
In terms of $\tau$ we will write also \newglossaryentry{gfrak}{name=\ensuremath{\mathfrak g(\tau)},description={primitive of the function $g(\tau) e^{\bar c (2\pi + \bar \alpha)(\tau+1)}$}}
\begin{equation*}
\gls{gfrak} = \int_{-\infty}^\tau g(\tau') e^{\bar c (2\pi + \bar \alpha)(\tau'+1)} d\tau'.
\end{equation*}
Using the formula
\begin{equation*}
\int_0^x (-1)^k \frac{e^{x'} (x')^k}{k!} dx'= \sum_{\ell = 0}^k (-1)^\ell \frac{e^x x^\ell}{\ell!} - 1,
\end{equation*}
we obtain the explicit formulas
\begin{equation*}
\mathcal G(\phi) = \sum_{k=0}^\infty \frac{1}{\cos(\bar\alpha)^{k+1}} \bigg[ - 1 + \sum_{j = 0}^k (-1)^j e^{\cot(\bar\alpha) (\phi - (2\pi + \bar\alpha)k)} \frac{(\cot(\bar\alpha) (\phi - (2\pi +\bar \alpha) k))^j}{j!} \bigg] \ind_{[0,\infty)}(\phi - (2\pi +\bar \alpha) k).
\end{equation*}

For the kernel $M(\phi,\phi_1,\phi_2)$,
\begin{equation*}
M \bigg( \phi,\phi_0 + 2\pi + \frac{\pi}{2}, \phi_0 + 2\pi + \frac{\pi}{2} + \Delta \phi \bigg) = \int_{\phi_0 + 2\pi + \frac{\pi}{2}}^{\phi_0 + 2\pi + \frac{\pi}{2} + \Delta \phi} G(\phi - \phi') d\phi'.
\end{equation*}
we have with simple calculus
\begin{align*}
\int_0^\phi \frac{M(\phi',\phi_1,\phi_2)}{\sin(\bar \alpha)} d\phi' &= \int_{0}^\phi \bigg[ \int_{\phi_1}^{\phi'} - \int_{\phi_2}^{\phi'} \bigg] \frac{G(\phi' - \phi'')}{\sin(\bar \alpha)} d\phi'' d\phi' \\
&= \bigg[ \int_0^{\phi - \phi_1} - \int_0^{\phi - \phi_2} \bigg] \int_0^{\phi'} \frac{G(\phi' - \phi'')}{\sin(\bar \alpha)} d\phi'' d\phi' \\
&= \frac{(2\pi + \bar \alpha)^2}{\sin(\bar \alpha)} \bigg[ \int_0^{(2\pi + \bar \alpha)(\tau - \tau_1)} - \int_0^{(2\pi + \bar \alpha)(\tau - \tau_2)} \bigg] \int_0^{\tau'} g(\tau' - \tau'') e^{\bar c(2\pi + \bar \alpha) (\tau' - \tau'')} d\tau' d\tau'',
\end{align*}
where we have used the change of variable $\phi = (2\pi + \bar \alpha) \tau$. Recalling the definition of \gls{Gcal} and defining \newglossaryentry{Gfrak}{name=\ensuremath{\mathfrak G(\tau)},description={primitive of $\mathfrak g$}}
\begin{equation*}
\gls{Gfrak} = \int_0^\phi \mathcal G(\phi') d\phi' = (2\pi + \alpha) \int_0^\tau \mathcal G(\tau') d\tau' = (2\pi + \alpha) \int_0^\tau \int_0^{\tau'} g(\tau' - \tau'') e^{\bar c(2\pi + \bar \alpha)(\tau' - \tau'' + 1)} d\tau' d\tau'',
\end{equation*}
\begin{equation*}
\begin{split}
\mathfrak G(\phi) = \sum_{k=0}^\infty \frac{1}{\cos(\bar\alpha)^{k+1}} \Bigg\{& - (\phi - (2\pi + \bar\alpha) k) \\
& + \sum_{j = 0}^k \frac{1}{\cot(\bar\alpha)} \bigg[ -1 + \sum_{\ell = 0}^j (-1)^\ell e^{\cot(\bar\alpha) (\phi - (2\pi + \bar\alpha)k)} \frac{(\cot(\bar\alpha) (\phi - (2\pi +\bar \alpha) k))^\ell}{\ell!} \bigg] \Bigg\} \ind_{\phi \geq (2\pi + \bar\alpha) k}(\phi).
\end{split}
\end{equation*}
we obtain thus
\begin{align*}
\int_0^\phi \frac{M(\phi',\phi_1,\phi_2)}{\sin(\bar \alpha)}d\phi' = \big( \mathfrak G(\tau - \tau_1) - \mathfrak G(\tau - \tau_2) \big), \quad \tau_1 = \frac{\phi_1}{2\pi + \bar \alpha}, \ \tau_2 = \frac{\phi_2}{2\pi + \bar \alpha}.
\end{align*}
Finally by \eqref{Equa:eq_for_calG} we obtain
\begin{equation}
\label{Equa:eq_for_frakG}
\frac{d}{d\phi} \mathfrak G(\phi) = \cot(\bar \alpha) \mathfrak G(\phi) - \frac{\mathfrak G(\phi - 2\pi - \bar \alpha)}{\sin(\bar \alpha)} + \frac{\max\{0,\phi\}}{\sin(\bar \alpha)}.
\end{equation}

\section{Existence of optimal closing barriers at a given angle}
\label{S:exist_opti_traj}

One of the difficulties we face when studying this problem is that the optimization problem \eqref{eq:opt:problem} may be empty (which is actually what we want to prove for $\sigma \leq \bar \sigma$). Aim of this section is to give a different minimization problem (Definition \ref{Def:minimi_sect_4}) on the set of admissible spirals with the following property: if the minimum is $0$ or negative then there is a spiral which is blocking the fire spreading.

Let $Z$ be a given admissible spiral strategy, let $\bar L$ be a large constant and let \newglossaryentry{phi0}{name=\ensuremath{\phi_0},description={initial angle for the construction of the fastest/optimal closing spirals}} $\gls{phi0},\gls{phibar}$ be some rotation angles with $0 \leq {\phi_0} \leq {\bar \phi}$.

\begin{definition}
\label{Def:continuation_spirals}
The \emph{set of admissible continuations} \newglossaryentry{AcalZbarphi}{name=\ensuremath{\mathcal{A}_S(Z,{\phi_0})},description={set of admissible continutation of a spiral $Z$ from the angle ${\phi_0}$}} \gls{AcalZbarphi} is the set of admissible spirals that coincide with $Z$ up to the rotation angle ${\phi_0}$, can be prolonged to the rotation angle ${\bar \phi}$ and have length uniformly bounded by $\bar L$.
\end{definition}

\begin{remark}
\label{Rem:not_appearance_bar_L}
In the previous definition the parameters $\bar L,{\bar \phi}$ do not appear in the notation $\mathcal A_S(Z,{\phi_0})$ because, if there is a blocking spiral, then $\bar L$ and ${\bar \phi}$ can be chosen so that it belongs to $\mathcal A_S(Z,{\phi_0})$.

The parameter $\bar L$ is needed because we need some compactness when using the angle variable $\phi$ in the representation of the spiral: clearly $\bar L \to \infty$ suggests that we are not blocking the fire. An estimate of the parameter $L$ can be done by using the saturated spiral of Section \ref{Sss:equation_satur}, which is always admissible, but not optimal: in particular there is at least on $L$ such that the set $\mathcal A_S(Z,\phi_0)$ is not empty.

In Section \ref{S:optimal_sol_candidate} we will construct explicitly the optimal solution, thus obtaining also an estimate on $\bar L$: we remark that it will have a longer length than the saturated spiral. This spiral is optimal among all other admissible spirals, and then the parameter $L$ will not play any role in the minimization problem below.
\end{remark}

\begin{definition}[Minimization problem]
\label{Def:minimi_sect_4}
We consider the following minimization problem:
\begin{equation}
\label{eq:min:probl_pert}
\min \big\{ r_{\tilde Z}({\bar \phi}), \tilde Z \in \mathcal A_S(Z,{\phi_0}) \big\}, 
\end{equation}
where $(r_{\tilde Z}(\phi),\phi)$ is the angle-ray parametrization of the spiral $\tilde Z$.
\end{definition}

Aim of this section is to prove that there is an optimal spiral for the problem \eqref{eq:min:probl_pert} (for  some $L \gg 1$ fixed), and to construct the optimal solution in the case $\bar \phi < 2\pi + \phi_0 + \beta^-(\phi_0)$.

One of the main ingredient is a compactness result in $\mathcal A_S(Z,{\phi_0})$. This compactness is independent of ${\phi_0}$, in the sense that $\mathcal A_S(Z,\phi_0)$ is a closed subset of $\mathcal A_S$. Hence we state it for the latter: this is slightly more technical because of the presence of the initial segment, where the angle-ray parametrization degenerates. It can be clearly adapted to become an independent proof of the existence of an optimal admissible spiral.

\subsection{Study of the minimization problem}
\label{Ss:existence_minimum}

Given the control parameter $\beta : [{\phi_0},{\bar \phi}) \to (0,\frac{\pi}{2}]$ with the convexity condition \eqref{Equa:beta_convex_1}, the corresponding spiral is obtained by solving
\begin{equation}
\label{Equa:Dr+_dR-_again_1}
Dr = \cot(\beta) r \mathscr L^1 - Ds^-, \quad r(0) = r_0 \geq 1,
\end{equation}
where $s^-(\phi)$ is constructed according to Proposition \ref{Prop:construct_spiral}, i.e. at the previous round. The initial segment is $[(1,0),(r_0,0)]$, which we assume to belong to the initial given spiral arc: we will see that it is not the optimal solution.

The set where the functions $\beta(\phi)$ can be chosen is the intersection of the two sets
\begin{equation*}
\begin{split}
&\bigg\{ \beta : \forall \phi \in [0,{\bar \phi}] \bigg( \mathcal A(\phi) = 1 + s^-(\phi) + r(\phi) - \frac{s^+(\phi)}{\sigma} \geq 0 \bigg) \bigg\} \cap \bigg\{ \beta \in \bigg( 0,\frac{\pi}{2} \bigg] : r(\phi) > 0, D\beta + 1 \geq 0, r_0 - 1 + s^+(\phi) \leq \bar L \bigg\}.
\end{split}
\end{equation*}
The first set is the set of admissible spirals, while the second means that the solution to \eqref{Equa:Dr+_dR-_again_1} is a spiral, that the spiral is convex, and that its total length is less than $\bar L$. The condition
\begin{equation*}
\int_0^\epsilon \frac{1}{\beta(\phi)} d\phi < \infty
\end{equation*}
is hidden in the requirement that the total length is bounded and the curve arrives at angle ${\bar \phi} > 0$, see Proposition \ref{Prop:construct_spiral}.

\begin{proposition}
\label{Prop:compact_curves}
Consider a sequence of $\beta^n(\phi)$ and initial segments $[(0,0),(r_0^n,0)]$ such that
\begin{equation}
\label{Equa:new_one_sided_prop}
\beta^n \in \bigg(0, \frac{\pi}{2} \bigg], \quad D\beta^n + 1 \geq 0, \quad \int_0^\epsilon \frac{1}{\beta^n(\phi)} d\phi < \infty,
\end{equation}
and let $\zeta^n$ be the spiral constructed by Proposition \ref{Prop:construct_spiral} by solving \eqref{eq:spiral:angle}. Assume that $L(\zeta^n) \leq \bar L$ and for all $\epsilon > 0$ there exists $\delta > 0$ such that
\begin{equation}
\label{Equa:r_bounded_below}
r^n(\phi) \geq \delta \quad \forall \phi \in (\epsilon,{\bar \phi} - \epsilon).
\end{equation}
Then there is a subsequence (not relabeled) such that $\beta^n \to \beta$, $r^n \to r$ in $L^1(0,{\bar \phi})$ and $\zeta^n \to \zeta$ in $W^{1,1}(0,{\bar \phi})$, with $\zeta(0) - 1 + L(\zeta) \leq \bar L$.
\end{proposition}

The uniform positivity of $r^n$ implies the curves $\zeta^n$ are not converging to a curve which closes before the angle ${\bar \phi}$, thus making the angle-ray description meaningless. For the fire blocking problem this ${\bar \phi}$ could be every angle before the minimal angle for which there is a blocking spiral. 

\begin{proof}
We start with a series of compactness estimates. At each step we will extract the corresponding converging subsequence: we will often avoid relabeling the subsequences.

\medskip

\noindent{Step 1.} The total variation of $\beta^n$ is uniformly controlled by
\begin{equation*}
\TV(\beta^n,[0,{\bar \phi}]) \leq \frac{\pi}{2} + 2 {\bar \phi}.
\end{equation*}
This follows from the one-sided Lipschitz condition $D\beta^n + 1 \geq 0$ and $\beta \in [0,\pi/2]$. Hence, being $\beta^n \in [0,\frac{\pi}{2}]$ bounded, we can assume that $\beta^n \to \beta$ in $L^1(0,{\bar \phi})$: moreover the points where it is not convergent are the discontinuity points of $\beta$, which form a countable family.

\medskip

\noindent{\it Step 2.} Having finite length, it follows that the final points $(r^n_0,0)$ of the initial segment are compact, and then $r^n_0 \to r_0$. This is not in general the final point of the initial segment, because $\zeta^n$ may add additional intervals to this segment.

\medskip

\noindent{\it Step 3.} The unit vectors $\dot \zeta^n$ rotate monotonically about $0$ for a total angle $\leq {\bar \phi} + \frac{\pi}{2}$, so that the measure second derivative of $\zeta^n$ satisfies
\begin{equation*}
|D^2 \zeta^n| \leq {\bar \phi} + \frac{\pi}{2}.
\end{equation*}
Hence we can assume that $\zeta^n \to \zeta$ in $W^{1,1}(0,{\bar \phi})$, and also $D^2 \zeta^n \rightharpoonup D^2 \zeta$ weakly on $(0,{\bar \phi})$. 

\medskip

\noindent{\it Step 4.} Since $\dot \zeta^n \to \dot \zeta$ in $L^1(0,{\bar \phi})$, the monotone function $\phi \mapsto s^{-,n}(\phi)$ converges in $L^1(0,{\bar \phi})$ to the corresponding function $\phi \mapsto s^-(\phi)$: indeed $s^{-,n}(\phi) \to s^-(\phi)$ in every point $\phi$ where the latter is continuous. In particular $D s^{-,n} \rightharpoonup D s^-$ weakly in the interval $(0,{\bar \phi})$. Observe also that $\TV(s^{-,n},[0,{\bar \phi}]) \leq \bar L$.

\noindent{\it Step 5.} By the length formula \eqref{eq:var:length:1}, using the one-sided Lipschitz condition \eqref{Equa:new_one_sided_prop} for $\beta^n$ and the fact that $r^n(\phi)$ is uniformly positive for $\phi \in (\epsilon,{\bar \phi} -\epsilon)$ by \eqref{Equa:r_bounded_below} and bounded above by the $\bar L + 1$, we conclude that there exists $\epsilon' > 0$ such that $\beta^n(\phi) \geq \epsilon'$ for all $\phi \in (\epsilon,{\bar \phi} - \epsilon)$. Hence
\begin{equation*}
\frac{1}{\sin(\beta^n(\phi))} \to \frac{1}{\sin(\beta(\phi))} \quad \text{in} \ L^1(\epsilon,{\bar \phi} - \epsilon), \ \forall \epsilon > 0.
\end{equation*}
Since for $\phi \in [0,2\pi]$ it holds $r^n(\phi) \geq 1$, we also conclude that
\begin{equation*}
\frac{\mathscr L^1}{\sin(\beta^n)} \rightharpoonup c \Diracd_0 + \frac{\mathscr L^1}{\sin(\beta)} \quad \text{weakly in} \ [0,{\bar \phi} - \delta], \ \forall \delta > 0.
\end{equation*}
Observe also that
\begin{equation*}
\cot(\beta^n) \mathscr L^1 \rightharpoonup c \Diracd_0 + \cot(\beta) \mathscr L^1 \quad \text{weakly in} \ [0,{\bar \phi} - \delta], \ \forall \delta > 0,
\end{equation*}
because $\cot(\beta) \sim \frac{1}{\sin(\beta)}$ when $\beta \searrow 0$.

\medskip

\noindent{\it Step 6.} Consider now the ODE \eqref{eq:spiral:angle}: since
\begin{equation*}
D(r^n + s^{-,n}) = \cot(\beta^n) r^n \geq 0, \quad 1 \leq r^n \leq \bar L,
\end{equation*}
it follows that
%
\begin{equation*}
\TV(r^n,[0,{\bar \phi}]) \leq 2\bar L.
\end{equation*}
Hence also $r^n \to r$ in $L^1(0,{\bar \phi})$.

\medskip

\noindent{\it Step 7.} The ODE \eqref{eq:spiral:angle} passes to the limit as
\begin{equation*}
Dr = \cot(\beta) r \mathscr L^1 - Ds^-, \quad \phi \in (0,{\bar \phi}),
\end{equation*}
where we observed that $s^{-,n}(\phi) = 0$ for $\phi \in [0,2\pi)$. It remains to analyze the initial data, where a Dirac's delta can accumulate.

Using the estimates of Step 5 for $\phi \to 0$ and that $s^{-,n} = 0$ in the first round, we can compute
\begin{equation*}
r^n(\phi) = r^n_0 e^{\int_0^\phi \cot(\beta^n) \mathscr L^1}.
\end{equation*}
Passing to the limit we obtain
\begin{equation*}
r(\phi) = (r_0 e^c) e^{\int_0^\phi \cot(\beta) \mathscr L^1},
\end{equation*}
so that the initial segment for the spiral is $[(1,0),(r_0 e^c,0)]$.

\medskip

\noindent{\it Step 8.} Finally, we show that the limit curve $\zeta$ is represented by the limit $r(\phi)$. First one observes that from
\begin{equation*}
\frac{ds^{+,n}}{d\phi} = \frac{r^n(\phi)}{\sin(\beta^n(\phi))}
\end{equation*}
and the convergence of $\frac{\mathscr L^1}{\sin(\beta^n)}$ of Step 5 and of $r^n$, it follows that $s^{+,n}$ converges uniformly in $(\delta,{\bar \phi}-\delta)$.

One then passes to the limit
\begin{equation*}
\zeta^n(s^{+,n}(\phi)) = \zeta^n(s^{-,n}(\phi)) + r^n(\phi) e^{\i \phi},
\end{equation*}
obtaining that for $\mathscr L^1$-a.e. $\phi$
\begin{equation*}
\zeta(s^{+}(\phi)) = \zeta(s^{-}(\phi)) + r(\phi) e^{\i \phi}.
\end{equation*}
The remaining $\phi$ are obtained by taking the limit from the right (we want $r(\phi)$ to be right continuous) and the angle ${\bar \phi}$ is obtained by taking the limit $\phi \nearrow {\bar \phi}$: we are using $D\zeta = e^{\i \phi} D(s^- + r) + \i e^{i\phi} r$, which gives that the limits exist because $D(r + s^-) \in L^\infty$ (and the curve is Lipschitz).

\medskip

\noindent{\it Step 9.} In this last step we just verify the length estimate: again using the convergence estimates of Steps 5, 7 and 8 we obtain
\begin{equation*}
r_0^n - 1 + s^{+,n}(\phi) \quad \to \quad (r_0 e^c) - 1 + s^+(\phi),
\end{equation*}
so that the length bound depends continuously.
\end{proof}

We single out the convergences we obtained in the proof.

\begin{corollary}
\label{Cor:convergs}
Under the assumptions of the previous proposition, we can extract a subsequence $\beta^n$ such that:
\begin{enumerate}
\item the functions $\beta^n$, $r^n$, $s^{-,n}$ are uniformly BV and converge in $L^1(0,{\bar \phi})$ to $\beta,r,s^-$;
\item the functions $s^{+,n},r^n+s^{-,n}$ converges locally uniformly in $(0,{\bar \phi})$ to $s^+,r + s^-$;
\item the spirals $\zeta^n$ converges in $W^{1,1}(0,{\bar \phi})$ to a spiral $\zeta$, corresponding to the angle-ray representation $r$ with control angle $\beta$;
\item the time function $u$ is converging locally uniformly in the set $\R^2 \setminus \zeta([0,{\bar \phi}])$ if $r({\bar \phi}) > 0$, otherwise converges locally uniformly in the open set $\cup_\phi \zeta(s^-(\phi)) + e^{\i \phi} (0,r(\phi))$.
\end{enumerate}
Hence the admissibility functional $\mathcal A(\phi)$ for $\zeta^n$ converges to the admissibility functional for $\zeta$ for all $\phi > 0$.
\end{corollary}

In particular, if $\zeta^n$ are admissible, so it is $\zeta$.

\begin{proof}
The convergence of the admissibility functional follows from the uniform convergence of $s^{+,n},r^n + s^{-,n}$.

We now study the convergence of the time function $u$. By the uniform convergence of $\zeta^n$, it is enough to study the behavior of $u$ in each round, i.e. the convergence of the time function \eqref{Equa:u_ell_constr} when the region $S_\ell^n$ is bounded by a simple curve converging in $W^{1,1}$ (the case after the last round (where $u$ has to be computed outside $S_\ell^n$) can be easily be handled in a similar way). The weak convergence of $D^2 \zeta^{n}$ implies that the map
$$
x \mapsto s^n(x) \quad \text{such that} \quad \frac{x - \zeta^n(s^n(x))}{|x - \zeta^n(s^n(x))|} \in \partial^- \zeta^n(s^n(x)) \ \text{and} \ (\zeta^n)^{-1}((\zeta^n(s^n(x)),x)) = \emptyset
$$
converges pointwise, and then locally uniformly. By repeating this reasoning for the finite number of rounds we obtain the statement on $u$.
\end{proof}

We can now state the main theorem of this section.

\begin{theorem}
\label{Theo:exists_min}
Fix a rotation angle ${\bar \phi} \geq {\phi_0}$ and a length $\bar L$ sufficiently large. Then either there exists an admissible spiral in $\mathcal A_S(Z,{\phi_0})$ blocking the fire before or at ${\bar \phi}$ or there exists an admissible spiral $\zeta$ such that $r({\bar \phi})$ is minimal among all admissible spirals of length bounded by $\bar L$.
\end{theorem}

\begin{remark}
\label{Rem:lenght}
In the main theorem of the paper we will give an explicit form for the optimal spiral, which is independent on the parameter $\bar L$, so that in particular the length $\bar L$ where to search for the minimum can be explicitly estimated.
\end{remark}

\begin{proof}
If there is an admissible spiral blocking the fire before or at ${\bar \phi}$ there is nothing to prove. Otherwise consider the number
\begin{equation*}
\inf \Big\{ r({\bar \phi}) = \big| \zeta(s^+({\bar \phi})) - \zeta(s^-({\bar \phi})) \big|, \zeta \in \mathcal A_S(Z,{\phi_0}) \Big\}.
\end{equation*}
Notice that the set is not empty if $\bar L \gg 1$, since it contains the saturated spiral of the previous section, whose length is bounded by $e^{\cot(\alpha) \bar \phi}$.

If this number if $0$, then there are a sequence of admissible spirals $\zeta^n$ such that
\begin{equation*}
\zeta^n(s^{+,n}({\bar \phi})) - \zeta^n(s^{-,n}({\bar \phi})) \to 0.
\end{equation*}
Hence by the convergence of Corollary \ref{Cor:convergs} it follows that up to subsequences $\zeta^n$ converges to a spiral such that
\begin{equation*}
\lim_{\phi \nearrow {\phi_0}} \zeta(s^+({\bar \phi})) - \zeta(s^-({\bar \phi})) = 0,
\end{equation*}
so that it blocks the fire at ${\bar \phi}$.

The same reasoning proves that there exists a minimizer $\zeta$ if the minimum is strictly positive.
\end{proof}

In the next sections we will address this optimization problem, and construct explicitly the optimal solution for spiral arcs corresponding to the first round after $\bar \phi_0$.

\subsection{Study of the distance function \texorpdfstring{$u$}{u}}
\label{Ss:distance_function_u}

In this subsection we study more carefully the time function $u$: we recall that the time function $u$ is given by 
\begin{equation*}
u(x) = \min \big\{ L(\gamma), \gamma \ \text{admissible} \big\}.
\end{equation*}
In this subsection we will prove that $u$ is locally convex in $\R^2 \setminus Z$. This follows from the structure of $Z$, which is a connected curve. The normal to the level set is $\nabla u(x) = e^{\i \phi(x)}$, where $\phi(x)$ is the angle of the optimal ray passing through $x$. Moreover, the radius of curvature of the level set is $x - \zeta(s^-(\phi))$, i.e. the distance of the point $x$ from the last point where the optimal ray touches the spiral.

We will prove these properties in Lemma \ref{Lem:convexity} and Lemma \ref{lem:C:1:1}.

\begin{lemma}
\label{Lem:convexity}
The function $u$ is of class $C^{1,1}$ and locally convex in $\R^2 \setminus Z$. If
$$
x \in \zeta(s^-(\phi)) + e^{\i \phi} r(\phi),
$$
then
\begin{equation*}
\nabla u(x) = \frac{x - \zeta(s^-(\phi))}{|x - \zeta(s^-(\phi))|} = \mathbf t(x), \quad \nabla^2 u(x) = \frac{1}{|x- \zeta(s^-(\phi))|} \mathbf n(x) \otimes \mathbf n(x), \text{ a.e. } x,
\end{equation*}
with $\mathbf n(x) = \mathbf t(x)^\perp$, $\mathbf t(x) \in \partial^- \zeta(s^-(x))$.

Moreover $\nabla^2 u$ is locally BV with jumps on the directions where $Z$ is flat, corresponding to the points where the radius of curvature $|x-\zeta(s(x))|$ has a jump: more precisely
\begin{equation*}
\begin{split}
D \nabla^2 u(x) \mathbf t(x) &= \bigg[ \mathbf t(x) \cdot D \bigg( \frac{1}{|x - \zeta(s^-(\phi))|} \bigg) \bigg] \mathbf n(x) \otimes \n(x) = - \frac{\n(x) \otimes \n(x)}{|x - \zeta(s^-(\phi))|^2} \mathscr L^2,
\end{split}
\end{equation*}
\begin{equation*}
\begin{split}
D \nabla^2 u(x) \mathbf n(x) &= - \bigg[ \n(x) \cdot D \bigg( \frac{1}{|x - \zeta(s^-(\phi))|} \bigg) \bigg] \n(x) \otimes \n(x) - \frac{1}{|x - \zeta(s^-(\phi)|} \big( \mathbf t(x) \otimes \mathbf n(x) + \mathbf n(x) \otimes \mathbf t(x) \big) \mathscr L^2.
\end{split}
\end{equation*}
\end{lemma}

\begin{proof}
All statements follow from the computation of the formulas for $\nabla u$ and $\nabla^2 u$. Moreover, we can assume that $\zeta$ is smooth with $s \mapsto e^{\i \theta}$ smooth with Lipschitz inverse, and then pass to the limit, because of the convergence of $s^-(\phi)$ as in Corollary \ref{Cor:convergs}.

In this case the function $u$ is computed by the formula
\begin{equation*}
u(\zeta(s) + \dot \zeta(s) r) = 1 + s + r.
\end{equation*}
Differentiating one has
\begin{equation*}
\nabla u \cdot (\dot \zeta(s) + \ddot \zeta(s) r, \dot \zeta(s)) = (1,1) \quad \Rightarrow \quad \nabla u(\zeta(s) + \dot \zeta(s) r) = \dot \zeta(s),
\end{equation*}
which gives the first formula, and
\begin{equation*}
\nabla^2 u : \dot \zeta(s) \otimes (\dot \zeta(s) + \ddot \zeta(s) r) + \nabla u \ddot \zeta(s) = 0,
\end{equation*}
\begin{equation*}
\nabla^2 u : (\dot \zeta(s) + \ddot \zeta(s) r) \otimes (\dot \zeta(s) + \ddot \zeta(s) r) + \nabla u (\ddot \zeta(s) + \dddot \zeta(s) r) = 0,
\end{equation*}
which, together with
$$
\nabla^2 u : \dot \zeta(s) \otimes \dot \zeta(s) = 0, \quad \dot \zeta(s) \cdot \dddot \zeta(s) = - \frac{1}{R^2},
$$
gives the second formula, with $R$ local radius of curvature of $\zeta$.

Taking one further derivative one obtains
\begin{equation*}
\partial_r \nabla^2 u = - \frac{\n \otimes \n}{r^2}.
\end{equation*}
\begin{equation*}
\begin{split}
\partial_s \nabla^2 u &= \partial_s \frac{\n \otimes \n}{r} = - \frac{1}{r} \big( \ddot \zeta^\perp \otimes \mathbf n + \mathbf n \otimes \ddot \zeta^\perp \big),
\end{split}
\end{equation*}
We thus obtain
\begin{equation*}
\nabla^3 u \cdot \mathbf t = - \frac{\n \otimes \n}{r^2},
\end{equation*}
\begin{equation*}
\nabla^3 u \cdot \mathbf n = \frac{R}{r^3} \n \otimes \n - \frac{1}{r^2} (\mathbf t \otimes \n + \n \otimes \mathbf t).
\end{equation*}
Using that $\mathbf t = e^{\i \theta}$ and
\begin{equation*}
\frac{d\phi}{ds} = \frac{1}{R}, \quad R = \frac{ds}{d\phi},
\end{equation*}
\begin{equation*}
\begin{split}
\nabla \bigg( \frac{1}{|x - \zeta(s^-(\phi))|} \bigg) &= - \frac{\mathbf t(x)}{|x - \zeta(s^-(\phi))|^2} - \frac{R}{|x - \zeta(s^-(\phi))|^3} \mathbf n(x),
\end{split}
\end{equation*}
we get the formula in the statement, because when $\frac{ds}{d\theta}$ is singular then $\theta(s)$ is continuous and thus also both $\mathbf t,\n$ are continuous there.
\end{proof}

\begin{corollary}
\label{lem:C:1:1}
The level sets of the minimum time function $u$ are $\mathcal{C}^{1,1}$ curves, with second derivative BV.
\end{corollary}

\begin{proof}Let $x(\upsilon)$ be a parametrization by arc length of the level set $u = \text{const}$. Then differentiating $u \circ x(\upsilon) = \text{const}$, under the assumption that $\upsilon \mapsto e^{\i \theta}$ is Lipschitz with Lipschitz inverse we obtain
\begin{equation*}
\nabla u \cdot \dot {x}(\upsilon) = 0 \quad \Longrightarrow \quad \dot{x}(\upsilon) = \mathbf n,
\end{equation*}
\begin{equation*}
\nabla^2 u : \dot{x} \times \dot{x} + \nabla u \ddot{x} = 0 \quad \Longrightarrow \quad \ddot{x} = - \frac{\mathbf t}{|x - \zeta(s^-(\phi))|},
\end{equation*}
\begin{equation*}
\dddot{x} = \nabla \bigg( - \frac{\mathbf t}{|x - \zeta(s^-(\phi))|} \bigg) \mathbf n = \frac{R}{r^3} \mathbf t - \frac{\n}{r^2}, 
\end{equation*}
which gives the statement as in the previous proof because $\nabla^2 u$ is locally BV.
\end{proof}

\subsection{Construction of the curve of minimal reachable points for the first round}
\label{Ss:minimal_reach}

In this section we will construct the optimal solution to Theorem \ref{Theo:exists_min} in the following situation: the spiral $Z = \zeta([0,S]) \in \mathcal{A}_S$ is fixed up to a rotation angle ${\phi_0}$, and we will assume that the angle ${\bar \phi}$ in Theorem \ref{Theo:exists_min} belongs to the first round after $\phi_0$, i.e. it satisfies
$$
0 \leq {\bar \phi} - {\phi_0} \leq 2\pi + \beta^-({\phi_0}). 
$$
This constraint is easily explained: we optimize $r(\phi)$ whenever the choice of $\beta(\phi)$, $\phi \in ({\phi_0},{\bar \phi})$, is not having influence on the behavior of $Z$ afterwards. In other words, we do not need to care about the delay behavior of the ODE, since $s^-(\phi)$ is assigned by the fixed part of the spiral.

In this situation, the only constraints we have to abide are the monotonicity of $u \circ \zeta$, the admissibility and the convexity of the curve: these two constraints are sufficient to determine uniquely the solution. Indeed, the optimal solution will be made of at most these three parts, in sequence:
\begin{itemize}
\item  an arc of the level set of $u$ passing through $\zeta({\phi_0})$,
\item a segment,
\item an arc of a saturated spiral.
\end{itemize}
These components are connected by requiring to be tangent at the junction points, and determine a family of convex curves. The curve of the end points ${\bar \phi} \mapsto \hat \zeta = \zeta({\bar \phi})$ is the \emph{optimality curve}, i.e. for each angle ${\bar \phi}$ the value $r({\bar \phi})$ (i.e. the length of the last segment of the optimal ray) is the minimal value of $r$ in Theorem \ref{Theo:exists_min}, and the corresponding sequence of (level-set arc)-(segment)-(arc of saturated spiral) is the unique optimal barrier achieving the minimal $r({\bar \phi})$. The optimality curve determines a non-admissibility region, where no admissible barrier can be constructed, namely the region
$$
\Big\{ \zeta(s^-(\phi)) + (0,r(\bar \phi)) e^{\i \bar \phi}, \bar \phi \in \phi_0 + (0,2\pi + \beta^-(\phi_0)) \Big\}.
$$

\begin{definition}
\label{Def:minmi_reach_points}
The set \newglossaryentry{Rcalhatphi}{name=\ensuremath{\mathcal R({\bar \phi})},description={curve of minimal reachable points}} {\gls{Rcalhatphi} of minimal reachable points} (Figure \ref{fig:best:first:1}) is the set of end points of curves made of a piece of fire level set $\lbrace u=u(\zeta({\phi_0}))\rbrace$ plus a tangent segment to the level set: all the points on these curves should be admissible, should not cross $Z$ and the last point of these curves must be saturated (meaning that $\mathcal A(\mathcal R) = 0$).
\end{definition}

It is possible that none of the points on the fire level set \newglossaryentry{xilevel}{name=\ensuremath{\xi},description={parametrization of the level set $u^{-1}(\zeta({\phi_0}))$}} $\gls{xilevel} = u^{-1}(u(\zeta({\phi_0})))$ is saturated, which means that an admissible strategy is to construct a wall on the whole level set $u^{-1}(u(\zeta({\phi_0}))$: in this case there is no starting point of the curve $\mathcal{R}(\bar\zeta)$. We thus assume that there is a point $\xi(a)$, $a \geq 0$ is the arc length parametrization, such that
\begin{equation*}
u(\xi(a)) = u(\zeta({\phi_0})) = \frac{L(\zeta([0,{\phi_0}])) + L(\xi([0,\bar a]))}{\sigma} = \frac{L(\zeta([0,{\phi_0}])) + a}{\sigma},
\end{equation*}
which is the saturation condition. Clearly such a point is unique, see Figure \ref{fig:best:first:1}, and will be the starting point of the curve $\mathcal R$.

\begin{figure}
\centering
\resizebox{.75\textwidth}{!}{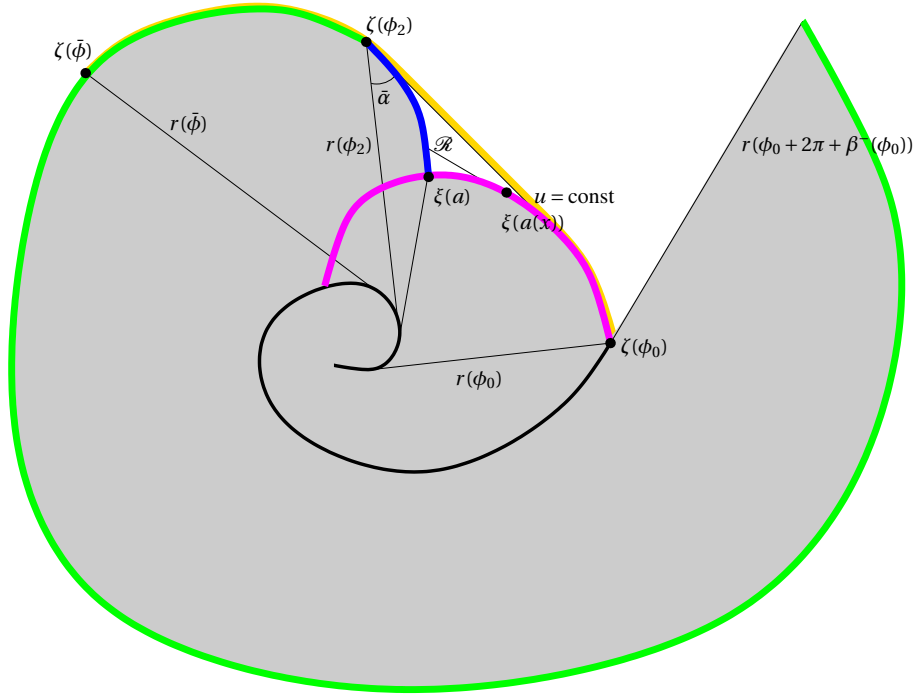}
\caption{The magenta line is the level set of the time function $u$ passing through $\zeta({\phi_0})$, the green line is the set $\mathcal{R}$ where all the points are saturated. It divides the region into the admissible and the non admissible part (gray). The yellow curve is the optimal solution for the angle $\bar \phi > \phi_2$: it is made by an arc of level set, a segment and an arc of the saturated spiral.}
\label{fig:best:first:1}
\end{figure}

To construct this curve $\mathcal R$, we will write the tangent in each point. Let $x$ be a point in $\mathcal R$, let $[\zeta(s^-(x)),x]$ be the last part (a segment) of the optimal ray (with \newglossaryentry{s(x)}{name=\ensuremath{s^-(x)},description={initial point of the last segment in an optimal ray}} $x \mapsto \gls{s(x)}$ being the map to the initial point of the last segment in the optimal ray), and $[\xi(a(x)),x]$ be the segment part of the admissible curve (again $a(x)$ is the point on $\xi$ whose tangent passes through $x$), with respective length $r,b$: more precisely, using the same notation for the subdifferential for the convex curve $\xi$,
$$
x = \zeta(s) + r \zeta'(s) = \xi(a) + b \xi'(a), \quad \zeta'(s) \in \partial^- \zeta(s), \ \xi'(a) \in \partial^- \xi(a).
$$
The first equalities express that $x$ belongs to the optimal ray and the half-line tangent to $\xi$, and the second equality is the tangency properties of the rays.

We will also denote with \newglossaryentry{uxi}{name=\ensuremath{u_\xi},description={time function for the level set $u^{-1}(\zeta({\phi_0}))$}} \gls{uxi} the solution to the Eikonal equation with barrier $\xi$ and value $u(\zeta(\phi_0)) = 0$, which is explicitly given by
\begin{equation*}
u_\xi(\xi(a) + b \xi'(a)) = a + b.
\end{equation*}
We recall also that the function $u$ is computed similarly as
\begin{equation*}
u(\zeta(s) + r \zeta'(s)) = s + r.
\end{equation*}
Define also the vectors \newglossaryentry{tbf(x)}{name=\ensuremath{\mathbf t(x)},description={gradient of the time function $u$}} \newglossaryentry{nbf(x)}{name=\ensuremath{\mathbf n(x)},description={vector $\mathbf n(x)$ rotated by $\frac{\pi}{2}$}} \newglossaryentry{tbfxi(x)}{name=\ensuremath{\mathbf t_\xi(x)},description={gradient of the function $u_\xi$}} \newglossaryentry{nbfxi(x)}{name=\ensuremath{\mathbf n_\xi(x)},description={vector $\mathbf n_\xi(x)$ rotated by $\frac{\pi}{2}$}}
\begin{equation*}
\gls{tbf(x)} = \nabla u(x) = \frac{x-\zeta(s)}{|x - \zeta(s)|}, \ \gls{nbf(x)} = \mathbf t^\perp, \quad \gls{tbfxi(x)} = \dot \xi(a) = \frac{x-\xi(a)}{|x - \xi(a)|}, \ \gls{nbfxi(x)} = \mathbf t^\perp_\xi.
\end{equation*}

We first consider the curve \newglossaryentry{Rtt}{name=\ensuremath{\mathtt R},description={curve where $u - \frac{u_\xi}{\sigma} = u(\zeta({\phi_0}))$}}
\begin{equation*}
\gls{Rtt} = \bigg\{ x : u(x) - \frac{u_\xi(x)}{\sigma} = \frac{L(\zeta(0,{\phi_0}))}{\sigma} \bigg\}. 
\end{equation*}
This curve coincides with $\mathcal R$ if all the points in the segment $\xi(a) + (0,b) \xi'(a)$ are all admissible. By definition the points in $\mathtt R$ are saturated, in the sense that the admissibility functional $\mathcal A$ is $0$.

\begin{proposition}
\label{Prop:der_calR}
The tangent curve $\mathtt R$ in $x \in \mathtt R$ is given by \newglossaryentry{rbftangRtt}{name=\ensuremath{\mathbf r(x)},description={tangent to the curve $\mathcal R$}}
\begin{equation*}
\gls{rbftangRtt} = \frac{\mathbf n(x) - \frac{\mathbf n_\xi(x)}{\sigma}}{|\mathbf n(x) - \frac{\mathbf n_\xi(x)}{\sigma}|},
\end{equation*}
oriented according to $u \circ \mathtt R,u_\xi \circ \mathtt R$ increasing.

With this orientation, the angle $\angle(\mathbf t,\mathbf r)$ is decreasing in the admissible region, and the last admissible point \newglossaryentry{xsatur}{name=\ensuremath{x_\mathrm{sl}},description={point on the curve $\mathtt R$ tangent to the saturated spiral}} \gls{xsatur} of $\mathtt R$ corresponds to the slope of the saturated spiral starting from that point (i.e. it forms an angle $\alpha$ with the optimal ray of $\zeta$). After this point the curve $\mathtt R$ is not admissible, in the sense that the segment $\xi(a) + (0,b) \xi'(a)$ has points outside the admissibility region (i.e. $\mathcal A < 0$ on some point of the segment $\xi(a) + b \xi'(a)$).

Finally, the curve is convex in the direction of $\mathbf r^\perp$ in the admissible part.
\end{proposition}

\begin{figure}
\centering
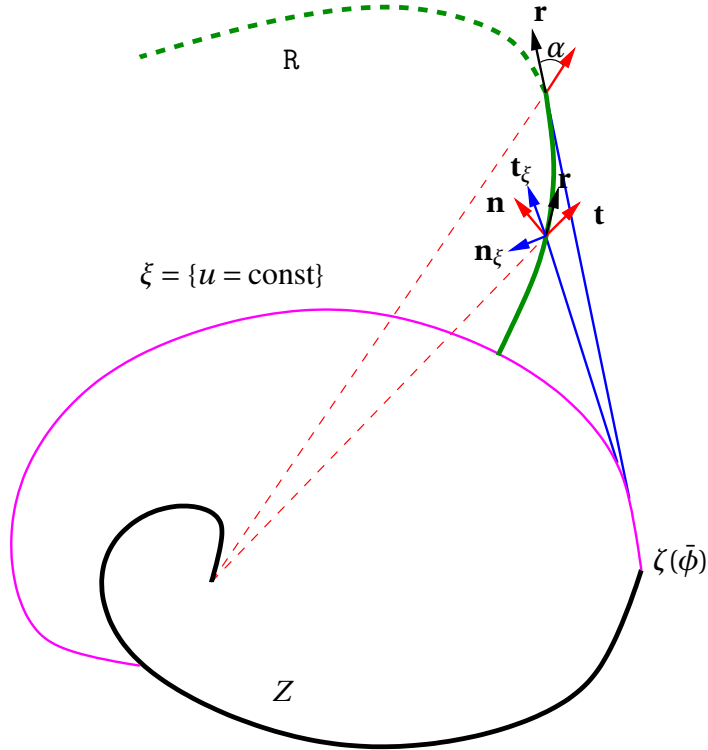
\caption{Construction of the curve $\mathtt R$: the curve is made of saturated points when the spiral is prolonged by an arc of the level set and the blue segment. The last admissibility point is when $\mathtt R$ is tangent to the saturated spiral, and after that point it is not admissible.}
\label{fig:non:adm}
\end{figure}

\begin{proof}
Using the explicit formula for the two solutions of the Eikonal equation with barrier $\zeta,\xi$,
\begin{equation*}
u(\zeta(s) + r \zeta'(s)) = s + r, \quad u_\xi(\xi(a) + b \xi'(a)) = a + b,
\end{equation*}
the admissibility functional takes the form
\begin{equation*}
\mathcal A(x) = s + r - \frac{a+b + L(\zeta(0,{\phi_0}))}{\sigma} = \text{const} + u(x) - \cos(\alpha) u_\xi(x) = 0.
\end{equation*}

Differentiating and using Lemma \ref{Lem:convexity} we obtain
\begin{equation*}
\nabla \mathcal A = \mathbf t - \cos( \alpha) \mathbf t_\xi,
\end{equation*}
and then, being the points of $\mathtt R$ saturated,
\begin{equation*}
\big( \nabla u - \cos( \alpha) \nabla u_\xi \big) \cdot \mathbf r = \big( \mathbf t - \cos( \alpha) \mathbf t_\xi \big) \cdot \mathbf r = 0,
\end{equation*}
with $\textbf r = \gls{rbftangRtt}$ the vector of the statement (the tangent to the curve $\mathtt R$), and this gives precisely the first formula in the statement up to the orientation. 

We compute the derivative of $u \circ \mathtt R$:
\begin{equation*}
(u \circ \mathtt R)' = - \cos(\alpha) \frac{(\mathbf t \cdot \mathbf n_\xi)}{|\mathbf n - \frac{\mathbf n_\xi}{\sigma}|}
\end{equation*}
which is positive because of the orientation of the vectors (see Figure \ref{fig:non:adm}). The case when $\mathbf t \perp \mathbf n_\xi$ means that the lines are parallel, and then $\mathtt R \to \infty$, which is impossible for $\sigma > 1$. Similarly,
\begin{equation*}
(u_\xi \circ \mathtt R)' = \frac{(\mathbf t_\xi \cdot \mathbf n)}{|\mathbf n - \mathbf n_\xi/\sigma|} = - \frac{(\mathbf t \cdot \mathbf n_\xi)}{|\mathbf n - \mathbf n_\xi/\sigma|} > 0,
\end{equation*}
as the saturation condition requires. This concludes the first part of the proposition.

In the starting point on the level set $u^{-1}(\zeta({\phi_0}))$ we obtain $\mathbf n_\xi = - \mathbf t$, so that the vector of the statement becomes
\begin{equation*}
\mathbf r = \frac{\mathbf n + \cos(\alpha) \mathbf t}{|\mathbf n + \cos(\alpha) \mathbf t|}.
\end{equation*}
In particular the initial angle $\angle (\mathbf{r},\mathbf{t})$ is greater than $\alpha$. In general, in the fire frame $(\mathbf t,\mathbf n)$ the position of $\mathbf n - \cos(\alpha) \mathbf n_\xi$ is contained in the circle $(0,1) + \{|x| = \cos(\alpha)\}$, and the minimal slope is when $\mathbf r = e^{\i \alpha}$. Writing the derivative of $(\nabla u \cdot \nabla u_\xi) \circ \mathtt R$ explicitly one gets by Lemma \ref{Lem:convexity}
\begin{equation*}
\begin{split}
((\mathbf \nabla u \cdot \nabla u_\xi) \circ \mathtt {R})' &= \mathbf r^T \frac{\mathbf n \otimes \mathbf n}{r} \mathbf t_\xi + \mathbf r^T \frac{\mathbf n_\xi \otimes \mathbf n_\xi}{b} \mathbf t \\
&= \frac{\mathbf t_\xi \cdot \mathbf n}{|\mathbf n - \cos(\alpha) \mathbf n_\xi|} \bigg( \frac{1 - \cos(\alpha) \mathbf n \cdot \mathbf n_\xi}{r} - \frac{\mathbf n \cdot \mathbf n_\xi - \cos(\alpha)}{b} \bigg) \\
&= \frac{\mathbf t_\xi \cdot \mathbf n}{|\mathbf n - \cos(\alpha) \mathbf n_\xi|} \bigg[ \bigg( \frac{1}{r} + \frac{\cos(\alpha)}{b} \bigg) - ((\mathbf \nabla u \cdot \nabla u_\xi) \circ \mathtt {R}) \bigg( \frac{\cos(\alpha)}{r} + \frac{1}{b} \bigg) \bigg].
\end{split}
\end{equation*}
The quantity
\begin{equation*}
\frac{1 - \cos(\alpha) \mathbf n \cdot \mathbf n_\xi}{r} - \frac{\mathbf n \cdot \mathbf n_\xi - \cos(\alpha)}{b}
\end{equation*}
is positive for $b \ll 1$, because $\n \cdot \n_\xi \searrow 0$ as we approach the starting point, and whenever this quantity is positive the vectors $\mathbf t,\mathbf t_\xi$ are approaching, which means that $\textbf{r} \cdot \textbf{t}$ decreases. Note that at the minimal angle
\begin{equation*}
((\nabla u \cdot \nabla u_\xi) \circ \mathtt R)' = \frac{1 - \cos^2(\alpha)}{r} > 0,
\end{equation*}
so this behavior is valid up to the last admissible point $x_\mathrm{sl}$. This proves the first part of the second statement of the proposition. In the point $x_\mathrm{sl}$ we have that $\mathbf{r}\parallel\mathbf{t}_{\xi}$. 

After the critical angle, the vector $\mathbf t_\xi$ moves counterclockwise having derivative $\propto \n_\xi (\n_\xi \cdot \n - \cos(\alpha))$, which means that we are going to take values of the parameter $a$ (the starting point of the segment from the barrier $u^{-1}(u(\zeta({\phi_0})))$) already used. 
However the gradient of the admissibility functional satisfies
\begin{equation*}
\nabla \mathcal A = \mathbf t - \cos( \alpha) \mathbf t_\xi,
\end{equation*}
and thus for directions with angle $> \alpha$
\begin{equation*}
\nabla \mathcal A \cdot \mathbf t_\xi = \mathbf t \cdot \mathbf t_\xi - \cos(\alpha) < 0,
\end{equation*}
because $\nabla u \cdot \nabla u_\xi$ is still increasing. Hence in this part of the curve the functional is not admissible. This concludes the second part of the proposition.

We are left to prove that $\mathtt R$ is convex. The second derivative of $\mathcal A$ along $\mathbf v$ gives
\begin{equation*}
\begin{split}
\mathbf r^T \nabla^2 \mathcal A \mathbf r &= \frac{(\mathbf n \cdot \mathbf r)^2}{r} - \cos(\alpha) \frac{(\mathbf n_\xi \cdot \mathbf r)^2}{b} \\
&= \frac{1}{|\mathbf n - \cos(\alpha) \mathbf n_\xi|^2} \bigg( \frac{(1 - \cos(\alpha) \mathbf n \cdot \mathbf n_\xi)^2}{r} - \frac{(\mathbf n \cdot \mathbf n_\xi - \cos(\alpha))^2}{b} \bigg) \\
&\geq \frac{1}{|\mathbf n - \cos(\alpha) \mathbf n_\xi|^2} \frac{\mathbf n \cdot \mathbf n_\xi - \cos(\alpha)}{b} \big( 1 - \cos(\alpha) \mathbf n \cdot \mathbf n_\xi - \mathbf n \cdot \mathbf n_\xi + \cos(\alpha) \big) > 0,
\end{split}
\end{equation*}
where in the last line we have used the estimate on $((\nabla u \cdot \nabla u_\xi) \circ \mathtt R)' > 0$ in the admissibility region. Hence the curve is convex in the direction of $\mathbf r^\perp$.
\end{proof}

Using the above proposition, we construct the optimal closing curve as follows, see Fig. \ref{Fi:three_cases}.
\begin{description}
\item[Case 1] the level set $u^{-1}(u(\zeta({\phi_0}))$ of the time function is always admissible, ending on a point of the spiral $Z$. \\
In this case for all angle ${\bar \phi} \geq {\phi_0}$ the optimal closing curve will be the level set, and the optimality curve $\mathcal R$ is the level set itself.

\item[Case 2] there are non admissible points on the level set: thus we can construct the curve $\mathtt R$, and this curve remains admissible until it touches again $Z$. \\
In this case, define the optimality curve as
\begin{equation*}
\mathcal R = \big\{ x: u(x) = u(\zeta({\phi_0})), \mathcal A(x) > 0 \big\} \cup \mathtt R.
\end{equation*}
For each angle ${\bar \phi} \geq {\phi_0}$ there is only one intersection of $\zeta(s^-({\bar \phi})) + r e^{\i {\bar \phi}}$ with the above optimality curve $\mathcal R$ (being $\mathtt R$ transversal to $\mathbf t$ by the above proposition). If the intersection occurs along the level set, then the optimal closing curve is the arc of the level set; if it occurs along the curve $\mathtt R$ in the point $\xi(a) + b \xi'(a)$, then the optimal closing curve is the arc $[\zeta(\phi_0),\xi(a)]$ of level set plus the segment $[\xi(a),\xi(a) + b \xi'(a)]$.

\item[Case 3] in this case the admissible part of the curve $\mathtt R$ ends in the final admissibile point $x_\mathrm{sl}$, where its tangent makes an angle $\alpha$ with the fire direction $\mathbf t$. \\
Define the optimality curve by
\begin{equation*}
\mathcal R = \big\{ x: u(x) = u(\zeta({\phi_0})), \mathcal A(x) > 0 \big\} \cup \big\{ x \in \mathtt R, \mathcal A(x) \geq 0 \big\} \cup \big\{ \text{saturated spiral starting from $x_\mathrm{sl}$} \big\}. 
\end{equation*}
For each angle ${\bar \phi} \geq {\phi_0}$, if the intersection of $\zeta(s^-({\bar \phi})) + r e^{\i {\bar \phi}}$ with $\mathcal R$ occurs on the arc of the saturated spiral starting from $x_\mathrm{sl}$, then the solution is the optimal solution as in Case 2 above up to $x_\mathrm{sl}$ (i.e. an arc of the level set plus a segment ending in $x_\mathrm{sl}$), and then the arc of saturated spiral ending in the angle ${\bar \phi}$. The other cases (i.e. when the intersection occurs before $\bar x$) are treated as in the previous point.
\end{description}

\begin{figure}
\resizebox{\textwidth}{!}{\input{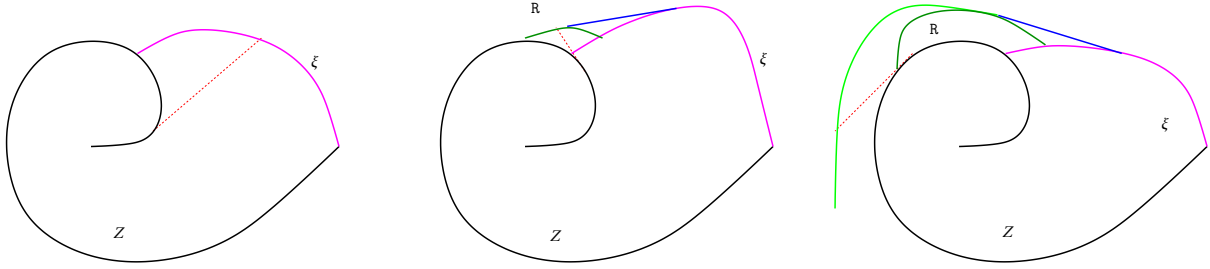}}
\caption{The three cases of Theorem \ref{Cor:curve_cal_R_sat_spiral}: in the first picture all level set $\xi$ is admissible, and then for every angle $\geq {\phi_0}$ the optimal solution is an arc of the level set; in the second case all curve $\mathtt R$ is admissible, and then the optimal solutions for points on $\mathtt R$ is an arc of $\xi$ and a segment; in the third case for large angles the minimal point lies on a saturated spiral, and the optimal solution is an arc of $\xi$, a segment and an arc of saturated spiral.}
\label{Fi:three_cases}
\end{figure}

\begin{theorem}
\label{Cor:curve_cal_R_sat_spiral}
The curves constructed in the three cases above are the unique solution to the minimum problem considered in Theorem \ref{Theo:exists_min} for ${\bar \phi} \in [{\phi_0},{\phi_0} + 2\pi + \beta^-({\phi_0}))$.
\end{theorem}

\begin{proof}
The proof is based on a contradiction: if an admissible spiral crosses the optimal solution given in the above Cases 1,2,3, then it is not admissible or not convex (Figure \ref{Fi:three_cases}).

Fix an angle ${\bar \phi}$, and let $\hat \zeta(\phi) = \zeta(s^-(\phi)) + \hat r(\phi) e^{\i \phi}$, $\phi \in [{\phi_0},{\bar \phi}]$ be the optimal solution candidate constructed as in Cases 1,2,3 above. Let $\check \zeta(\phi) = \zeta(s^-(\phi)) + \check r(\phi) e^{\i \phi}$ be another admissible spiral solution, and $\check \phi$ be the last angle for which $\hat r(\phi) \leq \check r(\phi)$ for $\phi \in [{\phi_0},\check \phi]$.

The convexity assumption of $\hat \zeta$ yields immediately that
\begin{equation*}
L(\hat \zeta([{\bar \phi},\check \phi])) \leq L(\check \zeta([{\bar \phi},\check \phi])),
\end{equation*}
and then this crossing cannot occur in the saturated part of $\hat \zeta$: indeed in this part $\beta(\phi) \leq \alpha$ is forced. By the monotonicity assumption on $u \circ \zeta$, the only part where $\check \zeta$ can cross $\hat \zeta$ is along the segment part of $\hat \zeta$.

The second observation is that $\check \zeta$ cannot cross $\mathtt R$ in the admissible region: indeed, it is fairly easy to see that the shortest path from $\zeta({\phi_0})$ to a point in $\mathtt R$ outside $\{u \leq u(\zeta({\phi_0}))\}$ is the solution constructed in Case 2 above. Hence after the crossing the spiral $\check \zeta$ would not be admissible by Proposition \ref{Prop:der_calR}.

Hence, the only situation to be studied is the following: the spiral $\check \zeta$ crosses the segment $(\xi(a),x_\mathrm{sl} = \xi(a) + b \xi'(a))$ in some point $\check x$. Considering the region delimited by $\{u = u(\zeta(\phi_0))\} \cup \mathcal R \cup [\xi(a),x_\mathrm{sl} = \xi(a) + b \xi'(a)]$, by convexity of $\hat \zeta$ and the above observation the curve $\check \zeta$ can only exit along the segment $(\xi(a),x_\mathrm{sl})$. This contradicts the convexity of $\check \zeta$. 
\end{proof}

\begin{remark}
\label{Rem:optima_for_others}
The argument of the above proof is that the shortest admissible convex curve for $\phi \in [{\phi_0},{\phi_0} + 2\pi + \beta^-({\phi_0}))$ is the one constructed in Cases 1,2,3. Since it is possible to prove that, in the range of angles considered here, the optimal closing barrier is convex, then the above statement is valid for generic simple curves such that $u$ is monotone along the curve.
\end{remark}

\begin{remark}
\label{Rem:other_angles}
It is important to notice that the time function $u$ used to compute the barrier is fixed, i.e. it does not depend on the barrier we are constructing after $\phi_0$: indeed the choice of $\phi \in [{\phi_0},{\phi_0} + 2\pi + \beta^-({\phi_0}))$ gives that we are only using the optimal rays leaving the spiral $\zeta([0,{\phi_0}])$ to construct the next arc of the spiral, and thus the rays are assigned. When we go above this angle, i.e. after one single round, then the choice of $\zeta$ at the first round impacts the behavior of $u$ in the next. We will indeed see that it is better to make $\zeta$ longer at the previous round in order to have more time at the next. This is easily explained by observing that if we increase the length of $\zeta$ by $L$ we gain a time $L$, i.e. we can construct a barrier with length $\sigma L > L$ at the next round.
\end{remark}

\section{A case study: the fastest saturated spiral}
\label{S:case:study}

In this section we consider as initial data for the saturated spiral the solution considered in Theorem \ref{Cor:curve_cal_R_sat_spiral}, namely an arc of the level set of $u$, a segment and the saturated spiral, where the initial fire set is the ball $\newglossaryentry{att}{name=\ensuremath{\mathtt a},description={initial radius for the case study of Section \ref{S:case:study}}} B_{\gls{att}}(0)$, $\gls{att} \in (0,1)$. The case $\mathtt a =1$ is the previous case of Section \ref{Sss:equation_satur}.

The aim of this analysis is twofold: on the one hand it is not clear if with this "optimized" initial data the spiral will close after some time, on the other hand the solution constructed here will be the fundamental block for the analysis of the general case. Since the spiral we construct is the spiral which becomes saturated before any other spiral, we call it the \emph{fastest saturated spiral}. We are not asking the question of its optimality here.

The construction speed we consider in this situation is the critical speed $\bar\sigma = \frac{1}{\cos(\bar \alpha)} = 2.6144..$ (Corollary \ref{Cor:eigen_angle}). The results proved here will be valid also for any other $\sigma \leq \bar\sigma$: more precisely, in the critical case the spirals diverges exponentially like $\phi e^{\lambda \phi}$, being $\lambda = \ln (\frac{2\pi + \bar 
\alpha}{\sin(\bar \alpha)}) = 0.27995...$ the only real eigenvalue of the RDE, and above the critical case it diverges like $e^{\lambda_1 \phi}$, being $\lambda_1$ the greater real eigenvalue. 

According to Cases 1,2,3 above Theorem \ref{Cor:curve_cal_R_sat_spiral} with $\hat \phi = 0$, the firefighter constructs a spiral $\mathtt Z_\mathtt a$ ($\mathtt a$ being the radius of the initial burning region $B_\mathtt a(0)$) starting from a point $P_0 = (1,0)$ with the following structure: it is the union of
\begin{enumerate}
\item a circle (the arc of the level set of $u$, in some cases consisting of a single point),
\item a segment of endpoints $P_1(\mathtt a),P_2(\mathtt a)$ tangent to the circle in $P_1(\mathtt a)$ such that the point $P_2(\mathtt a)$ is saturated,
\begin{equation*}
\mathcal{A}(P_2(\mathtt a)) = u(P_2(\mathtt a)) - \frac{1}{\bar\sigma} L(P_2(\mathtt a)) = 0,
\end{equation*}
where $L(P_2(\mathtt a))$ is the length of the spiral $\tilde Z_{\mathtt a}$ up to the point $P_2(\mathtt a)$,
\item a saturated spiral starting in $P_2(\mathtt a)$.
\end{enumerate}
The relevant quantities determining the exact structure of the spiral barrier are the following:
\begin{itemize}
\item the length of the arc of circle \newglossaryentry{Deltaphia}{name=\ensuremath{\Delta \phi_{\mathtt a}},description={opening angle of the initial starting in $(1,0)$}}
$$
\widearc{P_0 P_1(\mathtt a)} = \gls{Deltaphia},
$$
where $\Delta \phi_{\mathtt a}$ is the angle opening (remember that the level set $u$ passing by $(1,0)$ is $\partial B_1(0)$);
\item the length of the segment in the arc case (see below)
$$
|P_2(\mathtt a) - P_1(\mathtt a)| = \cot(\bar\alpha), \quad |P_2(\mathtt a)| = \frac{1}{\sin(\bar \alpha)|},
$$
where we have used that by convexity and admissibility the segment is tangent to the saturated spiral in $P_2(\mathtt a)$, i.e. $\angle(P_2(\mathtt a),P_2(\mathtt a) - P_1(\mathtt a)) = \bar \alpha$;
\item the saturation condition at $P_2(\mathtt a)$, which can be written now as
\begin{align*}
&u(P_2(\mathtt a)) - \frac{1}{\bar \sigma} L(P_2(\mathtt a)) = \left( \frac{1}{\sin(\bar \alpha)} - \mathtt a \right) - \frac{1}{\bar \sigma} \left( \Delta \phi_{\mathtt a} + \frac{\cos(\bar \alpha)}{\sin(\bar \alpha)} \right) = 0.
\end{align*}
\end{itemize}

Using $\bar \sigma = \frac{1}{\cos(\bar \alpha)}$, the last equation becomes
\begin{equation}
\label{eq:Delta:phi:a}
\Delta \phi_{\mathtt a} = \bar \sigma \left( \frac{1}{\sin(\bar \alpha)} - \mathtt a \right) - \frac{\cos(\bar \alpha)}{\sin(\bar \alpha)} = \tan(\bar \alpha) - \frac{\mathtt a}{\cos(\bar \alpha)}.
\end{equation}
Depending on the value of the initial radius $\mathtt a$ of the fire initial position, we have the following situations (see Fig. \ref{fig:two:cases}).

\begin{description}
\item[Arc case, $\mathtt a < \sin(\bar \alpha)$, i.e. $\Delta\phi_{\mathtt a} > 0$] then the angle at the origin corresponding to the segment $[P_1,P_2]$ is $\theta_{\mathtt a} = \frac{\pi}{2} - \bar\alpha$.
\item[Segment case, $\mathtt a \geq \sin(\bar \alpha)$, i.e. $\Delta \phi_{\mathtt a} \leq 0$] then $P_1(\mathtt a) = P_0 = (1,0)$, and the tangency and saturation conditions for $P_2(\mathtt a)$ give
\begin{equation*}
(|P_2(\mathtt a)| -  \mathtt a) - \frac{1}{\bar \sigma} |P_1(\mathtt a) - P_2(\mathtt a)| = \frac{\sin(\theta_{\mathtt a} + \bar \alpha)}{\sin(\bar \alpha)} -  \mathtt a - \frac{1}{\bar \sigma} \frac{\sin(\theta_{\mathtt a})}{\sin(\bar \alpha)} = 0.
\end{equation*}
Using again $\frac{1}{\bar \sigma} = \cos(\bar \alpha)$ we obtain
\begin{equation*}
\cos(\theta_{\mathtt a}) = \mathtt a, \quad \theta_{\mathtt a} \in \bigg[ 0 , \frac{\pi}{2} - \bar \alpha \bigg],
\end{equation*}
because $\mathtt a \in [\sin(\bar \alpha),1]$.
\end{description}

\begin{figure}
\centering
\includegraphics[scale=0.6]{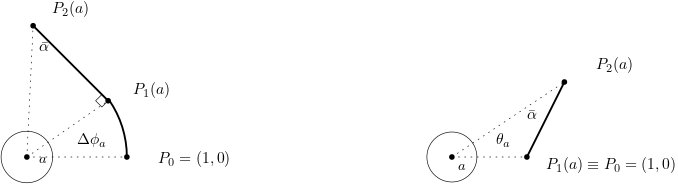}
\caption{The two cases with $ \mathtt a \in (0,\sin(\bar \alpha))$ and $ \mathtt a \in [\sin(\bar \alpha),1]$.}
\label{fig:two:cases}.
\end{figure}

The main result of this section is the following

\begin{theorem}
\label{thm:case:a}
For any value of $\mathtt a \in (0,1]$, the spiral \newglossaryentry{Stta}{name=\ensuremath{\tilde Z_{\mathtt a}},description={fastest saturated spiral, with the initial burning region $B_\mathtt{a}(0)$}} \gls{Stta} does not confine the fire: more precisely, the spiral diverges exponentially as
$$
\tilde r_\mathtt a(\phi) \sim 2 \phi e^{\bar c \phi} \begin{cases}
\frac{e^{-\bar c(\Delta \phi_{\mathtt a} + \frac{\pi}{2} - \bar \alpha)}}{\sin(\bar \alpha)} - e^{-\bar c(2\pi)} - \frac{e^{-\bar c(2\pi + \frac{\pi}{2})} - e^{- \bar c (2\pi + \frac{\pi}{2} + \Delta \phi_\mathtt a)}}{\bar c} - \cot(\bar \alpha) e^{-\bar c(2\pi + \frac{\pi}{2} + \Delta \phi_\mathtt a)} & 0 \leq \mathtt a < \sin(\bar \alpha), \\
\frac{\sin(\theta_\mathtt a + \bar \alpha) e^{-\bar c \theta_\mathtt a}}{\sin(\bar \alpha)} - e^{-\bar c(2\pi)} - \frac{\sin(\theta_\mathtt a) e^{-\bar c(2\pi + \theta_\mathtt a + \bar \alpha)}}{\sin(\bar \alpha)} & \sin(\bar \alpha) \leq \mathtt a \leq 1,
\end{cases} 
$$
being $\bar c = \ln(\frac{2\pi+\bar \alpha}{\sin(\bar \alpha)})$ given by Corollary \ref{Cor:eigen_angle}.
\end{theorem}

The proof of this theorem relies on a careful application of Lemma \ref{lem:key}.

\begin{remark}
\label{Rem:notation}
We will call the spiral $\tilde Z_\mathtt a$ as the \emph{fastest saturated spiral}: indeed it is the spiral which becomes saturated as fast as possible (i.e. at the smallest angle) and remains saturated from that angle onward. In the next section we consider the same problem when the initial configuration is an arc of a fixed barrier $\bar \zeta \rest_{\phi \in [0,\phi_0]}$. In both cases we will use the math accent \emph{tilde} will denote the fastest saturated spiral.
\end{remark}

\begin{figure}
\centering
\includegraphics[scale=0.6]{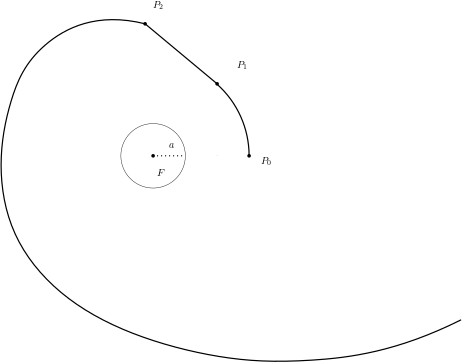}
\caption{Here $f_0$ is made by an arc, a segment and a saturated spiral.}
\label{fig:base:case:Arc}
\end{figure}

\begin{proof}[Proof of Theorem \ref{thm:case:a}]
We analyze separately the two cases $\Delta\phi_{\mathtt a} > 0$ and $\Delta\phi_{\mathtt a}\leq 0$.

\medskip

\noindent{\bf Case $\Delta\phi_{\mathtt a}>0$.} This case is described in Figure \ref{fig:base:case:Arc}. We remark that from \eqref{eq:Delta:phi:a} we find that $\Delta \phi_{\mathtt a} \in (0,\tan(\bar \alpha)]$, the largest possible arc-length being
\begin{equation*}
\Delta \phi_0 = \tan(\bar \alpha)
\end{equation*}
corresponding to $\mathtt a = 0$. We will denote by $\tilde r_{\mathtt a}$ the radius of the spiral $\tilde Z_{\mathtt a}$ in the fire coordinates of Theorem \ref{thm:param_1}, when we need to underline the dependence w.r.t. $\mathtt a$, otherwise we will use just $\tilde r$ to lighten the notation.

\medskip

\noindent{\it Step 1.} We compute the first rounds:
\begin{itemize}
\item for $\phi \in [0,\Delta \phi_{\mathtt a})$ we have
\begin{equation*}
\tilde r_{\mathtt a}(\phi) = 1;
\end{equation*}
\item for $\phi \in [\Delta \phi_{\mathtt a}, \Delta \phi_{\mathtt a} + \frac{\pi}{2} -\bar \alpha)$ we have that
\begin{equation*}
\tilde r_{\mathtt a}(\phi) = \frac{1}{\cos(\phi - \Delta\phi_{\mathtt a})},
\end{equation*}
in particular we call
\begin{equation*}
\kappa_1(\mathtt a) = \tilde r_{\mathtt a} \left( \Delta \phi_{\mathtt a} + \frac{\pi}{2} -\bar \alpha \right) = \frac{1}{\sin(\bar\alpha)}
\end{equation*}
the final distance for $0$;
\item for $\phi \in [\Delta \phi_{\mathtt a} + \frac{\pi}{2} - \bar \alpha, 2\pi)$ it holds
\begin{equation*}
\frac{d}{d\phi} \tilde r_{\mathtt a}(\phi) = \cot(\bar\alpha) \tilde r_{\mathtt a}(\phi),
\end{equation*}
since the spiral is saturated, giving
\begin{equation*}
\tilde r_{\mathtt a}(\phi) = \kappa_1(\mathtt a) e^{\cot(\bar\alpha)(\phi - \frac{\pi}{2} + \bar\alpha - \Delta \phi_{\mathtt a})};
\end{equation*}
\item the radius $\tilde r_{\mathtt a}$ has a jump at $\phi = 2\pi$,
\begin{equation*}
\tilde r_{\mathtt a}(2\pi)-\tilde r_{\mathtt a}(2\pi-)=-1,
\end{equation*}
in particular we denote by
\begin{equation*}
\kappa_2(\mathtt a) = \tilde r_{\mathtt a}(2\pi) = \kappa_1(\mathtt a) e^{\cot(\bar \alpha) (2\pi - \frac{\pi}{2} + \bar \alpha-\Delta \phi_{\mathtt a})}-1;
\end{equation*}
\item for $\phi \in [2\pi,2\pi+\frac{\pi}{2})$ it holds
\begin{equation*}
\frac{d}{d\phi} \tilde r_{\mathtt a}(\phi) = \cot(\bar\alpha) \tilde r_{\mathtt a}(\phi),
\end{equation*}
therefore
\begin{equation*}
\tilde r_{\mathtt a}(\phi) = \kappa_2(\mathtt a) e^{\cot(\bar\alpha)(\phi - 2\pi)} = \left[ \kappa_1(\mathtt a) e^{\cot(\bar \alpha) (2\pi - \frac{\pi}{2} + \bar \alpha - \Delta \phi_{\mathtt a})} - 1 \right] e^{\cot(\bar\alpha)(\phi - 2\pi)};
\end{equation*}
\item define
$$
\kappa_3(\mathtt a) = \tilde r_{\mathtt a} \left( 2\pi + \frac{\pi}{2} \right) = \kappa_2(\mathtt a) e^{\frac{\pi}{2}\cot(\bar\alpha)};
$$
\item for $\phi \in \left[ 2\pi + \frac{\pi}{2}, 2\pi + \frac{\pi}{2} + \Delta\phi_{\mathtt a} \right]$ the radius $\tilde r_{\mathtt a}(\phi)$ satisfies
\begin{equation*}
\frac{d}{d\phi} \tilde r_{\mathtt a}(\phi) = \cot(\bar\alpha) \tilde r_{\mathtt a}(\phi) - 1,
\end{equation*}
whose solution is
\begin{equation*}
\tilde r_{\mathtt a}(\phi) = k_3(\mathtt a) e^{\cot(\bar \alpha)(\phi - 2\pi - \frac{\pi}{2})} + \tan(\bar \alpha) \big( 1 - e^{\cot(\bar \alpha)(\phi - 2\pi - \frac{\pi}{2})} \big);
\end{equation*}
\item the radius $\tilde r_{\mathtt a}$ has a jump in $\phi = 2\pi + \frac{\pi}{2} + \Delta\phi_{\mathtt a}$: 
\begin{equation*}
\tilde r_{\mathtt a} \left( 2\pi + \frac{\pi}{2} + \Delta \phi_{\mathtt a} \right) -\tilde r_{\mathtt a} \left( 2\pi + \frac{\pi}{2} + \Delta \phi_{\mathtt a} - \right) = - \cot(\bar \alpha),
\end{equation*}
and also here we will introduce the quantity
\begin{equation*}
\kappa_4(\mathtt a) = \tilde r_{\mathtt a} \left(2\pi+\frac{\pi}{2}+\Delta\phi_{\mathtt a} \right) = e^{\cot(\bar \alpha) \Delta \phi_{\mathtt a}} \kappa_3(\mathtt a) + \tan(\bar \alpha) \left( 1 - e^{\cot(\bar \alpha)\Delta \phi_{\mathtt a}} \right) - \cot(\bar \alpha);
\end{equation*}
\item finally, for $\phi \geq 2\pi + \frac{\pi}{2} + \Delta \phi$ the function solves
\begin{equation*}
\frac{d}{d\phi} \tilde r_{\mathtt a}(\phi) = \cot(\bar \alpha) \tilde r_{\mathtt a}(\phi) -\frac{\tilde r_{\mathtt a}(\phi - 2\pi - \bar \alpha)}{\sin(\bar \alpha)}.
\end{equation*}
\end{itemize}

Using the kernel $G(\phi),M(\phi,\phi_1,\phi_2)$ of Remark \ref{Rem:oringal_phi_kernel}, the solution can be written for $\phi \geq \Delta \phi_\mathtt a + \frac{\pi}{2} - \bar \alpha$ as
\begin{equation}
\label{Equa:expl_sol_sat_or}
\tilde r_{\mathtt a}(\phi) = \frac{1}{\sin(\bar \alpha)} G(\phi - \bar \phi_0(\mathtt a)) - G(\phi - 2\pi) - M(\phi,\bar \phi_1(\mathtt a),\bar \phi_2(\mathtt a)) - \cot(\bar \alpha) G(\phi - \phi_0(\mathtt a) - 2\pi - \bar \alpha),
\end{equation}
where the angles are given by
\begin{equation*}
\bar \phi_0(\mathtt a) = \Delta \phi_\mathtt a + \frac{\pi}{2} - \bar \alpha, \quad \bar \phi_1(\mathtt a) = 2\pi + \frac{\pi}{2}, \quad \bar \phi_2(\mathtt a) = \bar \phi_0(\mathtt a) + 2\pi + \bar \alpha = \Delta \phi_a + 2\pi + \frac{\pi}{2}.
\end{equation*}
We can translate the solution by $\varphi = \phi - \bar \phi_0(\mathtt a)$ obtaining
\begin{equation}
\label{Equa:expl_sol_sat}
\tilde r_{\mathtt a}(\varphi) = \frac{1}{\sin(\bar \alpha)} G(\varphi) - G \bigg( \varphi + \frac{\pi}{2} - \bar \varphi_1(\mathtt a) \bigg) - M(\varphi,\bar \varphi_1(\mathtt a),2\pi + \bar \alpha) - \cot(\bar \alpha) G(\varphi - 2\pi - \bar \alpha),
\end{equation}
with
\begin{equation*}
\bar \varphi_1(\mathtt a) = 2\pi + \frac{\pi}{2} - \bar \varphi_0 = 2\pi + \bar \alpha - \Delta \phi_\mathtt a.
\end{equation*}

\medskip

\noindent{\it Step 2.}
The next step is to optimize the solution w.r.t. $\mathtt a$: this is done by taking the derivative w.r.t. $\mathtt a$ of the solution and using the linearity of the Delay ODE. We use the relations \eqref{Equa:ker_der_G}, \eqref{Equa:der_M_phi1} and
\begin{equation*}
\frac{d}{d\mathtt a} \bar \varphi_1 = \frac{1}{\cos(\bar \alpha)} \quad \text{because of (\ref{eq:Delta:phi:a})},
\end{equation*}
we obtain that
\begin{align*}
\partial_{\mathtt a} \tilde r_{\mathtt a}(\varphi) &= \frac{1}{\cos(\bar \alpha)} \bigg[ - \Diracd_{- \frac{\pi}{2} + \bar \varphi_1} - \cot(\bar \alpha) G \bigg( \varphi + \frac{\pi}{2} - \bar \varphi_1 \bigg) + \frac{G \big( \varphi + \frac{\pi}{2} - \bar \varphi_1 - 2\pi - \bar \alpha) \big)}{\sin(\bar \alpha)} + G(\varphi - \bar \varphi_1) \bigg].
\end{align*}
We have thus to study the function
\begin{equation*}
\check r(\psi) = \cot(\bar \alpha) G(\psi) - \frac{G(\psi - 2\pi - \bar \alpha)}{\sin(\bar \alpha)} - G \bigg( \psi - \frac{\pi}{2} \bigg), \quad \psi = \varphi + \frac{\pi}{2} - \bar \varphi_1,
\end{equation*}
which is the solution to the RDE for the saturated spiral
\begin{equation*}
\frac{d}{d\psi} \check r(\psi) = \cot(\bar \alpha) \check r(\psi) - \frac{\check r(\psi- 2\pi - \bar \alpha)}{\sin(\bar \alpha)} + \cot(\bar \alpha) \Diracd_0 - \Diracd_{\frac{\pi}{2}} - \frac{\Diracd_{2\pi + \bar \alpha}}{\sin(\bar \alpha)}. 
\end{equation*}
In terms the function $\check \rho(\tau) = \check r((2\pi+\bar \alpha) \tau) e^{- \bar c (2\pi + \bar \alpha) \tau}$ the RDE becomes
\begin{equation}
\label{Equa:ckrho_delph_casest}
\check \rho(\tau) = \cot(\bar \alpha) g(\tau) - \frac{e^{-\bar c (2\pi + \bar \alpha)} g(\tau - 1)}{\sin(\bar \alpha)} - e^{-\bar c \frac{\pi}{2}} g \bigg( \tau - \frac{\frac{\pi}{2}}{2\pi + \bar \alpha} \bigg),
\end{equation}
which is the solution to the RDE for the saturated spiral
\begin{equation*}
\frac{d}{d\psi} \check \rho(\tau) = \check \rho(\tau) - \check \rho(\tau-1) + \cot(\bar \alpha) \Diracd_0 - e^{-\bar c \frac{\pi}{2}} \Diracd_{\frac{\pi}{2(2\pi+\bar \alpha)}} - \frac{e^{-\bar c(2\pi+\bar \alpha)} \Diracd_{1}}{\sin(\bar \alpha)}. 
\end{equation*}

\begin{figure}
\includegraphics[scale=.6]{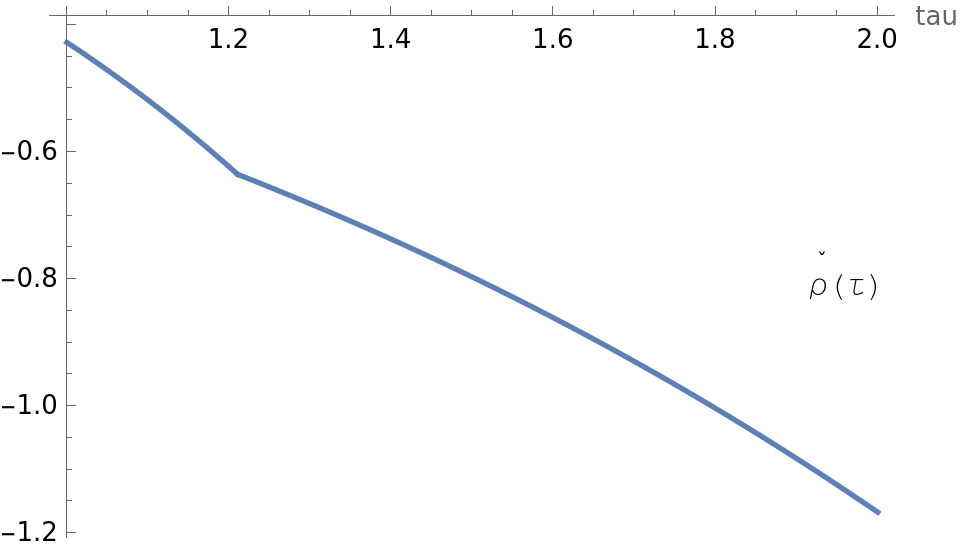}\hfill\includegraphics[scale=.6]{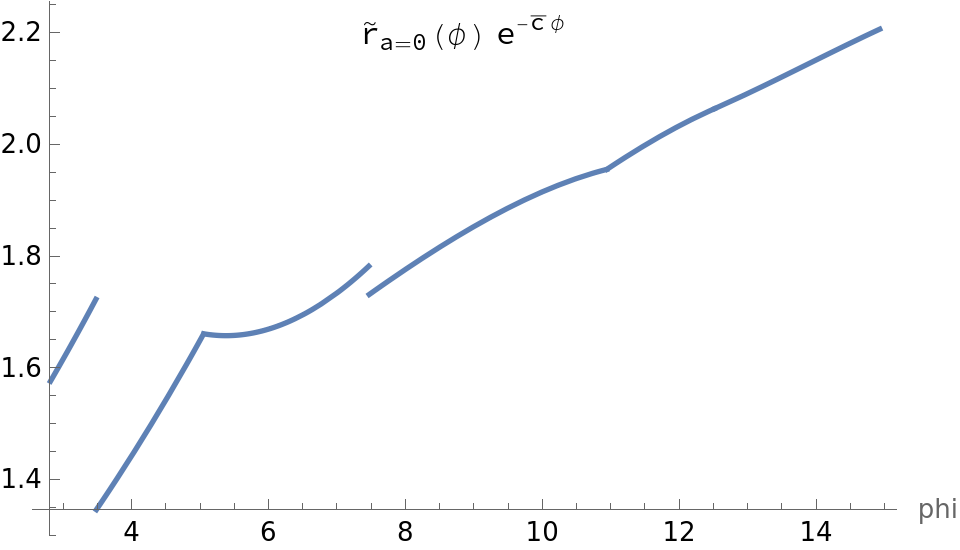}
\caption{Plot of the function \eqref{Equa:ckrho_delph_casest} for $\tau \in [1,2]$ on the left and of the first two rounds of the function $\tilde r_{\mathtt a = 0}(\phi)$ (Equation \eqref{Equa:expl_sol_sat_or}) multiplied by $e^{-\bar c \phi}$.}
\label{Fig:ckrho_delph_casst}
\end{figure}

A numerical plot of the solution for $\tau \in [1,2]$ is in Fig. \ref{Fig:ckrho_delph_casst}: we notice that it is strictly negative together with its derivative, and from Lemma \ref{lem:key} we deduce that $\check \rho < 0$ for all $\tau \geq 1$, corresponding to the angle
\begin{equation*}
\phi = 4\pi + \bar \alpha, \quad \text{because of the map} \ \tau \mapsto \phi = (2\pi + \bar \alpha) \tau + 2\pi.
\end{equation*}
For the first two rounds, a numerical plot of $\tilde r_{\mathtt a}(\phi)$ as function for $(\phi,\Delta \phi) \in [\Delta \phi + \frac{\pi}{2} - \bar \alpha] \times [0,\tan(\bar \alpha)]$ shows that the solution is strictly positive (and actually decreasing w.r.t. $\Delta \phi_\mathtt a$, i.e. increasing w.r.t. $\mathtt a$), see Fig. \ref{Fig:ckrho_delph_casst}.

\medskip


\noindent{\it Step 3.} From now on we will assume $\mathtt a = 0$, therefore $\Delta \phi_0 = \tan(\bar \alpha)$: a numerical plot of the function $\tilde r_{\mathtt a = 0}(\phi) e^{-\bar c \phi}$ is in Fig. \ref{Fig:ckrho_delph_casst}. We observe that the function is strictly positive and increasing in an interval larger than $2\pi+\bar \alpha$, and then by Lemma \ref{lem:key} we deduce that the solution is positive for all $\phi$.

\medskip

\noindent{\it Case $\Delta\phi_{\mathtt a}\leq 0$.} This case is described in Figure \ref{fig:base:case:185}. We will prove that the worst case is for $\mathtt a = \sin(\bar \alpha)$, which corresponds to $\Delta \phi_{\mathtt a} = 0$ in the previous analysis: hence again the worst case is for $\mathtt a = 0$, as shown above, and the spiral does not close.

The approach is completely similar, the only variation being the missing initial arc.

{\it Step 1.} We compute $\tilde r_{\mathtt a}(\phi)$ explicitly for the first rounds. We call $\bar \theta_{\mathtt a} = \theta_{\mathtt a} + \bar\alpha \in [0,\frac{\pi}{2}]$ as the slope of the segment $[P_1(\mathtt a),P_2(\mathtt a)]$:

\begin{figure}
\centering
\includegraphics[scale=0.75]{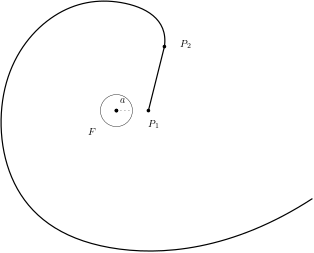}
\caption{Segment case for the fastest saturated spiral.}
\label{fig:base:case:185}
\end{figure}

\begin{itemize}
\item for $\phi \in[0,\bar \theta_{\mathtt a} - \bar \alpha]$ we have that
\begin{equation*}
\tilde r_{\mathtt a}(\phi) = \frac{\sin(\bar \theta_{\mathtt a})}{\sin(\bar \theta_{\mathtt a} - \phi)},
\end{equation*}
and we call
\begin{equation*}
\kappa'_1(\mathtt a) = \tilde r_{\mathtt a}(\bar \theta_{\mathtt a} - \bar \alpha) = \frac{\sin(\bar \theta_{\mathtt a})}{\sin(\bar \alpha)};
\end{equation*}
\item for $\phi \in [\bar \theta_{\mathtt a} - \bar \alpha, 2\pi)$ it holds
\begin{equation*}
\frac{d}{d\phi} \tilde r_{\mathtt a}(\phi) = \cot(\bar \alpha) \tilde r_{\mathtt a}(\phi),
\end{equation*}
which gives
\begin{equation*}
\tilde r_{\mathtt a}(\phi) = \kappa'_1(\mathtt a) e^{\cot(\bar \alpha) (\phi - \bar \theta_{\mathtt a} + \bar \alpha)};
\end{equation*}
\item the function $\tilde r_{\mathtt a}(\phi)$ has a jump at $\phi = 2\pi$, 
\begin{equation*}
\tilde r_{\mathtt a}(2\pi) - \tilde r_{\mathtt a}(2\pi-) = - 1,
\end{equation*}
and we set
$$
\kappa'_2(\mathtt a) = \tilde r_{\mathtt a}(2\pi) = \frac{\sin(\bar \theta_{\mathtt a})}{\sin(\bar \alpha)} e^{\cot(\bar \alpha) (2\pi - \bar \theta_{\mathtt a} + \bar \alpha)} - 1;
$$
\item for $\phi \in [2\pi,2\pi + \bar \theta_{\mathtt a})$ it holds
\begin{equation*}
\frac{d}{d\phi} \tilde r_{\mathtt a}(\phi) = \cot(\bar\alpha) \tilde r_{\mathtt a}(\phi),
\end{equation*}
so that
\begin{equation*}
\tilde r_{\mathtt a}(\phi) = \kappa'_2(\mathtt a) e^{\cot(\bar\alpha) (\phi - 2\pi)};
\end{equation*}
\item in $\phi=2\pi+\bar\theta_{\mathtt a}$ the function $\tilde r_{\mathtt a}$ has a jump, computed as
\begin{equation*}
\tilde r_{\mathtt a}(2\pi + \bar \theta_{\mathtt a}) - \tilde r_{\mathtt a}(2\pi + \bar \theta_{\mathtt a}-) = - \frac{\sin(\bar\theta_{\mathtt a} - \bar \alpha)}{\sin(\bar \alpha)},
\end{equation*}
and we call
$$
\kappa_3'(\mathtt a) = \tilde r_{\mathtt a}(2\pi + \bar \theta_{\mathtt a}) = \kappa'_2(\mathtt a) e^{\cot(\bar \alpha) \bar \theta_{\mathtt a}} - \frac{\sin(\bar \theta_{\mathtt a} - \bar \alpha)}{\sin(\bar \alpha)};
$$
\item finally, for $\phi \geq 2\pi+\bar\theta_{\mathtt a}$ the function solves
\begin{equation*}
\frac{d}{d\phi} \tilde r_{\mathtt a}(\phi) = \cot(\bar \alpha) \tilde r_{\mathtt a}(\phi) - \frac{\tilde r_{\mathtt a}(\phi - 2\pi - \bar \alpha)}{\sin(\bar\alpha)}.
\end{equation*}
\end{itemize}

As before, by using the kernel $G(\phi)$ of Remark \ref{Rem:oringal_phi_kernel}, the solution can be written for $\phi \geq \bar \theta_\mathtt a - \bar \alpha$ as
\begin{equation*}
\tilde r_{\mathtt a}(\phi) = \frac{\sin(\bar \theta_\mathtt a)}{\sin(\bar \alpha)} G(\phi - \bar \phi_0(\mathtt a)) - G(\phi - 2\pi) - \frac{\sin(\bar \theta_\mathtt a - \bar \alpha)}{\sin(\bar \alpha)} G(\phi - \bar \phi_0(\mathtt a) - 2\pi - \bar \alpha),
\end{equation*}
where the angle $\bar \phi_0(\mathtt a)$ given by
\begin{equation*}
\bar \phi_0(\mathtt a) = \bar \theta_\mathtt a - \bar \alpha.
\end{equation*}

\medskip

\noindent{\it Step 2.} \label{Page:step_2_case_study_segment} Taking the derivative w.r.t. $\bar \theta_\mathtt a$ and using \eqref{Equa:ker_der_G} we get
\begin{align*}
\check r_\mathtt a(\phi) &= \frac{d}{d\bar \theta_\mathtt a} \tilde r_{\mathtt a}(\phi) \\
&= \bigg( \frac{\cos(\theta_\mathtt a)}{\sin(\bar \alpha)} - \frac{\sin(\theta_\mathtt a)}{\sin(\bar \alpha)} \cot(\bar \alpha) \bigg) G(\phi - \phi_0) \\
& \quad + \bigg( \frac{\sin(\theta_\mathtt a)}{\sin(\bar \alpha)^2} - \frac{\cos(\theta_\mathtt a - \bar \alpha)}{\sin(\bar \alpha)} + \frac{\sin(\theta_\mathtt a - \bar \alpha)}{\sin(\bar \alpha)} \cot(\bar \alpha) \bigg) G(\phi - \phi_0 - (2\pi + \bar \alpha)) \\
& \quad  - \frac{\sin(\theta_\mathtt a - \bar \alpha)}{\sin(\bar \alpha)^2} G(\phi - \phi_0 - 2(2\pi + \bar \alpha)) \\
&= - \frac{\sin(\theta_\mathtt a - \bar \alpha)}{\sin(\bar \alpha)^2} \Big( G(\phi - \phi_0) - 2 \cos(\bar \alpha) G(\phi - \phi_0 - (2\pi + \bar \alpha)) + G(\phi - \phi_0 - 2(2\pi + \bar \alpha)) \Big).
\end{align*}
The latter function, when rescaled by $\rho(\tau) = \check r_\mathtt a(\phi_0 + (2\pi + \bar \alpha) \tau) e^{-\bar c(2\pi + \bar \alpha) \tau}$ is proportional to
\begin{equation}
\label{Equa:case_study_thetader2}
g(\tau) - 2 \cos(\bar \alpha) e^{- \bar c (2\pi +\bar \alpha)} g(\tau - 1) + e^{- 2 \bar c (2 \pi + \bar \alpha)} g(\tau - 2),
\end{equation}
The numerical plot (Fig. \ref{Fig:ckrho_delph_case_tg}) shows that it is positive and strictly increasing, hence by Lemma \ref{lem:key} we conclude that $\frac{d\tilde r_\mathtt a}{d\bar \theta_\mathtt a} < 0$, so the worst case is $\bar \theta_\mathtt a = \frac{\pi}{2}$, which corresponds to the arc case when $\Delta \phi_\mathtt a = 0$.

\begin{figure}
\includegraphics[scale=.5]{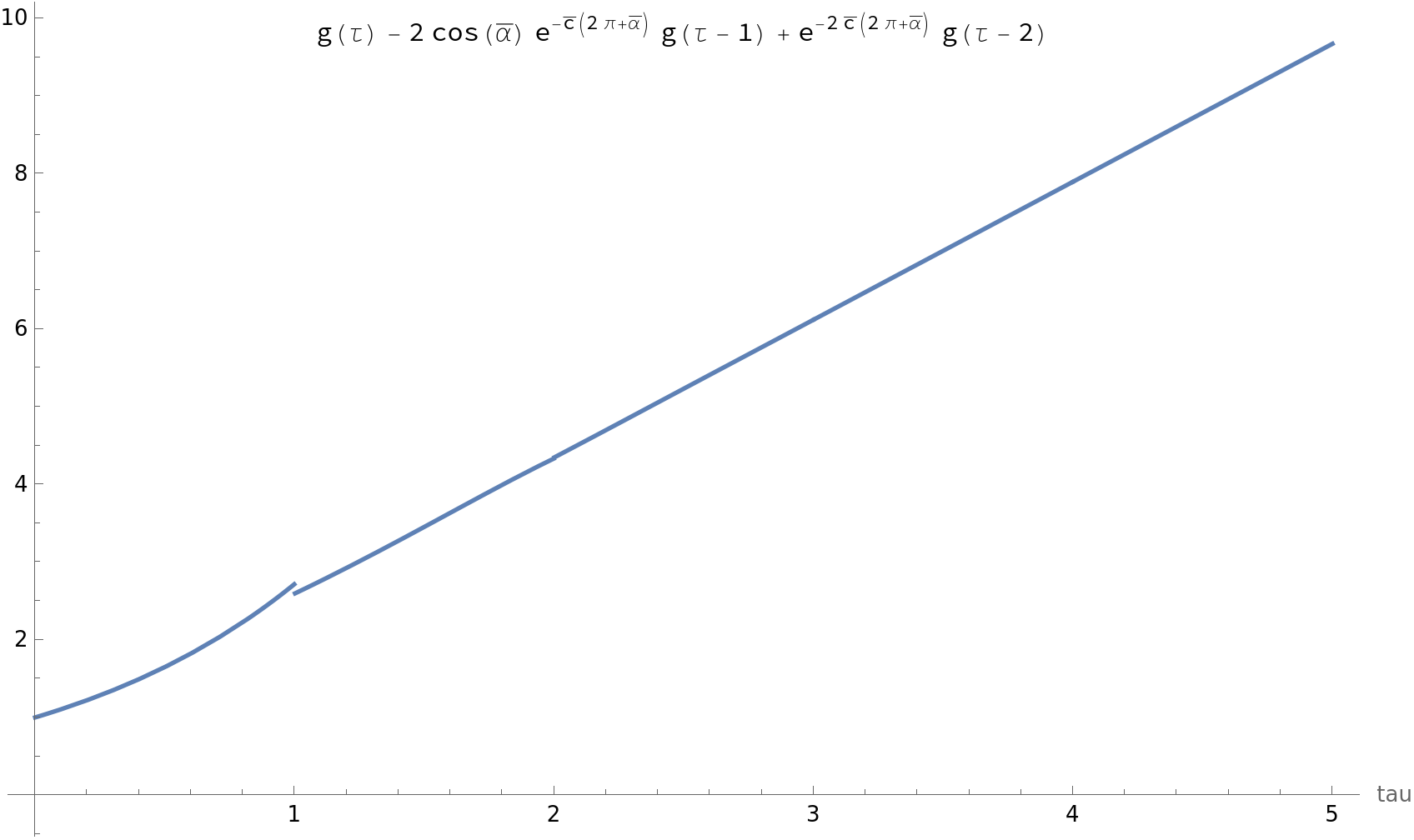}
\caption{Plot of the function \eqref{Equa:case_study_thetader2}.}
\label{Fig:ckrho_delph_case_tg}
\end{figure}

\medskip

\noindent{\it Asymptotic behavior.} For $\mathtt a \in [0,\sin(\bar \alpha)]$ the solution is given by \eqref{Equa:expl_sol_sat_or}:
\begin{equation*}
\tilde r_{\mathtt a}(\phi) = \frac{1}{\sin(\bar \alpha)} G(\phi - \bar \phi_0(\mathtt a)) - G(\phi - 2\pi) - M(\phi,\bar \phi_1(\mathtt a),\bar \phi_2(\mathtt a)) - \cot(\bar \alpha) G(\phi - \phi_0(\mathtt a) - 2\pi - \bar \alpha).
\end{equation*}
Using the asymptotic expansions of Corollary \ref{Cor:crit_G_M}, we obtain that the coefficient in front of the principal term $\phi e^{\bar c \phi}$ is
\begin{align*}
\frac{2 e^{-\bar c \bar \phi_0}}{\sin(\bar \alpha)} - 2 e^{-\bar c (2\pi)} - \int_{\bar \phi_1}^{\bar \phi_2} 2 e^{- \bar c \phi_0} d\phi_0 - 2 \cot(\bar \alpha) e^{-\bar c \bar \phi_2},
\end{align*}
which is the first case of the asymptotic expansion. The other case is completely similar.
\end{proof}

\subsection{An estimate on the length of the fastest saturated spiral}
\label{Ss:length_spiral_est}

We conclude this section with an estimate on the length $L$ of the fastest saturated spiral when $\mathtt a = 0$. Writing the solution \eqref{Equa:expl_sol_sat} with $\Delta \phi_\mathtt a = \tan(\bar \alpha)$ in the variable
\begin{equation*}
\rho(\tau) = \tilde r_{\mathtt a = 0}(\phi_0 + (2\pi + \bar \alpha) \tau) e^{- \bar c (2\pi + \bar \alpha) \tau}, \quad \phi_0(0) = \tan(\bar \alpha) + \frac{\pi}{2} - \bar \alpha,
\end{equation*}
we obtain
\begin{equation}
\label{Equa:rhobase_01}
\rho(\tau) = \frac{1}{\sin(\bar \alpha)} g(\tau) - e^{-\bar c(2\pi + \bar \alpha - \frac{\pi}{2} - \tan(\bar \alpha))} g(\tau - \tau_1) - m(\tau,\tau_2,1) - \cot(\bar \alpha) e^{-\bar c(2\pi + \bar \alpha)} g(\tau - 1),
\end{equation}
where
\begin{equation*}
\tau_1 = 1 - \frac{\frac{\pi}{2} + \tan(\bar \alpha)}{2\pi + \bar \alpha}, \quad \tau_2 = 1 - \frac{\tan(\bar \alpha)}{2\pi + \bar \alpha}.
\end{equation*}
Using the definitions of \gls{gfrak} and \gls{Gfrak} and that
$$
\tilde L_{\mathtt a = 0} = \tan(\bar \alpha) + \cot(\bar \alpha) + \int_{\phi_0}^\phi \frac{\tilde r_{\mathtt a=0}(\phi')}{\sin(\bar \alpha)} d\phi', \quad \phi \geq \phi_0,
$$
we thus have that the length of the spiral is
\begin{equation}
\label{Equa:Lbase_01}
\tilde L_{\mathtt a = 0}(\tau) = \begin{cases}
0 & \tau < - \frac{\frac{\pi}{2} - \bar \alpha + \tan(\bar \alpha)}{2\pi + \bar \alpha}, \\
(2\pi + \bar \alpha) \tau + \frac{\pi}{2} - \bar \alpha + \tan(\bar \alpha) & - \frac{\frac{\pi}{2} - \bar \alpha + \tan(\bar \alpha)}{2\pi + \bar \alpha} \leq \tau < - \frac{\frac{\pi}{2} - \bar \alpha}{2\pi + \bar \alpha}, \\
\tan(\bar \alpha) + \cot(\bar \alpha - (2\pi + \bar \alpha) \tau) & - \frac{\frac{\pi}{2} - \bar \alpha}{2\pi + \bar \alpha} \leq \tau < 0, \\
\begin{aligned}
&\tan(\bar \alpha) + \cot(\bar \alpha) + \frac{\mathfrak g(\tau)}{\sin(\bar \alpha)} - \mathfrak g(\tau - \tau_1) \\
& \quad - (\mathfrak G(\tau - \tau_2) - \mathfrak G(\tau - 1)) - \cot(\bar \alpha) \mathfrak g(\tau - 1)
\end{aligned} & \tau \geq 0.
\end{cases}
\end{equation}
The first terms take into account the initial arc and the initial segment.

\begin{proposition}
\label{Prop:length_compar}
It holds for $\phi \geq \tan(\bar \alpha) + \frac{\pi}{2} - \bar \alpha$
\begin{equation*}
\tilde r_{\mathtt a = 0}(\phi) - 2.08 \tilde L_{\mathtt a = 0}(\phi - 2\pi - \bar \alpha) \geq 0.67 e^{\bar c(\phi - \tan(\bar \alpha) - \frac{\pi}{2} + \bar \alpha)}.
\end{equation*}
\end{proposition}

\begin{figure}
\begin{subfigure}{.475\textwidth}
\resizebox{\textwidth}{!}{\includegraphics{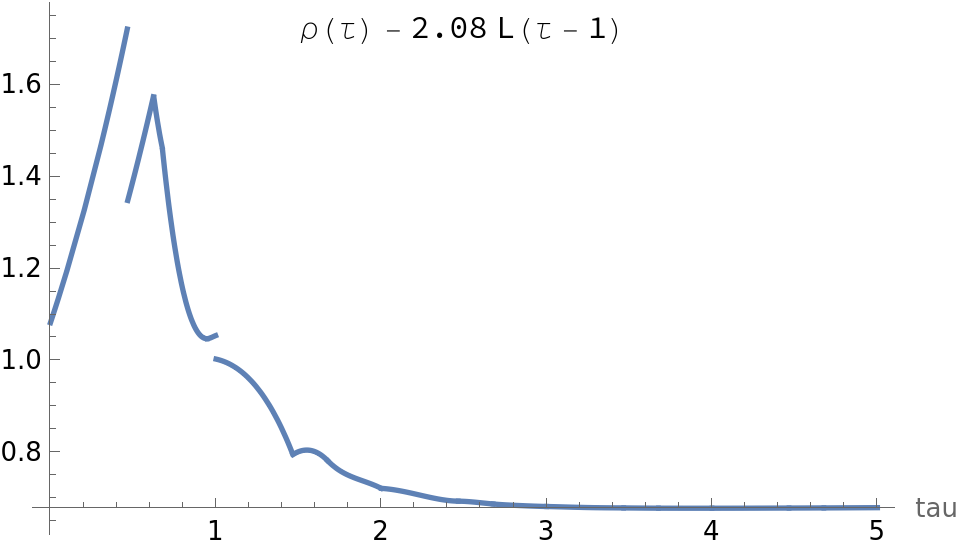}}
\end{subfigure} \hfill
\begin{subfigure}{.475\textwidth}
\resizebox{\textwidth}{!}{\includegraphics{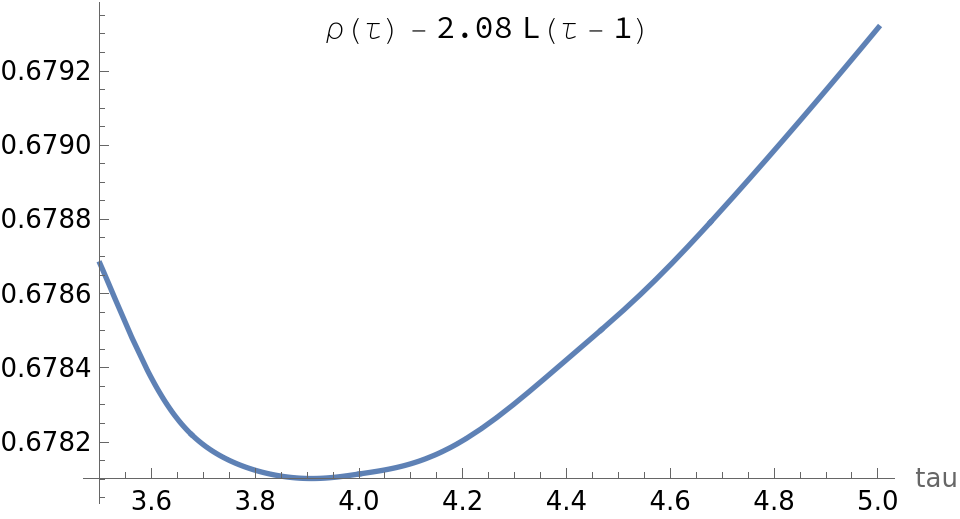}}
\end{subfigure}
\caption{Numerical computations of the function \eqref{Equa:rhobase_Lbase_comp}, with a blow-up about the minimum.}
\label{Fig:rhobase_Lbase_comp}
\end{figure}

\begin{proof}
We numerically compute the function
\begin{equation}
\label{Equa:rhobase_Lbase_comp}
\tau \mapsto \rho(\tau) - 2.08 e^{-\bar c (2\pi + \bar \alpha) \tau} \tilde L_{\mathtt a = 0}(\tau - 1),
\end{equation}
which is the rescaling of the function in the statement: Fig. \ref{Fig:rhobase_Lbase_comp} shows that the function has a minimum $>0.67$ for $\tau \in [3,4]$ and it is increasing for $\tau \in [4,5]$.

In order to apply Lemma \ref{lem:key}, we write the equation for $r - 2.08 L$. By using the equations for $\mathcal G, \mathfrak G$, Equations \eqref{Equa:eq_for_calG} and \eqref{Equa:eq_for_frakG}, we have that the $\phi$-derivative of $L(\phi)$ for $\phi \geq \phi_0$ is
\begin{align*}
\frac{d}{d\phi} \tilde L_{\mathtt a = 0}(\phi) &= \cot(\bar \alpha) \tilde L_{\mathtt a = 0}(\phi) - \frac{\tilde L_{\mathtt a = 0}(\phi - 2\pi - \bar \alpha)}{\sin(\bar \alpha)} \\
& \quad + \big( \tan(\bar \alpha) + \cot(\bar \alpha) \big) \delta_{\frac{\pi}{2} - \bar \alpha + \tan(\bar \alpha)} + \frac{H(\phi - \tan(\bar \alpha) - \frac{\pi}{2} + \bar \alpha)}{\sin(\bar \alpha)} - H(\phi - 2\pi) \\
& \quad - \frac{\max\{0,\phi - \frac{5\pi}{2}\} - \max\{0,\phi - \frac{5\pi}{2} - \tan(\bar \alpha)\}}{\sin(\bar \alpha)} - \cot(\bar \alpha) H \bigg( \phi - \frac{5\pi}{2} - \tan(\bar \alpha) \bigg).
\end{align*}
Recall that \gls{Heavyside} is the Heaviside function. Hence for $\phi \geq \frac{5\pi}{2} + \tan(\bar \alpha)$ we have \newglossaryentry{SL}{name=\ensuremath{S_L},description={source term for the length $L$ of a spiral}}
\begin{equation*}
\frac{d}{d\phi} \tilde L_{\mathtt a = 0}(\phi) = \cot(\bar \alpha) \tilde L_{\mathtt a = 0}(\phi) - \frac{\tilde L_{\mathtt a = 0}(\phi - 2\pi - \bar \alpha)}{\sin(\bar \alpha)} + \gls{SL},
\end{equation*}
where
\begin{equation}
\label{Equa:SL_const}
\gls{SL} = \frac{1}{\sin(\bar \alpha)} - 1 - \tan(\bar \alpha) - \cot(\bar \alpha) = - 2.7473 < 0.
\end{equation}
Hence for $\phi \geq \tan(\bar \alpha) + \frac{\pi}{2} - \bar \alpha + 2(2\pi + \bar \alpha)$ we have
\begin{align*}
\frac{d}{d\phi} (\tilde r_{\mathtt a = 0}(\phi) - 2.08 \tilde L_{\mathtt a = 0}(\phi - 2\pi - \bar \alpha)) &= \cot(\bar \alpha) (\tilde r_{\mathtt a = 0}(\phi) - 2.08 \tilde L_{\mathtt a = 0}(\phi - 2\pi - \bar \alpha)) \\
& \quad - \frac{\tilde r_{\mathtt a = 0}(\phi - 2\pi - \bar \alpha) - 2.08 \tilde L_{\mathtt a = 0}(\phi - 2(2\pi + \bar \alpha))}{\sin(\bar \alpha)} - S_L \\
&\geq \cot(\bar \alpha) (\tilde r_{\mathtt a = 0}(\phi) - 2.08 \tilde L_{\mathtt a = 0}(\phi - 2\pi - \bar \alpha)) \\
& \quad - \frac{\tilde r_{\mathtt a = 0}(\phi - 2\pi - \bar \alpha) - 2.08 \tilde L_{\mathtt a = 0}(\phi - 2(2\pi + \bar \alpha))}{\sin(\bar \alpha)}.
\end{align*}
Hence in terms of $\tau$ we obtain
\begin{equation*}
\begin{split}
\frac{d}{d\tau} (\rho(\tau) - 2.08 e^{-\bar c(2\pi + \bar \alpha) \tau} \tilde L_{\mathtt a = 0}(\tau-1)) &\geq (\rho(\tau) - 2.08 e^{-\bar c(2\pi + \bar \alpha) \tau} \tilde L_{\mathtt a = 0}(\tau-1)) \\
& \quad - (\rho(\tau-1) - 2.08 e^{-\bar c(2\pi + \bar \alpha) (\tau-1)} \tilde L_{\mathtt a = 0}(\tau-2)),
\end{split}
\end{equation*}
so that Lemma \ref{lem:key} and Remark \ref{rmk:exponentially:explod} can be applied.
\end{proof}

\begin{remark}
\label{Rem:asymtp_consta_for_comp}
The asymptotic constant for the length should be
\begin{equation}
\label{Equa:exact_Lbaserho}
\frac{\tilde r_{\mathtt a = 0}(\phi)}{\tilde L_{\mathtt a = 0}(\phi-2\pi-\bar \alpha)} \sim e^{\bar c} \bar c \sin(\bar \alpha) = 2.08884,
\end{equation}
obtained by using $\tilde r_{\mathtt a = 0}(\phi) \sim \phi e^{\bar c \phi}$. However for this exact value we cannot use Lemma \ref{lem:key}, obviously.
\end{remark}


\section{A differentiable path of spirals}
\label{S:family}

Let $\bar \zeta(s)$, $s \in [0,L]$ with $L \gg 1$, be a given admissible spiral. In this section we construct a family of admissible spirals \newglossaryentry{zetatildefam}{name=\ensuremath{\tilde \zeta(s;s_0)},description={family of admissible spirals, parametrized by $s_0$}} \gls{zetatildefam}, $s \in [0,L]$, with the following properties:
\begin{enumerate}
\item for $s \in [0,s_0]$ it holds $\tilde \zeta(s;s_0) = \bar \zeta(s)$;
\item in the first round starting from $\bar \zeta(s_0)$, the spiral $\tilde \zeta(s;s_0)$ is the spiral given by Theorem \ref{Cor:curve_cal_R_sat_spiral};
\item after the first round, the spiral $\tilde \zeta(s;s_0)$ is the saturated spiral.
\end{enumerate}
We call each spiral \newglossaryentry{Ztildes0}{name=\ensuremath{\tilde Z(s_0)},description={fastest saturated spiral starting from $s_0$}} $\gls{Ztildes0} = \tilde \zeta([0,L];s_0)$ the \emph{fastest saturated spiral starting from $s_0$}. The map $s_0 \mapsto \tilde Z(s_0)$ defines a path in the space of admissible spirals, connecting the fastest saturated spiral of the previous section (corresponding to $\tilde \zeta(s;0) = \tilde \zeta_{\mathtt a = 0}(s)$) with the given spiral $\bar Z = \bar \zeta([0,s])$.

If we denote by \newglossaryentry{phi0initangle}{name=\ensuremath{\phi_0 = (s^+)^{-1}(s_0)},description={starting angle for the admissible family of spiral $\tilde \zeta(s,;s_0)$}} \gls{phi0initangle} the starting angle of the above construction, then the first round is determined by
\begin{equation*}
\phi_0 \leq \phi \leq \phi_0 + \beta^-(\phi_0),
\end{equation*}
In the following we will write both $\tilde \zeta(\cdot;s_0),\tilde r(\cdot;s_0)$ and $\tilde \zeta(\cdot;\phi_0),\tilde r(\cdot;\phi_0)$ to denote this admissible family of spirals, and the argument $\cdot$ can be either the arc length $s$ or the rotation angle $\phi$.

It is immediate to deduce the following lemma.

\begin{lemma}
\label{Lem:tilde_zeta_prop}
The family of spirals $\tilde \zeta(s;s_0)$ satisfies the following properties:
\begin{enumerate}
\item $\tilde \zeta(\cdot;L) = \bar \zeta(\cdot)$;
\item $\tilde \zeta(s;0)$ is the fastest saturated spiral constructed in Section \ref{S:case:study};
\item if $s$ belongs to the first round starting from $\bar \zeta(s_0)$, then the solution is given by Theorem \ref{Cor:curve_cal_R_sat_spiral};
\item after the first round, denoting with \newglossaryentry{rtildefam}{name=\ensuremath{\tilde r(\phi;s_0)},description={family of admissible spirals in the rotation angle coordinate $\phi$, parametrized by $s_0$}} \gls{rtildefam} the angle-ray representation, we have
\begin{equation*}
\frac{d}{d\phi} \tilde r(\phi;s_0) = \cot(\bar \alpha) \tilde r(\phi;s_0) - \frac{\tilde r(\phi - 2\pi + \bar \alpha;s_0)}{\sin(\bar \alpha)}.
\end{equation*}
\end{enumerate}
\end{lemma}

Since the construction of Theorem \ref{Cor:curve_cal_R_sat_spiral} is explicit, the last point above yields that the family is completely determined, up to computing the solution in the first round after $s_0$ and solving the saturated RDE.

The main result of this section is to prove that this family is $C^1$-w.r.t. the parameter $s_0$: more precisely, we will show that the limit of the incremental ratio \newglossaryentry{deltartildefam}{name=\ensuremath{\delta \tilde r(\phi;s_0)},description={derivative of the family $\tilde r(s;s_0)$ w.r.t. $s_0$}}
\begin{equation*}
\gls{deltartildefam} := \lim_{\delta s_0 \searrow 0} \frac{\tilde r(\phi;s_0 + \delta s_0) - \tilde r(\phi;s_0)}{\delta s_0}
\end{equation*}
exists outside a finite set of angles $\phi$, and it satisfies a precise RDE, which allows the explicit computation of the $\delta \tilde r$. The dependence on the base spiral $\bar \zeta$ appears only in two real parameters describing the initial data of $\delta \tilde r$, allowing to write its explicit form with computations similar to Section \ref{S:case:study}: more precisely, the form of the RDE satisfied by the perturbation is very similar to the RDE satisfied by the solution $r(\phi)$ or its derivative $\frac{d}{d\mathtt a} r(\phi)$ considered in Section \ref{S:case:study}, and will be studied with similar techniques. In particular, depending whether the non-saturated part of the fastest closing spiral at the first round is a segment or is an arc+segment, the derivative takes different form, and we will study them separately in two different subsections, Subsections \ref{S:segment} and \ref{S:arc}: these derivatives coincide at the transition point between segment and arc.

\begin{remark}
\label{Rem:segmet_sub_arc0}
We remark that the segment case can be recovered from the arc case when the initial level set is just one point, see Remark \ref{Rem:same_as_segment}: however we believe that it is simpler and very instructive to perform the analysis in the segment case, so that we present it in the next subsection.

We anticipate also that out that $\delta \tilde r(\phi;s_0)$ is not positive in an initial interval of the second round, and thus it cannot be optimal for some angles. Hence to prove Theorem \ref{Theo:main_extended} we will need an additional perturbation of the last round: we address this fact in Section \ref{S:optimal_sol_candidate}.
\end{remark}

Since in the last round we have already the explicit form of the optimal closing solution (Theorem \ref{Cor:curve_cal_R_sat_spiral}), we can assume that $\zeta$ makes at least one round after $\phi_0$, so that there will always be a saturated part.

\subsection{Analysis of \texorpdfstring{$\tilde \zeta(s;s_0)$}{tildezeta(s;s0)} - Segment Case}
\label{S:segment}

The assumption on $\tilde \zeta(\phi;s_0)$ is that for $\phi \in \phi_0 + [0,2\pi + \beta^-(\phi_0))$ it is made of
\begin{itemize}
\item a segment $[P_1,P_2]$ where $P_1 = \bar \zeta(s_0) = \bar \zeta(s^+_0(\phi_0))$, followed by
\item a saturated spiral with $\beta(\phi) = \bar \alpha$.
\end{itemize}
The spiral $\tilde r(\phi,s_0)$ will then proceed for all angles $\phi \geq \phi_0 + 2\pi + \beta^-(\phi_0)$ as a saturated spiral, i.e. $\beta(\phi) = \bar \alpha$ constant. Following the notation of Step 1 in Case $\Delta \phi_\mathtt a = 0$ of the proof of Theorem \ref{thm:case:a}, we will denote by $\bar \theta$ the angle \newglossaryentry{thetatildeseg}{name=\ensuremath{\bar \theta},description={angle of the initial segment $[P_1,P_2]$ for the family $\tilde \zeta(\cdot;s_0)$}}
\begin{equation*}
\gls{thetatildeseg} = \angle(e^{i\phi_0},[P_1,P_2]) \in \bigg[ 0,\frac{\pi}{2} \bigg].
\end{equation*}
The case $\bar \theta = \frac{\pi}{2}$ corresponds to the transition between the segment-case and the arc-case.

Due to uniqueness of the construction, the same structure (i.e. an initial segment followed by a saturated spiral) holds for the spiral $\tilde r(\phi;s_0 + \delta s_0)$, starting from the point $\bar \zeta(s_0 + \delta s_0)$ if $\delta s_0 \ll 1$ and $\bar \theta < \frac{\pi}{2}$, i.e. there will be an initial segment
$$
[P_1',P_2'] = [P_1 + \delta P_1, P_2 + \delta P_2],
$$
with the angles
$$
s^+(\phi_0') = s_0^+(\phi_0 + \delta \phi_0) = s_0 + \delta s_0, \quad \bar \theta' = \bar \theta + \delta \bar \theta = \angle(e^{i\phi_0},[P_1',P_2']),
$$
followed by a saturated spiral. The analysis of this section refers to Figure \ref{segm:per}. 

\begin{figure}
\centering
\includegraphics[scale=0.5]{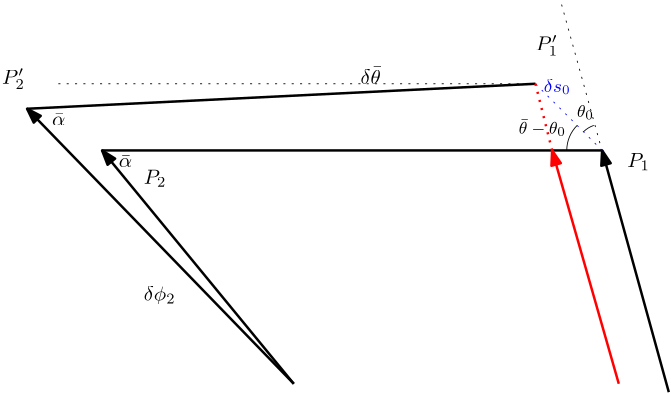}
\caption{Analysis of perturbations for the segment case.} 
\label{segm:per}
\end{figure}

Since the spirals are convex and there are no smaller admissible convex spirals for the first round ($\phi \in [\phi_0,\phi_0 + 2\pi + \beta^-(\phi_0))$), by the convexity of the level sets of $u$ one deduces that \newglossaryentry{theta0pertsegm}{name=\ensuremath{\theta_0 = \beta^+(s_0)},description={minimal initial angle for the segment in the segment case, Section \ref{S:segment}}}
\begin{equation*}
0 < \gls{theta0pertsegm} \leq \bar \theta, \quad \bar \alpha \leq \bar \theta,
\end{equation*}
where the second is due to the fact that the point $\bar \zeta(s_0)$ is not saturated if $P_1 \not= P_2$. We will denote by $\theta$ the angle \newglossaryentry{thetapertsegm}{name=\ensuremath{\theta},description={variation of the direction of the spiral at $s_0$}}
\begin{equation*}
\gls{thetapertsegm} = \bar \theta - \theta_0.
\end{equation*}

We compute now the variations up to second order terms w.r.t. $\delta s_0$:
\begin{equation*}
\delta P_1 = \delta s_0 e^{\i (\phi_0 + \theta_0)}, \quad \delta \phi_0 = \frac{\delta s_0 \sin(\theta_0)}{r(\phi_0)},
\end{equation*}
\begin{align*}
\delta P_2 &= \delta P_1 + \delta |P_2 - P_1| e^{\i (\phi_0 + \bar \theta)} + |P_2 - P_1| \delta \bar \theta e^{\i (\phi_0 + \bar \theta + \frac{\pi}{2})} \\
\notag &= \delta s_0 e^{\i (\phi_0 + \theta_0)} + \delta |P_2 - P_1| e^{\i (\phi_0 + \bar \theta)} + |P_2 - P_1| \delta \bar \theta e^{\i (\phi_0 + \bar \theta + \frac{\pi}{2})}.
\end{align*}
Since the point $P_2'$ is saturated, we obtain
\begin{equation*}
0 = u(P_2') - \cos(\bar \alpha) (\delta s_0 + |P_2' - P_1'|+L(P_0)),
\end{equation*}
where $L(P_0)$ is the length of the spiral up to the angle $\phi_0$. In particular, observing that $\nabla u(P_2) = e^{\i (\phi_0 + \bar \theta - \bar \alpha)}$,
\begin{align*}
 0 &= e^{\i (\phi_0 + \bar \theta - \bar \alpha)} \cdot \delta P_2 - \cos(\bar \alpha) (\delta s_0 + \delta |P_2 - P_1|) \\
&= \delta s_0 \cos(\bar \theta - \theta_0 - \bar \alpha) + \delta |P_2 - P_1| \cos(\bar \alpha) - |P_2 - P_1| \delta \theta \sin(\bar \alpha) - \cos(\bar \alpha) (\delta s_0 + \delta |P_2 - P_1|) \\
&= \delta s_0 \cos(\bar \alpha - \theta) + \delta |P_2 - P_1| \cos(\bar \alpha) - |P_2 - P_1| \delta \theta \sin(\bar \alpha) - \cos(\bar \alpha) (\delta s_0 + \delta |P_2 - P_1|),
\end{align*}
where we recall that \gls{vcdotwscalarprod} denotes the scalar product of $\mathbf v,\mathbf w \in \R^2$. We thus deduce that
\begin{equation}
\label{Equa:delta_theta_segm_1}
\frac{|P_2 - P_1| \delta \theta}{\delta s_0} = \frac{\cos(\bar \alpha - \theta) - \cos(\bar \alpha)}{\sin(\bar \alpha)}.
\end{equation}
Imposing that the angle at $P_2'$ with the optimal ray is equal to the critical angle $\bar \alpha$, we deduce that $P_2'$ corresponds to the angle $\phi_2 + \delta \theta$, where $\phi_2 = \phi_0 + \bar \theta - \bar \alpha$ is the angle corresponding to $P_2$ in the angle-ray representation for $\tilde r(\phi;s_0)$. Hence, imposing this relation we obtain
\begin{align*}
r(\phi_2) \delta \theta &= e^{\i (\phi_0 + \bar \theta - \bar \alpha + \frac{\pi}{2})} \cdot \delta P_2 \\
&= \delta s_0 \sin(\bar \alpha - \theta) + \delta |P_2 - P_1| \sin(\bar \alpha) + |P_2 - P_1| \delta \theta \cos(\bar \alpha),
\end{align*}
so that using \eqref{Equa:delta_theta_segm_1} we conclude that
\begin{equation}
\label{Equa:delta_P_2P_1_segm}
\begin{split}
\frac{\delta |P_2 - P_1|}{\delta s_0} - \frac{r(\phi_2)}{\sin(\bar \alpha)} \frac{\delta \theta}{\delta s_0} &= \cot(\bar \alpha) \frac{\cos(\bar \alpha) - \cos(\bar \alpha - \theta)}{\sin(\bar \alpha)} + \frac{\sin(\theta - \bar \alpha)}{\sin(\bar \alpha)} \\
&= \frac{\cos(\bar \alpha)^2 - \cos(\theta)}{\sin(\bar \alpha)^2} \\
&= - 1 + \frac{1 - \cos(\theta)}{\sin(\bar \alpha)^2}.
\end{split}
\end{equation}

To compute the perturbation $\delta \tilde r(\phi;s_0)$ for $\phi \geq \phi_2$, we observe that $\beta(\phi) = \bar \alpha$ in this region, so that, denoting with $R(\phi)$ the radius of curvature of the base spiral $r(\phi)$, one obtains

\begin{lemma}
\label{Lem:equation_segment_tilde_r}
The RDE satisfied by $\tilde r$ for $\phi \geq \phi_0 + \bar \theta - \bar \alpha$ is
\begin{equation}
\label{Equa:segm_equ_tilde_r}
\frac{d}{d\phi} \tilde r(\phi;s_0) = \cot(\bar \alpha) \tilde r(\phi;s_0) - \begin{cases}
R(\phi) & \phi_0 + \bar \theta - \bar \alpha \leq \phi < \phi_0 + 2\pi + \beta^-(\phi_0), \\
0 & \phi_0 + 2\pi + \beta^-(\phi_0) \leq \phi < \phi_0 + 2\pi + \bar \theta, \\
|P_2 - P_1| \Diracd_{\phi_0 + 2\pi + \bar \theta} + \frac{\tilde r(\phi - 2\pi - \bar \alpha)}{\sin(\bar \alpha)} & \phi \geq \phi_0 + 2\pi + \bar\theta.
\end{cases}
\end{equation}
\end{lemma}

\begin{proof}
We have just to compute the radius of curvature, since $\beta = \bar \alpha$ fixed. Formula \eqref{Equa:segm_equ_tilde_r} follows easily, observing that if $\beta^-(s_0) < \bar \theta$ then there is corner in the spiral $\tilde r(\phi;s_0)$ whose subdifferential is
\begin{equation*}
\partial^- \tilde r(\phi_0;s_0) = \big\{ e^{i\phi}, \phi \in \phi_0 + [\beta^-(\phi_0),\bar \theta] \big\}.
\end{equation*}
The jump at $\phi_0 + 2\pi + \bar \theta$ is due to the segment $[P_1,P_2]$, corresponding in a jump in $s^-(\phi)$ at that angle.
\end{proof}

In the next proposition we compute the RDE satisfied by the derivative $\delta \tilde r(\phi;s_0)$ w.r.t. $s_0$, for angles $\phi > \phi_0 + \bar \theta - \bar \alpha$.

\begin{proposition}
\label{Prop:equa_delta_tilde_r_segm}
For $\phi > \phi_2 = \phi_0 + \bar \theta - \bar \alpha$, the derivative 
$$
\delta \tilde r(\phi;s_0) = \lim_{\delta s_0 \searrow 0} \frac{\tilde r(\phi;s_0 + \delta s_0) - \tilde r(\phi;s_0)}{\delta s_0}
$$
satisfies the RDE on $\R$
\begin{equation*}
\frac{d}{d\phi} \delta \tilde r(\phi;s_0) = \cot(\bar \alpha) \delta \tilde r(\phi;s_0) - \frac{\delta \tilde r(\phi - 2\pi - \bar \alpha)}{\sin(\bar \alpha)} + S(\phi;s_0),
\end{equation*}
with source
\begin{align*}
S(\phi;s_0) &= \cot(\bar \alpha) \frac{1 - \cos(\theta)}{\sin(\bar \alpha)} \Diracd_{\phi_0 + \bar \theta - \bar \alpha} - \Diracd_{\phi_0 + 2\pi + \theta_0} \\
& \quad + \frac{\cos(\theta) - \cos(\bar \alpha)^2}{\sin(\bar \alpha)^2} \Diracd_{\phi_0 + 2\pi + \bar \theta} + \frac{\cos(\bar \alpha - \theta) - \cos(\bar \alpha)}{\sin(\bar \alpha)} \Diracd'_{\phi_0 + 2\pi + \bar \theta},
\end{align*}
where $\Diracd'$ is the distributional derivative of the Dirac's delta $\Diracd$.
\end{proposition}

\begin{proof}
%
We compute first the variation of the initial data of \eqref{Equa:segm_equ_tilde_r}: observing that the spiral $\tilde r(\phi;s_0)$ is saturated at $\phi_2$ and hence its tangent vector is $e^{\i (\phi_0 + \bar \theta)}$, one has up to second order terms for $\phi$ close to $\phi_2$ 
\begin{align*}
\tilde r(\phi;s_0 + \delta s_0) - \tilde r(\phi;s_0) &= \frac{e^{\i (\phi_0 + \bar \theta - \frac{\pi}{2})} \cdot \delta P_2}{\sin(\bar \alpha)} \\
&= \delta s_0 \frac{\sin(\theta)}{\sin(\bar \alpha)} - \frac{|P_2 - P_1| \delta \theta}{\sin(\bar \alpha)} \\
&= \delta s_0 \cot(\bar \alpha) \frac{1 - \cos(\theta)}{\sin(\bar \alpha)},
\end{align*}
which gives the first jump at $\phi_0 + \bar \theta - \bar \alpha$.

Being $R(\phi)$ the same up to $\phi_0 + 2\pi + \theta_0$, the RDE for $\delta \tilde r$ is actually the ODE
\begin{equation*}
\frac{d}{d\phi} \delta \tilde r(\phi) = \cot(\bar \alpha) \delta \tilde r(\phi)
\end{equation*}
for $\phi_0 + \bar \theta - \bar\alpha \leq \phi < \phi_0 + 2\pi + \theta_0$.

Using $R = Ds^-$, the variation at $\phi_0 + 2\pi + \theta_0$ is
\begin{align*}
\tilde r(\phi;s_0 + \delta s_0) - \tilde r(\phi;s_0) &= \cot(\bar \alpha) \big( \tilde r(\phi;s_0 + \delta s_0) - \tilde r(\phi;s_0) \big) - \big( s^-(\phi) - s^-(\phi_0 + 2\pi + \theta_0) \big) \\
&= \cot(\bar \alpha) \big( \tilde r(\phi;s_0 + \delta s_0) - \tilde r(\phi;s_0) \big) - \delta s_0.
\end{align*}
Hence $\delta \tilde r(\phi;s_0)$ has a jump of size $-1$ at the angle $\phi_0 + 2\pi + \theta_0$. 
From \eqref{Equa:segm_equ_tilde_r} it follows that
\begin{equation*}
\frac{d}{d\phi} \delta \tilde r(\phi) = \cot(\bar \alpha) \delta \tilde r(\phi)
\end{equation*}
for $\phi_0 + 2\pi + \theta_0 < \phi < \phi_0 + 2\pi + \bar \theta$.

At the angle $\phi_0 + 2\pi + \bar \theta$, computing the first order approximations for $0 \leq \phi - \phi_0 + 2\pi + \bar \theta \ll 1$ one gets up to second order terms
\begin{align*}
\tilde r(\phi;s_0 + \delta s_0) - \tilde r(\phi;s_0) &= \bigg[ \tilde r(\phi_0 + 2\pi + \bar\theta; s_0 + \delta s_0) + \cot(\bar \alpha) \tilde r(\phi_0 + 2\pi + \bar\theta;s_0 + \delta s_0) (\phi - \phi_0 - 2\pi - \bar \theta) \\
& \quad \quad - |P_2' - P_1'| \big( 1 + \cot(\bar \alpha) (\phi - \phi_0 - 2\pi - \bar \theta - \delta \bar \theta) \big) \ind_{\phi \geq \phi_0 + 2\pi + \bar \theta + \delta \bar \theta} \\
& \quad \quad - \frac{\tilde r(\phi_0 + \bar \theta - \bar \alpha;s_0 + \delta s_0)}{\sin(\bar \alpha)} (\phi - \phi_0 - 2\pi - \bar \theta - \delta \bar \theta) \bigg] \\
& \quad - \bigg[ \tilde r(\phi_0 + 2\pi + \theta-;s_0) + \cot(\bar \alpha) \tilde r(\phi_0 + 2\pi +\theta;s_0) (\phi - \phi_0 - 2\pi - \bar \theta) \\
& \quad \quad - |P_2 - P_1| \big( 1 + \cot(\bar \alpha) (\phi - \phi_0 - 2\pi - \bar \theta) \big) \ind_{\phi \geq \phi_0 + 2\pi + \bar \theta}  \\
& \quad \quad - \frac{\tilde r(\phi_0 + \bar \theta - \bar \alpha;s_0)}{\sin(\bar \alpha)} (\phi - \phi_0 - 2\pi - \bar \theta) \bigg] \\
&= \big( \tilde r(\phi_0 + 2\pi + \bar\theta;s_0 + \delta s_0) - \tilde r(\phi_0 + 2\pi + \bar\theta;s_0) \big) \big( 1 + \cot(\bar \alpha) (\phi - \phi_0 - 2\pi - \bar \theta) \big) \\
& \quad - \frac{\tilde r(\phi_0 + \bar \theta - \bar \alpha;s_0 + \delta s_0) - \tilde r(\phi_0 + \bar \theta - \bar \alpha;s_0)}{\sin(\bar \alpha)} (\phi - \phi_0 - 2\pi - \bar \theta) \\
& \quad + \bigg( - \delta |P_2 - P_1| + \frac{\tilde r(\phi_0 + \bar \theta - \bar \alpha;s_0)}{\sin(\bar \alpha)} \bigg) \ind_{\phi \geq \phi_0 + 2\pi + \bar \theta} + |P_2 - P_1| \ind_{[\phi_0 + 2\pi + \bar \theta,\phi_0 + 2\pi + \bar \theta + \delta \theta)}.
\end{align*}
For the above formula, we have used that the jumps $|P_2 - P_1|$ are in different angles and that the curvature term $- \frac{\tilde r(\phi_0+\bar \theta -\bar \alpha;s_0)}{\sin(\bar \alpha)}$ is applied starting from different angles. \\
Passing to the limit $\delta s_0 \searrow 0$ we obtain
\begin{align*}
\delta \tilde r(\phi;s_0) &= \delta \tilde r(\phi_0 + 2\pi + \bar\theta;s_0) \big( 1 + \cot(\bar \alpha) (\phi - \phi_0 - 2\pi - \bar \theta) \big) - \frac{\delta \tilde r(\phi_0 + \bar \theta - \bar \alpha;s_0)}{\sin(\bar \alpha)} (\phi - \phi_0 - 2\pi - \bar \theta) \\
& \quad + \bigg( - \frac{\delta |P_2 - P_1|}{\delta s_0} + \frac{\tilde r(\phi_0 + \bar \theta - \bar \alpha;s_0)}{\sin(\bar \alpha)} \frac{\delta \theta}{\delta s_0} \bigg) \ind_{\phi \geq \phi_0 + 2\pi + \bar \theta} + |P_2 - P_1| \frac{\delta \theta}{\delta s_0} \Diracd_{\phi_0 + 2\pi + \bar \theta}.
\end{align*}
Hence its derivative w.r.t. $\phi$ satisfies
\begin{align*}
\frac{d}{d\phi} \delta \tilde r(\phi;s_0) &= \cot(\bar \alpha) \delta \tilde r(\phi;s_0) - \frac{\delta \tilde r(\phi - 2\pi - \bar \alpha;s_0)}{\sin(\bar \alpha)} \\
& \quad + \bigg( - \frac{\delta |P_2 - P_1|}{\delta s_0} + \frac{\tilde r(\phi_0 + \bar \theta - \bar \alpha;s_0)}{\sin(\bar \alpha)} \frac{\delta \theta}{\delta s_0}  \bigg) \Diracd_{\phi_0 + 2\pi + \bar\theta} + |P_2 - P_1| \frac{\delta \theta}{\delta s_0} \Diracd'_{\phi_0 + 2\pi + \bar \theta},
\end{align*}
where $\Diracd'$ is the derivative of the Dirac's delta:
\begin{equation*}
\langle \Diracd'_x, f \rangle = - \frac{df(x)}{dx}.
\end{equation*}
Using \eqref{Equa:delta_theta_segm_1} and \eqref{Equa:delta_P_2P_1_segm} we obtain
\begin{align*}
- \frac{\delta |P_2 - P_1|}{\delta s_0} + \frac{\bar r(\phi_0 + \bar \theta - \bar \alpha;s_0)}{\sin(\bar \alpha)} \frac{\delta \theta}{\delta s_0} &= \cot(\bar \alpha) \frac{\cos(\bar \alpha - \theta) - \cos(\bar \alpha)}{\sin(\bar \alpha)} - \frac{\sin(\theta - \bar \alpha)}{\sin(\bar \alpha)} \\
&= \frac{\cos(\theta) - \cos(\bar \alpha)^2}{\sin(\bar \alpha)^2},
\end{align*}
\begin{equation*}
|P_2 - P_1| \frac{\delta \theta}{\delta s_0} = \frac{\cos(\bar \alpha - \theta) - \cos(\bar \alpha)}{\sin(\bar \alpha)}.
\end{equation*}
It is now easy to verify that the derivative $\delta \tilde r(\phi;s_0) \rest_{\phi_0 + \bar \theta - \bar \alpha}$ satisfies the RDE in the statement, since in the remaining part of $\R$ it solves the linear RDE of the saturated spiral.
\end{proof}

Since
\begin{equation*}
\sin(\theta) - \frac{\cos(\bar \alpha - \theta) - \cos(\bar \alpha)}{\sin(\bar \alpha)} = \cot(\bar \alpha) (1 - \cos(\theta)) \geq 0,
\end{equation*}
it follows that the perturbation is positive for $\phi \in [\phi_0,\phi_0+\bar \theta - \bar \alpha]$, and it is given explicitly by
\begin{equation*}
\delta \tilde r(\phi;s_0) = \frac{1}{\sin(\phi_0 + \bar \theta - \phi)} \bigg( \sin(\theta) - \frac{|P(\phi) - P_1|}{|P_2 - P_1|} \frac{\cos(\bar \alpha - \theta) - \cos(\bar \alpha)}{\sin(\bar \alpha)} \bigg),
\end{equation*}
where $P(\phi)$ is the position of the point on $[P_1,P_2]$ corresponding to the angle $\phi \in [\theta_0,\theta_0 + \bar \theta - \bar \alpha]$.

\begin{corollary}
\label{Cor:evolv_r_after_phi_2}
The solution $\delta \tilde r(\phi)$ for $\phi \geq \phi_2 = \phi_0 + \bar \theta - \bar \alpha$ is explicitly given by
\begin{equation*}
\begin{split}
\delta \tilde r(\phi) &= \cot(\bar \alpha) \frac{1 - \cos(\theta)}{\sin(\bar \alpha)} G(\phi - \phi_0 - \bar \theta + \bar \alpha) - G(\phi - \phi_0 - 2\pi - \theta_0) \\
&\quad + \bigg( 1 - \frac{1 - \cos(\theta)}{\sin(\bar \alpha)^2} + \cot(\bar \alpha) \frac{\cos(\bar \alpha - \theta) - \cos(\bar \alpha)}{\sin(\bar \alpha)} \bigg) G(\phi - 2\pi - \phi_0 - \bar \theta) \\
& \quad + \frac{\cos(\bar \alpha - \theta) - \cos(\bar \alpha)}{\sin(\bar \alpha)} \Diracd_{\phi - 2\pi - \phi_0 - \bar \theta} - \frac{\cos(\bar \alpha - \theta) - \cos(\bar \alpha)}{\sin(\bar \alpha)^2} G(\phi - 4\pi - \bar \theta - \bar \alpha),
\end{split}
\end{equation*}
where $G(\phi)$ is the Green kernel for the RDE.
\end{corollary}

\begin{proof}
The formula follows from Remark \ref{Rem:oringal_phi_kernel}.
\end{proof}

To proceed further, we rewrite the RDE translating $\phi_0 + \bar \theta - \bar \alpha$ to $0$ and then use the new variable
\begin{equation}
\label{Equa:rho_segm_r}
\delta \rho(\tau) = \delta \tilde r(\phi_0 + \bar \theta - \bar \alpha + (2\pi + \bar \alpha) \tau;s_0) e^{- \bar c (2\pi + \bar \alpha) \tau}.
\end{equation}

\begin{corollary}
\label{Cor:equation}
The values of the perturbation $\delta r(s;s_0)$ depend uniquely on the angle $\theta = \bar \theta - \theta_0 \in [0,\frac{\pi}{2}]$, and the function $\delta \rho(\tau)$ of Equation \eqref{Equa:rho_segm_r} is given explicitly by the formula \newglossaryentry{tau1}{name=\ensuremath{\tau_1},description={time of the first discontinuity for the perturbation in the segment and arc case}}
\begin{equation}
\label{Equa:equa_rho_segm_666}
\delta \rho(\tau) = S_0 g(\tau) + S_1 g(\tau - \tau_1) + S_2 g(\tau-1) + D \Diracd_1 + S_3 g(\tau - 2), \quad \gls{tau1} = 1 - \frac{\theta}{2\pi + \bar \alpha},
\end{equation}
with
\begin{subequations}
\label{Equa:source_segmnt_pert}
\begin{equation*}
S_0(\theta) = \cot(\bar \alpha) \frac{1 - \cos(\theta)}{\sin(\bar \alpha)},
\end{equation*}
\begin{equation*}
S_1(\theta) = - e^{- \bar c (2\pi + \bar \alpha - \theta)} = - \frac{\sin(\bar \alpha)}{2\pi + \bar \alpha} e^{\bar c \theta},
\end{equation*}
\begin{equation*}
\begin{split}
S_2(\theta) &= \bigg( 1 - \frac{1 - \cos(\theta)}{\sin(\bar \alpha)^2} + \cot(\bar \alpha) \frac{\cos(\bar \alpha - \theta) - \cos(\bar \alpha)}{\sin(\bar \alpha)} \bigg) e^{-\bar c (2\pi + \bar \alpha)} \\
&= \frac{\sin(\bar \alpha)}{2\pi + \bar \alpha} - \frac{1 - \cos(\theta)}{(2\pi + \bar \alpha) \sin(\bar \alpha)} + \frac{\cot(\bar \alpha)}{2\pi + \bar \alpha} \big( \cos(\bar \alpha - \theta) - \cos(\bar \alpha) \big),
\end{split}
\end{equation*}
\begin{equation*}
D(\theta) = \frac{\cos(\bar \alpha - \theta) - \cos(\bar \alpha)}{\sin(\bar \alpha)} e^{-\bar c (2\pi + \bar \alpha)} = \frac{\cos(\bar \alpha -\theta) - \cos(\bar \alpha)}{2\pi + \bar \alpha},
\end{equation*}
\begin{equation*}
S_3(\theta) = - \frac{\cos(\bar \alpha -\theta) - \cos(\bar \alpha)}{\sin(\bar \alpha)^2} e^{-\bar c (4\pi + 2 \bar \alpha)} = - \frac{\cos(\bar \alpha - \theta) - \cos(\bar \alpha)}{(2\pi + \bar \alpha)^2}.
\end{equation*}
\end{subequations}
\end{corollary}

We remark that we have used Equation \eqref{Equa:c_gene_defi} to replace $e^{- \bar c(2\pi + \bar \alpha)}$ with $\frac{\sin(\bar \alpha)}{2\pi + \bar \alpha}$.

The study of the sign of $\delta \rho$ given by \eqref{Equa:equa_rho_segm_666} is done in the next proposition, which is the main statement of this section concerning the segment case. For completeness, we let $\theta$ to vary in $[0,\pi]$.

\begin{proposition}
\label{Prop:regions_pos_neg_segm}
The solution $\delta \rho(\tau) = \delta \rho(\tau,\theta)$ is strictly positive outside the region \newglossaryentry{Nsegm}{name=\ensuremath{N_\mathrm{segm}},description={negative region for the perturbation of the segment case}}
\begin{equation*}
\gls{Nsegm} := \bigg\{ 1 - \frac{\theta}{2\pi + \bar \alpha} \leq \tau \leq 1, 0 \leq \theta \leq \hat \theta \bigg\} \cup \big\{ \theta = 0 \} \cup \{ \tau = 1, \theta \in [2\bar \alpha,\pi] \big\} \subset \R \times [0,\pi],
\end{equation*}
where the angle \newglossaryentry{thetahat}{name=\ensuremath{\hat \theta},description={angle corresponding to the boundary of the negativity region of the perturbation in the segment case, Equation (\ref{Equa:hat_theta_det_666})}} \gls{thetahat} is determined by the unique solution to
\begin{equation}
\label{Equa:hat_theta_det_666}
\cot(\bar \alpha) \frac{1 - \cos(\hat \theta)}{\sin(\bar \alpha)} -  e^{- \cot(\bar \alpha) (2\pi + \bar \alpha - \hat \theta)} = 0, \qquad \hat \theta \in [0,\pi].
\end{equation}
Moreover, as $\tau \to \infty$, the function $\delta \rho(\tau)$ diverges like
\begin{equation*}
\lim_{\tau \to \infty} \frac{\delta \rho(\tau) e^{-\bar c \tau}}{\tau} = 2 (S_0 + S_1 + S_2 + S_3), \quad 2 (S_0 + S_1 + S_2 + S_3) \in \theta^2 (0.08,028).
\end{equation*}
\end{proposition}

\begin{figure}
\includegraphics[scale=.6]{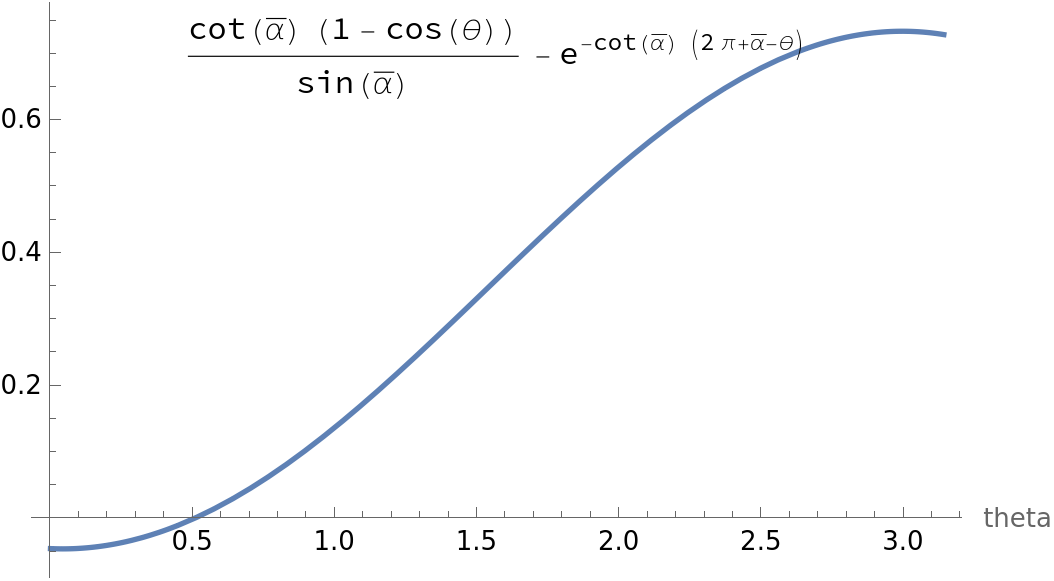}\hfill\includegraphics[scale=.6]{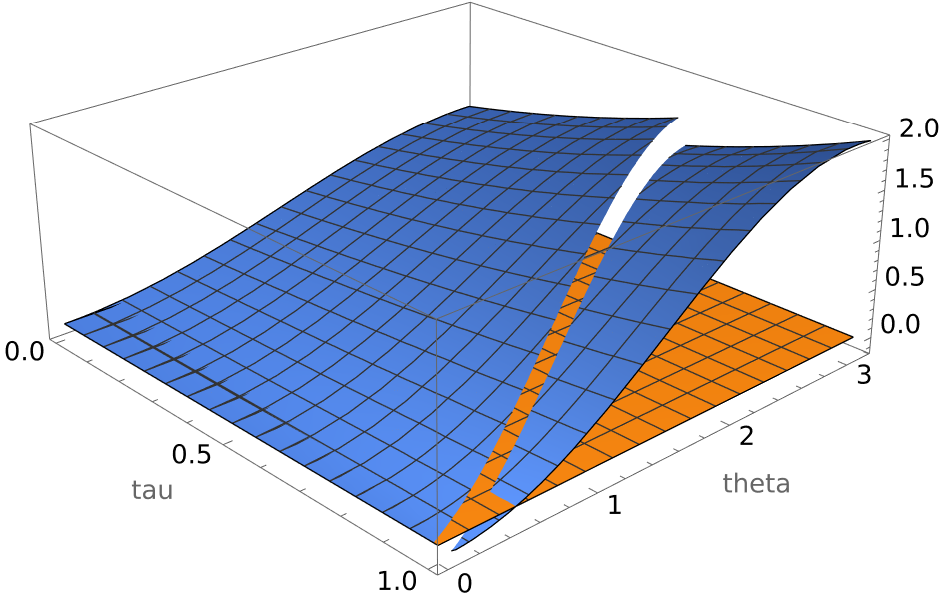}
\caption{Plot of the function (\ref{Equa:hat_theta_det_666}) on the left side, and of the solution (\ref{Equa:rho_segm_first}) on the right side.}
\label{Fig:hattheta_comput}
\end{figure}

A numerical computation gives the value (see Fig. \ref{Fig:hattheta_comput})
\begin{equation}
\label{Equa:hat_theta_value}
\hat \theta = 0.506134.
\end{equation}

\begin{proof}
We start by observing that for $\theta = 0$ the coefficients \eqref{Equa:source_segmnt_pert} are all $0$, so that the solution \eqref{Equa:equa_rho_segm_666} is $0$. Hence $\{\theta = 0\} \subset N_\mathrm{segm}$.

\medskip

{\it Step 1: computation of the first round.}

\noindent For $\tau \in [0,1)$ the solution is explicitly computed as
\begin{equation}
\label{Equa:rho_segm_first}
\delta \rho(\tau) = S_0 e^\tau + S_1 e^{\tau - \tau_1} \ind_{\tau \geq \tau_1} = e^{\tau} (S_0 + S_1 e^{-\tau_1} \ind_{\tau \geq \tau_1}).
\end{equation}
The solution $\delta \rho$ is thus negative only when $\tau_1 \leq \tau < 1$  and iff
\begin{equation*}
S_0 + S_1 e^{-\tau_1} = \cot(\bar \alpha) \frac{1- \cos(\theta)}{\sin(\bar \alpha)} - \frac{\sin(\bar \alpha)}{2\pi + \bar \alpha} e^{(\bar c + \frac{1}{2\pi + \bar \alpha})  \theta - 1} \leq 0,
\end{equation*}
which is \eqref{Equa:hat_theta_det_666}. This gives the angle $\hat \theta$: indeed
the function
\begin{equation*}
\theta \mapsto f(\theta) = (1 - \cos(\theta)) e^{- (\bar c + \frac{1}{2\pi + \bar \alpha}) \theta} = (1 - \cos(\theta)) e^{- \cot(\bar \alpha) \theta}
\end{equation*}
has first derivative
\begin{align*}
f'(\theta) &= \big( \sin(\theta) - \cos(\bar \alpha) (1 - \cos(\theta)) \big) e^{-\cot(\bar \alpha) \theta} = \frac{\cos(\bar \alpha - \theta) - \cos(\bar \alpha)}{\sin(\bar \alpha)} e^{- \cot(\bar \alpha) \theta} > 0 \quad \text{if} \ \theta > 0,
\end{align*}
because $\bar \alpha > \frac{\pi}{4}$, and
$$
f(0) = -1, \quad f \bigg( \frac{\pi}{2} \bigg) = e^{-\cot(\bar \alpha) \frac{\pi}{2}} > 0.
$$
Hence there is only one zero for $f$, and numerically one obtains the value $\hat \theta = 0.506134$, see Figure \ref{Fig:hattheta_comput}.

\medskip

\begin{figure}
\includegraphics[scale=.6]{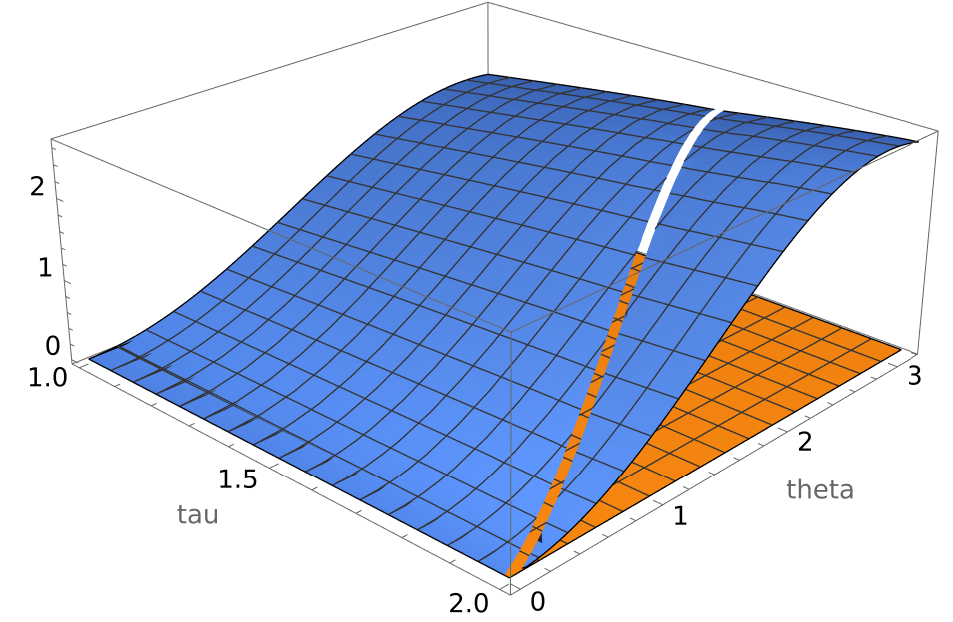}\hfill\includegraphics[scale=.6]{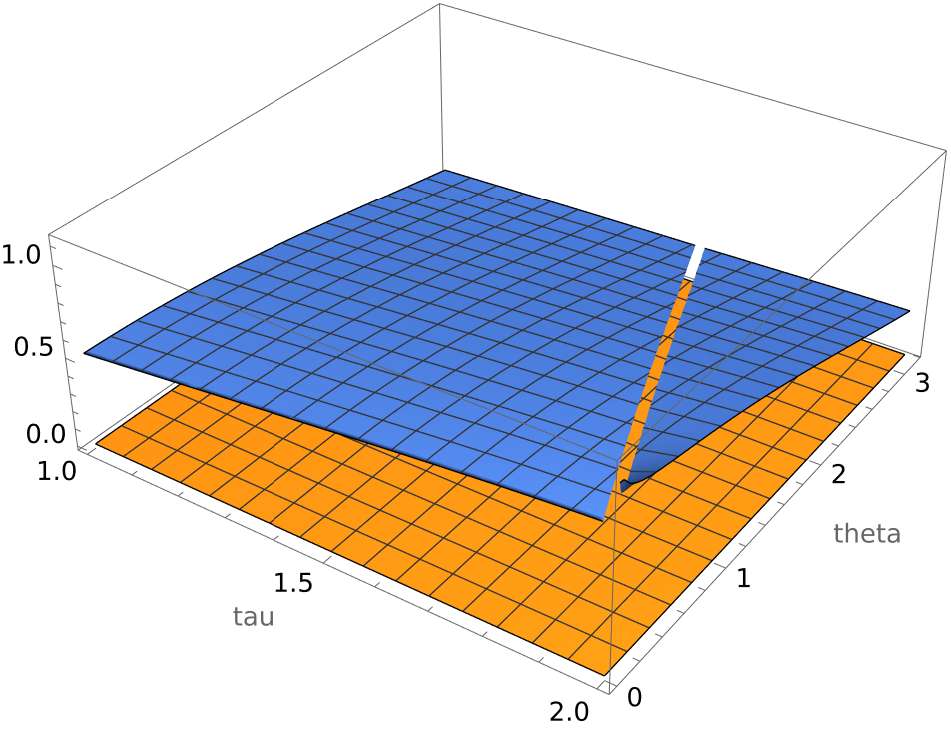}
\caption{Plot of the solution $\delta \rho(\tau;\theta)$ and of $\frac{\delta \rho(\tau;\theta)}{\theta^2}$ in the segment case for $\tau \in [1,2)$. Notice that the discontinuity for $\tau = \tau_1$ becomes here a discontinuity in the derivative in $1 + \tau_1$, and this is the reason why $\delta \rho \sim \theta$ as $\theta \in [1+\tau_1,2)$. In the next step we will have $\delta \rho \sim \theta^2$ in $[2 + \tau_1,3)$.}
\label{Fig:segm_round_2}
\end{figure}

{\it Step 2: the function $\frac{\delta \rho(\tau,\theta)}{\theta^2}$ is strictly positive for $\tau \in [1,1+\tau_1)$ and increasing w.r.t. $\tau$.}

\noindent For $\tau = [1,1+\tau_1)$ it holds
\begin{align*}
\delta \rho(\tau) &= D \Diracd_1 + S_0 \big( e^\tau - e^{\tau - 1} (\tau - 1) \big) + S_1 e^{\tau - \tau_1} + S_2 e^{\tau - 1}.
\end{align*}
Elementary computations show as before that since $\bar \alpha > \frac{\pi}{4}$
$$
D = \frac{\cos(\bar \alpha - \theta) - \cos(\bar \alpha)}{2\pi + \bar \alpha} \leq 0 \quad \Longleftrightarrow \quad \theta \in \{0\} \cup [2\bar \alpha,\pi],
$$
and in the interval $\theta \in [0,\pi]$ the only $0$ is for $\theta = 0$ which corresponds to $D = 0$.

When $1 < \tau < 1 + \tau_1$ the positivity of $\delta \rho$ depends on
\begin{align*}
f(\theta) &= S_0 (e - \tau + 1) + S_1 e^{1 - \tau_1} + S_2,
\end{align*}
whose worst case is for $\tau = 1 + \tau_1$:
\begin{equation*}
f(\theta) = S_0 (e - \tau_1) + S_1 e^{1-\tau_1} + S_2.
\end{equation*}
Near $\bar \theta = \theta_0$ we have the expansion
\begin{equation}
\label{Equa:first_round_exp_segm}
\begin{split}
f(\theta) &= \bigg( \frac{\cot(\bar \alpha)}{\sin(\bar \alpha)} (e-1) - \frac{1}{(2\pi + \bar \alpha) \sin(\bar \alpha)} - 2 \frac{\cot(\bar \alpha) \cos(\bar \alpha)}{2\pi + \bar \alpha} \bigg) \frac{(\bar \theta - \theta_0)^2}{2} + \mathcal O(\theta^3) \\
&\sim 0.582367 \frac{\theta^2}{2} + \mathcal O(\theta^3),
\end{split}
\end{equation}
so that we consider the function
\begin{equation*}
f'(\theta) = \frac{\delta \rho(\tau,\theta)}{\theta^2}, \quad 0 \leq \theta \leq \pi, 1 \leq \tau \leq 1 + \tau_1. 
\end{equation*}
A numerical plot shows that this function is increasing w.r.t. $\tau$ for all fixed $\theta$, and its minimal value is
\begin{equation*}
f'(1,\pi) = 0.17959.
\end{equation*}

\medskip

{\it Step 3: the function $\frac{\delta \rho(\tau,\theta)}{\theta^2}$ is strictly positive and increasing for $\tau \in [1+\tau_1(\theta),2)$, and diverging near $\theta = 0$.}

\noindent For $\tau \in [1+\tau_1,2)$ the solution is
\begin{align*}
\delta \rho(\tau) &= S_0 \big( e^\tau - e^{\tau - 1} (\tau - 1) \big) + S_1 \big( e^{\tau - \tau_1} - e^{\tau - 1 - \tau_1} (\tau - 1 - \tau_1) \big) + S_2 e^{\tau - 1}.
\end{align*}
Again numerically the function
\begin{equation*}
f(\tau,\theta) = \frac{\delta \rho(\tau,\theta)}{\theta^2}
\end{equation*}
has a strictly positive derivative w.r.t. $\tau$, and it is strictly positive, with minimal value at 
$$
f(1+\tau_1,\pi) = 0.226635.
$$

We compute the expansion at $\theta = 0$ finding
\begin{align*}
\delta \rho(\tau,\theta) &= \frac{\sin(\bar \alpha)}{2\pi + \bar \alpha} \bigg( \tau - 2 + \frac{\theta}{2\pi + \bar \alpha} \bigg),
\end{align*}
so that $f$ is diverging as $\frac{\sin(\bar \alpha) (1-\xi)}{(2\pi + \bar \alpha)^2 \theta}$ as $\theta \to 0$ and $\tau = 2 - \xi \frac{\theta}{2\pi + \bar \alpha}$, $\xi \in [0,1]$.

\medskip

\begin{figure}
\includegraphics[scale=.6]{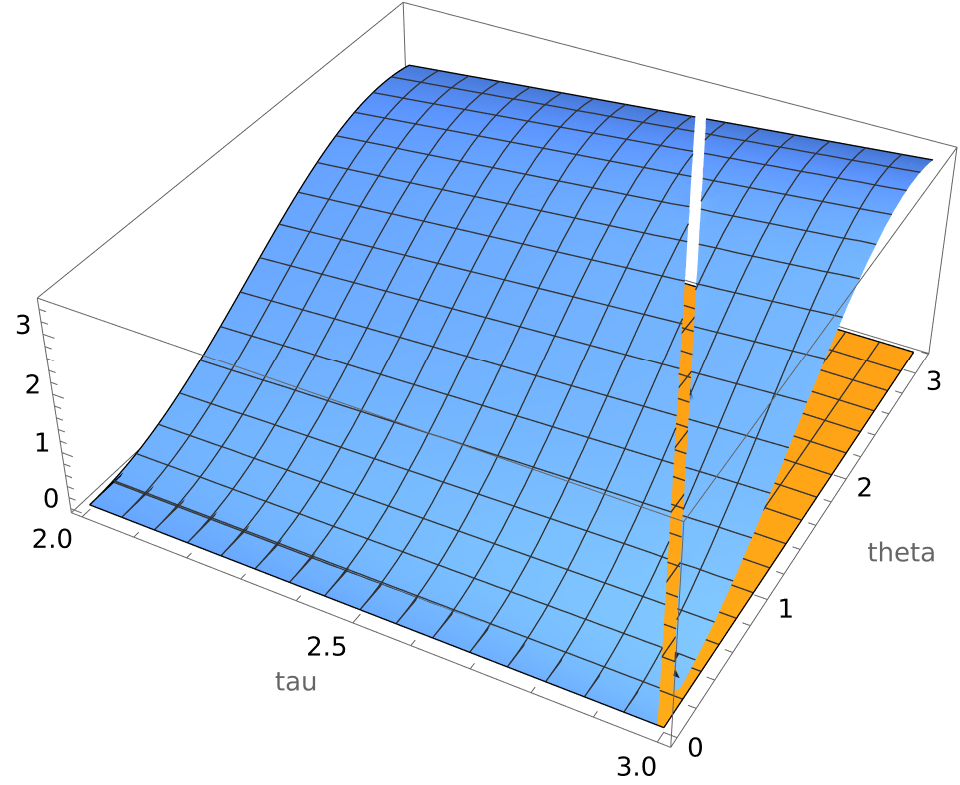}\hfill\includegraphics[scale=.6]{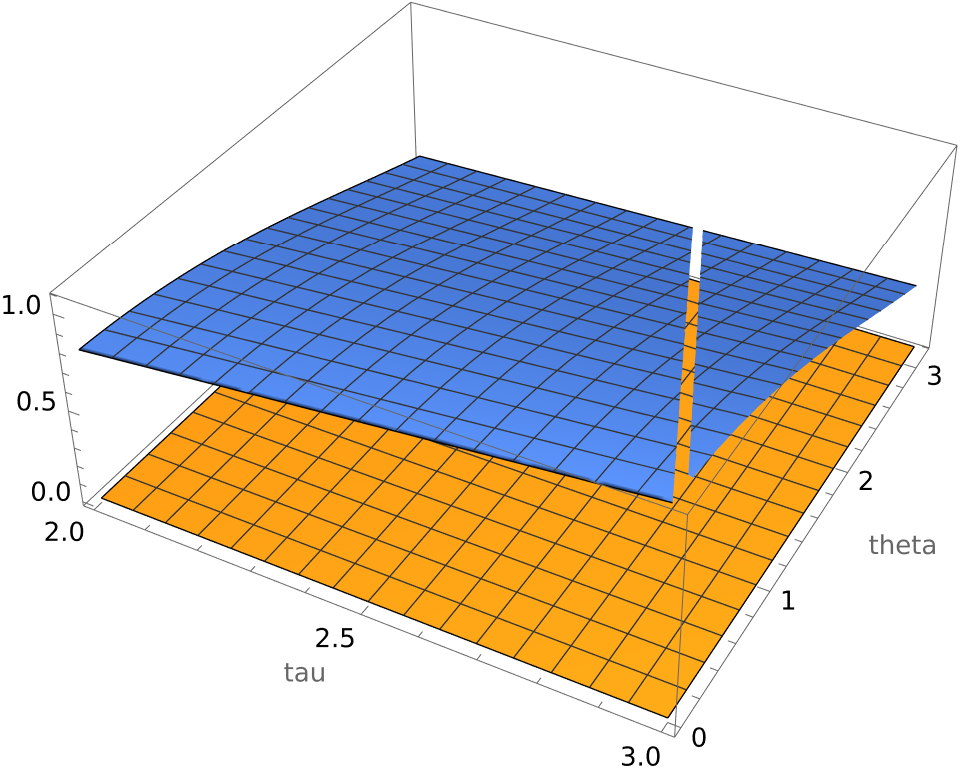}
\caption{Plot of the function $\delta \rho(\tau;\theta)$ (Equation (\ref{Equa:rhosegm_third})) on the left side and of $\frac{\delta \rho(\tau;\theta)}{\theta^2}$ on the right side for $\tau \in [2,3)$.}
\label{Fig:thordround_rhosegm}
\end{figure}

{\it Step 4: for $\tau \in [2,3]$ the function $\delta \rho(\tau,\theta)/\theta^2$ is bounded and strictly positive.}

\noindent The solution in this interval is written as
\begin{equation}
\label{Equa:rhosegm_third}
\begin{split}
\delta \rho(\tau,\theta) &= S_0 \bigg( e^\tau - e^{\tau - 1} (\tau-1) + e^{\tau-2} \frac{(\tau-2)^2}{2} \bigg) \\
& \quad + S_1 \bigg( e^{\tau - \tau_1} - e^{\tau - 1 - \tau_1} (\tau - 1 - \tau_1) + e^{\tau - 2 - \tau_1} \frac{(\tau - 2 - \tau_1)^2}{2} \ind_{\tau \geq 2 + \tau_1} \bigg) \\
&\quad + S_2 \big( e^{\tau-1} - e^{\tau-2} (\tau-2) \big) + S_3 e^{\tau-2}.
\end{split}
\end{equation}
Its expansion at $\theta = 0$ is quadratic in $\theta$, 
\begin{align*}
\delta \rho(\tau,\theta) &= \Bigg[ \frac{\cot(\bar \alpha)}{\sin(\bar \alpha)} \bigg( e^\tau - e^{\tau - 1} (\tau-1) + e^{\tau-2} \frac{(\tau-2)^2}{2} \bigg) \\
& \quad \quad - \frac{2 \cos(\bar \alpha)^2 + 1}{(2\pi + \bar \alpha) \sin(\bar \alpha)} \big( e^{\tau - 1} - e^{\tau - 2} (\tau - 2) \big) \\
& \quad \quad + \frac{3 \cos(\bar \alpha)}{(2\pi + \bar \alpha)^2} e^{\tau - 2} + \frac{\sin(\bar \alpha)}{2\pi + \bar \alpha} e^{\tau -3} \frac{(\frac{\tau - 3}{\theta} - \frac{1}{2\pi + \bar \alpha})^2}{2} \ind_{0 \geq \tau - 3 \geq - \frac{\theta}{2\pi + \bar \alpha}} \bigg) \Bigg] \frac{\theta^2}{2} + o(\theta^2).
\end{align*}
We can thus study numerically the function
\begin{equation*}
f(\tau,\theta) = \frac{\delta \rho(\tau,\theta)}{\theta^2},
\end{equation*}
which shows that the minimal value is
\begin{equation*}
f(2,\pi) = 0.262163.
\end{equation*}

\medskip

\begin{figure}
\includegraphics[scale=.6]{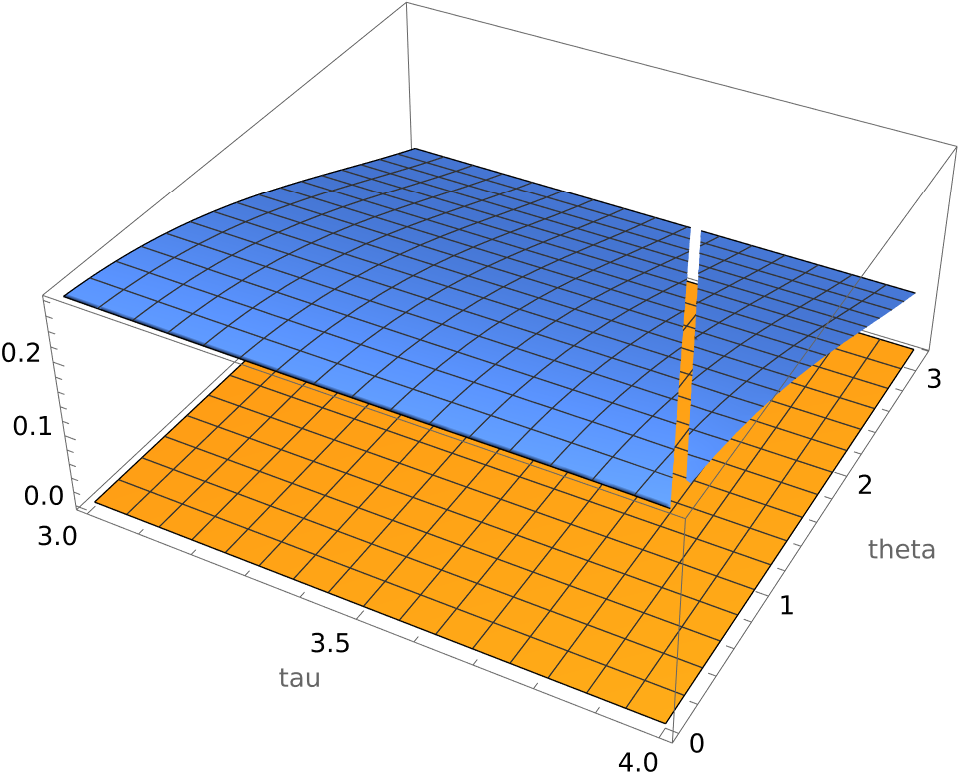}\hfill\includegraphics[scale=.6]{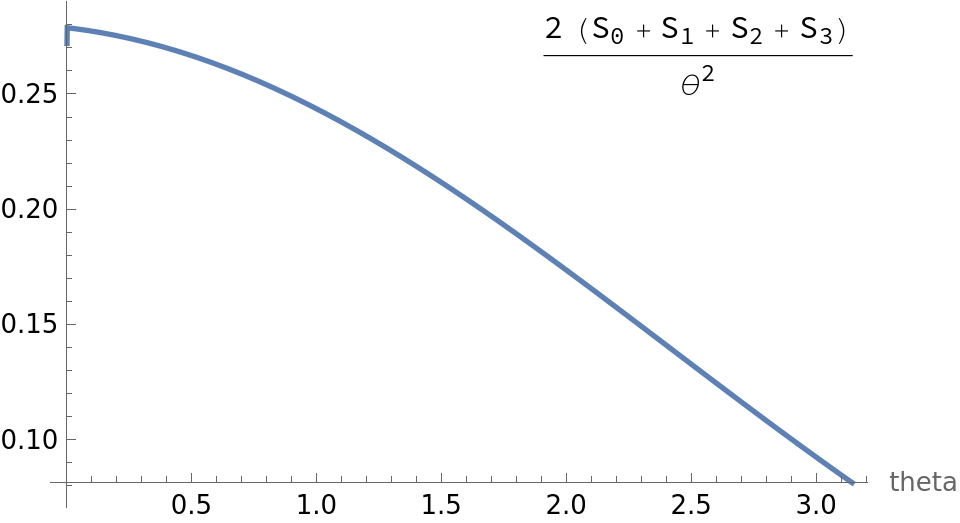}
\caption{Plot of the function $\frac{\delta \rho(\tau;\theta)}{\theta^2}$ for $\tau \geq 3$ (Equation (\ref{Equa:f_increas_segm})) on the left side and of the r.h.s. of Equation (\ref{Equa:asympt_segm}).}
\label{Fig:final_growth_rhosegm}
\end{figure}

{\it Step 5: the function $\frac{\delta \rho(\tau,\theta)}{\theta^2}$ is strictly increasing for $\tau \geq 3$.}

\noindent It is enough to study the interval $\tau \in [3,4]$ for the function
\begin{equation}
\label{Equa:f_increas_segm}
f(\tau,\theta) = \frac{\delta \rho(\tau,\theta)}{\theta^2}
\end{equation}
and apply Lemma \ref{lem:key}. The previous point, plus the fact that $g$ satisfies
\begin{equation*}
\dot f(\tau) = f(\tau) - f(\tau -1),
\end{equation*}
implies that it is bounded, so we do not need to study the expansion about $\theta = 0$.

The solution $\delta \rho$ is written as
\begin{align*}
\delta \rho(\tau,\theta) &= S_0 \bigg( e^\tau - e^{\tau - 1} (\tau-1) + e^{\tau-2} \frac{(\tau-2)^2}{2} - e^{\tau-3} \frac{(\tau-3)^3}{3!} \bigg) \\
& \quad + S_1 \bigg( e^{\tau - \tau_1} - e^{\tau - 1 - \tau_1} (\tau - 1 - \tau_1) + e^{\tau - 2 - \tau_1} \frac{(\tau - 2 - \tau_1)^2}{2} - e^{\tau - 3 - \tau_1} \frac{(\tau - 3 - \tau_1)^3}{3!} \ind_{\tau \geq 3 + \tau_1} \bigg) \\
&\quad + S_2 \bigg( e^{\tau-1} - e^{\tau-2} (\tau-2) + e^{\tau-3} \frac{(\tau-3)^2}{2} \bigg) + S_3 \big( e^{\tau-2} - e^{\tau-3} (\tau-3) \big).
\end{align*}
Numerical computations show that
\begin{equation*}
\frac{\partial_\tau \delta \rho(\tau,\theta)}{\theta^2} \geq \frac{\partial_\tau \delta \rho(3,\pi)}{\pi^2} = 0.142304,
\end{equation*}
and since the solution is continuous in the closed interval and positive at $\tau = 3$ from the previous point, we conclude that $\frac{\delta \rho(\tau,\theta)}{\theta^2}$ is strictly increasing in $\tau$, see Fig. \ref{Fig:final_growth_rhosegm}.

\medskip

{\it Step 6: the asymptotic behavior.} By using Lemma \ref{Lem:aympt_g}, the solution behaves as
\begin{equation}
\label{Equa:asympt_segm}
\lim_{\tau \to \infty} \frac{\delta \rho(\tau,\theta) e^{-\bar c(2\pi + \bar \alpha) \tau}}{\tau \theta^2} = 2 \frac{S_0(\theta) + S_1(\theta) + S_2(\theta) + S_3(\theta)}{\theta^2},
\end{equation}
and numerical plots gives that the function is decreasing w.r.t. $\theta$ values with (Fig. \ref{Fig:final_growth_rhosegm})
\begin{equation*}
0.0816077 \leq 2 \frac{S_0 + S_1 + S_2 + S_3}{\theta^2} \leq 0.278936. \qedhere
\end{equation*}
\end{proof}

%
%

\subsubsection{Estimate on the length for the fastest saturated spiral $\delta r(\phi;s_0)$, segment case}
\label{Sss:additional_estim_opt_segm}

In this section we study the functional
\begin{equation}
\label{Equa:Lpert_segm}
\delta \tilde r(\phi;s_0) - 2.08 \delta \tilde L(\phi - 2\pi + \bar \alpha;s_0) \quad \text{when} \quad \phi \geq 2\pi + \bar \alpha,
\end{equation}
where we denote with $\delta \tilde L$ the variation of length: more precisely, for angles larger than $\phi_0 + \bar \theta - \bar \alpha$ the spirals are saturated so that
\begin{equation*}
\delta \tilde L(\phi;s_0) = \frac{1 - \cos(\theta)}{\sin(\bar \alpha)^2} + \int_{\phi_0 + \bar \theta - \bar \alpha}^\phi \frac{\delta \tilde r(\phi';s_0)}{\sin(\bar \alpha)} d\phi'.
\end{equation*}
The first term is the due to the computation
\begin{equation*}
1 + \delta |P_2 - P_1| - \frac{r_2}{\sin(\bar \alpha)} \delta \bar \theta = \frac{1 - \cos(\theta)}{\sin(\bar \alpha)^2}
\end{equation*}
by \eqref{Equa:delta_P_2P_1_segm}.

Using the explicit expression of $\delta \tilde r$ of Corollary \ref{Cor:evolv_r_after_phi_2}, the length $\delta \tilde L$ can be written by
\begin{equation}
\label{Equa:evolv_L_after_phi_2}
\begin{split}
\delta \tilde L(\phi;s_0) &= \frac{1 - \cos(\theta)}{\sin(\bar \alpha)^2} + \cot(\bar \alpha) \frac{1 - \cos(\theta)}{\sin(\bar \alpha)} \mathcal G(\phi - \phi_0 - \bar \theta + \bar \alpha) - \mathcal G(\phi - \phi_0 - 2\pi - \theta_0) \\
&\quad + \bigg( 1 - \frac{1 - \cos(\theta)}{\sin(\bar \alpha)^2} + \cot(\bar \alpha) \frac{\cos(\bar \alpha - \theta) - \cos(\bar \alpha)}{\sin(\bar \alpha)} \bigg) \mathcal G(\phi - 2\pi - \phi_0 - \bar \theta) \\
& \quad + \frac{\cos(\bar \alpha - \theta) - \cos(\bar \alpha)}{\sin(\bar \alpha)} H(\phi - 2\pi - \phi_0 - \bar \theta) - \frac{\cos(\bar \alpha - \theta) - \cos(\bar \alpha)}{\sin(\bar \alpha)^2} \mathcal G(\phi - 4\pi - \phi_0 - \bar \theta - \bar \alpha),
\end{split}
\end{equation}
where we have used the definition of \gls{Gcal}. The equation it satisfies is
\begin{align*}
\frac{d}{dt} \delta \tilde L(\phi) &= \cot(\bar \alpha) \delta \tilde L(\phi) - \frac{\delta \tilde L(\phi - 2\pi - \bar \alpha)}{\sin(\bar \alpha)} + \cot(\bar \alpha) \frac{1 - \cos(\theta)}{\sin(\bar \alpha)} \frac{H(\phi - \phi_0 - \bar \theta + \bar \alpha)}{\sin(\bar \alpha)} - \frac{H(\phi - \phi_0 - 2\pi - \theta_0)}{\sin(\bar \alpha)} \\
&\quad + \bigg( 1 - \frac{1 - \cos(\theta)}{\sin(\bar \alpha)^2} + \cot(\bar \alpha) \frac{\cos(\bar \alpha - \theta) - \cos(\bar \alpha)}{\sin(\bar \alpha)} \bigg) \frac{H(\phi - 2\pi - \phi_0 - \bar \theta)}{\sin(\bar \alpha)} \\
& \quad + \frac{\cos(\bar \alpha - \theta) - \cos(\bar \alpha)}{\sin(\bar \alpha)} \frac{\Diracd(\phi - 2\pi - \phi_0 - \bar \theta)}{\sin(\bar \alpha)} - \frac{\cos(\bar \alpha - \theta) - \cos(\bar \alpha)}{\sin(\bar \alpha)^2} \frac{H(\phi - 4\pi - \phi_0 - \bar \theta - \bar \alpha)}{\sin(\bar \alpha)},
\end{align*}
which becomes for $\phi \geq \phi_0 + \bar \theta + 4\pi + \bar \alpha$
\begin{equation*}
\frac{d}{dt} \delta \tilde L(\phi) = \cot(\bar \alpha) \delta \tilde L(\phi) - \frac{\delta \tilde L(\phi - 2\pi - \bar \alpha)}{\sin(\bar \alpha)} + S_L,
\end{equation*}
\begin{align*}
S_L &= \frac{1}{\sin(\bar \alpha)} \bigg[ \cot(\bar \alpha) \frac{1 - \cos(\theta)}{\sin(\bar \alpha)} - 1 + \bigg( 1 - \frac{1 - \cos(\theta)}{\sin(\bar \alpha)^2} + \cot(\bar \alpha) \frac{\cos(\bar \alpha - \theta) - \cos(\bar \alpha)}{\sin(\bar \alpha)} \bigg) - \frac{\cos(\bar \alpha - \theta) - \cos(\bar \alpha)}{\sin(\bar \alpha)^2} \bigg] \\
&= - \frac{1 - \cos(\bar \alpha)}{\sin(\bar \alpha)^3} \big( (1 - \cos(\bar \alpha)) (1 - \cos(\theta) + \sin(\bar \alpha) \sin(\theta) \big) < 0.
\end{align*}
In particular we have for the same range of angles
\begin{align*}
\frac{d}{d\phi} \big( \delta \tilde r(\phi) - 2.08 \delta \tilde L(\phi) \big) &= \cot(\bar \alpha) \big( \delta \tilde r(\phi) - 2.08 \delta \tilde L(\phi) \big) - \frac{\delta \tilde r(\phi - 2\pi - \bar \alpha ) - 2.08 \delta \tilde L(\phi - 2\pi - \bar \alpha)}{\sin(\bar \alpha)} - S_L \\
&\geq \cot(\bar \alpha) \big( \delta \tilde r(\phi) - 2.08 \delta \tilde L(\phi) \big) - \frac{\delta \tilde r(\phi - 2\pi - \bar \alpha ) - 2.08 \delta \tilde L(\phi - 2\pi - \bar \alpha)}{\sin(\bar \alpha)}.
\end{align*}

\begin{figure}
\begin{subfigure}{.33\textwidth}
\resizebox{\textwidth}{!}{\includegraphics{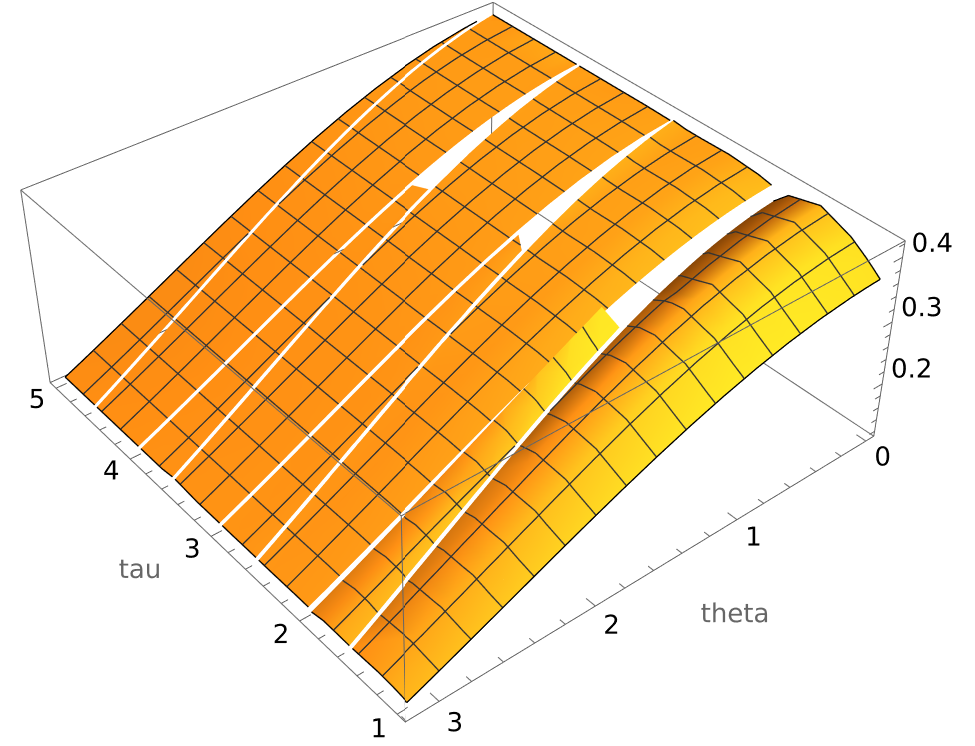}}
\end{subfigure} \hfill
\begin{subfigure}{.33\textwidth}
\resizebox{\textwidth}{!}{\includegraphics{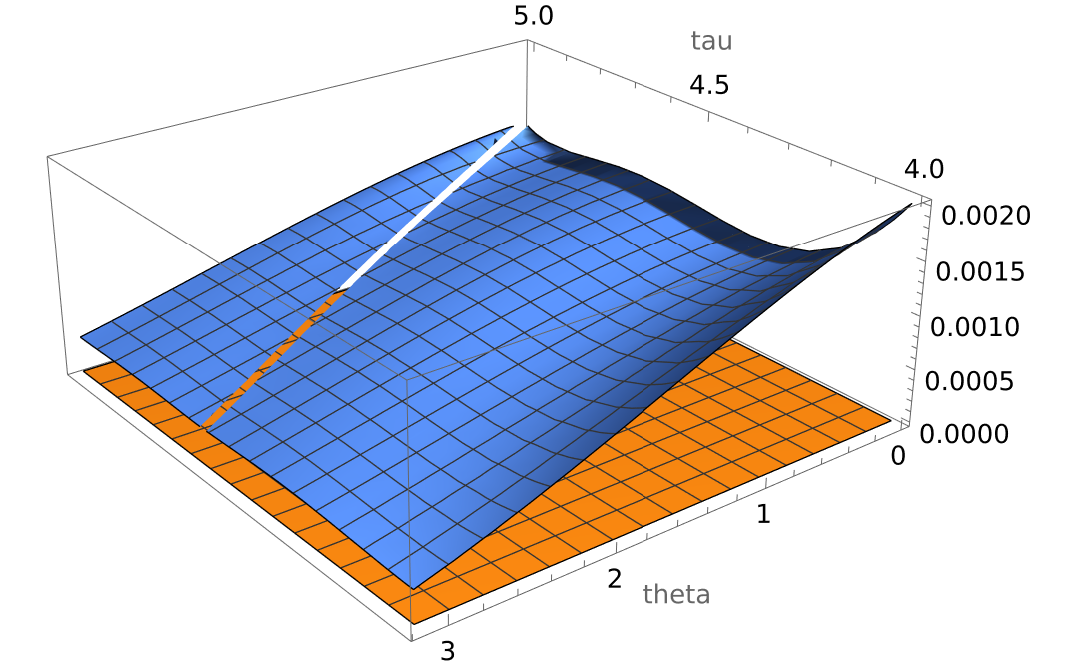}}
\end{subfigure} \hfill
\begin{subfigure}{.33\textwidth}
\resizebox{\textwidth}{!}{\includegraphics{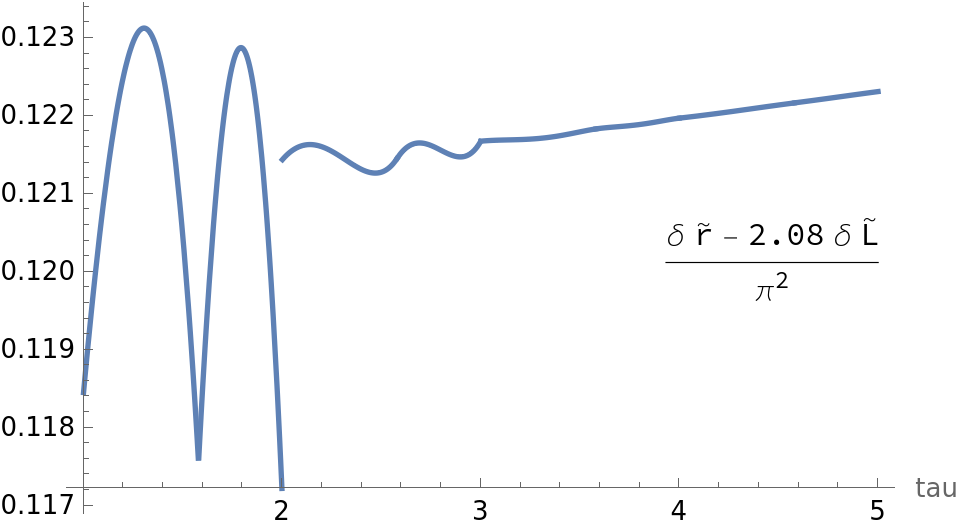}}
\end{subfigure}
\caption{Numerical plot of the function \eqref{Equa:Lpert_segm} rescaled by $e^{-\bar c \phi}$, divided by $\theta^2$ and using the variable $\tau$: from left to right, the first 4 rounds, the derivative for the round $\tau \in [4,5]$ and the minimal value.}
\label{Fig:Lpert_segm}
\end{figure}

In Fig. \ref{Fig:Lpert_segm} the numerical plot of the function \eqref{Equa:Lpert_segm} shows that it is positive, increasing after $\tau = 4$ and thus by Lemma \ref{lem:key} and Remark \ref{rmk:exponentially:explod} it holds

\begin{proposition}
\label{Prop:lemgth_segm_pert}
It holds for the perturbation in the segment case
\begin{equation*}
\delta \rho(\tau) - 2.08 e^{- \bar c(2\pi + \bar \alpha) \tau} \delta \tilde L(\tau) \geq 0.117 \theta^2, \quad \tau \geq 1. 
\end{equation*}
\end{proposition}

\subsection{Analysis of Perturbations - Arc Case}
\label{S:arc}

In this section we repeat the computations in the case that the fastest saturated solution $\tilde \xi$ for $\phi \in \phi_0 + [0,2\pi + \beta^-(\phi_0))$ has the following structure:
\begin{itemize}
\item it is an arc of the level set $\lbrace u=u(\bar \zeta(s_0)) \rbrace$, starting in $P_0 = \bar \zeta(\phi_0)$ and ending in \newglossaryentry{Deltaphi}{name=\ensuremath{\Delta \phi},description={opening angle for the arc on the level set}} $P_1 = \tilde \zeta(\phi_1) = \tilde \zeta(\phi_0 + \gls{Deltaphi})$;
\item a segment $[P_1,P_2]$ tangent to the level set $\lbrace u=u(\bar \zeta(s_0))\rbrace$ in $P_1$ and such that $P_2 = \tilde \zeta(\phi_2) = \bar \zeta(\phi_1 + \frac{\pi}{2} -\bar \alpha)$ is saturated;
\item a saturated spiral with $\beta(\phi) = \bar \alpha$ for $\phi \geq \phi_2$.
\end{itemize}
Due to the presence of the initial arc, we call it the \emph{arc case}. The analysis of this section refers to the geometric situation of Figure \ref{fig:enpert:arc}.

We observe that for angles $\phi\geq \phi_1$ the spiral has the same structure of the segment case (Subsection \ref{S:segment}), with $\bar\theta=\frac{\pi}{2}$, though we will show that the analysis of the derivative is different due to the presence of the initial arc: the formulas will coincide when $\bar \theta = \frac{\pi}{2}$ and $\Delta \phi = 0$, yielding as corollary that the derivative $\delta r$ is continuous across $\bar \theta = \frac{\pi}{2}$ (Remark \ref{Rem:same_as_segment}).

\begin{figure}
\centering
\includegraphics[scale=0.5]{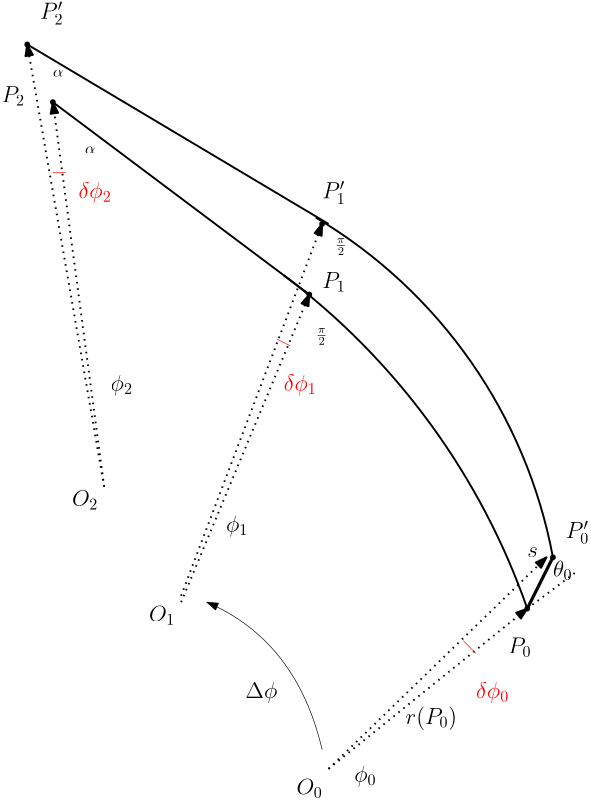}
\caption{Analysis of perturbations of the arc case, Subsection \ref{S:arc}.}
\label{fig:enpert:arc}
\end{figure}

Let us denote as before
\begin{equation*}
\gls{thetapertsegm} = \frac{\pi}{2} - \theta_0
\end{equation*}
We compute now the variations up to second order terms w.r.t. $\delta s_0$:
\begin{equation*}
\delta P_0 = \delta s_0 e^{\i (\phi_0 + \frac{\pi}{2} - \theta)}, \quad \delta \phi_0 = \delta s_0 \frac{\cos(\theta)}{r(\phi_0)},
\end{equation*}
\begin{equation}\label{eq:P:1:arc}
\delta P_1 = \tilde r(\phi_1) \delta \phi_1 e^{\i (\phi_0 + \Delta\phi + \frac{\pi}{2})} + \delta s_0 \sin(\theta) e^{\i (\phi_0 + \Delta \phi)},
\end{equation}
\begin{equation}\label{eq:P:2:arc}
\delta P_2 = \delta P_1 + \delta |P_2 - P_1| e^{\i (\phi_0 + \Delta\phi + \frac{\pi}{2})} - |P_2 - P_1|e^{i\phi_1} \delta \phi_1.
\end{equation}

By the saturation condition on $P_2$ and $P_2'$ we find that
\begin{equation*}
0 = u(P_2') - \cos(\bar \alpha) (\delta s_0 +L(P'_0P'_1) +|P_2' - P_1'|+L(P_0)),
\end{equation*}
where $L(P_0)$ is the length of the spiral up to the angle $\phi_0$ and $L(P'_0P'_1)$ is the length of the arc $P'_0P_1'$, in particular
\begin{align*}
0 = e^{\i (\phi_0 + \Delta\phi+\frac{\pi}{2} - \bar \alpha)} \cdot \delta P_2 - \cos(\bar \alpha) (\delta s_0 + \delta L(P_0P_1)+ \delta |P_2 - P_1|).
\end{align*}
We now notice that
\begin{equation*}
\delta L(P_0P_1) = \delta s_0 \sin(\theta) \Delta \phi + \tilde r(\phi_1) \delta \phi_1 - \delta s_0 \cos(\theta),
\end{equation*}
so that the saturation condition becomes 
\begin{align*}
0 &= \tilde r(\phi_1) \cos(\bar \alpha) \delta \phi_1 + \delta s_0 \sin(\theta) \sin(\bar \alpha) + \delta |P_2 - P_1| \cos(\bar \alpha) - |P_2 - P_1| \delta \phi_1 \sin(\bar \alpha) \\
& \quad - \cos(\bar \alpha) \big( \delta s_0 + \delta s_0 \sin(\theta) \Delta \phi + \tilde r(\phi_1) \delta \phi_1 - \delta s_0 \cos(\theta) + \delta |P_2 - P_1| \big) \\
&= \delta s_0 \big( \cos(\bar \alpha - \theta) - \cos(\bar \alpha) - \cos(\bar \alpha) \sin(\theta) \Delta\phi \big) - |P_2 - P_1| \delta \phi_1 \sin(\bar \alpha),
\end{align*}
in particular
\begin{equation}
\label{Equa:DeltaP_arc}
|P_2 - P_1| \frac{\delta \phi_1}{\delta s_0} = \sin(\theta) - \cot(\bar \alpha) \big( 1 - \cos(\theta) + \Delta \phi \sin(\theta) \big).
\end{equation}
Moreover, from the relations
\begin{equation*}
r(\phi_2) \delta \phi_2 = \delta P_2 \cdot e^{\i (\phi_0 + \Delta \phi + \pi - \bar \alpha)} \quad \text{and} \quad \phi_2 = \phi_1 + \frac{\pi}{2} - \bar \alpha,
\end{equation*}
we obtain $\delta \phi_2 = \delta \phi_1$ and
\begin{equation}
\label{Equa:Deltar_arc}
\begin{split}
\frac{r(\phi_2) \delta \phi_1}{\sin(\bar \alpha)} - \delta |P_2 - P_1| - r(\phi_1) \delta \phi_1 &= \cot(\bar \alpha) |P_2 - P_1| \delta \phi_1 - \cot(\bar \alpha) \sin(\theta) \delta s_0 \\
&= - \cot(\bar \alpha)^2 \big( 1 - \cos(\theta) + \Delta \phi \sin(\theta) \big).
\end{split}
\end{equation}
The variation of the angle $\Delta \phi$ along the arc of the level set is then
\begin{equation}
\label{Equa:delta_Delta_phi}
\begin{split}
\delta \Delta \phi &= \delta\phi_1-\delta\phi_0 \\
&= \frac{\delta s_0}{|P_2-P_1|} \left( \sin(\theta) (1 - \cot(\bar \alpha) \Delta \phi) - \cot(\bar \alpha) (1 - \cos(\theta)) \right) - \frac{\delta s_0}{r(\phi_0)} \cos(\theta).
\end{split}
\end{equation}
The variation $\delta P(\phi)$ along the segment $[P_1,P_2]$ can be computed as 
\begin{align*}
\delta P(\phi) &= \frac{\delta P_1 \cdot e^{\i \phi_1} - |P(\phi) - P_1| \delta \phi_1}{\cos(\phi - \phi_1)} \\
&= \frac{1}{\cos(\phi - \phi_1)} \bigg( \sin(\theta) \bigg(1 - \frac{|P(\phi) - P_1|}{|P_2 - P_1|} \bigg) + \frac{|P(\phi) - P_1|}{|P_2 - P_1|} \cot(\bar \alpha) \big( 1 - \cos(\theta) + \Delta \phi \sin(\theta) \big) \bigg),
\end{align*}
with $\phi \in (\phi_1,\phi_2)$.

\begin{remark}
\label{rmk:deltaphi:max}
We remark that $\Delta\phi\leq\tan\bar\alpha,$ where the equality corresponds to the situation of the case study of Section \ref{S:case:study} with $\mathtt a = 0$. Indeed for such a choice
\begin{equation*}
|P_2 - P_1| \frac{\delta \Delta \phi}{\delta s_0} = - \cot(\bar \alpha) (1 - \cos(\theta)) - \frac{|P_2 - P_1|}{\tilde r(\phi_0)} \cos(\theta) < 0
\end{equation*}
unless $\theta = 0$ and $|P_2 - P_1| = 0$. Notice that for or $\tilde r_{\mathtt a}$ we do not have the initial segment $[(0,0),(1,0)]$, and this is the reason why the limit value $\Delta \phi = \tan(\bar \alpha)$ is allowed.
\end{remark}

To compute $\delta \tilde r(\phi;s_0)$ for $\phi \geq \phi_2$, we observe that the angle $\beta(\phi)$ is equal to $\bar \alpha$ in this region, so that one obtains

\begin{lemma}
\label{Lem:equation_arc_tilde_r}
The RDE satisfied by $\tilde r$ for $\phi \geq \phi_0 +\Delta\phi+\frac{\pi}{2} - \bar \alpha$ is
\begin{equation}
\label{Equa:arc_equ_tilde_r}
\dot{\tilde r}(\phi;s_0) = \cot(\bar \alpha) \tilde r(\phi;s_0) - \begin{cases}
R(\phi) & \phi_0 +\Delta\phi+\frac{\pi}{2} - \bar \alpha \leq \phi < \phi_0 + 2\pi + \beta^-(\phi_0), \\
0 & \phi_0 + 2\pi + \beta^-(\phi_0) \leq \phi < \phi_0 + 2\pi + \frac{\pi}{2}, \\
\tilde r(\phi-2\pi-\frac{\pi}{2}) & \phi_0 + 2\pi + \frac{\pi}{2}\leq \phi\leq \phi_0 + 2\pi + \frac{\pi}{2}+\Delta\phi, \\
|P_2 - P_1| \Diracd_{\phi_0 + 2\pi + \frac{\pi}{2}+\Delta\phi} + \frac{\tilde r(\phi - 2\pi - \bar \alpha)}{\sin(\bar \alpha)} & \phi \geq \phi_0 + 2\pi + \frac{\pi}{2}+\Delta\phi.
\end{cases}
\end{equation}
where $R(\phi)$ is the radius of curvature of $r(\phi)$ (the original spiral).
\end{lemma}

\begin{proof}
We have just to compute the radius of curvature, since $\beta(\phi) = \bar \alpha$ fixed for $\phi \geq \phi_2$. Formula \eqref{Equa:arc_equ_tilde_r} follows easily, observing that if $\beta^-(s_0) < \frac{\pi}{2}$ then there is corner in the spiral $\tilde r(\phi;s_0)$ whose subdifferential is
\begin{equation*}
\partial^- \tilde r(\phi;s_0) = \bigg\{ e^{i\phi}, \phi \in \phi_0 + 2\pi + \bigg[ \beta^-(\phi_0),\frac{\pi}{2} \bigg] \bigg\}.
\end{equation*}
For the angles $\phi_0 + 2\pi + \frac{\pi}{2}\leq \phi\leq \phi_0 + 2\pi + \frac{\pi}{2}+\Delta\phi$ the radius of curvature is precisely the radius of the spiral, computed at the previous angle (Formula \eqref{eq:subs:angle}). The jump at $\phi_0 + 2\pi + \frac{\pi}{2}+\Delta\phi$ is due to the segment $[P_1,P_2]$, corresponding in a jump in $s^-(\phi)$ at that angle.
\end{proof}

Computing the derivative w.r.t. $s_0$ for $\phi > \phi_0 + \Delta\phi+\frac{\pi}{2}- \bar \alpha$ we can prove the following

\begin{proposition}
\label{Prop:equa_delta_tilde_r_arc}
For $\phi >  \phi_0 + \Delta\phi+\frac{\pi}{2}- \bar \alpha$, the derivative
$$
\delta \tilde r(\phi;s_0) = \lim_{\delta s_0 \searrow 0} \frac{\tilde r(\phi;s_0 + \delta s_0) - \tilde r(\phi;s_0)}{\delta s_0}
$$
satisfies the RDE on $\R$
\begin{equation*}
\frac{d}{d\phi} \delta \tilde r(\phi;s_0) = \cot(\bar \alpha) \delta \tilde r(\phi;s_0) - \frac{\delta \tilde r(\phi - 2\pi - \bar \alpha)}{\sin(\bar \alpha)} + S(\phi;s_0),
\end{equation*}
with source
\begin{align*}
S(\phi;s_0) &= \frac{\cot(\bar \alpha)}{\sin(\bar \alpha)} \big( 1 - \cos(\theta) + \Delta \phi \sin(\theta) \big) \Diracd_{\phi_0 + \Delta \phi + \frac{\pi}{2} - \bar \alpha} - \Diracd_{\phi_0 + 2\pi + \frac{\pi}{2} - \theta} + \cos(\theta) \Diracd_{\phi_0 + 2\pi + \frac{\pi}{2}} \\
& \quad - \sin(\theta) \ind_{\phi_0 + 2\pi + \frac{\pi}{2} + [0,\Delta \phi]} - \cot(\bar \alpha)^2 \big( 1 - \cos(\theta) + \Delta \phi \sin(\theta) \big) \Diracd_{\phi_0 + 2\pi + \Delta \phi + \frac{\pi}{2}} \\
& \quad + \big( \sin(\theta) - \cot(\bar \alpha) (1 - \cos(\theta) + \Delta \phi \sin(\theta)) \big) \Diracd'_{\phi_0 + 2\pi + \frac{\pi}{2} + \Delta \phi}.
\end{align*}
\end{proposition}

We recall the notation
$$
\phi_1 = \phi_0 + \Delta \phi, \quad \phi_2 = \phi_1 + \frac{\pi}{2} -\bar \alpha.
$$
which we will use in the following.

\begin{proof}
We compute first the variation of the initial data of \eqref{Equa:arc_equ_tilde_r}: observing that the spiral $\tilde r(\phi;s_0)$ is saturated at $\phi_2$ and hence its tangent vector is $e^{\i (\phi_0 +\Delta\phi+\frac{\pi}{2})}$, one has up to second order terms for $\phi$ close to $\phi_2$
\begin{align*}
\delta \tilde r(\phi_2) &= \frac{\delta P_2 \cdot  e^{i\phi_1}}{\sin(\bar \alpha)} \\
&= \frac{\delta s_0 \sin(\theta) - |P_2 - P_1| \delta \phi_1}{\sin(\bar \alpha)} \\
&= \delta s_0 \frac{\cot(\bar \alpha)}{\sin(\bar \alpha)}(1 - \cos(\theta) + \Delta \phi \sin(\theta)),
\end{align*}
which gives the first jump at $\phi_2$.

Being $R(\phi)$ the same up to $\phi_0 + 2\pi + \theta_0$, the RDE for $\delta \tilde r$ is actually the ODE
\begin{equation*}
\frac{d}{d\phi} \delta \tilde r(\phi) = \cot(\bar \alpha) \delta \tilde r(\phi)
\end{equation*}
for $\phi_0 + \Delta\phi + \frac{\pi}{2} - \bar\alpha \leq \phi < \phi_0 + 2\pi + \theta_0$.

Using $R = Ds^-$ and taking $\delta \phi \ll 1$, the variation at $\phi_0 + 2\pi + \theta_0$ is up to second order terms w.r.t. $\delta \phi$
\begin{align*}
&\tilde r(\phi_0+2\pi+\theta_0 + \delta \phi; s_0 + \delta s_0) - \tilde r(\phi_0+2\pi+\theta_0 + \delta \phi;s_0) \\
\quad &= (1 + \cot(\bar \alpha) \delta \phi) \big( \tilde r(\phi_0+2\pi+\theta_0;s_0 + \delta s_0) - \tilde r(\phi_0+2\pi+\theta_0;s_0) \big) - \big( s^-(\phi_0 + 2\pi + \theta_0 + \delta \phi) - s^-(\phi_0 + 2\pi + \theta_0) \big) \\
\quad &= (1 + \cot(\bar \alpha) \delta \phi) \big( \tilde r(\phi_0 + 2\pi + \theta_0;s_0 + \delta s_0) - \tilde r(\phi_0 + 2\pi + \theta_0;s_0) \big) - \delta s_0.
\end{align*}
Dividing by $\delta s_0$, letting $\delta s_0 \searrow 0$ first and then $\phi \searrow 0$, we conclude that $\delta \tilde r(\phi;s_0)$ has a jump of size $-1$ at the angle $\phi_0 + 2\pi + \theta_0$.

From \eqref{Equa:arc_equ_tilde_r} it follows that
\begin{equation*}
\frac{d}{d\phi} \delta \tilde r(\phi) = \cot(\bar \alpha) \delta \tilde r(\phi)
\end{equation*}
for $\phi \in \phi_0 + 2\pi + [\theta_0,\frac{\pi}{2}]$.

At about the angle $\phi = \phi_0 + 2\pi + \frac{\pi}{2}$, we can repeat the analysis for a small angle as we did at the angle $\phi_0 + 2\pi + \theta_0$ above,  obtaining that
there is an upward jump of value $\cos(\theta)$, giving the source $\cos(\theta) \Diracd_{\phi_0 + 2\pi + \frac{\pi}{2}}$.

Again, from \eqref{Equa:arc_equ_tilde_r} it follows that
\begin{equation*}
\frac{d}{d\phi} \delta \tilde r(\phi) = \cot(\bar \alpha) \delta \tilde r(\phi) - \sin(\theta),
\end{equation*}
for $\phi_0 + 2\pi + \frac{\pi}{2} < \phi < \phi_0 + 2\pi + \frac{\pi}{2}+\Delta\phi$, because $\delta \tilde r(\phi_0 + \frac{\pi}{2} + [0,\Delta \phi)) = \sin(\theta)$.

At the angle $\phi_0 + 2\pi + \frac{\pi}{2} + \Delta \phi$ the solution has a discontinuity due to the variation of $\delta \phi_1$ and $|P_2 - P_1|$: computing the first order approximations for $0 \leq \delta \phi \ll 1$ one gets
\begin{align*}
&\tilde r \bigg( \phi_0 + 2\pi + \frac{\pi}{2} + \Delta \phi + \delta \phi, s_0 + \delta s_0 \bigg) - \tilde r \bigg( \phi_0 + 2\pi + \frac{\pi}{2} + \Delta\phi +\delta \phi,s_0 \bigg) \\
& \quad = \left[ \tilde r \bigg( \phi_0 + 2\pi + \frac{\pi}{2} + \Delta \phi - \delta \phi, s_0 + \delta s_0 \bigg) - \tilde r \bigg( \phi_0 + 2\pi + \frac{\pi}{2} + \Delta \phi - \delta \phi,s_0 \bigg) \right] (1 + 2 \cot(\bar \alpha) \delta \phi) \\
& \quad \quad - |P'_2 - P'_1| H(\phi - \phi_2 + \delta \phi_1) - (\delta \phi + \delta \phi_1) \tilde r(\phi_1,s_0 + \delta s_0) - (\delta \phi - \delta \phi_1) \frac{\tilde r(\phi_2,s_0 + \delta s_0)}{\sin(\bar \alpha)} \\
& \quad \quad + |P_2 - P_1| H(\phi - \phi_2) + \delta \phi \tilde r(\phi_1,s_0) + \delta \phi \frac{\tilde r(\phi_2,s_0)}{\sin(\bar \alpha)},
\end{align*}
with \gls{Heavyside} being the Heaviside function. Dividing by $\delta s_0$, using that $\delta \tilde(\phi_1) = \sin(\theta)$ and passing to the limit $\delta s_0 \searrow 0$ we obtain
\begin{align*}
&\delta \tilde r \bigg( \phi_0 + 2\pi + \frac{\pi}{2} + \Delta \phi + \delta \phi;s_0 \bigg) \\
& \quad = \delta \tilde r \bigg( \phi_0 + 2\pi + \frac{\pi}{2} + \Delta \phi - \delta \phi;s_0 \bigg) \big( 1 + 2 \cot(\bar \alpha) \delta \phi \big) - \frac{\delta \tilde r(\phi_2;s_0)}{\sin(\bar \alpha)} \delta \phi - \sin(\theta) \delta \phi \\
& \quad \quad + |P_2 - P_1| \frac{\delta \phi_1}{\delta s_0} \Diracd_{\phi_0+2\pi + \frac{\pi}{2} + \Delta \phi} + \bigg[ - \frac{\delta |P_2 - P_1|}{\delta s_0} - \frac{\delta \phi_1}{\delta  s_0} \tilde r(\phi_1) + \frac{\delta \phi_1}{\delta {s_0}} \frac{\tilde r(\phi_2)}{\sin(\bar \alpha)} \bigg].
\end{align*}
Hence its derivative w.r.t. $\phi$ satisfies
\begin{align*}
\frac{d}{d\phi} \delta \tilde r(\phi;s_0) &= \cot(\bar \alpha) \delta \tilde r(\phi;s_0) - \frac{\delta \tilde r(\phi - 2\pi - \bar \alpha;s_0)}{\sin(\bar \alpha)} \\
& \quad + \bigg( - \frac{\delta |P_2 - P_1|}{\delta s_0} + \frac{\tilde r(\phi_2;s_0)}{\sin(\bar \alpha)} \frac{\delta \phi_1}{\delta s_0} -\tilde r(\phi_1)\frac{\delta\phi_1}{\delta s_0} \bigg) \Diracd_{\phi_0 + 2\pi + \frac{\pi}{2}+\Delta \phi} + |P_2 - P_1| \frac{\delta \phi_1}{\delta s_0} \Diracd'_{\phi_0 + 2\pi +\frac{\pi}{2}+\Delta\phi},
\end{align*}
where \newglossaryentry{Diracprime}{name=\ensuremath{\Diracd'},description={derivative of the Dirac delta $\Diracd$}} \gls{Diracprime} is the derivative of the Dirac's delta:
\begin{equation*}
\langle \Diracd'_x, f \rangle = - \frac{df(x)}{dx}.
\end{equation*}
Using \eqref{Equa:DeltaP_arc} and \eqref{Equa:Deltar_arc} we obtain source in the statement.

For $\phi \geq \phi_2 + 2\pi + \bar \alpha$ the equation is linear, and then the conclusion follows.
\end{proof}

We remark that, since
\begin{equation*}
\frac{\cot(\bar \alpha)}{\sin(\bar \alpha)} \left( 1 - \cos(\theta) + \Delta \phi \sin(\theta) \right) \geq 0,
\end{equation*}
being $\Delta\phi\geq 0$, it follows that the perturbation is positive for $\phi \in [0,\phi_0 + \frac{\pi}{2} + \Delta \phi -\bar\alpha]$, as the optimality of Theorem \ref{Cor:curve_cal_R_sat_spiral} implies.

\begin{remark}
\label{Rem:same_as_segment}
Observe that when $\Delta \phi = 0$ the source becomes
\begin{align*}
S(\phi;s_0) &= \frac{\cos(\bar \alpha)}{\sin(\bar \alpha)} (1 - \cos(\theta)) \Diracd_{\phi_0 + \frac{\pi}{2} - \bar \alpha} - \Diracd_{\phi_0 + 2\pi + \theta_0} \\
& \quad + \big( \cos(\theta) - \cot(\bar \alpha)^2 (1 - \cos(\theta)) \big) \Diracd_{\phi_0 + 2\pi + \frac{\pi}{2}} + \big( \sin(\theta) - \cos(\bar \alpha) (1 - \cos(\theta)) \big) \Diracd'_{\phi_0 + 2\pi + \frac{\pi}{2}},
\end{align*}
which is the same as the source of the segment case. Hence the derivative $\delta r$ is continuous when passing from the segment case to the arc case, i.e. when $\bar \theta = \frac{\pi}{2}$.
\end{remark}

Using the kernel \gls{Gkernel} of Remark \ref{Rem:oringal_phi_kernel} we obtain the following explicit representation of the solution $\delta \tilde r$.

\begin{corollary}
\label{Cor:evolv_r_after_phi_2:arc}
The solution $\delta \tilde r(\phi)$ for $\phi \geq \phi_2 = \phi_0 + \frac{\pi}{2} - \bar \alpha + \Delta \phi$ is explicitly given by
\begin{equation*}
\begin{split}
\delta \tilde r(\phi;s_0) &= \frac{\cot(\bar \alpha)}{\sin(\bar \alpha)} \left( 1 - \cos(\theta) + \Delta \phi \sin(\theta) \right) G \bigg( \phi - \phi_0 - \Delta \phi - \frac{\pi}{2} + \bar \alpha \bigg) - G(\phi - \phi_0 - 2\pi - \theta_0) \\
& \quad + \cos(\theta) G \bigg( \phi - \phi_0 - 2\pi - \frac{\pi}{2} \bigg) - \sin(\theta) M \bigg( \phi; \phi_0 + 2\pi + \frac{\pi}{2}; \phi_0 + 2\pi + \frac{\pi}{2} + \Delta \phi \bigg) \\
& \quad + \bigg( - \cot(\bar \alpha)^2 \big( 1 - \cos(\theta) + \Delta \phi \sin (\theta) \big) \\
& \quad \qquad \qquad + \cot(\bar \alpha) \big( \sin(\theta) - \cot(\bar \alpha) (1 - \cos(\theta) + \Delta \phi \sin(\theta)) \big) \bigg) G \bigg( \phi -\phi_0 - 2\pi - \Delta \phi - \frac{\pi}{2} \bigg) \\
& \quad + \big( \sin(\theta) - \cot(\bar \alpha) (1 - \cos(\theta) + \Delta \phi \sin(\theta)) \big) \Diracd_{\phi_0 + 2\pi + \frac{\pi}{2} + \Delta \phi} \\
& \quad - \frac{1}{\sin(\bar \alpha)} \big( \sin(\theta) - \cot(\bar \alpha) (1 - \cos(\theta) + \Delta \phi \sin(\theta)) \big) G \bigg( \phi - \phi_0 - 2\pi - \Delta \phi - \frac{\pi}{2} - 2\pi  - \bar \alpha \bigg).
\end{split}
\end{equation*}
\end{corollary}

We next make the change of variables
\begin{equation*}
\delta \tilde r(\phi_2 + (2\pi + \bar \alpha) \tau;s_0) = \delta \rho(\tau) e^{\bar c (2\pi + \bar \alpha) \tau}, \quad \phi_2 = \phi_0 + \Delta \phi + \frac{\pi}{2} - \bar \alpha,
\end{equation*}
obtaining the following corollary.

\begin{corollary}
\label{Cor:arc_rho_eq}
The values of the perturbation depend only on the angles $\theta = \frac{\pi}{2} - \theta_0$ and $\Delta \phi$, and the function $\delta \rho(\tau)$ is given by
\begin{equation*}
\delta \rho(\tau) = S_0 g(\tau) + S_1 g(\tau - \tau_1) + S_2 g(\tau - \tau_2) + S_3 m(\tau,\tau_2,1) + S_4 g(\tau-1) + D \Diracd_1 + S_5 g(\tau-2),
\end{equation*}
where \newglossaryentry{tau2}{name=\ensuremath{\tau_2},description={time of the second discontinuity in the perturbation in the arc case}}
\begin{equation*}
\gls{tau1} = 1 - \frac{\Delta \phi + \theta}{2\pi + \bar \alpha}, \quad \gls{tau2} = 1 - \frac{\Delta \phi}{2\pi + \bar \alpha},
\end{equation*}
and
\begin{equation*}
S_0(\theta,\Delta \phi) = \frac{\cot(\bar \alpha)}{\sin(\bar \alpha)} \big( 1 - \cos(\theta) + \Delta \phi \sin(\theta) \big),
\end{equation*}
\begin{equation*}
S_1(\theta,\Delta \phi) = - e^{- \bar c(2\pi + \bar \alpha) \tau_1} = - \frac{\sin(\bar \alpha)}{2\pi + \bar \alpha} e^{\bar c(\Delta \phi + \theta)},
\end{equation*}
\begin{equation*}
S_2(\theta,\Delta \phi) = \cos(\theta) e^{-\bar c(2\pi + \bar \alpha) \tau_2} = \frac{\sin(\bar \alpha) \cos(\theta)}{2\pi + \bar \alpha} e^{\bar c \Delta \phi},
\end{equation*}
\begin{equation*}
S_3(\theta) = - \sin(\theta), \qquad m(\tau,\tau_2,1) = \int_{\tau_2}^1 g(\tau - \tau') e^{- \bar c(2\pi + \bar \alpha) \tau'} (2\pi + \bar \alpha) d\tau',
\end{equation*}
\begin{equation*}
S_4(\theta,\Delta \phi) = \frac{\cos(\bar \alpha)}{2\pi + \bar \alpha} \Big( - 2 \cot(\bar \alpha) (1 - \cos(\theta) + \Delta \phi \sin(\theta)) + \sin(\theta) \Big)
\end{equation*}
\begin{equation*}
D(\theta,\Delta \phi) = \frac{\sin(\bar \alpha)}{2\pi + \bar \alpha} \big( \sin(\theta) - \cot(\bar \alpha) (1 - \cos(\theta) + \Delta \phi \sin(\theta)) \big), 
\end{equation*}
\begin{equation*}
S_5(\theta,\Delta \phi) = - \frac{\sin(\bar \alpha)}{(2\pi + \bar \alpha)^2} \Big( \sin(\theta) - \cot(\bar \alpha) (1 - \cos(\theta) + \Delta \phi \sin(\theta)) \Big).
\end{equation*}
\end{corollary}

The next result is essential for the analysis in the following sections: as in the previous case we let $\theta$ to vary in $[0,\pi]$, even if for spiral strategies $\theta \leq \frac{\pi}{2}$.

\begin{proposition}
\label{Prop:neg_region}
The function $\delta \rho = \delta \rho(\tau;\theta,\Delta \phi)$ is strictly positive outside the region \newglossaryentry{h1Deltaphi}{name=\ensuremath{h_1(\Delta \phi)},description={boundary of the first negativity region in the arc case}} \newglossaryentry{h2Deltaphi}{name=\ensuremath{h_2(\Delta \phi)},description={boundary of the first negativity region in the arc case}} \newglossaryentry{h3Deltaphi}{name=\ensuremath{h_3(\Delta \phi)},description={boundary of the first negativity region in the arc case}} \newglossaryentry{tau3thetaDeltaphi}{name=\ensuremath{\tau_3(\theta,\Delta \phi)},description={initial time of the third negativity region in the arc case}} \newglossaryentry{Narc}{name=\ensuremath{N_\mathrm{arc}},description={negative region for the perturbation of the arc case}} \newglossaryentry{Nresarc}{name=\ensuremath{N_\mathrm{res}},description={residual negative region for the arc case}}
\begin{align*}
\gls{Narc} &= \{\theta = 0\} \cup \big\{ \tau_1(\theta,\Delta \phi) \leq \tau < \tau_2(\Delta \phi), 0 \leq \theta \leq \gls{h1Deltaphi} \big\} \\
&\quad \cup \big\{\tau_2(\Delta \phi) \leq \tau < 1, 0 \leq \theta \leq \gls{h2Deltaphi} \big\} \\
&\quad \cup \big\{ \gls{tau3thetaDeltaphi} \leq \tau < 1, h_2(\Delta \phi) \leq \theta < \gls{h3Deltaphi} \big\} \cup \gls{Nresarc}, 
\end{align*}
where the functions $\gls{h1Deltaphi},\gls{h2Deltaphi},\gls{h3Deltaphi},\gls{tau3thetaDeltaphi}$ are determined by solving the equations
\begin{itemize}
\item for $\tau_1 \leq \tau \leq \tau_2$, the function $h_1(\Delta \phi)$ is determined implicitly by
\begin{equation}
\label{Equa:theta_Deltaphi_rel}
\frac{\cot(\bar \alpha)}{\sin(\bar \alpha)} (1 - \cos(\theta) + \sin(\theta) \Delta \phi) e^{\cot(\bar \alpha) (2\pi + \bar \alpha - \Delta \phi - \theta)} - 1 = 0,
\end{equation}
\item for $\tau_2 \leq \tau \leq 1$, the function $h_2(\Delta \phi)$ is determined by
\begin{equation*}
\frac{\cot(\bar \alpha)}{\sin(\bar \alpha)} (1 - \cos(\theta) + \sin(\theta) \Delta \phi) e^{\cot(\bar \alpha) (2\pi + \bar \alpha - \Delta \phi)} - e^{\cot(\bar \alpha) \theta} + \cos(\theta) = 0,
\end{equation*}
\item for $\tau_3 \leq \tau \leq 1$, the function $h_3(\Delta \phi)$ is determined by
\begin{equation*}
\frac{\cot(\bar \alpha)}{\sin(\bar \alpha)} (1 - \cos(\theta) + \sin(\theta) \Delta \phi) e^{\cot(\bar \alpha) (2\pi + \bar \alpha - \Delta \phi)} - e^{\cot(\bar \alpha) \theta} + \cos(\theta) - \sin(\theta) \frac{e^{\cot(\bar \alpha) \Delta \phi} - 1}{\cot(\bar \alpha)} = 0,
\end{equation*}
\item the function $\tau_3(\theta,\Delta \phi) = 1 - \frac{\theta_3}{2\pi + \bar \alpha}$ is given by solving in the region $\{h_2(\Delta \phi) \leq \theta \leq h_3(\Delta \phi)\}$
\begin{equation*}
\frac{\cot(\bar \alpha)}{\sin(\bar \alpha)} (1 - \cos(\theta) + \sin(\theta) \Delta \phi) e^{\cot(\bar \alpha) (2\pi + \bar \alpha - \theta_3)} - e^{\cot(\bar \alpha) (\Delta \phi + \theta - \theta_3)} + \cos(\theta) e^{\cot(\bar \alpha) (\Delta \phi - \theta_3} - \sin(\theta) M(\Delta \phi - \theta_3;0,\Delta \phi) = 0.
\end{equation*}
\end{itemize}
Finally, the region \gls{Nresarc} is contained in the set
\begin{equation*}
N_\mathrm{res} \subset \big\{ \tau \geq 3, \theta \in (3.12,\pi], \Delta \phi \in (2.37,\tan(\bar \alpha)] \big\} \cup \bigg\{ \tau = 1, \Delta \phi \geq \tan(\bar \alpha) - \frac{1 - \cos(\theta)}{\sin(\theta)} \bigg\}.
\end{equation*}

The asymptotic behavior is given by
\begin{equation}
\label{Equa:asympt_arc_sta}
\lim_{\tau \to \infty} \frac{\delta \rho(\tau) e^{-\bar c \tau}}{\tau \theta(\theta + \Delta \phi)} = \frac{2}{\theta ( \theta+ \Delta \phi)} \bigg( S_0 + S_1 + S_2 + S_3 \frac{e^{-\bar c (2\pi + \bar \alpha) \tau_2} - e^{-\bar c(2\pi + \bar \alpha)}}{\bar c} + S_4 + S_5 \bigg).
\end{equation}
\end{proposition}

The region of negativity for the arc case is depicted in Fig. \ref{Fig:nega_arc_case} in the first 3 rounds: the last region is due to the fact that the sum of the sources in \eqref{Equa:asympt_arc_sta} is negative in some region, see Fig. \ref{Fig:asympt_arc} below. We will not use this region in our perturbations.

\begin{figure}
\resizebox{.75\textwidth}{!}{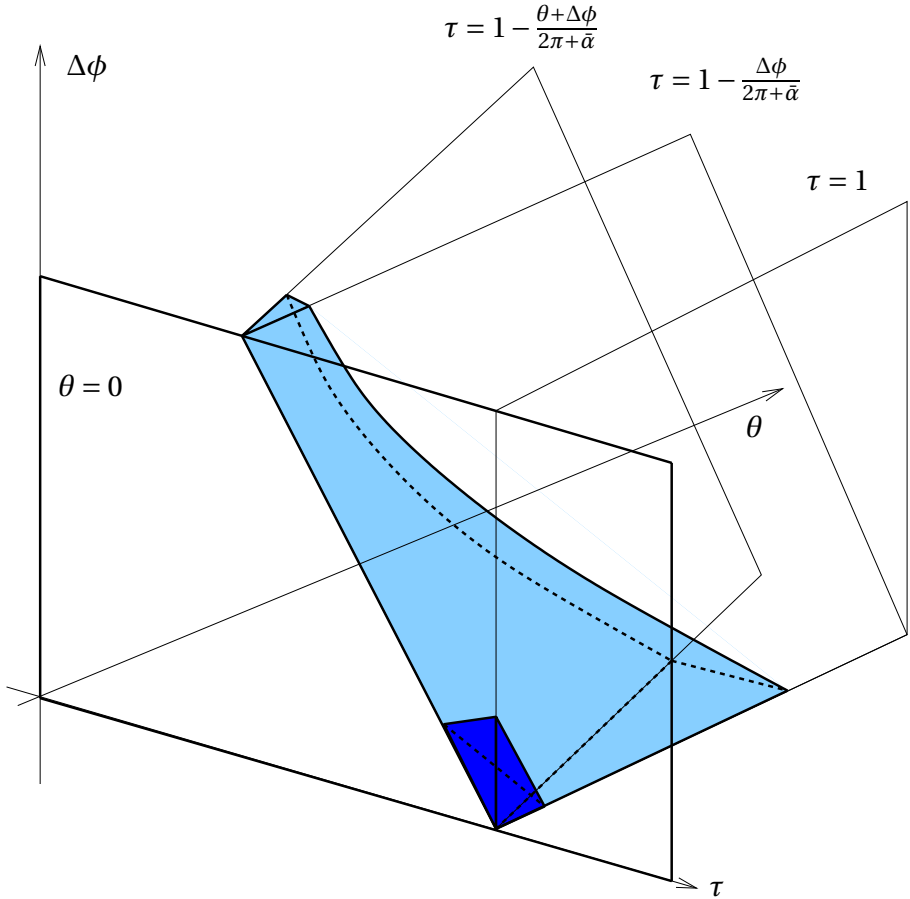}
\caption{The region where $\delta \tilde r \leq 0$ in the arc case for $\tau \in [0,1]$, neglecting $N_\mathrm{res}$ for clarity: in addition to the plane $\{\theta = 0\}$, it contains the cyan region $\{0 \leq \theta \leq h_1(\Delta \phi), \tau_1 \leq \tau \leq \tau_2\}$ and the trapezoidal blue region for $\{\tau_2 \leq \tau \leq 1\}$.}
\label{Fig:nega_arc_case}
\end{figure}

\begin{proof}
The proof will be given in several steps.

\medskip

{\it Case 0: $\theta = 0$.}

\noindent It is easy to see that in this case $\delta \rho(\tau) = 0$, so that $\{\theta = 0\} \subset N_{\mathrm{arc}}$.

\medskip

{\it Case 1: $\tau \in [0,\tau_1)$.}

\noindent Being $S_0 > 0$ for $\theta > 0$, it follows that $\delta \rho(\tau) > 0$ in this region.

\medskip

{\it Case 2: $\tau \in [\tau_1,\tau_2)$.}

\noindent The solution $\delta r(\phi;s_0)$ is written explicitly as
\begin{align*}
\delta r(\phi;s_0) &= \frac{\cot(\bar \alpha)}{\sin(\bar \alpha)} (1 - \cos(\theta) + \sin(\theta) \Delta \phi) e^{\cot(\bar \alpha) (\phi - \phi_2)} - e^{\cot(\bar \alpha) (\phi - \phi_2 - 2\pi - \bar \alpha + \Delta \phi + \theta} \\
&= e^{\cot(\bar \alpha)(\phi - \phi_2 - 2\pi - \bar \alpha + \Delta \phi + \theta} \bigg( \frac{\cot(\bar \alpha)}{\sin(\bar \alpha)} (1 - \cos(\theta) + \sin(\theta) \Delta \phi) e^{\cot(\bar \alpha) (2\pi + \bar \alpha - \Delta \phi - \theta)} - 1 \bigg). 
\end{align*}
The study of the above equation is done in the following lemma.

\begin{lemma}
\label{Lem:study_g_1_arc}
The equation
\begin{equation*}
\frac{\cot(\bar \alpha)}{\sin(\bar \alpha)} (1 - \cos(\theta) + \sin(\theta) \Delta \phi) e^{\cot(\bar \alpha) (2\pi + \bar \alpha - \Delta \phi - \theta)} - 1 = 0
\end{equation*}
determines implicitly a convex function \gls{h1Deltaphi} such that
\begin{equation*}
h_1(\Delta \phi) < \hat \theta \quad \text{unless} \ \Delta \phi = 0,
\end{equation*}
and moreover
\begin{equation*}
\Delta \phi \mapsto \hat \theta(\Delta \phi) := \Delta \phi + h_1(\Delta \phi)
\end{equation*}
is strictly increasing unless $\Delta \phi = 0$ and has derivative $<1$.
\end{lemma}

\begin{figure}
\begin{subfigure}{.475\linewidth}
\includegraphics[width=\linewidth]{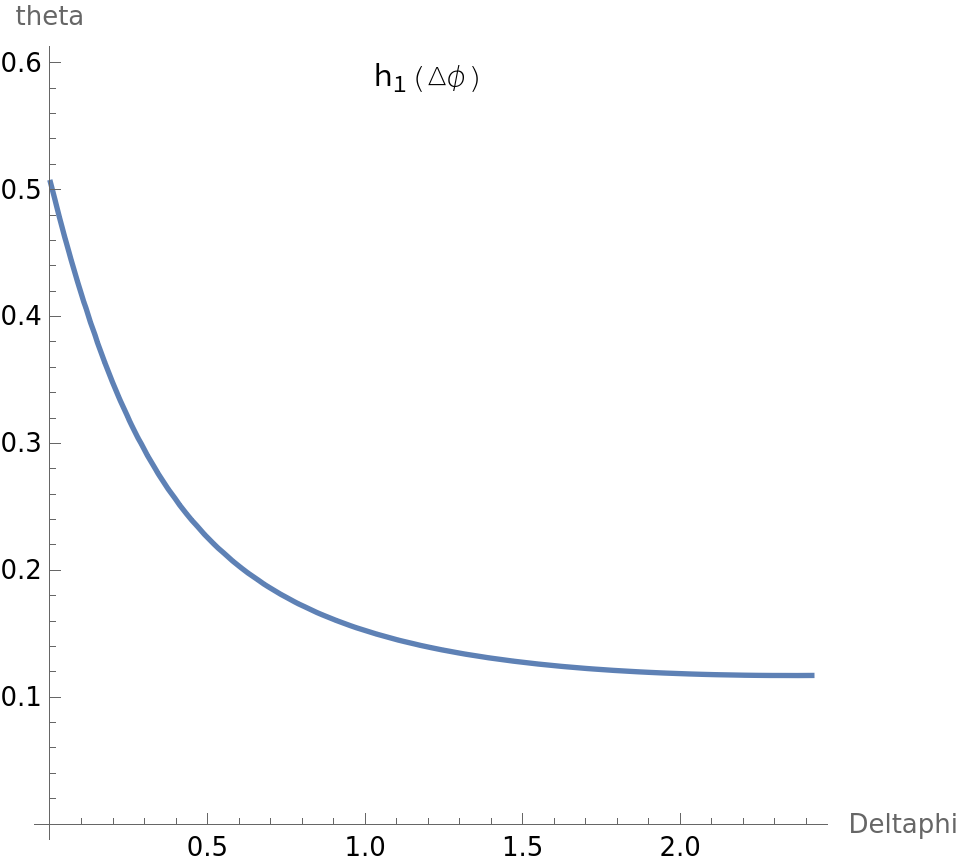}
\caption{Plot of the function $h_1(\Delta \phi)$ of Lemma \ref{Lem:study_g_1_arc}.}
\label{Fig:h1_plot}
\end{subfigure}
\hfill
\begin{subfigure}{.475\linewidth}
\includegraphics[width=\linewidth]{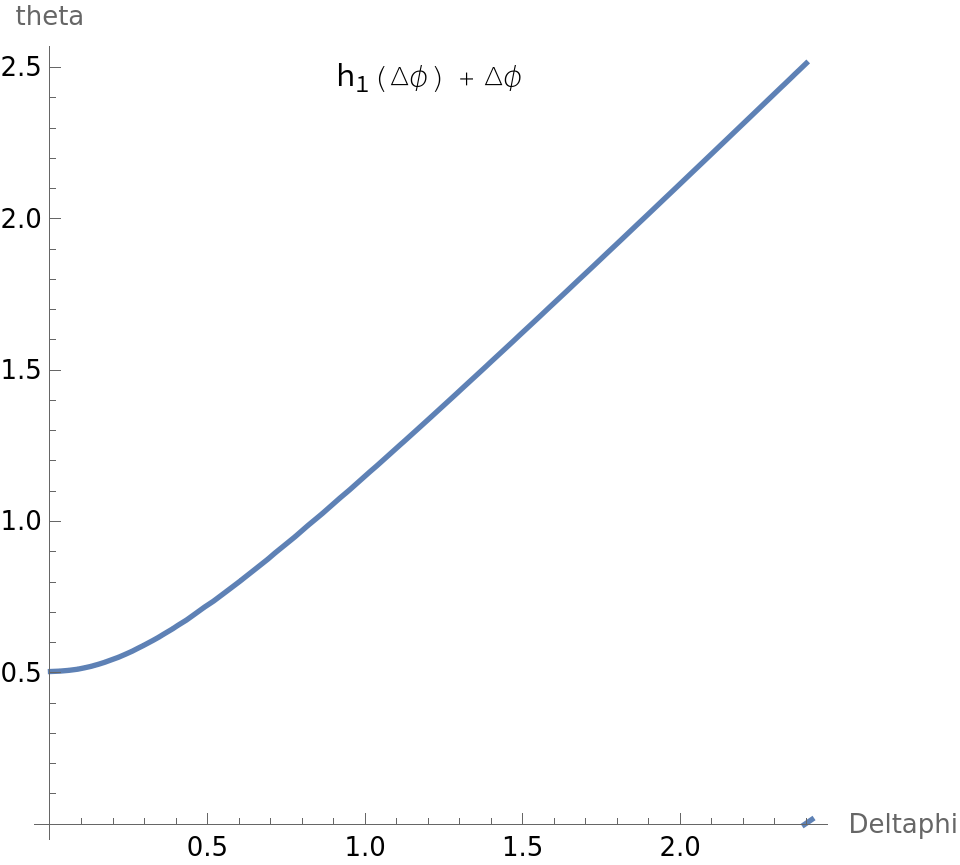}
\caption{Plot of the function $h_1(\Delta \phi) + \Delta \phi$ of Lemma \ref{Lem:study_g_1_arc}.}
\label{Fig:h1Deltaphi_plot}
\end{subfigure}
\caption{Numerical plots of the function $h_1(\Delta \phi)$ of Lemma \ref{Lem:study_g_1_arc}.}
\label{Fig:h1_total}
\end{figure}

A numerical plot of $h_1(\Delta \phi)$ is in Fig. \ref{Fig:h1_plot}, while in Fig. \ref{Fig:h1Deltaphi_plot} it is plotted the function $h_1(\Delta \phi) + \Delta \phi$.

\begin{proof}
Let
\begin{equation}
\label{Equa:f_1_arc_first_neg}
f_1(\theta,\Delta \phi) = \frac{\cot(\bar \alpha)}{\sin(\bar \alpha)} (1 - \cos(\theta) + \sin(\theta) \Delta \phi) e^{\cot(\bar \alpha) (2\pi + \bar \alpha - \Delta \phi - \theta)} - 1.
\end{equation}
Taking the derivative w.r.t. $\theta$ we obtain
\begin{equation}
\label{Equa:partialf_1_arc_first_neg}
\begin{split}
\partial_\theta f_1 &= \frac{\cot(\bar \alpha)}{\sin(\bar \alpha)} \big( \sin(\theta) + \cos(\theta) \Delta \phi - \cot(\bar \alpha) (1 - \cos(\theta) + \sin(\theta) \Delta \phi) \big) e^{\cot(\bar \alpha) (2\pi + \bar \alpha - \Delta \phi - \theta)} \\
&= \frac{\cot(\bar \alpha)}{\sin(\bar \alpha)^2} \big( \cos(\bar \alpha - \theta) - \cos(\bar \alpha) + \sin(\bar \alpha - \theta) \Delta \phi \big) e^{\cot(\bar \alpha) (2\pi + \bar \alpha - \Delta \phi - \theta)}.
\end{split}
\end{equation}
Since numerically one observe that (see Fig. \ref{Fig:partialf_1_arc_first_neg})
\begin{equation*}
f_1 \bigg( \bigg[ \frac{\pi}{4}, \pi \bigg] \times [0,\tan(\bar \alpha)] \bigg) \geq f_1(\pi,\tan(\bar \alpha)) = 0.971097,
\end{equation*}
\begin{equation*}
\partial_\theta f_1 \bigg( \bigg[0,\frac{\pi}{4} \bigg],\Delta \phi \bigg) \geq \partial_\theta f_1 \bigg( \frac{\pi}{4},0 \bigg) \geq 1.29579 (\theta + \Delta \phi),
\end{equation*}
then $\partial_\theta f_1 > 0$ in the region $(0,\frac{\pi}{4}) \times [0,\tan(\bar \alpha)]$ (observe that $f1(0,\Delta \phi) = -1$). We deduce that the equation $f_1(\theta,\Delta \phi) = 0$ defines an implicit function $h_1(\Delta \phi)$, such that
\begin{equation*}
h_1(\Delta \phi) < \hat \theta \quad \text{unless} \ \Delta \phi = 0.
\end{equation*}

\begin{figure}
\begin{subfigure}{.475\linewidth}
\includegraphics[width=\linewidth]{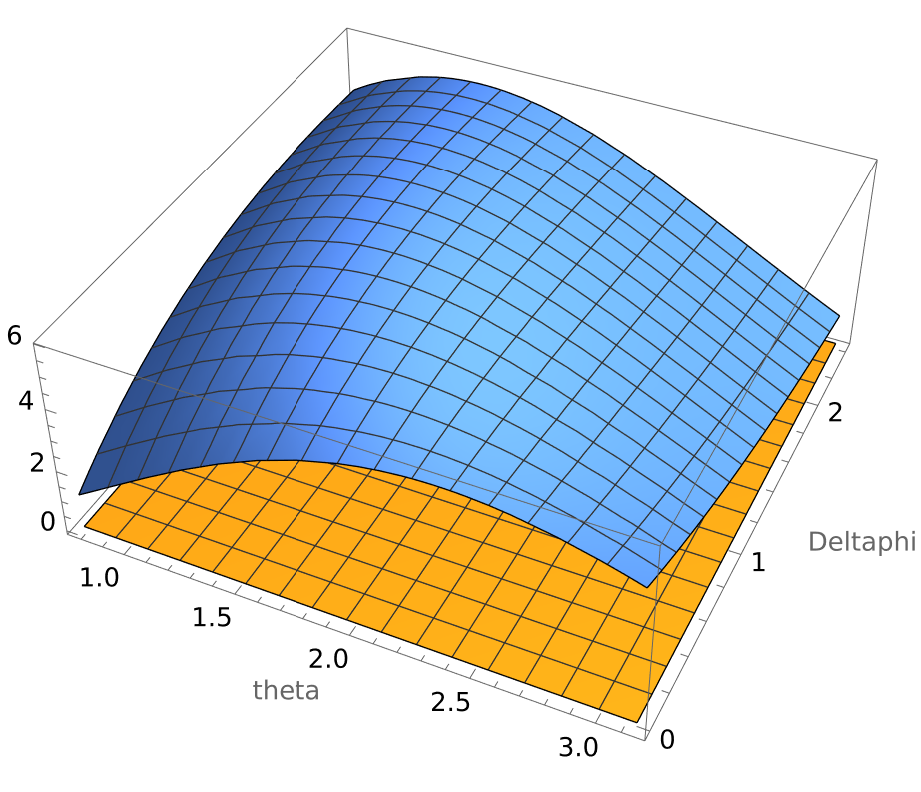}
\caption{Plot of the function $f_1(\theta,\Delta \phi)$, Equation (\ref{Equa:f_1_arc_first_neg}), in the region $[\frac{\pi}{4},\pi] \times [0,\tan(\bar \alpha)]$.}
\label{Fig:f_1_arc_first_neg}
\end{subfigure}\hfill
\begin{subfigure}{.475\linewidth}
\includegraphics[width=\linewidth]{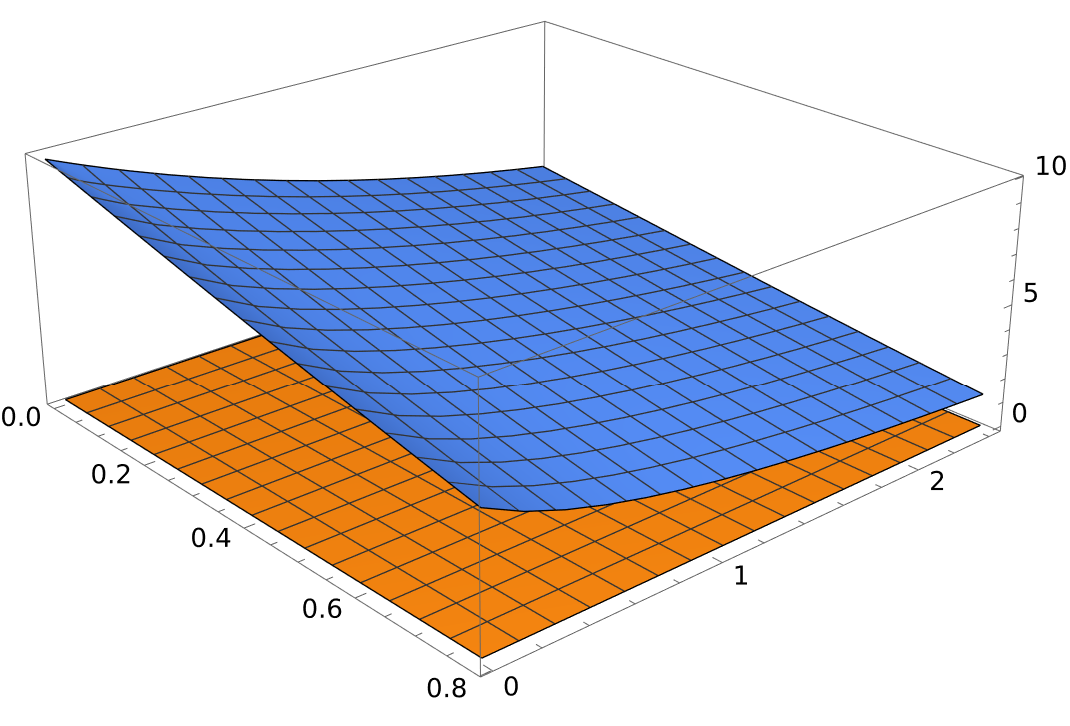}
\caption{Plot of the function $\partial_\theta f_1(\theta,\Delta \phi)$, Equation (\ref{Equa:partialf_1_arc_first_neg}), in the region $[0,\frac{\pi}{4}] \times [0,\tan(\bar \alpha)]$.}
\label{Fig:partialf_1_arc_first_neg}
\end{subfigure}
\medskip
\begin{subfigure}{.475\linewidth}
\includegraphics[width=\linewidth]{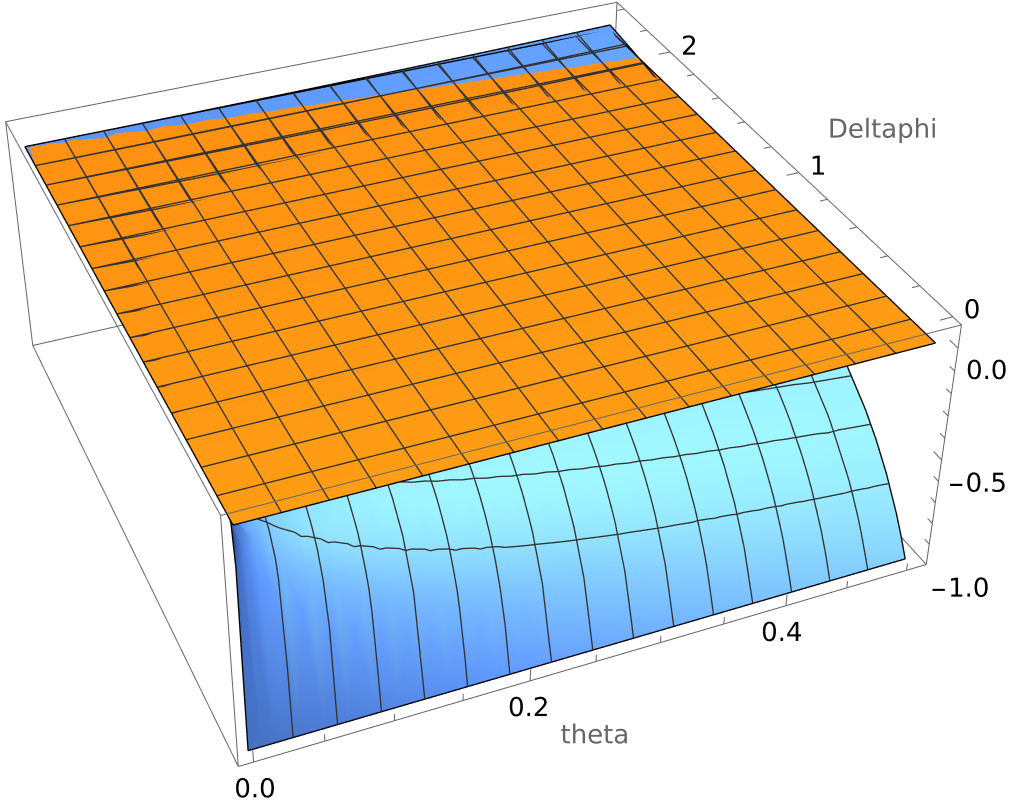}
\caption{Plot of the function (\ref{Equa:dh1_dDeltaphi}).}
\label{Fig:dh1_dDeltaphi}
\end{subfigure}\hfill
\begin{subfigure}{.475\linewidth}
\includegraphics[width=\linewidth]{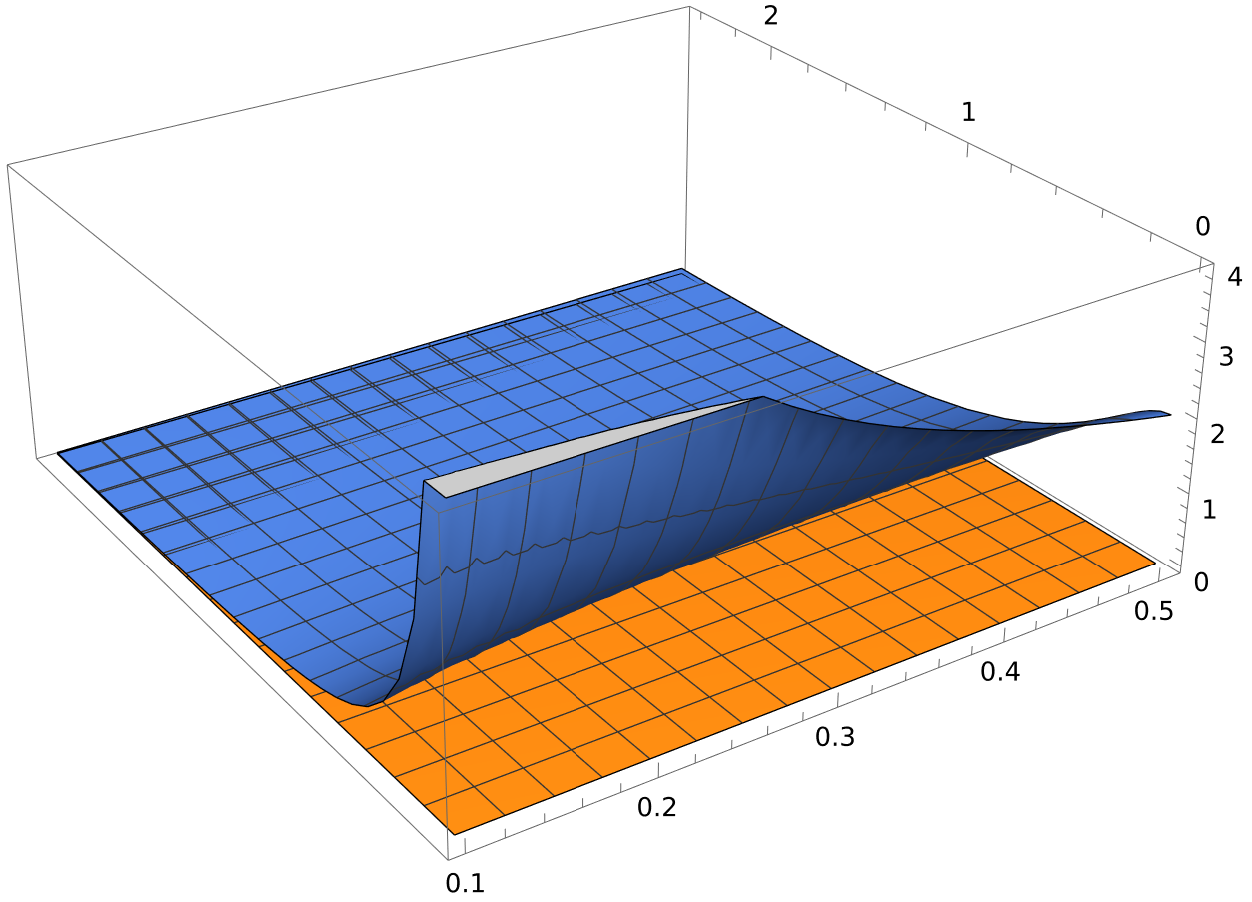}
\caption{Plot of the function (\ref{Equa:d2h1_dDeltaphi}).}
\label{Fig:d2h1_dDeltaphi}
\end{subfigure}
\caption{Numerical plots for the proof of Lemma \ref{Lem:study_g_1_arc}.}
\label{Fig:ddh1}
\end{figure}

By the Implicit Function Theorem we obtain
\begin{equation}
\label{Equa:dh1_dDeltaphi}
\begin{split}
\frac{dh_1}{d\Delta \phi} &= - \frac{\cos(\bar \alpha - \theta) - \cos(\bar \alpha) - \cos(\bar \alpha) \sin(\theta) \Delta \phi}{\cos(\bar \alpha - \theta) - \cos(\bar \alpha) + \sin(\bar \alpha - \theta) \Delta \phi} \\
&= - 1 + \frac{\sin(\bar \alpha) \cos(\theta) \Delta \phi}{\cos(\bar \alpha - \theta) - \cos(\bar \alpha) + \sin(\bar \alpha - \theta) \Delta \phi}.
\end{split}
\end{equation}
so that in the region of interest one can check numerically that (see Fig. \ref{Fig:dh1_dDeltaphi})
\begin{equation*}
-1 \leq \frac{dh_1(\Delta \phi)}{d\Delta \phi} < - 1 + \frac{\sin(\bar \alpha) \cos(\hat \theta) \tan(\bar \alpha)}{\cos(\bar \alpha - \hat \theta) - \cos(\bar \alpha) + \sin(\bar \alpha - \hat \theta) \tan(\bar \alpha)} = 0.0251844.
\end{equation*}
Differentiating w.r.t $\Delta \phi$ one obtains
\begin{equation}
\label{Equa:d2h1_dDeltaphi}
\begin{split}
\frac{d^2h_1}{d\Delta \phi^2} &= \partial_{\Delta \phi} \frac{dh_1}{d\Delta \phi} + \partial_{\theta} \frac{dh_1}{d\Delta \phi} \frac{dh_1}{d\Delta \phi} \\
&= \frac{\sin(\bar \alpha) \cos(\theta) (\cos(\bar \alpha - \theta) - \cos(\bar \alpha))}{(\cos(\bar \alpha - \theta) - \cos(\bar \alpha) + \sin(\bar \alpha - \theta) \Delta \phi)^2} \\
& \quad + \frac{\sin(\bar \alpha) \Delta \phi (- \sin(\bar \alpha) - \cos(\bar \alpha) \sin(\theta) + \cos(\bar \alpha) \Delta \phi)}{(\cos(\bar \alpha - \theta) - \cos(\bar \alpha) + \sin(\bar \alpha - \theta) \Delta \phi)^2} \bigg( - 1 + \frac{\sin(\bar \alpha) \cos(\theta) \Delta \phi}{\cos(\bar \alpha - \theta) - \cos(\bar \alpha) + \sin(\bar \alpha - \theta) \Delta \phi} \bigg),
\end{split}
\end{equation}
and numerically one sees that (see Fig. \ref{Fig:d2h1_dDeltaphi})
\begin{equation*}
\frac{d^2 h_1}{d\Delta \phi^2} > 0.0168525,
\end{equation*}
so that the function is convex, with a minimum near $\Delta \phi = 2.35681$.

Finally, the only case for which $\frac{dh_1}{d\Delta \phi} = -1$ is for $\Delta \phi = 0$, $g_3(\Delta \phi) = \hat \theta$. This means that the function
\begin{equation*}
\delta \phi \mapsto \hat \theta(\Delta \phi) = \Delta \phi + h_1(\Delta \phi)
\end{equation*}
is strictly increasing, convex and has derivative $0$ for $\Delta \phi = 0$, see Fig. \ref{Fig:h1Deltaphi_plot}.
\end{proof}

\medskip

{\it Case 3: $\tau \in [\tau_2,1)$.}

\noindent The solution to be studied is
\begin{align*}
\delta \tilde r(\phi;s_0) &= \frac{\cot(\bar \alpha)}{\sin(\bar \alpha)} (1 - \cos(\theta) + \sin(\theta) \Delta \phi) e^{\cot(\bar \alpha) (\phi - \phi_2)} - e^{\cot(\bar \alpha) (\phi - \phi_2 - 2\pi - \bar \alpha + \Delta \phi + \theta)} \\
&\quad + \cos(\theta) e^{\cot(\bar \alpha) (\phi - \phi_2 - 2\pi - \bar \alpha + \Delta \phi)} - \sin(\theta) \int_{\phi_2 + 2\pi + \bar \alpha + \Delta \phi}^\phi e^{\cot(\bar \alpha) (\phi - \phi')} d\phi' \\
&= e^{\cot(\bar \alpha)(\phi - \phi_2 - 2\pi - \bar \alpha + \theta)} \bigg[ \frac{\cot(\bar \alpha)}{\sin(\bar \alpha)} (1 - \cos(\theta) + \sin(\theta) \Delta \phi) e^{\cot(\bar \alpha) (2\pi + \bar \alpha - \theta)} - e^{\cot(\bar \alpha) \theta} + \cos(\theta) \bigg] \\
& \quad - \sin(\theta) \tan(\bar \alpha) \big( e^{\cot(\bar \alpha) (\phi - \phi_2 - 2\pi - \bar \alpha + \theta)} - 1 \big).
\end{align*}
At the initial angle $\phi = \phi_2 + 2\pi + \bar \alpha - \Delta \phi$ we have
\begin{equation}
\label{Equa:f2_arc}
\begin{split}
f_2(\theta,\Delta \phi) &= \frac{\cot(\bar \alpha)}{\sin(\bar \alpha)} (1 - \cos(\theta) + \sin(\theta) \Delta \phi) e^{\cot(\bar \alpha) (2\pi + \bar \alpha - \Delta \phi)} - e^{\cot(\bar \alpha) \theta} + \cos(\theta) \\
&= f_1(\theta,\Delta \phi) e^{\cot(\bar \alpha) \theta} + \cos(\theta).
\end{split}
\end{equation}
Numerical plots (see Fig. \ref{Fig:f_2_arc}) show that $f_2 < 0$ is a triangular shaped region $N_2$ with corners
\begin{equation*}
(0,0), \quad (0.0956622,0), \quad (0,0.0428438).
\end{equation*}

\begin{figure}
\begin{subfigure}{.475\linewidth}
\includegraphics[width=\linewidth]{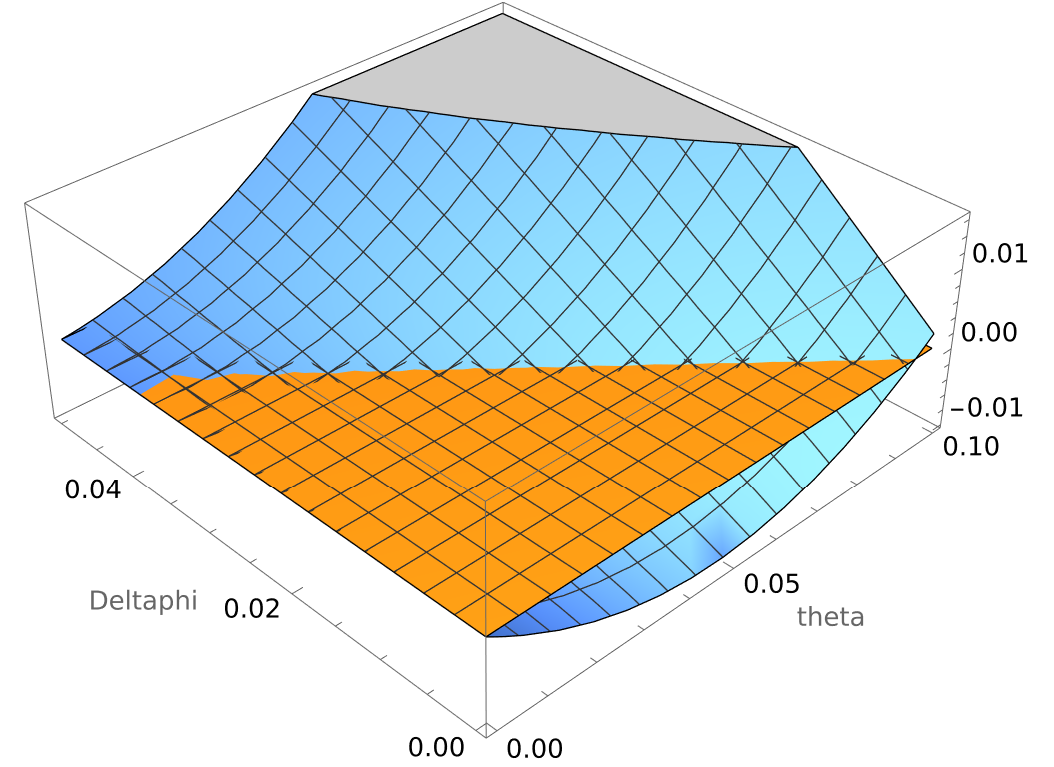}
\caption{Plot of the function $f_2(\theta,\Delta \phi)$ (\ref{Equa:f2_arc}).}
\label{Fig:f_2_arc}
\end{subfigure}\hfill
\begin{subfigure}{.475\linewidth}
\includegraphics[width=\linewidth]{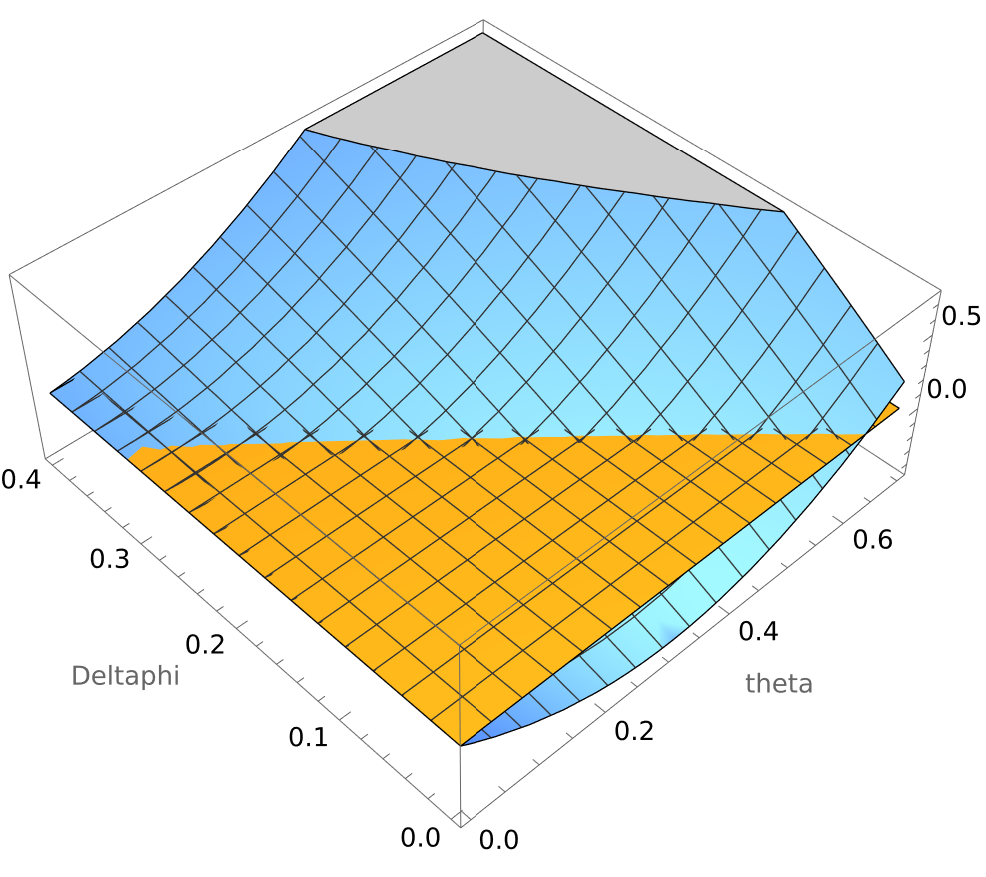}
\caption{Plot of the function $f_3(\theta,\Delta \phi)$ (\ref{Equa:f3_arc}).}
\label{Fig:f_3_arc}
\end{subfigure}
\medskip
\begin{subfigure}{.475\linewidth}
\includegraphics[width=\linewidth]{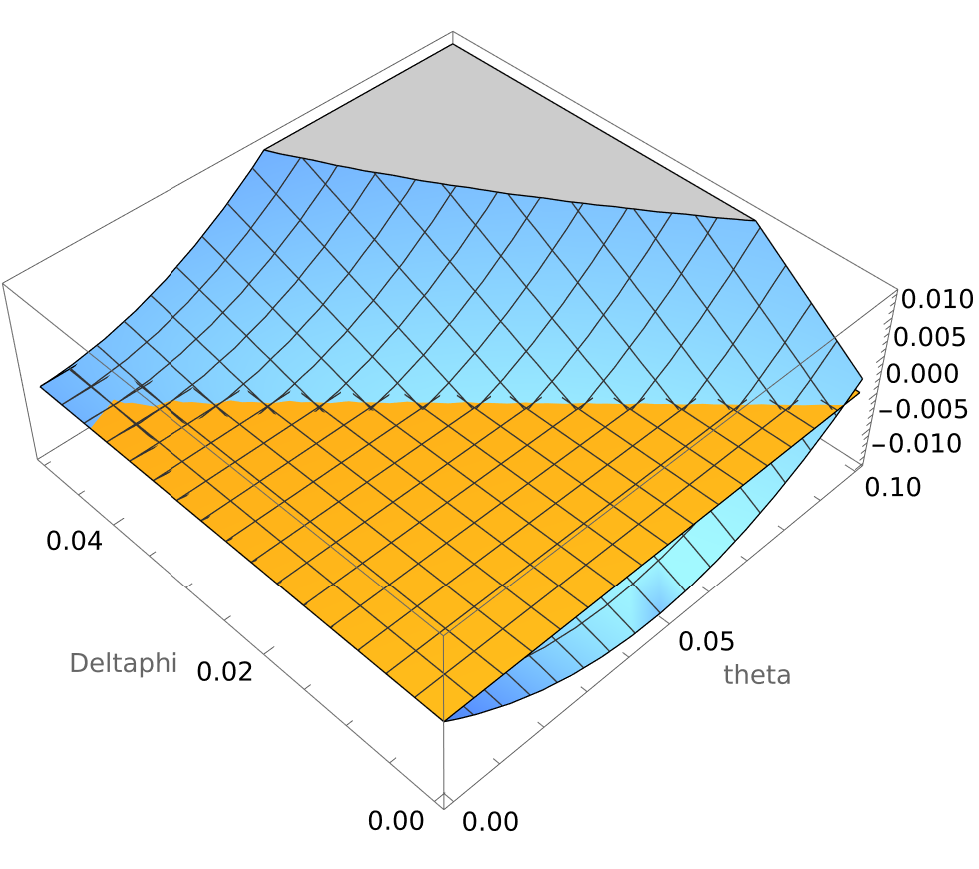}
\caption{Plot of the function (\ref{Equa:f4_arc}).}
\label{Fig:f_4_arc}
\end{subfigure}\hfill
\begin{subfigure}{.475\linewidth}
\includegraphics[width=\linewidth]{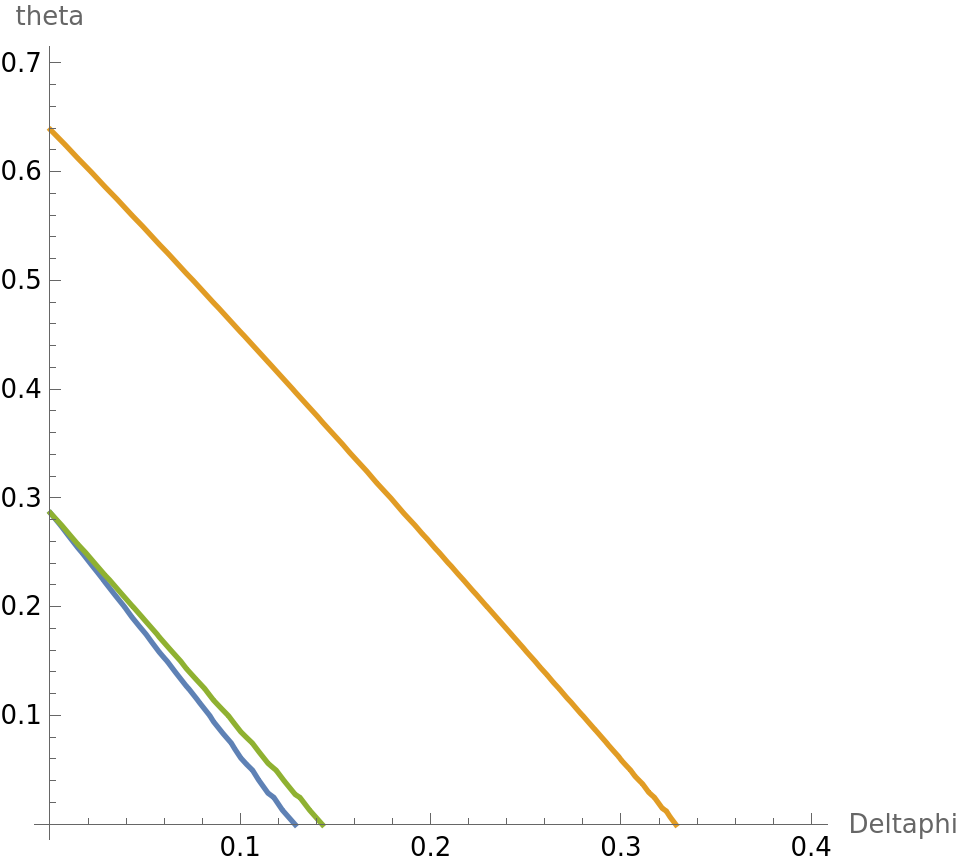}
\caption{Plot of $f_2(\cdot/3) = 0$ (blue), $f_3 = 0$ (orange), $f_4(\cdot/3) = 0$ (green). We have rescaled the first and the third for clarity.}
\label{Fig:f234}
\end{subfigure}
\caption{Numerical plots for the analysis of the negative region with $\tau \in [\tau_2,1]$.}
\label{Fig:f234_tot}
\end{figure}

Next, we take the derivative of $\delta \tilde r(\phi;s_0)$ w.r.t. $\phi$, obtaining
\begin{equation}
\label{Equa:f3_arc}
\begin{split}
\partial_\phi \delta \tilde r(\phi;s_0) &= e^{\cot(\bar \alpha)(\phi - \phi_2 - 2\pi - \bar \alpha + \Delta \phi)} \cot(\bar \alpha) \bigg[ \frac{\cot(\bar \alpha)}{\sin(\bar \alpha)} (1 - \cos(\theta) + \sin(\theta) \Delta \phi) e^{\cot(\bar \alpha) (2\pi + \bar \alpha - \Delta \phi)} \\
& \quad \qquad \qquad \qquad \qquad \qquad \qquad \qquad \qquad - e^{\cot(\bar \alpha) \theta} + \cos(\theta) - \sin(\theta) \tan(\bar \alpha) \bigg] \\
&= e^{\cot(\bar \alpha)(\phi - \phi_2 - 2\pi - \bar \alpha + \Delta \phi)} \cot(\bar \alpha) f_3(\theta,\Delta \phi).
\end{split}
\end{equation}
Numerically (see Fig. \ref{Fig:f_3_arc}), the negativity region is a triangular shaped region $N_3$ with corners
\begin{equation*}
    (0,0), \quad (0.639078,0), \quad (0,0.329784).
\end{equation*}
Being this region larger than the negativity region for $f_2$, it follows that
\begin{equation*}
    \big\{ \tau \in [\tau_2,1], (\theta,\Delta \phi) \in N_2 \big\} \subset N_{\mathrm{arc}},
\end{equation*}
and
\begin{equation*}
    \big\{ \tau \in [\tau_2,1], (\theta,\Delta \phi) \notin N_3 \big\} \cap N_\mathrm{arc} = \emptyset. 
\end{equation*}
In particular the derivative of $\delta \tilde r(\phi;s_0)$ is strictly negative in the region of interest, and thus this determines a function $\tau_2(\theta,\Delta \phi)$ whose graph is the boundary of the region of negativity.

The maximal region is computed by considering
\begin{equation}
\label{Equa:f4_arc}
\begin{split}
\delta \tilde r(\phi_2 + 2\pi + \bar \alpha -;s_0) &= \frac{\cot(\bar \alpha)}{\sin(\bar \alpha)} (1 - \cos(\theta) + \sin(\theta) \Delta \phi) e^{\cot(\bar \alpha) (2\pi + \bar \alpha)} - e^{\cot(\bar \alpha) (\theta + \Delta \phi)} \\
& \quad + \cos(\theta) e^{\cot(\bar \alpha) \Delta \phi} - \sin(\theta) \tan(\bar \alpha) \big( e^{\cot(\bar \alpha) \Delta \phi} - 1 \big) \\
&= f_4(\theta,\Delta \phi).
\end{split}
\end{equation}
Numerically (see Fig. \ref{Fig:f_4_arc}) the negativity region is triangular shaped with corners
\begin{equation*}
    (0,0), \quad (0,0.0478458), \quad (0.0956622,0).
\end{equation*}

The contours of the three regions of negativity are in Fig. \ref{Fig:f234}.

\medskip

{\it Case 4: $\tau \in [1,1+\tau_1)$} 

\noindent The solution $\delta \rho$ is
\begin{align*}
\delta \rho(\tau) &= S_0 (e^{\tau} - e^{\tau - 1} (\tau - 1)) + S_1 e^{\tau - \tau_1} + S_2 e^{\tau - \tau_2} + S_3 e^{- \bar c (2\pi + \bar \alpha)} \frac{e^{\cot(\bar \alpha) \Delta \phi} - 1}{\cot(\bar \alpha)} e^{\tau - 1} + S_4 e^{\tau - 1} \\
&= e^{\tau - 1} \bigg[ S_0 (e - (\tau - 1)) + S_1 e^{1 - \tau_1} + S_2 e^{1 - \tau_2} + S_3 e^{-\bar c(2\pi + \bar \alpha)} \frac{e^{\cot(\bar \alpha) \Delta \phi} - 1}{\cot(\bar \alpha)} + S_4 \bigg].
\end{align*}
The worst case is for $\tau = 1 + \tau_1$, so that we have to study the function
\begin{align*}
f(\theta,\Delta \phi) &= S_0 (e - \tau_1) + S_1 e^{1 - \tau_1} + S_2 e^{1 - \tau_2} + S_3 \frac{\sin(\bar \alpha)}{2\pi + \bar \alpha} \frac{e^{\cot(\bar \alpha) \Delta \phi} - 1}{\cot(\bar \alpha)} + S_4.
\end{align*}
One observes numerically that the quantity inside the bracket is strictly positive ($> 0.44$), and also numerically one obtains that
\begin{equation*}
\frac{f(\theta,\Delta \phi)}{\theta(\theta + \Delta \phi)} \geq \frac{f(\pi,\tan(\bar \alpha))}{\pi(\pi + \tan(\bar \alpha)} = 0.031,
\end{equation*}
and that (see Fig. \ref{Fig:second_round_arc_1})
\begin{equation*}
\frac{\delta \rho([1,1+\tau_1),\theta,\Delta \phi)}{\theta(\theta + \Delta \phi)} \geq \frac{\delta \rho(1+\tau_1,\pi,\tan(\bar \alpha))}{\pi(\pi+\tan(\bar \alpha)} > 0.04.
\end{equation*}

\begin{figure}
\begin{subfigure}{.475\linewidth}
\includegraphics[width=\linewidth]{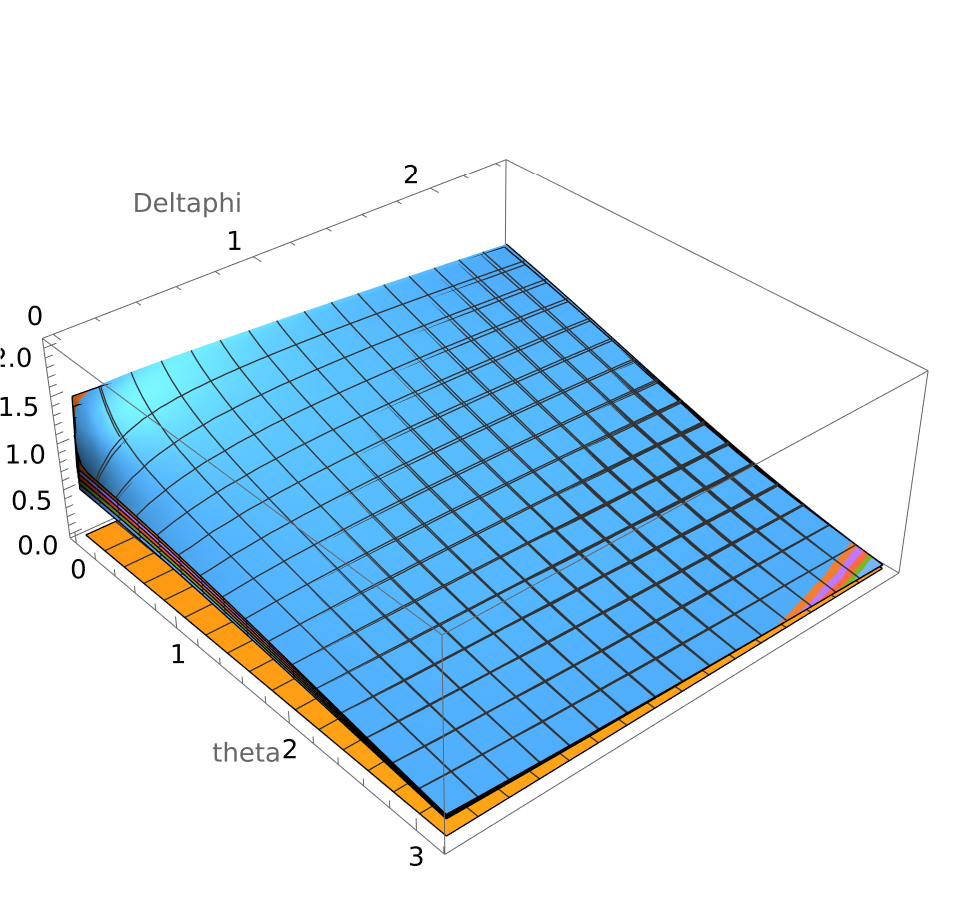}
\caption{Plot of the function $\frac{\delta \rho(1 + s \tau_1,\theta,\Delta \phi)}{\theta (\theta + \Delta \phi)}$ for $s$ assuming the values $\{0,.2,.4,.6,.8,1\}$ (monotonically bottom to top graph).}
\label{Fig:rho_first_1_arc}
\end{subfigure}\hfill
\begin{subfigure}{.475\linewidth}
\includegraphics[width=\linewidth]{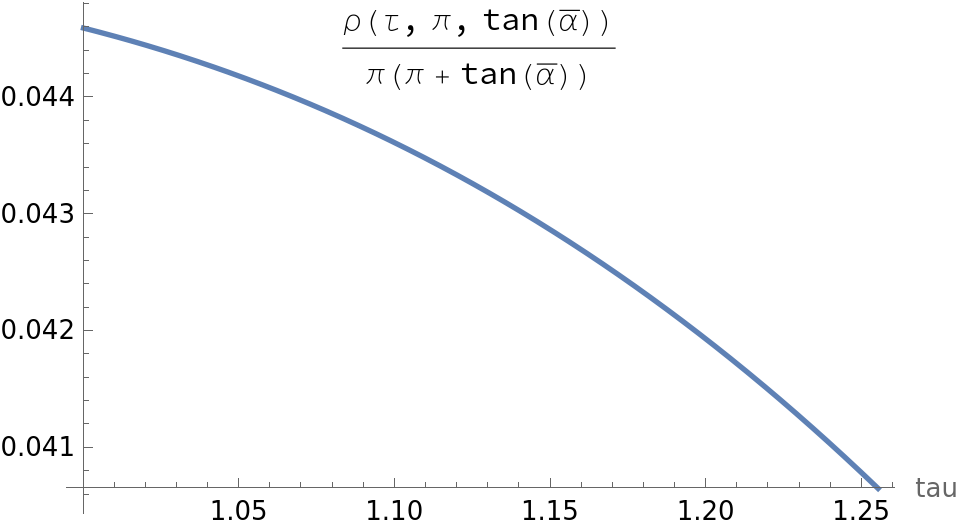}
\caption{Plot of the worst case $\frac{\delta \rho(\tau,\pi,\tan(\bar \alpha))}{\pi(\pi + \tan(\bar \alpha))}$, $\tau \in [1,1+\tau_1)$.}
\label{Fig:rho_first_1_arc_worst}
\end{subfigure}
\caption{Study of $\delta \rho$ for $\tau \in [1,1+\tau_1)$.}
\label{Fig:second_round_arc_1}
\end{figure}

\medskip

{\it Case 5: $\tau \in 1 + [\tau_1,\tau_2)$.}

\noindent The solution is given by
\begin{align*}
\delta \rho(\tau) &= S_0 (e^{\tau} - e^{\tau - 1} (\tau - 1)) + S_1 \big( e^{\tau - \tau_1} - e^{\tau - \tau_1 - 1} (\tau - \tau_1 -1) \big) + S_2 e^{\tau - \tau_2} + S_3 e^{- \bar c(2\pi + \bar \alpha)} \frac{e^{\cot(\bar \alpha) \Delta \phi} - 1}{\cot(\bar \alpha)} e^{\tau - 1} + S_4 e^{\tau - 1} \\
&= e^{\tau - 1} \bigg[ S_0 (e - (\tau - 1)) + S_1 (e^{1 - \tau_1} - e^{-\tau_1} (\tau - \tau_1 - 1)) + S_2 e^{1 - \tau_2} + S_3 e^{-\bar c(2\pi + \bar \alpha)} \frac{e^{\cot(\bar \alpha) \Delta \phi} - 1}{\cot(\bar \alpha)} + S_4 \bigg].
\end{align*}
A numerical evaluation shows that (see Fig. \ref{Fig:second_round_arc_2})
\begin{equation*}
\frac{\delta \rho([1+\tau_1,1+\tau_2),\theta,\Delta \phi)}{\theta(\theta + \Delta \phi)} \geq \frac{\delta \rho(\tau_1,\pi,\tan(\bar \alpha))}{\pi(\pi + \tan(\bar \alpha))} > 0.04.
\end{equation*}

\begin{figure}
\begin{subfigure}{.475\linewidth}
\includegraphics[width=\linewidth]{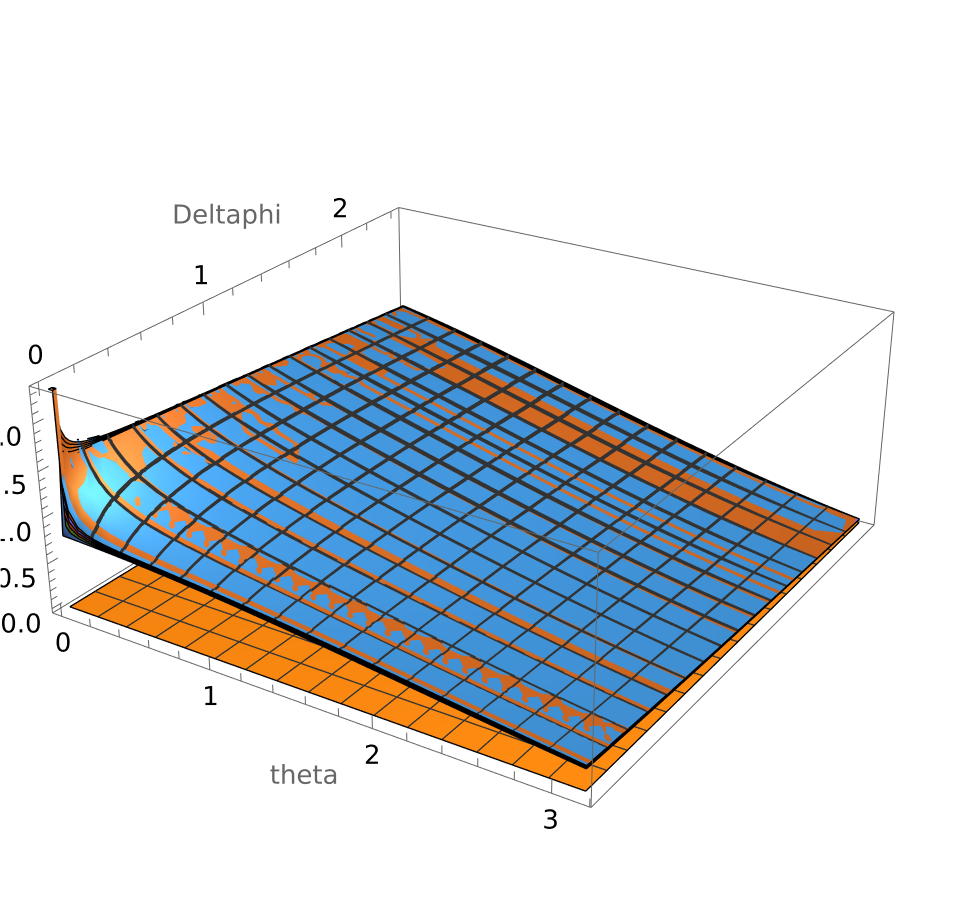}
\caption{Plot of the function $\frac{\delta \rho(1 + (1-s) \tau_1 + s \tau_2,\theta,\Delta \phi)}{\theta (\theta + \Delta \phi)}$ for $s$ assuming the values $\{0,.2,.4,.6,.8,1\}$ (monotonically bottom to top graph).}
\label{Fig:rho_first_2_arc}
\end{subfigure}\hfill
\begin{subfigure}{.475\linewidth}
\includegraphics[width=\linewidth]{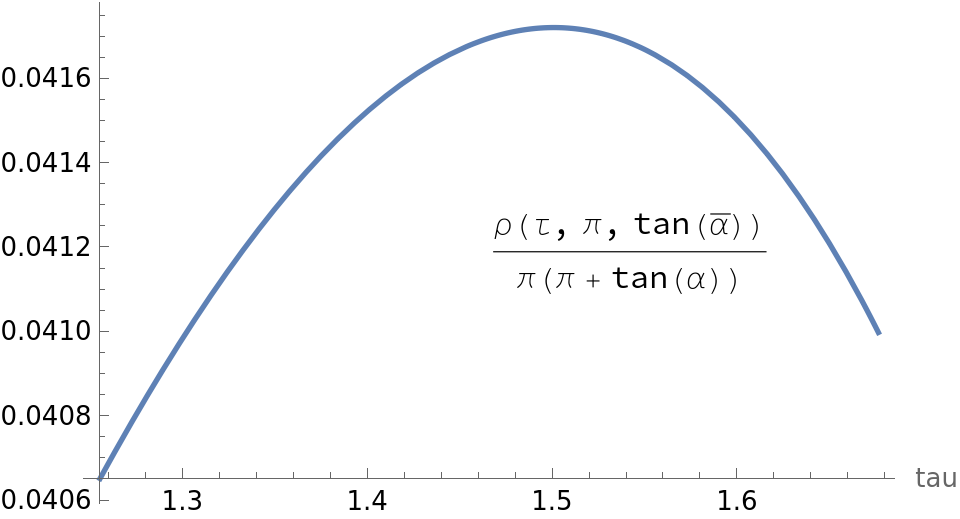}
\caption{Plot of the worst case $\frac{\delta \rho(\tau,\frac{\pi}{2},\tan(\bar \alpha))}{\frac{\pi}{2}(\frac{\pi}{2} +\tan(\bar \alpha))}$, $\tau \in [1+\tau_1,1+\tau_2)$.}
\label{Fig:rho_first_2_arc_worst}
\end{subfigure}
\caption{Study of $\delta \rho$ for $\tau \in [1+\tau_1,1+\tau_2)$.}
\label{Fig:second_round_arc_2}
\end{figure}

\medskip

{\it Case 6: $\tau \in [1+\tau_2,2]$.}

\noindent The solution is given by 
\begin{align*}
\delta \rho(\tau) &= S_0 (e^{\tau} - e^{\tau - 1} (\tau - 1)) + S_1 \big( e^{\tau - \tau_1} - e^{\tau - \tau_1 - 1} (\tau - \tau_1 -1) \big) + S_2 \big( e^{\tau - \tau_2} - e^{\tau - \tau_2 - 1} (\tau - \tau_2 - 1) \big) \\
& \quad + S_3 e^{- \bar c(2\pi + \bar \alpha)} \bigg( \frac{e^{\cot(\bar \alpha) \Delta \phi} - 1}{\cot(\bar \alpha)} e^{\tau - 1} \\
& \quad \qquad \qquad \qquad \qquad \qquad - \frac{e^{\cot(\bar \alpha) \Delta \phi} \cot(\bar \alpha) (2\pi + \bar \alpha) (\tau - 2 + \frac{\Delta \phi}{2\pi + \bar \alpha}) - e^{\cot(\bar \alpha) \Delta \phi} + e^{- \cot(\bar \alpha) (2\pi + \bar \alpha) (\tau - 2)}}{(2\pi + \bar \alpha) \cot(\bar \alpha)^2} e^{\tau - 2} \bigg) \\
& \quad + S_4 e^{\tau - 1}.
\end{align*}
Numerically one observes that (see Fig. \ref{Fig:second_round_arc_3})
\begin{equation*}
\frac{\delta \rho([1+\tau_2,2),\theta,\Delta \phi)}{\theta(\theta + \Delta \phi)} \geq \frac{\delta \rho(1+\tau_2,\pi,\tan(\bar \alpha))}{\pi(\pi + \tan(\bar \alpha))} > 0.04.
\end{equation*}

\begin{figure}
\begin{subfigure}{.475\linewidth}
\includegraphics[width=\linewidth]{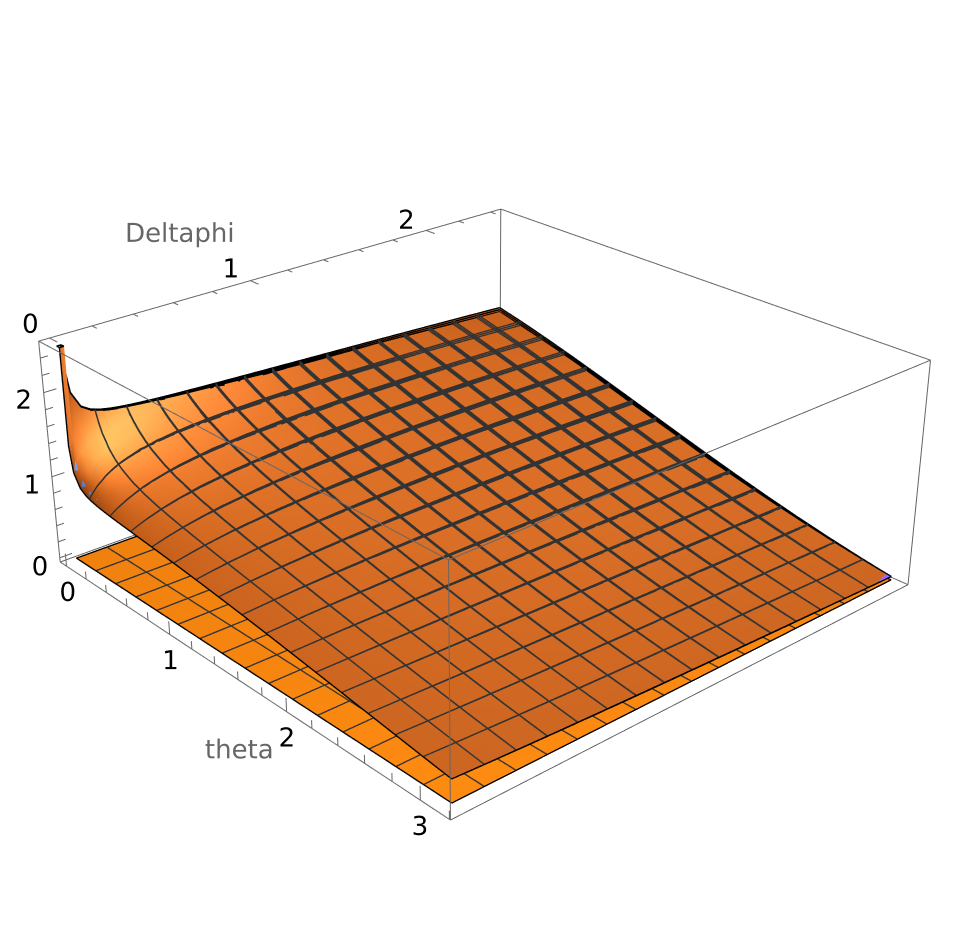}
\caption{Plot of the function $\frac{\delta \rho(1 + (1-s) \tau_2 + s,\theta,\Delta \phi)}{\theta (\theta + \Delta \phi)}$ for $s$ assuming the values $\{0,.2,.4,.6,.8,1\}$ (monotonically bottom to top graph).}
\label{Fig:rho_first_3_arc}
\end{subfigure}\hfill
\begin{subfigure}{.475\linewidth}
\includegraphics[width=\linewidth]{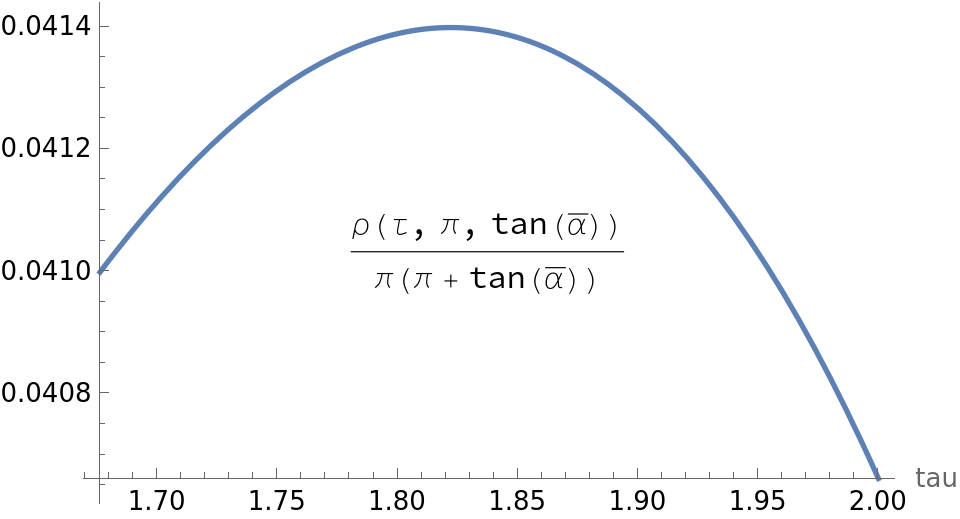}
\caption{Plot of the worst case $\frac{\delta \rho(\tau,\pi,\tan(\bar \alpha))}{\pi(\pi + \tan(\bar \alpha))}$, $\tau \in [1+\tau_2,2)$.}
\label{Fig:rho_first_3_arc_worst}
\end{subfigure}
\caption{Study of $\delta \rho$ for $\tau \in [1+\tau_2,2)$.}
\label{Fig:second_round_arc_3}
\end{figure}

\medskip

{\it Case 7: $\tau \in [2,2+\tau_1]$.}

\noindent The explicit form of the solution is
\begin{align*}
\delta \rho(\tau) &= S_0 \bigg( e^{\tau} - e^{\tau - 1} (\tau - 1) + e^{\tau-2}\frac{(\tau - 2)^2}{2} \bigg) + S_1 \big( e^{\tau - \tau_1} - e^{\tau - \tau_1 - 1} (\tau - \tau_1 -1) \big) + S_2 \big( e^{\tau - \tau_2} - e^{\tau - \tau_2 - 1} (\tau - \tau_2 - 1) \big) \\
& \quad + S_3 \bigg[ \frac{e^{- \bar c \tau}}{\cot(\bar \alpha)} \big( e^{\cot(\bar \alpha)(2\pi + \bar \alpha) (\tau - \tau_2)} - e^{\cot(\bar \alpha) (2\pi + \bar \alpha) (\tau-1)} \big) \\
& \quad \qquad \qquad + \frac{e^{\bar c(2\pi + \bar \alpha)(\tau-1)}}{\cot(\bar \alpha)^2 (2\pi + \bar \alpha)} \bigg( - e^{-\cot(\bar \alpha) (\tau - \tau_2 - 1)} \cot(\bar \alpha) (2\pi + \bar \alpha) (\tau - \tau_2 - 1) \\
& \quad \qquad \qquad \qquad \qquad \qquad + e^{-\cot(\bar \alpha) (\tau - 2)} \cot(\bar \alpha) (2\pi + \bar \alpha) (\tau - 2) + e^{\cot(\bar \alpha) (2\pi + \bar \alpha) (\tau - \tau_2 - 1} - e^{\cot(\bar \alpha) (2\pi + \bar \alpha) (\tau - 2)} \bigg) \bigg] \\
& \quad + S_4 \big( e^{\tau - 1} - e^{\tau-2} (\tau -2) \big) + S_5 e^{\tau-2}.
\end{align*}
Numerical evaluations show that (see Fig. \ref{Fig:third_round_arc_1})
\begin{equation*}
\frac{\delta \rho([2,2+\tau_1),\theta,\Delta \phi)}{\theta(\theta+\Delta \phi)} \geq \frac{\delta \rho(2,\pi,\tan(\bar \alpha))}{\pi(\pi + \tan(\bar \alpha))} > 0.04,
\end{equation*}
so that the function is positive.

\begin{figure}
\begin{subfigure}{.475\linewidth}
\includegraphics[width=\linewidth]{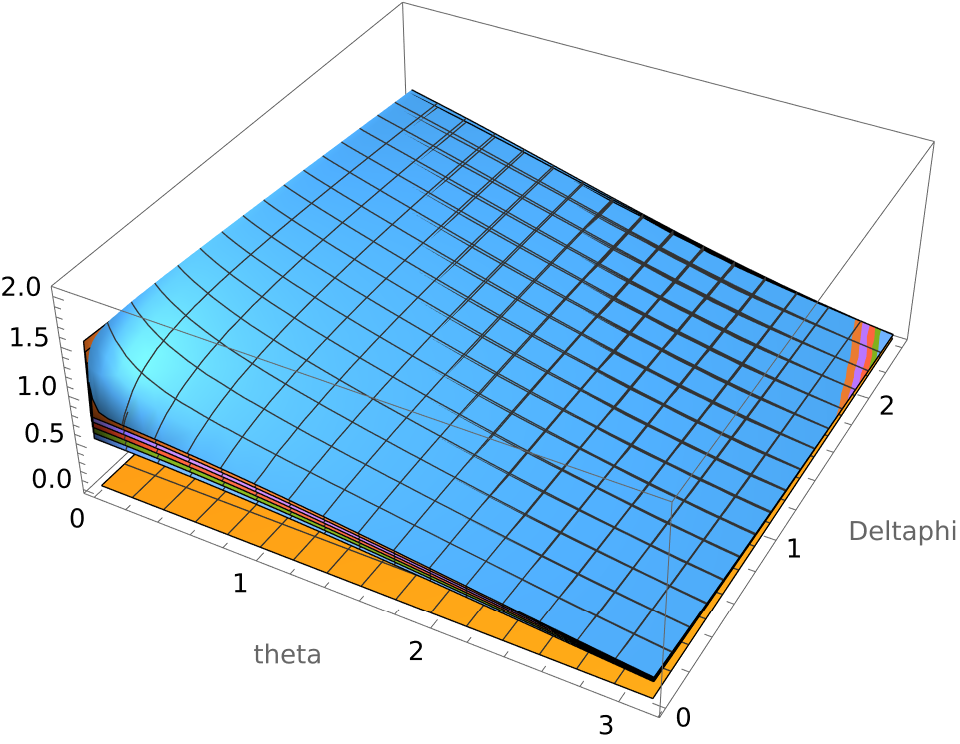}
\caption{Plot of the function $\frac{\delta \rho(2 + s \tau_1,\theta,\Delta \phi)}{\theta (\theta + \Delta \phi)}$ for $s$ assuming the values $\{0,.2,.4,.6,.8,1\}$ (monotonically bottom to top graph).}
\label{Fig:rho_third_1_arc}
\end{subfigure}\hfill
\begin{subfigure}{.475\linewidth}
\includegraphics[width=\linewidth]{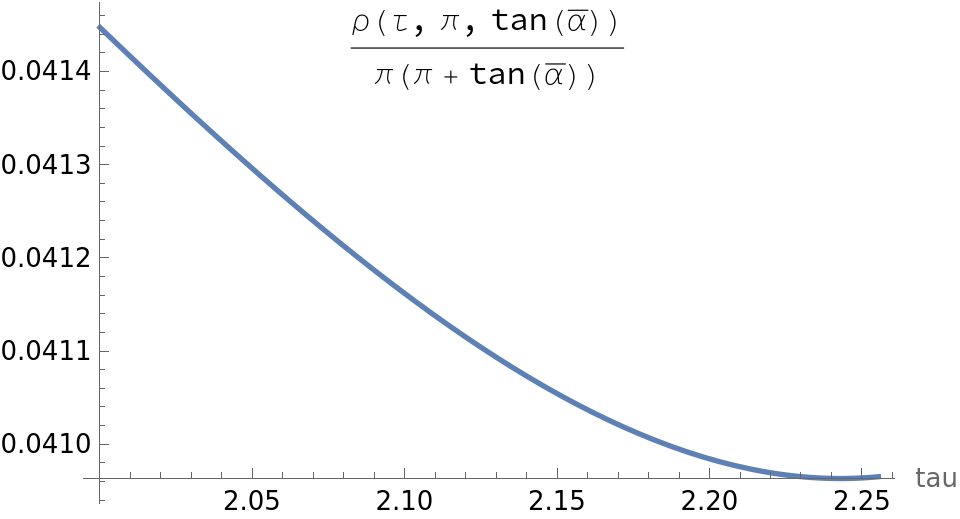}
\caption{Plot of the worst case $\frac{\delta \rho(\tau,\pi,0)}{\pi(\pi + \tan(\bar \alpha))}$, $\tau \in [2,2+\tau_1)$.}
\label{Fig:rho_third_1_arc_worst}
\end{subfigure}
\caption{Study of $\delta \rho$ for $\tau \in [2,2+\tau_1)$.}
\label{Fig:third_round_arc_1}
\end{figure}

\medskip

{\it Case 8: $\tau \in [2+\tau_1,2+\tau_2]$.}

\noindent The explicit form of the solution is
\begin{align*}
\delta \rho(\tau) &= S_0 \bigg( e^{\tau} - e^{\tau - 1} (\tau - 1) + e^{\tau-2}\frac{(\tau - 2)^2}{2} \bigg) + S_1 \bigg( e^{\tau - \tau_1} - e^{\tau - \tau_1 - 1} (\tau - \tau_1 -1) + e^{\tau - \tau_1 - 2} \frac{(\tau - 2 - \tau_1)^2}{2} \bigg) \\
& \quad + S_2 \big( e^{\tau - \tau_2} - e^{\tau - \tau_2 - 1} (\tau - \tau_2 - 1) \big) \\
& \quad + S_3 \bigg[ \frac{e^{- \bar c \tau}}{\cot(\bar \alpha)} \big( e^{\cot(\bar \alpha)(2\pi + \bar \alpha) (\tau - \tau_2)} - e^{\cot(\bar \alpha) (2\pi + \bar \alpha) (\tau-1)} \big) \\
& \quad \qquad \qquad + \frac{e^{\bar c(2\pi + \bar \alpha)(\tau-1)}}{\cot(\bar \alpha)^2 (2\pi + \bar \alpha)} \bigg( - e^{-\cot(\bar \alpha) (\tau - \tau_2 - 1)} \cot(\bar \alpha) (2\pi + \bar \alpha) (\tau - \tau_2 - 1) \\
& \quad \qquad \qquad \qquad \qquad \qquad + e^{-\cot(\bar \alpha) (\tau - 2)} \cot(\bar \alpha) (2\pi + \bar \alpha) (\tau - 2) + e^{\cot(\bar \alpha) (2\pi + \bar \alpha) (\tau - \tau_2 - 1} - e^{\cot(\bar \alpha) (2\pi + \bar \alpha) (\tau - 2)} \bigg) \bigg] \\
& \quad + S_4 \big( e^{\tau - 1} - e^{\tau-2} (\tau -2) \big) + S_5 e^{\tau-2}.
\end{align*}
Numerical evaluations show that (see Fig. \ref{Fig:third_round_arc_2})
\begin{equation*}
\frac{\delta \rho([2+\tau_1,2+\tau_2),\theta,\Delta \phi)}{\theta(\theta+\Delta \phi)} \geq \frac{\delta \rho(2+\tau_1,\pi,\tan(\bar \alpha))}{\pi(\pi + \tan(\bar \alpha))} > 0.04,
\end{equation*}
so that the function is positive.

\begin{figure}
\begin{subfigure}{.475\linewidth}
\includegraphics[width=\linewidth]{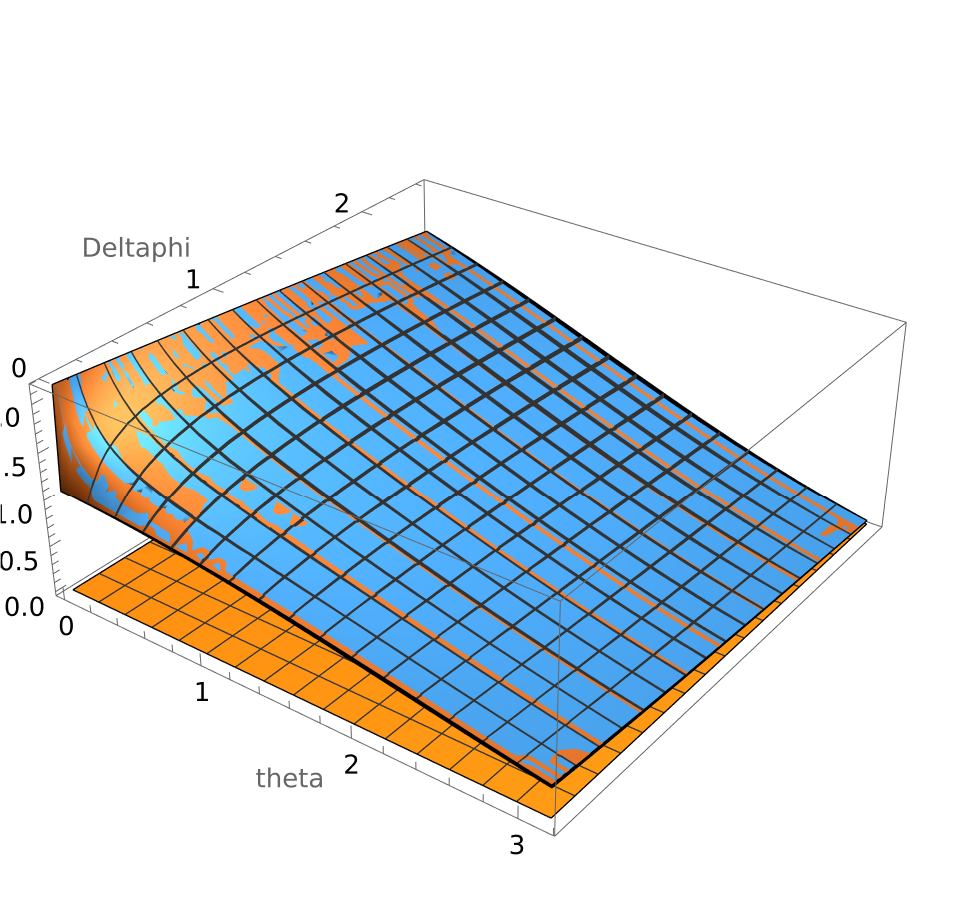}
\caption{Plot of the function $\frac{\delta \rho(2 + (1-s) \tau_1 + s \tau_2,\theta,\Delta \phi)}{\theta (\theta + \Delta \phi)}$ for $s$ assuming the values $\{0,.2,.4,.6,.8,1\}$ (monotonically bottom to top graph).}
\label{Fig:rho_third_2_arc}
\end{subfigure}\hfill
\begin{subfigure}{.475\linewidth}
\includegraphics[width=\linewidth]{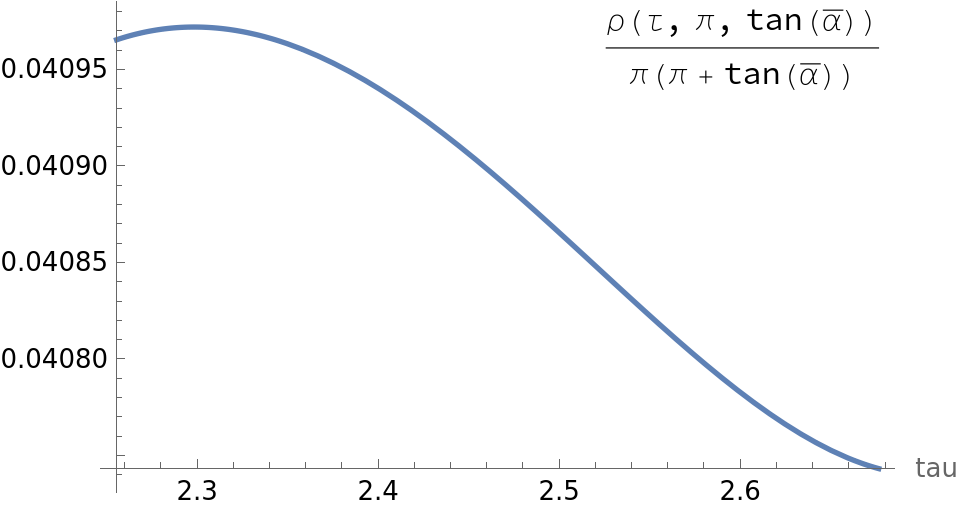}
\caption{Plot of the worst case $\frac{\delta \rho(\tau,\frac{\pi}{2},\tan(\bar \alpha))}{\frac{\pi}{2}(\frac{\pi}{2} + \tan(\bar \alpha))}$, $\tau \in [2+\tau_1,2+\tau_2)$.}
\label{Fig:rho_third_2_arc_worst}
\end{subfigure}
\caption{Study of $\delta \rho$ for $\tau \in [2+\tau_1,2+\tau_2)$.}
\label{Fig:third_round_arc_2}
\end{figure}

\medskip

{\it Case 9: $\tau \in [2+\tau_2,3)$.}

\noindent The explicit form of the solution is
\begin{align*}
\delta \rho(\tau) &= S_0 \bigg( e^{\tau} - e^{\tau - 1} (\tau - 1) + e^{\tau-2}\frac{(\tau - 2)^2}{2} \bigg) + S_1 \bigg( e^{\tau - \tau_1} - e^{\tau - \tau_1 - 1} (\tau - \tau_1 -1) + e^{\tau - \tau_1 - 2} \frac{(\tau - 2 - \tau_1)^2}{2} \bigg) \\
& \quad + S_2 \bigg( e^{\tau - \tau_2} - e^{\tau - \tau_2 - 1} (\tau - \tau_2 - 1) + e^{\tau - \tau_2 -2} \frac{(\tau - \tau_2 -2)^2}{2} \bigg) \\
& \quad + S_3 \bigg[ \frac{e^{- \bar c \tau}}{\cot(\bar \alpha)} \big( e^{\cot(\bar \alpha)(2\pi + \bar \alpha) (\tau - \tau_2)} - e^{\cot(\bar \alpha) (2\pi + \bar \alpha) (\tau-1)} \big) \\
& \quad \qquad + \frac{e^{\bar c(2\pi + \bar \alpha)(\tau-1)}}{\cot(\bar \alpha)^2 (2\pi + \bar \alpha)} \bigg( - e^{-\cot(\bar \alpha) (\tau - \tau_2 - 1)} \cot(\bar \alpha) (2\pi + \bar \alpha) (\tau - \tau_2 - 1) \\
& \quad \qquad \qquad \qquad \qquad \qquad \qquad + e^{-\cot(\bar \alpha) (\tau - 2)} \cot(\bar \alpha) (2\pi + \bar \alpha) (\tau - 2) + e^{\cot(\bar \alpha) (2\pi + \bar \alpha) (\tau - \tau_2 - 1} - e^{\cot(\bar \alpha) (2\pi + \bar \alpha) (\tau - 2)} \bigg) \\
& \quad \qquad + \frac{e^{\bar c(2\pi + \bar \alpha)(\tau-2)}}{\cot(\bar \alpha)^3 (2\pi + \bar \alpha)^2} \bigg( e^{-\cot(\bar \alpha) (\tau - \tau_2 - 2)} \frac{(\cot(\bar \alpha) (2\pi + \bar \alpha) (\tau - \tau_2 - 2))^2}{2} \\
& \quad \qquad \qquad \qquad \qquad \qquad \qquad - e^{-\cot(\bar \alpha) (\tau - \tau_2 - 2)} \cot(\bar \alpha) (2\pi + \bar \alpha) (\tau - \tau_2 - 2) + e^{\cot(\bar \alpha) (2\pi + \bar \alpha) (\tau - \tau_2 - 2)} - 1 \bigg) \bigg] \\
& \quad + S_4 \big( e^{\tau - 1} - e^{\tau-2} (\tau -2) \big) + S_5 e^{\tau-2}.
\end{align*}
Numerical evaluations show that (see Fig. \ref{Fig:third_round_arc_3})
\begin{equation*}
\frac{\delta \rho([2+\tau_2,3),\theta,\Delta \phi)}{\theta(\theta+\Delta \phi)} \geq \frac{\delta \rho(3,\pi,\tan(\bar \alpha))}{\pi(\pi+ \tan(\bar \alpha))} > 0.04,
\end{equation*}
so that the function is positive.

\begin{figure}
\begin{subfigure}{.475\linewidth}
\includegraphics[width=\linewidth]{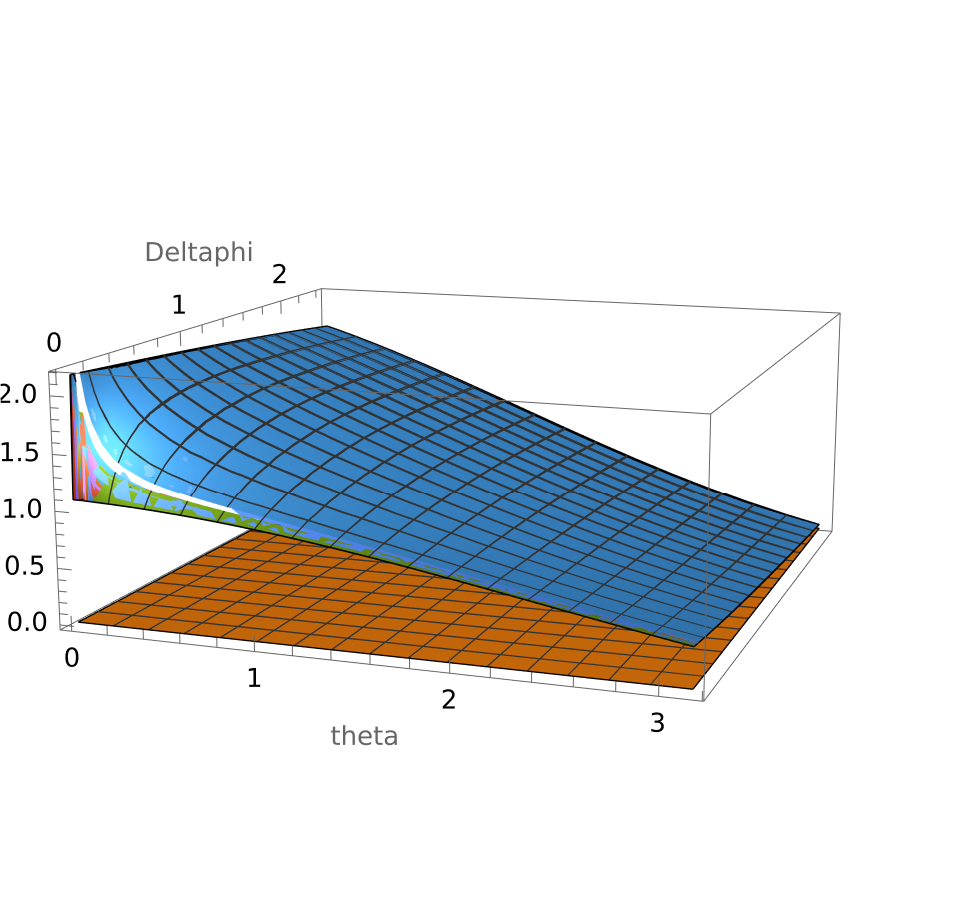}
\caption{Plot of the function $\frac{\delta \rho(2 + (1-s) \tau_2 + s,\theta,\Delta \phi)}{\theta (\theta + \Delta \phi)}$ for $s$ assuming the values $\{0,.2,.4,.6,.8,1\}$ (monotonically bottom to top graph).}
\label{Fig:rho_third_3_arc}
\end{subfigure}\hfill
\begin{subfigure}{.475\linewidth}
\includegraphics[width=\linewidth]{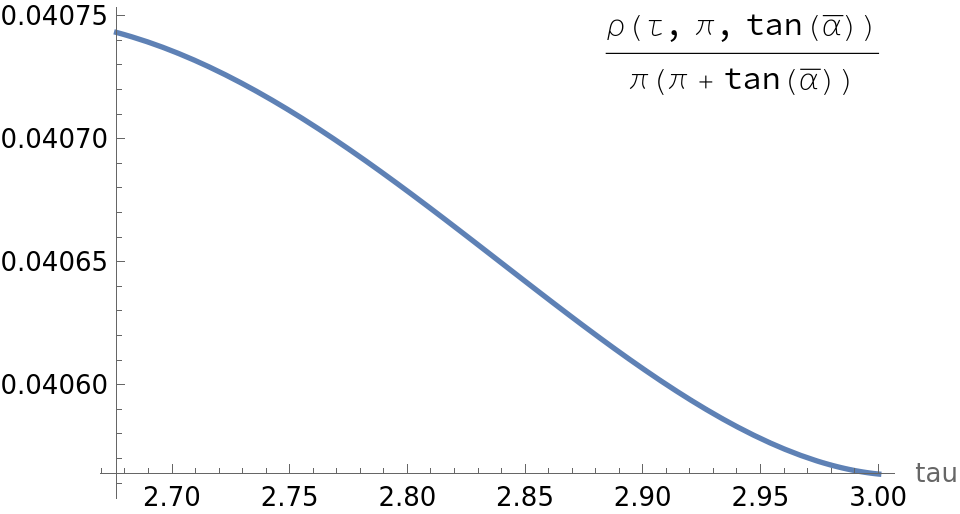}
\caption{Plot of the worst case $\frac{\delta \rho(\tau,\frac{\pi}{2},\tan(\bar \alpha))}{\frac{\pi}{2}(\frac{\pi}{2} + \tan(\bar \alpha))}$, $\tau \in [2+\tau_2,3)$.}
\label{Fig:rho_third_3_arc_worst}
\end{subfigure}
\caption{Study of $\delta \rho$ for $\tau \in [2+\tau_2,3)$.}
\label{Fig:third_round_arc_3}
\end{figure}

\medskip

{\it Case 10: $\tau \geq 3$.}

It is enough to find the region where that
\begin{equation*}
\dot \delta \rho(\tau) = \delta \rho(\tau) - \delta \rho(\tau-1) \geq 0
\end{equation*}
and apply Lemma \ref{lem:key}. Numerically it can be shown that (see Fig. \ref{Fig:fourth_round_arc})
\begin{equation*}
\frac{\delta \rho(\tau,\theta,\Delta \phi) - \delta \rho(\tau,\theta,\Delta \phi)}{\theta(\theta + \Delta \phi)} \geq 10^{-4}
\end{equation*}
for
\begin{equation*}
(\theta,\Delta \phi) \in (0,\pi] \times [0,\tan(\bar \alpha)] \setminus (3.12,\pi) \times (2.37,\tan(\bar \alpha)).
\end{equation*}

\begin{figure}
\begin{subfigure}{.32\linewidth}
\includegraphics[width=\linewidth]{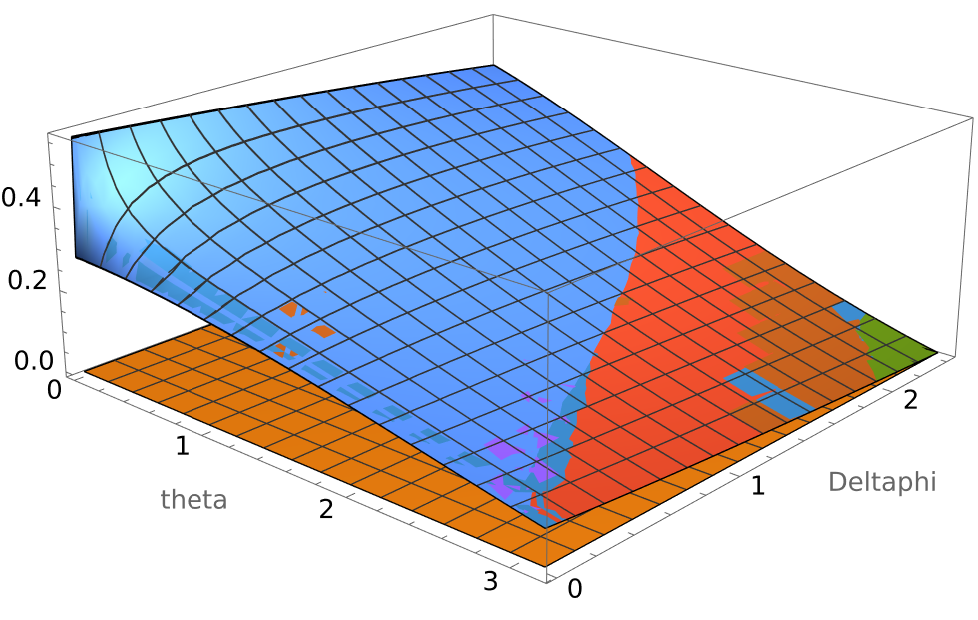}
\caption{Numerical plot of the function $\frac{\delta \rho(\tau,\theta,\Delta \phi)-\delta \rho(\tau-1,\theta,\Delta \phi)}{\theta (\theta + \Delta \phi)}$, $\tau$ assuming the values $\{3,3.2,3.4,3.6,3.8,4\}$.}
\label{Fig:rho_fourth_1_arc}
\end{subfigure}\hfill
\begin{subfigure}{.32\linewidth}
\includegraphics[width=\linewidth]{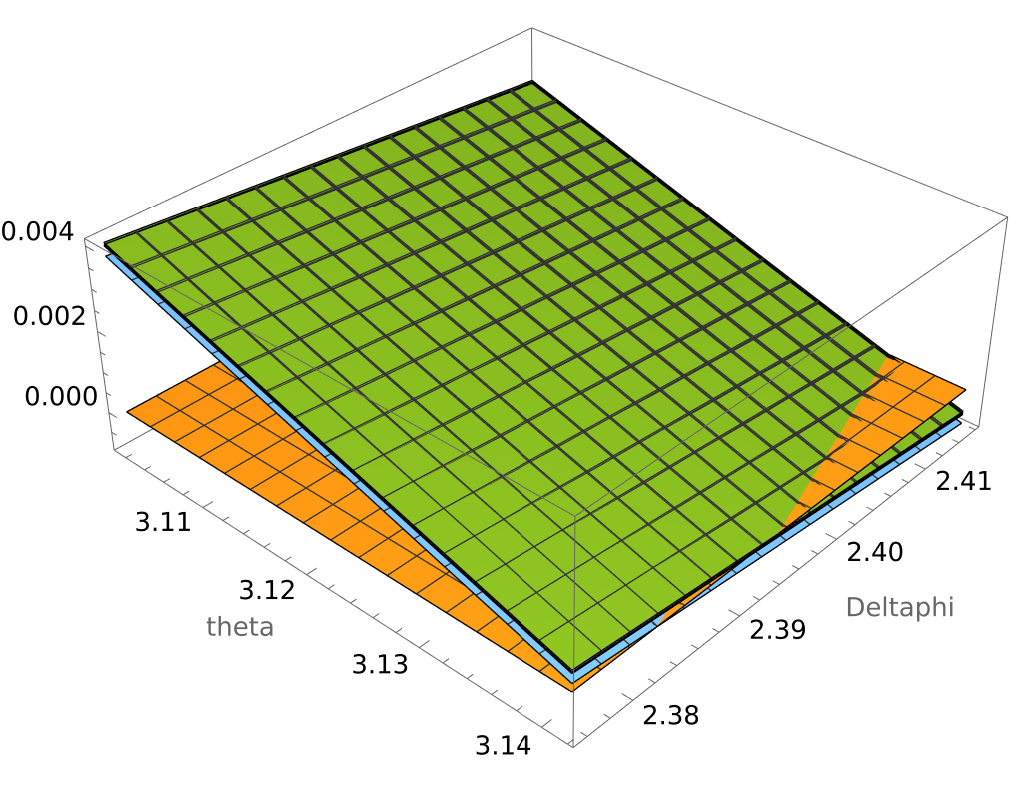}
\caption{Close-in plot of the previous graph near the negativity region $(\theta,\Delta \phi) \in (\pi,\tan(\bar \alpha)) + (-.4,0) \times (-4,0)$.}
\label{Fig:rho_fourth_2_arc}
\end{subfigure} \hfill
\begin{subfigure}{.32\linewidth}
\includegraphics[width=\linewidth]{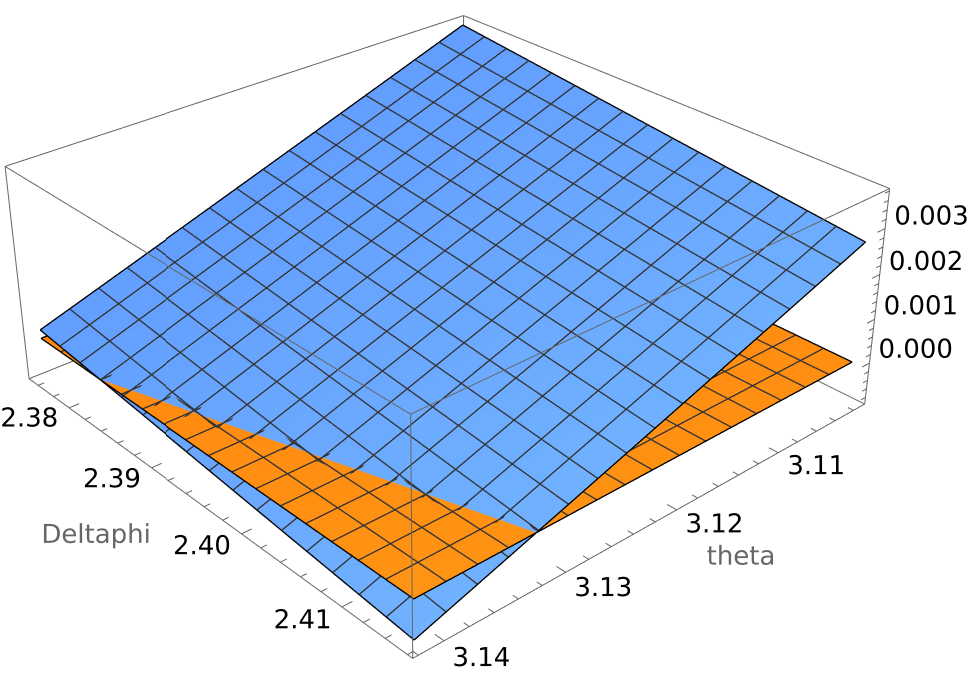}
\caption{Numerical plot of the worst case $\frac{\delta \rho(3,\theta,\Delta \phi) - \delta \rho(2,\theta,\Delta \phi)}{\theta(\theta + \Delta \phi)}$ near the negativity region.}
\label{Fig:rho_fourth_3_arc}
\end{subfigure}
\caption{Study of $\delta \rho$ for $\tau \in [3,4)$.}
\label{Fig:fourth_round_arc}
\end{figure}
The application of Lemma \ref{lem:key} concludes the proof.

\medskip

{\it Asymptotic behavior.}

\begin{figure}
\begin{subfigure}{.475\linewidth}
\includegraphics[width=\linewidth]{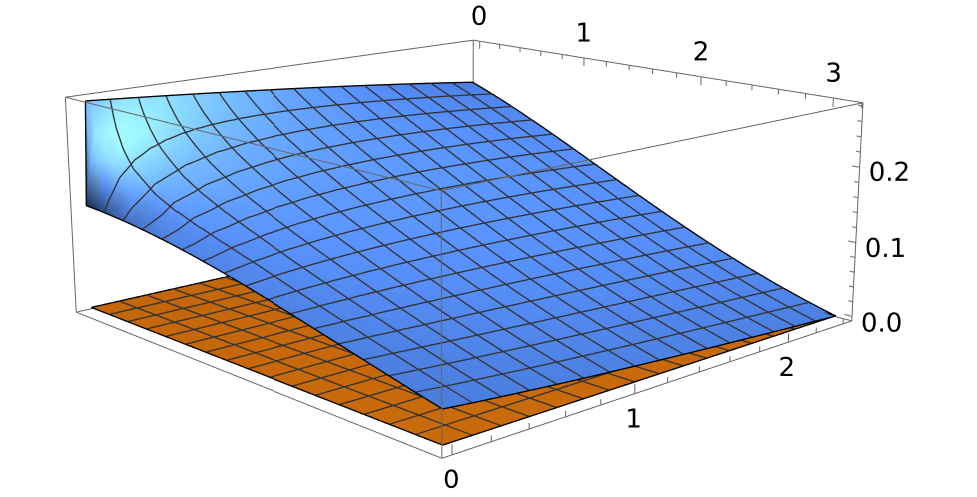}
\end{subfigure}\hfill
\begin{subfigure}{.475\linewidth}
\includegraphics[width=\linewidth]{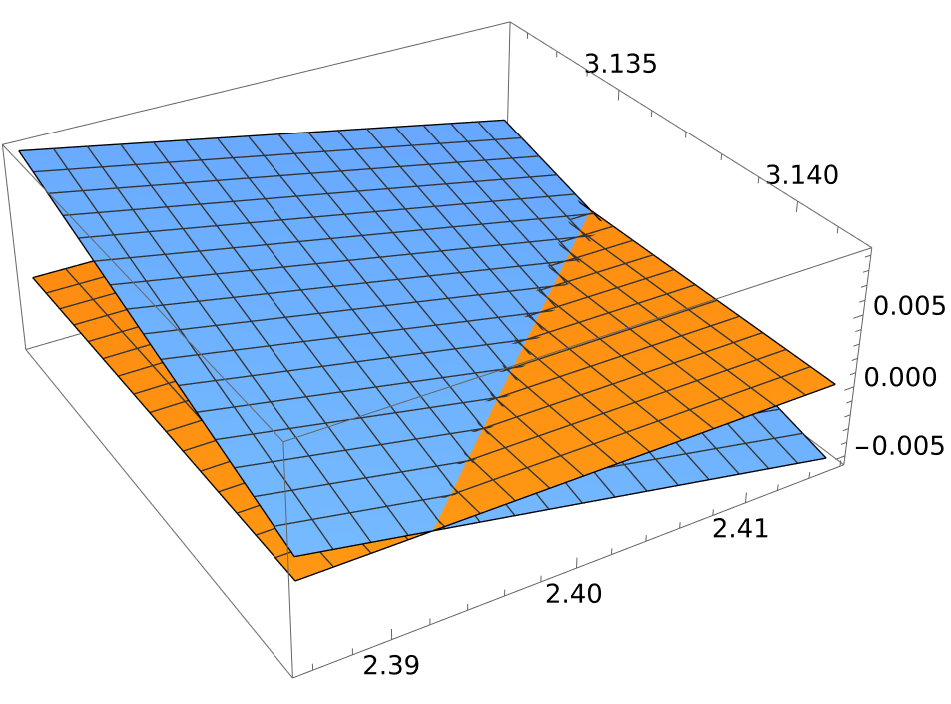}
\end{subfigure}
\caption{Plot of the function \ref{Equa:asympt_arc}, and a blowup about the negativity region.}
\label{Fig:asympt_arc}
\end{figure}

Using Lemmata \ref{Lem:aympt_g} and \ref{Lem:m_aymptotic} we obtain
\begin{equation}
\label{Equa:asympt_arc}
\lim_{\tau \to \infty} \frac{\delta \rho(\tau) e^{-\bar c \tau}}{\tau \theta(\theta + \Delta \phi)} = \frac{2}{\theta ( \theta+ \Delta \phi)} \bigg( S_0 + S_1 + S_2 + S_3 \frac{e^{-\bar c (2\pi + \bar \alpha) \tau_2} - e^{-\bar c(2\pi + \bar \alpha)}}{\bar c} + S_4 + S_5 \bigg).
\end{equation}
A plot fo this function is in Fig. \ref{Fig:asympt_arc}.
\end{proof}

%

\subsubsection{Estimate on the length for the fastest saturated spiral $\delta r(\phi;s_0)$, arc case}
\label{Sss:additional_estim_opt_arc}

We will repeat some of the estimates for the functional
\begin{equation*}
\delta \tilde r(\phi;s_0) - 2.08 \delta \tilde L(\phi;s_0), \quad \phi \geq \phi_0 + \Delta \phi + \frac{\pi}{2} + 2\pi.
\end{equation*}
Using computations similar to Section \ref{Sss:additional_estim_opt_segm} we obtain
\begin{equation}
\label{Equa:lenght_arc}
\begin{split}
\delta \tilde L(\phi;s_0) &= \frac{1 - \cos(\theta) + \sin(\theta) \Delta \phi}{\sin(\bar \alpha)^2} \\
& \quad + \frac{\cot(\bar \alpha)}{\sin(\bar \alpha)} \left( 1 - \cos(\theta) + \Delta \phi \sin(\theta) \right) \mathcal G \bigg( \phi - \phi_0 - \Delta \phi - \frac{\pi}{2} + \bar \alpha \bigg) - \mathcal G(\phi - \phi_0 - 2\pi - \theta_0) \\
& \quad + \cos(\theta) \mathcal G \bigg( \phi - \phi_0 - 2\pi - \frac{\pi}{2} \bigg) - \sin(\theta) \bigg( \mathfrak G \bigg( \phi - \phi_0 - 2\pi - \frac{\pi}{2} \bigg) - \mathfrak G \bigg( \phi - \phi_0 - 2\pi - \frac{\pi}{2} - \Delta \phi \bigg) \bigg)  \\
& \quad + \bigg( - \cot(\bar \alpha)^2 \big( 1 - \cos(\theta) + \Delta \phi \sin (\theta) \big) \\
& \quad \qquad \qquad \qquad + \cot(\bar \alpha) \Big( \sin(\theta) - \cot(\bar \alpha) \big( 1 - \cos(\theta) + \Delta \phi \sin(\theta) \big) \Big) \bigg) \mathcal G \bigg( \phi -\phi_0 - 2\pi - \Delta \phi - \frac{\pi}{2} \bigg) \\
& \quad + \Big( \sin(\theta) - \cot(\bar \alpha) \big( 1 - \cos(\theta) + \Delta \phi \sin(\theta) \big) \Big) H \bigg( \phi - \phi_0 + 2\pi + \frac{\pi}{2} + \Delta \phi \bigg) \\
& \quad - \frac{1}{\sin(\bar \alpha)} \Big( \sin(\theta) - \cot(\bar \alpha) \big( 1 - \cos(\theta) + \Delta \phi \sin(\theta) \big) \Big) \mathcal G \bigg( \phi - \phi_0 - 2\pi - \Delta \phi - \frac{\pi}{2} - 2\pi  - \bar \alpha \bigg).
\end{split}
\end{equation}
where we have used the definition of \gls{Gfrak} and for the first term we compute by \eqref{Equa:Deltar_arc}
\begin{equation*}
1 - \cos(\theta) + \sin(\theta) \Delta \phi + r(\phi_1) \delta \phi_1 + \delta |P_2 - P_1| - \frac{r(\phi_2) \delta \phi_2}{\sin(\bar \alpha)} = \frac{1 - \cos(\theta) + \sin(\theta) \Delta \phi}{\sin(\bar \alpha)^2}.
\end{equation*}

For $\phi \geq \phi_0 + 2\pi + \Delta \phi + \frac{\pi}{2} + 2\pi + \bar \alpha$, the RDE satisfied by $\delta \tilde L$ is
\begin{align*}
\frac{d}{d\phi} \delta \tilde L(\phi) &= \cot(\bar \alpha) \delta \tilde L(\phi) - \frac{\delta \tilde L(\phi - 2\pi - \bar \alpha)}{\sin(\bar \alpha)} + S_L,
\end{align*}
with
\begin{align*}
S_L &= \frac{\cot(\bar \alpha)}{\sin(\bar \alpha)^2} \left( 1 - \cos(\theta) + \Delta \phi \sin(\theta) \right) - \frac{1}{\sin(\bar \alpha)} + \frac{\cos(\theta)}{\sin(\bar \alpha)} - \frac{\sin(\theta) \Delta \phi}{\sin(\bar \alpha)} \\
& \quad + \frac{1}{\sin(\bar \alpha)} \bigg( - \cot(\bar \alpha)^2 \big( 1 - \cos(\theta) + \Delta \phi \sin (\theta) \big) + \cot(\bar \alpha) \Big( \sin(\theta) - \cot(\bar \alpha) \big( 1 - \cos(\theta) + \Delta \phi \sin(\theta) \big) \Big) \bigg) \\
& \quad - \frac{1}{\sin(\bar \alpha)^2} \Big( \sin(\theta) - \cot(\bar \alpha) \big( 1 - \cos(\theta) + \Delta \phi \sin(\theta) \big) \Big) \\
&= - \big( 1 - \cos(\theta) + \sin(\theta) \Delta \phi \big) \frac{(1 - \cos(\bar \alpha))^2}{\sin(\bar \alpha)^3} + \sin(\theta) \frac{\cos(\bar \alpha) - 1}{\sin(\bar \alpha)^2} < 0.
\end{align*}
Hence the RDE for $\delta \tilde r - 2.08 \delta \tilde L$ is
\begin{equation*}
\frac{d}{d\phi} (\delta \tilde r(\phi) - 2.08 \delta \tilde L(\phi)) > \cot(\bar \alpha) (\delta \tilde r(\phi) - 2.08 \delta \tilde L(\phi)) - \frac{\delta \tilde r(\phi - 2\pi - \bar \alpha) - 2.08 \delta \tilde L(\phi - 2\pi - \bar \alpha)}{\sin(\bar \alpha)}
\end{equation*}
for $\phi \geq \phi_0 + 2\pi + \Delta \phi + \frac{\pi}{2} + 2\pi + \bar \alpha$. We can thus apply Lemma \ref{lem:key} and Remark \ref{rmk:exponentially:explod} if we can show numerically that the function is positive and strictly increasing for $\tau \in [4,5]$.

\begin{figure}
\begin{subfigure}{.33\textwidth}
\resizebox{\textwidth}{!}{\includegraphics{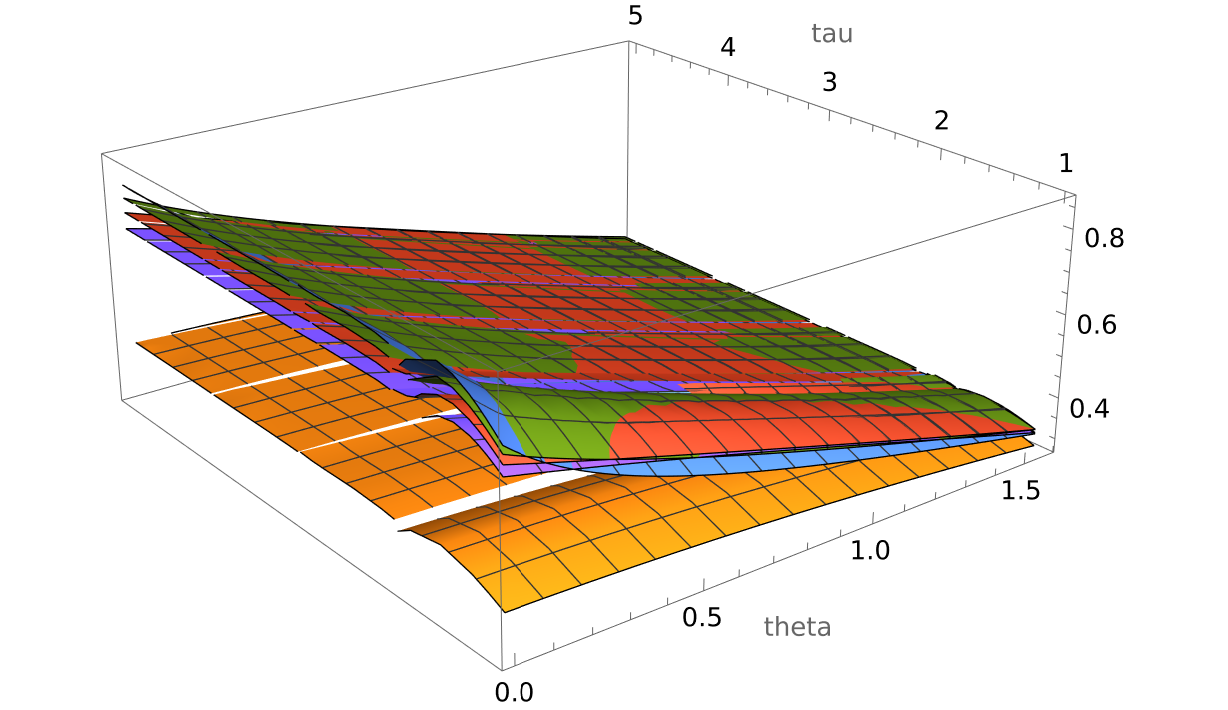}}
\end{subfigure} \hfill
\begin{subfigure}{.33\textwidth}
\resizebox{\textwidth}{!}{\includegraphics{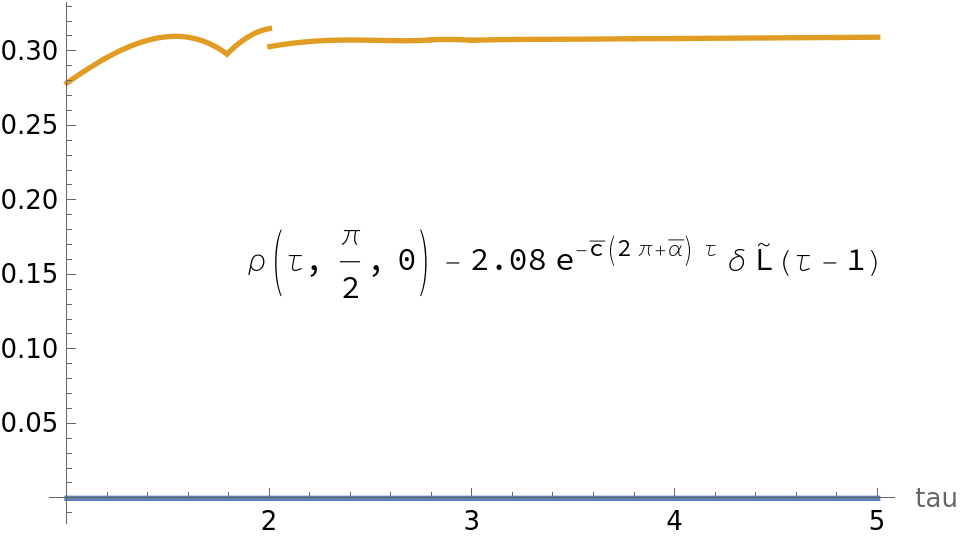}}
\end{subfigure} \hfill
\begin{subfigure}{.33\textwidth}
\resizebox{\textwidth}{!}{\includegraphics{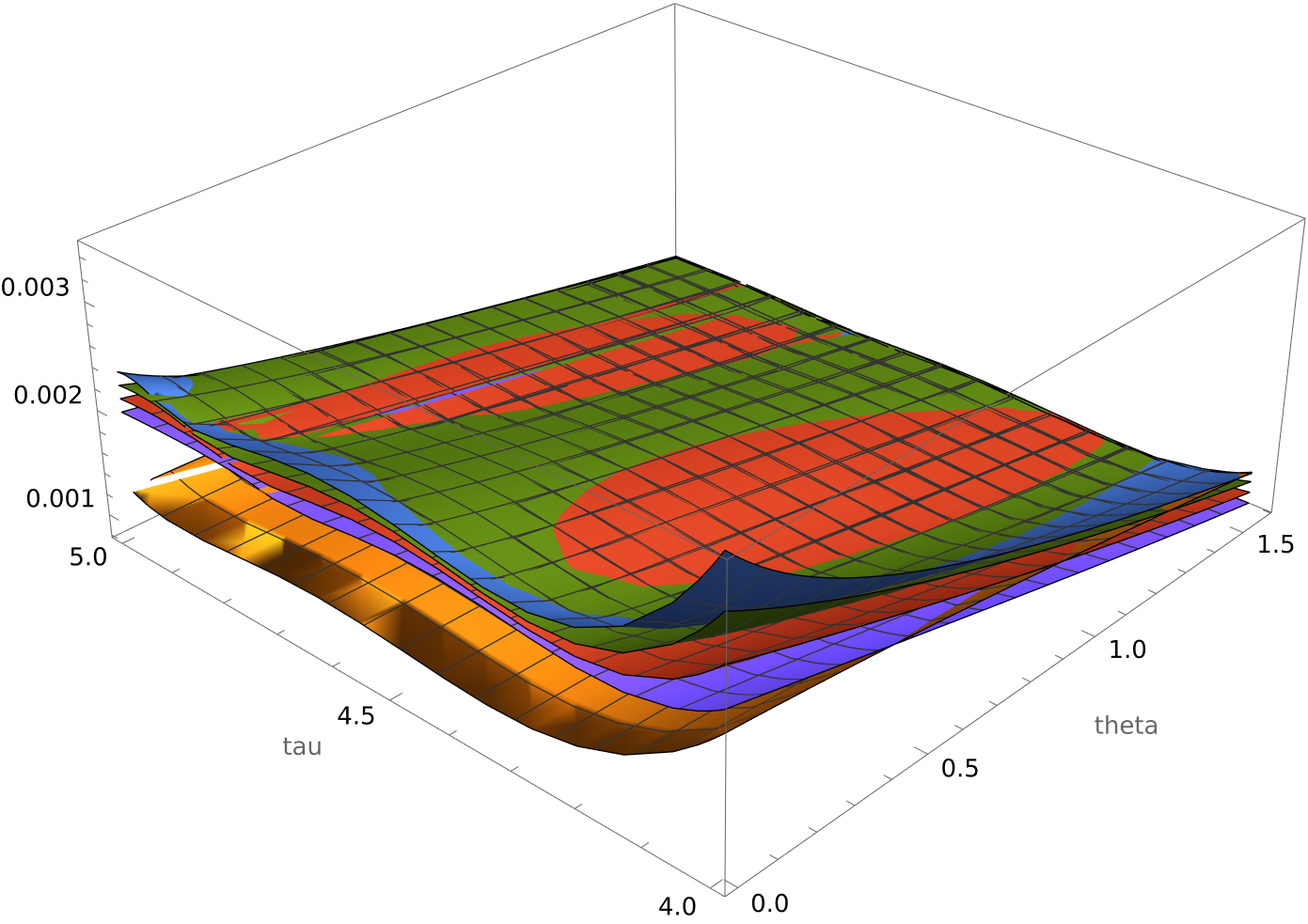}}
\end{subfigure}
\caption{Numerical plots of the function $\frac{\delta \rho(\tau) - e^{-\bar c (2\pi + \bar \alpha) \tau} \delta \tilde L(\tau)}{\theta(\theta + \Delta \phi)}$: the fist 5 rounds, the minimal values for $\theta = \frac{\pi}{2}, \Delta \phi = 0$, and its derivative for $\tau \in [4,5]$.}
\label{Fig:Lpert_arc}
\end{figure}

A numerical plot of the function $\delta \rho(\tau) - 2.08 e^{-\bar c (2\pi + \bar \alpha) \tau} \delta L(\tau)$, $\phi = \phi_0 + \Delta \phi + \frac{\pi}{2} - \bar \alpha$ and $(\theta,\Delta \phi) \in [0,\frac{\pi}{2}] \times [0,\tan(\bar \alpha)]$, is in Fig. \ref{Fig:Lpert_arc}: it is shown that the function is positive, increasing for $\tau \in [4,5]$ and its minimal value: by Lemma \ref{lem:key} we conclude that

\begin{proposition}
\label{Prop:lemgth_arc_pert}
It holds for the perturbation in the arc case
\begin{equation*}
\delta \rho(\tau) - 2.08 e^{-\bar c (2\pi + \bar \alpha) \tau} \delta \tilde L(\tau) \geq 0.27 e^{\bar c (2\pi + \bar \alpha) \tau} \theta(\theta + \Delta \phi), \quad \tau \geq 1.
\end{equation*}
\end{proposition}

\subsection{An bound for the length of a spiral}
\label{Ss:bound_lenth}

We can put Propositions \ref{Prop:lemgth_segm_pert} and \ref{Prop:lemgth_arc_pert} together, obtaining the following

\begin{corollary}
\label{Cor:bound_length_gen}
Consider the spiral $\tilde r(\phi;\phi_0)$: if
$$
\phi \geq \phi_2 = \begin{cases}
\phi_0 + \bar \theta + 2\pi & \text{segment case}, \\
\phi_0 + \Delta \phi + \frac{\pi}{2} + 2\pi & \text{arc case},
\end{cases}
$$
then
\begin{equation*}
\tilde r(\phi) - 2.08 \tilde L(\phi - 2\pi - \bar \alpha) \geq \tilde r_{\mathtt a = 0}(\phi) - 2.08 \tilde L_{\mathtt a = 0}(\phi - 2\pi - \bar \alpha) > 0.67 e^{\bar c(\phi - \tan(\bar \alpha) - \frac{\pi}{2} + \bar \alpha)},
\end{equation*}
\end{corollary}

\begin{proof}
We can write
\begin{align*}
\tilde r(\phi) - 2.08 \tilde L(\phi - 2\pi - \bar \alpha) &= \tilde r_{\mathtt a = 0}(\phi) - 2.08 \tilde L_{\mathtt a = 0}(\phi - 2\pi - \bar \alpha) + \int_0^{\phi_0} \delta \tilde r(\phi;s_0) - 2.08 \delta \tilde L(\phi - 2\pi - \bar \alpha;s_0) ds_0 \\
&\geq \tilde r_{\mathtt a = 0}(\phi) - 2.08 \tilde L_{\mathtt a = 0}(\phi - 2\pi - \bar \alpha). \qedhere
\end{align*}
\end{proof}

\section{The optimal solution}
\label{S:optimal_solut}

In this section we construct the optimal candidate spiral for the minimization problem \eqref{Equa:minimum_rphi_intro}, \eqref{eq:min:probl_pert}. The previous section showed that the fastest saturated spiral $\tilde \zeta(\phi;s_0)$ is optimal outside a small interval of rotation angle in the first round after $\phi_0$ (for $\theta \in [0,\frac{\pi}{2}]$): we thus construct a new family of spirals $\check r$ whose structure takes into account that for some angles $\phi$ it is better not to follow the fastest saturated spiral $\tilde r$.

\subsection{Definition and construction of the optimal solution candidate}
\label{Ss:optimal_sol_def}

To describe the possible situations which can occur, depending on the relative position of $\phi_0, \bar \phi$ and the quantities \gls{thetatildeseg}, \gls{Deltaphi} (recall the definitions of the two quantities from Section \ref{S:family}), we define the quantity \newglossaryentry{omegaopt}{name=\ensuremath{\omega},description={rotation angle for the unsaturated part of the fastest saturated spiral $\tilde r$}}
\begin{equation}
\label{Eq:def:omega:bar:theta}
\gls{omegaopt} = \begin{cases}
\bar \theta & \text{segment case}, \\
\frac{\pi}{2} + \Delta \phi & \text{arc case},
\end{cases} \qquad \text{so that} \ \bar \theta = \min \bigg\{ \omega, \frac{\pi}{2} \bigg\},
\end{equation}
and consider the following cases (see Figure \ref{fig:tent:solution} as a reference):
\begin{description}
\item[Point $(1)$ $\bar \phi -\phi_0 \in [0,2\pi +\beta^-(\phi_0))$] in this case we define
$$
\hat r(\phi;s_0) = \tilde r(\phi;s_0), \quad \phi \in [\phi_0,\bar \phi],
$$
i.e. we follow the fastest saturated spiral, which we already know it is the optimal solution for the first round by Theorem \ref{Cor:curve_cal_R_sat_spiral}.
\item[Point $(2)$ $\bar \phi - \phi_0 - 2\pi \in [\beta^-(\phi_0),\bar \theta)$] the solution $\hat r(\phi;s_0)$ will be a segment $[\check P_0,\check P^-]$ in the direction $\bar \phi - 2\pi$. We stop at a length $\ell_0 = |\check P^- - \check P_0|$ such that the fastest saturated spiral starting at the end point $\check P^-$ forms an initial angle whose value is on the boundary of the negativity region (it may happen that $\ell_0 = 0$). From this point onward, the solution is the fastest saturated spiral, which can start either with a segment or with a level set arc.


\item[Point $(2')$  $\bar \phi - \phi_0 - 2\pi \in [\frac{\pi}{2},\omega)$] here we are in the arc case, and we follow the fastest saturated solution $\tilde r(\phi;s_0)$ in the interval
\begin{equation*}
\phi \in \bigg[ \phi_0,\bar \phi - 2\pi - \frac{\pi}{2} \bigg),
\end{equation*}
then we repeat the construction of the previous point: follow the direction $\bar \phi - 2\pi$ until we remain in the negativity region. Also in this case, after reaching the boundary of the negativity region, we close the last round starting either with a segment or with a level set arc.

\item[Point $(3)$    $\bar \phi - \phi_0 - 2\pi - \omega \geq 0$] we follow the fastest saturated solution $\tilde \zeta$ up to the angle $\bar \phi - 2\pi - \bar \alpha$, then we repeat Point (2) and Point $(2')$ above starting from a saturated point $\bar \phi - 2\pi - \bar \alpha$. Note that when $\bar \phi = \phi_0 + \omega + 2\pi$, then
\begin{equation*}
P_2 = \tilde \zeta \bigg( \phi_0 + \Delta \phi + \frac{\pi}{2} - \bar \alpha \bigg), \quad P_1 = \tilde \zeta(\phi_0 + \Delta \phi),
\end{equation*}
and then $\bar \phi - 2\pi \in \partial^- \tilde \zeta(\phi_1) \cap \partial^- \tilde \zeta(\phi_2)$, so that the two constructions above are the same. Here, we have indicated by $\phi_1=\phi_0+\Delta\phi$ and by $\phi_2=\phi_0+\Delta\phi+\frac{\pi}{2}-\bar\alpha$ as in Section \ref{S:arc}.
\end{description}
\begin{figure}
	\centering
	\includegraphics[scale=0.6]{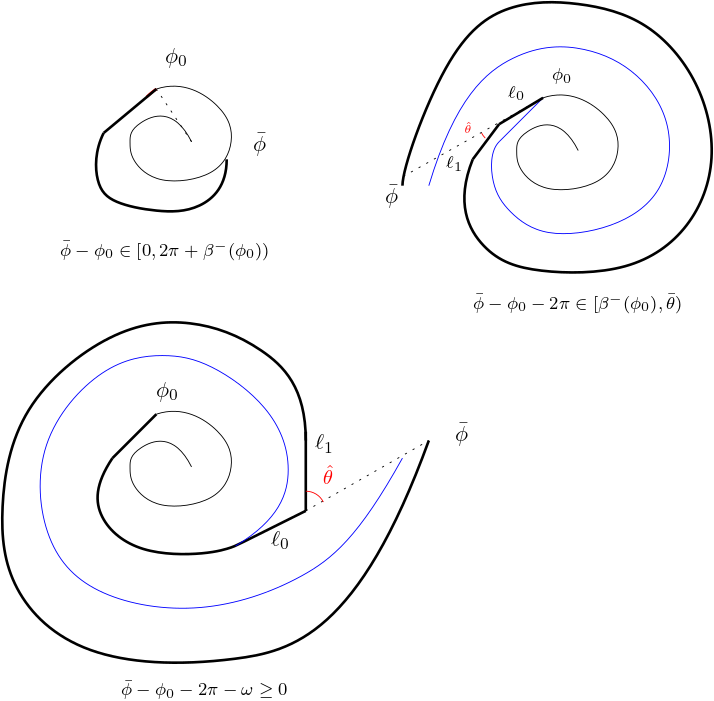}
	\caption{The three possible situations occurring in the construction of an element of the family, in the segment case ($\omega<\frac{\pi}{2}$). The blue curve represents the strategy $\tilde r(\phi;s_0)$ constructed in the previous section. In red the angle $\hat\theta$, representing the opening of the tent}
	\label{fig:tent:solution}
\end{figure}

We can summarize the construction as follows, Fig. \ref{Fig:checkr_notation}: we follow the fastest saturated solution $\tilde r$ up to the angle $\check \phi_0$ such that
\begin{equation}
\label{Equa:barphi_checkphi_rel}
\bar \phi \in \check \phi_0 + 2\pi + [\beta^-(\check \phi_0),\beta^+(\check \phi_0)],
\end{equation}
then a segment of length $\ell_0 = |\check P^- - \check P_0|$ determined by requiring that it is the last point such that the derivative of the fastest saturated spiral $\tilde r'$ starting from that point satisfies $\delta \tilde r' \leq 0$, and then the fastest saturated solution for the last round.

\begin{figure}
\resizebox{.75\textwidth}{!}{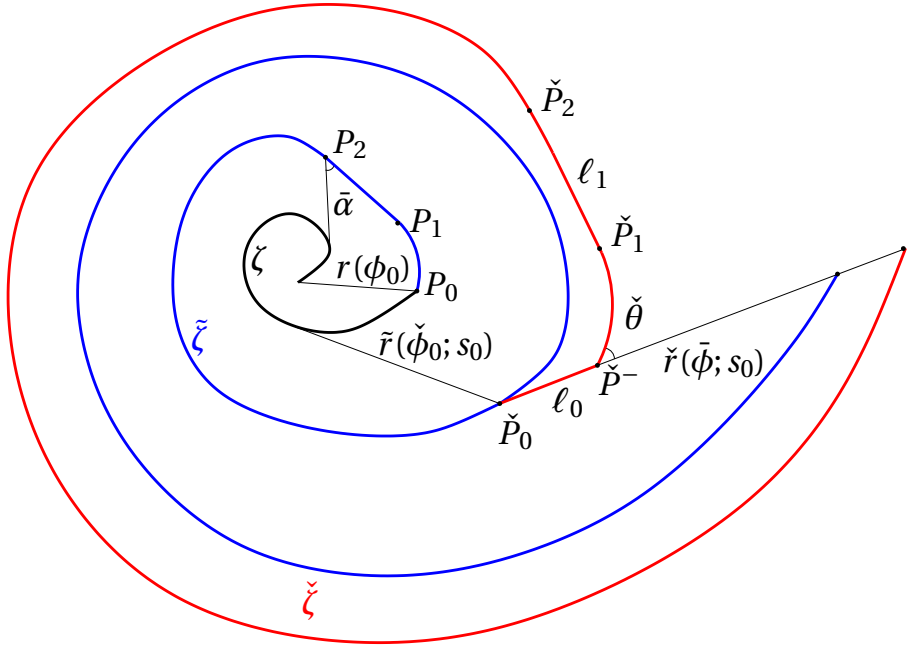}
\caption{The construction of the optimal closing solution: the starting spiral is $\zeta$ (black) up to the angle $\phi_0$. Then we follow the fastest saturated spiral $\tilde \zeta$ (blue) up to the angle $\check \phi_0$, determined by \eqref{Equa:barphi_checkphi_rel}. At this point we follow the segment with direction $\bar \phi$ until the angle between $\bar \phi$ and the new fastest saturated spiral $\check \zeta$ (red) is on the boundary of the negativity region. Note that while the blue spiral is inside the red one, the ray $\check r(\bar \phi;s_0)$ is smaller than $\tilde r(\bar \phi;s_0)$.}
\label{Fig:checkr_notation}
\end{figure}

Apriori, the perturbation of $\tilde r$ in the last case (i.e. on the saturated part of $\tilde r$) can be either a segment of an arc, i.e. we are exiting the negativity region in the arc or segment case. This may only happen in the first round, because afterwards the last round corresponds to the segment case (Proposition \ref{Prop:tent_admissible}): due to the particular form of the solution in this case (two consecutive segments followed by a saturated spiral), we will call this construction a \emph{tent}. In particular, Fig. \ref{Fig:checkr_notation} is correct when $\check P_1 = \check P^-$, but its purpose was just for notation.

We next fix the notation for the optimal candidate solution, see Fig. \ref{Fig:checkr_notation}.
\begin{itemize}
\item \gls{phibar} is the final angle, where we want to minimize $r(\bar \phi)$;
\item \gls{phi0} is the starting angle, with $\phi_0 \leq \bar\phi$, from which the arc/segment of the fastest saturated spiral \gls{rtildefam} starts;
\item the point \newglossaryentry{P0check}{name=\ensuremath{\check P_0},description={initial point of the segment of length $\ell_0$ in the direction $\bar \phi$}} \gls{P0check} is the point of the spiral $\tilde r(\phi;s_0)$ where a segment of length $\ell_0$ starts with direction $\bar \phi$; the corresponding angle is \newglossaryentry{phi0check}{name=\ensuremath{\check \phi_0},description={angle of the point $\check P_0$}} \gls{phi0check};
\item the segment of length \newglossaryentry{ell0}{name=\ensuremath{\ell_0},description={length of the segment $[\check P_0,P^-]$ of direction $\bar \phi$}} \gls{ell0} has endpoints \newglossaryentry{P-check}{name=\ensuremath{\check P^-},description={final point of the segment of length $\ell_0$ and direction $\bar \phi$}} $\check P_0,\gls{P-check}$, and its direction belongs to the subdifferential to the fastest saturated spiral $\tilde r(\phi;s_0),\phi)$ in the point $\check P_0$;
\item in the point $\check P^{-}$ we compute the fastest saturated spiral
: two cases again can occur (arc and segment, see however Proposition \ref{Prop:tent_admissible});
\item if the segment case occurs in the point $\check P^{-}$ we will denote it by \newglossaryentry{P1check}{name=\ensuremath{\check P_1},description={initial point of the last segment of the fastest saturated spiral}} $\gls{P1check} = \check P^{-}$;
\item if the arc case occurs in the point $\check P^{-}$ we will denote by \gls{P1check} the second endpoint of the arc of the level set of $u$, with opening \newglossaryentry{Deltaphicheck}{name=\ensuremath{\Delta \check \phi},description={angle of the level set for the last round of the optimal solution $\check r(\bar \phi;s_0)$}} \gls{Deltaphicheck};
\item from the point $\check P_1$ the spiral is a segment of length \newglossaryentry{ell1}{name=\ensuremath{\ell_1},description={length of the last segment $[\check P_1,\check P_2]$ of the fastest saturated spiral}} \gls{ell1} and endpoints \newglossaryentry{P2check}{name=\ensuremath{\check P_2},description={final point of the last segment $[\check P_1,\check P_2]$ in the fastest saturated spiral}} $\check P_1,\gls{P2check}$;
\item the point $\check P_2$ is saturated;
\item the exterior angle between the segments $[\check P_0,\check P_1]$ and $[\check P_1,\check P_2]$ is the critical angle \gls{thetahat} in the segment case, while in the arc case the angle between $[\check P_0,\check P^-]$ and the level set arc is $\gls{h1Deltaphi} \leq \hat \theta$.
\end{itemize}

The first result of this section is that the above construction defines a unique spiral \newglossaryentry{rcheckphis0}{name=\ensuremath{\check r(\phi;s_0)},description={optimal spiral at the angle $\bar \phi$}} \gls{rcheckphis0}.

\begin{proposition}
\label{Prop:unique_optimal_guess}
Given $s_0 = s(\phi_0) \geq 0$ and an angle $\bar \phi \geq \phi_0$, there exists a unique spiral $\check r(\phi;s_0)$ with the properties above.
\end{proposition}
%
%
%
%

\begin{proof}
The spiral $\check r(\phi;s_0)$ is defined once we know at what length $\ell_0$ of the segment starting in $\check P_0$ and pointing in the direction $\bar \phi$ has to stop, and the spiral becomes the fastest saturated afterward. This point is when the derivative $\delta \tilde r(\phi;\ell_0)$ of the fastest saturated spiral $\tilde r(\phi;\ell_0)$ starting from $\check P^-$, $|\check P^- - \check P_0| = \ell_0$, is $0$. Recall the definition of $\omega,\bar\theta$ in \eqref{Eq:def:omega:bar:theta}.

\begin{figure}
\resizebox{.75\textwidth}{!}{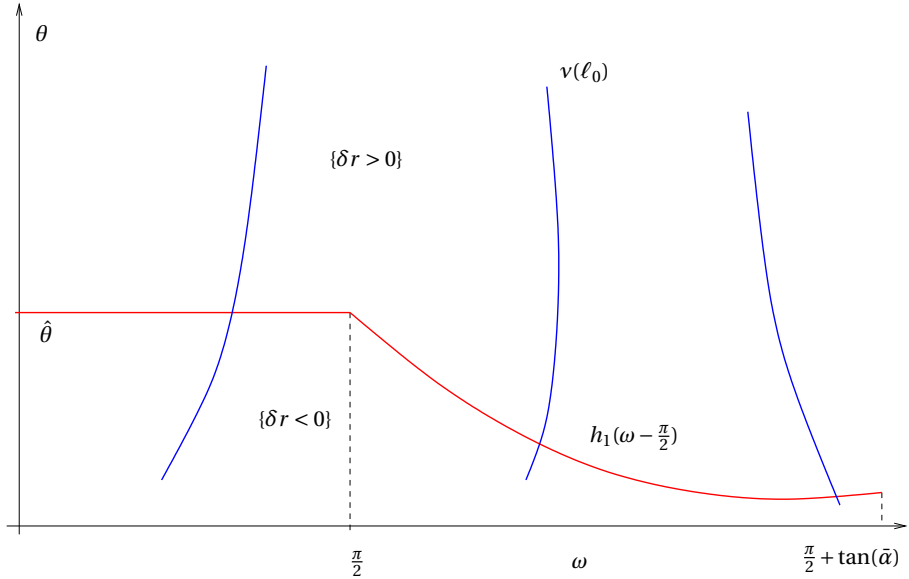}
\caption{The curves $\nu(\ell_0)$ of Proposition \ref{Prop:unique_optimal_guess} in blue, and the boundary of the negativity region in red.}
\label{Fig:nu_funct}
\end{figure}

If we define $\omega(\ell_0), \theta(\ell_0)$ as the values of the angles $\omega$, $\theta$ formed by the fastest saturated spiral along the segment $\ell_0 e^{\i \bar \phi}$ (recall that \gls{thetapertsegm} is the angle $\gls{thetatildeseg} -\theta_0$, where $\theta_0 = \beta^-(\check P_0)$ is the direction of the perturbation), we have defined a curve \newglossaryentry{nuell0}{name=\ensuremath{\nu(\ell_0)},description={curve representinvg the variation of the angles for the fastest saturated spiral w.r.t. the length $\ell_0$ of the segment $[\check P_0.\check P^-]$}}
\begin{equation*}
\R^+ \ni \ell_0 \mapsto \gls{nuell0} = (\theta(\ell_0),\omega(\ell_0)) \in \R^2.
\end{equation*}
The statement is proved if we show that this curve crosses the boundary of the negativity region $N = N_\mathrm{segm} \cup N_\mathrm{arc}$ (Propositions \ref{Prop:neg_region} and \ref{Prop:regions_pos_neg_segm}) only in one point: if this happens, the crossing point is a minimum for the quantity $\tilde r(\bar \phi;\ell_0)$, see Fig. \ref{Fig:nu_funct}.

In the case the initial value is already outside the negativity region (i.e. $\theta > h_1(\Delta \phi)$), then we deduce that $\ell_0 = 0$. This could happen for example in the case $\theta>\hat\theta$, as for instance in the second picture of Figure \ref{fig:tent:solution}, where the angle $\theta$ is sufficiently big. In this case $\ell_0=0$ and the optimal solution coincides with the fastest saturated spiral.

We observe that after the point $P_0$, the segment $[\check P_0,\check P^-]$ starts tangent to the spiral $\tilde \zeta$, i.e. $\theta = 0$: thus we always start in the negativity region of $\delta \tilde r(\bar \phi)$ (i.e. $\ell_0 > 0$) apart from the case where $\check P_0 = P_0$. Summarizing, we encounter two different situations, $P_0=\check P_0$ and $P_0\not= \check P_0$. In the first case we compute $\theta$, and if $\theta>\hat\theta$ then $\ell_0=0$ and the optimal solution is the fastest saturated spiral (if it is not, we find $\ell_0$ such that the angle formed with the solution is exactly $\hat\theta$). In the second situation, since we start tangent to the fastest saturated spiral in the point $\check P_0$, then $\theta=0$ and we are in the negativity region. Again, $\ell_0$ is such that the fastest saturated spiral computed at the endpoint $\check P_1$ forms exactly an angle $\hat\theta$ in the segment case and  $h_1(\Delta \phi)$ in the arc case.


We can compute the derivative of the curve $\nu(\ell_0)$ by using the formulas of the previous section.

\begin{description}
\item[Segment case] we have from \eqref{Equa:delta_theta_segm_1}
\begin{equation*}
\frac{d\bar \theta}{d\ell_0} = \frac{1}{|\check P_2 - \check P_1|} \big( \sin(\theta) - \cot(\bar \alpha) ( 1 - \cos(\theta)) \big) - \frac{\sin(\check \theta_1)}{\check r(\check \phi_1)},
\end{equation*}
\begin{equation*}
\frac{d\theta}{d\ell_0} = \frac{1}{|\check P_2 - \check P_1|} \big( \sin(\theta) - \cot(\bar \alpha) \big( 1 - \cos(\theta) \big)\big),
\end{equation*}
where we have used Equation \eqref{Equa:delta_theta_segm_1} and where $\check r(\check \phi_1)$ is the length of the segment part of the optimal ray ending in the initial point $\check P_1$ and $\check \theta_1 = \beta^-(\check \phi_1) =\beta^+(\check\phi_1)-\theta$. See Figure \ref{fig:capitolo:7} as a reference. We point out that the derivative is computed in the point $\check P_1$.
\begin{figure}
	\centering
	\includegraphics[scale=0.7]{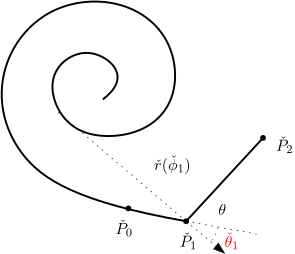}
	\caption{The situation of Proposition \ref{Prop:unique_optimal_guess} in the segment case.}
	\label{fig:capitolo:7}
\end{figure}
 Hence, since $\frac{d\theta}{d\ell_0} > 0$ unless $\theta = 0$, the curve $\nu(\ell_0)$ which in the segment case is $(\theta(\ell_0),\bar\theta(\ell_0))$ can be represented as the graph of a differentiable function $\bar \theta(\theta)$, parameterized with $\theta$ increasing. Since in the segment case the boundary of the negativity region is $\theta = \hat \theta$, this curve exits the negativity region only in one point, corresponding to a unique value $\ell_0$.
 
 In the case $\check P_0 = \check P_1 = \check P_2$ one can check by a local expansion about a saturated point that
 \begin{equation*}
 \frac{d\theta}{d\ell_0} = \frac{\sin(\bar \alpha)}{\tilde r(\check \phi_0)}, \quad \frac{d\bar \theta}{d\ell_0} = 0.
 \end{equation*}
 Hence the curve $\nu(0)$ starts with direction $(0,1)$.

\item[Arc case] here we use \eqref{Equa:delta_Delta_phi} and the fact that the angle of the level set with the ray is $\frac{\pi}{2}$ to write
\begin{equation*}
\frac{d\omega}{d\ell_0} = \frac{d\Delta \check \phi}{d\ell_0} = \frac{1}{|\check P_2 - \check P_1|} \big( \sin(\theta) - \cot(\bar \alpha) (1 - \cos(\theta)) + \cot(\bar\alpha) \sin(\theta) \Delta \check \phi \big) - \frac{\cos(\theta)}{\check r(\check \phi^-)},
\end{equation*}
\begin{equation*}
\frac{d\theta}{d\ell_0} = \frac{\cos(\theta)}{\check r(\check \phi^-)},
\end{equation*}
with $\check r(\check \phi^-)$ the segment part of the optimal ray ending in $\check P^-$. Hence the curve can be represented again as the graph of a function $\Delta \check \phi(\theta)$, oriented along the increasing $\theta$ axis, and such that
\begin{equation*}
\frac{d\Delta \check \phi}{d\theta} = \frac{\check r(\check \phi^-)}{|\check P_2 - \check P_1|} \frac{\sin(\theta) - \cot(\bar \alpha) (1 - \cos(\theta) + \sin(\theta) \Delta \phi)}{\cos(\theta)} - 1.
\end{equation*}
The computation of Lemma \ref{Lem:study_g_1_arc} gives that the derivative of the boundary curve $h_1(\Delta \check \phi)$ is (see Equation \eqref{Equa:dh1_dDeltaphi})
\begin{equation*}
\bigg( \frac{dh_1}{d\Delta \check \phi} \bigg)^{-1} = - 1 - \frac{\cos(\theta) \Delta \check \phi}{\sin(\theta) - \cot(\bar \alpha) (1 - \cos(\theta) + \sin(\theta) \Delta \check \phi)},
\end{equation*}
and then it would follow that the curves are not transversal only if
\begin{equation*}
\frac{\check r(\check \phi^-)}{|\check P_2 - \check P_1|} = - \frac{\cos(\theta)^2 \Delta \check \phi}{(\sin(\theta) - \cot(\bar \alpha) (1 - \cos(\theta) + \sin(\theta) \Delta \check \phi))^2} \leq 0,
\end{equation*}
which is impossible.
\end{description}
Thus the curves are transversal, and the crossing occurs from the negative region to the positive region: indeed this is clear in the segment case, while in the arc case
$$
\text{if} \quad \sin(\theta) - \cot(\bar \alpha)(1 - \cos(\theta) + \sin(\theta) \Delta \check \phi) < 0 \quad \text{then} \quad \frac{d\Delta \check \phi}{d\theta} < - 1 < \bigg( \frac{dh_1}{d\Delta \check \phi} \bigg)^{-1},
$$
$$
\text{if} \quad \sin(\theta) - \cot(\bar \alpha)(1 - \cos(\theta) + \sin(\theta) \Delta \check \phi) > 0 \quad \text{then} \quad \frac{d\Delta \check \phi}{d\theta} > -1 > \bigg( \frac{dh_1}{d\Delta \check \phi} \bigg)^{-1},
$$
meaning that if you cross the negative region, then the crossing occurs always with the same direction.
The above inequalities yield the statement.
\end{proof}

\begin{corollary}
\label{Cor:unique}
Assume that along the line $\check P_0 + \R^+ e^{i \bar \phi}$ we find a point $\check P^-$ such that the fastest saturated spiral starting from $\check P^-$ form the critical angle $\theta = \hat \theta$ in the segment case (i.e. $\check P^- = \check P_1$) or $\theta = h_1(\Delta \phi)$ in the arc case. Then this is the starting point of the unique spiral of Proposition \ref{Prop:unique_optimal_guess}.
\end{corollary}

\begin{proof}
Indeed, the map $\ell_0 \to (\theta,\omega)$ is a smooth curve strictly increasing in $\theta$ and crossing the line of critical angle $\bar \theta$ (segment case) or $h_1(\Delta \phi)$ (arc case) exactly in one point.
\end{proof}

In particular, the segment case is admissible if and only if the tip of the tent $\check P_1$ is such that the angle of the optimal ray ending in $\check P_1$ and the segment $\ell_1$ is $\leq \frac{\pi}{2}$: this is the criterion we will check in Section \ref{Sss:having_tent_check}.


\begin{remark}
\label{Rem:trans_segm_arc}
Observe that along the line $\Delta \check \phi = 0$, i.e. $\bar \theta = \frac{\pi}{2}$ and $\check \theta_1 = \frac{\pi}{2} - \theta$, we have
\begin{equation*}
\frac{d\omega}{d\theta} =\frac{\check r(\check\phi^-)}{|\check P_2-\check P_1|}\frac{\sin\theta-\cot\bar\alpha(1-\cos\theta)}{\cos\theta}-1.
\end{equation*}
We want to check whether from the arc case (meaning that the fastest saturated spiral starts with an arc in the point $\check P^-$) there is a transition from the segment solution.
By observing that the spirals are convex and the angle at the top of the segment $\ell_0$ with the optimal ray is $\bar \alpha$, by elementary geometrical reasons we must have 
\begin{equation*}
|\check P_2-\check P_1|\tan\bar\alpha=\ell_1\tan\bar\alpha\leq \check r(\check \phi_-),
\end{equation*}
in particular the previous inequality becomes 
\begin{equation*}
\frac{d\omega}{d\theta} > \tan(\bar\alpha)\frac{\sin(\theta)-\cot(\bar\alpha)(1-\cos(\theta))}{\cos(\theta)}-1 = \frac{\tan(\bar \alpha) \sin(\theta) - 1}{\cos(\theta)}.
\end{equation*}
This quantity is positive whenever $\theta\geq\sin^{-1}(\cot(\bar\alpha)) = 0.426813$. In particular, in this range there is no transition from arc solutions to segment solutions. 
\end{remark}

The previous proposition proves that construction of the spiral $\check r(\phi;s_0)$ at the beginning of the section is well defined. To prove that it is the minimizer for the minimization problem \eqref{eq:min:probl_pert}, we need to study the positivity of $\check r(\bar \phi;s_0)$ w.r.t. to perturbations. More precisely, we want to compute the derivative \newglossaryentry{deltarcheck}{name=\ensuremath{\delta \check r(\phi;s_0)},description={derivative of the candidate optimal slosing spiral w.r.t. the starting position $s_0$}} 
\begin{equation*}
\gls{deltarcheck} = \frac{d}{ds_0} \check r(\phi;s_0).
\end{equation*}
Similarly to the computation of $\delta \tilde r(\phi;s_0)$, depending on the structure of the solution $\check r$ several cases are possible: the initial fastest saturated part of $\check r(\phi;s_0)$ can be a segment or a level-set arc, and the segment starting in $\check P_0$ with length $\ell_0$ may be followed by a segment (tent case) or an arc.

The most important case is when the last round starts with segment (tent case), which we analyze in the next section.

\subsection{Formulas for the tent case}
\label{Ss:comput_tent_new}

\begin{figure}
\resizebox{.5\textwidth}{!}{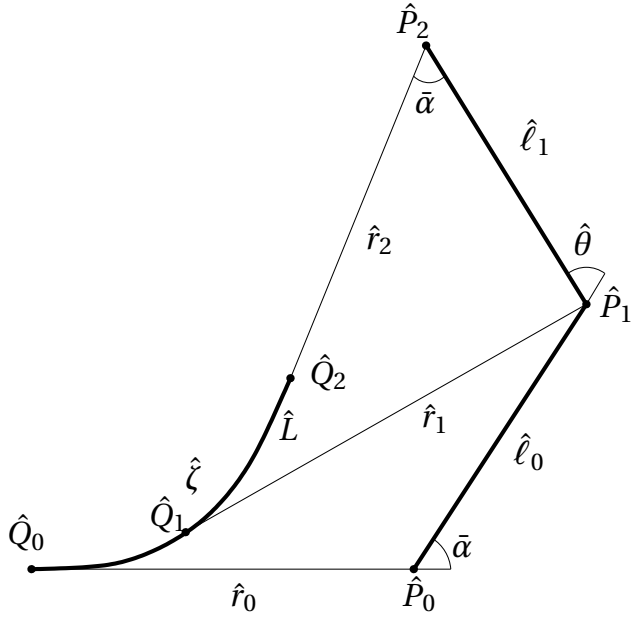}
\caption{The geometry used to obtain the formulas for the tent case, Section \ref{Ss:comput_tent_new}.}
\label{Fig:tent_comput_geo}
\end{figure}

In order to study the correction to the fastest saturated spiral $\tilde r$ after the angle $\check \phi_0$, we need to study some geometric relations among the length $\ell_0$ of this segment and the underlying spiral $\tilde r$. This section is thus devoted to obtain these formulas, depending on the following quantities, which in order to distinguish from the previous ones are denoted with a hat, see Fig. \ref{Fig:tent_comput_geo}:
\begin{itemize}
\item the entrance ray \newglossaryentry{rhat0}{name=\ensuremath{\hat r_0},description={entry ray for the formulas for tent or arc perturbation}} \gls{rhat0} at angle $0$ initial position \newglossaryentry{Qhat0tent}{name=\ensuremath{\hat Q_0},description={base point for $\hat r_0$ in the tent or arc perturbation}} \gls{Qhat0tent}, the end point \newglossaryentry{Phat0}{name=\ensuremath{\hat P_0},description={initial point of the tent or arc perturbation}} \gls{Phat0} of $\hat r_0$ is saturated;
\item the exit ray \newglossaryentry{rhat2}{name=\ensuremath{\hat r_2},description={exit ray for the tent or arc computations}} \gls{rhat2} at angle $\hat \theta$ with initial position \newglossaryentry{Qhat2tent}{name=\ensuremath{\hat Q_2},description={base point for $\hat r_2$ in the tent or arc perturbation}} \gls{Qhat2tent}, the end point of $\hat r_2$ is saturated;
\item the length of segments \newglossaryentry{ellhat0}{name=\ensuremath{\hat \ell_0},description={length of the first segment in the tent or arc computations}} \newglossaryentry{ellhat1}{name=\ensuremath{\hat \ell_1},description={length of the second segment in the tent or arc computations}} \gls{ellhat0}, \gls{ellhat1};
\item the length \newglossaryentry{Lhattent}{name=\ensuremath{\hat L},description={length of the spiral between $Q_0,\hat Q_2$ in the tent or arc computations}} \gls{Lhattent} of the spiral at the base between the points $\hat Q_0, \hat Q_2$ (angle opening $[0,\hat \theta]$).
\end{itemize}
We will also give conditions to verify that it is admissible, which is equivalent to check that the angle formed by segment \newglossaryentry{Qhat1tent}{name=\ensuremath{\hat Q_1},description={point of the curve $\zeta$ where the optimal ray ending in the tip $\hat P_1$ of the ten or arc is detaching}} \newglossaryentry{Phat1tent}{name=\ensuremath{\hat P_1},description={tip of the tent or arc}} $[\gls{Qhat1tent},\gls{Phat1tent}]$ and \newglossaryentry{Phat2tent}{name=\ensuremath{\hat P_2},description={saturated final point of the tent or arc}} the second segment $[\gls{Phat1tent},\gls{Phat2tent}]$ is $\angle \hat Q_1\hat P_1\hat P_2\geq\frac{\pi}{2}$.

Since $\hat P_0$ is saturated, the saturation condition at the end point of the segment $\ell_1$ reads as
\begin{equation}
\label{Equa:saturation_tent}
\hat r_2 + \hat L - \hat r_0 - \cos(\bar \alpha) (\hat \ell_0 + \hat \ell_1) = 0,
\end{equation}
and the vector relation among the various points is
\begin{equation}
\label{Equa:vector_Q_1_Q_2}
\hat Q_0 + \hat r_0 (1,0) + \hat \ell_0 e^{\i \bar \alpha} + \hat \ell_1 e^{\i (\bar \alpha + \hat \theta)} = \hat Q_2 + \hat r_2 e^{\i \hat \theta}.
\end{equation}

Projecting the second equation in the direction $e^{\i \hat \theta}$ we obtain
\begin{equation*}
(\hat Q_2 - \hat Q_0) \cdot e^{\i \hat \theta} + \hat r_2 - \hat \ell_1 \cos(\bar \alpha) - \hat \ell_0 \cos(\bar \alpha - \hat \theta) - \hat r_0 \cos(\hat \theta) = 0,
\end{equation*}
and then using the saturation condition
\begin{equation*}
(\hat Q_2 - \hat Q_0) \cdot e^{\i \hat \theta} - \hat \ell_0 \cos(\bar \alpha - \hat \theta) - \hat r_0 \cos(\hat \theta) = \hat L - \hat r_0 - \cos(\bar \alpha) \hat \ell_0,
\end{equation*}
\begin{equation*}
\begin{split}
\hat \ell_0 &= \frac{(\hat Q_2 - \hat Q_0) \cdot e^{\i \hat \theta} - \hat L + \hat r_0 (1 - \cos(\hat \theta))}{\cos(\bar \alpha - \hat \theta) - \cos(\bar \alpha)}. 
\end{split}
\end{equation*}
Projecting again \eqref{Equa:vector_Q_1_Q_2} on the direction $e^{\i (\alpha + \hat \theta - \frac{\pi}{2})}$
\begin{equation*}
(\hat Q_2 - \hat Q_0) \cdot e^{\i (\bar \alpha + \hat \theta - \frac{\pi}{2})} + \hat r_2 \sin(\bar \alpha) - \hat \ell_0 \sin(\hat \theta) - \hat r_0 \sin(\bar \alpha + \hat \theta) = 0,
\end{equation*}
which gives
\begin{equation*}
\hat r_2 = \frac{\hat r_0 \sin(\bar \alpha + \hat \theta) + \hat \ell_0 \sin(\hat \theta) - (\hat Q_2 - \hat Q_0) \cdot e^{\i (\bar \alpha + \hat \theta - \frac{\pi}{2})}}{\sin(\bar \alpha)}.
\end{equation*}
We can now compute $\hat \ell_1$ from \eqref{Equa:saturation_tent} as
\begin{equation*}
\hat \ell_1 = \frac{\hat r_2 + \hat L  -\hat r_0 - \cos(\bar \alpha) \hat \ell_0}{\cos(\bar \alpha)}.
\end{equation*}
Finally, the relation between $\hat L$ and $\hat Q_2 - \hat Q_0$ is given by integrating the curve \newglossaryentry{zetahattent}{name=\ensuremath{\hat \zeta},description={convex curve for computing the formulas for the tent or arc case}} \gls{zetahattent}: 
\begin{equation*}
\frac{d\hat \zeta}{d\theta} = e^{\i \theta} \hat R(\theta), \quad \hat \zeta(\theta) - \hat \zeta(\theta_0) = \int_{\theta_0}^{\theta} e^{\i \omega} \hat R(\omega) d\omega,
\end{equation*}
where \newglossaryentry{Rhatcurvsegm}{name=\ensuremath{\hat R},description={curvature of the convex curve $\hat \zeta$ for the computations of the tent or arc formulas}} \gls{Rhatcurvsegm} is the radius of curvature of $\hat \zeta$. The above formula holds also when the curvature is $\infty$, by considering $ds = \hat R d\theta$ as a measure.

We summarize the results obtained in the following lemma.

\begin{lemma}
\label{Lem:relation_admissibility_tent}
Consider the geometric configuration as in Fig. \ref{Fig:tent_comput_geo}, with $\hat P_0,\hat P_2$ saturated. Then the following geometric relations hold:
\begin{equation*}
\hat Q_2 - \hat Q_0 = \int_0^{\hat \theta} e^{\i \omega} \hat R(\omega) d\omega, \quad \hat R(\theta) \ \text{curvature of } \ \hat \zeta,
\end{equation*}
\begin{equation*}
\hat L = \int_0^{\hat \theta} \hat R(\theta) d\theta,
\end{equation*}
\begin{equation}
\label{Equa:ell_0_tent_general}
\hat \ell_0 = \frac{(\hat Q_2 - \hat Q_0) \cdot e^{\i \hat \theta} - \hat L + \hat r_0 (1 - \cos(\hat \theta))}{\cos(\bar \alpha - \hat \theta) - \cos(\bar \alpha)},
\end{equation}
\begin{equation*}
\hat r_2 = \frac{\hat r_0 \sin(\bar \alpha + \hat \theta) + \hat \ell_0 \sin(\hat \theta) - (\hat Q_2 - \hat Q_0) \cdot e^{\i (\bar \alpha + \hat \theta - \frac{\pi}{2})}}{\sin(\bar \alpha)},
\end{equation*}
\begin{equation*}
\hat \ell_1 = \frac{\hat r_2 + \hat L -\hat r_0 - \cos(\bar \alpha) \hat \ell_0}{\cos(\bar \alpha)}.
\end{equation*}
\end{lemma}

Since all formulas are linear, we obtain the following corollary.

\begin{corollary}
\label{Cor:tent_corol_for_pert}
The same formulas of Lemma \ref{Lem:relation_admissibility_tent} hold for the perturbation $\delta \check r(\phi;s_0)$.
\end{corollary}

\newglossaryentry{tauhat}{name=\ensuremath{\hat \tau},description={angle shift \gls{thetahat} in the $\tau$ coordinates}}
For future use, define the shift \gls{tauhat} in the $\tau$-coordinate corresponding to the angle shift $\hat \theta$:
\begin{equation}
\label{Equa:hattau_def}
\hat \tau = \frac{\hat \theta}{2\pi + \bar \alpha}.
\end{equation}

\subsubsection{Estimates for the admissibility of the tent}
\label{Sss:having_tent_check}

 The next series of lemmas give sufficient conditions which ensure that the construction above is admissible, i.e. the angle between the segments $[\hat Q_1,\hat P_1]$ and $[\hat P_1,\hat P_2]$ is $\angle \hat Q_1\hat P_1\hat P_2\geq\frac{\pi}{2}$. If \newglossaryentry{phihat}{name=\ensuremath{\hat \phi},description={angular coordinate correponding to $\hat Q_1 = \hat \zeta(\hat \phi)$ for the tent or arc case}} $\hat Q_1 = \hat \zeta(\gls{phihat})$, then we must require 
\begin{equation*}
\bar \alpha + \hat \theta - \hat \phi < \frac{\pi}{2}, \quad \hat \phi \geq 0.113642.
\end{equation*}
\begin{lemma}
\label{Lem:when_tent}
Assume that we are in the configuration of Fig. \ref{Fig:tent_comput_geo} {\color{red}with $\hat P_0$ and $\hat P_2$ saturated}: if \newglossaryentry{Kadmis}{name=\ensuremath{K_\mathrm{tent}},description={constant for verifying whether the tent is admissible}}
\begin{equation*}
\frac{\hat L}{\hat r_0} \leq \gls{Kadmis} = \frac{\sin(\bar \alpha) (\sin(\bar \alpha + \hat \theta) - \sin(\bar \alpha))}{\cos(\hat \theta) (1 - \sin(\bar \alpha))} < 0.966795,
\end{equation*}
then the tent is admissible.
\end{lemma}


\begin{proof}
We have to study the minimum of the function $\frac{\hat L}{\hat r_0}$ in order to have that $\hat P_1$ belongs to the critical line $\hat Q_1 + \R^+ e^{\i (\bar \alpha + \hat \theta - \frac{\pi}{2})}$: we will show that in order to have this, we must have $\frac{\hat L}{\hat r_0} \geq K_\mathrm{tent}$.

If $\hat Q_1 + \R^+ e^{\i (\bar \alpha + \hat \theta - \frac{\pi}{2})}$, we must satisfy the vector relation
\begin{equation*}
\hat Q_0 + \hat r_0 e^{\i 0} + \hat \ell_0 e^{i \bar \alpha} = \hat P_1,
\end{equation*}
with by \eqref{Equa:ell_0_tent_general}
\begin{equation*}
\hat \ell_0 \big( \cos(\bar \alpha - \hat \theta) - \cos(\bar \alpha) \big) = (\hat Q_2 - \hat Q_0) \cdot e^{\i \hat \theta} - \hat L + \hat r_0 (1 - \cos(\hat \theta)).
\end{equation*}

We can fix the point $\hat Q_1$, and observe that the geometric configuration with minimal length is when $\hat Q_2 \in \hat Q_1 + \R^+ e^{\i (\bar \alpha + \hat \theta - \frac{\pi}{2})}$: indeed in the formula for $\hat \ell_0$ we can write
\begin{equation*}
(\hat Q_2 - \hat Q_0) \cdot e^{\i \hat \theta} - \hat L = \int_0^{\hat \theta} (\cos(\hat \theta - \theta) - 1) d\hat L(\theta) = \bigg[ \int_0^{\bar \alpha + \hat \theta - \frac{\pi}{2}} + \int_{\bar \alpha + \hat \theta - \frac{\pi}{2}}^{\hat \theta} \bigg] (\cos(\hat \theta - \theta) - 1) d\hat L(\theta).
\end{equation*}
In the second integral one can reduce the quantity
\begin{equation*}
\int_{\bar \alpha + \hat \theta - \frac{\pi}{2}}^{\hat \theta} d\hat L(\theta)
\end{equation*}
if we take the maximal value of $1 - \cos(\hat \theta - \theta)$, which is when $\theta = \bar \alpha + \hat \theta - \frac{\pi}{2}$. This means that we can assume $\hat Q_2 \in \hat Q_0 + \R^+ e^{\i (\bar \alpha + \hat \theta - \frac{\pi}{2})}$.

The saturation gives then
\begin{equation*}
\hat L + |\hat P_2 - \hat Q_2| - \hat r_0 = \hat L + \frac{|\hat P_1 - \hat Q_2|}{\sin(\bar \alpha)} - \hat r_0 = \cos(\bar \alpha) (\hat \ell_0 + \hat \ell_1) = \cos(\bar \alpha) \big( \hat \ell_0 + \cot(\bar \alpha) |\hat P_1 - \hat Q_2| \big),
\end{equation*}
which is
\begin{equation*}
\hat L + \sin(\bar \alpha) |\hat P_1 - \hat Q_2| - \hat r_0 = \cos(\bar \alpha) \hat \ell_0.
\end{equation*}

A configuration for which the quantity $K_\mathrm{tent}$ of the statement is reached is in Fig. \ref{Fig:worst_tent_case}. For this configuration, elementary geometry gives 
\begin{equation*}
\hat r_0 \frac{\sin(\hat \phi)}{\sin(\bar \alpha - \hat \phi)} = \hat \ell_0 = \frac{\hat r_0 (1 - \cos(\hat \theta)) - \hat L (1 - \cos(\hat \theta - \hat \phi))}{\cos(\bar \alpha - \hat \theta) - \cos(\bar \alpha)},
\end{equation*}
where we have used Equation \eqref{Equa:ell_0_tent_general}, i.e.
\begin{equation*}
\frac{\hat L}{\hat r_0} = \frac{\sin(\bar \alpha - \hat \phi)(1 - \cos(\hat \theta)) - \sin(\hat \phi) (\cos(\bar \alpha - \hat \theta) - \cos(\bar \alpha))}{\sin(\bar \alpha - \hat \phi)(1 - \cos(\hat \theta - \hat \phi))} = \frac{\sin(\bar \alpha) (\cos(\hat \phi) - \cos(\hat \theta - \hat \phi))}{\sin(\bar \alpha - \hat \phi) (1 - \cos(\hat \theta - \hat \phi))}.
\end{equation*}
The function
\begin{equation}
\label{Equa:Loverr0_tent}
\hat \phi \mapsto \frac{\sin(\bar \alpha) (\cos(\hat \phi) - \cos(\hat \theta - \hat \phi))}{\sin(\bar \alpha - \hat \phi) (1 - \cos(\hat \theta - \hat \phi))}
\end{equation}
is such that if $\hat L_0/\hat r_0<0.966795$, then the angle corresponding to $\hat P_1$ is larger than $\bar \alpha + \hat \theta - \frac{\pi}{2}$, see Fig. \ref{Fig:Loverr0_tent}.

\begin{figure}
\begin{subfigure}{.32\textwidth}
\resizebox{\linewidth}{!}{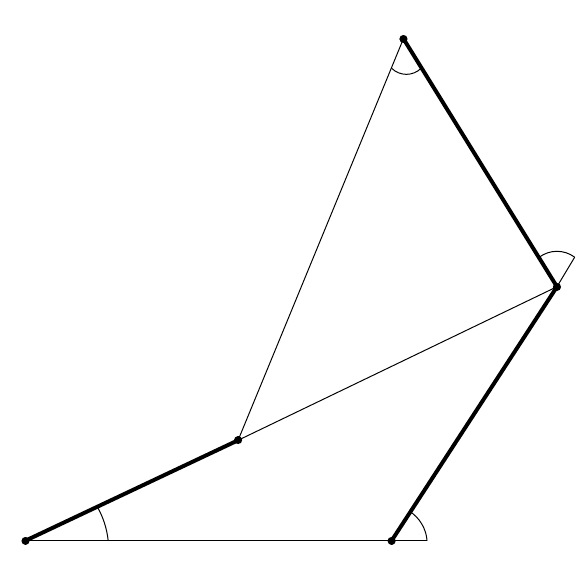}
\caption{The worst scenario of Lemma \ref{Lem:when_tent}.}
\label{Fig:worst_tent_case}
\end{subfigure}
\hfill
\begin{subfigure}{.32\linewidth}
\includegraphics[width=\linewidth]{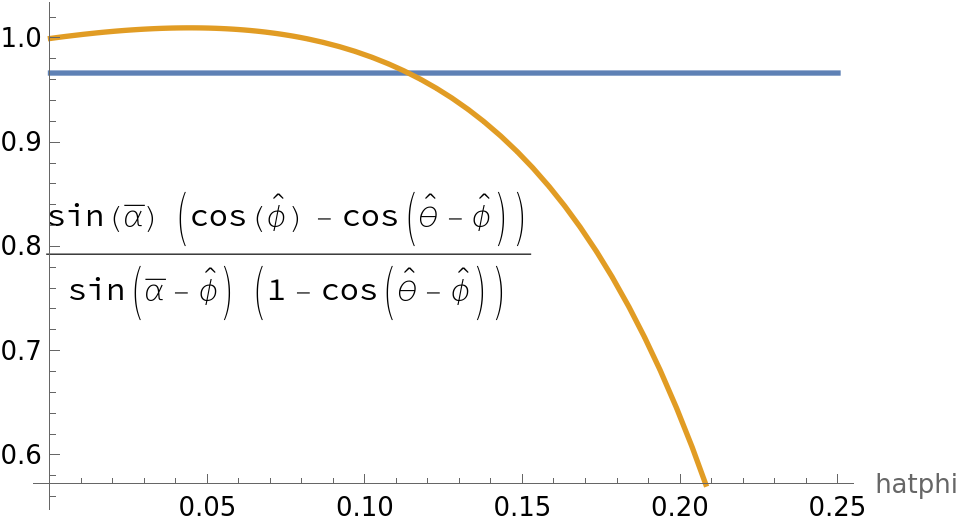}
\caption{Plot of the function (\ref{Equa:Loverr0_tent}) in blue, in orange the value of $\frac{\hat L}{\hat r_0}$ for the critical angle $\hat \phi = \bar \alpha + \hat \theta - \frac{\pi}{2}$.}
\label{Fig:Loverr0_tent}
\end{subfigure}
\hfill
\begin{subfigure}{.32\linewidth}
\resizebox{\linewidth}{!}{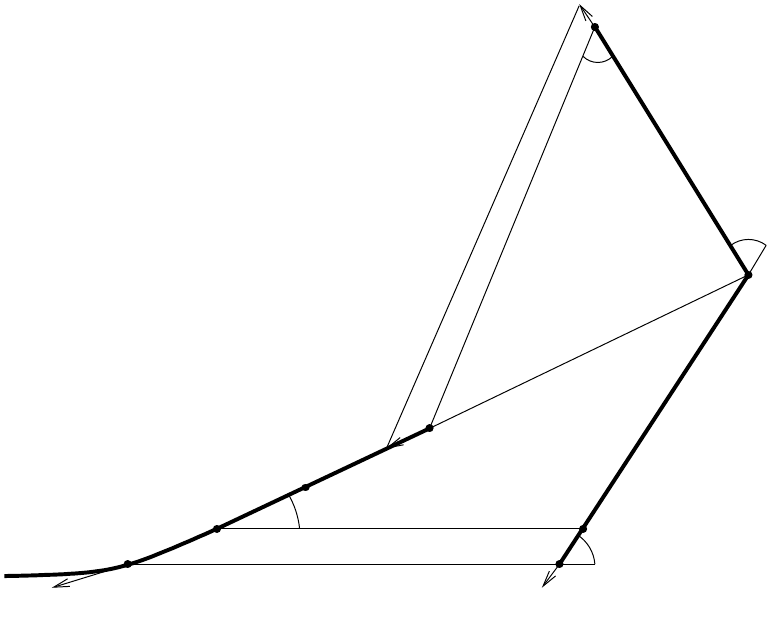}
\caption{The geometry of the perturbation used to estimate the tent in the general case.}
\label{Fig:worst_case_tent_L}
\end{subfigure}
\caption{Analysis of Lemma \ref{Lem:when_tent}.}
\label{Fig:tent_2_case}
\end{figure}

For a generic configuration, we consider a variation of the geometry as follows (see Fig. \ref{Fig:worst_case_tent_L}): keeping $\hat Q_1$, $\hat P_1$ fixed, we move backward the point $\hat Q_0$ along a convex curve and vary $\hat r_0,\hat \ell_0$ accordingly in order to preserve their directions. In particular, if $\theta$ is the direction of the tangent of the convex curve in $\hat Q_0$, we obtain from the vector relations
\begin{equation*}
\delta \hat \ell_0 = \delta \hat Q_0 \frac{\sin(\theta)}{\sin(\bar \alpha)}, \quad \delta \hat r_0 = \delta \hat Q_0 \cos(\theta) - \delta \hat \ell_0 \cos(\bar \alpha) = \delta \hat Q_0 (\cos(\theta) - \sin(\theta) \cot(\bar \alpha)).
\end{equation*}
Plugging this relation into the formula for $\hat \ell_0$ we obtain
\begin{equation*}
\delta \hat \ell_0 (\cos(\bar \alpha - \hat \theta) - \cos(\bar \alpha)) = \delta \hat Q_0 (- 1 + \cos(\hat \theta - \theta)) + \delta \hat Q_2 ( 1 - \sin(\bar \alpha)) + \delta \hat r_0 (1 - \cos(\hat \theta)),
\end{equation*}
because $\hat Q_2$ is assumed to move along the curve $\hat Q_1 + \R^+ e^{\i (\bar \alpha + \hat \theta - \frac{\pi}{2})}$ towards $\hat Q_1$. Substituting we obtain
\begin{equation*}
\begin{split}
\delta \hat Q_2 (1 - \sin(\bar \alpha)) &= \delta \hat \ell_0 (\cos(\bar \alpha - \hat \theta) - \cos(\bar \alpha)) + \delta \hat Q_0 (1 - \cos(\hat \theta - \theta)) - \delta \hat r_0 (1 - \cos(\hat \theta)) \\
&= \delta \hat Q_0 \bigg[ \frac{\sin(\theta)}{\sin(\bar \alpha)} (\cos(\bar \alpha - \hat \theta) - \cos(\bar \alpha)) + (1 - \cos(\hat \theta - \theta)) - (\cos(\theta) - \cot(\bar \alpha) \sin(\theta)) (1 - \cos(\hat \theta)) \bigg] \\
&= \delta \hat Q_0 (1 - \cos(\theta)).
\end{split}
\end{equation*}
Hence we have
\begin{equation*}
\frac{\delta \hat L}{\delta \hat r_0} = \frac{\delta \hat Q_0 - \delta \hat Q_2}{\delta \hat r_0} = \frac{\sin(\bar \alpha) (\cos(\theta) - \sin(\bar \alpha))}{(1 - \sin(\bar \alpha)) \sin(\bar \alpha - \theta)}.
\end{equation*}
Note that this formula coincides with $K_\mathrm{tent}$ for $\theta = \bar \alpha + \hat \theta - \frac{\pi}{2}$. By taking of the r.h.s. we have that
\begin{equation*}
\frac{d}{d\theta} \frac{\sin(\bar \alpha) (\cos(\theta) - \sin(\bar \alpha))}{(1 - \sin(\bar \alpha)) \sin(\bar \alpha - \theta)} = \frac{\sin(\bar \alpha)}{1 - \sin(\bar \alpha)} \frac{\cos(\bar \alpha) - \cos(\bar \alpha - \theta) \sin(\bar \alpha)}{\sin(\bar \alpha - \theta)^2}
\end{equation*}
whose maximum in the interval $\theta \in [0,\bar \alpha + \hat \theta - \frac{\pi}{2}]$ is for $\theta = \bar \alpha + \hat \theta - \frac{\pi}{2}$ and equal to
\begin{equation*}
\frac{\sin(\bar \alpha)}{1 - \sin(\bar \alpha)} \frac{\cos(\bar \alpha) - \sin(\hat \theta) \sin(\bar \alpha)}{\cos(\hat \theta)^2} = - 1.03949 < 0.
\end{equation*}
Hence
\begin{equation*}
\delta \bigg( \frac{\hat L}{\hat r_0} - K_\mathrm{tent} \bigg) = \frac{\delta \hat r_0}{\hat r_0} \bigg( \frac{\delta \hat L}{\delta \hat r_0} - \frac{\hat L}{\hat r_0} \bigg) > \frac{\delta \hat r_0}{\hat r_0} \bigg( K_\mathrm{tent} - \frac{\hat L}{\hat r_0} \bigg) \quad \text{for} \ \theta < \bar \alpha + \hat \theta - \frac{\pi}{2},
\end{equation*}
and then we conclude that
\begin{equation*}
\frac{\hat L}{\hat r_0} - K_\mathrm{tent} > 0
\end{equation*}
unless we are in the critical configuration where $\frac{\hat L}{\hat r_0} - K_\mathrm{tent} = 0$.

We conclude that is $\frac{\hat L}{\hat r_0} < K_\mathrm{tent}$ then the angle $\hat \zeta(\hat \phi) = \hat Q_1$ is larger than the critical angle $\bar \alpha + \hat \theta - \frac{\pi}{2}$.
\end{proof}

We have proved that, under certain conditions on the length $\hat L$ and on $\hat r_0$, the tent is admissible, meaning that $\angle\hat Q_1\hat P_1\hat P_2\geq\frac{\pi}{2}$. In particular, if $\frac{\hat L}{\hat r_0}\geq 0.966795..$, then the fasted saturated spiral computed at $\hat P_1$ starts with an arc.

We next prove that in such a case (i.e. when we start with an arc), after one round the functional $r(\phi) - 2.08 L(\phi - 2\pi - \bar \alpha)$ will be still negative. This is a consequence of the fact that $\hat r_0 - \hat L \sim 0.04$, so that we need $\hat L$ to be almost of the size of $\hat r_0$, and we need at least one round to make $r(\phi)$ more than twice the length $L(\phi - 2\pi - \bar \alpha)$.

Consider indeed the spiral $\tilde r(\phi;\phi_0)$, and let $\hat \phi_0$ be the initial angle for the tent and $\hat \phi_2 = \hat \phi_0 + \hat \theta$ the final angle, and let \newglossaryentry{phihat3}{name=\ensuremath{\hat \phi_3},description={maximal angle for using the length bound as admissibility criterion for the tent}} $\gls{phihat3} \geq \hat \phi_0$ an angle where we are going to estimate the functional $r(\cdot) - 2.08 L(\cdot - 2\pi - \bar \alpha)$. By considering the fastest growing spiral
\begin{equation*}
\dot r = \cot(\bar \alpha) r,
\end{equation*}
and observing that the best situation is the one of the previous lemma (because the length $\hat L$ is the minimal and is applied at the latest possible angle), we obtain that if the tent is not admissible and $\hat \phi_3 \geq \hat \phi_2$
\begin{equation*}
\begin{split}
\tilde r(\hat \phi_3) &\leq \tilde r(\hat \phi_0) e^{\cot(\bar \alpha)(\hat \phi_3 - \hat \phi_0)} - (\tilde L(\hat \phi_0 + \hat \theta) - \tilde L(\hat \phi_0)) e^{\cot(\bar \alpha)(\hat \phi_3 - \hat \phi_0 + \frac{\pi}{2} - \bar \alpha - \hat \theta)} \\
&\leq (\tilde L(\hat \phi_0 + \hat \theta) - \tilde L(\hat \phi_0)) \bigg( \frac{e^{\cot(\bar \alpha) (\bar \alpha + \hat \theta - \frac{\pi}{2})}}{K_\mathrm{tent}} - 1 \bigg) e^{\cot(\bar \alpha) (\hat \phi_3 - \hat \phi_0 + \frac{\pi}{2} - \bar \alpha - \hat \theta)} \\
&\leq \tilde L(\hat \phi_3 - 2\pi - \bar \alpha) \bigg( \frac{e^{\cot(\bar \alpha) (\bar \alpha + \hat \theta - \frac{\pi}{2})}}{K_\mathrm{tent}} - 1 \bigg) e^{\cot(\bar \alpha) (\hat \phi_3 - \hat \phi_0 + \frac{\pi}{2} - \bar \alpha - \hat \theta)}.
\end{split}
\end{equation*}
If now
\begin{equation*}
\hat \phi_3 > \phi_0 + \omega + 2\pi,
\end{equation*}
we must have that Corollary \ref{Cor:bound_length_gen} is valid, i.e.
\begin{equation*}
\bigg( \frac{e^{\cot(\bar \alpha) (\bar \alpha + \hat \theta - \frac{\pi}{2})}}{K_\mathrm{tent}} - 1 \bigg) e^{\cot(\bar \alpha) (\hat \phi_3 - \hat \phi_0 + \frac{\pi}{2} - \bar \alpha - \hat \theta))} \geq 2.08.
\end{equation*}
Numerically one obtains
\begin{equation*}
\hat \phi_3 - \hat \phi_0 \geq \bar \alpha + \hat \theta - \frac{\pi}{2} + \tan(\bar \alpha) \ln \bigg( \frac{2.08}{\frac{e^{\cot(\bar \alpha) (\bar \alpha + \hat \theta - \frac{\pi}{2})}}{K_\mathrm{tent}} - 1} \bigg) = 7.86127,
\end{equation*}
which in the $\tau$ coordinates corresponds to
\begin{equation}
\label{Equa:hattau3_cond}
\hat \tau_3 - \hat \tau_0 \geq 1.05358.
\end{equation}
We thus have the following

\begin{proposition}
\label{Prop:tent_admissible}
The tent is admissible for if the starting angle $\hat \phi_0$ is larger that $\phi_0 + \bar \theta - \bar \alpha$ (segment case) of $\phi_0 + \frac{\pi}{2} + \Delta \phi - \bar \alpha$.
\end{proposition}

\begin{proof}
The initial angle corresponds to $\tau = 0$ in the $\tau$-coordinates, and Corollary \ref{Cor:bound_length_gen} must be true at $\tau = 1$: 

Hence if the tent is not admissible for $\tau > 0$, there is a $\tau > 1$ such that Corollary \ref{Cor:bound_length_gen} is not true by \eqref{Equa:hattau3_cond}, which is a contradiction. 
%
\end{proof}

\begin{remark}
\label{Rem:inital_tent}
%
It is easy to see that the tent is not always admissible: if we assume that the spiral $\zeta$ has a very long segment, $\zeta$ asymptotically becomes by scaling
\begin{equation*}
\Omega_0 = B_{\cos(\bar \alpha)}(0), \quad \zeta = [0,1],
\end{equation*}
for which the fastest saturated is an arc. Indeed, in this situation the saturation condition for the tent reads as
\begin{align*}
0 &= \hat \ell_1 \cos(\bar \alpha) + \hat \ell_0 \cos(\bar \alpha - \hat \theta) + \cos(\hat \theta) - \cos(\bar \alpha) - \cos(\bar \alpha) (1 + \hat \ell_0 + \hat \ell_1) \\
&= \hat \ell_0 (\cos(\bar \alpha - \hat \theta) - \cos(\bar \alpha)) - 2 \cos(\bar \alpha) + \cos(\hat \theta).
\end{align*}
Being
\begin{equation*}
2 \cos(\bar \alpha) - \cos(\hat \theta) = - 0.10964 < 0,
\end{equation*}
the tent is not admissible.
\end{remark}


\subsubsection{An estimate on the final value $\hat r(\bar \phi)$}
\label{Sss:estima_final_value}

To conclude this section, we give a simplified formula for estimating the final value of $\hat r$ and the angle $\bar \phi = \hat \phi_0 + 2\pi + \bar \alpha$.

\begin{lemma}
\label{Lem:final_value_satu}
Assume that base of the tent is a convex curve, $\hat P_0$ is saturated and that $\tilde r$ is the saturated spiral starting from $\hat \phi_0$ tangent to the tent in the point $\hat P_0$. Then it holds
\begin{equation*}
\hat r(\bar \phi) \geq \tilde r(\bar \phi) - \hat \ell_0 \geq \tilde r(\bar \phi) - \frac{1 - \cos(\hat \theta)}{\cos(\bar \alpha - \hat \theta) - \cos(\bar \alpha)} \hat r_0.
\end{equation*}
\end{lemma}

\begin{proof}
Since the base of the spiral is a convex curve, then
\begin{equation*}
|\hat Q_2 - \hat Q_0| \leq \hat L,
\end{equation*}
and then by Equation \eqref{Equa:ell_0_tent_general}
\begin{equation*}
\hat \ell_0 \leq \frac{1 - \cos(\hat \theta)}{\cos(\bar \alpha - \hat \theta) - \cos(\bar \alpha)} \hat r_0.
\end{equation*}
By Theorem \ref{Cor:curve_cal_R_sat_spiral} it holds also
\begin{equation*}
\hat r_2 \geq \tilde r(\hat \phi_2) = \tilde r(\hat \phi_0 + \hat \theta).
\end{equation*}
Since $\hat r$ and $r$ satisfy after the angle $\hat\phi_2$ the same RDE, their difference satisfies 
\begin{equation*}
	\frac{d}{d\theta}(\hat r(\phi)-r(\phi))=\cot\bar\alpha(\hat r(\phi)-r(\phi)).
	\end{equation*}
Therefore,
\begin{align*}
\hat r(\bar \phi) &= (\hat r_2 - \tilde r(\hat \phi_2)) e^{\cot(\bar \alpha)(2\pi + \bar \alpha - \hat \theta)} + \tilde r(\bar \phi) - \hat \ell_0 \\
&\geq \tilde r(\bar \phi) - \frac{1 - \cos(\hat \theta)}{\cos(\bar \alpha - \hat \theta) - \cos(\bar \alpha)} \hat r_0. \qedhere
\end{align*}
\end{proof}

The last computation of the above proof will be used several times in the next sections:
\begin{align}
\label{eq:comparaison:Saturated}
\hat r(\bar \phi) &= (\hat r_2 - \tilde r(\hat \phi_2)) e^{\cot(\bar \alpha)(2\pi + \bar \alpha - \hat \theta)} + \tilde r(\bar \phi) - \hat \ell_0.
\end{align}

\section{Optimal solution candidate}
\label{S:optimal_sol_candidate}

In this section we construct the optimal candidate \newglossaryentry{roptim}{name=\ensuremath{r_\mathrm{opt}(\phi)},description={the optimal solution for blocingk the fire at a given angle}} \gls{roptim} when the initial burning set is $\{0\}$ and the spiral is constructed outside $B_1(0)$. We will show the structure of the optimal solution candidate depending on the final angle $\bar \phi$ and its asymptotic behavior.

\subsection{Optimal solution candidate for \texorpdfstring{$\bar \phi \in [0,2\pi)$}{barphi in [0,2pi]}: see Fig. \ref{fig:base:case:Arc}}
\label{Ss:optimal_cand_1}

In this case the optimal solution is the fastest saturated spiral $r_\mathrm{opt} = \tilde r_{\mathtt a = 0}$, according to Theorem \ref{Cor:curve_cal_R_sat_spiral}, and by Theorem \ref{thm:case:a} the solution is positive.

\subsection{Optimal solution candidate for \texorpdfstring{$\bar \phi \in [2\pi,2\pi + \frac{\pi}{2})$}{bar phi in [2pi,2pi + pi/2]}: see Fig. \ref{Fig:opt_sol_1_1}}
\label{Ss:optimal_cand_2}

\begin{figure}
\begin{subfigure}{.475\textwidth}
\resizebox{\textwidth}{!}{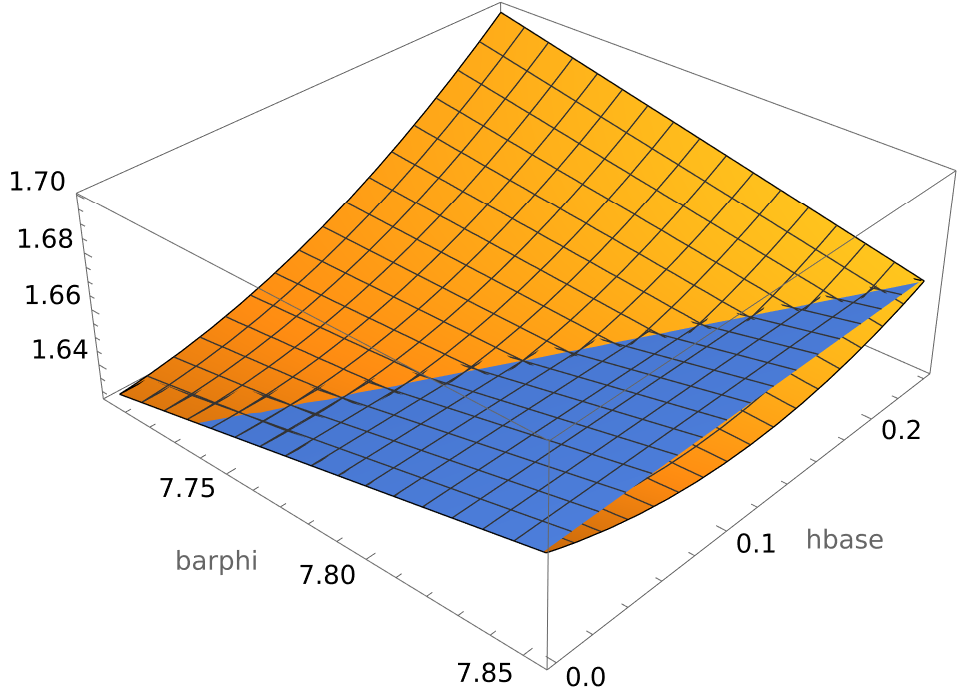}
\caption{Geometric situation of Section \ref{Ss:optimal_cand_2}.}
\label{Fig:opt_sol_1_1}
\end{subfigure} \hfill
\begin{subfigure}{.475\textwidth}
\resizebox{\textwidth}{!}{\includegraphics{opt_sol_1.png}}
\caption{Numerical plot of the function $r_\mathrm{opt}$ (orange) and $S_\mathtt a$ with $\mathtt a = 0$ (blue).}
\label{Fig:opt_sol_1_2}
\end{subfigure}
\caption{Analysis of Section \ref{Ss:optimal_cand_2}.}
\label{Fig:opt_sol_1}
\end{figure}

In this case we move along the direction $\bar \phi$ until we are in the critical curve $\check \theta = h_1(\Delta \check \phi)$: the saturation condition at the point $\check P_2$ is
\begin{equation*}
\frac{\check \ell}{\sin(\bar \alpha)} - \cos(\bar \alpha) \big( h + \check \ell \Delta \check \phi + \check \ell \cot(\bar \alpha) \big) = 0, \quad \check \ell = \sqrt{1 + h^2 + 2 h \cos(\bar \phi)}, \quad \sin(\check \phi) = \frac{h}{\check \ell} \sin(\bar \phi),
\end{equation*}
which gives
\begin{equation}
\label{Equa:opt_2_checkDeltaphi}
\Delta \check \phi = \tan(\bar \alpha) - \frac{h}{\check \ell}.
\end{equation}
The final value is then
\begin{equation}
\label{Equa:ropt_2}
r_\mathrm{opt}(\bar \phi,h) = \frac{\check \ell}{\sin(\bar \alpha)} e^{\cot(\bar \alpha) (\bar \phi - (\check \phi + \Delta \check \phi + \frac{\pi}{2} - \bar \alpha))} - e^{\cot(\bar \alpha)(\bar \phi - 2\pi)} - h.
\end{equation}
A plot is in Fig. \ref{Fig:opt_sol_1_2}, where the function $e^{-\bar c \phi} r_\mathrm{opt}$ and the fastest saturated spiral $e^{-\bar c \phi} \tilde r_{\mathtt a = 0}$ of Section \ref{S:case:study} are plotted together: the figure presents only the part where $r_\mathrm{opt}$ is below $\tilde r_{\mathtt a = 0}$. Note that the region where it is convenient to use $r_\mathrm{opt}$ is for
$$
\bar \phi > 2\pi + \frac{\pi}{2} - h_1(\tan(\bar \alpha)) = 7.73645,
$$
as Proposition \ref{Prop:neg_region} requires. The optimal value is computed by imposing  $h = h_1(\Delta \check \phi)$ together with \eqref{Equa:opt_2_checkDeltaphi}.

In order to verify that the tent is not an admissible solution, one could check that $\check \Delta \phi > 0$ or try to compute the geometric quantities for a tent: indeed, if we try to replace the arc with the tent and we use the saturation condition and the optimality condition $\theta = \hat \theta$ of Proposition \ref{Prop:unique_optimal_guess}, we obtain that if $P_2$ is the first saturated point of the fastest closing spiral starting in $P_1 = (1 + h \cos(\bar \phi),h \sin(\bar \phi))$, then from geometric consideration and the saturation condition
\begin{equation*}
|P_2| = \cos(\bar \alpha - \hat \theta - \bar \phi + 2\pi) + h \cos(\bar \alpha - \hat \theta) + \ell_1 \cos(\bar \alpha), \quad |P_2| - \cos(\bar \alpha) (h + \ell_1) = 0,
\end{equation*}
where $\ell_1=|P_2-P_1|$, so that
\begin{equation*}
h = - \frac{\cos(\bar \alpha - \hat \theta - \bar \phi + 2\pi)}{\cos(\bar \alpha - \hat \theta) - \cos(\bar \alpha)} < 0
\end{equation*}
for $\phi \in [2\pi,2\pi + \frac{\pi}{2}]$, so it cannot be an admissible solution.

We summarize these results into the following lemma.

\begin{lemma}
\label{Lem:optimal_2}
If $\bar \phi \in 2\pi + [0,\frac{\pi}{2} - h_1(\tan(\bar \alpha))]$, the optimal solution is the saturated spiral. For $\bar \phi \in 2\pi + [\frac{\pi}{2} - h_1(\tan(\bar \alpha)),\frac{\pi}{2}]$ the optimal solution is a segment in the direction $\bar \phi$ followed by an arc. The solution remains positive.
\end{lemma}

\subsection{Optimal solution candidate for \texorpdfstring{$\bar \phi \in [2\pi + \frac{\pi}{2},2\pi+\frac{\pi}{2}+\tan(\bar \alpha))$}{bar phi in [2pi + pi/2,2pi + pi/2+ tan(bar alpha)}: see Fig. \ref{Fig:opt_sol_2_1}, \ref{Fig:opt_sol_2_2}}
\label{Ss:optimal_cand_3}

\begin{figure}
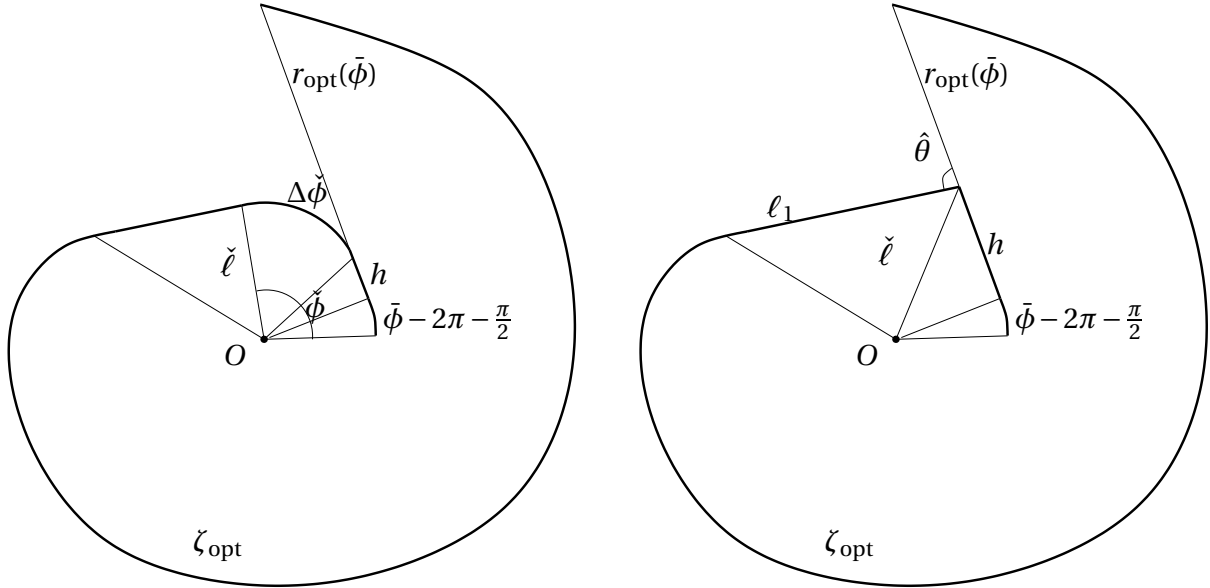

\begin{subfigure}{.475\textwidth}
\resizebox{\textwidth}{!}{\input{opt_sol_2.pdf_t}}
\caption{Geometric situation of Section \ref{Ss:optimal_cand_3} in the arc case.}
\label{Fig:opt_sol_2_1}
\end{subfigure} \hfill
\begin{subfigure}{.475\textwidth}
\resizebox{\textwidth}{!}{\input{opt_sol_2_tent.pdf_t}}
\caption{Geometric situation of Section \ref{Ss:optimal_cand_3} in the tent case.}
\label{Fig:opt_sol_2_2}
\end{subfigure}
\caption{Analysis of Section \ref{Ss:optimal_cand_3}.}
\label{Fig:opt_sol_2}
\end{figure}

In this case the solution $r_\mathrm{opt}$ switches from the arc to the tent as the angle $\bar \phi$ increases: indeed from Proposition \ref{Prop:tent_admissible} we know that in the saturated region the tent is the admissible solution. For the arc case (Figure \ref{Fig:opt_sol_2_1}), the saturation condition at the point $\check P_2$ is
\begin{equation*}
\frac{\check \ell}{\sin(\bar \alpha)} - \cos(\bar \alpha) \bigg( \bar \phi - 2\pi - \frac{\pi}{2} + h + \check \ell \Delta \check \phi + \check \ell \cot(\bar \alpha) \bigg) = 0, \quad \check \ell = \sqrt{1 + h^2}, \quad \sin(\check \phi) = \frac{h}{\check \ell},
\end{equation*}
i.e.
\begin{equation}
\label{Equa:opt_sol_2_arc_1}
\Delta \check \phi = \tan(\bar \alpha) - \frac{\bar \phi - 2\pi - \frac{\pi}{2} + h}{\check \ell}.
\end{equation}
which is positive for (see Fig. \ref{Fig:opt_sol_2_arc_1})
\begin{equation*}
\bar \phi \leq 2\pi + \frac{\pi}{2} + \tan(\bar \alpha) \check \ell - h.
\end{equation*}

\begin{figure}
\begin{subfigure}{.32\textwidth}
\resizebox{\textwidth}{!}{\includegraphics{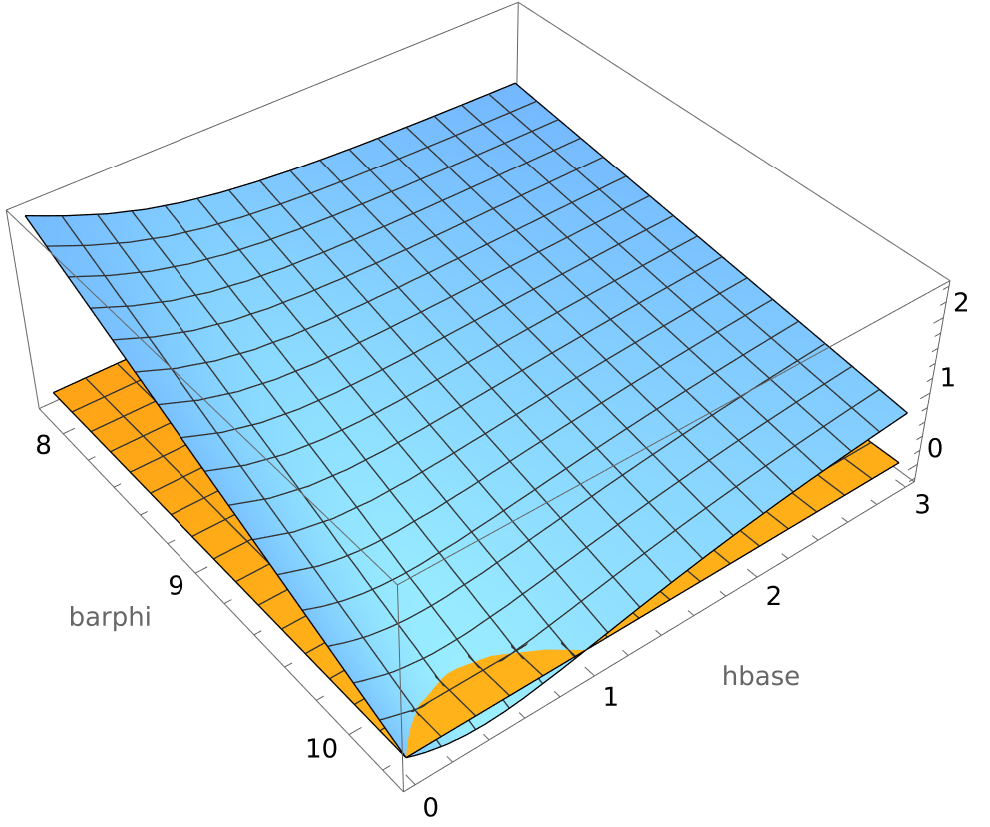}}
\caption{Plot of the function \eqref{Equa:opt_sol_2_arc_1}.}
\label{Fig:opt_sol_2_arc_1}
\end{subfigure} \hfill
\begin{subfigure}{.32\textwidth}
\resizebox{\textwidth}{!}{\includegraphics{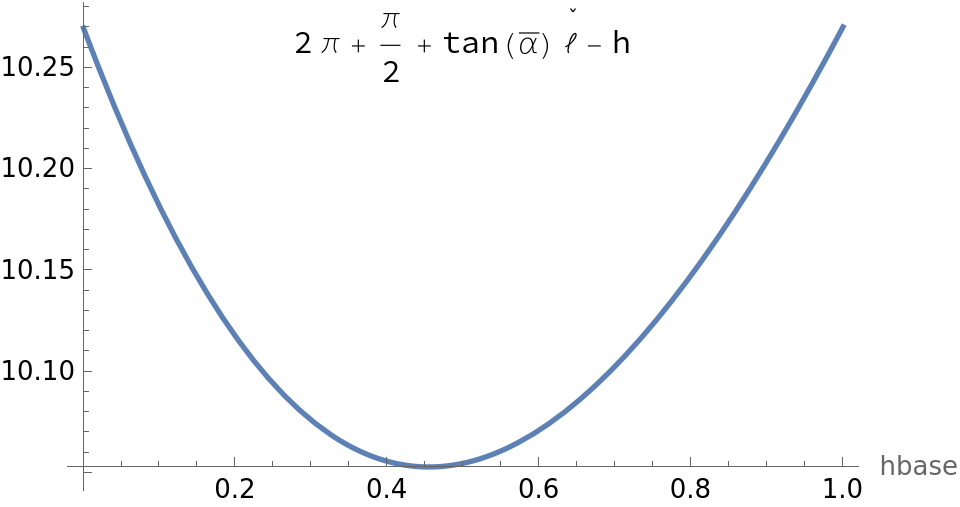}}
\caption{Plot of the boundary of the region \eqref{Equa:opt_sol_2_arc_2}.}
\label{Fig:opt_sol_2_arc_2}
\end{subfigure} \hfill
\begin{subfigure}{.32\textwidth}
\resizebox{\textwidth}{!}{\includegraphics{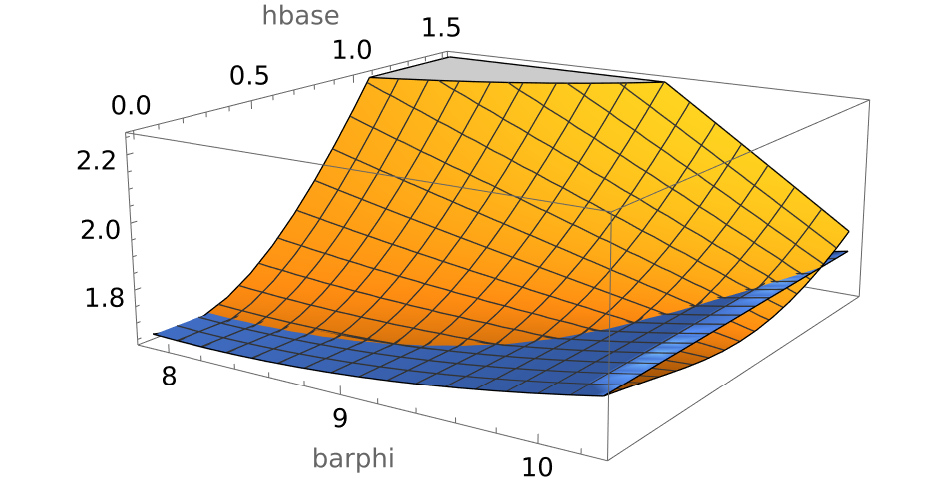}}
\caption{Plot of the function \eqref{Equa:opt_sol_3_finarc} (orange) and the base solution (blue), both rescaled by $e^{-\bar c(\bar \phi - (\tan(\bar \alpha) + \frac{\pi}{2} - \bar \alpha))}$.}
\label{Fig:opt_sol_2_arc_3}
\end{subfigure}
\begin{subfigure}{.32\textwidth}
\resizebox{\textwidth}{!}{\includegraphics{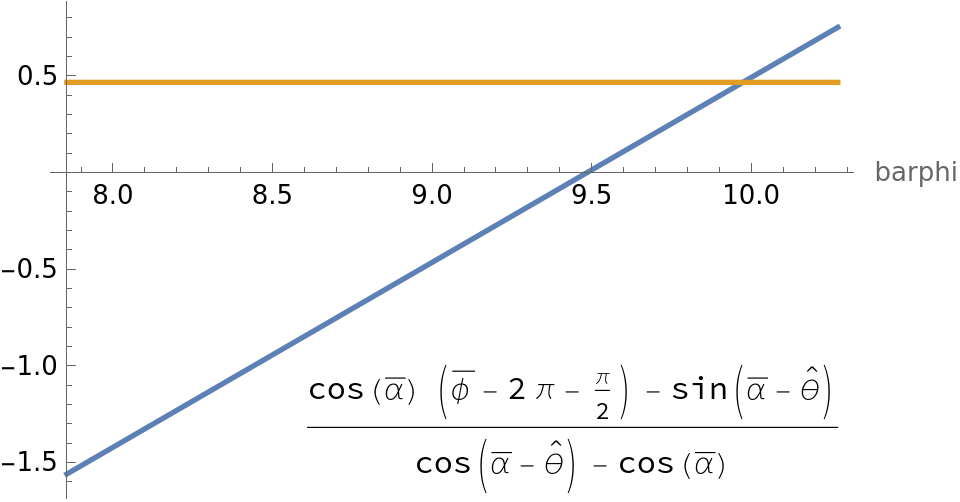}}
\caption{Plot of the function \eqref{Equa:opt_sol_3_h_tent} (blue): the orange line is $\tan(\hat \theta)$.}
\label{Fig:opt_sol_2_tent_1}
\end{subfigure} \hfill
\begin{subfigure}{.32\textwidth}
\resizebox{\textwidth}{!}{\includegraphics{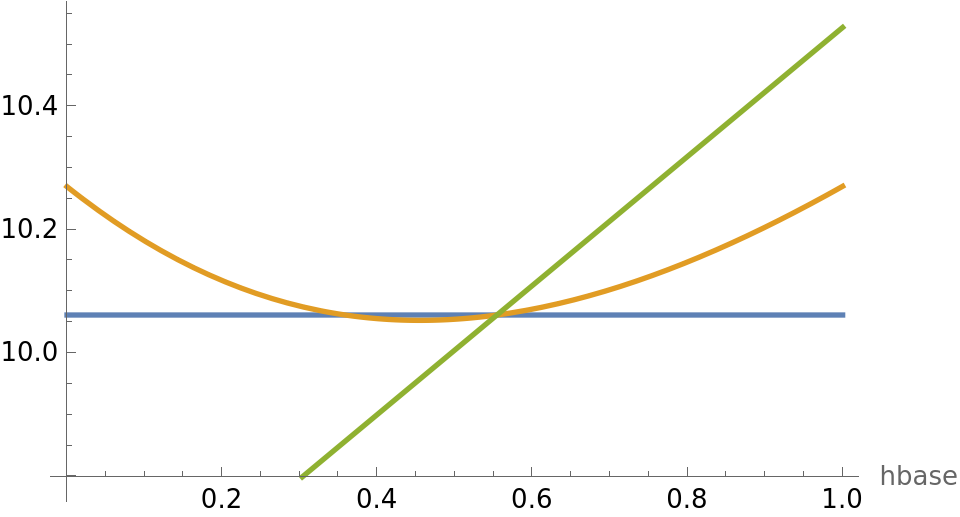}}
\caption{Plot of Equation \eqref{Equa:opt_sol_3_h_tent} (green), \eqref{Equa:opt_sol_2_arc_2} (orange) and the line for admissibility (Equation \eqref{Equa:opt_sol_3_admi_tent}, blue).}
\label{Fig:opt_sol_2_tent_2}
\end{subfigure} \hfill
\begin{subfigure}{.32\textwidth}
\resizebox{\textwidth}{!}{\includegraphics{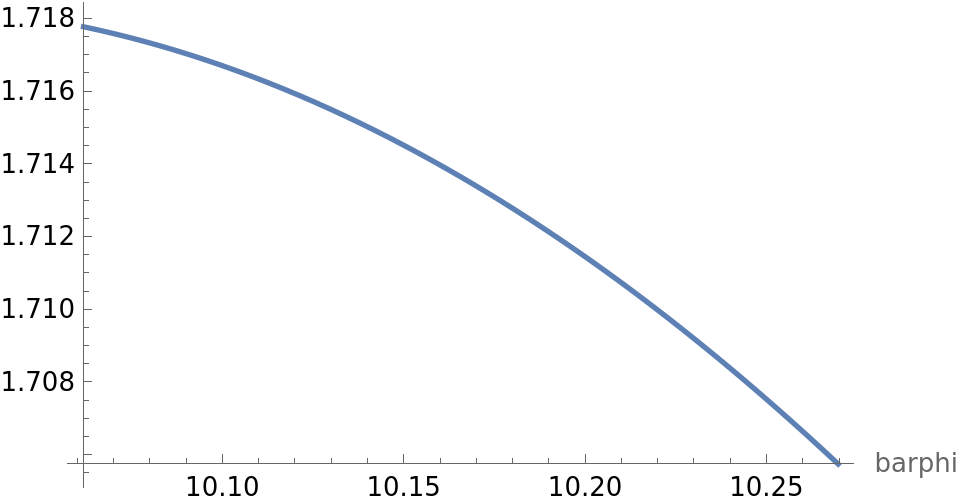}}
\caption{Plot of the function \eqref{Equa:opt_sol_3_tent_final} multiplied by $e^{-\bar c(\bar \phi - (\tan(\bar \alpha) + \frac{\pi}{2} - \bar \alpha))}$.}
\label{Fig:opt_sol_2_tent_3}
\end{subfigure}
\caption{Numerical analysis of Section \ref{Ss:optimal_cand_3}.}
\label{Fig:opt_sol_2_arc_1_tot}
\end{figure}

In the region (see Fig. \ref{Fig:opt_sol_2_arc_2})
\begin{equation}
\label{Equa:opt_sol_2_arc_2}
\bar \phi > 2\pi + \frac{\pi}{2} + \tan(\bar \alpha) \check \ell - h
\end{equation}
the arc is not admissible having $\Delta \check \phi < 0$, so that the solution will be the tent. The final value in the arc case is
\begin{equation}
\label{Equa:opt_sol_3_finarc}
\begin{split}
r_\mathrm{opt}(\bar \phi) &= \frac{\check \ell}{\sin(\bar \alpha)} e^{\cot(\bar \alpha) (\bar \phi - (\bar \phi - 2\pi - \frac{\pi}{2} + \Delta \check \phi + \check \phi + \frac{\pi}{2} - \bar \alpha))} - e^{\cot(\bar \alpha)(\bar \phi - 2\pi)} - \int_0^{\bar \phi - 2\pi - \frac{\pi}{2}} e^{\cot{\bar \alpha} \phi'} d\phi' - h \\
&= \frac{\check \ell}{\sin(\bar \alpha)} e^{\cot(\bar \alpha) (2\pi + \bar \alpha - (\check \phi + \Delta \check \phi))} - e^{\cot(\bar \alpha)(\bar \phi - 2\pi)} - \frac{e^{\cot{\bar \alpha} (\bar \phi - 2\pi - \frac{\pi}{2})}-1}{\cot(\bar \alpha)} - h.
\end{split}
\end{equation}
which is positive $\geq 1.6 e^{\bar c(\bar \phi - (\frac{\pi}{2} - \bar \alpha + \tan(\bar \alpha)))}$ also in the region where $\Delta \check \phi < 0$, see Fig. \ref{Fig:opt_sol_2_arc_3}.

We next compute the tent case with final angle $\hat \theta$, which we know to be the optimal angle at the end point of the segment $h$: the saturation gives here
\begin{align*}
0 &= \big( \ell_1 \cos(\bar \alpha) + h \cos(\bar \alpha - \hat \theta) + \sin(\bar \alpha - \hat \theta) \big) - \cos(\bar \alpha) \bigg( \bar \phi - 2\pi - \frac{\pi}{2} + h + \ell_1 \bigg) \\
&= h (\cos(\bar\alpha-\hat\theta) - \cos(\bar \alpha)) + \sin(\bar \alpha - \hat \theta) - \cos(\bar \alpha) \bigg( \bar \phi - 2\pi - \frac{\pi}{2} \bigg),
\end{align*}
\begin{equation}
\label{Equa:opt_sol_3_h_tent}
h = \frac{\cos(\bar \alpha) (\bar \phi - 2\pi - \frac{\pi}{2}) - \sin(\bar \alpha - \hat \theta)}{\cos(\bar \alpha - \hat \theta) - \cos(\bar \alpha)}.
\end{equation}
This quantity is positive only for (see Fig. \ref{Fig:opt_sol_2_tent_1})
\begin{align*}
\bar \phi &\geq 2\pi + \frac{\pi}{2} + \frac{\sin(\bar \alpha - \hat \theta)}{\cos(\bar \alpha)} = 2\pi + \frac{\pi}{2} + \tan(\bar \alpha) - \frac{\sin(\bar \alpha) - \sin(\bar \alpha - \hat \theta)}{\cos(\bar \alpha)} = 2\pi + \frac{\pi}{2} + \tan(\bar \alpha) - .787658 = 9.48195.
\end{align*}
The angle between the ray departing from the origin and arriving at the end of the segment $h$ and the segment with angle $\hat \theta$ is
\begin{equation*}
\pi - \hat \theta - \frac{\pi}{2} + \arctan(h) = \frac{\pi}{2} + \arctan(h) - \hat \theta,
\end{equation*}
so that the requirement to be $\geq \frac{\pi}{2}$ for admissibility gives
\begin{equation}
\label{Equa:opt_sol_3_admi_tent}
h \geq \tan(\hat \theta) = .554249, \quad \bar \phi \geq 2\pi + \frac{\pi}{2} + \frac{\sin(\bar \alpha) - \sin(\hat \theta) \cos(\bar\alpha)}{\cos(\bar \alpha) \cos(\hat \theta)} = 10.0161.
\end{equation}
The point $(h,\bar \phi) = (.554249,10.0161)$ belongs to the curve where $\Delta \check \phi = 0$, and from that point onward the value $h$ given by Equation \eqref{Equa:opt_sol_3_finarc} is inside the region where $\Delta \check \phi < 0$, see Fig. \ref{Fig:opt_sol_2_tent_2}.

The final value of the solution for the tent case in the admissibility region is
\begin{equation}
\label{Equa:opt_sol_3_tent_final}
r_\mathrm{opt}(\bar \phi) = \bigg( \frac{\cos(\arctan(h) - \hat \theta)}{\sin(\bar \alpha)} \sqrt{1+h^2} \bigg) e^{\cot(\bar \alpha)(2 \pi + \bar \alpha - \hat \theta)} - e^{\cot(\bar \alpha)(\bar \phi - 2\pi)} - \frac{e^{\cot(\bar \alpha)(\bar \phi - 2\pi - \frac{\pi}{2})} - 1}{\cot(\bar \alpha)} - h,
\end{equation}
with $h$ given by \eqref{Equa:opt_sol_3_h_tent}. A numerical plot is in Fig. \ref{Fig:opt_sol_2_tent_3}.

We summarize the results in the following lemma.

\begin{lemma}
\label{Lem:optimal_3}
If $\bar \phi \in 2\pi + \frac{\pi}{2} + [0,\tan(\bar \alpha) - \frac{\sin(\bar \alpha) - \sin(\bar \alpha - \hat \theta)}{\cos(\bar \alpha)}]$, the optimal solution is a segment in the direction $\bar \phi$ followed by an arc. If $\bar \phi \in 2\pi + \frac{\pi}{2} + [\tan(\bar \alpha) - \frac{\sin(\bar \alpha) - \sin(\bar \alpha - \hat \theta)}{\cos(\bar \alpha)},\tan(\bar \alpha)]$, the optimal solution is the tent. Both solutions remain positive.
\end{lemma}

\subsection{Optimal solution candidate for \texorpdfstring{$\bar \phi \geq 2\pi + \frac{\pi}{2} + \tan(\bar \alpha)$}{bar phi geq 2pi + pi/2 + tan(bar alpha)}: see Fig. \ref{Fig:opt_sol_3_tent}}
\label{Ss:optimal_cand_4}

\begin{figure}
\begin{subfigure}{.475\textwidth}
\resizebox{\textwidth}{!}{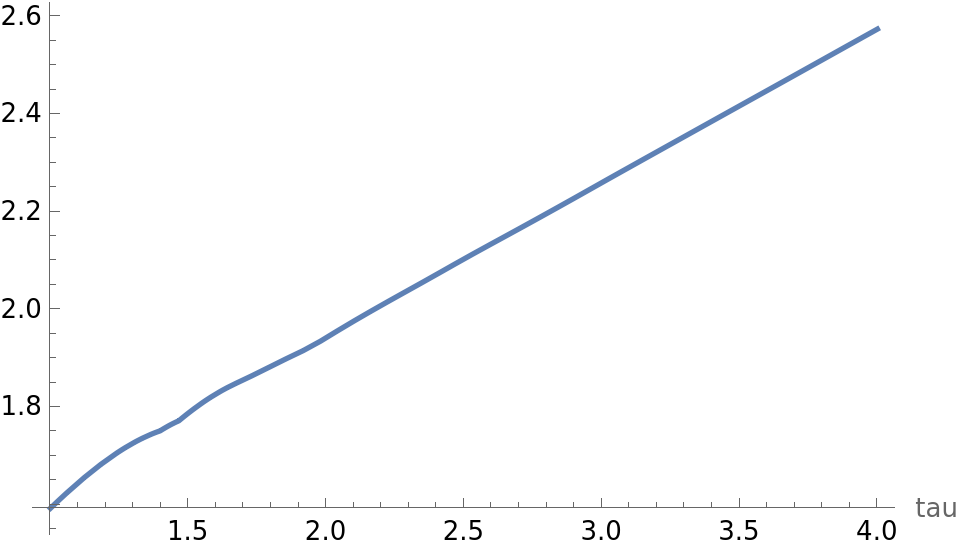}
\caption{The optimal candidate $r_\mathrm{opt}(\phi)$ for $\bar \phi \geq 2\pi + \tan(\bar \alpha) + \frac{\pi}{2}$.}
\label{Fig:opt_sol_3_tent}
\end{subfigure} \hfill
\begin{subfigure}{.475\textwidth}
\resizebox{\textwidth}{!}{\includegraphics{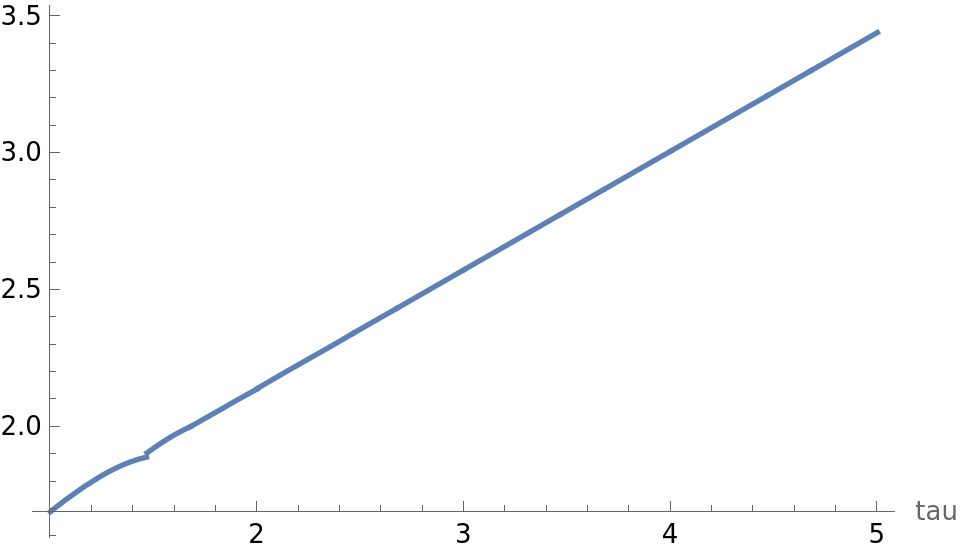}}
\caption{Numerical plot of the functional \eqref{Equa:opt_sol_3_pos_1}.}
\label{Fig:opt_sol_3_pos}
\end{subfigure}
\caption{Analysis of Section \ref{Ss:optimal_cand_4}.}
\label{Fig:opt_sol_3}
\end{figure}

We first observe that by Proposition \ref{Prop:tent_admissible} the tent is admissible.

Next we have to apply the formulas of Lemma \ref{Lem:relation_admissibility_tent} in the saturated part, in order to compute the final value $r_{\mathrm{opt}}(\bar \phi)$. Using Lemma \ref{Lem:final_value_satu}, we need to estimate the function
\begin{equation}
\label{Equa:opt_sol_3_pos_1}
\rho(\bar \tau) - \frac{1 - \cos(\hat \theta)}{\cos(\bar \alpha - \hat \theta) - \cos(\bar \alpha)} e^{-\bar c (2\pi + \bar \alpha) } \rho(\bar \tau-1), \quad \bar \tau \geq 1.
\end{equation}

A numerical plot is in Fig. \ref{Fig:opt_sol_3_pos}: we observe that it is positive and increasing for $\tau \in [4,5]$: then by Lemma \ref{lem:key} we conclude that

\begin{proposition}
\label{Prop:e_opt_large_phi}
For $\bar \tau \geq 1$, the optimal solution $r_\mathrm{opt}$ is made by a final tent with initial point in $\tau - 1$ and moreover $r_\mathrm{opt}(\bar \phi) > 0$.
\end{proposition}

\subsection{Asymptotic behavior of the optimal candidate}
\label{Ss:optimal_asympt}

We can compute the asymptotic value of the optimal solution, using the asymptotic value of the fastest saturated spiral $\tilde r_{\mathtt a = 0}(\phi)$ given by Theorem \ref{thm:case:a} for $\bar \phi \gg 1$: \newglossaryentry{Kasymptosat}{name=\ensuremath{K_\mathrm{asympt}},description={constant for the linear part of the asymptotic behavior of the fastest saturated spiral}}
\begin{align*}
\tilde r_{\mathtt a = 0}(\phi) &\sim 2 \phi e^{\bar c \phi} \bigg( \frac{e^{-\bar c(\tan(\bar \alpha) + \frac{\pi}{2} - \bar \alpha)}}{\sin(\bar \alpha)} - e^{-\bar c(2\pi)} - \frac{e^{-\bar c(2\pi + \frac{\pi}{2})} - e^{- \bar c (2\pi + \frac{\pi}{2} + \tan(\bar \alpha))}}{\bar c} - \cot(\bar \alpha) e^{-\bar c(2\pi + \frac{\pi}{2} + \tan(\bar \alpha))} \bigg) \\
&= \gls{Kasymptosat} \phi e^{\bar c \phi}.
\end{align*}
with $\gls{Kasymptosat} = 0.149681$.

One thus obtain the estimates for the tent computations of Lemma \ref{Lem:relation_admissibility_tent} 
\begin{align*}
\hat Q_2 - \hat Q_0 &= e^{\i (\bar \phi - 2\pi - \bar \alpha)} \int_0^{\hat \theta} e^{\i \omega} \frac{\tilde r_{\mathtt a = 0}(\omega + (\bar \phi - 2(2\pi + \bar \alpha)))}{\sin(\bar \alpha)} d\omega \\
&\sim K_\mathrm{asympt} e^{\i (\bar \phi - 2\pi - \bar \alpha)} \frac{\bar \phi e^{\bar c (\bar \phi - 2(2\pi + \bar \alpha))}}{\sin(\bar \alpha)} \int_0^{\hat \theta} e^{\i \omega} e^{\bar c \omega} d\omega \\
&= K_\mathrm{asympt} e^{\i (\bar \phi - 2\pi - \bar \alpha)} \frac{\bar \phi e^{\bar c (\bar \phi - 2(2\pi + \bar \alpha))}}{\sin(\bar \alpha)} \frac{e^{(\bar c + \i) \hat \theta} - 1}{\bar c + \i}, 
\end{align*}
\begin{align*}
(\hat Q_2 - \hat Q_0) \cdot e^{\i (\bar \phi - (2\pi + \bar \alpha) + \hat \theta)} &\sim K_\mathrm{asympt} \frac{\bar \phi e^{\bar c (\bar \phi - 2(2\pi + \bar \alpha))}}{\sin(\bar \alpha)} \frac{\bar c (e^{\bar c \hat \theta} - \cos(\hat \theta)) + \sin(\hat \theta)}{1 + \bar c^2} \\
&= 0.564562 K_\mathrm{asympt} \bar \phi e^{\bar c (\bar \phi - 2(2\pi + \bar \alpha))},
\end{align*}
\begin{align*}
(\hat Q_2 - \hat Q_0) \cdot e^{\i (\bar \phi - (2\pi + \bar \alpha) + \bar \alpha + \hat \theta - \frac{\pi}{2})} &\sim K_\mathrm{asympt} \frac{\bar \phi e^{\bar c (\bar \phi - 2(2\pi + \bar \alpha))}}{\sin(\bar \alpha)} \frac{\bar c (e^{\bar c \hat \theta} \sin(\bar \alpha) - \sin(\bar \alpha + \hat \theta)) + (e^{\bar c \hat \theta} \cos(\bar \alpha) - \cos(\bar \alpha + \hat \theta))}{1 + \bar c^2} \\
&= 0.576096 K_\mathrm{asympt} \bar \phi e^{\bar c (\bar \phi - 2(2\pi + \bar \alpha))},
\end{align*}
\begin{align*}
\hat L &= \int_0^{\hat \theta} \frac{\tilde r_{\mathtt a = 0}(\omega + (\bar \phi - 2(2\pi + \bar \alpha)))}{\sin(\bar \alpha)} d\omega \sim K_\mathrm{asympt} \frac{\bar \phi e^{\bar c (\bar \phi - 2(2\pi + \bar \alpha))}}{\sin(\bar \alpha)} \frac{e^{\bar c \hat \theta} - 1}{\bar c} \\
&= 0.588497 K_\mathrm{asympt} \bar \phi e^{\bar c (\bar \phi - 2(2\pi + \bar \alpha))},
\end{align*}
\begin{align*}
\hat \ell_0 &= \frac{(\hat Q_2 - \hat Q_0) \cdot e^{\i (\hat \theta + (\bar \phi - 2(2\pi + \bar \alpha)))} - \hat L + (1 - \cos(\hat \theta)) r(\bar \phi - 2(2\pi + \bar \alpha))}{\cos(\bar \alpha - \hat \theta) - \cos(\bar \alpha)} \\
&\sim K_\mathrm{asympt} \bar \phi e^{\bar c (\bar \phi - (2\pi + \bar \alpha))} \frac{1 - \cos(\hat \theta) + \frac{e^{-\bar c(2\pi + \bar \alpha)}}{\sin(\bar \alpha)} (\frac{e^{(\bar c + \i) \hat \theta} - 1}{\bar c + \i} \cdot e^{\i \theta} - \frac{e^{\bar c \hat \theta} - 1}{\bar c})}{\cos(\bar \alpha - \hat \theta) - \cos(\bar \alpha)} \\
&= 0.306042 K_\mathrm{asympt} \bar \phi e^{\bar c (\bar \phi - (2\pi + \bar \alpha))},
\end{align*}
\begin{align*}
\hat r_2 &= \frac{\hat r_0 \sin(\bar \alpha + \hat \theta) + \hat \ell_0 \sin(\hat \theta) - (\hat Q_2 - \hat Q_0) \cdot e^{\i (\bar \alpha + \hat \theta - \frac{\pi}{2})}}{\sin(\bar \alpha)} \\
&\sim 1.15869 K_\mathrm{asympt} \bar \phi e^{\bar c (\bar \phi - (2\pi + \bar \alpha) + \hat \theta)},
\end{align*}
so that the final value is computed via the comparison with the saturated spiral, see Equation \eqref{eq:comparaison:Saturated}, as
\begin{align*}
r_\mathrm{opt}(\bar \phi) &= (\hat r_2 - \tilde r_{\mathtt a = 0}(\bar \phi - (2\pi + \bar \alpha) + \hat \theta)) e^{\cot(\bar \alpha)(2\pi + \bar \alpha - \hat \theta)} + \tilde r_{\mathtt a = 0}(\bar \phi) - \hat \ell_0 \\
&\sim \big[ (1.15869 - e^{\bar c \hat \theta}) e^{\cot(\bar \alpha)(2 \pi + \bar \alpha - \hat \theta) - \bar c(2\pi + \bar \alpha)} + 1 - 0.306042 e^{- \bar c (2\pi + \bar \alpha)} \big] K_\mathrm{asympt} \bar \phi e^{\bar c \bar \phi} \\
&= 0.976359 K_\mathrm{asympt} \bar \phi e^{\bar c \bar \phi}.
\end{align*}

We thus have proved the following result.

\begin{proposition}
\label{Prop:postiv_aympt}
The optimal candidate solution at $\bar \phi$ can be estimated as 
\begin{equation*}
r_\mathrm{opt}(\bar \phi) = 0.976359 \tilde r_{\mathtt a = 0}(\bar \phi) + \mathcal O(1)
\end{equation*}
for $\bar \phi \to \infty$.
\end{proposition}

In particular, we gain 2.4\% using the tent w.r.t. the saturated spiral. However observe that asymptotically $\hat\ell_0 \sim 0.3 \tilde r_{\mathtt a = 0}$, so that the spirals $\tilde r_{\mathtt a = 0}$ and $r_\mathrm{opt}$ are quite different.

%
%
%
%
%
%
%
%
%
%
%
%
%
%

\section{Segment-segment case}
\label{S:segment_case}

In the following subsections, we analyze the case where the fastest saturated spiral starts with a segment, and the perturbation of the optimal solution is a tent: we call it \emph{segment-segment case} (or \emph{segment-tent case}). This is the unique possibility for $\bar \tau \geq 1$ by Proposition \ref{Prop:tent_admissible}.

Depending on the relative position of $\phi_0$ and $\bar \phi$, $8$ cases have to be considered: we will study them in the next subsections, and we will show that the derivative of the optimal candidate is always positive for all admissible perturbations. The fact that we need to consider several cases is due to the presence of the negativity region: indeed in this case we cannot rely on the estimates of Lemma \ref{Lem:final_value_satu}, and the only way is to use the explicit formulas of Lemma \ref{Lem:relation_admissibility_tent}. For more clarity, see also Figure \ref{fig:tent:solution}, where there are the different cases based on the choice of $\bar\phi$.
We recall that the angle $\bar \theta$ is the angle formed by the initial segment $[P_1,P_2]$ of the fastest saturated spiral, with $P_1 = \zeta(s_0)$.

\subsection{Segment-segment with \texorpdfstring{$\bar \phi \in \phi_0 + [\beta^-(\phi_0),2\pi + \bar \theta - \hat \theta]$}{bar phi in beta-,2pi + bar theta - hat theta}, see Fig. \ref{Fig:segm_arc_1}}
\label{Sss:segm-segm_low}

In this case $\hat r(\phi;s_0) = \tilde r(\phi;s_0)$, because the derivative $\delta \tilde r(\phi;s_0) \geq 0$: indeed $\bar \phi$ is outside the region of negativity $\gls{Nsegm}$. This fact gives by Proposition \ref{Prop:neg_region} the following

\begin{lemma}
\label{Lem:segm_segm_1_case}
If $\bar \phi \in \phi_0 + [\beta^-(\phi_0),2\pi + \bar \theta - \hat \theta]$ and $\tilde r = \check r$ is in the segment case, any non-null perturbation of the optimal solution is strictly positive at $\bar \phi$.
\end{lemma}

\begin{proof}
Fix indeed $\bar \phi \in [\beta^{-}(\phi_0),2\pi+\bar\theta - \hat \theta]$. In $P_0=\zeta(s_0)$ compute the fastest saturated spiral: we know it is optimal whenever $\bar\theta-\theta_0=\theta>\hat\theta$, that is, it is outside the negativity region. Thus in the tent solution $\hat\ell_0$=0, and this corresponds precisely to the case of Figure \ref{Fig:segm_arc_1}.
\end{proof}

\begin{figure}
\resizebox{.75\textwidth}{!}{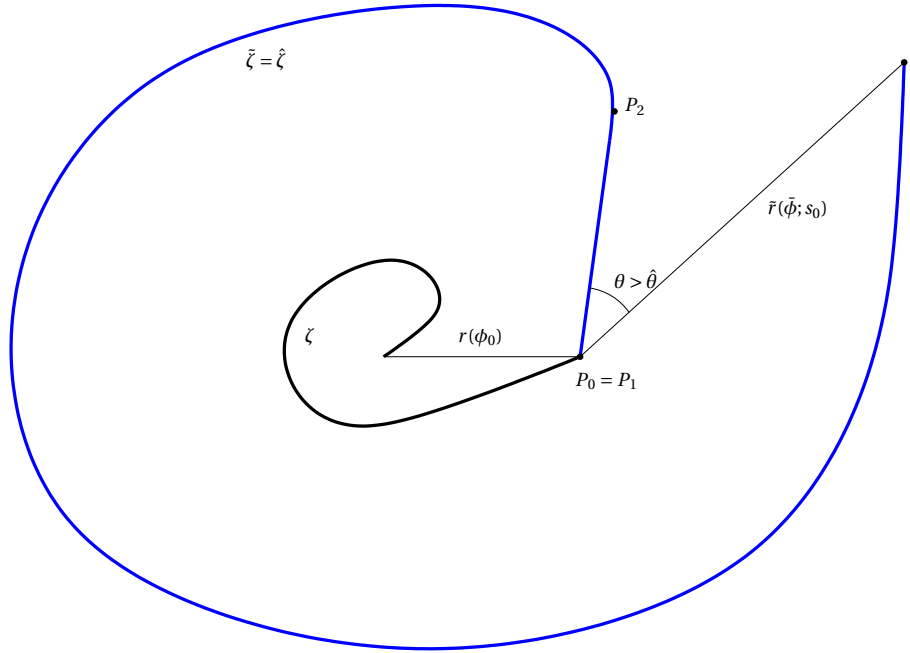}
\caption{The geometric situation of Section \ref{Sss:segm-segm_low}: being $\bar \phi - 2\pi - \phi_0 - \bar \theta > \hat \theta$, $\hat r = \tilde r$ and the derivative is $\leq 0$.}
\label{Fig:segm_arc_1}
\end{figure}

%
%
%
%

\subsection{Segment-segment with \texorpdfstring{$\bar \phi \in \phi_0 + 2\pi + (\bar \theta-\hat \theta,\bar \theta)$}{bar phi in phi0 + 2pi + (bar theta - hat theta, bar theta)}, See Fig. \ref{Fig:segm_segm_2_case}.}
\label{Sss:segm_semg_neg}

This situation occurs when $\bar \phi$ is inside the negativity region $\gls{Nsegm}$: it is thus more convenient to take an initial segment in the direction $\bar \phi$, because the derivative $\delta \tilde r$ is negative according to Proposition \ref{Prop:regions_pos_neg_segm}.

The spiral $\check \zeta$, corresponding to $\check r(\phi;s_0)$, is thus made of a segment $[P_0 = P_1 = \check P_0, \check P^- = \check P_1]$ and a second segment $[\check P_1,\check P_2]$, forming an angle equal to the critical value $\hat \theta$. The assumption to be in the segment-segment case enters here because the second segment corresponds to the fastest saturated spiral starting in $\check P_1$: in the next Section \ref{Sss:segm_arc_2} we will address the segment-arc case, where after $\check P_1$ the fastest saturated spiral starts with a level-set arc.

In order to prove that any perturbation is positive, we compute the derivative w.r.t. an initial perturbation.

First of all, if $\phi_0 + \beta^+(\phi_0) + 2\pi > \bar \phi$, we are not in the optimal solution: indeed this would imply that
$$
\delta \check r(\bar \phi;s_0) -= \delta \tilde r(\bar \phi;s_0) > 0.
$$
Hence the situation to be considered is as in Fig. \ref{Fig:segm_segm_2_case}. 

\begin{figure}
\begin{subfigure}{.475\textwidth}
\resizebox{\textwidth}{!}{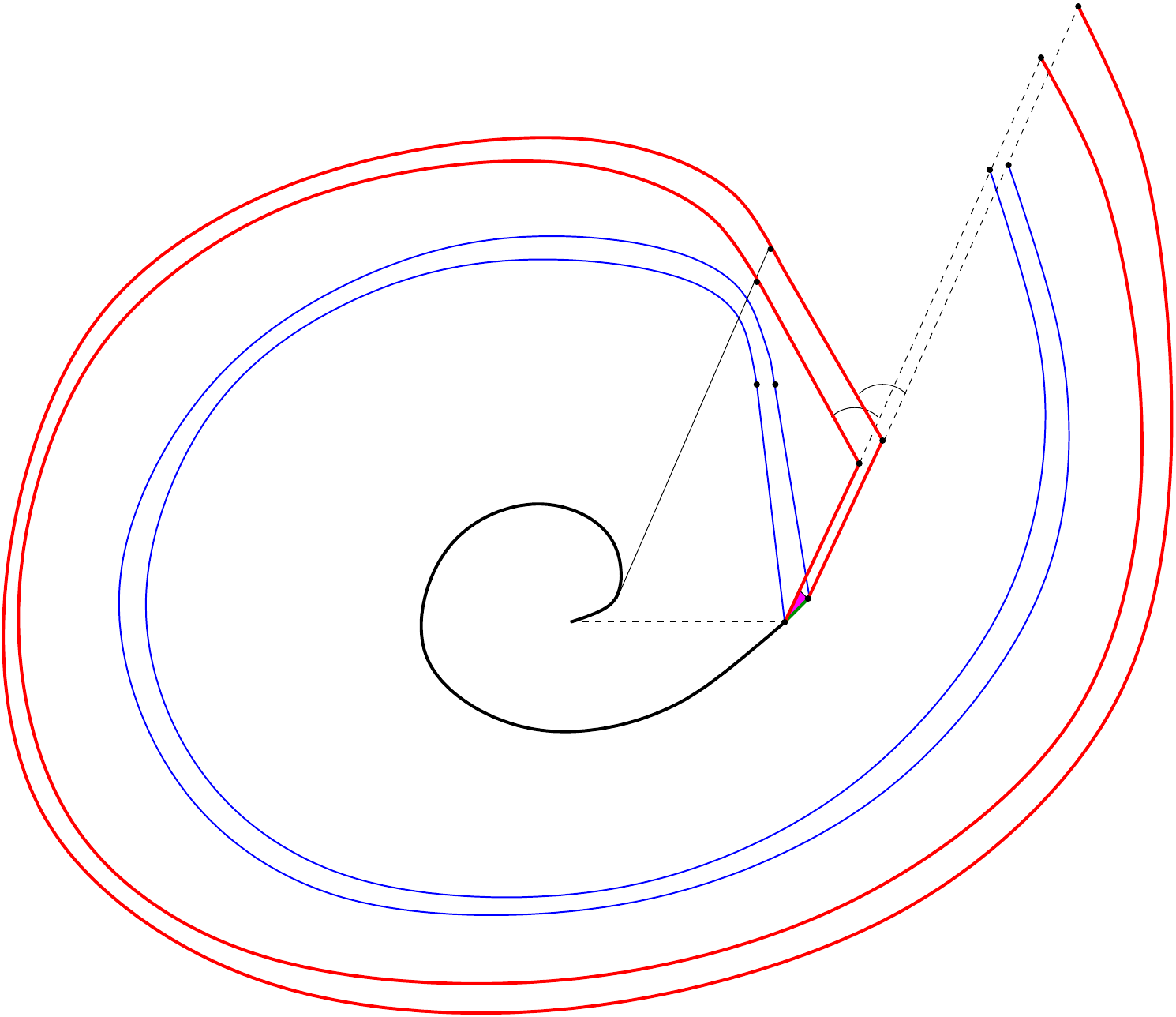}
\caption{The perturbation of the optimal solution $\check r$ for the case of Section \ref{Sss:segm_semg_neg}.}
\label{Fig:segm_segm_2_case}
\end{subfigure} \hfill
\begin{subfigure}{.475\textwidth}
\resizebox{\textwidth}{!}{\includegraphics{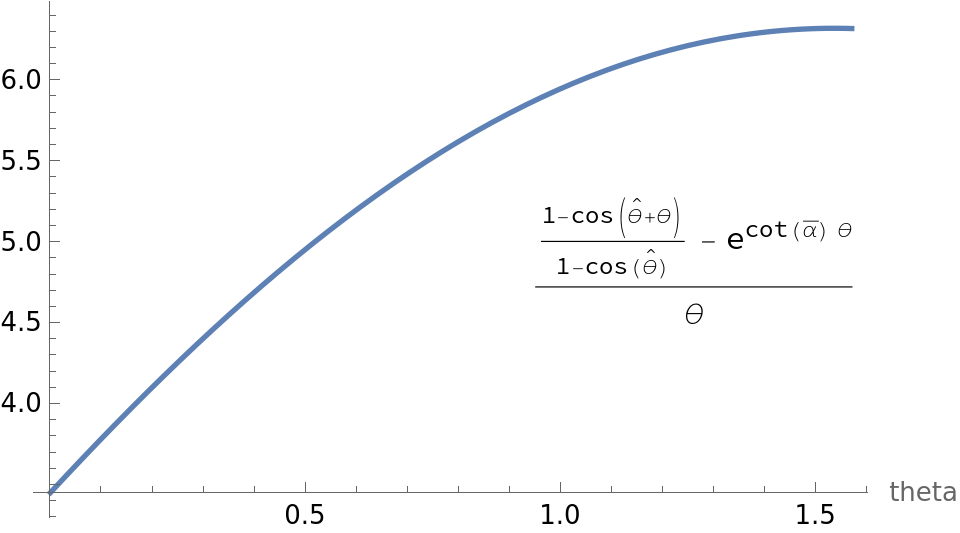}}
\caption{Plot of the function (\ref{Equa:segm_segm_2_final}), whose minimal value is $3.45283$.}
\label{Fig:segm_segm_case_2_final}
\end{subfigure} 
\caption{In the first figure, in blue the fastest saturated solution $\tilde r$ is represented, while $\check r$ is in red. While $\delta \tilde r(\bar \phi;s_0) < 0$ (variation of the distance of the final point from $P_0$), it holds $\delta \check r(\phi;s_0) \geq 0$ (variation of the distance from $\check P^-$), giving the optimality in this case. It is also represented the angle $\theta$ (magenta), formed by the direction of the perturbation and $\bar \phi$.}
\end{figure}

The variations are
\begin{equation*}
\delta \check P_0 = \delta s_0 e^{\i (\bar \phi - \theta)}, \quad 
\delta \check P_1 = \delta \check P_0 + \delta \ell_0 e^{\i \bar \phi} = \delta s_0 e^{\i (\bar \phi - \theta)} + \delta \ell_0 e^{\i \bar \phi},
\end{equation*}
\begin{equation*}
\delta \check P_2 = \delta \check P_1 + \delta \ell_1 e^{\i (\bar \phi + \hat \theta)} = \delta s_0 e^{\i (\bar \phi - \theta)} + \delta \ell_0 e^{\i \bar \phi} + \delta \ell_1 e^{\i (\bar \phi + \hat \theta)},
\end{equation*}
where we recall that
\begin{equation*}
\ell_0 = |\check P_1 - \check P_0|, \quad \ell_1 = |\check P_2 - \check P_1|.
\end{equation*}
The vector relation among the above quantities is
\begin{equation*}
\delta \check r_2(\bar \phi - 2\pi - \bar \alpha + \hat \theta) e^{\i (\bar \phi + \hat \theta - \bar \alpha)} = \delta \check P_2,
\end{equation*}
because the angle of the optimal ray with $\check \zeta$ at $\check P_2$ is constantly equal to $\bar\alpha$ by construction, see Fig. \ref{Fig:segm_segm_2_case}.
Projecting the above equation one obtains
\begin{align*}
0 &= e^{\i (\bar \phi + \frac{\pi}{2} + \hat \theta - \bar \alpha)} \cdot \delta \check P_2 \\
&= \delta s_0 \sin(\bar \alpha - \hat \theta - \theta) + \delta \ell_0 \sin(\bar \alpha - \hat \theta) + \delta \ell_1 \sin(\bar \alpha), 
\end{align*}
\begin{align*}
\delta \check r_2 = \delta \check P_2 \cdot e^{\i (\bar \phi + \hat \theta - \bar \alpha)} &= \delta s_0 \cos(\bar \alpha - \hat \theta - \theta) + \delta \ell_0 \cos(\bar \alpha - \hat \theta) + \delta \ell_1 \cos(\bar \alpha) \\
&= \delta s_0 \frac{\sin(\hat \theta + \theta)}{\sin(\bar \alpha)} + \delta \ell_0 \frac{\sin(\hat \theta)}{\sin(\bar \alpha)},
\end{align*}
where we have used the first projection to replace $\delta \ell_1$.

The saturation condition at the point $\check P_2 + \delta \check P_2$ gives
\begin{align*}
0 &= \delta \check r_2 - \cos(\bar \alpha) \big( \delta s_0 + \delta \ell_0 + \delta \ell_1 \big) \\
&= \delta s_0 \big( \cos(\bar \alpha - \hat \theta - \theta) - \cos(\bar \alpha) \big) + \delta \ell_0 ( \cos(\bar \alpha - \hat \theta) - \cos(\bar \alpha)),
\end{align*}
so that
\begin{equation*}
\delta \ell_0 = - \delta s_0 \frac{\cos(\bar \alpha - \hat \theta - \theta) - \cos(\bar \alpha)}{\cos(\bar \alpha - \hat \theta) - \cos(\bar \alpha)},
\end{equation*}
\begin{align*}
\delta \check r_2 &= \delta s_0 \frac{\sin(\hat \theta + \theta) (\cos(\bar \alpha - \hat \theta) - \cos(\bar \alpha)) - \sin(\hat \theta) (\cos(\bar \alpha - \hat \theta - \theta) - \cos(\bar \alpha))}{\sin(\bar \alpha) (\cos(\bar \alpha - \hat \theta) - \cos(\bar \alpha))} \\
&= \delta s_0 \frac{\cot(\bar \alpha)}{\cos(\bar \alpha - \hat \theta) - \cos(\bar \alpha)} \big( \sin(\theta) (1 - \cos(\hat \theta)) + (1 - \cos(\theta)) \sin(\hat \theta) \big).
\end{align*}

As in Corollary \ref{Cor:equation}, the ODE for the perturbation divided by $\delta s_0$ becomes
\begin{equation*}
\delta \check r(\bar \phi - \theta -) = e^{\cot(\bar \alpha) (2\pi + \bar \alpha - \hat \theta - \theta)} \delta \check r_2,
\end{equation*}
\begin{equation*}
\delta \check r(\bar \phi -) = e^{\cot(\bar \alpha) (2\pi + \bar \alpha - \hat \theta)} \delta \check r_2 - e^{\cot(\bar \alpha) \theta},
\end{equation*}
and then the final value is computed as
\begin{align*}
\delta \check r(\bar \phi) &= \delta \check r_2 e^{\cot(\bar \alpha) (2\pi + \bar \alpha - \hat \theta)} - e^{\cot(\bar \alpha) \theta} - \delta \ell_0 \\
&= \frac{\cot(\bar \alpha)}{\cos(\bar \alpha - \hat \theta) - \cos(\bar \alpha)} \big( \sin(\theta) (1 - \cos(\hat \theta)) + (1 - \cos(\theta)) \sin(\hat \theta) \big) e^{\cot(\bar \alpha) (2\pi + \bar \alpha - \hat \theta)} \\
& \quad \qquad \qquad \qquad - e^{\cot(\bar \alpha) \theta} + \frac{\cos(\bar \alpha - \hat \theta - \theta) - \cos(\bar \alpha)}{\cos(\bar \alpha - \hat \theta) - \cos(\bar \alpha)} \\
&= \frac{1 - \cos(\hat \theta + \theta)}{1 - \cos(\hat \theta)} - e^{\cot(\bar \alpha) \theta},
\end{align*}
where we have used the definition of \gls{thetahat}, namely (see \eqref{Equa:hat_theta_det_666})
$$
\frac{\cot(\bar \alpha)}{\sin(\bar \alpha)}(1 - \cos(\hat \theta)) e^{\cot(\bar \alpha)(2\pi + \bar \alpha - \hat \theta)} = 1.
$$
A numerical plot of
\begin{equation}
\label{Equa:segm_segm_2_final}
\theta \mapsto \frac{\delta \check r(\bar \phi)}{\theta} = \frac{1}{\theta} \bigg( \frac{1 - \cos(\hat \theta + \theta)}{1 - \cos(\hat \theta)} - e^{\cot(\bar \alpha) \theta} \bigg)
\end{equation}
is in Fig. \ref{Fig:segm_segm_case_2_final}: the function is strictly positive, with minimal value
\begin{equation*}
\lim_{\theta \to 0} \frac{1}{\theta} \bigg( \frac{1 - \cos(\hat \theta + \theta)}{1 - \cos(\hat \theta)} - e^{\cot(\bar \alpha) \theta} \bigg) = \frac{\sin(\hat \theta)}{1 - \cos(\hat \theta)} - \cot(\bar \alpha) = 3.45283.
\end{equation*}

We have proved the following result.

\begin{lemma}
\label{Lem:segm_segm_2_case}
If $\bar \phi \in \phi_0 + 2\pi + (\bar \theta-\hat \theta,\bar \theta)$, any non-null perturbation of the optimal solution is strictly positive at $\bar \phi$ for $\theta \in [0,\pi]$.
\end{lemma}

%
%
%

\begin{figure}
\begin{subfigure}{.453\textwidth}
\resizebox{\linewidth}{!}{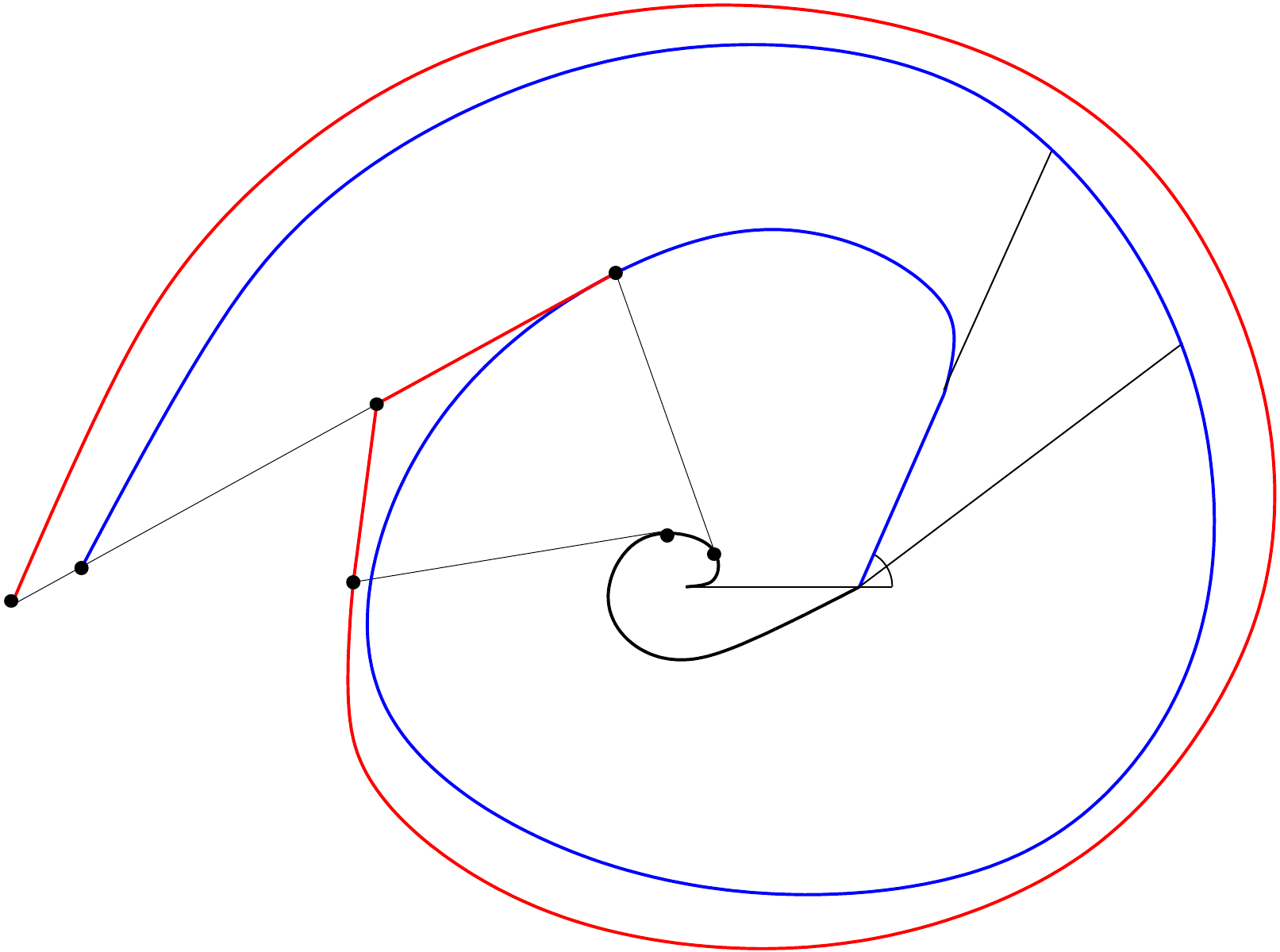}
\caption{The geometric situation of Section \ref{Sss:segm_semg_after_neg}.}
\label{Fig:segm_segm_sat_1}
\end{subfigure}
\hfill
\begin{subfigure}{.453\linewidth}
\resizebox{\linewidth}{!}{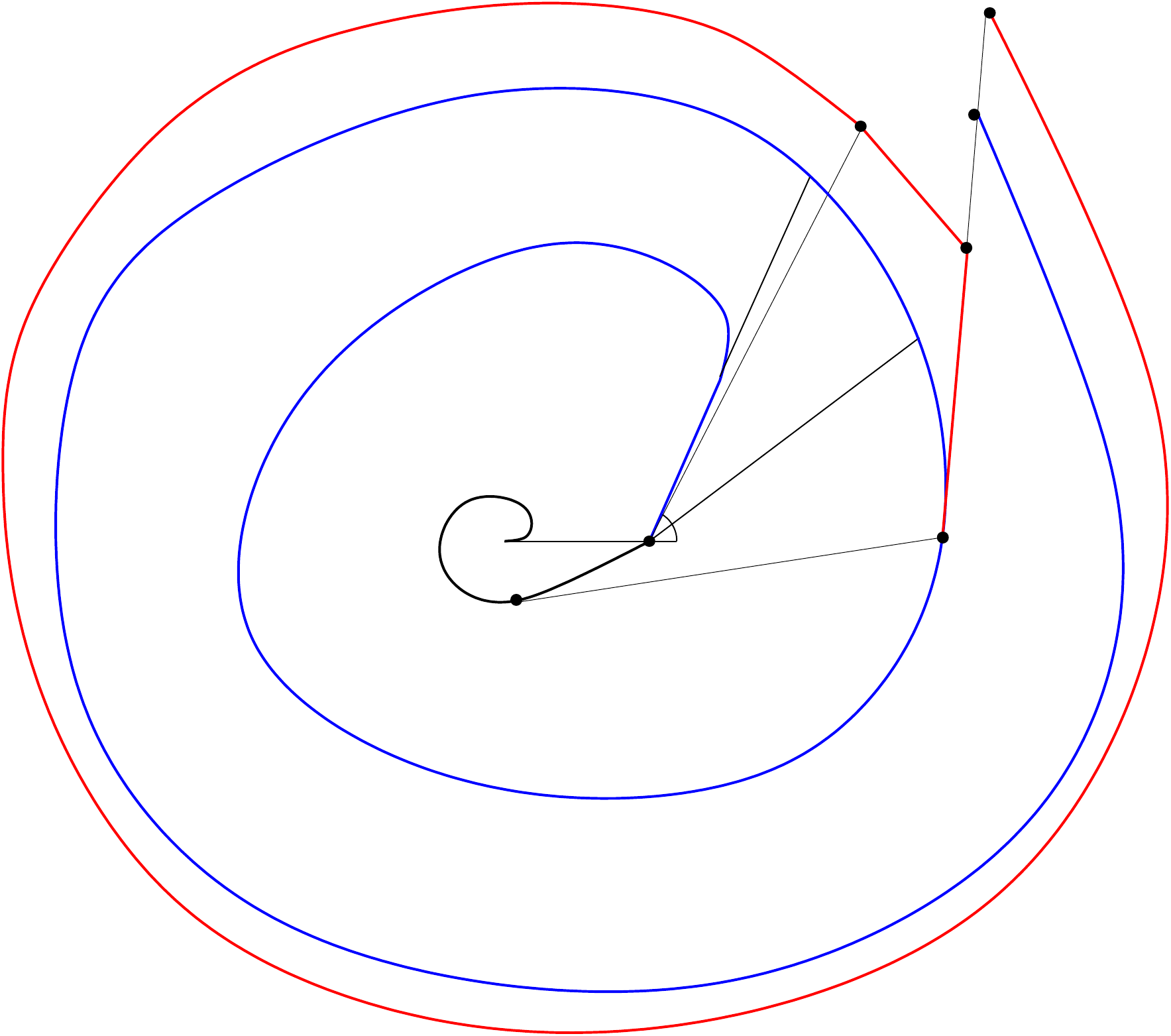}
\caption{The geometric situation of Section \ref{Sss:segm_semg_after_neg_2}.}
\label{Fig:segm_segm_sat_2}
\end{subfigure}
\begin{subfigure}{.453\textwidth}
\resizebox{\linewidth}{!}{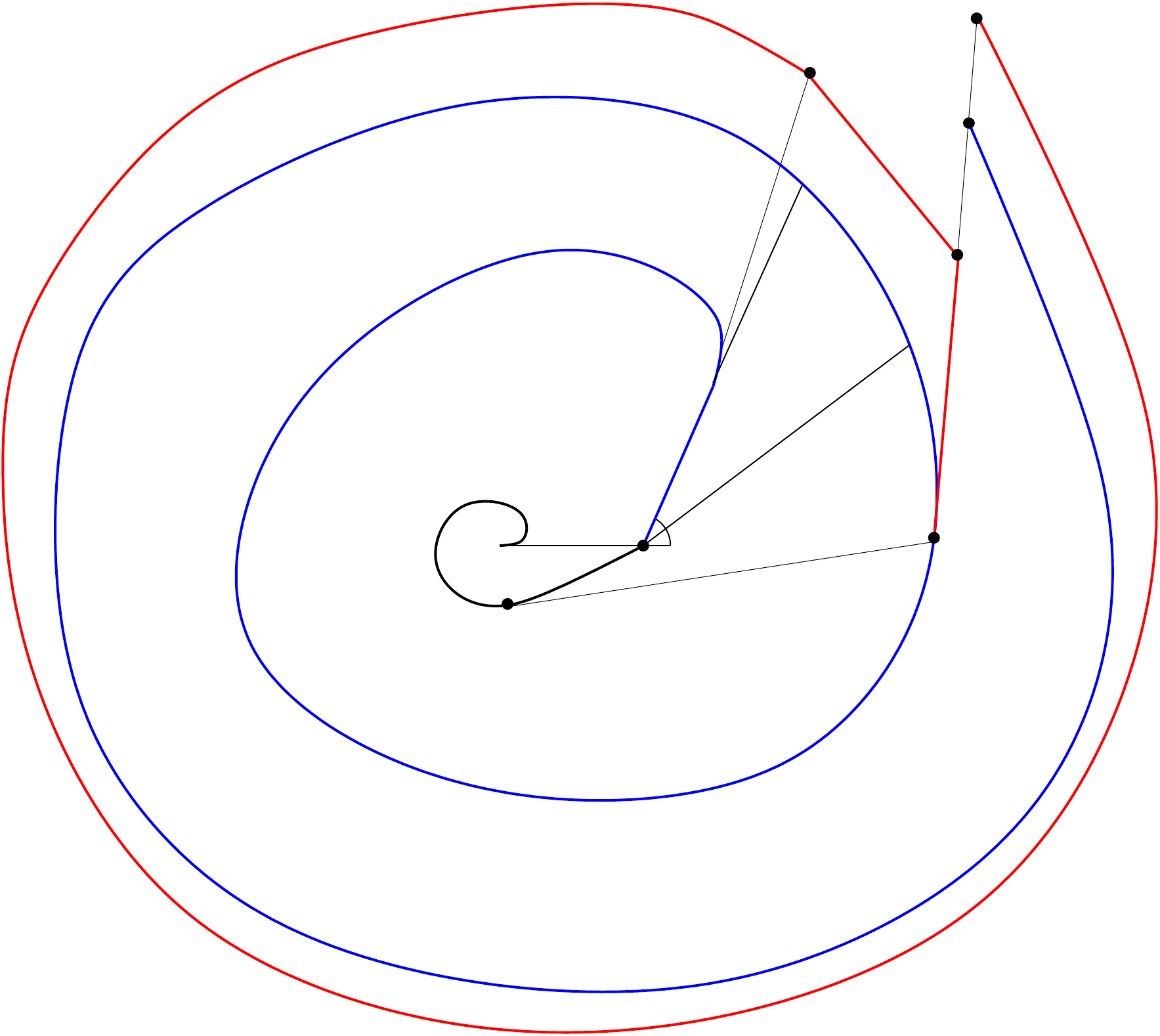}
\caption{The geometric situation of Section \ref{Sss:segm_semg_after_neg_3}.}
\label{Fig:segm_segm_sat_3}
\end{subfigure}
\hfill
\begin{subfigure}{.453\linewidth}
\resizebox{\linewidth}{!}{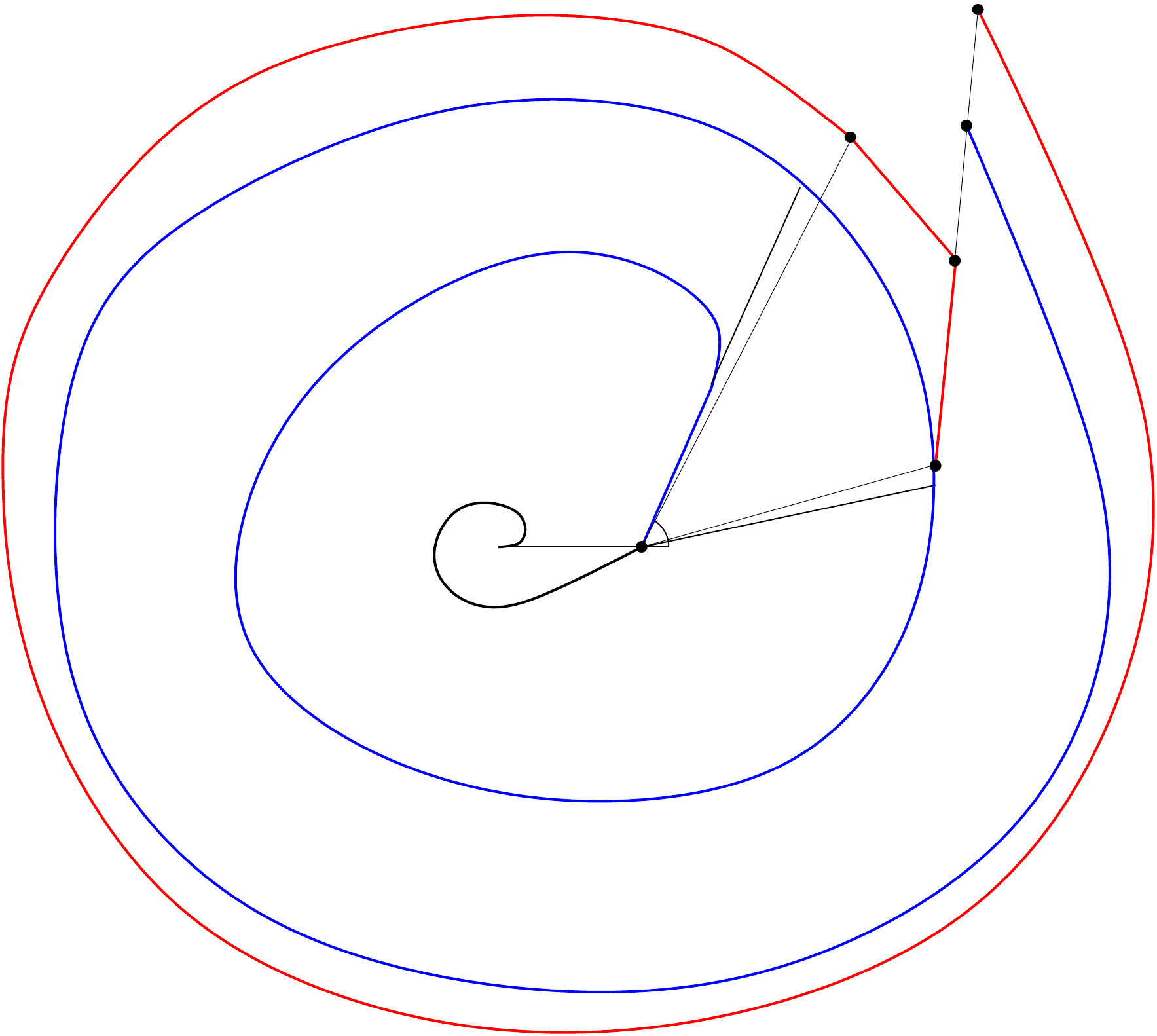}
\caption{The geometric situation of Section \ref{Sss:segm_semg_after_neg_4}.}
\label{Fig:segm_segm_sat_4}
\end{subfigure}
\begin{subfigure}{.453\textwidth}
\resizebox{\linewidth}{!}{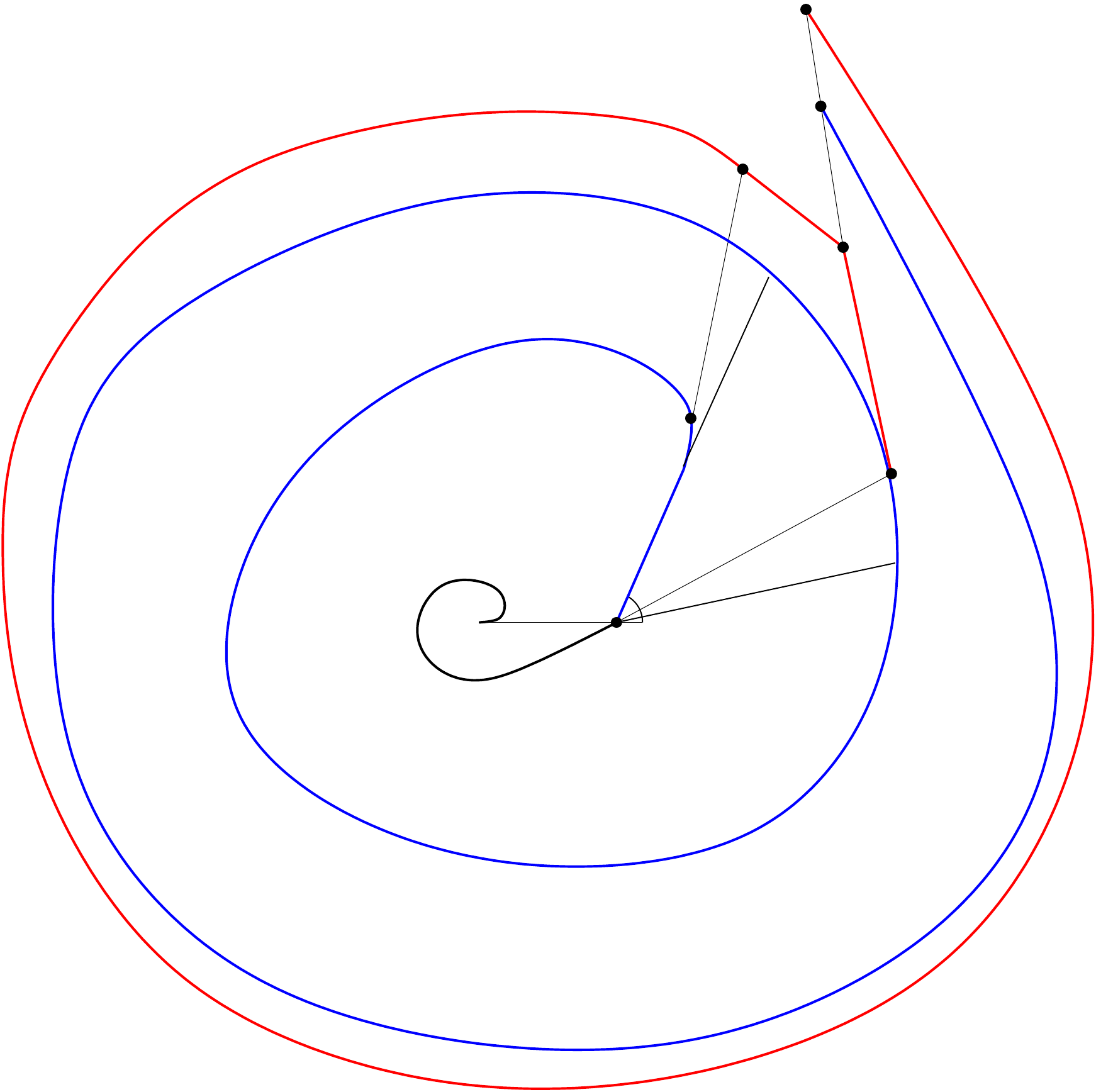}
\caption{The geometric situation of Section \ref{Sss:segm_semg_after_neg_5}.}
\label{Fig:segm_segm_sat_5}
\end{subfigure}
\hfill
\begin{subfigure}{.453\linewidth}
\resizebox{\linewidth}{!}{\input{segm_segm_sat_6.pdf_t}}
\caption{The geometric situation of Section \ref{Sss:segm_semg_after_neg_6}.}
\label{Fig:segm_segm_sat_6}
\end{subfigure}
\caption{The various geometric situation for the segment-segment case}
\label{Fig:segm_segm_sat_tot}
\end{figure}

\bigskip

\begin{remark}
\label{Rem:divide_deltas_0}
From now on we will write the formulas for the perturbation of $\tilde \zeta, \check \zeta$ omitting the term $\delta s_0$.
\end{remark}

\bigskip

\subsection{Segment-segment with \texorpdfstring{$\bar \phi \in \phi_0 + 2\pi + [\bar \theta,2\pi + \bar \alpha + \bar \theta - \theta - \hat \theta)$}{bar phi in phi0 + 2pi + [bar theta,2pi + bar alpha + bar theta - theta- hat theta]}: see Fig. \ref{Fig:segm_segm_sat_1}}
\label{Sss:segm_semg_after_neg}

In this case the full tent is constructed, so that we can use the formulas of Lemma \ref{Lem:relation_admissibility_tent} to compute the perturbation $\check r(\phi;s_0)$. In this case, we have that the base of the tent is not perturbed, so that we are in the situation of Fig. \ref{Fig:tent_comput_geo} with
\begin{equation*}
\hat \zeta = \{(0,0)\} \quad \text{because} \quad \check \phi_0 = \bar \phi - 2\pi - \bar \alpha \geq \bar \phi_0 + \bar \theta - \bar \alpha, \ \check \phi_1 = \check \phi_0 + \hat \theta \leq \phi_0 + 2\pi - \bar \theta - \theta.
\end{equation*}
The computations are thus explicit: the quantities of Lemma \ref{Lem:relation_admissibility_tent} are computed in this case as
\begin{equation*}
\hat r_0 = \delta \tilde r(\check \phi_0;s_0), \quad \hat \ell_0 = \frac{\hat r_0 (1 - \cos(\hat \theta))}{\cos(\bar \alpha - \hat \theta) - \cos(\bar \alpha)},
\end{equation*}
\begin{equation*}
\hat r_2 = \frac{\hat r_0 \sin(\bar \alpha + \hat \theta) + \hat \ell_0 \sin(\hat \theta)}{\sin(\bar \alpha)}, \quad \hat \ell_1 = \frac{\hat r_2 - \hat r_1}{\cos(\bar \alpha)} - \hat \ell_0.
\end{equation*}
The numerical plot of
\begin{align*}
\frac{e^{-\bar c(2\pi + \bar \alpha)}}{\theta^2} \delta \check r(\bar \phi;s_0) &= \frac{e^{-\bar c(2\pi + \bar \alpha)}}{\theta^2}  \bigg [ (\hat r_2 - \delta \tilde r(\check \phi_2;s_0)) e^{\cot(\bar \alpha)(2\pi + \bar \alpha - \hat \theta)} + \delta \tilde r(\bar \phi;s_0) - \hat \ell_0 \bigg] 
\end{align*}
is in Fig. \ref{Fig:segm_segm_after_neg_1}: its minimal value is $> 0.2$.

We have proved the following result.

\begin{lemma}
\label{Lem:segm_segm_3_case}
If $\bar \phi \in \phi_0 + 2\pi + [\bar \theta,2\pi + \bar \alpha + \bar \theta - \hat \theta)$, any non-null perturbation of the optimal solution is strictly positive at $\bar \phi$ for $\theta \in [0,\pi]$.
\end{lemma}

%
%

\begin{figure}
\begin{subfigure}{.475\textwidth}
\resizebox{\linewidth}{!}{\includegraphics{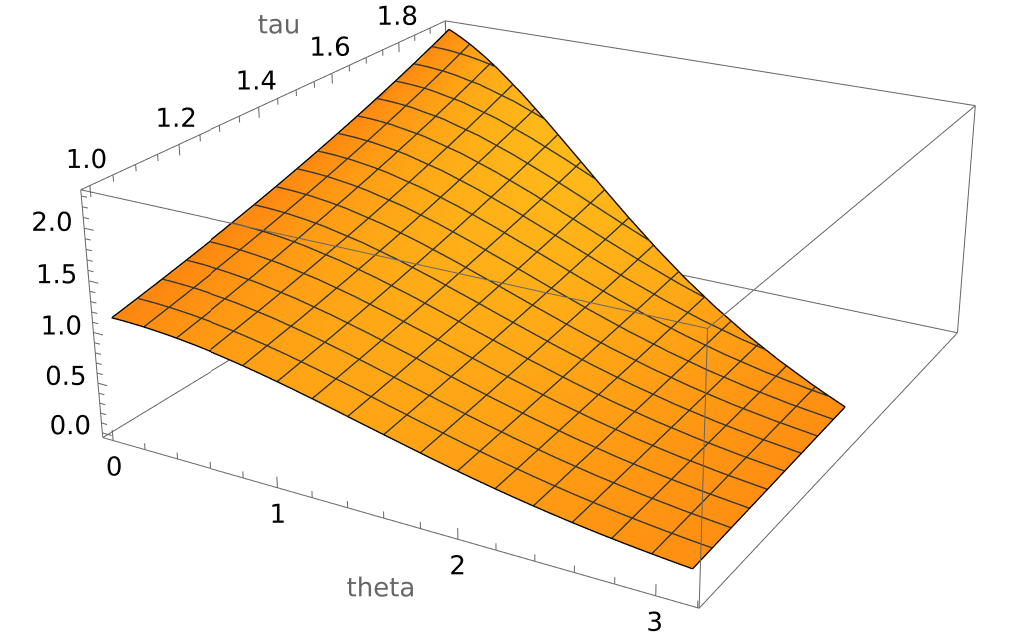}}
\caption{Plot of the function $\frac{e^{-\bar c \bar \phi} \delta \hat r(\bar \phi)}{\theta^2}$.}
\label{Fig:segm_segm_after_neg_1}
\end{subfigure}
\hfill
\begin{subfigure}{.475\linewidth}
\resizebox{\linewidth}{!}{\includegraphics{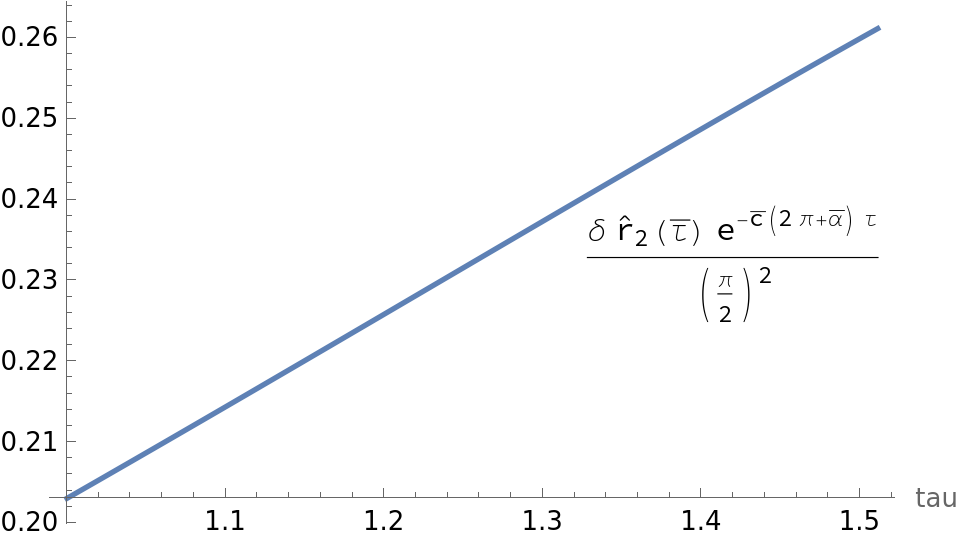}}
\caption{Minimal value of the function $\frac{e^{-\bar c \bar \phi} \delta \hat r(\bar \phi)}{\theta^2}$, i.e. when $\theta = \frac{\pi}{2}$.}
\label{Fig:segm_segm_after_neg_2}
\end{subfigure}
\caption{Numerical analysis of Section \ref{Sss:segm_semg_after_neg}.}
\label{Fig:segm_segm_after_neg}
\end{figure}


\subsection{Segment-segment with \texorpdfstring{$\check \phi_0 \leq \phi_0 + 2\pi + \bar \theta - \theta, \check \phi_2 \in \phi_0 + 2\pi + [\bar \theta - \theta,\bar \theta)$}{chec phi 0 leq phi0 + 2pi + bar theta - theta, check phi2 in phi0 + 2pi + [bar theta - theta,bar theta)}: see Fig. \ref{Fig:segm_segm_sat_2}}
\label{Sss:segm_semg_after_neg_2}

Since the angle $\bar\phi$ is fixed, the idea here is that we perturb the point $\hat Q_2$, while the perturbation of the point $\hat Q_0$ is $\delta\hat Q_0=0$. Because of the relative position of $\bar\phi$ with respect to the perturbed point $\hat Q_2+\delta\hat Q_2$, we see that the perturbation of the point $\hat Q_2$ is the unit vector $e^{i(\phi_0+2\pi+\bar\theta-\theta)}$, where $\theta=\bar\theta-\theta_0$ is the angle introduced in Section \ref{S:family}.
Hence the formulas of Lemma \ref{Lem:relation_admissibility_tent} become 
%
\begin{equation*}
\delta \hat Q_2 - \delta \hat Q_0 = e^{\i (\phi_0 + 2\pi + \bar \theta - \theta)} = e^{\i (\bar \phi - (2\pi + \bar \alpha))} e^{\i (\phi_0 + \bar \theta - \bar \alpha - \theta + 2 (2 \pi + \bar \alpha) - \bar \phi)} = e^{\i (\bar \phi - (2\pi + \bar \alpha))} e^{\i ((2 \pi + \bar \alpha)(2 - \bar \tau) - \theta)},
\end{equation*}
\begin{equation*}
\big( \delta \hat Q_2 - \delta \hat Q_0 \big) \cdot e^{\i (\bar \phi - (2\pi + \bar \alpha) + \hat \theta)} = \cos \big( (2\pi + \bar \alpha) (2 - \bar\tau) - \theta - \hat \theta \big),
\end{equation*}
\begin{equation*}
\big( \delta \hat Q_2 - \delta \hat Q_0 \big) \cdot e^{\i (\bar \phi - (2\pi + \bar \alpha) + \hat \theta)} e^{\i\left(\frac{\pi}{2}-\bar\alpha\right)}= \cos \bigg( (2\pi + \bar \alpha) (2 - \bar\tau) - \theta - \bar \alpha - \hat \theta + \frac{\pi}{2} \bigg),
\end{equation*}
\begin{equation*}
\delta \hat L = 1, \quad \delta \hat r_0 = \delta \tilde r(\check \phi_0) = \delta \rho(\check \tau_0) e^{\bar c (2\pi + \bar \alpha) \check \tau_0}.
\end{equation*}
In terms of $\check \tau_0$ the interval of interest is
\begin{equation*}
1 - \frac{\theta + \hat \theta}{2\pi + \bar \alpha} \leq \check \tau_0 < 1 - \frac{\max\{\theta,\hat \theta\}}{2\pi + \bar \alpha}.
\end{equation*}
The functions $\hat \ell_0,\hat r_2,\hat \ell_1$ are computed according to Lemma \ref{Lem:relation_admissibility_tent}.

In Fig. \ref{Fig:segm_segm_after_neg2_1} we numerically plot the function
\begin{equation}
\label{Equa:segm_tent_1}
\frac{1}{\theta^2} \delta \hat r(\bar \phi) e^{- \bar c (2\pi + \bar \alpha) \bar \tau} = \frac{e^{- \bar c (2\pi + \bar \alpha) \bar \tau}}{\theta^2} \big[ e^{\cot(\bar \alpha)(2\pi + \bar \alpha - \hat \theta)} (\delta \hat r(\check \phi_2) - \delta \tilde r(\check \phi_2)) + \delta \tilde r(\bar \phi) - \delta \hat \ell_0 \big],
\end{equation}
corresponding to the quantity computed in Equation \eqref{eq:comparaison:Saturated}, which follows with the comparison with the saturated spiral. We compute this quantity as a function of \newglossaryentry{taubar}{name=\ensuremath{\bar \tau},description={the value of $\tau$ corresponding to $\bar \phi$, $\bar \tau = \frac{\bar \phi - \phi_0 - \bar \theta + \bar \alpha}{2\pi + \bar \alpha}$}} $\gls{taubar} = (1-\varsigma) (1 - \frac{\theta+\hat \theta}{2\pi+\bar \alpha}) + \varsigma (2 - \frac{\max\lbrace{\theta,\hat \theta\rbrace}}{2\pi + \bar \alpha})$, $\varsigma \in [0,1]$. The function is strictly positive, with minimal value greater than $0.2$ in $\theta = \pi, \bar\tau = 2 - \frac{\pi + \hat \theta}{2\pi + \bar \alpha}$.
\begin{figure}
\begin{subfigure}{.475\textwidth}
\resizebox{\linewidth}{!}{\includegraphics{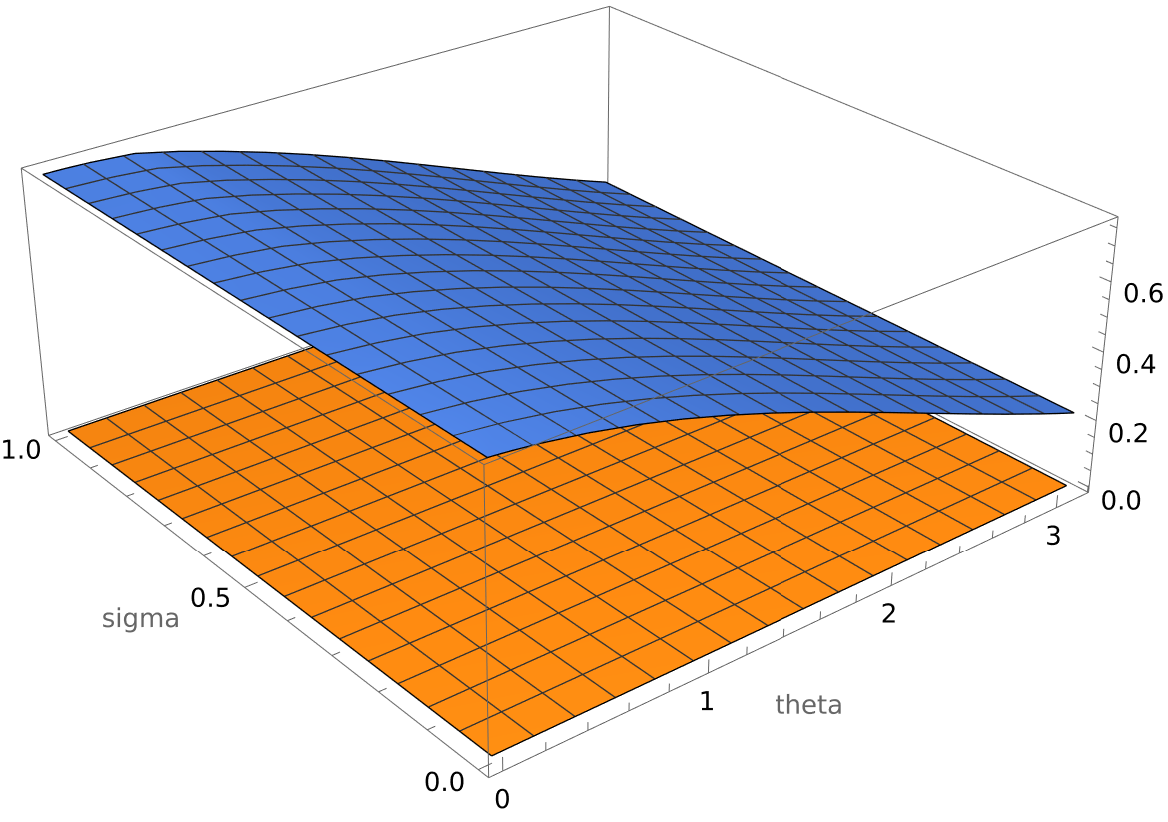}}
\caption{Plot of the function \eqref{Equa:segm_tent_1}.}
\label{Fig:segm_segm_after_neg2_1}
\end{subfigure}
\hfill
\begin{subfigure}{.475\linewidth}
\resizebox{\linewidth}{!}{\includegraphics{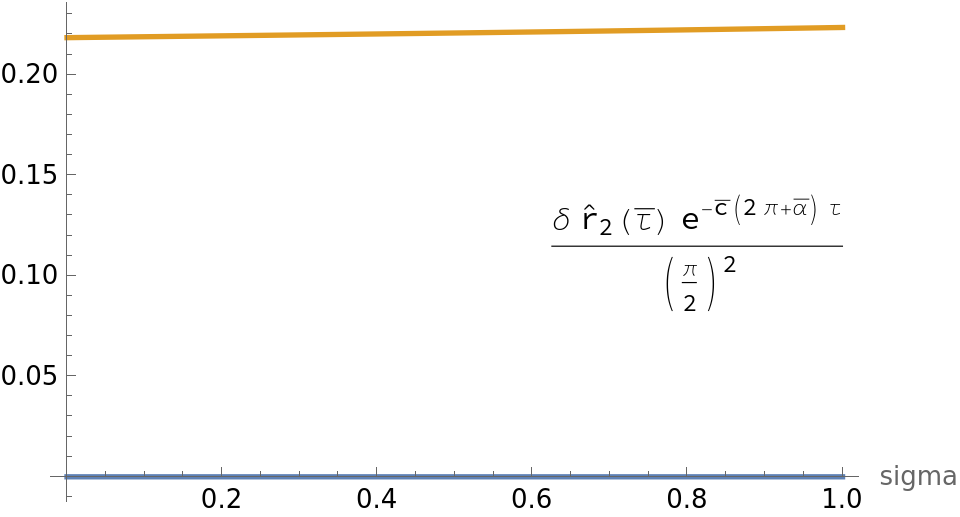}}
\caption{Minimal value for the function \eqref{Equa:segm_tent_1} with $\theta = \frac{\pi}{2}$.}
\label{Fig:segm_segm_after_neg2_2}
\end{subfigure}
\caption{Numerical analysis of Section \ref{Sss:segm_semg_after_neg_2}.}
\label{Fig:segm_tent_1}
\end{figure}

We thus have proved the following lemma.

\begin{lemma}
\label{Lem:segm_tent_1}
If $\check \phi_0 \leq \phi_0 + 2\pi + \bar \theta - \theta, \check \phi_2 \in \phi_0 + 2\pi + [\bar \theta - \theta,\bar \theta)$, then $\delta \hat r(\bar \phi) \geq 0$ for $\theta \in [0,\pi]$.
\end{lemma}

\begin{remark}
\label{Rem:more_regu}
Note that the asymptotic behavior $\delta \tilde r(\bar \phi) \simeq \theta$ as $\theta \to 0$ for these angles (see Fig. \ref{Fig:segm_round_2}) becomes now $\delta \hat r(\bar \phi) \simeq \theta^2$, showing that the tent has a regularizing effect on the dependence w.r.t. $\theta$, while $\delta \tilde r \sim -1$.
\end{remark}

%
%
%
%

\subsection{Segment-segment with \texorpdfstring{$\check \phi_0 \leq \phi_0 + 2\pi + \bar \theta - \theta, \check \phi_1 \geq \phi_0 + 2\pi + \bar \theta]$}{check phi0 leq phi0 + 2pi + bar theta - theta, check phi1 geq phi0 + 2pi + bar theta]}: see Fig. \ref{Fig:segm_segm_sat_3}}
\label{Sss:segm_semg_after_neg_3}

Because of the conditions on $\bar \phi$, we have in this case that the point $\hat Q_2$ and its perturbation lie in the saturated part, so that
\begin{equation*}
\delta \hat Q_2 - \delta \hat Q_0 = \delta \tilde r(\check \phi_2 - 2\pi - \bar \alpha) e^{\i (\phi_0 + \bar \theta - \bar \alpha + (2\pi + \bar \alpha) (\bar\tau - 2 + \hat \tau))} = \delta \rho(\bar \tau - 2 + \hat \tau) e^{\bar c (2\pi + \bar \alpha) (\bar \tau - 2 + \hat \tau)} e^{\i(\phi_0 + \bar \theta - \bar \alpha + (2\pi + \bar \alpha) (\bar\tau - 1))} e^{\i (\hat \theta - \bar \alpha)},
\end{equation*}
\begin{equation*}
\big( \delta \hat Q_2 - \delta \hat Q_0 \big) \cdot e^{\i (\bar \phi - (2\pi + \bar \alpha) + \hat \theta)} = \delta \rho(\bar \tau - 2 + \hat \tau) e^{\bar c (2\pi + \bar \alpha) (\bar \tau - 2 + \hat \tau)} \cos(\bar \alpha),
\end{equation*}
\begin{equation*}
		\big( \delta \hat Q_2 - \delta \hat Q_0 \big) \cdot e^{\i (\bar \phi - (2\pi + \bar \alpha) + \hat \theta)}e^{\i\left(-\frac{\pi}{2}+\bar\alpha\right)} = \delta \rho(\bar \tau - 2 + \hat \tau) e^{\bar c (2\pi + \bar \alpha) (\bar \tau - 2 + \hat \tau)} \cos \bigg( \frac{\pi}{2} - 2 \bar \alpha \bigg),
\end{equation*}
See the auxiliary picture \ref{fig:section:9:5} for reference. In particular we have also
\begin{figure}
	\centering
	\includegraphics[scale=0.4]{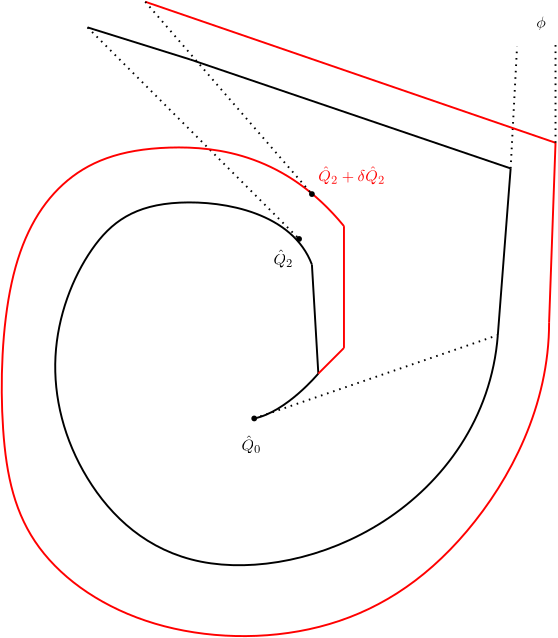}
	\caption{The spiral and its perturbed one in the case of Subsection \ref{Sss:segm_semg_after_neg_3}.}
		\label{fig:section:9:5}
\end{figure}
\begin{align*}
\delta \hat L 
&= 1 + \delta |P_2 - P_1| - \frac{r(\phi_2)}{\sin(\bar \alpha)} \delta \phi_2 + \int_0^{(2\pi + \bar \alpha)(\bar\tau - 2 + \hat \tau)} \frac{\delta \tilde r(\phi)}{\sin(\bar \alpha)} d\phi \\
\big[ \text{Equation} \ (\ref{Equa:delta_P_2P_1_segm}) \big] \quad &= \frac{1 - \cos(\theta)}{\sin(\bar \alpha)^2} + \int_0^{(2\pi + \bar \alpha)(\bar\tau - 2 + \hat \tau)} \frac{\delta \tilde r(\phi)}{\sin(\bar \alpha)} d\phi = \frac{1 - \cos(\theta)}{\sin(\bar \alpha)^2} e^{\cot(\bar \alpha)(2\pi + \bar \alpha)(\bar\tau-2+\hat \tau)},
\end{align*}
We also remark that the previous computation follows by the expression of $\delta\tilde r$ that can be recovered from Proposition \ref{Prop:equa_delta_tilde_r_segm}. Finally
\begin{equation*}
\delta \hat r_0 = \delta \rho(\bar \tau - 1) e^{\bar c (2\pi + \bar \alpha) (\bar \tau - 1)}.
\end{equation*}
In terms of $\check \tau_0$ the interval of interest is
\begin{equation*}
2 - \frac{\hat \theta}{2\pi + \bar \alpha} \leq \bar \tau < 2 - \frac{\theta}{2\pi + \bar \alpha}, \quad \theta < \hat \theta.
\end{equation*}
In Fig. \ref{Fig:segm_segm_after_neg3_1} we numerically plot the function
\begin{equation}
\label{Equa:segm_tent_2}
\frac{\delta \hat r(\bar \phi) e^{- \bar c (2\pi + \bar \alpha) \bar \tau}}{\theta^2} = \frac{e^{- \bar c (2\pi + \bar \alpha) \bar \tau}}{\theta^2} \big[ e^{\cot(\bar \alpha)(2\pi + \bar \alpha - \hat \theta)} (\delta \hat r(\check \phi_2) - \delta \tilde r(\check \phi_2)) + \delta \tilde r(\bar \phi) - \delta \hat \ell_0 \big],
\end{equation}
as a function of $\bar \tau = (1 - \varsigma)(2 - \hat \tau) + \varsigma(2 - \frac{\theta}{2\pi + \bar \alpha})$, $\varsigma \in (0,1)$. The function is strictly positive with minimum for $\theta = \hat \theta$ and value $> 0.14$. 

\begin{figure}
\begin{subfigure}{.475\textwidth}
\resizebox{\linewidth}{!}{\includegraphics{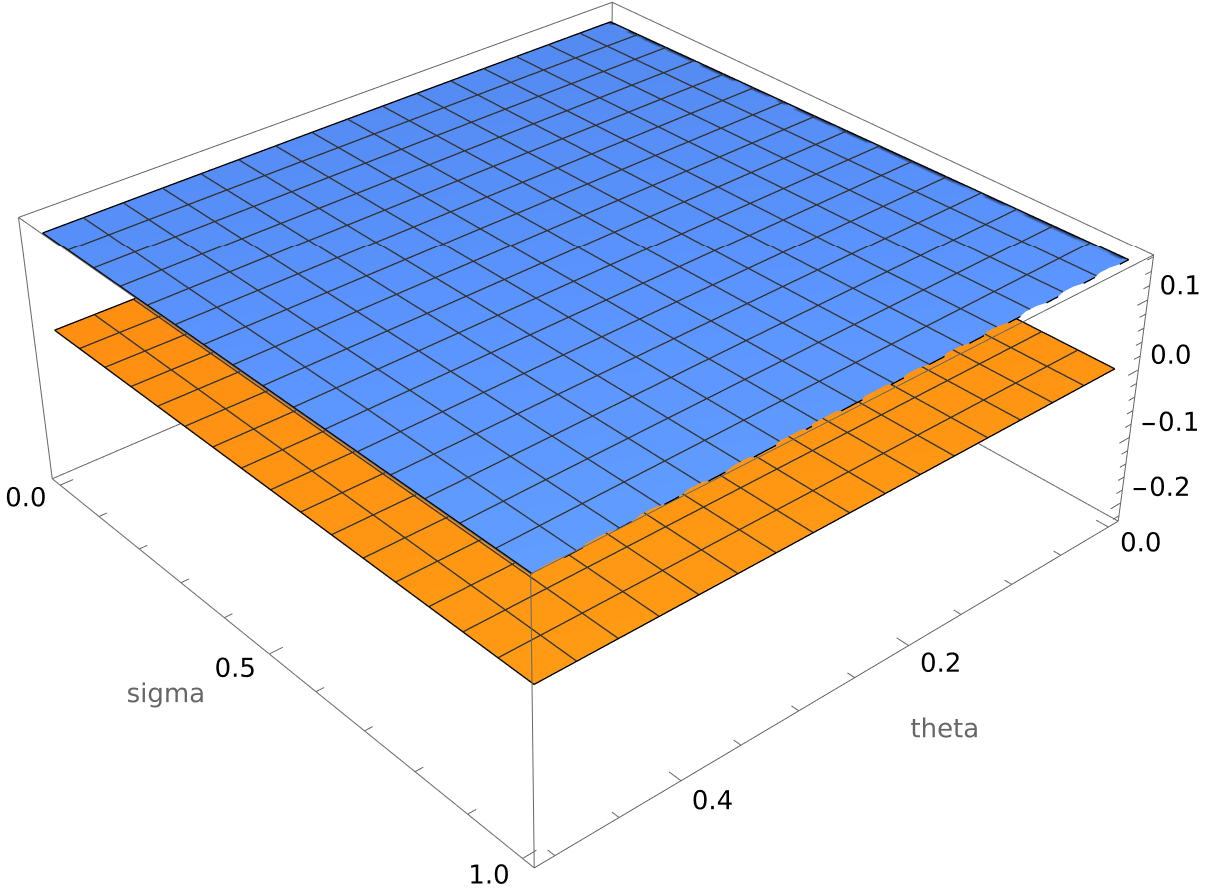}}
\caption{Plot of the function \eqref{Equa:segm_tent_2}.}
\label{Fig:segm_segm_after_neg3_1}
\end{subfigure}
\hfill
\begin{subfigure}{.475\linewidth}
\resizebox{\linewidth}{!}{\includegraphics{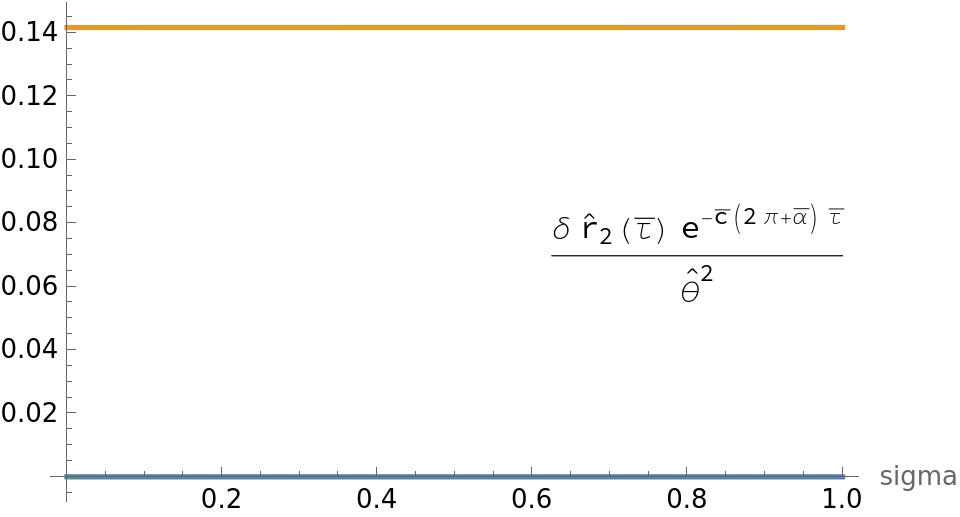}}
\caption{Plot of the function \eqref{Equa:segm_tent_2} when $\theta = \hat \theta$.}
\label{Fig:segm_segm_after_neg4_2}
\end{subfigure}
\caption{Numerical analysis of Section \ref{Sss:segm_semg_after_neg_3}.}
\label{Fig:segm_tent_2}
\end{figure}

We thus have proved the following lemma.

\begin{lemma}
\label{Lem:segm_tent_2}
If $\check \phi_0 \leq \phi_0 + 2\pi + \bar \theta - \theta, \check \phi_2 \geq \phi_0 + 2\pi + \bar \theta$, then $\delta \hat r(\bar \phi) \geq 0$.
\end{lemma}

%
%
%
%

\subsection{Segment-segment with \texorpdfstring{$\check \phi_0 \geq \phi_0 + 2\pi + \bar \theta - \theta, \check \phi_1 \leq \phi_0 + 2\pi + \bar \theta$}{check phi0 geq phi0 + 2pi + bar theta - theta, check phi1 leq phi0 + 2pi + bar theta}: see Fig. \ref{Fig:segm_segm_sat_4}}
\label{Sss:segm_semg_after_neg_4}

We compute in this case
\begin{equation*}
\delta \hat Q_2 - \delta \hat Q_0 = 0, \quad \delta \hat L = 0,
\end{equation*}
\begin{equation*}
\delta \hat r_0 = \delta \rho(\bar \tau - 1) e^{\bar c (2\pi + \bar \alpha) (\bar \tau - 1)}.
\end{equation*}
In terms of $\check \tau_0$ the interval of interest is
\begin{equation*}
2 - \frac{\theta}{2\pi + \bar \alpha} \leq \bar \tau < 2 - \hat \tau, \quad \theta > \hat \theta.
\end{equation*}
The situation is very similar to Section \ref{Sss:segm-segm_low}, and the function $\frac{\delta \check r(\bar \phi) e^{-\bar c(2\pi + \bar \alpha) \bar \tau}}{\theta^2}$ as a function of $\bar \tau = (1-\varsigma) (2 - \frac{\theta}{2\pi + \bar \alpha}) + \varsigma (2 - \hat \tau)$, $\varsigma \in (0,1)$, is plotted in Fig. \ref{Fig:segm_segm_after_neg4_1}, while in Fig. \ref{Fig:segm_segm_after_neg5_1} the minimal values are plotted for $\theta = \pi/2$.

\begin{figure}
\begin{subfigure}{.475\textwidth}
\resizebox{\linewidth}{!}{\includegraphics{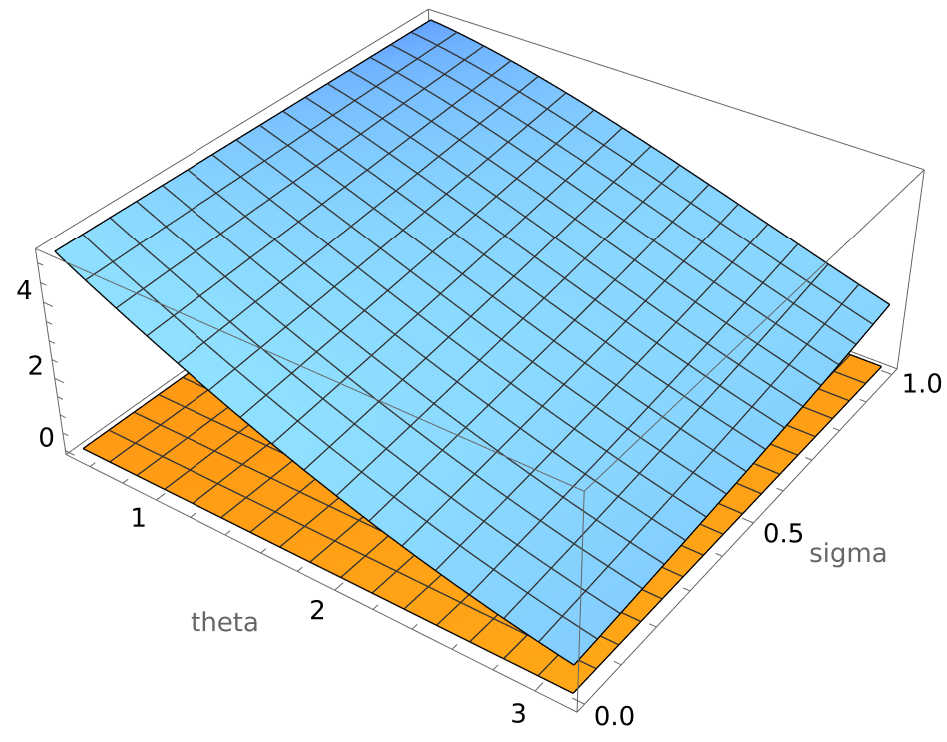}}
\caption{Plot of the function $\delta \check r(\bar \phi)/\theta^2$.}
\label{Fig:segm_segm_after_neg4_1}
\end{subfigure}
\hfill
\begin{subfigure}{.475\linewidth}
\resizebox{\linewidth}{!}{\includegraphics{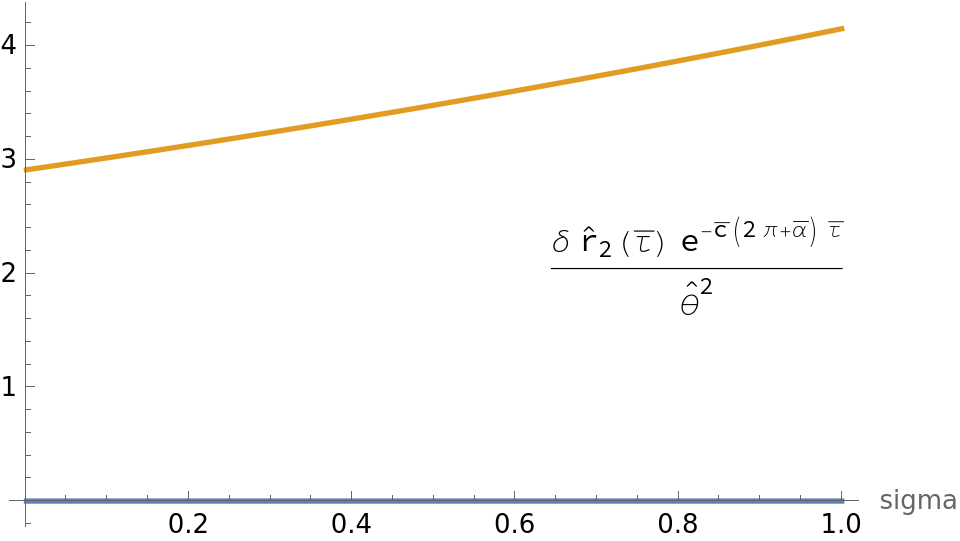}}
\caption{Plot of the function $\delta \check r(\bar \phi)/\pi^2$ (minimal values).}
\label{Fig:segm_segm_after_neg5_2}
\end{subfigure}
\caption{Numerical analysis of Section \ref{Sss:segm_semg_after_neg_4}.}
\label{Fig:segm_tent_4}
\end{figure}

%
%

\begin{lemma}
\label{Lem:segm_tent_3}
If $2 - \frac{\theta}{2\pi + \bar \alpha} \leq \bar\tau < 2 - \frac{\hat \theta}{2\pi + \bar \alpha}$, then $\delta \check r(\bar \phi) > 2.91 \theta^2$, $\theta \in (\hat \theta,\pi/2]$.
\end{lemma}

%
%
%
%
%
%

\subsection{Segment-segment with \texorpdfstring{$\check \phi_0 \geq \phi_0 + 2\pi + \bar \theta - \theta, \check \phi_1 \geq \phi_0 + 2\pi + \bar \theta$}{check phi0 geq phi0 + 2pi + bar theta - theta, check phi1 geq phi0 + 2pi + bar theta}: see Fig. \ref{Fig:segm_segm_sat_5}}
\label{Sss:segm_semg_after_neg_5}

In this case we have
\begin{align*}
\Delta\hat Q=\delta \hat Q_2 - \delta \hat Q_0 &= \delta \tilde r(\check \phi_2 - 2\pi - \bar \alpha) e^{\i (\phi_0 + \bar \theta -\bar \alpha + (2\pi + \bar \alpha) (\bar\tau - 2 + \hat \tau))} - e^{\i (\phi_0 + \bar \theta - \theta - (\phi_0 + \bar \theta - \bar \alpha + (2\pi + \bar \alpha) (\bar\tau - 1)))} \\
&= \delta \rho(\bar \tau - 2 + \hat \tau) e^{\bar c (2\pi + \bar \alpha) (\bar \tau - 2 + \hat \tau)} e^{\i (\phi_0 + \bar \theta + (2\pi + \bar \alpha) (\bar\tau - 1 + \hat \tau))} e^{\i (\hat \theta - \bar \alpha)} \\
& \quad - e^{\i (\phi_0 + \bar \theta - \bar \alpha + (2\pi + \bar \alpha) (\bar\tau - 1))} e^{\i (- \theta - (2\pi + \bar \alpha) (\bar\tau - 1)))},
\end{align*}
We remark that the change of variables we have used is
\begin{equation*}
	\phi=\phi_0+\bar\theta-\bar\alpha+(2\pi+\bar\alpha)\tau, \quad \hat \tau = \frac{\hat \theta}{2\pi + \bar \alpha}.
	\end{equation*}
In particular we have denoted by $\bar\tau$ that value corresponding to $\bar\phi$.
\begin{equation*}
\Delta \hat Q \cdot e^{\i (\bar \phi - (2\pi + \bar \alpha) + \hat \theta))} = \delta \rho(\bar \tau - 2 + \hat \tau) e^{\bar c (2\pi + \bar \alpha) (\bar \tau - 2 + \hat \tau)} \cos(\bar \alpha) - \cos(- \theta - (2\pi + \bar \alpha) (\bar \tau - 2) - \hat \theta),
\end{equation*}
\begin{equation*}
\Delta \hat Q \cdot e^{\i (\bar \phi - (2\pi + \bar \alpha) + (\bar \alpha + \hat \theta - \frac{\pi}{2}))} = \delta \rho(\bar \tau - 2 + \hat \tau) e^{\bar c (2\pi + \bar \alpha) (\bar \tau - 2 + \hat \tau)} \cos \bigg( \frac{\pi}{2} - 2 \bar \alpha \bigg) - \cos \bigg(- \theta - (2\pi + \bar \alpha) (\bar\tau - 2) - \bigg( \bar \alpha + \hat \theta - \frac{\pi}{2} \bigg) \bigg),
\end{equation*}
\begin{align*}
\delta \hat L 
&= \delta |P_2 - P_1| - \frac{r(\phi_2)}{\sin(\bar \alpha)} \delta \phi_2 + \int_0^{(2\pi + \bar \alpha)(\bar\tau - 2 + \hat \tau)} \frac{\delta \tilde r(\phi)}{\sin(\bar \alpha)} d\phi \\
\big[ (\ref{Equa:delta_P_2P_1_segm}) \big] \quad &= - 1 + \frac{1 - \cos(\theta)}{\sin(\bar \alpha)^2} + \int_0^{(2\pi + \bar \alpha)(\bar\tau - 2 + \hat \tau)} \frac{\delta \tilde r(\phi)}{\sin(\bar \alpha)} d\phi \\
&= - 1 + \frac{1 - \cos(\theta)}{\sin^2(\bar \alpha)} e^{\cot(\bar \alpha)(2\pi + \bar \alpha)(\bar\tau - 2 + \hat \tau)},
\end{align*}
\begin{equation*}
\delta \hat r_0 = \delta \rho(\bar \tau - 1) e^{\bar c (2\pi + \bar \alpha) (\bar \tau - 1)}.
\end{equation*}
The interval of interest is
\begin{equation*}
2 - \frac{\min\{\theta,\hat \theta\}}{2\pi + \bar \alpha} \leq \bar \tau < 2, \quad \theta \in (0,\pi].
\end{equation*}
In Fig. \ref{Fig:segm_segm_after_neg5_1} we numerically plot the function $\frac{\delta \check r(\bar \tau) e^{-\bar c(2\pi + \bar \alpha) \bar \tau}}{\theta^2}$ as a function of $\bar \tau = (1- \varsigma)(1- \frac{\min\{\theta,\hat \theta\}}{2\pi +\bar\alpha}) + 2\varsigma$, $\varsigma \in (0,1)$: its minimal value is for $\varsigma = 0,\theta = \pi/2$ and $> 0.59$ (Fig. \ref{Fig:segm_segm_after_neg6_2}).

\begin{figure}
\begin{subfigure}{.475\textwidth}
\resizebox{\linewidth}{!}{\includegraphics{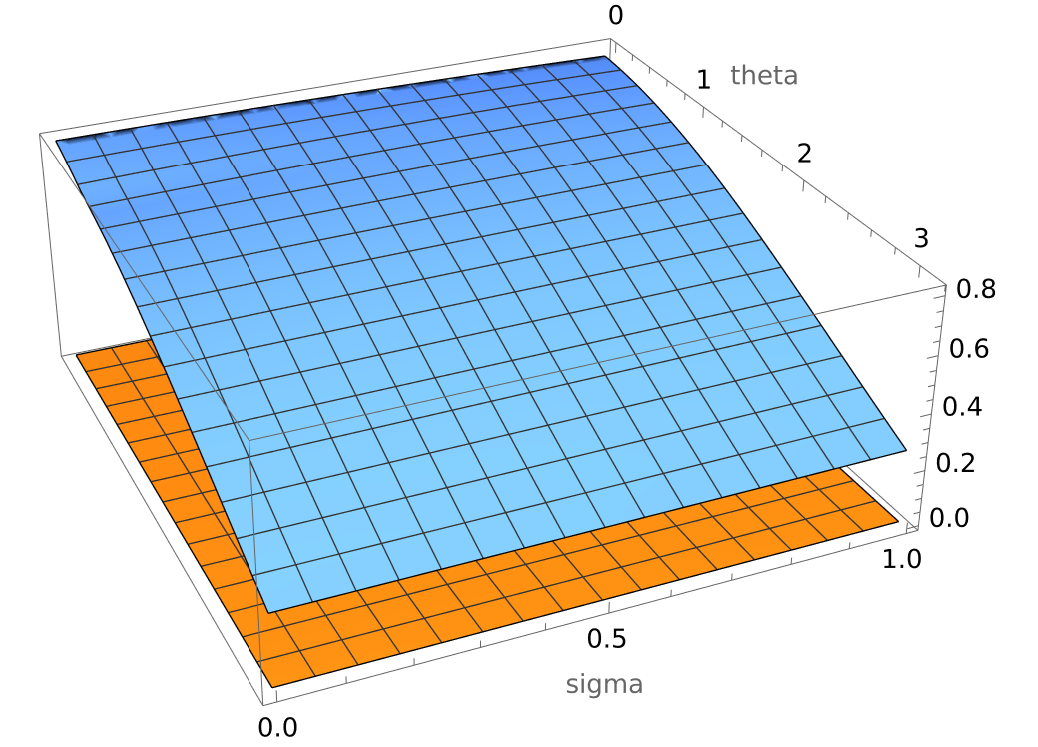}}
\caption{Plot of the function $\delta \check r(\bar \phi)/\theta^2$.}
\label{Fig:segm_segm_after_neg5_1}
\end{subfigure}
\hfill
\begin{subfigure}{.475\linewidth}
\resizebox{\linewidth}{!}{\includegraphics{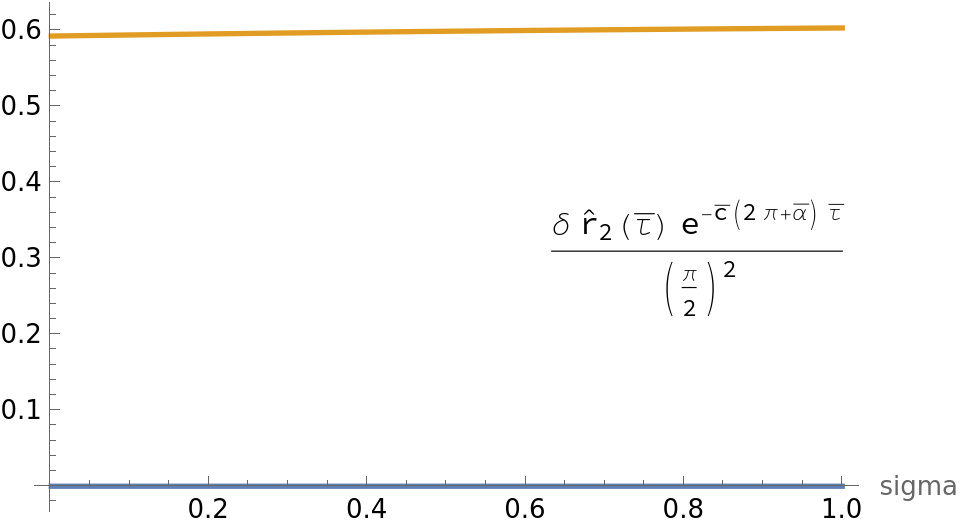}}
\caption{Plot of the function $\delta \check r(\bar \phi)/\theta^2$ for $\theta = \frac{\pi}{2}$.}
\label{Fig:segm_segm_after_neg6_2}
\end{subfigure}
\caption{Analysis of Section \ref{Sss:segm_semg_after_neg_5}.}
\label{Fig:segm_tent_3}
\end{figure}

We thus have proved the following lemma.

\begin{lemma}
\label{Lem:segm_tent_5}
If $\check \phi_0 \geq \phi_0 + 2\pi + \bar \theta - \theta, \check \phi_2 \geq \phi_0 + 2\pi + \bar \theta$, then $\delta \hat r(\bar \phi) \geq 0.59 \theta^2$.
\end{lemma}

%
%
%
%

\subsection{Segment-segment with \texorpdfstring{$\tau - 2 \geq 0$}{tau - 2 geq 0}: see Fig. \ref{Fig:segm_segm_sat_6}}
\label{Sss:segm_semg_after_neg_6}

In this region the base of the tent is a convex curve, because $\delta \tilde r \geq 0$: we can thus apply Lemma \ref{Lem:final_value_satu}, we need just the positivity of
\begin{equation}
\label{Equa:rho_7_segm_segm}
\frac{e^{-\bar c \bar \phi}}{\theta^2} \bigg( \delta \tilde r(\bar \phi) - \frac{1 - \cos(\hat \theta)}{\cos(\bar \alpha - \hat \theta) - \cos(\bar \alpha)} \delta \tilde r(\hat \phi_0) \bigg) = \frac{1}{\theta^2} \bigg( \delta \rho(\bar \tau) - \frac{(1 - \cos(\hat \theta)) e^{- \bar c (2\pi + \bar \alpha) 2}}{\cos(\bar \alpha - \hat \theta) - \cos(\bar \alpha)} \delta \rho(\bar \tau - 2) \bigg), \quad \bar \tau > 2.
\end{equation}
A numerical plot is in Fig. \ref{Fig:segm_tent_7}, where one observe that is is positive and strictly increasing.

\begin{figure}
\begin{subfigure}{.475\textwidth}
\resizebox{\linewidth}{!}{\includegraphics{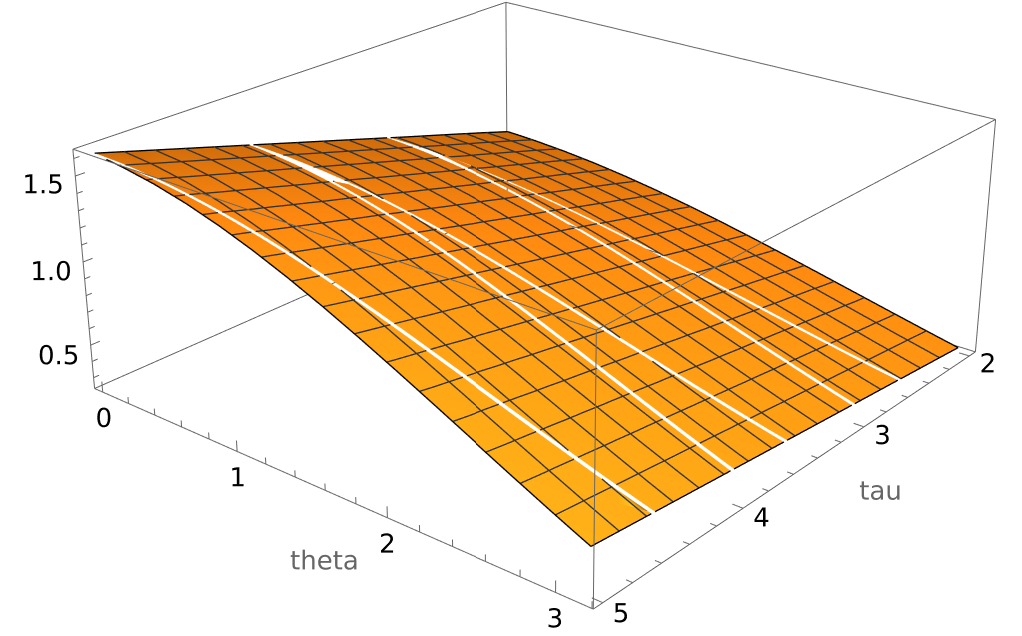}}
\caption{Plot of the function \eqref{Equa:rho_7_segm_segm}.}
\label{Fig:segm_segm_after_neg7_1}
\end{subfigure}
\hfill
\begin{subfigure}{.475\linewidth}
\resizebox{\linewidth}{!}{\includegraphics{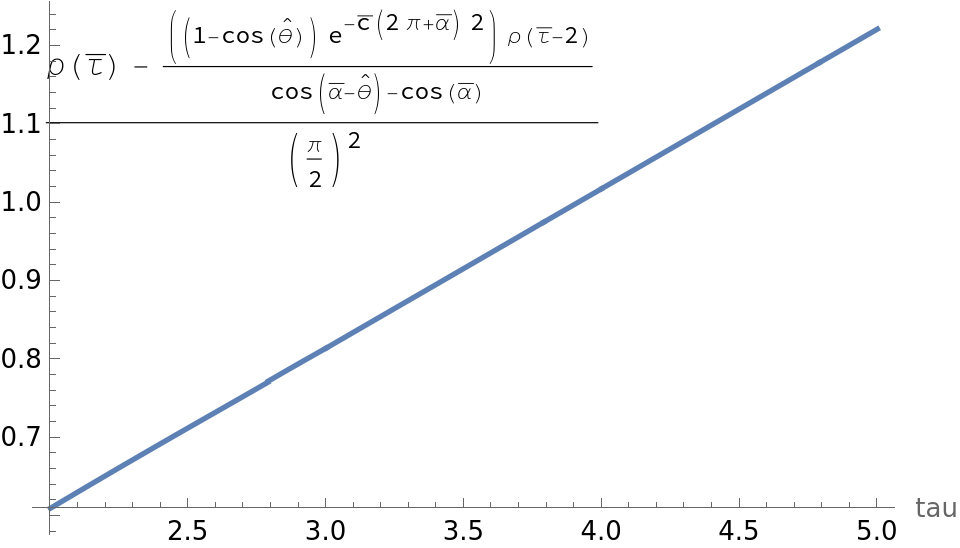}}
\caption{Plot of the function \eqref{Equa:rho_7_segm_segm} for $\theta = \frac{\pi}{2}$.}
\label{Fig:segm_segm_after_neg7_2}
\end{subfigure}
\caption{Analysis of Section \ref{Sss:segm_semg_after_neg_6}.}
\label{Fig:segm_tent_7}
\end{figure}

We thus conclude that

\begin{lemma}
\label{Lem:segm_tent_6}
If $\bar \tau \geq 2$, then $\delta \hat r(\bar \phi) \geq 0$.
\end{lemma}

\newpage

\section{Study of the perturbation to the optimal solution in the presence of an arc}
\label{S:arc_1}

In this section we address the case when the optimal candidate we are going to perturb contains a level set arc. More precisely, four situations can occur:
\begin{enumerate}
\item the optimal candidate spiral $\check r$ starts with a segment at $\phi_0$ in the direction $\bar \phi$, and exit the negativity region in the arc case, Section \ref{Sss:segm_arc_2} and Fig. \ref{Fig:segm_arc_2};
\item the optimal candidate spiral $\check r$ starts with an arc, then a segment at $\phi = \bar \phi - 2\pi - \frac{\pi}{2}$ in the direction $\bar \phi$, and exits the negativity region in the arc case, Section \ref{Sss:segm_arc_3} and Fig. \ref{Fig:segm_arc_3};
\item the optimal candidate spiral $\check r$ starts with an arc, then a segment at $\phi = \bar \phi - 2\pi - \frac{\pi}{2}$ in the direction $\bar \phi$, and exits the negativity region in the segment case, Section \ref{Sss:arc_segm_1} and Fi. \ref{Fig:arc_segm_1};
\item the optimal candidate spiral $\check r$ starts with an arc, then a segment and a saturated spiral, next at $\phi = \bar \phi - 2\pi - \frac{\pi}{2}$ we have a tent, Section \ref{Sss:arc_tent_satur}: the 10 different respective positions of the arc and the tent are presented in Fig. \ref{Fig:arc_segm_2}.
\end{enumerate}
By Proposition \ref{Prop:tent_admissible}, only the last case can occur for $\bar \tau \geq 1$. In particular, when we are in the saturated region for $\hat \phi_0$ (the angle at which the tent starts), the tent is admissible, so that the following situation is not happening: the optimal candidate starts with an arc, then a segment and a saturated spiral, next at $\phi = \bar \phi - 2\pi - \bar \alpha$ there is a segment in the direction $\bar \phi$ and the last round starts with an arc.

\subsection{Segment-arc or arc-arc cases}
\label{Ss:arc-arc}

These situations can happen only when $\bar \phi - 2\pi - \phi_0 \leq \omega$, i.e. $\bar \tau < 1$ by Proposition \ref{Prop:tent_admissible}.

%
%
%
%

\subsubsection{Segment-arc with $\bar \phi \in \phi_0 + 2\pi + [\bar \theta - \theta,\bar \theta)$: see Fig. \ref{Fig:segm_arc_2}}
\label{Sss:segm_arc_2}

\begin{figure}
\begin{subfigure}{.475\textwidth}
\resizebox{\textwidth}{!}{\input{segm_arc_2.pdf_t}}
\caption{The geometric situation of Section \ref{Sss:segm_arc_2}.}
\label{Fig:segm_arc_2}
\end{subfigure}
\hfill
\begin{subfigure}{.475\textwidth}
\resizebox{\textwidth}{!}{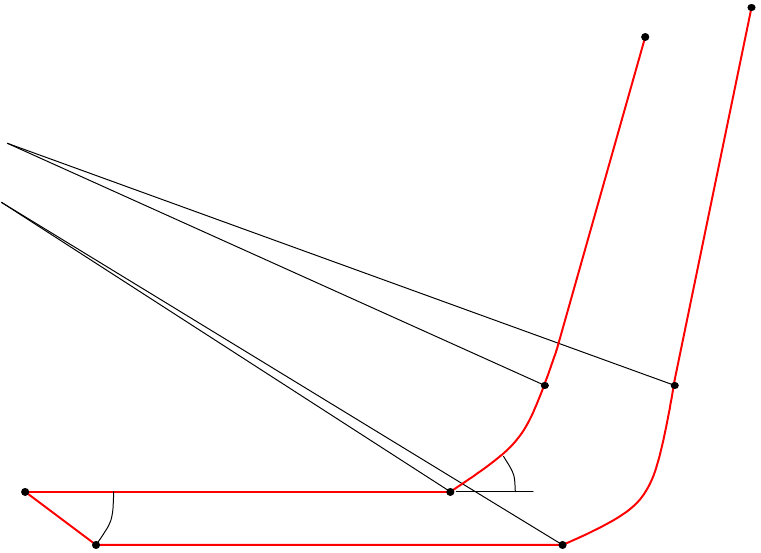}
\caption{The geometry to compute the variations of the points in the arc case.}
\label{Fig:segm_arc_2_2}
\end{subfigure}
\caption{Analysis of Section \ref{Sss:segm_arc_2}.}
\label{Fig:segm_arc_2_tot}
\end{figure}

Setting for simplicity $\bar \phi = 0$ and with $\theta$ the angle of the perturbation of the point $P_0$, the  variations of the positions of the points is (see Fig. \ref{Fig:segm_arc_2_2}):
\begin{equation*}
\delta P_0 = e^{- \i \theta}, \quad \delta \check P^- = \delta P_0 + \delta \check \ell_0 e^{\i \bar\phi}=\delta P_0 + \delta \check \ell_0 e^{\i 0}, \quad \delta \check \phi^- = \frac{\delta \check P^- \cdot e^{\i \check \theta}}{\check r(\check \phi^-)},
\end{equation*}
\begin{equation*}
\delta \check P_1 = \big[ \delta \check P^- \cdot e^{\i (\check \theta - \frac{\pi}{2})} \big] e^{\i (\check \theta + \Delta \check \phi - \frac{\pi}{2})} + \check r(\check \phi_1) \delta \check \phi_1 e^{\i (\check \theta + \Delta \check \phi)}, \quad \delta \check \phi_1 = \frac{\delta \check P^- \cdot e^{\i \check \theta}}{\check r(\check \phi^-)} + \delta \Delta \check \phi,
\end{equation*}
where the first term is the variation of the level set and the second is the variation of the angle along the level set. Using Formula \eqref{eq:P:1:arc} we have used the fact that $\check\phi^-=-\frac{\pi}{2}+\check\theta-2\pi$. Moreover we have
\begin{equation*}
\delta \check P_2 = \delta \check P_1 + \delta \check \ell_1 e^{\i (\check \theta + \Delta \check \phi)} + \delta \check \phi_1 \check \ell_1 e^{i (\check \theta + \Delta \check \phi + \frac{\pi}{2})}.
\end{equation*}
For reference compare the previous formulas with \eqref{eq:P:1:arc},\eqref{eq:P:2:arc}.
The optimality condition requires by Lemma \ref{Lem:study_g_1_arc} that
\begin{equation*}
\check \theta = h_1(\Delta \check \phi), \quad \delta \check \theta = \delta \check \phi^- = \frac{dh_1(\Delta \check \phi)}{d\Delta \check\phi} \delta \Delta \check \phi,
\end{equation*}
and the saturation condition reads as
\begin{equation*}
\begin{split}
0 &= \delta \check P_2 \cdot e^{\i (\check \theta + \Delta \check \phi - \bar \alpha)} - \cos(\bar \alpha) \bigg( 1 + \delta \check \ell_0 - \check r(\check \phi^-) \delta \check \phi^- + \Delta \check \phi \big( \delta \check P^- \cdot e^{\i (\check \theta - \frac{\pi}{2})} \big) + \check r(\check \phi_1) \delta \check \phi_1 + \delta \check \ell_1 \bigg) \\
&= \big( \delta \check P_2 \cdot e^{\i (\check \theta + \Delta \check \phi)} \cos(\bar \alpha) + \delta \check P_2 \cdot e^{\i (\check \theta + \Delta \check \phi - \frac{\pi}{2})} \sin(\bar \alpha) \big) \\
& \quad \qquad \qquad \qquad - \cos(\bar \alpha) \bigg( 1 + \delta \check \ell_0 - \check r(\check \phi^-) \delta \check \phi^- + \Delta \check \phi \big( \delta \check P^- \cdot e^{\i (\check \theta - \frac{\pi}{2})} \big) + \check r(\check \phi_1) \delta \check \phi_1 + \delta \check \ell_1 \bigg) \\
&= \big( (\hat r_1(\hat \phi_1) \delta \phi_1 + \delta \hat \ell_1) \cos(\bar \alpha) + \delta \check P_2 \cdot e^{\i (\check \theta + \Delta \check \phi - \frac{\pi}{2})} \sin(\bar \alpha) \big) \\
&\quad \qquad \qquad \qquad - \cos(\bar \alpha) \bigg( 1 + \delta \check \ell_0 - \check r(\check \phi^-) \delta \check \phi^- + \Delta \check \phi \big( \delta \check P^- \cdot e^{\i (\check \theta - \frac{\pi}{2})} \big) + \check r(\check \phi_1) \delta \check \phi_1 + \delta \check \ell_1 \bigg) \\
&= \delta \check P_2 \cdot e^{\i (\check \theta + \Delta \check \phi - \frac{\pi}{2})} \sin(\bar \alpha) - \cos(\bar \alpha) \bigg( 1 + \delta \check \ell_0 - \check r(\check \phi^-) \delta \check \phi^- + \Delta \check \phi \big( \delta \check P^- \cdot e^{\i (\check \theta - \frac{\pi}{2})} \big) \bigg).
\end{split}
\end{equation*}
%
The final value is computed by
\begin{equation*}
\begin{split}
\delta \hat r(\bar \phi) &= \frac{\delta \check P_2 \cdot e^{\i (\check \theta + \Delta \check \phi - \frac{\pi}{2})}}{\sin(\bar \alpha)} e^{\cot(\bar \alpha) (2\pi + \bar \alpha - \check \theta - \Delta \check \phi)} - e^{\cot(\bar \alpha) \theta} - \delta \check \ell_0 \\
&= \frac{\cot(\bar \alpha)}{\sin(\bar \alpha)} \bigg( 1 + \delta \check \ell_0 - \big( \cos(\theta + \check \theta) + \delta \check \ell_0 \cos(\check \theta) \big) + \Delta \check \phi \big( \sin(\theta + \check \theta) + \delta \check \ell_0 \sin(\check \theta) \big) \bigg) e^{\cot(\bar \alpha) (2\pi + \bar \alpha - \check \theta - \Delta \check \phi)} \\&- e^{\cot(\bar \alpha) \theta} - \delta \check\ell_0.
\end{split}
\end{equation*}
Using \eqref{Equa:theta_Deltaphi_rel} we simplify into
\begin{equation*}
\begin{split}
\delta \hat r(\bar \phi) &= \frac{\cot(\bar \alpha)}{\sin(\bar \alpha)} \big( 1 - \cos(\theta + \check \theta) + \Delta \check \phi  \sin(\theta + \check \theta) \big) e^{\cot(\bar \alpha) (2\pi + \bar \alpha - \check \theta - \Delta \check \phi)} - e^{\cot(\bar \alpha) \theta} \\
&= e^{\cot(\bar \alpha) \theta} \bigg( \frac{\cot(\bar \alpha)}{\sin(\bar \alpha)} \big( 1 - \cos(\theta + \check \theta) + \Delta \check \phi \sin(\theta + \check \theta) \big) e^{\cot(\bar \alpha) (2\pi + \bar \alpha - (\check \theta + \theta) - \Delta \check \phi)} - 1 \bigg),
\end{split}
\end{equation*}
which is $\geq 0$ unless $\theta = 0$ by Lemma \ref{Lem:study_g_1_arc}.

We thus have the following result.

\begin{lemma}
\label{Lem:segm_arc_2}
If $\bar \phi \in \phi_0 + 2\pi + [\frac{\pi}{2} - \theta,\frac{\pi}{2})$, then $\delta \check r(\bar \phi) \geq 0$.
\end{lemma}

\subsubsection{Arc-arc with $2\pi + \frac{\pi}{2} \leq \bar \phi - \phi_0 \leq 2\pi + \frac{\pi}{2} + \Delta \phi$: see Fig. \ref{Fig:segm_arc_3}}
\label{Sss:segm_arc_3}

\begin{figure}
\resizebox{.75\textwidth}{!}{\input{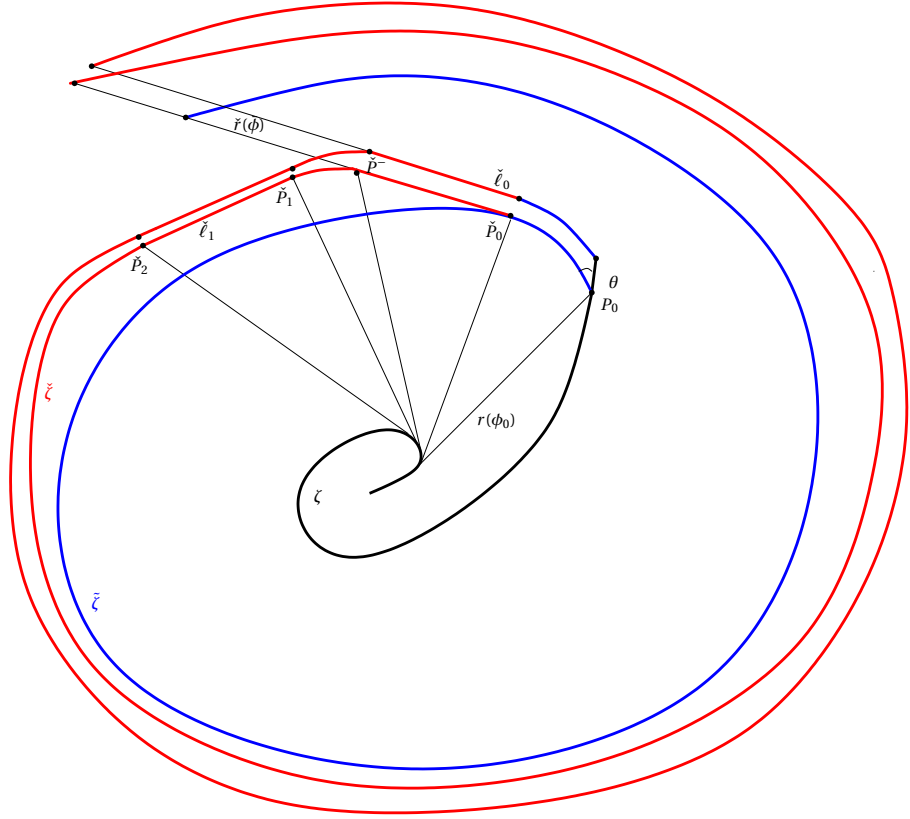}}
\caption{The geometric situation of Section \ref{Sss:segm_arc_3}.}
\label{Fig:segm_arc_3}
\end{figure}

The angle $\Delta \phi$ is the angle corresponding to the arc of the level set for the fastest saturated spiral $\tilde \zeta$. We first observe that the angle at the point $\check P_2$ is
\begin{equation}
\label{Equa:max_arc_arc}
\begin{split}
\check \phi_2 - \phi_0 &\leq \tan(\bar \alpha) + \check \theta + \Delta \check \phi + \frac{\pi}{2} - \bar \alpha \\
&\leq \tan(\bar \alpha) + \max \big\{ \Delta \check \phi + h_1(\Delta \check \phi) \big\} + \frac{\pi}{2} - \bar \alpha < 5.4 < 2\pi + \frac{\pi}{2} - \theta,
\end{split}
\end{equation}
so that the end point is always inside the first round. We thus have
\begin{equation*}
\delta P_0 = e^{\i (\phi_0 + \frac{\pi}{2} - \theta)}, \quad \delta \check P_0 = \sin(\theta) e^{\i (\bar \phi - 2\pi - \frac{\pi}{2})},
\end{equation*}
\begin{equation*}
\delta \check P^- = \delta \check \ell_0 e^{\i (\bar \phi - 2\pi)} + \sin(\theta) e^{\i (\bar \phi - 2\pi - \frac{\pi}{2})}, \quad \delta \check P_1 = \check r_1 \delta \check \omega e^{\i (\bar \phi - 2\pi + \check \omega)} +  \delta \check P^- \cdot e^{\i (\bar \phi - 2\pi + \check \theta - \frac{\pi}{2})} e^{\i (\bar \phi - 2\pi + \check \omega - \frac{\pi}{2})},
\end{equation*}
\begin{equation*}
\delta \check P_2 = \delta \check P_1 + \delta \check \ell_1 e^{\i (\bar \phi + \check \omega)} - \check \ell_1 \delta \check \omega e^{\i (\bar \phi + \check \omega - \frac{\pi}{2})}.
\end{equation*}
Using $\Delta \phi = \bar \phi - 2\pi - \phi_0 - \frac{\pi}{2}$ and denoting
$$
\check \omega = \check \theta + \Delta \check \phi = h_1(\Delta \check \phi) + \Delta \check \phi,
$$
the saturation condition is
\begin{align*}
0 &= \delta \check P_2 \cdot e^{\i(\bar \phi -2 \pi + \check \omega - \bar \alpha)} - \cos(\bar \alpha) \Big( (1 - \cos(\theta) + \sin(\theta) \Delta \phi) + \delta \check \ell_0 - \check r^- \delta \check \phi^- + \delta \check P^- \cdot e^{\i (\bar \phi - 2\pi + \check \theta - \frac{\pi}{2})} \Delta \check \phi + \check r_1 \delta \check \phi_1 + \delta \check \ell_1 \Big) \\
&= \delta \check P_1 \cdot e^{\i (\bar \phi - 2\pi + \check \omega - \bar \alpha)} - \check \ell_1 \delta \check \omega \sin(\bar \alpha) \\
& \quad - \cos(\bar \alpha) \Big( (1 - \cos(\theta) + \sin(\theta) \Delta \phi) + \delta \check \ell_0 - \delta \check P^- \cdot e^{i(\bar \phi - 2\pi + \check \theta)} + \delta \check P^- \cdot e^{\i (\bar \phi - 2\pi + \check \theta - \frac{\pi}{2})} \Delta \check \phi + \delta \check P_1 \cdot e^{i (\bar \phi - 2\pi + \check \omega)} \Big) \\
&= \delta \check P^- \cdot e^{\i (\bar \phi - 2\pi + \check \theta - \frac{\pi}{2})} \sin(\bar \alpha) - \check \ell_1 \delta \check \omega \sin(\bar \alpha) \\
&\quad - \cos(\bar \alpha) \Big( \big( 1 - \cos(\theta) + \sin(\theta) \Delta \phi \big) + \sin(\theta) \sin(\check \theta) + \sin(\theta) \cos(\check \theta) \Delta \check \phi + \delta \check \ell_0 \big( 1 - \cos(\check \theta) + \sin(\check \theta) \Delta \check \phi \big) \Big). \\
\end{align*}
The final value is
\begin{equation}
\label{Equa:final_arc_arc}
\begin{split}
\delta \check r(\bar \phi) &= \frac{\delta \check P_2 \cdot e^{\i (\bar \phi - 2\pi + \check \omega - \frac{\pi}{2})}}{\sin(\bar \alpha)} e^{\cot(\bar \alpha) (2\pi + \bar \alpha - \check \omega)} - e^{\cot(\bar \alpha)(\Delta \phi + \theta)} + \cos(\theta) e^{\cot(\bar \alpha) \Delta \phi} - \sin(\theta) \frac{e^{\cot(\bar \alpha) \Delta \phi} - 1}{\cot(\bar \alpha)} - \delta \check \ell_0 \\
&= \frac{\delta \check P^- \cdot e^{\i (\bar \phi - 2\pi + \check \theta - \frac{\pi}{2})} - \check \ell_1 \delta \check \omega}{\sin(\bar \alpha)} e^{\cot(\bar \alpha) (2\pi + \bar \alpha - \check \omega)} - e^{\cot(\bar \alpha)(\Delta \phi + \theta)} + \cos(\theta) e^{\cot(\bar \alpha) \Delta \phi} - \sin(\theta) \frac{e^{\cot(\bar \alpha) \Delta \phi} - 1}{\cot(\bar \alpha)} - \delta \check \ell_0 \\
&= \frac{\cot(\bar \alpha)}{\sin(\bar \alpha)} e^{\cot(\bar \alpha) (2\pi + \bar \alpha - \check \omega)} \Big( \big( 1 - \cos(\theta) + \sin(\theta) \Delta \phi \big) + \delta \check \ell_0 - \delta \check P^- \cdot e^{i (\bar \phi - 2\pi + \check \theta)} + \delta \check P^- \cdot e^{\i (\bar \phi - 2\pi + \check \theta - \frac{\pi}{2})} \Delta \check \phi \Big) \\
& \quad - e^{\cot(\bar \alpha)(\Delta \phi + \theta)} + \cos(\theta) e^{\cot(\bar \alpha) \Delta \phi} - \sin(\theta) \frac{e^{\cot(\bar \alpha) \Delta \phi} - 1}{\cot(\bar \alpha)} - \delta \check \ell_0 \\
&= \frac{\cot(\bar \alpha)}{\sin(\bar \alpha)} e^{\cot(\bar \alpha) (2\pi + \bar \alpha - \check \omega)} \Big( \big( 1 - \cos(\theta) + \sin(\theta) \Delta \phi \big) \\
& \quad \qquad \qquad \qquad \qquad \qquad \qquad + \sin(\theta) \sin(\check \theta) + \sin(\theta) \cos(\check \theta) \Delta \check \phi + \delta \check \ell_0 \big( 1 - \cos(\check \theta) + \sin(\check \theta) \Delta \check \phi \big) \Big) \\
& \quad - e^{\cot(\bar \alpha)(\Delta \phi + \theta)} + \cos(\theta) e^{\cot(\bar \alpha) \Delta \phi} - \sin(\theta) \frac{e^{\cot(\bar \alpha) \Delta \phi} - 1}{\cot(\bar \alpha)} - \delta \check \ell_0 \\
&= \frac{1 - \cos(\theta + \check \theta) + \sin(\theta + \check \theta) \Delta \check \phi + \sin(\theta) \Delta \phi}{1 - \cos(\check \theta) + \sin(\check \theta) \Delta \check \phi} - e^{\cot(\bar \alpha)(\Delta \phi + \theta)} + \cos(\theta) (e^{\cot(\bar \alpha) \Delta \phi} - 1) - \sin(\theta) \frac{e^{\cot(\bar \alpha) \Delta \phi} - 1}{\cot(\bar \alpha)} \\
&= \frac{1 - \cos(\theta + \check \theta) + \sin(\theta + \check \theta) \Delta \check \phi + \sin(\theta) \Delta \phi}{1 - \cos(\check \theta) + \sin(\check \theta) \Delta \check \phi} - e^{\cot(\bar \alpha)(\Delta \phi + \theta)} + \cos(\theta + \bar \alpha) \frac{e^{\cot(\bar \alpha) \Delta \phi} - 1}{\cos(\bar \alpha)},
\end{split}
\end{equation}
where we have used \eqref{Equa:theta_Deltaphi_rel}, and we rely on the observation that the final angle is before $2\pi + \frac{\pi}{2} - \theta$. The second derivative w.r.t. $\Delta \phi$ is
\begin{equation*}
\partial^2_{\Delta \phi} \delta \check r(\bar \phi) = - \frac{\cot(\bar \alpha)}{\sin(\bar \alpha)} e^{\cot(\bar \alpha)(\Delta \phi + \theta)} \big( \cos(\bar \alpha) - \cos(\bar \alpha + \theta) e^{-\cot(\bar \alpha) \theta} \big) \leq 0,
\end{equation*}
so that the minimal values are for $\Delta \phi = 0,\tan(\bar \alpha)$. The first case corresponds to the previous section, while for $\Delta \phi = \tan(\bar \alpha)$ we plot the function  $\delta \check r(\bar \phi)/\theta$ in Fig. \ref{Fig:arcarc_2_2} in a neighborhood of the curve determined by \eqref{Equa:theta_Deltaphi_rel}: the function is positive $>4.5 \theta$.

%
%

\begin{figure}
\begin{subfigure}{.475\textwidth}
\resizebox{\textwidth}{!}{\includegraphics{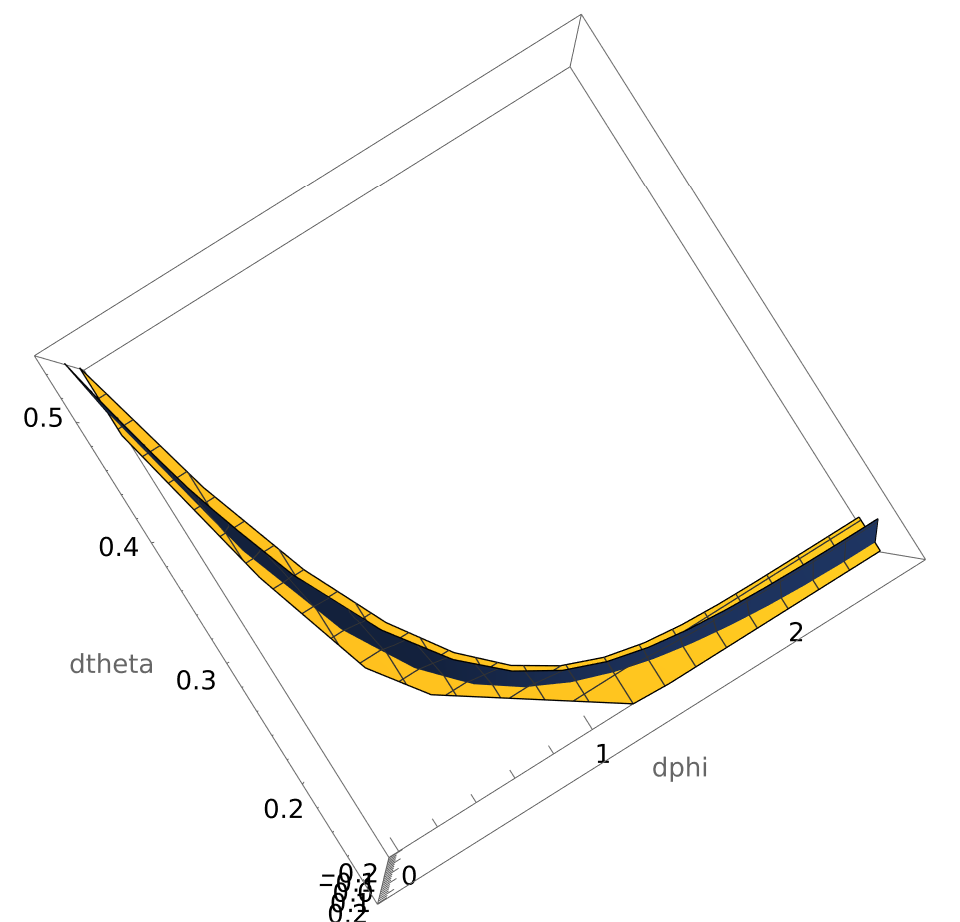}}
\caption{A small neighborhood of the region $\frac{\cot(\bar \alpha)}{\sin(\bar \alpha)} (1 - \cos(\check \theta) + \sin(\check \theta) \Delta \check \phi) e^{\cot(\bar \alpha) (2\pi + \bar \alpha - \check \theta - \Delta \check \phi)} = 1$: in orange the plane $\R^2 \times \{0\}$ and in blue the previous function.}
\label{Fig:arcarc_2_1}
\end{subfigure} \hfill
\begin{subfigure}{.475\textwidth}
\resizebox{\textwidth}{!}{\includegraphics{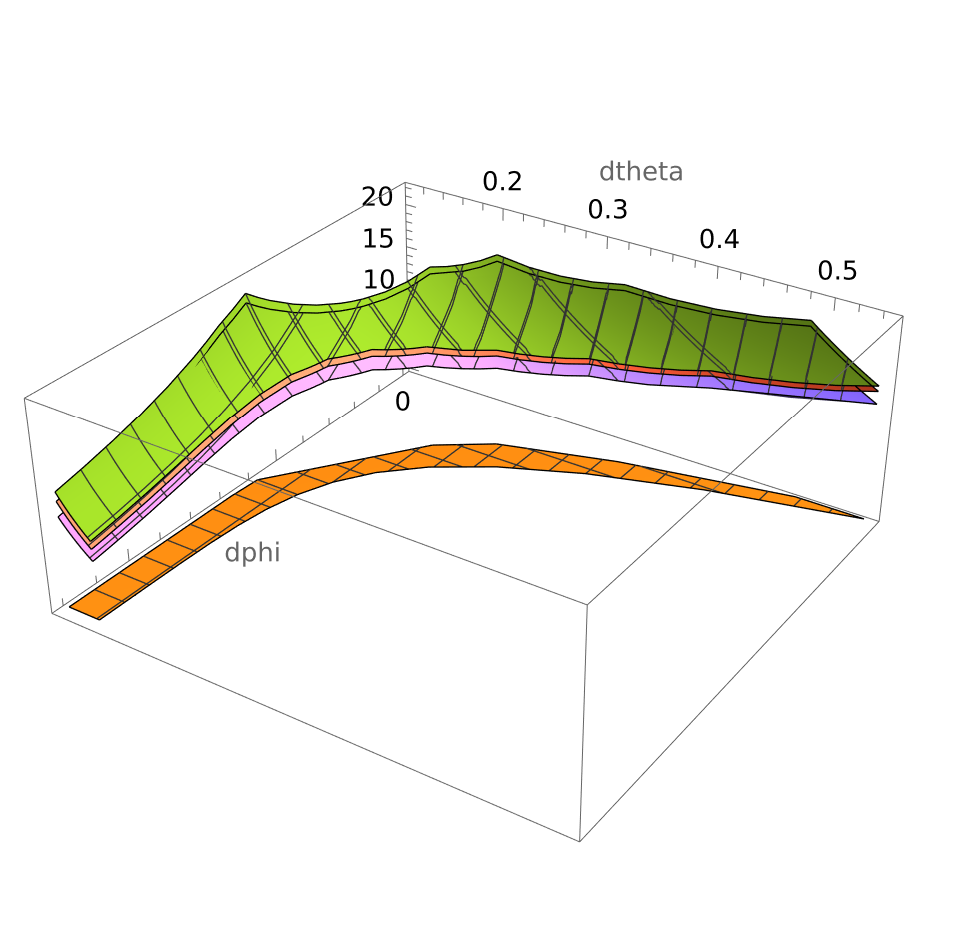}}
\caption{The function $\delta \check r(\bar \phi)/\theta$ for various values of $\theta \in (0,\pi/2]$ in the same neighborhood.}
\label{Fig:arcarc_2_2}
\end{subfigure}
\caption{Analysis of the final value of $\delta \check r(\bar \phi)$, Section  \ref{Sss:segm_arc_3}.}
\label{Fig:arcarc_2}
\end{figure}

\begin{lemma}
\label{Lem:segm_arc_3}
In this geometric setting it holds $\delta \check r(\bar \phi) \geq 4.5 \theta$.
\end{lemma}

%
%

\subsection{Arc-segment case}
\label{Ss:arc_segment}

The situation is when the initial spiral $\tilde r$ starts with an arc, and later $\check r$ closes with a tent. This is the unique situation in the saturated region. There are several cases depending on the position of the segment w.r.t. the arc.

\subsubsection{The tent starts inside the arc: Fig. \ref{Fig:arc_segm_1}}
\label{Sss:arc_segm_1}

\begin{figure}
\begin{subfigure}{.475\textwidth}
\resizebox{\textwidth}{!}{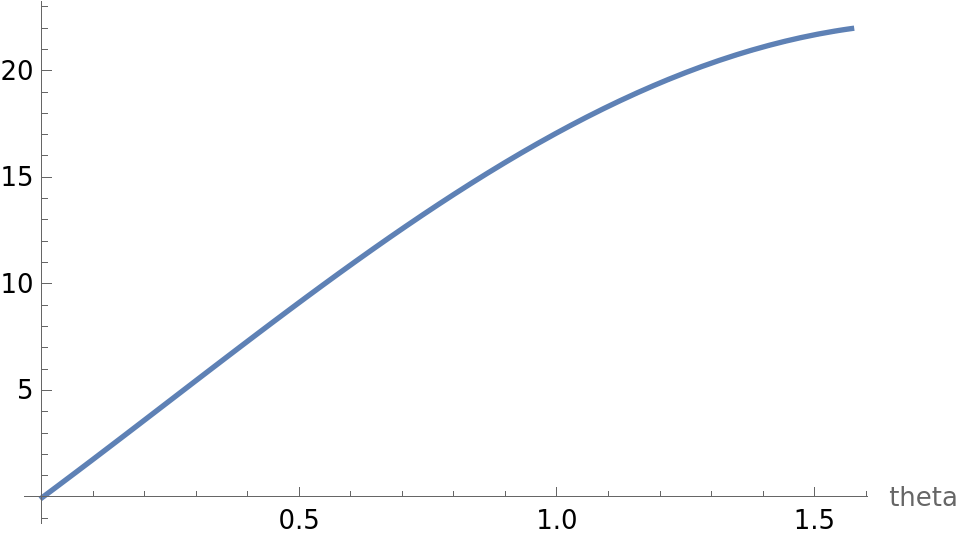}
\caption{The geometric situation of Section \ref{Sss:arc_segm_1}.}
\label{Fig:arc_segm_1}
\end{subfigure} \hfill
\begin{subfigure}{.475\textwidth}
\resizebox{\textwidth}{!}{\includegraphics{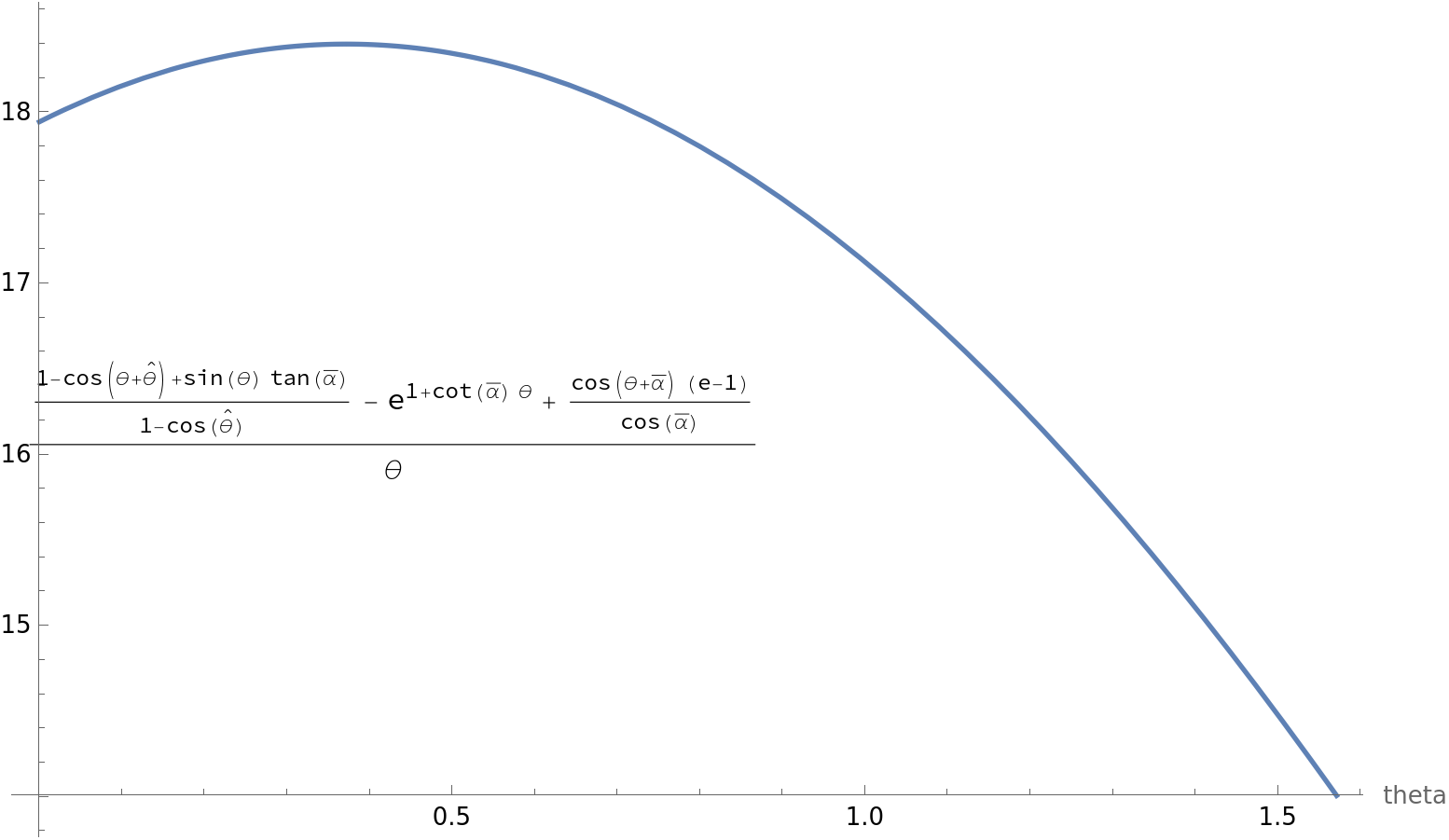}}
\caption{The plot of the function \eqref{Equa:arc_segm_1} divided by $\theta$.}
\label{Fig:arc_segm_1_1}
\end{subfigure}
\caption{Numerical analysis of Section \ref{Sss:arc_segm_1}.}
\label{Fig:arc_segm_1_1_1}
\end{figure}

In this case the analysis is the same as in the arc-arc case, by setting $\Delta \check \phi = 0$ and $\check \theta = \hat \theta$: as before we are left with the case $\Delta \phi = \tan(\bar \alpha)$, thus the final value is
\begin{align}
\label{Equa:arc_segm_1}
\delta \check r(\bar \phi) &= \frac{1 - \cos(\theta + \hat \theta) + \sin(\theta) \tan(\bar \alpha)}{1 - \cos(\hat \theta)} - e^{1 + \cot(\bar \alpha) \theta} + \frac{\cos(\theta + \bar \alpha)}{\cos(\bar \alpha)} (e-1).
\end{align}
The plot of this function is in Fig. \ref{Fig:arc_segm_1_1}.

\begin{lemma}
\label{Lem:arc_segm_1}
In the above situation the perturbation is positive $> 14 \theta$.
\end{lemma}

\subsubsection{Arc-tent case with the final tent in the saturated region}
\label{Sss:arc_tent_satur}

The 10 cases representing the relative positions of the initial arc and the final tent are represented in Fig. \ref{Fig:arcarc_2}.

\begin{figure}
\resizebox{.75\textwidth}{!}{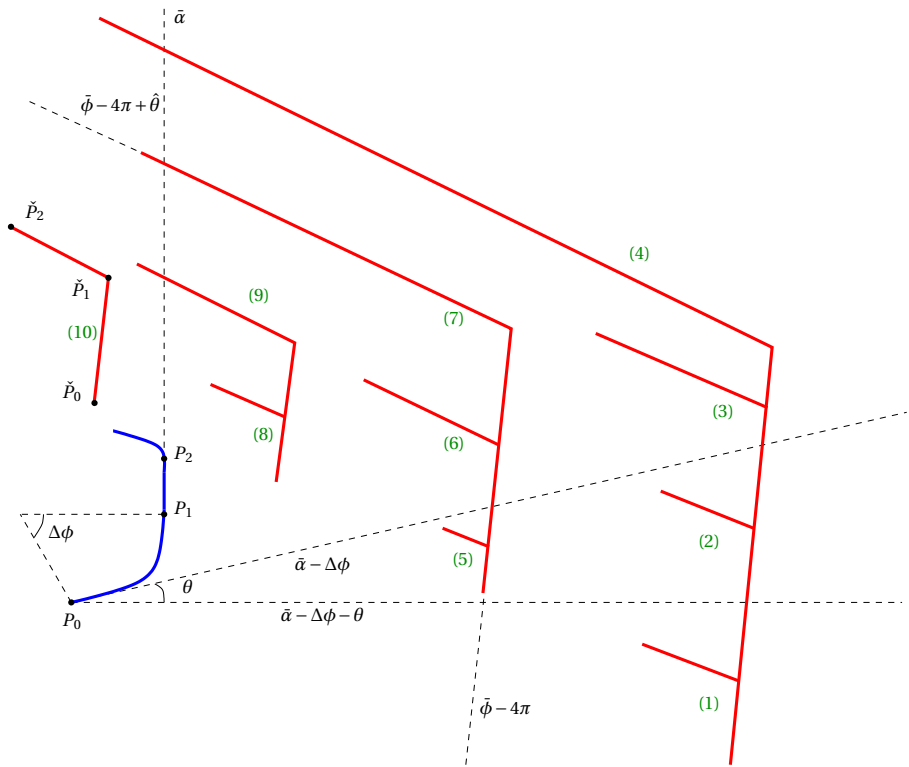}
\caption{The 10 possible relative positions of the arc and the tent in the saturated region: in green the number of the case considered in Section \ref{Sss:arc_tent_satur}.}
\label{Fig:arc_segm_2}
\end{figure}

We will use the tent formulas of Section \ref{Ss:comput_tent_new}.

\begin{description}
\item[Case 1: $\bar \phi - 2(2\pi - \bar \alpha) - \phi_0 \in (\omega - \bar \alpha,2\pi + \frac{\pi}{2} - \theta - \hat \theta)$] in this case $\delta\hat Q_0 =0$, $\delta\hat Q_2=0$, $\delta \hat L=0$, so that (see formulas in Lemma \ref{Lem:relation_admissibility_tent})
\begin{equation*}
\hat \ell_0 = \delta \tilde r_0 \frac{1 - \cos(\hat \theta)}{\cos(\bar \alpha - \hat \theta) - \cos(\bar \alpha)},
\end{equation*}
\begin{equation}
\label{Equa:check_worst_arc_segm_case_1}
\delta \hat r_2 - \delta \tilde r_2 = \delta \tilde r_0 \bigg( \frac{\cos(\bar \alpha) - \cos(\bar \alpha + \hat \theta)}{\cos(\bar \alpha - \hat \theta) - \cos(\bar \alpha)} - e^{\cot(\bar \alpha) \hat \theta)} \bigg) > 0.00669 \delta \tilde r_0 > 0,
\end{equation}
and we thus just estimate
\begin{align*}
\delta \hat r(\bar \phi) &= \big( \delta \hat r_2 - \delta \tilde r_2) e^{\cot(\bar \alpha)(2\pi + \bar \alpha - \hat \theta)} + \delta \tilde r(\bar \phi) - \hat \ell_0 \\
&> \delta \tilde r(\bar \phi) - \delta \tilde r_0 \frac{1 - \cos(\hat \theta)}{\cos(\bar \alpha - \hat \theta) - \cos(\bar \alpha)}.
\end{align*}
A numerical evaluation of the above function is in Fig. \ref{Fig:arccase1} as a function of $\bar \tau = (1 - \varsigma) + \varsigma (2 - \frac{\theta + \Delta \phi + \hat \theta}{2\pi + \bar \alpha})$, $\varsigma \in (0,1)$, where one sees that it is larger than $.35 e^{\bar c \phi} \theta ( \theta + \Delta \phi)$.

\begin{figure}
\begin{subfigure}{.475\textwidth}
\resizebox{\textwidth}{!}{\includegraphics{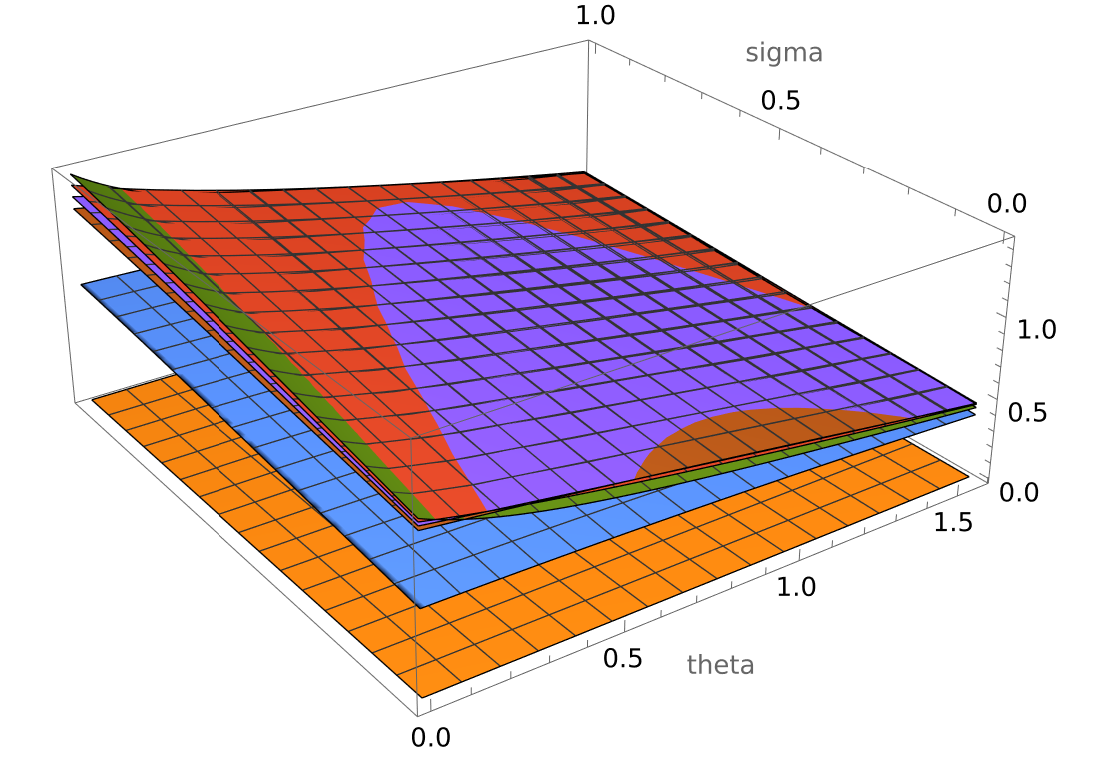}}
\end{subfigure} \hfill
\begin{subfigure}{.475\textwidth}
\resizebox{\textwidth}{!}{\includegraphics{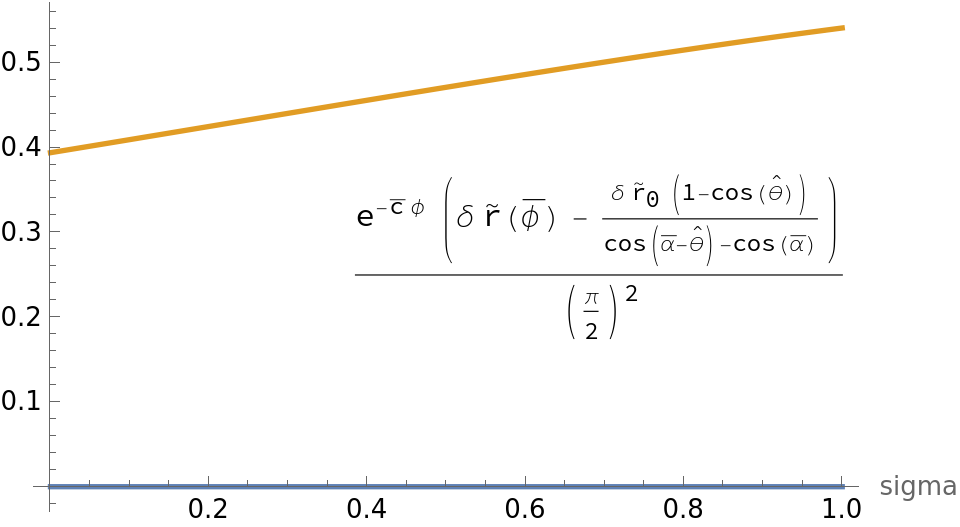}}
\end{subfigure}
\caption{Numerical plot of the function $e^{-\bar c \phi} (\delta \tilde r(\bar \phi) - \delta \tilde r_0 \frac{1 - \cos(\hat \theta)}{\cos(\bar \alpha - \hat \theta) - \cos(\bar \alpha)})$ and its worst case for $\theta = \frac{\pi}{2},\Delta \phi = 0$, Case 1 of Section \ref{Sss:arc_tent_satur}.}
\label{Fig:arccase1}
\end{figure}


\item[Case 2: $\bar \tau \in (1,1+\tau_1), \bar \tau + \hat \tau - 1 \in (\tau_1,\tau_2)$] here the end point of the tent is in the region $(\tau_1,\tau_2) = (1 - \frac{\theta + \Delta \phi}{2\pi + \bar \alpha},1 - \frac{\Delta \phi}{2\pi + \bar \alpha})$, while the initial point $\check P_0$ is in the saturated arc before, so that using $\tau = 0$ as the reference direction for computing the angles (i.e. $\phi_0 + \frac{\pi}{2} + \Delta \phi - \bar \alpha$ is mapped to $0$), we have:
\begin{equation*}
\bar \phi - 4 \pi - \bar \alpha < \bar \alpha - \Delta \phi - \theta, \quad \bar \alpha - \Delta \phi - \theta < \bar \phi - 4 \pi - \bar \alpha + \hat \theta < \bar \alpha - \Delta \phi,
\end{equation*}
\begin{equation*}
\delta \Delta Q = e^{\i(\bar \alpha - \Delta \phi - \theta)}, \quad \delta \Delta Q \cdot e^{\i (2\pi + \bar \alpha)(\bar \tau - 1 + \hat \tau)} = \cos \big( (2\pi + \bar \alpha)(\bar\tau - 2) + \Delta \phi + \theta + \hat \theta \big),
\end{equation*}
\begin{equation*}
\delta \Delta Q \cdot e^{\i (2\pi + \bar \alpha)(\bar \tau - 1 + \hat \tau) + \bar \alpha - \frac{\pi}{2}} = \cos \bigg( (2\pi + \bar \alpha)(\bar\tau - 2) + \Delta \phi + \theta + \hat \theta + \bar \alpha - \frac{\pi}{2} \bigg),
\end{equation*}
\begin{equation*}
\delta \hat L = 1, \quad \delta \hat r_0 = \delta \tilde r(\bar \phi - (2\pi + \bar \alpha)).
\end{equation*}
Recall that
\begin{equation*}
\hat \tau = \frac{\hat \theta}{2\pi + \bar \alpha}.
\end{equation*}
The plot of the function
\begin{equation}
\label{Equa:final_case_2}
\frac{e^{- \bar c (2\pi + \bar \alpha) \bar \tau} \delta \check r(\bar \phi)}{\theta(\theta + \Delta \phi)}, \quad \bar \tau = (1 - \varsigma) \bigg( 2 - \frac{\theta + \hat \theta + \Delta \phi}{2\pi + \bar \alpha} \bigg) + \varsigma \bigg( 2 - \frac{\max\{\theta,\hat \theta\} + \Delta \phi}{2\pi + \bar \alpha} \bigg), \ \varsigma \in (0,1),
\end{equation}
computed according to the formulas of Lemma \ref{Lem:relation_admissibility_tent}, is in Fig. \ref{Fig:PTarc_case2_1}, which is strictly positive:
\begin{equation*}
\frac{e^{- \bar c (2\pi + \bar \alpha) \bar \tau} \delta \hat r(\bar \phi)}{\theta(\theta + \Delta \phi)} \geq 0.54.
\end{equation*}

\begin{figure}
\begin{subfigure}{.475\textwidth}
\resizebox{\textwidth}{!}{\includegraphics{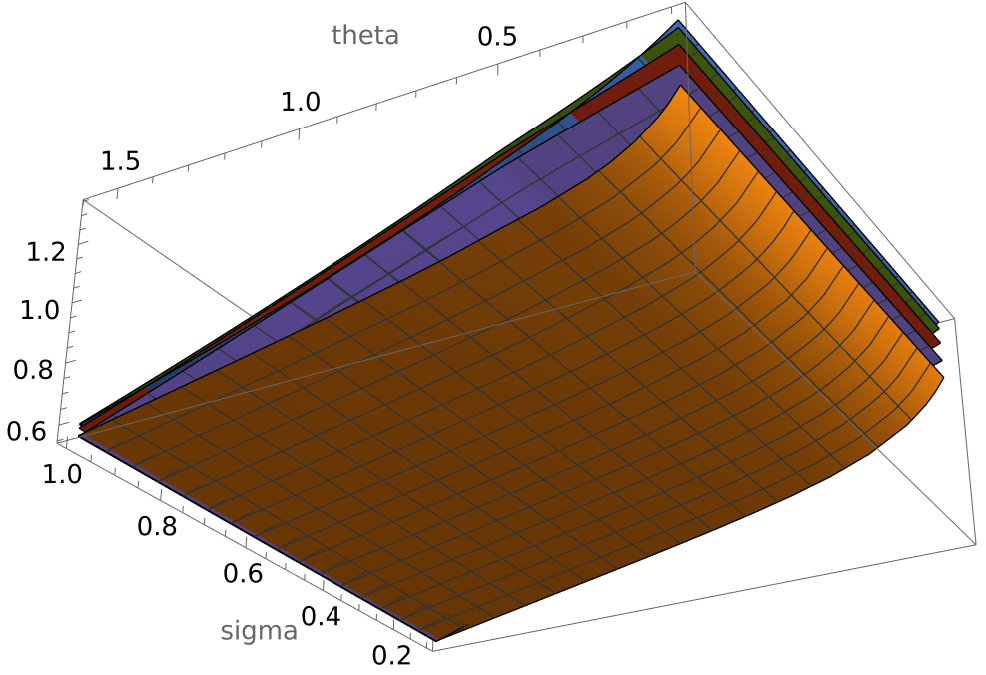}}
\end{subfigure} \hfill
\begin{subfigure}{.475\textwidth}
\resizebox{\textwidth}{!}{\includegraphics{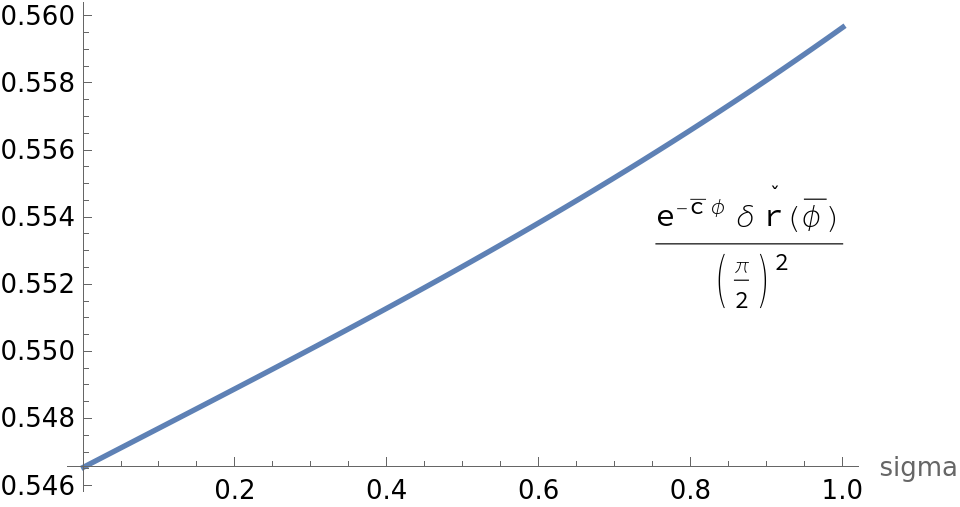}}
\end{subfigure}
\caption{Final value for Case 2 and of its minimum, (\ref{Equa:final_case_2}): the orange surface corresponds to $\Delta \phi = 0$ and the minimum to $\Delta \phi = 0, \theta = \frac{\pi}{2}$.}
\label{Fig:PTarc_case2_1}
\end{figure}

\item[Case 3: $\bar \tau \in (1,1+\tau_1), \bar \tau + \hat \tau - 1 \in (\tau_2,1)$] here the initial point of the tent is before $\tau_1 = 1 - \frac{\theta+\Delta \phi}{2\pi + \bar \alpha}$, and the final point is in $(\tau_2,1) = (1 - \frac{\theta}{2\pi + \bar \alpha},1)$. As before setting $\tau = 0$ as reference direction, we have thus
\begin{equation*}
\bar \phi - 4\pi - \bar \alpha < \bar \alpha - \Delta \phi - \theta, \quad \bar \alpha - \Delta \phi < \bar \phi - 4 \pi - \bar \alpha + \hat \theta < \bar \alpha, \quad \theta < \hat \theta,
\end{equation*}
\begin{equation*}
\delta \Delta Q = \sin(\theta) e^{\i (\hat \theta - \frac{\pi}{2})}, \quad \delta \hat L = 1 - \cos(\theta) + \sin(\theta) \bigg( \bar \phi - 4\pi - \bar \alpha + \hat \theta + \Delta \phi - \bar \alpha \bigg), \quad \delta \hat r_0 = \delta \tilde r(\bar \phi - (2\pi + \bar \alpha)).
\end{equation*}
The plot of the function
\begin{equation}
\label{Equa:case3_funct_fin}
\frac{e^{- \bar c (2\pi + \bar \alpha) \bar \tau} \delta \check r(\bar \phi)}{\theta(\theta + \Delta \phi)}, \quad \bar\tau = (1-\varsigma) \bigg( 2 - \frac{\hat \theta + \Delta \phi}{2\pi + \bar \alpha} \bigg) + \varsigma \bigg( 2 - \frac{\max\{\theta + \Delta \phi,\hat \theta\}}{2\pi + \bar \alpha} \bigg), \ \varsigma \in (0,1),
\end{equation}
is in Fig. \ref{Fig:PTarc_case3_1}, showing that
$$
e^{- \bar c (2\pi + \bar \alpha) \bar \tau} \delta \check r(\bar \phi) > 0.74 \theta(\theta + \Delta \phi).
$$

\begin{figure}
\begin{subfigure}{.475\textwidth}
\resizebox{\textwidth}{!}{\includegraphics{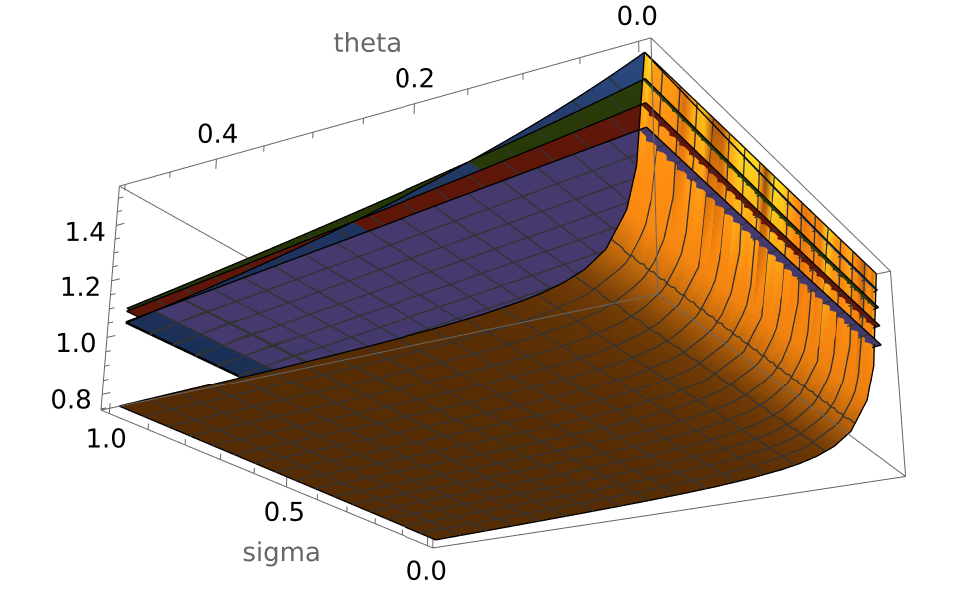}}
\end{subfigure} \hfill
\begin{subfigure}{.475\textwidth}
\resizebox{\textwidth}{!}{\includegraphics{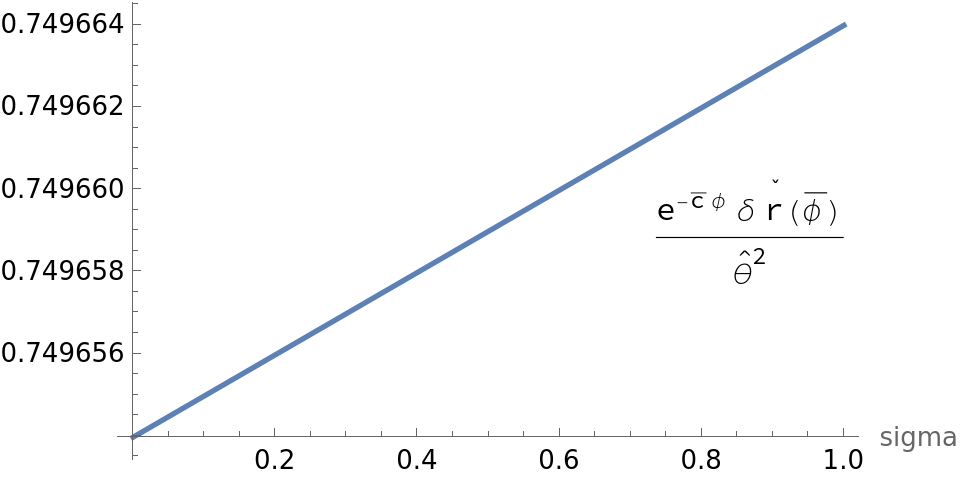}}
\end{subfigure}
\caption{Final value for Case 3, Equation (\ref{Equa:case3_funct_fin}): the orange surface corresponds to $\Delta \phi = 0$, and the minimal case is for $\Delta \phi = 0,\theta = \hat \theta$ (in the numerical plot we consider $\theta = \hat \theta - 10^{-4}$ to avoid numerical artifacts).}
\label{Fig:PTarc_case3_1}
\end{figure}

\item[Case 4: $\bar \tau \in (1,1+\tau_1), \bar \tau + \hat \tau > 2$] here the initial value is before $\tau_1 = 1 - \frac{\theta + \Delta \phi}{2\pi + \bar \alpha}$ region and the final value is after $\tau = 1$:
\begin{equation*}
\bar \phi - 4\pi - \bar \alpha < \bar \alpha - \Delta \phi - \theta, \quad \bar \phi > 2 (2\pi + \bar \alpha) - \hat \theta, \quad \Delta \phi + \theta < \hat \theta.
\end{equation*}
In particular $\delta\hat Q_0=0$, $\delta \Delta Q=\delta\hat Q_2$ which reads as
\begin{equation*}
\delta \Delta Q = \delta \tilde r(\bar \phi - 2(2\pi + \bar \alpha) + \hat \theta) e^{\i (\bar \phi - 2(2\pi + \bar \alpha) + \hat \theta)},
\end{equation*}
\begin{equation*}
\delta \hat L = \frac{1 - \cos(\theta) + \sin(\theta) \Delta \phi}{\sin(\bar \alpha)^2} e^{\cot(\bar \alpha)(\bar \phi - 2(2\pi + \bar \alpha) + \hat \theta)},
\end{equation*}
\begin{equation*}
\delta \hat r_0 = \delta \tilde r(\bar \phi - (2\pi + \bar \alpha)).
\end{equation*}
The plot of the function
\begin{equation}
\label{Equa:fina_case_4_arcsegm}
\frac{e^{- \bar c (2\pi + \bar \alpha) \bar \tau} \delta \check r(\bar \phi)}{\theta(\theta + \Delta \phi)}, \quad \bar\tau = (1 - \varsigma) \bigg( 2 - \frac{\hat \theta}{2\pi + \bar \alpha} \bigg) + \varsigma \bigg( 2 - \frac{\theta + \Delta \phi}{2\pi + \bar \alpha} \bigg), \ \varsigma \in (0,1),
\end{equation}
is in Fig. \ref{Fig:PTarc_case4_1}, which is strictly positive $> 0.74$.

\begin{figure}
\begin{subfigure}{.475\textwidth}
\resizebox{\textwidth}{!}{\includegraphics{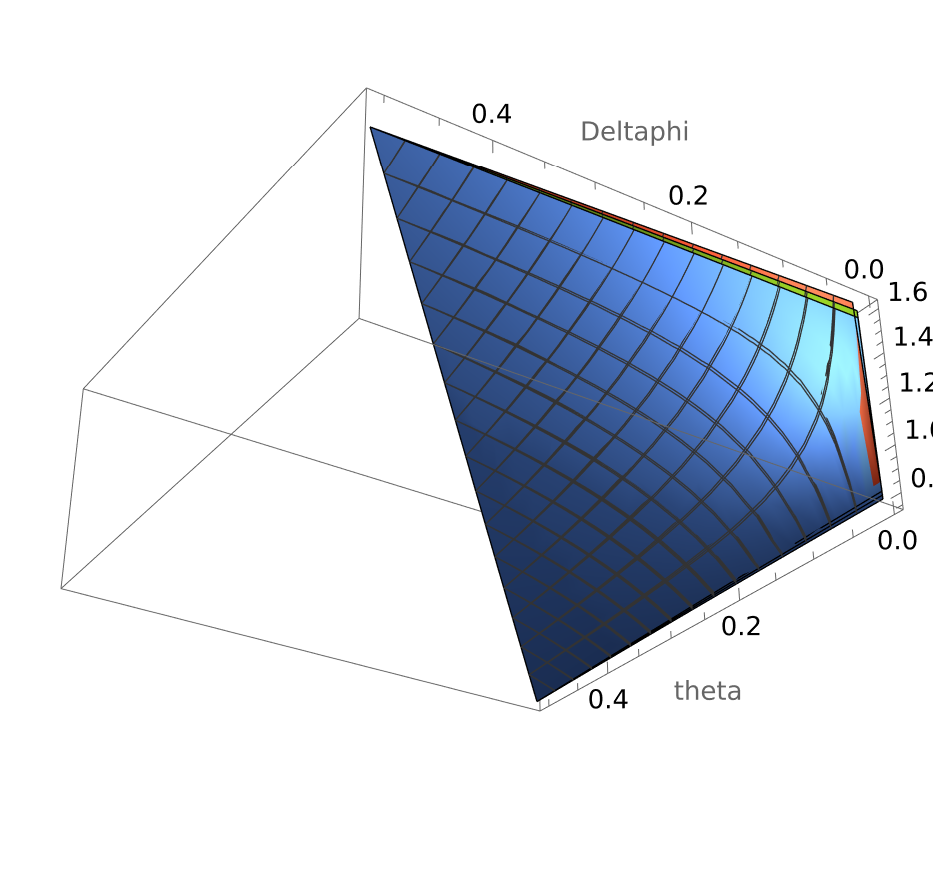}}
\end{subfigure} \hfill
\begin{subfigure}{.475\textwidth}
\resizebox{\textwidth}{!}{\includegraphics{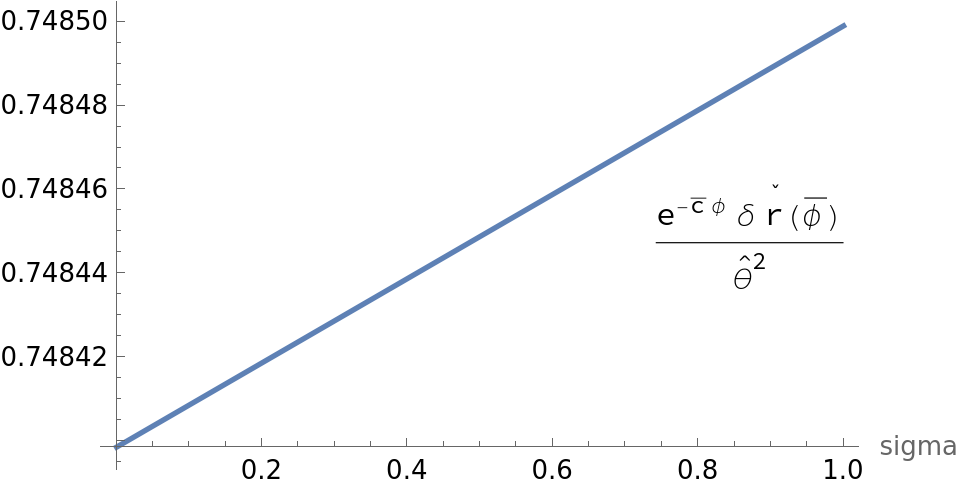}}
\end{subfigure}
\caption{Final value for Case 4, Equation (\ref{Equa:fina_case_4_arcsegm}): the minimal values occur for $\Delta \phi = 0,\theta = 0$.}
\label{Fig:PTarc_case4_1}
\end{figure}

\item[Case 5: $\bar \tau \in (1+\tau_1,1+\tau_2), \bar \tau + \hat \tau \in (1+\tau_1,1+\tau_2)$] here the tent is inside $(\tau_1,\tau_2) = (1 - \frac{\theta + \Delta \phi}{2\pi + \bar \alpha},1 - \frac{\Delta \phi}{2\pi + \bar \alpha})$:
\begin{equation*}
\bar \alpha - \Delta \phi - \theta < \bar \phi - 4\pi - \bar \alpha < \bar \phi - 4\pi - \bar \alpha + \hat \theta < \bar \alpha - \Delta \phi, \quad \theta > \hat \theta,
\end{equation*}
\begin{equation*}
\delta \Delta Q = 0, \quad \delta \hat L = 0, \quad \delta \hat r_0 = \delta \tilde r(\bar \phi - (2\pi + \bar \alpha)).
\end{equation*}
The above condition implies that $\theta \geq \hat \theta$. 
The plot of the function
\begin{equation}
\label{Equa:final_case_5}
\frac{e^{- \bar c (2\pi + \bar \alpha) \bar \tau} \delta \check r(\bar \phi)}{\theta(\theta + \Delta \phi)}, \quad \bar \tau = (1 - \varsigma) \bigg( 2 - \frac{\theta + \Delta \phi}{2\pi + \bar \alpha} \bigg) + \varsigma \bigg( 2 - \frac{\Delta \phi}{2\pi + \bar \alpha} \bigg), \ \varsigma \in (0,1),
\end{equation}
is in Fig. \ref{Fig:PTarc_case5_1}, which is strictly positive $> 0.55$.

\begin{figure}
\begin{subfigure}{.475\textwidth}
\resizebox{\textwidth}{!}{\includegraphics{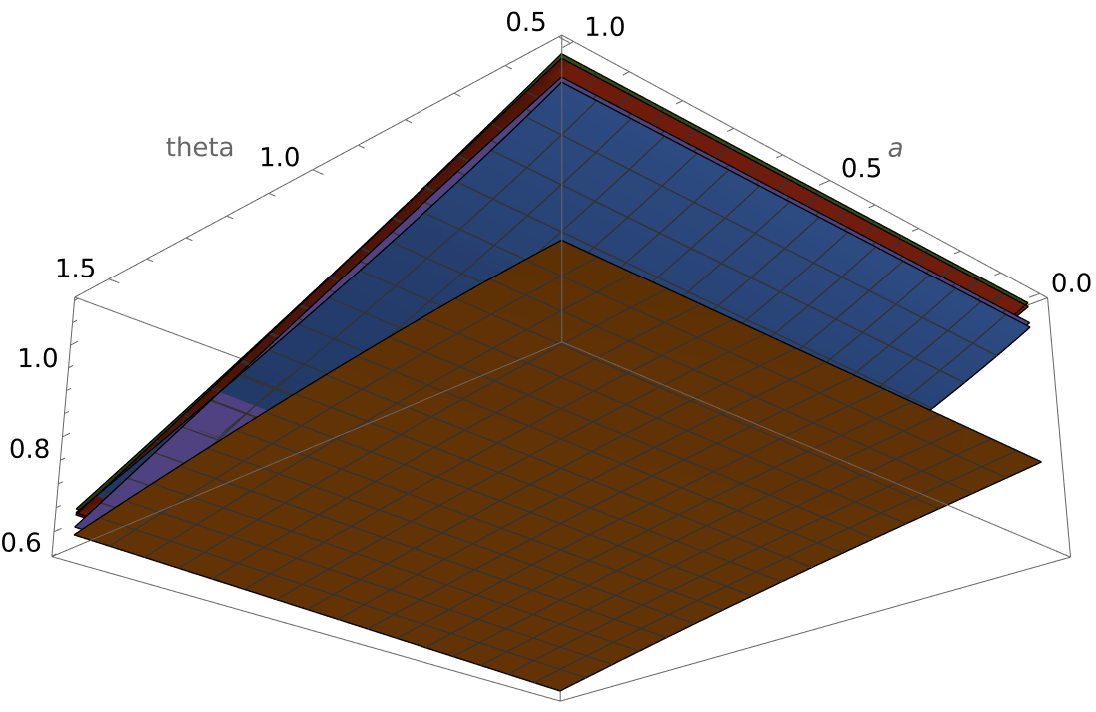}}
\end{subfigure} \hfill
\begin{subfigure}{.475\textwidth}
\resizebox{\textwidth}{!}{\includegraphics{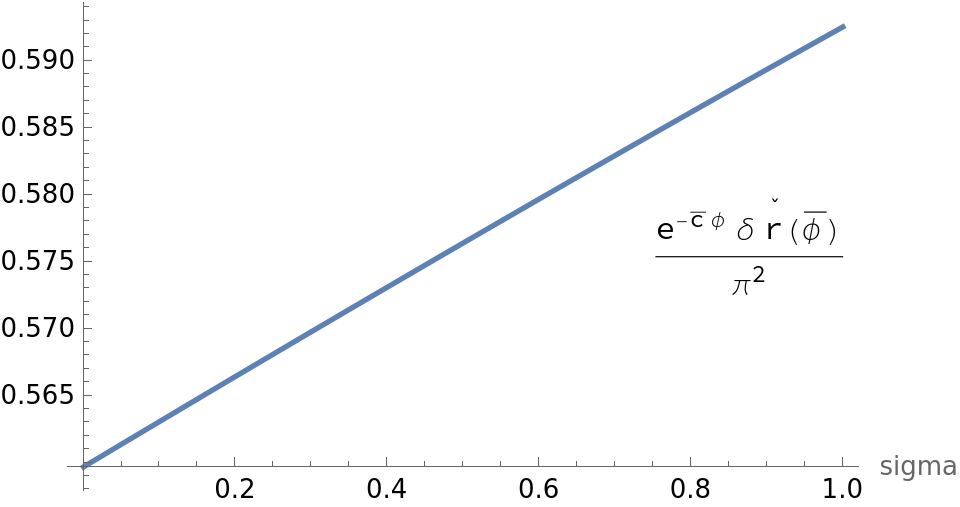}}
\end{subfigure}
\caption{Final value for Case 5, Equation (\ref{Equa:final_case_5}): the minimal values are obtained for $\Delta \phi = 0, \theta = \frac{\pi}{2}$.}
\label{Fig:PTarc_case5_1}
\end{figure}

\item[Case 6: $\bar \tau \in (1+\tau_1,1+\tau_2), \bar \tau + \hat \tau \in (1+\tau_2,1)$] here the tent starts inside $(\tau_1,\tau_2) = (1 - \frac{\theta + \Delta \phi}{2\pi + \bar \alpha},1 - \frac{\theta}{2\pi + \bar \alpha})$ and ends in $(\tau_2,1)$:
\begin{equation*}
\bar \alpha - \Delta \phi - \theta < \bar \phi - 4\pi - \bar \alpha < \bar \alpha - \Delta \phi, \quad \bar \alpha - \Delta \phi < \bar \phi - 4\pi - \bar \alpha + \hat \theta < \bar \alpha, \quad \Delta \phi + \theta > \hat \theta,
\end{equation*}
\begin{equation*}
\delta \Delta Q = \sin(\theta) e^{\i(\bar \phi - 4\pi - \bar \alpha + \hat \theta - \frac{\pi}{2})} - e^{\i(\bar \alpha - \Delta \phi - \theta)}, \quad \delta \hat L = - \cos(\theta) + \sin(\theta) (\bar \phi - 2(2\pi + \bar \alpha) + \hat \theta + \Delta \phi), \quad \delta \hat r_0 = \delta \tilde r(\bar \phi - (2\pi + \bar \alpha)).
\end{equation*}
The above condition implies that $\theta \geq \hat \theta$. The plot of the function
\begin{equation}
\label{Equa:final_case_6}
\frac{e^{- \bar c (2\pi + \bar \alpha) \bar \tau} \delta \check r(\bar \phi)}{\theta}, \quad \bar \tau = (1 - \varsigma) \bigg( 2 - \frac{\min\{\theta,\hat \theta\} + \Delta \phi}{2\pi + \bar \alpha} \bigg) + \varsigma \bigg( 2 - \frac{\max\{\Delta \phi,\hat \theta\}}{2\pi + \bar \alpha} \bigg), \ \varsigma \in (0,1),
\end{equation}
is in Fig. \ref{Fig:PTarc_case6_1}, which is strictly positive $>.38 \theta(\theta + \Delta \phi)$.

\begin{figure}
\begin{subfigure}{.475\textwidth}
\resizebox{\textwidth}{!}{\includegraphics{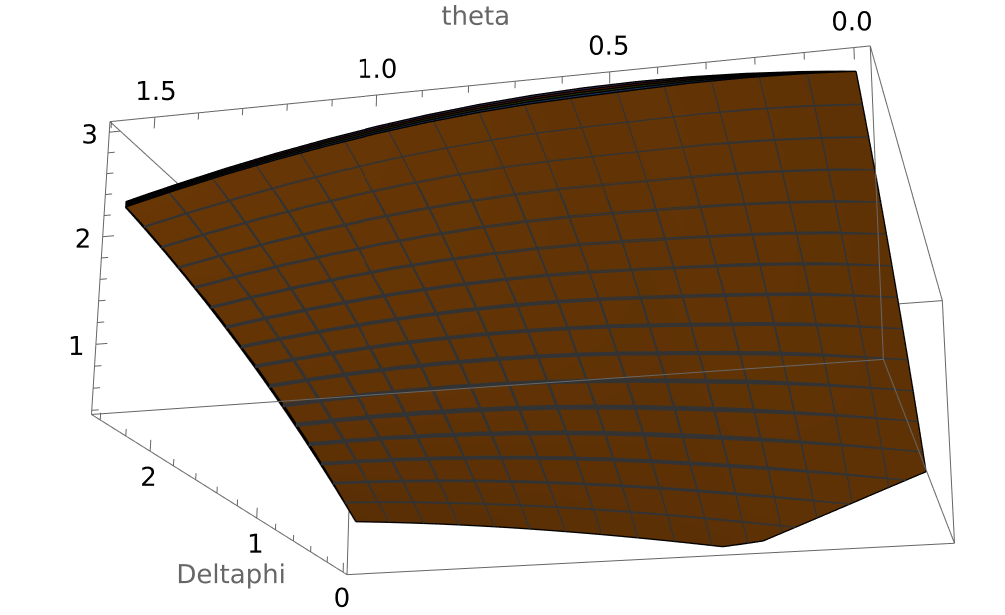}}
\end{subfigure} \hfill
\begin{subfigure}{.475\textwidth}
\resizebox{\textwidth}{!}{\includegraphics{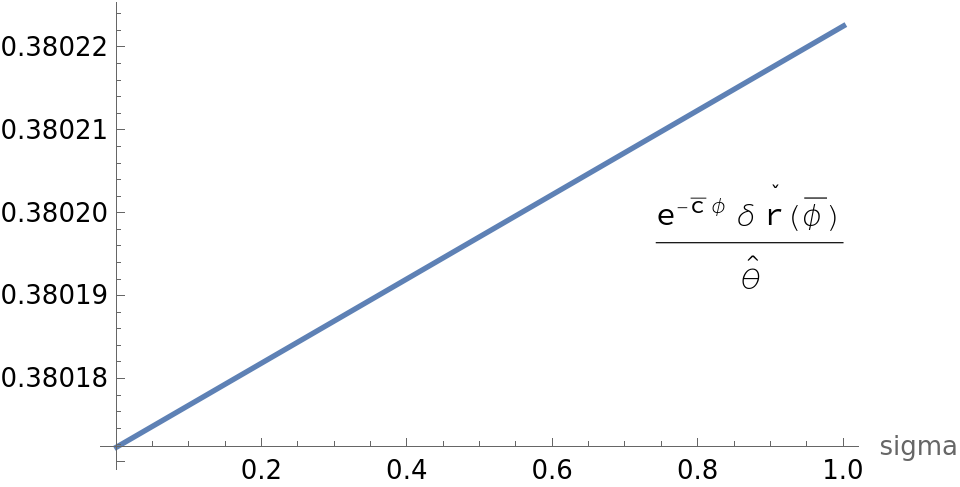}}
\end{subfigure}
\caption{Final value for Case 6, Equation (\ref{Equa:final_case_6}): the minimal values are for $\Delta \phi = 0, \theta = \hat \theta$.}
\label{Fig:PTarc_case6_1}
\end{figure}

\item[Case 7: $\bar \tau \in (1+\tau_1,1+\tau_2), \bar \tau + \hat \tau > 2$] here the tent starts inside ($\tau_1,\tau_2)$ and ends after $\tau = 1$:
\begin{equation*}
\bar \alpha - \Delta \phi - \theta < \bar \phi - 4\pi - \bar \alpha < \bar \alpha - \Delta \phi, \quad \bar \alpha < \bar \phi - 4\pi - \bar \alpha + \hat \theta, \quad \Delta \phi < \hat \theta,
\end{equation*}
\begin{equation*}
\delta \Delta Q = \delta \tilde r(\bar \phi - 2(2\pi+\bar \alpha) + \hat \theta) e^{\i(\bar \phi - 2(2\pi + \bar \alpha) + \hat \theta)} - e^{\i(\bar \alpha - \Delta \phi - \theta)},
\end{equation*}
\begin{equation*}
\delta \hat L = \frac{1 - \cos(\theta) + \sin(\theta) \Delta \phi}{\sin(\bar \alpha)^2} e^{\cot(\bar \alpha)(\bar \phi - 2(2 \pi + \bar \alpha) + \hat \theta)} - 1,
\end{equation*}
\begin{equation*}
\delta \hat r_0 = \delta \tilde r(\bar \phi - (2\pi + \bar \alpha)).
\end{equation*}
The above condition implies that $\Delta \phi < \hat \theta$. The plot of the function
\begin{equation}
\label{Equa:final_7_arcsegm}
\frac{e^{- \bar c (2\pi + \bar \alpha) \bar \tau} \delta \check r(\bar \phi)}{\theta(\theta + \Delta \phi)}, \quad \bar \tau = (1 - \varsigma) \bigg( 2 - \frac{\min\{\Delta \phi + \theta,\hat \theta\}}{2\pi + \bar \alpha} \bigg) + \varsigma \bigg( 2 - \frac{\Delta \phi}{2\pi + \bar \alpha} \bigg), \ \varsigma \in (0,1),
\end{equation}
is in Fig. \ref{Fig:PTarc_case7_1}, which is strictly positive $> 0.28$.

\begin{figure}
\begin{subfigure}{.475\textwidth}
\resizebox{\textwidth}{!}{\includegraphics{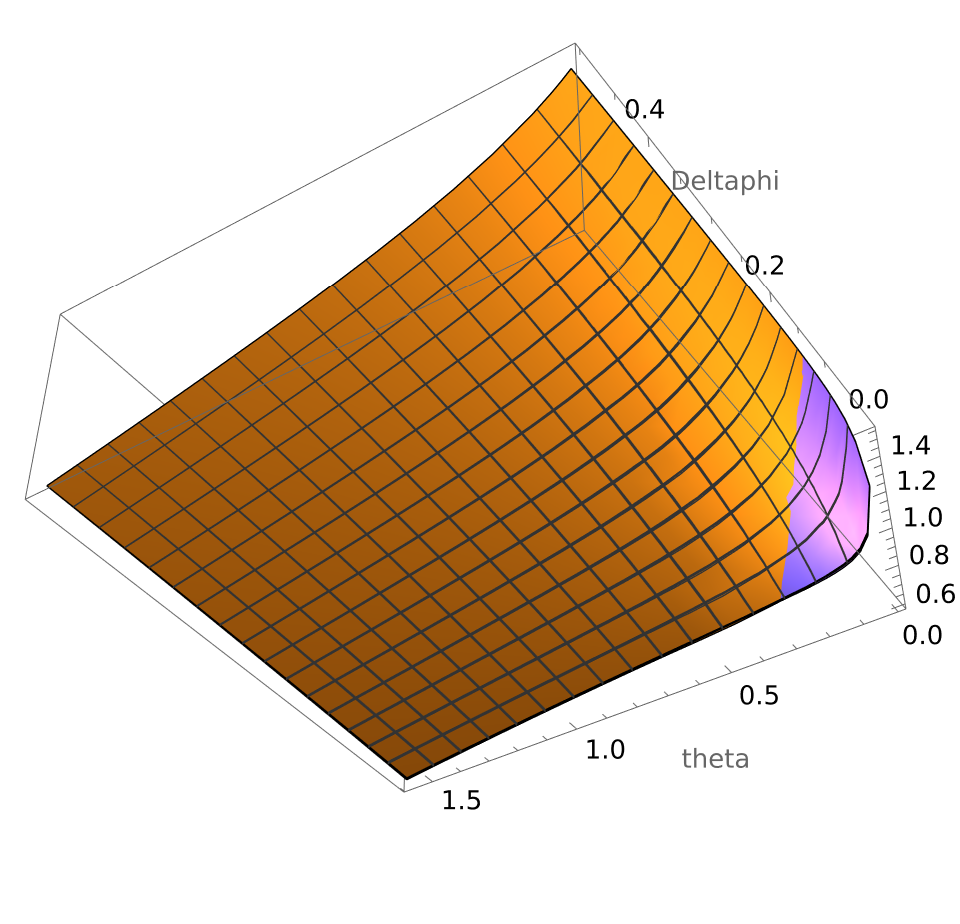}}
\end{subfigure} \hfill
\begin{subfigure}{.475\textwidth}
\resizebox{\textwidth}{!}{\includegraphics{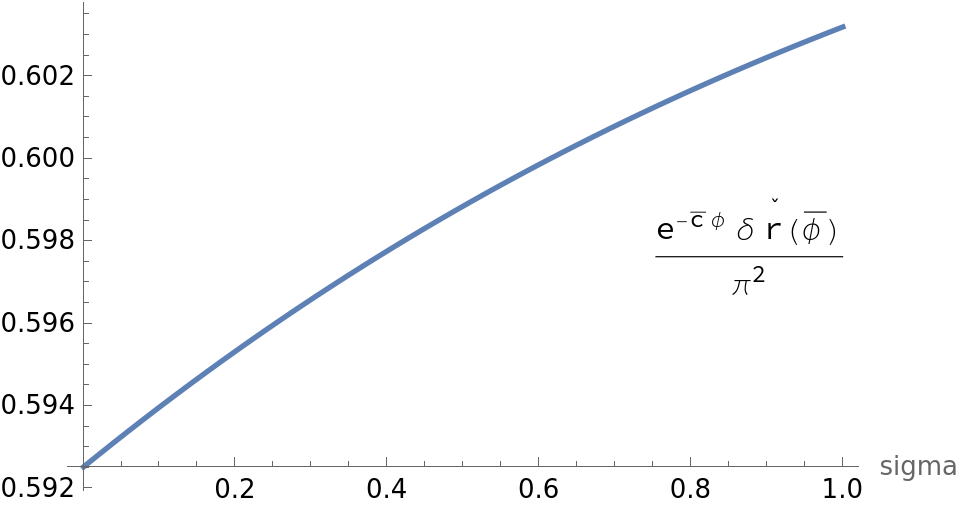}}
\end{subfigure}
\caption{Final value for Case 7, Equation (\ref{Equa:final_7_arcsegm}): the minimal values are for $\Delta \phi = 0,\theta = \frac{\pi}{2}$.}
\label{Fig:PTarc_case7_1}
\end{figure}

\item[Case 8: $\bar \tau > 1 + \tau_2,\bar \tau + \hat \tau < 2$] here the tent is inside $(\tau_2,1)$:
\begin{equation*}
\bar \alpha - \Delta \phi < \bar \phi - 4\pi - \bar \alpha < \bar \phi - 4\pi - \bar \alpha + \hat \theta < \bar \alpha, \quad \Delta \phi > \hat \theta,
\end{equation*}
\begin{equation*}
\delta \Delta Q = \sin(\theta) e^{\i (\bar \phi - 4\pi - \bar \alpha + \hat \theta - \frac{\pi}{2})} - \sin(\theta) e^{\i (\bar \phi - 4\pi - \bar \alpha - \frac{\pi}{2})}, \quad \delta \hat L = \sin(\theta) \hat \theta, \quad \delta \hat r_0 = \delta \tilde r(\bar \phi - (2\pi + \bar \alpha)).
\end{equation*}
The above condition implies that $\Delta \phi \geq \hat \theta$. The plot of the function
\begin{equation}
\label{Equa:final_case_8}
\frac{e^{- \bar c (2\pi + \bar \alpha) \bar \tau} \delta \check r(\bar \phi)}{\theta}, \quad \bar \tau = (1 - \varsigma) \bigg( 2 - \frac{\Delta \phi}{2\pi + \bar \alpha} \bigg) + \varsigma \bigg( 2 - \frac{\hat \theta}{2\pi + \bar \alpha} \bigg), \ \varsigma \in (0,1),
\end{equation}
is in Fig. \ref{Fig:PTarc_case8_1}, which is strictly positive $> 0.76$.

\begin{figure}
\begin{subfigure}{.475\textwidth}
\resizebox{\textwidth}{!}{\includegraphics{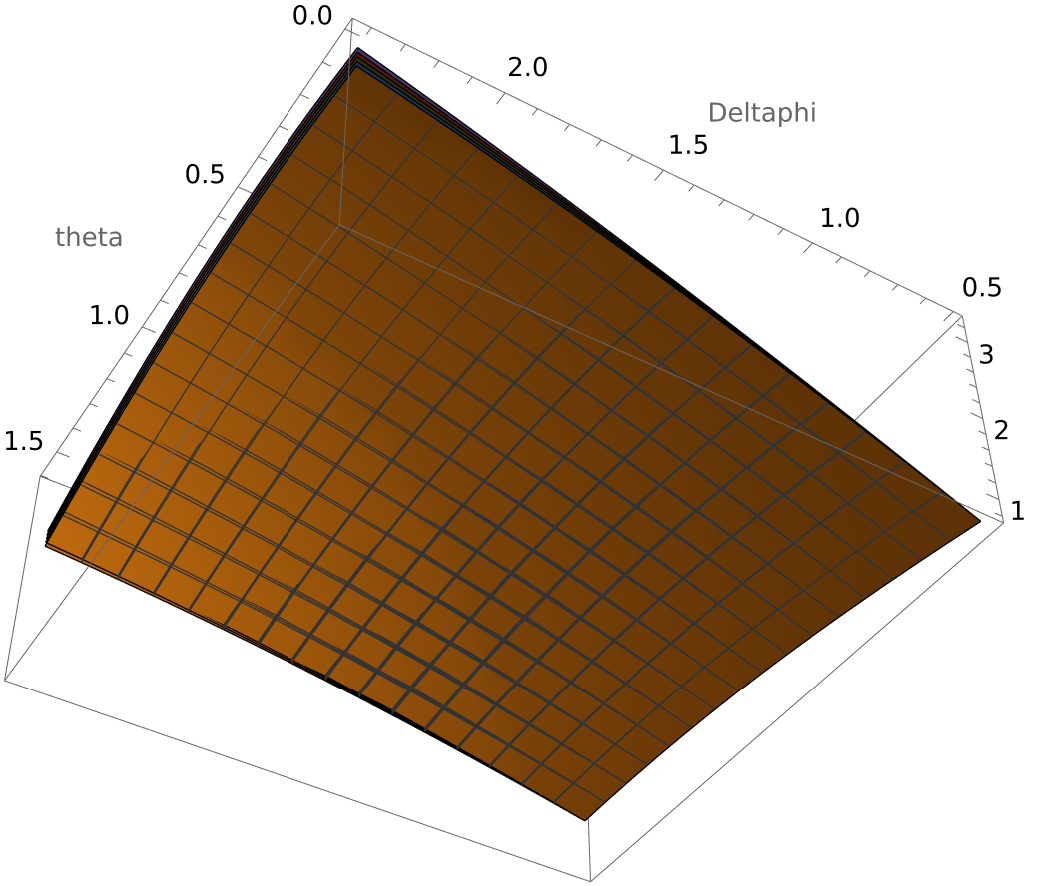}}
\end{subfigure} \hfill
\begin{subfigure}{.475\textwidth}
\resizebox{\textwidth}{!}{\includegraphics{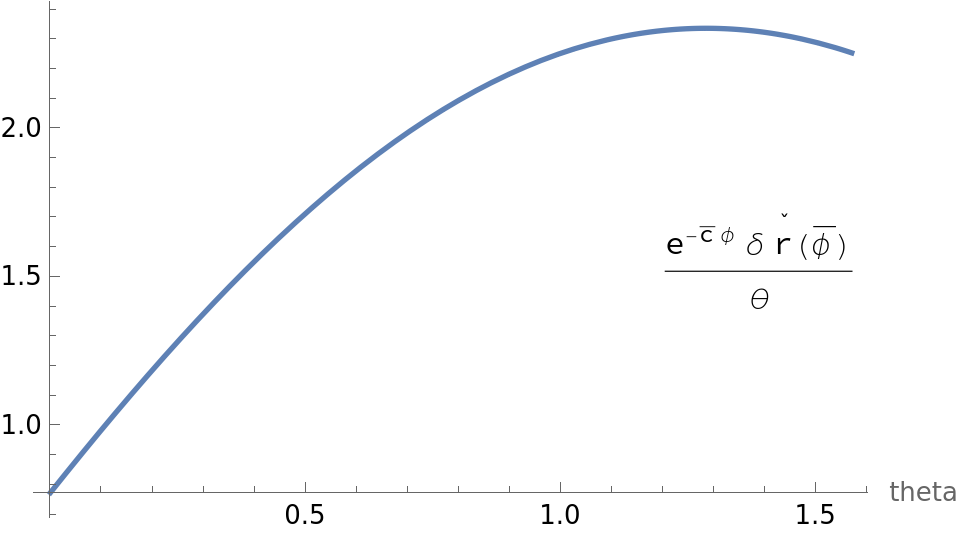}}
\end{subfigure}
\caption{Final value for Case 8, Equation (\ref{Equa:final_case_8}): the minimal values are for $\tau = 0,\Delta \phi = \hat \theta$.}
\label{Fig:PTarc_case8_1}
\end{figure}

\item[Case 9: $\bar \tau \in (1 + \tau_2,2),\bar \tau + \hat \tau > 2$] here the tent starts inside $(\tau_2,1)$ and ends after $\tau = 1$:
\begin{equation*}
\bar \alpha - \Delta \phi < \bar \phi - 4\pi - \bar \alpha < \bar \alpha, \quad \bar \alpha < \bar \phi - 4\pi - \bar \alpha + \hat \theta,
\end{equation*}
\begin{equation*}
\delta \Delta Q = \delta \tilde r(\bar \phi - 2(2\pi+\bar \alpha) + \hat \theta) e^{\i(\bar \phi - 2(2\pi + \bar \alpha) + \hat \theta)} - \sin(\theta) e^{\i (\bar \phi - 4\pi - \bar \alpha - \frac{\pi}{2})},
\end{equation*}
\begin{equation*}
\delta \hat L = \frac{1 - \cos(\theta) + \sin(\theta) \Delta \phi}{\sin(\bar \alpha)^2} (e^{\cot(\bar \alpha)(\bar \phi - 2(2 \pi + \bar \alpha) + \hat \theta)} - 1) + \sin(\theta) (- \bar \phi + 4\pi + \bar \alpha)), \quad \delta \hat r_0 = \delta \tilde r(\bar \phi - (2\pi + \bar \alpha)).
\end{equation*}
The above condition implies that $\theta \geq \hat \theta$. The plot of the function
\begin{equation}
\label{Equa:finalcase_9}
\frac{e^{- \bar c (2\pi + \bar \alpha) \bar \tau} \delta \check r(\bar \phi)}{\theta(\theta + \Delta \phi)}, \quad \bar \tau = (1 - \varsigma) \bigg( 2 - \frac{\theta + \hat \theta + \Delta \phi}{2\pi + \bar \alpha} \bigg) + \varsigma \bigg( 2 - \frac{\max\{\theta,\hat \theta\} + \Delta \phi}{2\pi + \bar \alpha} \bigg), \ \varsigma \in (0,1),
\end{equation}
is in Fig. \ref{Fig:PTarc_case9_1}, which is strictly positive $> 0.62$.

\begin{figure}
\begin{subfigure}{.475\textwidth}
\resizebox{\textwidth}{!}{\includegraphics{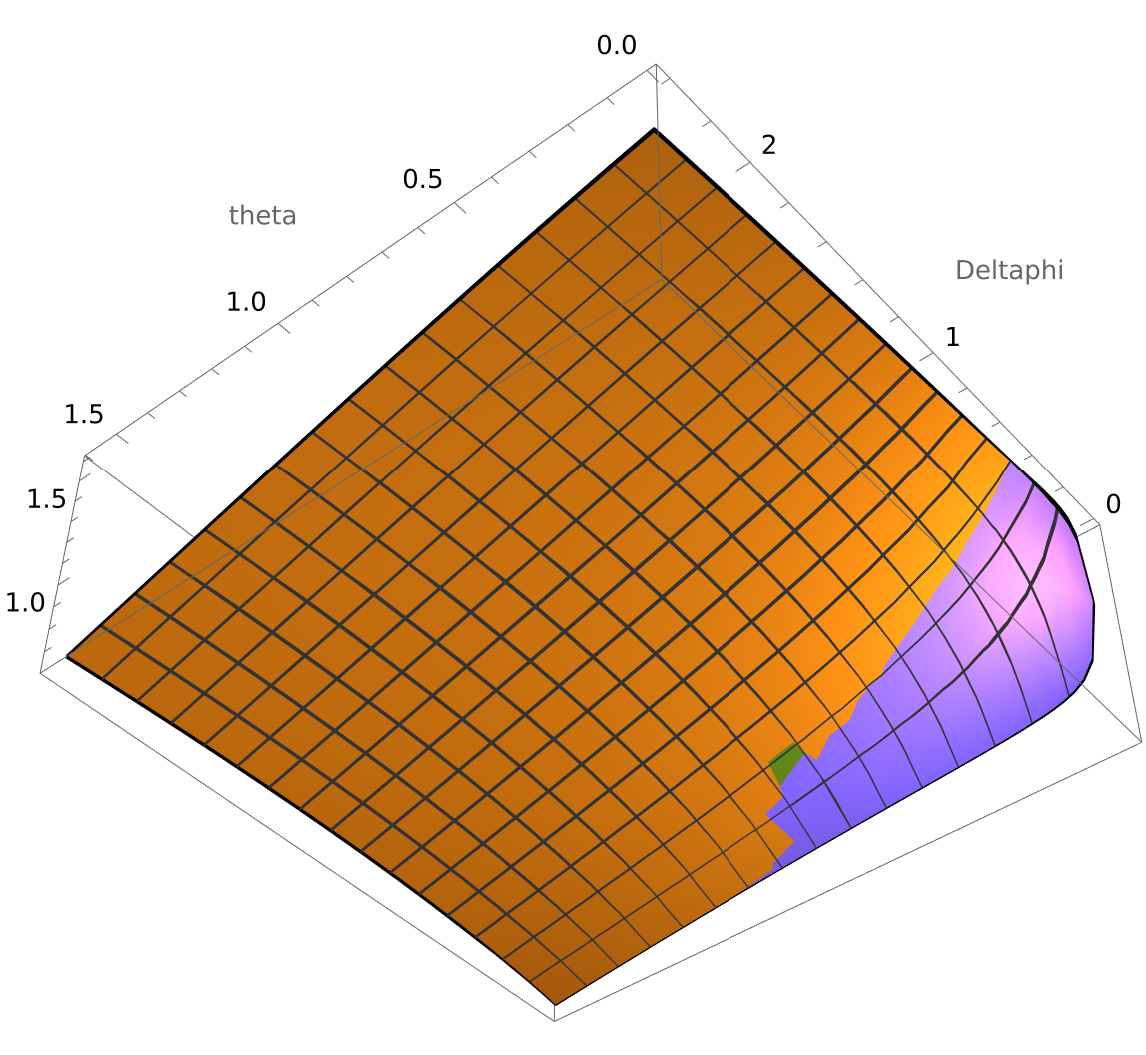}}
\end{subfigure} \hfill
\begin{subfigure}{.475\textwidth}
\resizebox{\textwidth}{!}{\includegraphics{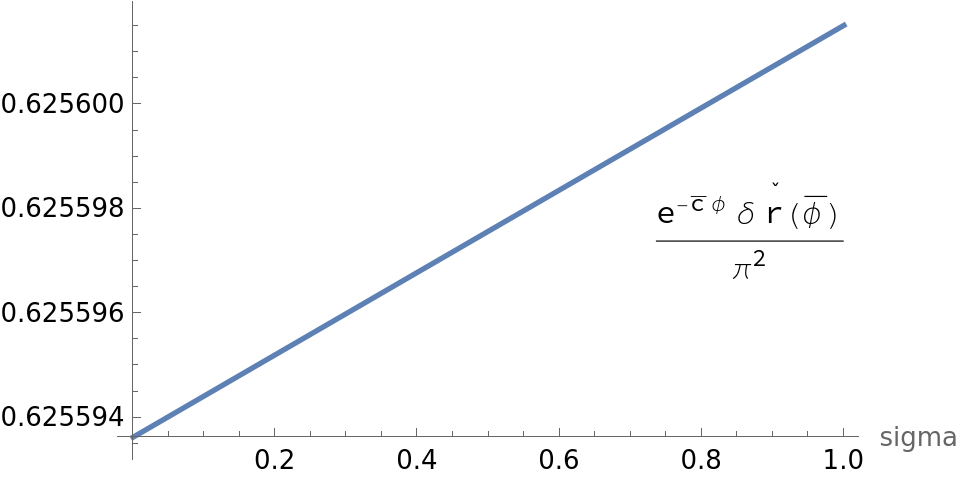}}
\end{subfigure}
\caption{Final value for Case 9, Equation (\ref{Equa:finalcase_9}): the minimal values are for $\Delta \phi = 0, \theta = \frac{\pi}{2}$.}
\label{Fig:PTarc_case9_1}
\end{figure}

\item[Case 10: $\bar \tau > 2$] here the base of the tent is after $\tau = 0$, and in particular the perturbed spiral $\delta \tilde r$ is positive. Similarly to Section \ref{Sss:segm_semg_after_neg}, we can thus use Lemma \ref{Lem:final_value_satu} and just study the function
\begin{equation}
\label{Equa:final_case_10}
\bar \tau \mapsto \frac{1}{\theta(\theta + \Delta \phi)} \bigg( \delta \rho(\bar \tau) - \frac{(1 - \cos(\hat \theta)) e^{-\bar c(2\pi + \bar \alpha)}}{\cos(\bar \alpha - \hat \theta) - \cos(\bar \alpha)} \delta \rho(\bar \tau - 1) \bigg),
\end{equation}
with $\rho(\tau)$ given by Corollary \ref{Cor:arc_rho_eq}. The numerical plot of the above function for $\tau \in (2,5)$ is in Fig. \ref{Fig:PTarc_case10_1} and of its $\tau$-derivative for $\tau \in (4,5)$ is in Fig. \ref{Fig:PTarc_case10bis_1}: we can thus apply Lemma \ref{lem:key} and deduce that
\begin{equation*}
\frac{1}{\theta(\theta + \Delta \phi)} \bigg( \delta \rho(\bar \tau) - \frac{(1 - \cos(\hat \theta)) e^{-\bar c(2\pi + \bar \alpha)}}{\cos(\bar \alpha - \hat \theta) - \cos(\bar \alpha)} \delta \rho(\bar \tau - 1) \bigg) > 0.59.
\end{equation*}

\begin{figure}
\begin{subfigure}{.475\textwidth}
\resizebox{\textwidth}{!}{\includegraphics{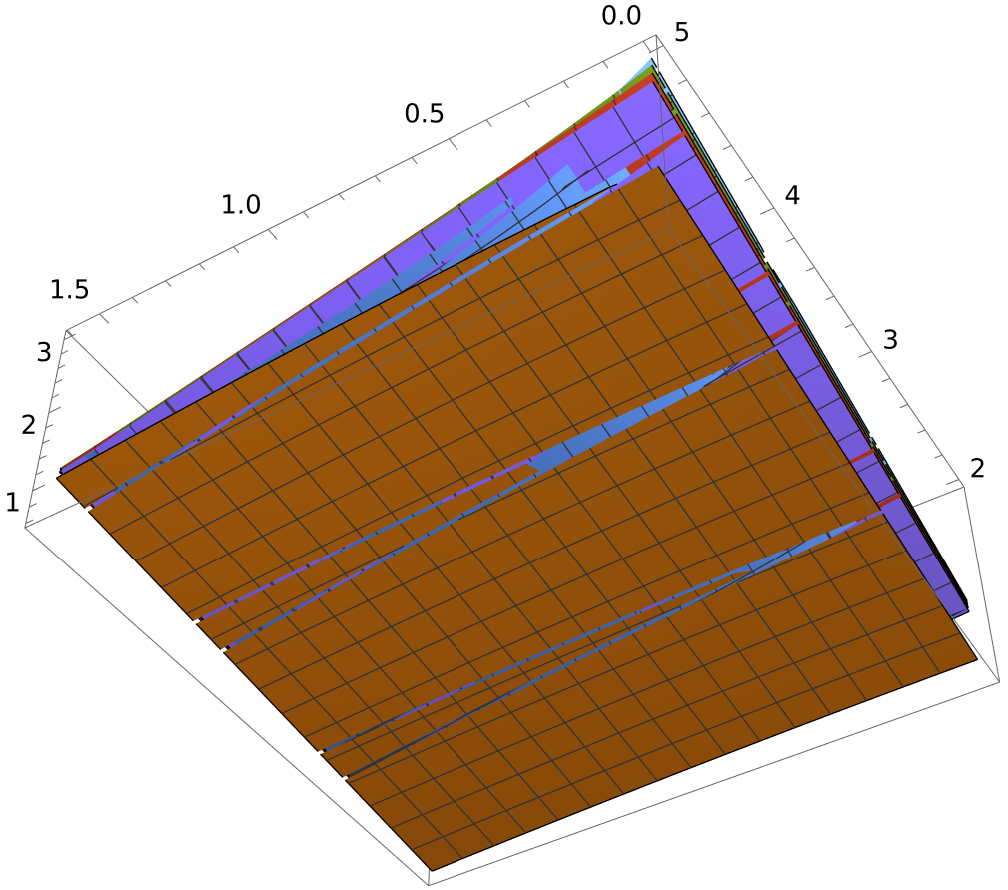}}
\end{subfigure} \hfill
\begin{subfigure}{.475\textwidth}
\resizebox{\textwidth}{!}{\includegraphics{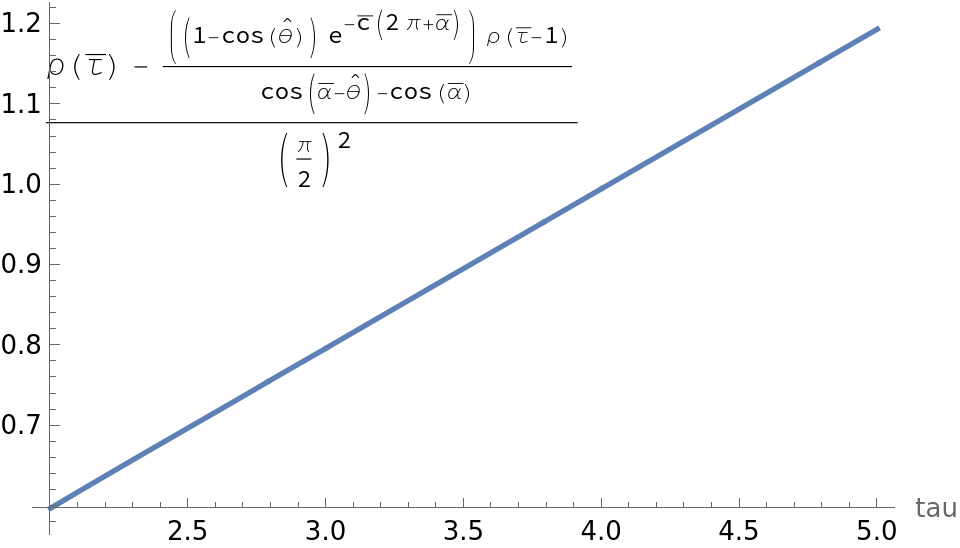}}
\end{subfigure}
\caption{Final value for Case 10, Equation (\ref{Equa:final_case_10}): the minimal values are for $\Delta \phi = 0, \theta = \frac{\pi}{2}$.}
\label{Fig:PTarc_case10_1}
\end{figure}

\begin{figure}
\begin{subfigure}{.475\textwidth}
\resizebox{\textwidth}{!}{\includegraphics{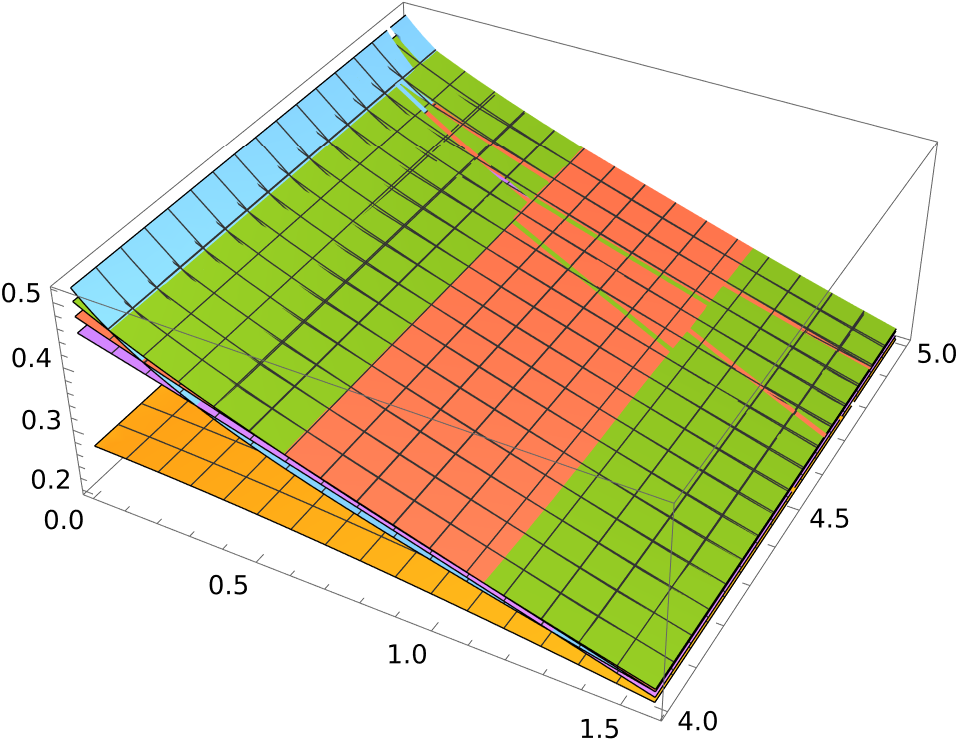}}
\end{subfigure} \hfill
\begin{subfigure}{.475\textwidth}
\resizebox{\textwidth}{!}{\includegraphics{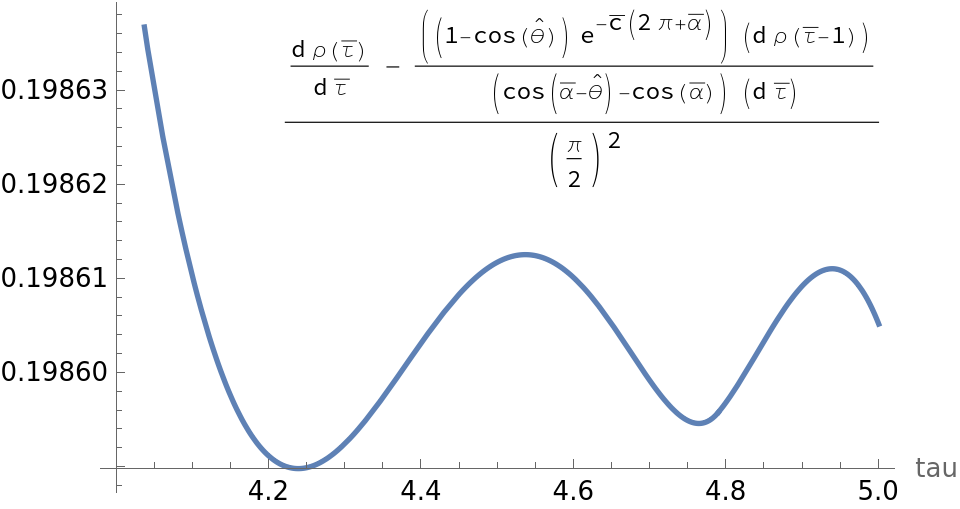}}
\end{subfigure}
\caption{Derivative for Case 10, Equation (\ref{Equa:final_case_10}): the minimal values are for $\Delta \phi = 0, \theta = \frac{\pi}{2}$.}
\label{Fig:PTarc_case10bis_1}
\end{figure}

\end{description}

We collect all the above statements into the following

\begin{lemma}
\label{Lem:arc_segm_satur}
If the final tent is in the saturated region of $\check r$, then $\delta \check r(\bar \phi) \geq 0$, and it is $0$ only if the angle $\theta = 0$.
\end{lemma}

This lemma, together with Lemmas \ref{Lem:segm_arc_2}, \ref{Lem:segm_arc_3}, \ref{Lem:arc_segm_1}, shows that $\delta \check r(\bar \phi) > 0$, and it is $0$ if and only if $\theta = 0$, concluding the proof of Theorem \ref{Theo:main_extended} using also the results of Section \ref{S:segment_case}.

\newpage

\appendix

\section{Numerical code for Wolfram Mathematica}
\label{S:num_code}

\noindent {\bf Section \ref{Sss:change_variables}}

\bigskip

\noindent Exponent $c(\alpha)$, Eq. \eqref{Equa:c_gene_defi}
\medskip
\begin{verbatim}
In[7]:= c[alpha_] := Log[(2 Pi + alpha)/Sin[alpha]]/(2 Pi + alpha)
\end{verbatim}
\medskip
Coefficient $a(\alpha)$, Eq. \eqref{Equa:a_const}
\medskip
\begin{verbatim}
a[alpha_] := (2 Pi + alpha) (Cot[alpha] - c[alpha])
\end{verbatim}
\medskip
Computation of the critical eigenvalue, Cor. \ref{Cor:eigen_angle}
\medskip
\begin{verbatim}
lam[alpha_] := Exp[(2 Pi + alpha) Cot[alpha] - 1]/(2 Pi + alpha) Sin[alpha] - 1
\end{verbatim}
\medskip
Critical angle $\bar \alpha$
\medskip
\begin{verbatim}
baralpha = FindRoot[lam[alpha], {alpha, 1.2}][[1, 2]]
1.1783
\end{verbatim}
\medskip
Critical speed $\bar \sigma$
\medskip
\begin{verbatim}
1/Cos[baralpha]
2.61443
\end{verbatim}
\medskip
Critical exponent $\bar c$
\medskip
\begin{verbatim}
barc = c[baralpha]
0.27995
\end{verbatim}
\medskip
Computations of the complex eigenvalues, Fig. \ref{Fig:complex_eigen_f}
\medskip
\begin{verbatim}
Plot[Exp[y Cot[y]] Sin[y]/y, {y, 0, 20}, PlotLegends -> 
  Placed[ToExpression["\\frac{e^{y \\cot(y)} \\sin(y)}{y}", TeXForm, 
    HoldForm], {Right, Top}], PlotRange -> {{-15, 15}}, 
 AxesLabel -> Automatic]
\end{verbatim}

\bigskip

\noindent{\bf Section \ref{ss:subset:saturated}}

\bigskip

\noindent Computation for Remark \ref{Rem:numer_rounds}
\medskip
\begin{verbatim}
Cot[baralpha] - 2 Pi/(2 Pi + baralpha)
-0.428111
\end{verbatim}
\medskip
Plot of the function \eqref{Equa:intern_prop_3.4}, Fig. \ref{Fig:proposition_3.4_proof}
\medskip \begin{verbatim}
Plot[{(Exp[Cot[alpha] 2 Pi] - 1) Exp[-2 Pi Cot[alpha] + a[alpha]] - 
   Exp[a[alpha] 2 Pi/(2 Pi + alpha)]}, {alpha, 1, 1.3}, 
 PlotLegends -> 
  Placed[Text[
    Style[ToExpression[
      "(e^{2\\pi \\cot(\\alpha)} - 1) e^{- 2\\pi \\cot(\\alpha) + \
a(\\alpha)} - e^{a(\\alpha) \\frac{2\\pi}{2\\pi + \\alpha}}", TeXForm,
       HoldForm]]], {Right, Top}]]
\end{verbatim}

\bigskip

\noindent{\bf Section \ref{Sss:green_kernels}}

\bigskip

\noindent Explicit form up to $\tau = 6$, Lemma \ref{Lem:explic_kernel_RDE}
\medskip
\begin{verbatim}
galpha[tau_, alpha_] := 
 Sum[(-1)^k  Exp[a[alpha] (tau - k)] (tau - k)^k/k! HeavisideTheta[
    tau - k], {k, 0, 5}]
\end{verbatim}
\medskip
Kernel in the critical case
\medskip
\begin{verbatim}
g[tau_] := galpha[tau, baralpha]
\end{verbatim}
\medskip
Asymptotic expansion, Point \ref{Point_2:expli_exp_kern} of Lemma \ref{Lem:aympt_g}
\medskip
\begin{verbatim}
gasympt[tau_] := 2 tau + 2/3
\end{verbatim}
\medskip
Plot for Fig. \ref{Fig:g_kernel_sympt}
\medskip
\begin{verbatim}
Plot[{galpha[tau, baralpha - .01], galpha[tau, baralpha], 
  galpha[tau, baralpha + .1]}, {tau, 0, 6}, 
 PlotLegends -> 
  Placed[ToExpression["g(\\tau,\\alpha)", TeXForm, HoldForm], {Center,
     Top}], AxesLabel -> Automatic]
\end{verbatim}
\begin{verbatim}
Plot[{g[tau] - gasympt[tau]}, {tau, 0, 5}, 
 PlotLegends -> 
  Placed[ToExpression["g(\\tau,\\bar \\alpha)-2\\tau-\\frac{2}{3}", 
    TeXForm, HoldForm], {Right, Top}], AxesLabel -> Automatic, 
 PlotRange -> {{-.1, 5}, {-.06, .06}}]
\end{verbatim}
\medskip
Construction of the solution $m$ of Lemma \ref{Lem:expli_source_RDE}
\medskip
\begin{verbatim}
m0alpha[tau_, tau1_, alpha_] := 
 Sum[(-1)^
     l (2 Pi + alpha) Exp[-c[alpha] (2 Pi + alpha) (tau - k)] Exp[
     Cot[alpha] (2 Pi + alpha) (tau - tau1 - 
        k)] (Cot[alpha] (2 Pi + alpha) (tau - tau1 - k))^
      l/(l! (Cot[alpha] (2 Pi + alpha))^{1 + k}) HeavisideTheta[
     tau - tau1 - k], {k, 0, 6}, {l, 1, k}] + 
  Sum[(2 Pi + alpha) Exp[-c[alpha] (2 Pi + alpha) (tau - k)] Max[0, 
      Exp[(Cot[alpha] (2 Pi + alpha)) (tau - tau1 - k)] - 
       1]/(Cot[alpha] (2 Pi + alpha))^{1 + k}, {k, 0, 5}]
\end{verbatim}
\begin{verbatim}
malpha[tau_, tau1_, tau2_, alpha_] := 
 m0alpha[tau, tau1, alpha] - m0alpha[tau, tau2, alpha]
\end{verbatim}
\medskip
In the critical case $\alpha = \bar \alpha$
\medskip
\begin{verbatim}
m[tau_, tau1_, tau2_] := malpha[tau, tau1, tau2, baralpha]
\end{verbatim}
\medskip
Asymptotic behavior of $M$ in the critical case, Lemma \ref{Lem:m_aymptotic}: coefficients for the linear and constant part
\medskip
\begin{verbatim}
masymptlin[tau1_, 
  tau2_] := -2 Exp[-barc (2 Pi + baralpha) tau2]/barc + 
  2 Exp[-barc (2 Pi + baralpha) tau1]/barc
\end{verbatim}
\begin{verbatim}
masymptconst[tau1_, tau2_] := 
 2 Exp[-barc (2 Pi + baralpha) tau2]/(barc^2 (2 Pi + 
        baralpha)) - (2 (-tau2) + 
     2/3) Exp[-barc (2 Pi + baralpha) tau2]/
    barc - (2 Exp[-barc (2 Pi + baralpha) tau1]/(barc^2 (2 Pi + 
      baralpha)) - (2 (-tau1) + 
       2/3) Exp[-barc (2 Pi + baralpha) tau1]/barc)
\end{verbatim}
\medskip
Asymptotic expansion of $m$
\medskip
\begin{verbatim}
masympt[tau_, tau1_, tau2_] := 
 2 Exp[-barc (2 Pi + baralpha) tau2]/(barc^2 (2 Pi + 
        baralpha)) - (2 (tau - tau2) + 
     2/3) Exp[-barc (2 Pi + baralpha) tau2]/
    barc - (2 Exp[-barc (2 Pi + baralpha) tau1]/(barc^2 (2 Pi + 
          baralpha)) - (2 (tau - tau1) + 
       2/3) Exp[-barc (2 Pi + baralpha) tau1]/barc)
\end{verbatim}
\medskip
Plot for Fig. \ref{Fig:m_m_sympt}
\medskip
\begin{verbatim}
Plot[{malpha[tau, 1/2, 1, baralpha - .01], 
  malpha[tau, 1/2, 1, baralpha], 
  malpha[tau, 1/2, 1, baralpha + .1]}, {tau, 0, 6}, 
 PlotLegends -> 
  Placed[ToExpression["m(\\tau,1/2,1,\\alpha)", TeXForm, 
    HoldForm], {Center, Top}], AxesLabel -> Automatic]
\end{verbatim}
\begin{verbatim}
Plot3D[{0, 
      m[tau, tau1, 1] - masymptlin[tau1, 1] tau - 
       masymptconst[tau1, 1]}, {tau, 0, 5}, {tau1, 
      1 - Tan[baralpha]/(2 Pi + baralpha), 1}, AxesLabel -> Automatic, 
     PlotRange -> {-.02, .02}]
\end{verbatim}
\medskip
Kernel in $\phi$, Remark \ref{Rem:oringal_phi_kernel}, up to $\phi = 6(2\pi+\alpha)$
\medskip
\begin{verbatim}
Galpha[phi_, alpha_] := 
  galpha[phi/(2 Pi + alpha), alpha] Exp[barc phi]
\end{verbatim}
\begin{verbatim}
Malpha[phi_, phi1_, phi2_, alpha_] := 
  malpha[phi/(2 Pi + alpha), phi1/(2 Pi + alpha), phi2/(2 Pi + alpha), 
    alpha] Exp[barc phi]
\end{verbatim}
\begin{verbatim}
G[phi_] := Galpha[phi, baralpha]
\end{verbatim}
\begin{verbatim}
M[phi_, phi1_, phi2_] := Malpha[phi, phi1, phi2, baralpha]
\end{verbatim}
\medskip
Plot of the solution for $\mathtt a = \sin(\bar \alpha)$, i.e. $\theta = \pi/2$, Fig. \ref{Fig:base_sol_intro}
\medskip
\begin{verbatim}
Plot[g[tau]/Sin[baralpha] - 
   Exp[-barc (2 Pi + baralpha - Pi/2)] g[
     tau - 1 + Pi/2/(2 Pi + baralpha)] - 
   Cot[baralpha] Exp[-barc (2 Pi + baralpha)] g[tau - 1], {tau, 0, 5}, 
  AxesLabel -> Automatic, 
  PlotLegends -> 
   Placed[ToExpression["\\rho(\\phi,\\sin(\\bar \\alpha))", TeXForm, 
     HoldForm], {Right, Center}]]
\end{verbatim}

\bigskip

\noindent{ \bf Section \ref{Sss:length_saturated}}

\bigskip

\noindent Primitive functions for the kernel $\mathcal G$
\medskip
\begin{verbatim}
Lgg[tau_, alpha_] := 
  Sum[1/Cos[
       alpha]^{k + 
        1} (Sum[(-1)^j Exp[
         Cot[alpha] (2 Pi + alpha) (tau - 
        k)] (Cot[alpha] (2 Pi + alpha) (tau - k))^j/j!, {j, 0, 
        k}] - 1) HeavisideTheta[tau - k], {k, 0, 5}]
 Lg[tau_] := Lgg[tau, baralpha]
\end{verbatim}
\medskip
This is $\mathfrak G$, used for the integral of $m$
\medskip
\begin{verbatim}
LLgg[tau_, alpha_] := 
  Sum[1/Cos[
       alpha]^{k + 
        1} (Sum[(Sum[(-1)^l Exp[
         Cot[alpha] (2 Pi + alpha) (tau - 
            k)] (Cot[alpha] (2 Pi + alpha) (tau - k))^l/l!, {l, 0, 
            j}]/Cot[alpha]) - 1/Cot[alpha], {j, 0, 
        k}] - (2 Pi + alpha) (tau - k)) HeavisideTheta[tau - k], {k, 0,
     5}]
 LLg[tau_] := LLgg[tau, baralpha]
\end{verbatim}
\medskip
This is the formula for the integral of $m$
\medskip
\begin{verbatim}
Lm[tau_, tau1_, tau2_] := LLg[tau - tau1] - LLg[tau - tau2]
\end{verbatim}

\bigskip

\noindent{\bf Section \ref{S:case:study}}

\bigskip

\noindent This is the plot of the derivative w.r.t. $\mathtt a$ in the arc case, Equation \eqref{Equa:ckrho_delph_casest} and Fig. \ref{Fig:ckrho_delph_casst}
\medskip
\begin{verbatim}
Plot[Cot[baralpha] g[tau] - 
  Exp[-barc Pi/2] g[tau - Pi/(2 (2 Pi + baralpha))] - 
  Exp[-barc (2 Pi + baralpha)] g[tau - 1]/Sin[baralpha], {tau, 1, 2}, 
 AxesLabel -> Automatic, 
 PlotLegends -> 
  Placed[ToExpression["\\check \\rho(\\tau)", TeXForm, 
    HoldForm], {Right, Center}]]
\end{verbatim}
\medskip
Evolution of the solution when $\mathtt a = 0$, Fig. \ref{Fig:ckrho_delph_casst}
\medskip
\begin{verbatim}
Plot[(G[phi]/Sin[baralpha] - 
    G[phi - (2 Pi + baralpha - Tan[baralpha] - Pi/2)] - 
    M[phi, 2 Pi + baralpha - Tan[baralpha], 2 Pi + baralpha] - 
    Cot[baralpha] G[phi - 2 Pi - baralpha]) Exp[-barc phi], {phi, 
  Pi/2 + Tan[baralpha] - baralpha, 4 Pi + 2 baralpha}, 
 AxesLabel -> Automatic, 
 PlotLegends -> 
  Placed[ToExpression[
    "\\tilde r_{\\mathtt a = 0}(\\phi) e^{- \\bar c \\phi}", TeXForm, 
    HoldForm], {Center, Top}]]
\end{verbatim}
\medskip
Segment case: plot of the function for the derivative w.r.t. $\mathtt a$, Fig. \ref{Fig:ckrho_delph_case_tg}
\medskip
\begin{verbatim}
Plot[g[tau] - Exp[-barc (2 Pi + baralpha)] g[tau - 1] + 
  2 Cos[baralpha] Exp[-2 barc (2 Pi + baralpha)] g[tau - 2], {tau, 0, 
  5}, AxesLabel -> Automatic, 
 PlotLegends -> 
  Placed[ToExpression[
    "g(\\tau) - 2\\cos(\\bar \\alpha) e^{-\\bar c(2\\pi + \\bar \
\\alpha)}g(\\tau-1)+e^{-2\\bar c(2\\pi+\\bar\\alpha)} g(\\tau-2)", 
    TeXForm, HoldForm], {Center, Top}]]
\end{verbatim}

\bigskip

\noindent{\bf Section \ref{Ss:length_spiral_est}}

\bigskip

\noindent Fastest saturated spiral with $\mathtt a = 0$, Equation \eqref{Equa:rhobase_01}
\medskip
\begin{verbatim}
rhobase[tau_] := 
 g[tau]/Sin[baralpha] - 
  Exp[-barc (2 Pi + baralpha - Pi/2 - Tan[baralpha])] g[
    tau - (2 Pi + baralpha - Pi/2 - Tan[baralpha])/(2 Pi + 
        baralpha)] - m[tau, 1 - Tan[baralpha]/(2 Pi + baralpha), 1] - 
  Cot[baralpha] Exp[-barc (2 Pi + baralpha)] g[tau - 1]
\end{verbatim}
\medskip
Length of the fastest saturated spiral, Equation \eqref{Equa:Lbase_01}
\medskip
\begin{verbatim}
Lbase[tau_] := 
 Piecewise[{{(2 Pi + baralpha) tau + Pi/2 - baralpha + 
     Tan[baralpha], -(Pi/2 - baralpha + Tan[baralpha])/(2 Pi + 
        baralpha) <= 
     tau < -(Pi/2 - baralpha)/(2 Pi + baralpha)}, {Tan[baralpha] + 
     Cot[baralpha - (2 Pi + baralpha) tau], -(Pi/2 - baralpha)/(2 Pi +
         baralpha) <= tau < 0}, {Tan[baralpha] + Cot[baralpha] + 
     Lg[tau]/Sin[baralpha] - 
     Lg[tau - (1 - (Pi/2 + Tan[baralpha])/(2 Pi + baralpha))] - 
     Lm[tau, 1 - Tan[baralpha]/(2 Pi + baralpha), 1] - 
     Cot[baralpha] Lg[tau - 1], tau >= 0}}]
\end{verbatim}
\medskip
Plot of the comparison of the ray and the length, Fig. \ref{Fig:rhobase_Lbase_comp}
\medskip
\begin{verbatim}
Plot[rhobase[tau] - 
  2.08 Exp[-barc (2 Pi + baralpha) tau ] Lbase[tau - 1], {tau, 0, 5}, 
 AxesLabel -> Automatic, 
 PlotLegends -> 
  Placed[ToExpression[
    "\\rho(\\tau) - 2.08 e^{-\\bar c(2\\pi + \\bar \\alpha) \\tau} L(\
\\tau-1)", TeXForm, HoldForm], {Center, Top}]]
\end{verbatim}
\begin{verbatim}
Plot[rhobase[tau] - 
  2.08 Exp[-barc (2 Pi + baralpha) tau ] Lbase[tau - 1], {tau, 3.5, 
  5}, AxesLabel -> Automatic, 
 PlotLegends -> 
  Placed[ToExpression[
    "\\rho(\\tau) - 2.08 e^{- \\bar c(2\\pi + \\bar \\alpha) \\tau} \
L(\\tau-1)", TeXForm, HoldForm], {Center, Top}]]
\end{verbatim}
\medskip
Estimate for the source of the length equation, Equation \eqref{Equa:SL_const}
\medskip
\begin{verbatim}
1/Sin[baralpha] - 1 - Tan[baralpha] - Cot[baralpha]

-2.7473
\end{verbatim}
\medskip
Asymptotic constant for the ratio $\frac{L(\tau-1)}{r_{\mathtt a=0}(\tau)}$, Equation \eqref{Equa:exact_Lbaserho}
\medskip
\begin{verbatim}
Exp[barc (2 Pi + baralpha)] barc Sin[baralpha]

2.08884
\end{verbatim}
\medskip

\bigskip

\noindent{\bf Section \ref{S:family}}

\bigskip

\noindent Arc (segment) coefficients: these are the coefficients for the general perturbation, Corollary \ref{Cor:arc_rho_eq}. The coefficients for the segment case are obtained from these by setting $\Delta \phi = 0$ as in Remark \ref{Rem:same_as_segment} and Corollary \ref{Cor:equation}
\smed
\begin{verbatim}
S0arc[theta_, Deltaphi_] := 
 Cot[baralpha] (1 - Cos[theta] + Sin[theta] Deltaphi)/Sin[baralpha]
\end{verbatim}
\begin{verbatim}
S1arc[theta_, 
  Deltaphi_] := -Sin[baralpha] Exp[
    barc (theta + Deltaphi)]/(2 Pi + baralpha)
\end{verbatim}
\begin{verbatim}
S2arc[theta_, Deltaphi_] := 
 Sin[baralpha] Cos[theta] Exp[barc Deltaphi]/(2 Pi + baralpha)
\end{verbatim}
\begin{verbatim}
S3arc[theta_] := - Sin[theta]
\end{verbatim}
\begin{verbatim}
S4arc[theta_, Deltaphi_] := 
 Cos[baralpha] (-2 Cot[
       baralpha] (1 - Cos[theta] + Sin[theta] Deltaphi) + 
     Sin[theta])/(2 Pi + baralpha)
\end{verbatim}
\begin{verbatim}
Darc[theta_, Deltaphi_] := 
 Sin[baralpha] (-Cot[baralpha] (1 - Cos[theta] + 
        Sin[theta] Deltaphi) + Sin[theta])/(2 Pi + baralpha)
\end{verbatim}
\begin{verbatim}
S5arc[theta_, 
  Deltaphi_] := -Sin[
    baralpha] (-Cot[baralpha] (1 - Cos[theta] + Sin[theta] Deltaphi) +
      Sin[theta])/(2 Pi + baralpha)^2
\end{verbatim}
\begin{verbatim}
tau1arc[theta_, Deltaphi_] := 1 - (theta + Deltaphi)/(2 Pi + baralpha)
\end{verbatim}
\begin{verbatim}
tau2arc[Deltaphi_] := 1 - Deltaphi/(2 Pi + baralpha)
\end{verbatim}
\medskip
This is the generic form of the solution, without the delta Dirac delta $\Diracd_{\tau=1}$, Corollaries \ref{Cor:equation} and \ref{Cor:arc_rho_eq}
\medskip
\begin{verbatim}
rhogeneric[tau_, theta_, Deltaphi_] := 
 S0arc[theta, Deltaphi] g[tau] + 
  S1arc[theta, Deltaphi] g[tau - tau1arc[theta, Deltaphi]] + 
  S2arc[theta, Deltaphi] g[tau - tau2arc[Deltaphi]] + 
  S3arc[theta] m[tau, tau2arc[Deltaphi], 1] + 
  S4arc[theta, Deltaphi] g[tau - 1] + 
  S5arc[theta, Deltaphi] g[tau - 2]
\end{verbatim}

\bigskip

\noindent{\bf Segment case, Section \ref{S:segment}}

\bigskip

\noindent A plot for the introduction, Fig. \ref{Fig:segment_intro_2}
\medskip
\begin{verbatim}
Plot[rhogeneric[tau, .3, 0], {tau, 0, 3}, AxesLabel -> Automatic, 
 PlotLegends -> 
  Placed[ToExpression["\\rho(\\tau)", TeXForm, HoldForm], {Center, 
    Top}]]
\end{verbatim}
\medskip
Computations for Proposition \ref{Prop:regions_pos_neg_segm} \\
First round segment: computation of the angle hat theta, Equation \eqref{Equa:hat_theta_value}
\medskip
\begin{verbatim}
hatheta = 
 FindRoot[
   Cot[baralpha] (1 - Cos[theta])/Sin[baralpha] Exp[
      Cot[baralpha] (2 Pi + baralpha - theta)] - 1, {theta, .5}][[1, 
  2]]
\end{verbatim}
\begin{verbatim}
0.506134
\end{verbatim}
\medskip
Plot of Equation \eqref{Equa:hat_theta_det_666}, Fig. \ref{Fig:hattheta_comput}
\medskip
\begin{verbatim}
Plot[Cot[baralpha] (1 - Cos[theta])/Sin[baralpha] - 
  Exp[-Cot[baralpha] (2 Pi + baralpha - theta)], {theta, 0, Pi}, 
 AxesLabel -> Automatic, 
 PlotLegends -> 
  Placed[ToExpression[
    "\\cot(\\bar \\alpha) \\frac{1 - \\cos(\\theta)}{\\sin(\\bar \
\\alpha)} - e^{-\\cot(\\bar \\alpha)(2\\pi + \\bar \\alpha - \
\\theta)}", TeXForm, HoldForm], {Center, Top}]]
\end{verbatim}
\medskip
Plot of the first round segment case, Fig. \ref{Fig:hattheta_comput}
\medskip
\begin{verbatim}
Plot3D[{0, 
  S0arc[theta, 0] Exp[tau] + 
   S1arc[theta, 0] Exp[tau - tau1arc[theta, 0]] HeavisideTheta[
     tau - tau1arc[theta, 0]]}, {tau, 0, 1}, {theta, 0, Pi}, 
 AxesLabel -> Automatic]
\end{verbatim}
\medskip
Second round segment: computation of the expansion for $\tau \in [1,1+\tau1]$ \eqref{Equa:first_round_exp_segm}
\medskip
\begin{verbatim}
Cot[baralpha] (Exp[1] - 1)/Sin[baralpha] - 
 1/((2 Pi + baralpha) Sin[baralpha]) - 
 2 Cot[baralpha] Cos[baralpha]/(2 Pi + baralpha)
\end{verbatim}
\begin{verbatim}
0.582367
\end{verbatim}
\medskip
Plot of the second round segment, Fig. \ref{Fig:segm_round_2}
\medskip
\begin{verbatim}
Plot3D[{0, 
  S0arc[theta, 0] g[tau] + 
   S1arc[theta, 0] g[tau - tau1arc[theta, 0]] + 
   S2arc[theta, 0] g[tau - tau2arc[0]] + 
   S4arc[theta, 0] g[tau - 1]}, {tau, 1, 2}, {theta, 0, Pi}, 
 AxesLabel -> Automatic]
\end{verbatim}
\begin{verbatim}
Plot3D[{0, rhogeneric[tau, theta, 0]/theta^2}, {tau, 1, 2}, {theta, 0,
   Pi}, AxesLabel -> Automatic]
\end{verbatim}
\medskip
Minimal value for the second round
\medskip
\begin{verbatim}
(S0arc[Pi, 0] g[1] + S1arc[Pi, 0] g[1 - tau1arc[Pi, 0]] + 
   S2arc[Pi, 0] + S4arc[Pi, 0])/(Pi)^2
\end{verbatim}
\begin{verbatim}
0.17959
\end{verbatim}
\medskip
Minimal value for $\tau = 1+\tau_1$
\medskip
\begin{verbatim}
(S0arc[Pi, 0] g[1 + tau1arc[Pi, 0]] + S1arc[Pi, 0] g[1] + 
   S2arc[Pi, 0] g[tau1arc[Pi, 0]] + 
   S4arc[Pi, 0] g[tau1arc[Pi, 0]])/(Pi)^2
\end{verbatim}
\begin{verbatim}
0.226635
\end{verbatim}
\medskip
Third round segment: plot of Fig. \ref{Fig:thordround_rhosegm}
\medskip
\begin{verbatim}
Plot3D[{0, rhogeneric[tau, theta, 0]}, {tau, 2, 3}, {theta, 0, Pi}, 
 AxesLabel -> Automatic]
\end{verbatim}
\begin{verbatim}
Plot3D[{0, rhogeneric[tau, theta, 0]/theta^2}, {tau, 2, 3}, {theta, 0,
   Pi}, AxesLabel -> Automatic]
\end{verbatim}
\medskip
Minimal value for the third round
\medskip
\begin{verbatim}
(S0arc[Pi, 0] g[2] + S1arc[Pi, 0] g[2 - tau1arc[Pi, 0]] + 
   S2arc[Pi, 0] g[1] + S4arc[Pi, 0] g[1] + S5arc[Pi, 0])/(Pi)^2
\end{verbatim}
\begin{verbatim}
0.262163
\end{verbatim}
\medskip
Increasing after the third round segment: plot of Fig. \ref{Fig:final_growth_rhosegm}
\medskip
\begin{verbatim}
Plot3D[{0, (rhogeneric[tau, theta, 0] - 
     rhogeneric[tau - 1, theta, 0])/theta^2}, {tau, 3, 4}, {theta, 0, 
  Pi}, AxesLabel -> Automatic]
\end{verbatim}
\medskip
Minimal value for the derivative in the third round
\medskip
\begin{verbatim}
(S0arc[Pi, 0] (g[3] - g[2]) + 
   S1arc[Pi, 0] (g[2 - tau1arc[Pi, 0]] - g[2 - tau1arc[Pi, 0]]) + 
   S2arc[Pi, 0] (g[2] - g[1]) + S4arc[Pi, 0] (g[2] - g[1]) + 
   S5arc[Pi, 0] (g[1] - 1))/(Pi)^2
\end{verbatim}
\begin{verbatim}
0.142304
\end{verbatim}
\medskip
Asymptotic behavior segment case: plot of Fig. \ref{Fig:final_growth_rhosegm}
\medskip
\begin{verbatim}
Plot[{2 (S0arc[theta, 0] + S1arc[theta, 0] + S2arc[theta, 0] + 
      S4arc[theta, 0] + S5arc[theta, 0])/theta^2}, {theta, 0, Pi}, 
 AxesLabel -> Automatic, 
 PlotLegends -> 
  Placed[ToExpression["2 \\frac{S_0 + S_1 + S_2 + S_3}{\\theta^2}", 
    TeXForm, HoldForm], {Right, Top}]]
\end{verbatim}
\medskip
Minimal value for the asymptotic value
\medskip
\begin{verbatim}
2 (S0arc[Pi, 0] + S1arc[Pi, 0] + S2arc[Pi, 0] + S4arc[Pi, 0] + 
    S5arc[Pi, 0])/(Pi)^2
\end{verbatim}
\begin{verbatim}
0.0816077
\end{verbatim}

\bigskip

\noindent{\bf Section \ref{Sss:additional_estim_opt_segm}}

\bigskip

\noindent Length of the perturbation, Equation \eqref{Equa:evolv_L_after_phi_2}
\medskip
\begin{verbatim}
Lrhosegment[tau_, theta_] := (1 - Cos[theta])/Sin[baralpha]^2 + 
  S0arc[theta, 0] Lg[tau] + 
  S1arc[theta, 0] Exp[barc (2 Pi + baralpha) tau1arc[theta, 0]] Lg[
    tau - tau1arc[theta, 0]] + 
  S2arc[theta, 0] Exp[barc (2 Pi + baralpha)] Lg[tau - 1] + 
  S4arc[theta, 0] Exp[barc (2 Pi + baralpha)] Lg[tau - 1] + 
  Darc[theta, 0] Exp[barc (2 Pi + baralpha)] HeavisideTheta[tau - 1]/
    Sin[baralpha] + 
  S5arc[theta, 0] Exp[barc (2 Pi + baralpha) 2] Lg[tau - 2]
\end{verbatim}
\medskip
Plot of Fig. \ref{Fig:Lpert_segm}
\medskip
\begin{verbatim}
Plot3D[{(rhogeneric[tau, theta, 0] - 
     2.08 Exp[-barc (2 Pi + baralpha) tau] Lrhosegment[tau - 1, 
       theta])/theta^2}, {tau, 1, 5}, {theta, 0.001, Pi}, 
 AxesLabel -> Automatic]
\end{verbatim}
\medskip
Here we verify some of the critical parts in order to avoid the Heaviside function
\medskip
\begin{verbatim}
Plot3D[{0, (rhogeneric[(1 - sigma) + sigma (1 + tau1arc[theta, 0]), 
      theta, 0] - 
     2.08 Exp[-barc (2 Pi + baralpha) ((1 - sigma) + 
          sigma (1 + tau1arc[theta, 0]))] Lrhosegment[(1 - sigma) + 
        sigma (1 + tau1arc[theta, 0]) - 1, theta])/theta^2}, {sigma, 
  0, 1}, {theta, 0, Pi}, AxesLabel -> Automatic]
\end{verbatim}
\begin{verbatim}
Plot3D[{0, (rhogeneric[(1 - sigma) (1 + tau1arc[theta, 0]) + sigma 2, 
      theta, 0] - 
     2.08 Exp[-barc (2 Pi + 
          baralpha) ((1 - sigma) (1 + tau1arc[theta, 0]) + 
          sigma 2)] Lrhosegment[(1 - sigma) (1 + tau1arc[theta, 0]) + 
        sigma 2 - 1, theta])/theta}, {sigma, 0, 1}, {theta, 0, Pi}, 
 AxesLabel -> Automatic]
\end{verbatim}
\medskip
This is to verify that the functional has positive derivative for $\tau \in [4,5]$, Fig. \ref{Fig:Lpert_segm}
\medskip
\begin{verbatim}
Plot3D[{0, ((rhogeneric[tau, theta, 0] - 
       2.08 Exp[-barc (2 Pi + baralpha) tau] Lrhosegment[tau - 1, 
         theta]) - (rhogeneric[tau - 1, theta, 0] - 
       2.08 Exp[-barc (2 Pi + baralpha)
          (tau - 1)] Lrhosegment[tau - 2, theta]))/theta^2}, {tau, 4, 
  5}, {theta, 0.0001, Pi}, AxesLabel -> Automatic]
\end{verbatim}
\medskip
Worst case
\medskip
\begin{verbatim}
Plot[{(rhogeneric[tau, Pi, 0] - 
     2.08 Exp[-barc (2 Pi + baralpha) tau] Lrhosegment[tau - 1, Pi])/
   Pi^2}, {tau, 1, 5}, AxesLabel -> Automatic, 
 PlotLegends -> 
  Placed[ToExpression[
    "\\frac{\\delta \\tilde r - 2.08 \\delta \\tilde L}{\\pi^2}", 
    TeXForm, HoldForm], {Right, Center}], PlotRange -> Full]
\end{verbatim}
\begin{verbatim}
(rhogeneric[2 - .001, Pi, 0] - 
   2.08 Exp[-barc (2 Pi + baralpha) (2 - .001)] Lrhosegment[
     2 - .001 - 1, Pi])/Pi^2
\end{verbatim}
\begin{verbatim}
{0.117205}
\end{verbatim}

\bigskip

\noindent{\bf Section \ref{S:arc}}

\bigskip

\noindent Study of the sign of the solution in the arc case, Proposition \ref{Prop:neg_region}. Function $h_1$ for first negativity region, Lemma \ref{Lem:study_g_1_arc}. Function $f_1$, Eq. \eqref{Equa:f_1_arc_first_neg}
\medskip
\begin{verbatim}
f1[theta_, Deltaphi_] := 
 S0arc[theta, Deltaphi] Exp[
    Cot[baralpha] (2 Pi + baralpha - Deltaphi - theta)] - 1
\end{verbatim}
\medskip
Derivative of $f_1$, Eq. \eqref{Equa:partialf_1_arc_first_neg}
\medskip
\begin{verbatim}
partialf1[theta_, 
  Deltaphi_] := -S0arc[theta, Deltaphi] Cot[baralpha] Exp[
    Cot[baralpha] (2 Pi + baralpha - Deltaphi - theta)] + 
  Cot[baralpha] (Sin[theta] + Deltaphi Cos[theta]) Exp[
     Cot[baralpha] (2 Pi + baralpha - Deltaphi - theta)]/Sin[baralpha]
\end{verbatim}
\medskip
Numerical plot of Fig. \ref{Fig:f_1_arc_first_neg}
\medskip
\begin{verbatim}
Plot3D[{0, f1[theta, Deltaphi]}, {theta, Pi/4, Pi}, {Deltaphi, 0, 
  Tan[baralpha]}, AxesLabel -> Automatic]
\end{verbatim}
\medskip
Minimal value
\smed
\begin{verbatim}
f1[Pi/4, 0]
\end{verbatim}
\begin{verbatim}
1.08109
\end{verbatim}
\begin{verbatim}
f1[Pi, Tan[baralpha]]
\end{verbatim}
\begin{verbatim}
0.971097
\end{verbatim}
\smed
Numerical plot of Fig. \ref{Fig:partialf_1_arc_first_neg}
\smed
\begin{verbatim}
Plot3D[{0, partialf1[theta, Deltaphi]/(theta + Deltaphi)}, {theta, 0, 
  Pi/4}, {Deltaphi, 0, Tan[baralpha]}]
\end{verbatim}
\smed
Minimum value
\smed
\begin{verbatim}
partialf1[Pi/4, Tan[baralpha]]/(Pi/4 + Tan[baralpha])
\end{verbatim}
\begin{verbatim}
1.29579
\end{verbatim}
Numerical plot of $h_1(\Delta \phi)$, Fig. \ref{Fig:h1Deltaphi_plot}
\smed
\begin{verbatim}
ContourPlot[
 f1[theta, Deltaphi] == 0, {Deltaphi, 0, Tan[baralpha]}, {theta, 
  0, .6}, Axes -> True, AxesLabel -> Automatic, Frame -> None, 
 PlotLegends -> 
  Placed[ToExpression["h_1(\\Delta \\phi)", TeXForm, 
    HoldForm], {Center, Top}]]
\end{verbatim}
\begin{verbatim}
ContourPlot[
 f1[theta - Deltaphi, Deltaphi] == 0, {Deltaphi, 0, 
  Tan[baralpha]}, {theta, 0, Tan[baralpha] + .1}, Axes -> True, 
 AxesLabel -> Automatic, Frame -> None, 
 PlotLegends -> 
  Placed[ToExpression["h_1(\\Delta \\phi) + \\Delta \\phi", TeXForm, 
    HoldForm], {Center, Top}]]
\end{verbatim}
\smed
Maximal value of $\omega = \Delta \phi + h_1(\Delta \phi)$
\smed
\begin{verbatim}
FindRoot[f1[theta, Tan[baralpha]], {theta, .15}]
\end{verbatim}
\begin{verbatim}
{theta -> 0.117533}
\end{verbatim}
\begin{verbatim}
Tan[baralpha] + 0.11753300046278962`
\end{verbatim}
\begin{verbatim}
2.53316
\end{verbatim}
\smed
This the implicit derivative of $h_1$, Eq. \eqref{Equa:dh1_dDeltaphi} and Fig. \ref{Fig:dh1_dDeltaphi}
\smed
\begin{verbatim}
dh1[theta_, Deltaphi_] := -1 + 
  Sin[baralpha] Cos[
    theta] Deltaphi/(Cos[baralpha - theta] - Cos[baralpha] + 
      Sin[baralpha - theta] Deltaphi)
\end{verbatim}
\smed
Maximal positive value
\smed
\begin{verbatim}
dh1[hatheta, Tan[baralpha]]
\end{verbatim}
\begin{verbatim}
0.0251844
\end{verbatim}
\smed
This the implicit second derivative of $h_1$, Eq. \eqref{Equa:d2h1_dDeltaphi} and Fig. \ref{Fig:d2h1_dDeltaphi}
\smed
\begin{verbatim}
d2h1[theta_, Deltaphi_] := 
 Sin[baralpha] (Cos[baralpha - theta] - Cos[baralpha]) Cos[
     theta]/(Cos[baralpha - theta] - Cos[baralpha] + 
       Sin[baralpha - theta] Deltaphi)^2 + 
  Sin[baralpha] Deltaphi (-Sin[baralpha] - Cos[baralpha] Sin[theta] + 
      Cos[baralpha] Deltaphi)/(Cos[baralpha - theta] - Cos[baralpha] +
        Sin[baralpha - theta] Deltaphi)^2 dh1[theta, Deltaphi]
\end{verbatim}
\begin{verbatim}
Plot3D[{0, d2h1[theta, Deltaphi]}, {theta, 0.1, hatheta}, {Deltaphi, 
  0, Tan[baralpha]}, PlotRange -> {-.1, 4}]
\end{verbatim}
\smed
Minimal value
\smed
\begin{verbatim}
Plot[{0, d2h1[theta, Tan[baralpha]]}, {theta, 0.1, hatheta}]
\end{verbatim}
\begin{verbatim}
d2h1[0.1, Tan[baralpha]]
\end{verbatim}
\begin{verbatim}
0.0168525
\end{verbatim}
Function $h_2$ for second negativity region. Second function $f_2$ for negative region, Eq. \eqref{Equa:f2_arc} and Fig. \ref{Fig:f_2_arc}
\smed
\begin{verbatim}
f2[theta_, Deltaphi_] := 
 f1[theta, Deltaphi] Exp[Cot[baralpha] theta] + Cos[theta]
\end{verbatim}
\begin{verbatim}
Plot3D[{0, f2[theta, Deltaphi]}, {theta, 0, .1}, {Deltaphi, 0, .05}, 
 AxesLabel -> Automatic]
\end{verbatim}
\smed
Function $f_3$ for second negative region, Eq. \eqref{Equa:f3_arc} and Fig. \ref{Fig:f_3_arc}
\smed
\begin{verbatim}
f3[theta_, Deltaphi_] := 
 Cot[baralpha] (1 - Cos[theta] + Sin[theta] Deltaphi) Exp[
     Cot[baralpha] (2 Pi + baralpha - Deltaphi)]/Sin[baralpha] - 
  Exp[Cot[baralpha] theta] + Cos[theta] - Sin[theta] Tan[baralpha]
\end{verbatim}
\begin{verbatim}
Plot3D[{0, f3[theta, Deltaphi]}, {theta, 0, .7}, {Deltaphi, 0, .4}, 
 AxesLabel -> Automatic]
\end{verbatim}
\smed
Implicit definition of the second negativity region at t = 1-, function $f_4$ of Eq. \eqref{Equa:f4_arc} and Fig. \ref{Fig:f_4_arc} and \ref{Fig:f234}
\smed
\begin{verbatim}
f4[theta_, Deltaphi_] := 
 Cot[baralpha] (1 - Cos[theta] + Sin[theta] Deltaphi) Exp[
     Cot[baralpha] (2 Pi + baralpha)]/Sin[baralpha] - 
  Exp[Cot[baralpha] (theta + Deltaphi)] + 
  Cos[theta] Exp[Cot[baralpha] Deltaphi] - 
  Sin[theta] Tan[baralpha] (Exp[Cot[baralpha] Deltaphi] - 1)
\end{verbatim}
\begin{verbatim}
Plot3D[{0, f4[theta, Deltaphi]}, {theta, 0, .1}, {Deltaphi, 0, .05}, 
 AxesLabel -> Automatic]
\end{verbatim}
\begin{verbatim}
ContourPlot[{f2[theta/3, Deltaphi/3] == 0, f3[theta, Deltaphi] == 0, 
  f4[theta/3, Deltaphi/3] == 0}, {Deltaphi, 0, .4}, {theta, 0, .7}, 
 AxesLabel -> Automatic, PlotLegends -> {"f2=0", "f3=0", "f4=0"}, 
 Frame -> None, Axes -> True]
\end{verbatim}
\smed
Computations for the second round arc case. Interval $[1,1+tau_1]$: numerical plot of $\rho/(\theta(\theta + \Delta \phi))$, Fig. \ref{Fig:rho_first_1_arc}
\smed
\begin{verbatim}
Plot3D[{0, 
  Evaluate@
   Table[rhogeneric[1.001 + tau tau1arc[theta, Deltaphi], theta, 
      Deltaphi]/(theta (theta + Deltaphi)), {tau, 0, 1, .2}]}, {theta,
   0, Pi}, {Deltaphi, 0, Tan[baralpha]}, PlotLegends -> Automatic, 
 AxesLabel -> Automatic]
\end{verbatim}
\smed
The function $\rho/(\theta(\theta + \Delta \phi))$ for $\theta = \pi/2$, minimal value
\smed
\begin{verbatim}
Plot3D[rhogeneric[tau, Pi, Deltaphi]/(Pi (Pi + Deltaphi)), {Deltaphi, 
  0, Tan[baralpha]}, {tau, 1, 1 + tau1arc[Pi, Deltaphi]}, 
 AxesLabel -> Automatic]
\end{verbatim}
\smed
The function $\rho/(\theta(\theta + \Delta \phi))$ for $\theta = \pi/2$, $\Delta \phi = \tan(\bar \alpha)$, minimal value Fig. \ref{Fig:rho_first_1_arc_worst}
\smed
\begin{verbatim}
Plot[rhogeneric[tau, Pi, 
   Tan[baralpha]]/(Pi (Pi + Tan[baralpha])), {tau, 1, 
  1 + tau1arc[Pi, Tan[baralpha]]}, AxesLabel -> Automatic, 
 PlotLegends -> 
  Placed[ToExpression[
    "\\frac{\\rho(\\tau,\\pi,\\tan(\\bar \
\\alpha))}{\\pi(\\pi+\\tan(\\bar \\alpha))}", TeXForm, 
    HoldForm], {Center, Top}]]
\end{verbatim}
\smed
Minimal value
\smed
\begin{verbatim}
rhogeneric[1 + tau1arc[Pi, Tan[baralpha] - .001], Pi, 
  Tan[baralpha]]/(Pi (Pi + Tan[baralpha]))
\end{verbatim}
\begin{verbatim}
{0.0406532}
\end{verbatim}
\smed
Interval $[1 + \tau_1,1 + \tau_2]$: here the singularity is due to the fact that at the previous round it converges to $-1$, and at the next round is like $\theta$
\smed
\begin{verbatim}
Plot3D[{0, 
  Evaluate@
   Table[rhogeneric[(1 - tau) (1 + tau1arc[theta, Deltaphi]) + 
       tau (1 + tau2arc[Deltaphi]), theta, 
      Deltaphi]/(theta (theta + Deltaphi)), {tau, 0, 1, .2}]}, {theta,
   0, Pi}, {Deltaphi, 0, Tan[baralpha]}, PlotLegends -> Automatic, 
 AxesLabel -> Automatic]
\end{verbatim}
\smed
This is for $\theta = \pi$
\smed
\begin{verbatim}
Plot3D[rhogeneric[(1 - tau) (1 + tau1arc[Pi, Deltaphi]) + 
    tau (1 + tau2arc[Deltaphi]), Pi, 
   Deltaphi]/(Pi (Pi + Deltaphi)), {tau, 0, 1}, {Deltaphi, 0, 
  Tan[baralpha]}, AxesLabel -> Automatic]
\end{verbatim}
\smed
This is for $\theta = \pi$, $\Delta \phi = \tan(\bar \alpha)$, Fig. \ref{Fig:rho_first_2_arc_worst}
\smed
\begin{verbatim}
Plot[rhogeneric[(1 - tau) (1 + tau1arc[Pi, Tan[baralpha]]) + 
    tau (1 + tau2arc[Tan[baralpha]]), Pi, 
   Tan[baralpha]]/(Pi (Pi + Tan[baralpha])), {tau, 0, 1}, 
 AxesLabel -> Automatic, 
 PlotLegends -> 
  Placed[ToExpression[
    "\\frac{\\rho(\\tau,\\pi,\\tan(\\bar \
\\alpha))}{\\pi(\\pi+\\tan(\\alpha))}", TeXForm, HoldForm], {Center, 
    Right}]]
\end{verbatim}
\smed
Interval $[1+\tau_2,2]$, Fig. \ref{Fig:rho_first_3_arc}, analysis of the minimal values and Fig. \ref{Fig:rho_first_3_arc_worst}
\smed
\begin{verbatim}
Plot3D[{0, 
  Evaluate@
   Table[rhogeneric[(1 - tau) (1 + tau2arc[Deltaphi]) + tau (2), 
      theta, Deltaphi]/(theta (theta + Deltaphi)), {tau, 0, 
     1, .2}]}, {theta, 0, Pi}, {Deltaphi, 0, Tan[baralpha]}, 
 PlotLegends -> Automatic, AxesLabel -> Automatic]
\end{verbatim}
\begin{verbatim}
Plot3D[{rhogeneric[(1 - tau) (1 + tau2arc[Deltaphi]) + tau (2), Pi, 
    Deltaphi]/(Pi (Pi + Deltaphi))}, {tau, 0, 1}, {Deltaphi, 0, 
  Tan[baralpha]}, AxesLabel -> Automatic]
\end{verbatim}
\begin{verbatim}
Plot[rhogeneric[tau, Pi, 
   Tan[baralpha]]/(Pi (Pi + Tan[baralpha])), {tau, 
  1 + tau2arc[Tan[baralpha]], 2}, AxesLabel -> Automatic, 
 PlotLegends -> 
  Placed[ToExpression[
    "\\frac{\\rho(\\tau,\\pi,\\tan(\\bar \
\\alpha))}{\\pi(\\pi+\\tan(\\bar\\alpha))}", TeXForm, 
    HoldForm], {Center, Center}]]
\end{verbatim}
\smed
Minimal value
\smed
\begin{verbatim}
rhogeneric[2 - .001, Pi, Tan[baralpha]]/(Pi (Pi + Tan[baralpha]))
\end{verbatim}
\begin{verbatim}
{0.0406695}
\end{verbatim}
\smed
Computations for the third round arc case. Interval $[2,2+\tau_1]$: Fig. \ref{Fig:rho_third_1_arc}, plot of the minimal values and Fig. \ref{Fig:rho_third_1_arc_worst}
\smed
\begin{verbatim}
Plot3D[{0, 
  Evaluate@
   Table[rhogeneric[(1 - tau) 2.01 + 
       tau (2 + tau1arc[theta, Deltaphi]), theta, 
      Deltaphi]/(theta (theta + Deltaphi)), {tau, 0, 1, .2}]}, {theta,
   0, Pi}, {Deltaphi, 0, Tan[baralpha]}, AxesLabel -> Automatic]
\end{verbatim}
\begin{verbatim}
Plot3D[{0, 
  rhogeneric[(1 - tau) 2 + tau (2 + tau1arc[Pi, Deltaphi]), Pi, 
    Deltaphi]/(Pi (Pi + Deltaphi))}, {tau, 0, 1}, {Deltaphi, 0, 
  Tan[baralpha]}, AxesLabel -> Automatic]
\end{verbatim}
\begin{verbatim}
Plot[rhogeneric[tau, Pi, 
   Tan[baralpha]]/(Pi (Pi + Tan[baralpha])), {tau, 2, 
  2 + tau1arc[Pi, Tan[baralpha]]}, AxesLabel -> Automatic, 
 PlotLegends -> 
  Placed[ToExpression[
    "\\frac{\\rho(\\tau,\\pi,\\tan(\\bar \
\\alpha))}{\\pi(\\pi+\\tan(\\bar\\alpha))}", TeXForm, 
    HoldForm], {Center, Top}]]
\end{verbatim}
\smed
Interval $[2+\tau_1,2+\tau_2]$: Fig. \ref{Fig:rho_third_2_arc} and \ref{Fig:rho_first_2_arc_worst}
\smed
\begin{verbatim}
Plot3D[{0, 
  Evaluate@
   Table[rhogeneric[(1 - tau) (2 + tau1arc[theta, Deltaphi]) + 
       tau (2 + tau2arc[Deltaphi]), theta, 
      Deltaphi]/(theta (theta + Deltaphi)), {tau, 0, 1, .2}]}, {theta,
   0, Pi}, {Deltaphi, 0, Tan[baralpha]}, PlotLegends -> Automatic, 
 AxesLabel -> Automatic]
\end{verbatim}
\begin{verbatim}
Plot[rhogeneric[tau, Pi, 
   Tan[baralpha]]/(Pi (Pi + Tan[baralpha])), {tau, 
  2 + tau1arc[Pi, Tan[baralpha]], 2 + tau2arc[Tan[baralpha]]}, 
 AxesLabel -> Automatic, 
 PlotLegends -> 
  Placed[ToExpression[
    "\\frac{\\rho(\\tau,\\pi,\\tan(\\bar \
\\alpha))}{\\pi(\\pi+\\tan(\\bar\\alpha))}", TeXForm, 
    HoldForm], {Right, Top}]]
\end{verbatim}
\begin{verbatim}
rhogeneric[2 + tau2arc[Tan[baralpha]], Pi, 
  Tan[baralpha]]/(Pi (Pi + Tan[baralpha]))
\end{verbatim}
\begin{verbatim}
{0.0407434}
\end{verbatim}
\smed
Interval $[2+\tau_2,3]$: Fig. \ref{Fig:rho_third_3_arc} and \ref{Fig:rho_third_3_arc_worst}
\smed
\begin{verbatim}
Plot3D[
 {0, Evaluate@
   Table[rhogeneric[(1 - tau) (2 + tau2arc[Deltaphi]) + 
       tau (1 + 1.99), theta, 
      Deltaphi]/(theta (theta + Deltaphi)), {tau, 0, 1, .2}]}, {theta,
   0, Pi}, {Deltaphi, 0, Tan[baralpha]}, PlotLegends -> Automatic, 
 AxesLabel -> Automatic]
\end{verbatim}
\begin{verbatim}
Plot[rhogeneric[tau, Pi, 
   Tan[baralpha]]/(Pi (Pi + Tan[baralpha])), {tau, 
  2 + tau2arc[Tan[baralpha]], 3}, AxesLabel -> Automatic, 
 PlotLegends -> 
  Placed[ToExpression[
    "\\frac{\\rho(\\tau,\\pi,\\tan(\\bar \
\\alpha))}{\\pi(\\pi+\\tan(\\bar\\alpha))}", TeXForm, 
    HoldForm], {Right, Top}]]
\end{verbatim}
\begin{verbatim}
rhogeneric[3, Pi, Tan[baralpha]]/(Pi (Pi + Tan[baralpha]))
\end{verbatim}
\begin{verbatim}
{0.0405639}
\end{verbatim}
\smed
After third round arc case we need to find the region where the function is strictly increasing: Fig. \ref{Fig:rho_fourth_1_arc}, \ref{Fig:rho_fourth_2_arc}, \ref{Fig:rho_fourth_3_arc}
\smed
\begin{verbatim}
Plot3D[{0, 
  Evaluate@
   Table[(rhogeneric[(1 - tau) (3.01) + tau (3.99), theta, Deltaphi] -
        rhogeneric[(1 - tau) (2.01) + tau (2.99), theta, 
        Deltaphi])/(theta (theta + Deltaphi)), {tau, 0, 
     1, .2}]}, {theta, 0, Pi}, {Deltaphi, 0, Tan[baralpha]}, 
 AxesLabel -> Automatic]
\end{verbatim}
\begin{verbatim}
Plot3D[{0, 
  Evaluate@
   Table[(rhogeneric[(1 - tau) (3.01) + tau (3.99), theta, Deltaphi] -
        rhogeneric[(1 - tau) (2.01) + tau (2.99), theta, 
        Deltaphi])/(theta (theta + Deltaphi)), {tau, 0, 
     1, .2}]}, {theta, Pi - .04, Pi}, {Deltaphi, Tan[baralpha] - .04, 
  Tan[baralpha]}, AxesLabel -> Automatic]
\end{verbatim}
\begin{verbatim}
Plot3D[{0, (rhogeneric[(3.01), theta, Deltaphi] - 
     rhogeneric[(2.01), theta, 
      Deltaphi])/(theta (theta + Deltaphi))}, {theta, Pi - .04, 
  Pi}, {Deltaphi, Tan[baralpha] - .04, Tan[baralpha]}, 
 AxesLabel -> Automatic]
\end{verbatim}
\smed
Asymptotic behavior in the arc case: Fig. \ref{Fig:asympt_arc}
\smed
\begin{verbatim}
Plot3D[{0, (S0arc[theta, Deltaphi] + S1arc[theta, Deltaphi] + 
     S2arc[theta, Deltaphi] + 
     S3arc[theta] (Exp[-barc (2 Pi + baralpha) tau2arc[Deltaphi]] - 
         Exp[-barc (2 Pi + baralpha)])/barc + S4arc[theta, Deltaphi] +
      S5arc[theta, Deltaphi])/(theta (theta + Deltaphi))}, {theta, 0, 
  Pi}, {Deltaphi, 0, Tan[baralpha]}]
\end{verbatim}
\begin{verbatim}
Plot3D[{0, 
  S0arc[theta, Deltaphi] + S1arc[theta, Deltaphi] + 
   S2arc[theta, Deltaphi] + 
   S3arc[theta] (Exp[-barc (2 Pi + baralpha) tau2arc[Deltaphi]] - 
       Exp[-barc (2 Pi + baralpha)])/barc + S4arc[theta, Deltaphi] + 
   S5arc[theta, Deltaphi]}, {theta, Pi - .01, Pi}, {Deltaphi, 
  Tan[baralpha] - .03, Tan[baralpha]}]
\end{verbatim}

\sbig

\noindent{\bf Section \ref{Sss:additional_estim_opt_arc}}

\sbig

\noindent This is the variation of length in the arc case \eqref{Equa:lenght_arc}
\smed
\begin{verbatim}
Lengtharc[tau_, theta_, 
  Deltaphi_] := (1 - Cos[theta] + Sin[theta] Deltaphi)/
   Sin[baralpha]^2 + S0arc[theta, Deltaphi] Lg[tau] + 
  S1arc[theta, Deltaphi] Exp[
    barc (2 Pi + baralpha) tau1arc[theta, Deltaphi]] Lg[
    tau - tau1arc[theta, Deltaphi]] + 
  S2arc[theta, Deltaphi] Exp[
    barc (2 Pi + baralpha) tau2arc[Deltaphi]] Lg[
    tau - tau2arc[Deltaphi]] + 
  S3arc[theta] Lm[tau, tau2arc[Deltaphi], 1] + 
  S4arc[theta, Deltaphi] Exp[barc (2 Pi + baralpha)] Lg[tau - 1] + 
  Darc[theta, Deltaphi] Exp[
    barc (2 Pi + baralpha)] HeavisideTheta[tau - 1] /Sin[baralpha] + 
  S5arc[theta, Deltaphi] Exp[barc (2 Pi + baralpha) 2] Lg[tau - 2]
\end{verbatim}
\smed
First plot of Fig. \ref{Fig:Lpert_arc}
\smed
\begin{verbatim}
Plot3D[{Evaluate@
   Table[(rhogeneric[tau, theta, Deltaphi] - 
       2.08 Exp[-barc (2 Pi + baralpha) tau] Lengtharc[tau - 1, theta,
          Deltaphi])/(theta (theta + Deltaphi)), {Deltaphi, 0, 
     Tan[baralpha], .5}]}, {theta, 0.001, Pi/2}, {tau, 1, 5}, 
 AxesLabel -> Automatic, PlotRange -> Full]
\end{verbatim}
\smed
A more detailed verification of the positivity: the worst case is for $\Delta \phi = 0$
\smed
\begin{verbatim}
Plot3D[{0, 
  Evaluate@
   Table[(rhogeneric[(1 - sigma) + 
         sigma (1 + tau1arc[theta, Deltaphi]), theta, Deltaphi] - 
       2.08 Exp[-barc (2 Pi + baralpha) ((1 - sigma) + 
            sigma (1 + tau1arc[theta, Deltaphi]))] Lengtharc[(1 - 
            sigma) + sigma (1 + tau1arc[theta, Deltaphi]) - 1, theta, 
         Deltaphi])/(theta (theta + Deltaphi)), {Deltaphi, 0, 
     Tan[baralpha], .5}]}, {theta, 0, Pi/2}, {sigma, 0, 1}, 
 AxesLabel -> Automatic]
\end{verbatim}
\begin{verbatim}
Plot3D[{0, 
  Evaluate@
   Table[(rhogeneric[(1 - sigma) (1 + tau1arc[theta, Deltaphi]) + 
         sigma (1 + tau2arc[Deltaphi]), theta, Deltaphi] - 
       2.08 Exp[-barc (2 Pi + 
            baralpha) ((1 - sigma) (1 + tau1arc[theta, Deltaphi]) + 
            sigma (1 + tau2arc[Deltaphi]))] Lengtharc[(1 - sigma) (1 +
              tau1arc[theta, Deltaphi]) + 
          sigma (1 + tau2arc[Deltaphi]) - 1, theta, 
         Deltaphi])/(theta), {Deltaphi, 0, 
     Tan[baralpha], .5}]}, {theta, 0, Pi/2}, {sigma, 0, 1}, 
 AxesLabel -> Automatic]
\end{verbatim}
\begin{verbatim}
Plot3D[{0, 
  Evaluate@
   Table[(rhogeneric[(1 - sigma) (1 + tau2arc[Deltaphi]) + 2 sigma, 
        theta, Deltaphi] - 
       2.08 Exp[-barc (2 Pi + 
            baralpha) ((1 - sigma) (1 + tau2arc[Deltaphi]) + 
            2 sigma)] Lengtharc[(1 - sigma) (1 + tau2arc[Deltaphi]) + 
          2 sigma - 1, theta, Deltaphi])/(theta (theta + 
         Deltaphi)), {Deltaphi, 0.001, Tan[baralpha], .5}]}, {theta, 
  0, Pi/2}, {sigma, 0, 1}, AxesLabel -> Automatic]
\end{verbatim}
\smed
Worst case and second plot of Fig. \ref{Fig:Lpert_arc}
\smed
\begin{verbatim}
Plot3D[{0, (rhogeneric[tau, Pi/2, Deltaphi] - 
     2.08 Exp[-barc (2 Pi + baralpha) tau] Lengtharc[tau - 1, Pi/2, 
       Deltaphi])/(Pi/2 (Pi/2 + Deltaphi))}, {Deltaphi, 0, 
  Tan[baralpha]}, {tau, 1, 5}, AxesLabel -> Automatic]
\end{verbatim}
\begin{verbatim}
Plot[{0, (rhogeneric[tau, Pi/2, 0] - 
     2.08 Exp[-barc (2 Pi + baralpha) tau] Lengtharc[tau - 1, Pi/2, 
       0])/(Pi/2 (Pi/2 + 0))}, {tau, 1, 5}, AxesLabel -> Automatic, 
 PlotLegends -> 
  Placed[ToExpression[
    "\\rho(\\tau,\\pi/2,0) - 2.08 e^{-\\bar c (2\\pi + \\bar \\alpha) \
\\tau} \\delta \\tilde L(\\tau-1)", TeXForm, HoldForm], {Right, 
    Center}]]
\end{verbatim}
\begin{verbatim}
Plot[{0, (rhogeneric[tau, Pi/2, Tan[baralpha]] - 
     2.08 Exp[-barc (2 Pi + baralpha) tau] Lengtharc[tau - 1, Pi/2, 
       Tan[baralpha]])/(Pi/2 (Pi/2 + Tan[baralpha]))}, {tau, 1, 5}, 
 AxesLabel -> Automatic]
\end{verbatim}
\smed
Increasing for $\tau \in [4,5]$, third plot of Fig. \ref{Fig:Lpert_arc}
\smed
\begin{verbatim}
Plot3D[{Evaluate@
   Table[(rhogeneric[tau, theta, Deltaphi] - 
       2.08 Exp[-barc (2 Pi + baralpha) tau] Lengtharc[tau - 1, theta,
          Deltaphi] - (rhogeneric[tau - 1, theta, Deltaphi] - 
         2.08 Exp[-barc (2 Pi + baralpha) (tau - 1)] Lengtharc[
           tau - 2, theta, Deltaphi]))/(theta (theta + 
         Deltaphi)), {Deltaphi, 0, Tan[baralpha], .5}]}, {theta, 
  0.001, Pi/2}, {tau, 4, 5}, AxesLabel -> Automatic]
\end{verbatim}
\smed
Worst case for increasing estimates
\smed
\begin{verbatim}
Plot3D[{0, (rhogeneric[tau, Pi/2, Deltaphi] - 
     2.08 Exp[-barc (2 Pi + baralpha) tau] Lengtharc[tau - 1, Pi/2, 
       Deltaphi] - (rhogeneric[tau - 1, Pi/2, Deltaphi] - 
       2.08 Exp[-barc (2 Pi + baralpha) (tau - 1)] Lengtharc[tau - 2, 
         Pi/2, Deltaphi]))/(Pi/2 (Pi/2 + Deltaphi))}, {Deltaphi, 0, 
  Tan[baralpha]}, {tau, 4, 5}, AxesLabel -> Automatic]
\end{verbatim}
\begin{verbatim}
Plot[{0, (rhogeneric[tau, Pi/2, 0] - 
     2.08 Exp[-barc (2 Pi + baralpha) tau] Lengtharc[tau - 1, Pi/2, 
       0] - (rhogeneric[tau - 1, Pi/2, 0] - 
       2.08 Exp[-barc (2 Pi + baralpha) (tau - 1)] Lengtharc[tau - 2, 
         Pi/2, 0]))/(Pi/2 (Pi/2 + 0))}, {tau, 4, 5}, 
 AxesLabel -> Automatic]
\end{verbatim}
\begin{verbatim}
Plot[{0, (rhogeneric[tau, Pi/2, Tan[baralpha]] - 
     2.08 Exp[-barc (2 Pi + baralpha) tau] Lengtharc[tau - 1, Pi/2, 
       Tan[baralpha]] - (rhogeneric[tau - 1, Pi/2, Tan[baralpha]] - 
       2.08 Exp[-barc (2 Pi + baralpha) (tau - 1)] Lengtharc[tau - 2, 
         Pi/2, Tan[baralpha]]))/(Pi/2 (Pi/2 + Tan[baralpha]))}, {tau, 
  4, 5}, AxesLabel -> Automatic]
\end{verbatim}
\begin{verbatim}
(rhogeneric[4.0001, Pi/2, Tan[baralpha]] - 
   2.08 Exp[-barc (2 Pi + baralpha) 4.0001] Lengtharc[4.0001 - 1, 
     Pi/2, Tan[baralpha]] - (rhogeneric[4.0001 - 1, Pi/2, 
      Tan[baralpha]] - 
     2.08 Exp[-barc (2 Pi + baralpha) (4.0001 - 1)] Lengtharc[
       4.0001 - 2, Pi/2, Tan[baralpha]]))/(Pi/
    2 (Pi/2 + Tan[baralpha]))
\end{verbatim}
\begin{verbatim}
{0.000786574}
\end{verbatim}

\sbig

\noindent{\bf Section \ref{S:optimal_solut}}

\sbig

\noindent Remark \ref{Rem:trans_segm_arc}
\smed
\begin{verbatim}
ArcSin[Tan[baralpha]]
\end{verbatim}
\begin{verbatim}
0.4226813
\end{verbatim}
\smed
Formulas for tent, Lemma \ref{Lem:relation_admissibility_tent}: these are the formulas for computing the final tent: here the entries are with the angle computed for $DeltaQ$
\smed
\begin{verbatim}
ell0tent[hatr0tent_, hatLtent_, 
  DeltaQtent_] := ((1 - Cos[hatheta]) hatr0tent + DeltaQtent - 
    hatLtent)/(Cos[baralpha - hatheta] - Cos[baralpha])
\end{verbatim}
\begin{verbatim}
hatr2tent[hatr0tent_, hatell0tent_, 
  DeltaQtentPi2_] := (hatell0tent Sin[hatheta] + 
    hatr0tent Sin[baralpha + hatheta] - DeltaQtentPi2)/Sin[baralpha]
\end{verbatim}
\begin{verbatim}
ell1tent[r2tent_, hatLtent_, hatr0tent_, 
  hatell0tent_] := (r2tent + hatLtent - hatr0tent - 
    Cos[baralpha] hatell0tent)/Cos[baralpha]
\end{verbatim}
\begin{verbatim}
hatrbarphi[hatr2tentv_, r2v_, rbarphiv_, 
  hatell0tent_] := (hatr2tentv - r2v) Exp[
    Cot[baralpha] (2 Pi + baralpha - hatheta)] + rbarphiv - 
  hatell0tent
\end{verbatim}
\smed
Time shift for tent, Eq. \eqref{Equa:hattau_def}
\smed
\begin{verbatim}
hattau := hatheta/(2 Pi + baralpha)
\end{verbatim}

\sbig

\noindent {\bf Admissibility of tent, Section \ref{Sss:having_tent_check}}: admissibility constant $K_\mathrm{tent}$, Lemma \ref{Lem:when_tent}

\sbig

\begin{verbatim}
Kadmitent = 
 Sin[baralpha] (Sin[baralpha + hatheta] - 
     Sin[baralpha])/(Cos[hatheta] (1 - Sin[baralpha]))
\end{verbatim}
\begin{verbatim}
0.966795
\end{verbatim}
\smed
Plot of the function \eqref{Equa:Loverr0_tent}, Fig. \ref{Fig:Loverr0_tent}
\smed
\begin{verbatim}
Plot[{Kadmitent, 
  Sin[baralpha] (Cos[hatphi] - 
      Cos[hatheta - hatphi])/(Sin[
       baralpha - hatphi] (1 - Cos[hatheta - hatphi]))}, {hatphi, 
  0, .25}, AxesLabel -> Automatic, 
 PlotLegends -> 
  Placed[ToExpression[
    "\\frac{\\sin(\\bar \\alpha) (\\cos(\\hat \\phi) - \\cos(\\hat \
\\theta - \\hat \\phi))}{\\sin(\\bar \\alpha - \\hat \\phi) (1 - \
\\cos(\\hat \\theta - \\hat \\phi))}", TeXForm, HoldForm], {Left, 
    Center}]]
\end{verbatim}
\smed
Computations for \eqref{Equa:hattau3_cond}
\smed
\begin{verbatim}
baralpha + hatheta - Pi/2 + 
 Tan[baralpha] Log[
   2.08/(Exp[Cot[baralpha] (baralpha + hatheta - Pi/2)]/Kadmitent - 
      1)]
\end{verbatim}
\begin{verbatim}
7.86127
\end{verbatim}
\begin{verbatim}
(baralpha + hatheta - Pi/2 + 
   Tan[baralpha] Log[
     2.08/(Exp[Cot[baralpha] (baralpha + hatheta - Pi/2)]/Kadmitent - 
        1)])/(2 Pi + baralpha)
\end{verbatim}
\begin{verbatim}
1.05358
\end{verbatim}
Computation for Remark \ref{Rem:inital_tent}
\smed
\begin{verbatim}
2 Cos[baralpha] - Cos[hatheta]
\end{verbatim}
\begin{verbatim}
-0.10964
\end{verbatim}

\sbig

\noindent{\bf Section \ref{S:optimal_sol_candidate}}

\sbig

\noindent {\bf Section \ref{Ss:optimal_cand_2}}

\sbig

\begin{verbatim}
checkellbase1[barphi_, hbase_] := 
 Sqrt[1 + hbase^2 + 2 hbase Cos[barphi - 2 Pi]]
\end{verbatim}
\begin{verbatim}
checkphibase1[barphi_, hbase_] := 
 ArcSin[hbase Sin[barphi - 2 Pi]/
    Sqrt[1 + hbase^2 + 2 hbase Cos[barphi]]]
\end{verbatim}
\smed
Formulas \eqref{Equa:opt_2_checkDeltaphi} and \eqref{Equa:ropt_2}
\smed
\begin{verbatim}
checkDeltaphibase1[barphi_, hbase_] := 
 Tan[baralpha] - hbase/Sqrt[1 + hbase^2 + 2 hbase Cos[barphi - 2 Pi]]
\end{verbatim}
\begin{verbatim}
checkrbarphibase1[barphi_, hbase_] := 
 checkellbase1[barphi, 
    hbase] Exp[
     Cot[baralpha] (barphi - (checkphibase1[barphi, hbase] + 
          checkDeltaphibase1[barphi, hbase] + Pi/2 - baralpha))]/
    Sin[baralpha] - Exp[Cot[baralpha] (barphi - 2 Pi)] - hbase
\end{verbatim}
\smed
Plot of Fig. \ref{Fig:opt_sol_1_2}
\smed
\begin{verbatim}
Plot3D[{Exp[-barc (barphi - Tan[baralpha] - Pi/2 + 
       baralpha)] checkrbarphibase1[barphi, hbase], 
  rhobase[(barphi - Tan[baralpha] - Pi/2 + baralpha)/(2 Pi + 
      baralpha)]}, {barphi, 2 Pi, 5 Pi/2}, {hbase, 0, .3}, 
 AxesLabel -> Automatic]
\end{verbatim}
\begin{verbatim}
Plot3D[{Exp[-barc (barphi - Tan[baralpha] - Pi/2 + 
       baralpha)] checkrbarphibase1[barphi, hbase], 
  rhobase[(barphi - Tan[baralpha] - Pi/2 + baralpha)/(2 Pi + 
      baralpha)]}, {barphi, 5 Pi/2 - .15, 5 Pi/2}, {hbase, 0, .24}, 
 AxesLabel -> Automatic]
\end{verbatim}
\smed
Computations for the initial point
\smed
\begin{verbatim}
2 Pi + Pi/2 - FindRoot[f1[theta, Tan[baralpha]], {theta, .15}][[1, 2]]
\end{verbatim}
\begin{verbatim}
7.73645
\end{verbatim}

\sbig

\noindent {\bf Section \ref{Ss:optimal_cand_3}}

\sbig

\begin{verbatim}
checkellbase2[hbase_] := Sqrt[1 + hbase^2 ]
\end{verbatim}
\begin{verbatim}
checkphibase2[hbase_] := ArcSin[hbase/checkellbase2[hbase]]
\end{verbatim}
\smed
Formula \eqref{Equa:opt_sol_2_arc_1}
\smed
\begin{verbatim}
checkDeltaphibase2[barphi_, hbase_] := 
 Tan[baralpha] - (barphi - 2 Pi - Pi/2 + hbase)/Sqrt[1 + hbase^2]
\end{verbatim}
\smed
Plots of Fig. \ref{Fig:opt_sol_2_arc_1}, \ref{Fig:opt_sol_2_arc_2}
\smed
\begin{verbatim}
Plot3D[{0, checkDeltaphibase2[barphi, hbase]}, {barphi, 5 Pi/2, 
  5 Pi/2 + Tan[baralpha]}, {hbase, 0, 3}, AxesLabel -> Automatic]
\end{verbatim}
\begin{verbatim}
Plot[{2 Pi + Pi/2 + Tan[baralpha] checkellbase2[hbase] - 
   hbase}, {hbase, 0, 1 }, AxesLabel -> Automatic, 
 PlotLegends -> 
  Placed[ToExpression[
    "2\\pi + \\frac{\\pi}{2} + \\tan(\\bar \\alpha) \\check \\ell - \
h", TeXForm, HoldForm], {Center, Top}]]
\end{verbatim}
\smed
Final value of the solution in the arc case, Eq. \eqref{Equa:opt_sol_3_finarc}
\smed
\begin{verbatim}
checkrbarphibase2[barphi_, hbase_] := 
 checkellbase2[
    hbase] Exp[
     Cot[baralpha] (2 Pi + 
        baralpha - (checkDeltaphibase2[barphi, hbase] + 
          checkphibase2[hbase]))]/Sin[baralpha] - 
  Exp[Cot[baralpha] (barphi - 2 Pi)] - (Exp[
      Cot[baralpha] (barphi - 2 Pi - Pi/2)] - 1)/Cot[baralpha] - hbase
\end{verbatim}
\smed
Plot of the boundary of the region Fig. \ref{Fig:opt_sol_2_arc_3}
\smed
\begin{verbatim}
Plot3D[{Exp[-barc (barphi - Pi/2 + baralpha - 
       Tan[baralpha])] checkrbarphibase2[barphi, hbase], 
  rhobase[(barphi - Pi/2 - Tan[baralpha] + baralpha)/(2 Pi + 
      baralpha)]}, {barphi, 5 Pi/2, 5 Pi/2 + Tan[baralpha]}, {hbase, 
  0, 1.5}, AxesLabel -> Automatic]
\end{verbatim}
\smed
For the tent case the length $h$ Eq. \eqref{Equa:opt_sol_3_h_tent} and Fig. \ref{Fig:opt_sol_2_tent_1}
\smed
\begin{verbatim}
checkhbase2[
  barphi_] := (Cos[baralpha] (barphi - 2 Pi - Pi/2) - 
    Sin[baralpha - hatheta])/(Cos[baralpha - hatheta] - Cos[baralpha])
\end{verbatim}
\begin{verbatim}
Plot[{checkhbase2[barphi], ArcTan[hatheta]}, {barphi, 2 Pi + Pi/2, 
  2 Pi + Pi/2 + Tan[baralpha]}, AxesLabel -> Automatic, 
 PlotLegends -> 
  Placed[ToExpression[
    "\\frac{\\cos(\\bar \\alpha) (\\bar \\phi - 2\\pi - \
\\frac{\\pi}{2}) - \\sin(\\bar \\alpha - \\hat \\theta)}{\\cos(\\bar \
\\alpha - \\hat \\theta) - \\cos(\\bar \\alpha)}", TeXForm, 
    HoldForm], {Right, Bottom}]]
\end{verbatim}
\smed
Computation of the beginning of the positivity region for the tent case
\smed
\begin{verbatim}
2 Pi + Pi/2 + 
 Tan[baralpha] - (Sin[baralpha] - Sin[baralpha - hatheta])/
  Cos[baralpha]
\end{verbatim}
\begin{verbatim}
9.48195
\end{verbatim}
\smed
Computation of admissibility of the tent Eq. \eqref{Equa:opt_sol_3_admi_tent}
\smed
\begin{verbatim}
Tan[hatheta]
\end{verbatim}
\begin{verbatim}
0.554294
\end{verbatim}
\begin{verbatim}
2 Pi + Pi/2 + Tan[baralpha]/Cos[hatheta] - Tan[hatheta]
\end{verbatim}
\begin{verbatim}
10.0616
\end{verbatim}
\smed
Position of the tent optimal value and the negativity region Fig. \ref{Fig:opt_sol_2_tent_2}
\smed
\begin{verbatim}
Plot[{5 Pi/
     2 + (Sin[baralpha] - 
      Sin[hatheta] Cos[baralpha])/(Cos[baralpha] Cos[hatheta]), 
  2 Pi + Pi/2 + Tan[baralpha] checkellbase2[hbase] - hbase, 
  2 Pi + Pi/2 + Sin[baralpha - hatheta]/Cos[baralpha] + 
   hbase (Cos[baralpha - hatheta] - Cos[baralpha])/
     Cos[baralpha]}, {hbase, 0, 1 }, AxesLabel -> Automatic]
\end{verbatim}
\smed
Final value of tent Eq. \eqref{Equa:opt_sol_3_tent_final} and Fig. \ref{Fig:opt_sol_2_tent_3}
\smed
\begin{verbatim}
checkrbarphibase2tent[barphi_, hbase_] := 
 Cos[ArcTan[hbase] - hatheta] Sqrt[
    1 + hbase^2] Exp[Cot[baralpha] (2 Pi + baralpha - hatheta)]/
    Sin[baralpha] - 
  Exp[Cot[baralpha] (barphi - 2 Pi)] - (Exp[
      Cot[baralpha] (barphi - 2 Pi - Pi/2)] - 1)/Cot[baralpha] - hbase
\end{verbatim}
\begin{verbatim}
Plot[Exp[-barc (barphi - Pi/2 + baralpha - 
      Tan[baralpha])] checkrbarphibase2tent[barphi, 
   checkhbase2[barphi]], {barphi, 
  2 Pi + Pi/2 + Tan[baralpha]/Cos[hatheta] - Tan[hatheta], 
  2 Pi + Pi/2 + Tan[baralpha]}, AxesLabel -> Automatic]
\end{verbatim}

\sbig

\noindent {\bf Section \ref{Ss:optimal_cand_4}}, Fig. \ref{Fig:opt_sol_3_pos}

\sbig

\begin{verbatim}
Plot[Kadmitent rhobase[tau] - 
  Exp[-barc (2 Pi + baralpha) tau] (Lbase[tau - 1 + hattau] - 
     Lbase[tau - 1]), {tau, 1, 5}]
\end{verbatim}

\sbig

\noindent {\bf Section \ref{Ss:optimal_asympt}}: asymptotic behavior for base solution with tent, computation of \gls{Kasymptosat}

\sbig

\begin{verbatim}
Kasympt = 
 2 (Exp[-barc (Tan[baralpha] + Pi/2 - baralpha)]/Sin[baralpha] - 
    Exp[-barc (2 Pi)] - (Exp[-barc (2 Pi + Pi/2)] - 
       Exp[-barc (2 Pi + Pi/2 + Tan[baralpha])])/barc - 
    Cot[baralpha] Exp[-barc (2 Pi + baralpha)])
\end{verbatim}
\smed
Vector for the base of the tent $\hat Q_2 - \hat Q_1$
\smed
\begin{verbatim}
Sqrt[(Exp[barc hatheta] Cos[hatheta] - 
      1)^2 + (Exp[barc hatheta] Sin[hatheta])^2]/(Sin[
    baralpha] (1 + barc^2))
\end{verbatim}
\begin{verbatim}
0.560686
\end{verbatim}
\begin{verbatim}
DeltaQhathetaaympt = (barc (Exp[barc hatheta] - Cos[hatheta]) + 
    Sin[hatheta])/(1 + barc^2)
\end{verbatim}
\begin{verbatim}
0.521632
\end{verbatim}
\begin{verbatim}
DeltaQell0 = DeltaQhathetaaympt/Sin[baralpha]
\end{verbatim}
\begin{verbatim}
0.564562
\end{verbatim}
\begin{verbatim}
DeltaQbarahathetasympt = (barc (Exp[barc hatheta] Sin[baralpha] - 
       Sin[baralpha + hatheta]) + (Exp[barc hatheta] Cos[baralpha] - 
      Cos[baralpha + hatheta]))/(1 + barc^2)
\end{verbatim}
\begin{verbatim}
0.532289
\end{verbatim}
\begin{verbatim}
DeltaQr2 = DeltaQbarahathetasympt/Sin[baralpha]
\end{verbatim}
\begin{verbatim}
0.576096
\end{verbatim}
\smed
Asymptotic length $L$
\smed
\begin{verbatim}
(Exp[barc hatheta] - 1)/barc
\end{verbatim}
\begin{verbatim}
0.543747
\end{verbatim}
\begin{verbatim}
DeltaL = (Exp[barc hatheta] - 1)/(Sin[baralpha] barc)
\end{verbatim}
\begin{verbatim}
0.588497
\end{verbatim}
\smed
Asymptotic $\hat \ell_0$
\smed
\begin{verbatim}
ell0asympt = (Exp[-barc (2 Pi + baralpha)] DeltaQell0 - 
    Exp[-barc (2 Pi + baralpha)] DeltaL + (1 - Cos[hatheta]))/(Cos[
     baralpha - hatheta] - Cos[baralpha]) 
\end{verbatim}
\begin{verbatim}
0.306042
\end{verbatim}
\smed
Asymptotic $\hat r_2$
\smed
\begin{verbatim}
r2asympt = (Sin[baralpha + hatheta] + Sin[hatheta] ell0asympt - 
    DeltaQr2 Exp[-barc (2 Pi + baralpha)])/Sin[baralpha]
\end{verbatim}
\begin{verbatim}
1.15869
\end{verbatim}
\smed
Final value $r_{\mathrm{opt}}(\bar \phi)$
\smed
\begin{verbatim}
rbarphasympt = (r2asympt - Exp[barc hatheta]) Exp[
    Cot[baralpha] (2 Pi + baralpha - hatheta) - 
     barc (2 Pi + baralpha)] + 1 - 
  ell0asympt Exp[-barc (2 Pi + baralpha)]
\end{verbatim}
\begin{verbatim}
0.976359
\end{verbatim}

\sbig

\noindent{\bf Section \ref{S:segment_case}}

\sbig

\noindent{\bf Section \ref{Sss:segm_semg_neg}}: plot of the final function Eq. \eqref{Equa:segm_segm_2_final}, Fig. \ref{Fig:segm_segm_case_2_final} and its minimal value

\sbig

\begin{verbatim}
Plot[((1 - Cos[theta + hatheta])/(1 - Cos[hatheta]) - 
    Exp[Cot[baralpha] theta])/theta, {theta, 0, Pi/2}, 
 AxesLabel -> Automatic, 
 PlotLegends -> 
  Placed[ToExpression[
    "\\frac{1}{\\theta} \\bigg( \\frac{1 - \\cos(\\hat \\theta + \
\\theta)}{1 - \\cos(\\hat \\theta)} - e^{\\cot(\\bar \\alpha) \
\\theta} \\bigg)", TeXForm, HoldForm], {Right, Center}]]
\end{verbatim}
\begin{verbatim}
Sin[hatheta]/(1 - Cos[hatheta]) - Cot[baralpha]
\end{verbatim}
\begin{verbatim}
3.45283
\end{verbatim}

\sbig

\noindent{\bf Section \ref{Sss:segm_semg_after_neg}}

\sbig

\noindent We first compute the relevant quantities for the tent
\smed
\begin{verbatim}
hatr0tentsegm1[tau_, theta_] := 
 rhogeneric[tau - 1, theta, 0] Exp[barc (2 Pi + baralpha) (tau - 1)]
\end{verbatim}
\begin{verbatim}
hatell0tentsegm1[tau_, theta_] := 
 ell0tent[hatr0tentsegm1[tau, theta], 0, 0]
\end{verbatim}
\begin{verbatim}
hatr2tentsegm1[tau_, theta_] := 
 hatr2tent[hatr0tentsegm1[tau, theta], hatell0tentsegm1[tau, theta], 
  0]
\end{verbatim}
\begin{verbatim}
hatell1tentsegm1[tau_, theta_] := 
 ell1tent[hatr2tentsegm1[tau, theta], 0, hatr0tentsegm1[tau, theta], 
  hatell0tentsegm1[tau, theta]]
\end{verbatim}
\begin{verbatim}
hatrbarphitentsegm1[tau_, theta_] := 
 hatrbarphi[hatr2tentsegm1[tau, theta], 
  S0arc[theta, 0] rhogeneric[tau - 1 + hattau, theta, 0] Exp[
    barc (2 Pi + baralpha) (tau - 1 + hattau)], 
  rhogeneric[tau, theta, 0] Exp[barc (2 Pi + baralpha) tau], 
  hatell0tentsegm1[tau, theta]]
\end{verbatim}
\smed
Plot of the function $\delta \check r(\bar \phi)$, Fig. \ref{Fig:segm_segm_after_neg_1} and its minimal value Fig. \ref{Fig:segm_segm_after_neg_2}
\smed
\begin{verbatim}
Plot3D[(hatrbarphitentsegm1[tau, 
     theta] Exp[-barc (2 Pi + baralpha) tau])/theta^2, {theta, 0, 
  Pi}, {tau, 1, 1 + tau1arc[theta, 0] - hattau - .001}, 
 AxesLabel -> Automatic]
\end{verbatim}
\begin{verbatim}
Plot[(hatrbarphitentsegm1[tau, Pi] Exp[-barc (2 Pi + baralpha) tau])/
  Pi^2, {tau, 1, 1 + tau1arc[Pi, 0] - hattau}, AxesLabel -> Automatic,
  PlotLegends -> 
  Placed[ToExpression[
    "\\frac{\\delta \\hat r_2(\\bar \\tau) e^{- \\bar c(2\\pi + \\bar \
\\alpha) \\tau}}{(\\pi/2)^2}", TeXForm, HoldForm], {Right, Center}]]
\end{verbatim}
\begin{verbatim}
(hatrbarphitentsegm1[1.00001, 
    Pi] Exp[-barc (2 Pi + baralpha) 1])/(Pi)^2
\end{verbatim}
\begin{verbatim}
{0.203057}
\end{verbatim}

\sbig

\noindent{\bf Section \ref{Sss:segm_semg_after_neg_2}}

\sbig

\noindent The angle in this case is before the perturbation at the beginning of the tent and inside the negativity cone for the final angle: the functions for $\Delta Q \cdot e^{\i\hat \theta}, \Delta \hat Q \cdot e^{i(\bar \alpha+ \hat \theta - \pi/2)}$ are
\smed
\begin{verbatim}
DeltaQtentsegm2[tau_, theta_] := 
 Cos[(2 Pi + baralpha) (2 - tau) - theta - hatheta]
\end{verbatim}
\begin{verbatim}
DeltaQtentsegmpi22[tau_, theta_] := 
 Cos[(2 Pi + baralpha) (2 - tau) - 
   theta - (baralpha + hatheta - Pi/2)]
\end{verbatim}
\begin{verbatim}
hatr0tentsegm2[tau_, theta_] := 
 rhogeneric[tau - 1, theta, 0] Exp[barc (2 Pi + baralpha) (tau - 1)]
\end{verbatim}
\begin{verbatim}
hatLtentsegm2[tau_, theta_] := 1
\end{verbatim}
\smed
Computation of the tent quantities
\smed
\begin{verbatim}
hatell0tentsegm2[tau_, theta_] := 
 ell0tent[hatr0tentsegm2[tau, theta], hatLtentsegm2[tau, theta], 
  DeltaQtentsegm2[tau, theta]]
\end{verbatim}
\begin{verbatim}
hatr2tentsegm2[tau_, theta_] := 
 hatr2tent[hatr0tentsegm2[tau, theta], hatell0tentsegm2[tau, theta], 
  DeltaQtentsegmpi22[tau, theta]]
\end{verbatim}
\begin{verbatim}
hatell1tentsegm2[tau_, theta_] := 
 ell1tent[hatr2tentsegm2[tau, theta], hatLtentsegm2[tau, theta], 
  hatr0tentsegm2[tau, theta], hatell0tentsegm2[tau, theta]]
\end{verbatim}
\begin{verbatim}
hatrbarphitentsegm2[tau_, theta_] := 
 hatrbarphi[hatr2tentsegm2[tau, theta], 
  rhogeneric[tau - 1 + hattau, theta, 0] Exp[
    barc (2 Pi + baralpha) (tau - 1 + hattau)], 
  rhogeneric[tau, theta, 0] Exp[barc (2 Pi + baralpha) tau], 
  hatell0tentsegm2[tau, theta]]
\end{verbatim}
\smed
Plot of the function $\delta \check r(\bar \phi)/\theta^2$, Fig. \ref{Fig:segm_segm_after_neg2_1}, \ref{Fig:segm_segm_after_neg2_2} and its minimal value
\smed
\begin{verbatim}
Plot3D[{0, 
  Exp[-barc (2 Pi + 
       baralpha) ((1 - 
          sigma) (2 - (theta + hatheta)/(2 Pi + baralpha)) + 
       sigma (2 - 
          Max[theta, 
            hatheta]/(2 Pi + baralpha)))] hatrbarphitentsegm2[((1 - 
          sigma) (2 - (theta + hatheta)/(2 Pi + baralpha)) + 
       sigma (2 - Max[theta, hatheta]/(2 Pi + baralpha))), theta]/
    theta^2}, {theta, 0.001, Pi}, {sigma, 0.001, 1 - .001}, 
 AxesLabel -> Automatic]
\end{verbatim}
\begin{verbatim}
Plot[{0, 
  Exp[-barc (2 Pi + 
       baralpha) ((1 - sigma) (2 - (Pi + hatheta)/(2 Pi + baralpha)) +
        sigma (2 - 
          Max[Pi, hatheta]/(2 Pi + 
             baralpha)))] hatrbarphitentsegm2[((1 - 
          sigma) (2 - (Pi + hatheta)/(2 Pi + baralpha)) + 
       sigma (2 - Max[Pi, hatheta]/(2 Pi + baralpha))), Pi]/
    Pi^2}, {sigma, 0, 1}, AxesLabel -> Automatic, 
 PlotLegends -> 
  Placed[ToExpression[
    "\\frac{\\delta \\hat r_2(\\bar \\tau) e^{- \\bar c(2\\pi + \\bar \
\\alpha) \\tau}}{(\\pi/2)^2}", TeXForm, HoldForm], {Right, Center}]]
\end{verbatim}
\begin{verbatim}
Exp[-barc (2 Pi + 
     baralpha) ((1 - .001) (2 - (Pi + hatheta)/(2 Pi + baralpha)) + 
     0.001 (2 - 
        Max[Pi, hatheta]/(2 Pi + 
           baralpha)))] hatrbarphitentsegm2[((1 - 
        0.001) (2 - (Pi + hatheta)/(2 Pi + baralpha)) + 
     0.001 (2 - Max[Pi, hatheta]/(2 Pi + baralpha))), Pi]/Pi^2
\end{verbatim}
\begin{verbatim}
{0.218402}
\end{verbatim}

\sbig

\noindent {\bf Section \ref{Sss:segm_semg_after_neg_3}}

\sbig

\noindent The angle in this case is before the perturbation at the beginning of the tent and after the negativity cone for the final angle: the relevant quantities are
\smed
\begin{verbatim}
DeltaQtentsegm3[tau_, theta_] := 
 rhogeneric[tau - 2 + hattau, theta, 0] Exp[
   barc (2 Pi + baralpha) (tau - 2 + hattau)] Cos[baralpha]
\end{verbatim}
\begin{verbatim}
DeltaQtentsegmpi23[tau_, theta_] := 
 rhogeneric[tau - 2 + hattau, theta, 0] Exp[
   barc (2 Pi + baralpha) (tau - 2 + hattau)] Cos[Pi/2 - 2 baralpha]
\end{verbatim}
\begin{verbatim}
hatr0tentsegm3[tau_, theta_] := 
 rhogeneric[tau - 1, theta, 0] Exp[barc (2 Pi + baralpha) (tau - 1)]
\end{verbatim}
\begin{verbatim}
hatLtentsegm3[tau_, theta_] := (1 - Cos[theta])/Sin[baralpha]^2 + 
  S0arc[theta, 
    0] (Exp[Cot[baralpha] (2 Pi + baralpha) (tau - 2 + hatheta)] - 1)/
    Cos[baralpha]
\end{verbatim}
\smed
Computation of the tent quantities
\smed
\begin{verbatim}
hatell0tentsegm3[tau_, theta_] := 
 ell0tent[hatr0tentsegm3[tau, theta], hatLtentsegm3[tau, theta], 
  DeltaQtentsegm3[tau, theta]]
\end{verbatim}
\begin{verbatim}
hatr2tentsegm3[tau_, theta_] := 
 hatr2tent[hatr0tentsegm3[tau, theta], hatell0tentsegm3[tau, theta], 
  DeltaQtentsegmpi23[tau, theta]]
\end{verbatim}
\begin{verbatim}
hatell1tentsegm3[tau_, theta_] := 
 ell1tent[hatr2tentsegm3[tau, theta], hatLtentsegm3[tau, theta], 
  hatr0tentsegm3[tau, theta], hatell0tentsegm3[tau, theta]]
\end{verbatim}
\begin{verbatim}
hatrbarphitentsegm3[tau_, theta_] := 
 hatrbarphi[hatr2tentsegm3[tau, theta], 
  rhogeneric[tau - 1 + hattau, theta, 0] Exp[
    barc (2 Pi + baralpha) (tau - 1 + hattau)], 
  rhogeneric[tau, theta, 0] Exp[barc (2 Pi + baralpha) tau], 
  hatell0tentsegm3[tau, theta]]
\end{verbatim}
\smed
Plot of the final value, Fig. \ref{Fig:segm_segm_after_neg3_1} and its minimal value for $\theta \searrow 0$ and $\theta = \hat \theta$, Fig. \ref{Fig:segm_segm_after_neg4_2}
\smed
\begin{verbatim}
Plot3D[{0, 
  Exp[-barc (2 Pi + baralpha) ((1 - sigma) (2 - hattau) + 
       sigma (2 - theta/(2 Pi + baralpha)))] hatrbarphitentsegm3[(1 - 
         sigma) (2 - hattau) + sigma (2 - theta/(2 Pi + baralpha)), 
     theta]/theta^2}, {theta, 0, hatheta - .001}, {sigma, 0, 1}, 
 AxesLabel -> Automatic]
\end{verbatim}
\begin{verbatim}
Plot[{0, 
  Exp[-barc (2 Pi + baralpha) ((1 - sigma) (2 - hattau) + 
       sigma (2 - (hatheta - .001)/(2 Pi + 
             baralpha)))] hatrbarphitentsegm3[(1 - sigma) (2 - 
         hattau) + 
      sigma (2 - (hatheta - .001)/(2 Pi + 
            baralpha)), (hatheta - .001)]/(hatheta - .001)^2}, {sigma,
   0, 1}, AxesLabel -> Automatic, 
 PlotLegends -> 
  Placed[ToExpression[
    "\\frac{\\delta \\hat r_2(\\bar \\tau) e^{- \\bar c(2\\pi + \\bar \
\\alpha) \\bar \\tau}}{\\hat \\theta^2}", TeXForm, HoldForm], {Right, 
    Center}]]
\end{verbatim}
\begin{verbatim}
Plot[{0, 
  Exp[-barc (2 Pi + baralpha) ((1 - sigma) (2 - hattau) + 
       sigma (2 - (.001)/(2 Pi + baralpha)))] hatrbarphitentsegm3[(1 -
          sigma) (2 - hattau) + 
      sigma (2 - (.001)/(2 Pi + baralpha)), (.001)]/(.001)^2}, {sigma,
   0, 1}, AxesLabel -> Automatic]
\end{verbatim}

\sbig

\noindent{\bf Section \ref{Sss:segm_semg_after_neg_4}}

\sbig

\noindent The tent is the negativity cone: computation of the relevant quantities
\smed
\begin{verbatim}
DeltaQtentsegm4[tau_, theta_] := 0
\end{verbatim}
\begin{verbatim}
DeltaQtentsegmpi24[tau_, theta_] := 0
\end{verbatim}
\begin{verbatim}
hatr0tentsegm4[tau_, theta_] := 
 rhogeneric[tau - 1, theta, 0] Exp[barc (2 Pi + baralpha) (tau - 1)]
\end{verbatim}
\begin{verbatim}
hatLtentsegm4[tau_, theta_] := 0
\end{verbatim}
\smed
Computation of the tent quantities
\smed
\begin{verbatim}
hatell0tentsegm4[tau_, theta_] := 
 ell0tent[hatr0tentsegm4[tau, theta], hatLtentsegm4[tau, theta], 
  DeltaQtentsegm4[tau, theta]]
\end{verbatim}
\begin{verbatim}
hatr2tentsegm4[tau_, theta_] := 
 hatr2tent[hatr0tentsegm4[tau, theta], hatell0tentsegm4[tau, theta], 
  DeltaQtentsegmpi24[tau, theta]]
\end{verbatim}
\begin{verbatim}
hatell1tentsegm4[tau_, theta_] := 
 ell1tent[hatr2tentsegm4[tau, theta], hatLtentsegm4[tau, theta], 
  hatr0tentsegm4[tau, theta], hatell0tentsegm4[tau, theta]]
\end{verbatim}
\begin{verbatim}
hatrbarphitentsegm4[tau_, theta_] := 
 hatrbarphi[hatr2tentsegm4[tau, theta], 
  rhogeneric[tau - 1 + hattau, theta, 0] Exp[
    barc (2 Pi + baralpha) (tau - 1 + hattau)], 
  rhogeneric[tau, theta, 0] Exp[barc (2 Pi + baralpha) tau], 
  hatell0tentsegm4[tau, theta]]
\end{verbatim}
\smed
Plot of the final value Fig. \ref{Fig:segm_segm_after_neg4_1}, Fig. \ref{Fig:segm_segm_after_neg5_2} and minimal value
\smed
\begin{verbatim}
Plot3D[{0, 
  Exp[-barc (2 Pi + 
       baralpha)] hatrbarphitentsegm4[(1 - sigma) (2 - 
         theta/(2 Pi + baralpha)) + sigma (2 - hattau), theta]/
    theta^2}, {theta, hatheta + .001, Pi}, {sigma, 0.001, 1}, 
 AxesLabel -> Automatic]
\end{verbatim}
\begin{verbatim}
Plot[{0, 
  Exp[-barc (2 Pi + 
       baralpha)] hatrbarphitentsegm4[(1 - sigma) (2 - 
         Pi/2/(2 Pi + baralpha)) + sigma (2 - hattau), 
     Pi/2]/(Pi/2)^2}, {sigma, 0.001, 1}, AxesLabel -> Automatic, 
 PlotLegends -> 
  Placed[ToExpression[
    "\\frac{\\delta \\hat r_2(\\bar \\tau) e^{- \\bar c(2\\pi + \\bar \
\\alpha) \\bar \\tau}}{\\hat \\theta^2}", TeXForm, HoldForm], {Right, 
    Center}]]
\end{verbatim}
\begin{verbatim}
Exp[-barc (2 Pi + 
     baralpha)] hatrbarphitentsegm4[(1 - .001) (2 - 
       Pi/2/(2 Pi + baralpha)) + .001 (2 - hattau), Pi/2]/(Pi/2)^2
\end{verbatim}
\begin{verbatim}
{2.91263}
\end{verbatim}

\sbig

\noindent{\bf Section \ref{Sss:segm_semg_after_neg_5}}

\sbig

\noindent The initial point of the tent is inside the negativity cone and the final point is outside: computation of the relevant quantities
\smed
\begin{verbatim}
DeltaQtentsegm5[tau_, theta_] := 
 rhogeneric[tau - 2 + hattau, theta, 0] Exp[
    barc (2 Pi + baralpha) (tau - 2 + hattau)] Cos[baralpha] - 
  Cos[-hatheta - theta - (2 Pi + baralpha) (tau - 2)]
\end{verbatim}
\begin{verbatim}
DeltaQtentsegmpi25[tau_, theta_] := 
 rhogeneric[tau - 2 + hattau, theta, 0] Exp[
    barc (2 Pi + baralpha) (tau - 2 + hattau)] Cos[
    Pi/2 - 2 baralpha] - 
  Cos[-(baralpha + hatheta - Pi/2) - 
    theta - (2 Pi + baralpha) (tau - 2)]
\end{verbatim}
\begin{verbatim}
hatr0tentsegm5[tau_, theta_] := 
 rhogeneric[tau - 1, theta, 0] Exp[barc (2 Pi + baralpha) (tau - 1)]
\end{verbatim}
\begin{verbatim}
hatLtentsegm5[tau_, 
  theta_] := -1 + (1 - Cos[theta]) Exp[
     Cot[baralpha] (2 Pi + baralpha) (tau - 2 + hattau)]/
    Sin[baralpha]^2
\end{verbatim}
\smed
Computation of the tent quantities
\smed
\begin{verbatim}
hatell0tentsegm5[tau_, theta_] := 
 ell0tent[hatr0tentsegm5[tau, theta], hatLtentsegm5[tau, theta], 
  DeltaQtentsegm5[tau, theta]]
\end{verbatim}
\begin{verbatim}
hatr2tentsegm5[tau_, theta_] := 
 hatr2tent[hatr0tentsegm5[tau, theta], hatell0tentsegm5[tau, theta], 
  DeltaQtentsegmpi25[tau, theta]]
\end{verbatim}
\begin{verbatim}
hatell1tentsegm5[tau_, theta_] := 
 ell1tent[hatr2tentsegm5[tau, theta], hatLtentsegm5[tau, theta], 
  hatr0tentsegm5[tau, theta], hatell0tentsegm5[tau, theta]]
\end{verbatim}
\begin{verbatim}
hatrbarphitentsegm5[tau_, theta_] := 
 hatrbarphi[hatr2tentsegm5[tau, theta], 
  rhogeneric[tau - 1 + hattau, theta, 0] Exp[
    barc (2 Pi + baralpha) (tau - 1 + hattau)], 
  rhogeneric[tau, theta, 0] Exp[barc (2 Pi + baralpha) tau], 
  hatell0tentsegm5[tau, theta]]
\end{verbatim}
\smed
Plot of the final value Fig. \ref{Fig:segm_segm_after_neg5_1}, Fig. \ref{Fig:segm_segm_after_neg6_2} and its minimal value
\smed
\begin{verbatim}
Plot3D[{0, 
  Exp[-barc (2 Pi + 
       baralpha) ((1 - sigma) (2 - 
          Min[theta, hatheta]/(2 Pi + baralpha)) + 
       2 sigma)] hatrbarphitentsegm5[(1 - sigma) (2 - 
         Min[theta, hatheta]/(2 Pi + baralpha)) + 2 sigma, theta]/
    theta^2}, {theta, 0, Pi}, {sigma, 0.001, 1}, 
 AxesLabel -> Automatic]
\end{verbatim}
\begin{verbatim}
Plot[{0, 
  Exp[-barc (2 Pi + 
       baralpha) ((1 - sigma) (2 - 
          Min[Pi/2, hatheta]/(2 Pi + baralpha)) + 
       2 sigma)] hatrbarphitentsegm5[(1 - sigma) (2 - 
         Min[Pi/2, hatheta]/(2 Pi + baralpha)) + 2 sigma, 
     Pi/2]/(Pi/2)^2}, {sigma, 0.001, 1}, AxesLabel -> Automatic, 
 PlotLegends -> 
  Placed[ToExpression[
    "\\frac{\\delta \\hat r_2(\\bar \\tau) e^{- \\bar c(2\\pi + \\bar \
\\alpha) \\bar \\tau}}{\\hat \\theta^2}", TeXForm, HoldForm], {Right, 
    Center}]]
\end{verbatim}
\begin{verbatim}
Exp[-barc (2 Pi + 
     baralpha) ((1 - .001) (2 - 
        Min[Pi/2, hatheta]/(2 Pi + baralpha)) + 
     2 (.001))] hatrbarphitentsegm5[(1 - .001) (2 - 
       Min[Pi/2, hatheta]/(2 Pi + baralpha)) + 2 (.001), 
   Pi/2]/(Pi/2)^2
\end{verbatim}
\begin{verbatim}
{2.91263}
\end{verbatim}

\sbig

\noindent{\bf Section \ref{Sss:segm_semg_after_neg_6}}

\sbig

\noindent Here we use only the standard functional of Lemma \ref{Lem:final_value_satu}, Eq. \eqref{Equa:rho_7_segm_segm} and Fig. \ref{Fig:segm_segm_after_neg7_1} and Fig. \ref{Fig:segm_segm_after_neg7_2}
\smed
\begin{verbatim}
Plot3D[{(rhogeneric[tau, theta, 0] - 
     Kadmitent Exp[-barc (2 Pi + baralpha) 2] rhogeneric[tau - 2, 
       theta, 0])/theta^2}, {tau, 2, 5}, {theta, 0.001, Pi}, 
 AxesLabel -> Automatic]
\end{verbatim}
\begin{verbatim}
Plot[{(rhogeneric[tau, Pi/2, 0] - 
     Kadmitent Exp[-barc (2 Pi + baralpha) 2] rhogeneric[tau - 2, 
       Pi/2, 0])/(Pi/2)^2}, {tau, 2, 5}, AxesLabel -> Automatic, 
 PlotLegends -> 
  Placed[ToExpression[
    "\\frac{\\rho(\\bar \\tau) - \\frac{(1 - \\cos(\\hat \\theta)) \
e^{- \\bar c (2\\pi + \\bar \\alpha) 2}}{\\cos(\\bar \\alpha - \\hat \
\\theta) - \\cos(\\bar \\alpha)} \\rho(\\bar \\tau - 2)}{(\\pi/2)^2}",
     TeXForm, HoldForm], {Left, Top}]]
\end{verbatim}

\sbig

\noindent {\bf Section \ref{S:arc}}

\sbig

\noindent {\bf Section \ref{Sss:segm_arc_3}}

\sbig

\noindent Maximal final angle, so that we are not in the perturbed region, Eq. \eqref{Equa:max_arc_arc}
\smed
\begin{verbatim}
Tan[baralpha] + Tan[baralpha] + 0.11753300046278962` + Pi/2 - baralpha
\end{verbatim}
\begin{verbatim}
5.34128
\end{verbatim}
\smed
Final value $\delta \check r(\bar \phi)$, Eq. \eqref{Equa:final_arc_arc}
\smed
\begin{verbatim}
arcarc[theta_, Deltaphi_, dtheta_, 
  dphi_] := (1 - Cos[theta + dtheta] + Sin[theta + dtheta] dphi + 
     Sin[theta] Deltaphi)/(1 - Cos[dtheta] + Sin[dtheta] dphi) - 
  Exp[Cot[baralpha] (Deltaphi + theta)] + 
  Cos[theta + baralpha] (Exp[Cot[baralpha] Deltaphi] - 1)/
    Cos[baralpha]
\end{verbatim}
\smed
Plot of the second derivative of the value function w.r.t. $\Delta \phi$, so that the minima are at the boundary
\smed
\begin{verbatim}
Plot[1 - 
  Cos[baralpha + theta] Exp[-Cot[baralpha] theta]/
    Cos[baralpha], {theta, 0, Pi/2}]
\end{verbatim}
\smed
Region where we have to study the final value: it is the boundary of the negativity region, Fig. \ref{Fig:arcarc_2_1}
\smed
\begin{verbatim}
Plot3D[{0, 
  Cot[baralpha] (1 - Cos[dtheta] + Sin[dtheta] dphi) Exp[
      Cot[baralpha] (2 Pi + baralpha - dtheta - dphi)]/
     Sin[baralpha] - 1}, {dphi, 0, Tan[baralpha]}, {dtheta, 
  Max[.11, (2.1 - dphi)/
    8, .08 + (hatheta - .1) (1 - 
        dphi/Tan[baralpha])^6], .14 + (hatheta - .1) (1 - 
       dphi/Tan[baralpha])^5}, AxesLabel -> Automatic, 
 PlotLegends -> Automatic]
\end{verbatim}
\smed
Plot of the final value for $\Delta \phi = \tan(\bar \alpha)$, Fig. \ref{Fig:arcarc_2_2}
\smed
\begin{verbatim}
Plot3D[{0, 
  Evaluate@
   Table[arcarc[theta, Tan[baralpha], dtheta, dphi]/theta, {theta, 0, 
     Pi/2, .4}]}, {dphi, 0, Tan[baralpha]}, {dtheta, 
  Max[.11, (2.1 - dphi)/
    8, .08 + (hatheta - .1) (1 - 
        dphi/Tan[baralpha])^6], .14 + (hatheta - .1) (1 - 
       dphi/Tan[baralpha])^5}, AxesLabel -> Automatic, 
 PlotLegends -> Automatic]
\end{verbatim}
\begin{verbatim}
Plot3D[{0, arcarc[Pi/2, Tan[baralpha], dtheta, dphi]/(Pi/2)}, {dphi, 
  0, Tan[baralpha]}, {dtheta, 
  Max[.11, (2.1 - dphi)/
    8, .08 + (hatheta - .1) (1 - 
        dphi/Tan[baralpha])^6], .14 + (hatheta - .1) (1 - 
       dphi/Tan[baralpha])^5}, AxesLabel -> Automatic, 
 PlotLegends -> Automatic]
\end{verbatim}
\begin{verbatim}
Plot[{0, 
  arcarc[Pi/2, Tan[baralpha], dtheta, Tan[baralpha]]/(Pi/2)}, {dtheta,
   Max[.11, (2.1 - Tan[baralpha])/
    8, .08 + (hatheta - .1) (1 - 
        Tan[baralpha]/Tan[baralpha])^6], .14 + (hatheta - .1) (1 - 
       Tan[baralpha]/Tan[baralpha])^5}, AxesLabel -> Automatic, 
 PlotLegends -> Automatic]
\end{verbatim}

\sbig

\noindent{\bf Section \ref{Sss:arc_segm_1}}

\sbig

\noindent Numerical plot of Fig. \ref{Fig:arc_segm_1_1}
\smed
\begin{verbatim}
Plot[((1 - Cos[theta + hatheta] + Sin[theta] Tan[baralpha])/(1 - 
       Cos[hatheta]) - Exp[1 + Cos[baralpha] theta] + 
    Cos[baralpha + theta] (Exp[1] - 1)/Cos[baralpha])/theta, {theta, 
  0, Pi/2}, AxesLabel -> Automatic, 
 PlotLegends -> 
  Placed[ToExpression[
    "(\\frac{1 - \\cos(\\theta + \\hat \\theta) + \\sin(\\theta) \
\\tan(\\bar \\alpha)}{1 - \\cos(\\hat \\theta)} - e^{1 + \\cot(\\bar \
\\alpha) \\theta} + \\frac{\\cos(\\theta + \\bar \
\\alpha)}{\\cos(\\bar \\alpha)} (e-1))/\\theta", TeXForm, 
    HoldForm], {Left, Center}]]
\end{verbatim}

\sbig

\noindent{\bf Section \ref{Sss:arc_tent_satur}}

\sbig

\noindent{\bf Case 1}: computation of \eqref{Equa:check_worst_arc_segm_case_1}
\smed
\begin{verbatim}
(Cos[baralpha] - Cos[baralpha + hatheta])/(Cos[baralpha - hatheta] - 
    Cos[baralpha]) - Exp[Cot[baralpha] hatheta]
\end{verbatim}
\begin{verbatim}
0.0066943
\end{verbatim}
\smed
Plot of the final value Fig. \ref{Fig:arccase1} and its minimal values
\smed
\begin{verbatim}
Plot3D[{0, 
  Evaluate@
   Table[(rhogeneric[(1 - sigma) 1 + 
         sigma (1 + tau1arc[theta, Deltaphi] - hattau), theta, 
        Deltaphi] - (1 - 
          Cos[hatheta]) Exp[-barc (2 Pi + 
            baralpha)] rhogeneric[(1 - sigma) 1 + 
           sigma (1 + tau1arc[theta, Deltaphi] - hattau) - 1, theta, 
          Deltaphi]/(Cos[baralpha - hatheta] - 
           Cos[baralpha]))/(theta (theta + Deltaphi)), {Deltaphi, 0, 
     Tan[baralpha], .5}]}, {theta, 0, Pi/2}, {sigma, 0, 1}, 
 AxesLabel -> Automatic]
\end{verbatim}
\begin{verbatim}
Plot3D[{0, (rhogeneric[(1 - sigma) 1 + 
       sigma (1 + tau1arc[theta, 0] - hattau), theta, 
      0] - (1 - 
        Cos[hatheta]) Exp[-barc (2 Pi + 
          baralpha)] rhogeneric[(1 - sigma) 1 + 
         sigma (1 + tau1arc[theta, 0] - hattau) - 1, theta, 
        0]/(Cos[baralpha - hatheta] - Cos[baralpha]))/(theta (theta + 
       0))}, {theta, 0, Pi/2}, {sigma, 0, 1}, AxesLabel -> Automatic]
\end{verbatim}
\begin{verbatim}
Plot[{0, (rhogeneric[(1 - sigma) 1 + 
       sigma (1 + tau1arc[Pi/2, 0] - hattau), Pi/2, 
      0] - (1 - 
        Cos[hatheta]) Exp[-barc (2 Pi + 
          baralpha)] rhogeneric[(1 - sigma) 1 + 
         sigma (1 + tau1arc[Pi/2, 0] - hattau) - 1, Pi/2, 
        0]/(Cos[baralpha - hatheta] - Cos[baralpha]))/(Pi/
      2 (Pi/2 + 0))}, {sigma, 0, 1}, AxesLabel -> Automatic, 
 PlotLegends -> 
  Placed[ToExpression[
    "\\frac{e^{-\\bar c \\phi}}{(\\pi/2)^2}(\\delta \\tilde r(\\bar \
\\phi) - \\delta \\tilde r_0 \\frac{1 - \\cos(\\hat \
\\theta)}{\\cos(\\bar \\alpha - \\hat \\theta) - \\cos(\\bar \
\\alpha)})", TeXForm, HoldForm], {Right, Center}]]
\end{verbatim}

\sbig

\noindent{\bf Case 2}: computations of the relevant quantities for the tent
\smed
\begin{verbatim}
DeltaQcase2arcarc[tau_, theta_, Deltaphi_] := 
 Cos[(2 Pi + baralpha) (tau - 2) + hatheta + theta + Deltaphi]
\end{verbatim}
\begin{verbatim}
DeltaQcase2arcarc2[tau_, theta_, Deltaphi_] := 
 Cos[(2 Pi + baralpha) (tau - 2) + hatheta + theta + Deltaphi + 
   baralpha - Pi/2]
\end{verbatim}
\begin{verbatim}
hatell0arc21[tau_, theta_, Deltaphi_] := 
 ell0tent[
  rhogeneric[tau - 1, theta, Deltaphi] Exp[
    barc (2 Pi + baralpha) (tau - 1)], 1, 
  DeltaQcase2arcarc[tau, theta, Deltaphi]]
\end{verbatim}
\begin{verbatim}
hatr2arc21[tau_, theta_, Deltaphi_] := 
 hatr2tent[
  rhogeneric[tau - 1, theta, Deltaphi] Exp[
    barc (2 Pi + baralpha) (tau - 1)], 
  hatell0arc21[tau, theta, Deltaphi], 
  DeltaQcase2arcarc2[tau, theta, Deltaphi]]
\end{verbatim}
\smed
Final value
\smed
\begin{verbatim}
hatrbarphiarc21[tau_, theta_, Deltaphi_] := 
 Exp[-barc (2 Pi + baralpha) tau] hatrbarphi[
   hatr2arc21[tau, theta, Deltaphi], 
   rhogeneric[tau - 1 + hattau, theta, Deltaphi] Exp[
     barc (2 Pi + baralpha) (tau - 1 + hattau)], 
   rhogeneric[tau, theta, Deltaphi] Exp[barc (2 Pi + baralpha) tau], 
   hatell0arc21[tau, theta, Deltaphi]]
\end{verbatim}
\smed
Numerical plot of Fig. \ref{Fig:PTarc_case2_1} and its minimal value
\smed
\begin{verbatim}
Plot3D[{Evaluate@
   Table[hatrbarphiarc21[(1 - 
          sigma) (2 + (-Deltaphi - theta - hatheta)/(2 Pi + 
             baralpha)) + 
       sigma (2 + (-Deltaphi - Max[theta, hatheta])/(2 Pi + 
             baralpha)), theta, 
      Deltaphi]/(theta (theta + Deltaphi)), {Deltaphi, 0.1, 
     Tan[baralpha], .5}]}, {theta, 0.1, Pi/2}, {sigma, 0.1, 1}, 
 AxesLabel -> Automatic]
\end{verbatim}
\begin{verbatim}
Plot3D[{hatrbarphiarc21[(1 - 
        sigma) (2 + (-0.0 - theta - hatheta)/(2 Pi + baralpha)) + 
     sigma (2 + (-0.0 - Max[theta, hatheta])/(2 Pi + baralpha)), 
    theta, 0.0]/(theta (theta + 0.0))}, {theta, 0.1, Pi/2}, {sigma, 
  0.1, 1}, AxesLabel -> Automatic]
\end{verbatim}
\begin{verbatim}
Plot[{hatrbarphiarc21[(1 - 
        sigma) (2 + (-0.0 - Pi/2 - hatheta)/(2 Pi + baralpha)) + 
     sigma (2 + (-0.0 - Max[Pi/2, hatheta])/(2 Pi + baralpha)), Pi/2, 
    0.0]/(Pi/2 (Pi/2 + 0.0))}, {sigma, 0.0, 1}, 
 AxesLabel -> Automatic, 
 PlotLegends -> 
  Placed[ToExpression[
    "\\frac{e^{-\\bar c \\phi} \\delta \\check r(\\bar \
\\phi)}{(\\pi/2)^2}", TeXForm, HoldForm], {Right, Center}]]
\end{verbatim}
\begin{verbatim}
hatrbarphiarc21[(1 - 
      0.0001) (2 + (-0 - Pi/2 - hatheta)/(2 Pi + baralpha)) + 
   0.0001 (2 + (-0 - Pi/2)/(2 Pi + baralpha)), Pi/2, 
  0]/(Pi/2 (Pi/2 + 0))
\end{verbatim}
\begin{verbatim}
{0.546565}
\end{verbatim}

\sbig

\noindent{\bf Case 3}: computation of the relevant quantities for the tent
\smed
\begin{verbatim}
hatell0arc2case3[tau_, theta_, Deltaphi_] := 
 ell0tent[
  rhogeneric[tau - 1, theta, Deltaphi] Exp[
    barc (2 Pi + baralpha) (tau - 1)], 
  1 - Cos[theta] + 
   Sin[theta] ((2 Pi + baralpha) (tau - 2) + hatheta + Deltaphi), 
  Sin[theta] Cos[-Pi/2]]
\end{verbatim}
\begin{verbatim}
hatr2arc2case3[tau_, theta_, Deltaphi_] := 
 hatr2tent[
  rhogeneric[tau - 1, theta, Deltaphi] Exp[
    barc (2 Pi + baralpha) (tau - 1)], 
  hatell0arc2case3[tau, theta, Deltaphi], Sin[theta] Cos[baralpha]]
\end{verbatim}
\begin{verbatim}
hatrbarphiarc2case3[tau_, theta_, Deltaphi_] := 
 Exp[-barc (2 Pi + baralpha) tau] hatrbarphi[
   hatr2arc2case3[tau, theta, Deltaphi], 
   rhogeneric[tau - 1 + hattau, theta, Deltaphi] Exp[
     barc (2 Pi + baralpha) (tau - 1 + hattau)], 
   rhogeneric[tau, theta, Deltaphi] Exp[barc (2 Pi + baralpha) tau], 
   hatell0arc2case3[tau, theta, Deltaphi]]
\end{verbatim}
\smed
Plot of the final value Fig. \ref{Fig:PTarc_case3_1}, first an enlargement about the singularity $\theta = \Delta \phi = 0$
\smed
\begin{verbatim}
Plot3D[{Evaluate@
   Table[hatrbarphiarc2case3[tau, theta, 
      Deltaphi]/(theta (theta + 
         Deltaphi)), {theta, .0001, .1, .02}]}, {Deltaphi, 
  0, .1}, {tau, 2 + (-Deltaphi - hatheta)/(2 Pi + baralpha), 
  2 - hatheta/(2 Pi + baralpha)}, AxesLabel -> Automatic]
\end{verbatim}
\begin{verbatim}
Plot3D[{{Evaluate@
    Table[hatrbarphiarc2case3[(1 - 
           sigma) (2 + (-Deltaphi - hatheta)/(2 Pi + baralpha)) + 
        sigma (2 + (-Max[Deltaphi + theta, hatheta])/(2 Pi + 
              baralpha)), theta, 
       Deltaphi]/(theta (theta + Deltaphi)), {Deltaphi, 0.01, 
      Tan[baralpha], .5}]}}, {theta, 0, hatheta - .01}, {sigma, 0, 1},
  AxesLabel -> Automatic]
\end{verbatim}
\begin{verbatim}
Plot3D[{hatrbarphiarc2case3[(1 - 
        sigma) (2 + (-0.001 - hatheta)/(2 Pi + baralpha)) + 
     sigma (2 + (-Max[0.001 + theta, hatheta])/(2 Pi + baralpha)), 
    theta, 0.001]/(theta (theta + 0.001))}, {theta, 0, 
  hatheta - .01}, {sigma, 0, 1}, AxesLabel -> Automatic]
\end{verbatim}
\begin{verbatim}
Plot3D[{hatrbarphiarc2case3[(1 - 
        sigma) (2 + (-0.001 - hatheta)/(2 Pi + baralpha)) + 
     sigma (2 + (-Max[0.001 + theta, hatheta])/(2 Pi + baralpha)), 
    theta, 0.001]/(theta (theta + 0.001))}, {theta, hatheta - .1, 
  hatheta - .01}, {sigma, 0, 1}, AxesLabel -> Automatic]
\end{verbatim}
\begin{verbatim}
Plot[{hatrbarphiarc2case3[(1 - 
        sigma) (2 + (-0.001 - hatheta)/(2 Pi + baralpha)) + 
     sigma (2 + (-Max[0.001 + (hatheta - .0001), hatheta])/(2 Pi + 
           baralpha)), hatheta - .0001, 
    0.001]/((hatheta - .0001) ((hatheta - .0001) + 0.001))}, {sigma, 
  0, 1}, AxesLabel -> Automatic, 
 PlotLegends -> 
  Placed[ToExpression[
    "\\frac{e^{-\\bar c \\phi} \\delta \\check r(\\bar \\phi)}{(\\hat \
\\theta)^2}", TeXForm, HoldForm], {Right, Center}]]
\end{verbatim}
\begin{verbatim}
hatrbarphiarc2case3[(1 - 
      0) (2 + (-0.001 - hatheta)/(2 Pi + baralpha)) + 
   0 (2 + (-Max[0.001 + (hatheta - .0001), hatheta])/(2 Pi + 
         baralpha)), hatheta - .0001, 
  0.001]/((hatheta - .0001) ((hatheta - .0001) + 0.001))
\end{verbatim}
\begin{verbatim}
{0.749654}
\end{verbatim}

\sbig

\noindent{\bf Case 4}: computation of the relevant quantities for the tent
\smed
\begin{verbatim}
hatell0arc2case4[tau_, theta_, Deltaphi_] := 
 ell0tent[
  rhogeneric[tau - 1, theta, Deltaphi] Exp[
    barc (2 Pi + baralpha) (tau - 1)], (1 - Cos[theta] + 
     Sin[theta] Deltaphi) Exp[
     Cot[baralpha] (2 Pi + baralpha) (tau - 2 + hattau)]/
    Sin[baralpha]^2, 
  rhogeneric[tau - 2 + hattau, theta, Deltaphi] Cos[baralpha] Exp[
    barc (2 Pi + baralpha) (tau - 2 + hattau)]]
\end{verbatim}
\begin{verbatim}
hatr2arc2case4[tau_, theta_, Deltaphi_] := 
 hatr2tent[
  rhogeneric[tau - 1, theta, Deltaphi] Exp[
    barc (2 Pi + baralpha) (tau - 1)], 
  hatell0arc2case4[tau, theta, Deltaphi], 
  rhogeneric[tau - 2 + hattau, theta, Deltaphi] Exp[
    barc (2 Pi + baralpha) (tau - 2 + hattau)] Sin[2 baralpha]]
\end{verbatim}
\begin{verbatim}
hatrbarphiarc2case4[tau_, theta_, Deltaphi_] := 
 Exp[-barc (2 Pi + baralpha) tau] hatrbarphi[
   hatr2arc2case4[tau, theta, Deltaphi], 
   rhogeneric[tau - 1 + hattau, theta, Deltaphi] Exp[
     barc (2 Pi + baralpha) (tau - 1 + hattau)], 
   rhogeneric[tau, theta, Deltaphi] Exp[barc (2 Pi + baralpha) tau], 
   hatell0arc2case4[tau, theta, Deltaphi]]
\end{verbatim}
\smed
Plot of the final value Fig. \ref{Fig:PTarc_case4_1} and computation of the minimal values
\smed
\begin{verbatim}
Plot3D[{Evaluate@
   Table[hatrbarphiarc2case4[(1 - sigma) (2 - hattau) + 
       sigma (2 + (-Deltaphi - theta)/(2 Pi + baralpha)), theta, 
      Deltaphi]/(theta (theta + Deltaphi)), {sigma, 
     0, .9, .3}]}, {theta, 0, hatheta - .001}, {Deltaphi, 0, 
  hatheta - theta - .001}, AxesLabel -> Automatic, 
 PlotLegends -> Automatic]
\end{verbatim}
\begin{verbatim}
Plot3D[hatrbarphiarc2case4[(1 - sigma) (2 - hattau) + 
    sigma (2 + (-0 - theta)/(2 Pi + baralpha)), theta, 
   0]/(theta (theta + 0)), {sigma, 0, 1}, {theta, 0.01, hatheta}, 
 AxesLabel -> Automatic, PlotLegends -> Automatic]
\end{verbatim}
\begin{verbatim}
Plot[{hatrbarphiarc2case4[(1 - sigma) (2 - hattau) + 
     sigma (2 + (-hatheta + .001 - .000)/(2 Pi + baralpha)), 
    hatheta - .001, 
    0]/((hatheta - .001) (hatheta - .001 + 0))}, {sigma, 0, 1}, 
 AxesLabel -> Automatic, 
 PlotLegends -> 
  Placed[ToExpression[
    "\\frac{e^{-\\bar c \\phi} \\delta \\check r(\\bar \\phi)}{(\\hat \
\\theta)^2}", TeXForm, HoldForm], {Right, Center}]]
\end{verbatim}
\begin{verbatim}
hatrbarphiarc2case4[(1 - 0.001) (2 - hattau) + 
   0.001 (2 + (-hatheta + .001 - .000)/(2 Pi + baralpha)), 
  hatheta - .001, 0]/((hatheta - .001) (hatheta - .001 + 0))
\end{verbatim}
\begin{verbatim}
{0.748399}
\end{verbatim}

\sbig

\noindent {\bf Case 5}: computation of the relevant quantities for the tent
\smed
\begin{verbatim}
hatell0arc2case5[tau_, theta_, Deltaphi_] := 
 ell0tent[
  rhogeneric[tau - 1, theta, Deltaphi] Exp[
    barc (2 Pi + baralpha) (tau - 1)], 0, 0]
\end{verbatim}
\begin{verbatim}
hatr2arc2case5[tau_, theta_, Deltaphi_] := 
 hatr2tent[
  rhogeneric[tau - 1, theta, Deltaphi] Exp[
    barc (2 Pi + baralpha) (tau - 1)], 
  hatell0arc2case5[tau, theta, Deltaphi], 0]
\end{verbatim}
\begin{verbatim}
hatrbarphiarc2case5[tau_, theta_, Deltaphi_] := 
 Exp[-barc (2 Pi + baralpha) tau] hatrbarphi[
   hatr2arc2case5[tau, theta, Deltaphi], 
   rhogeneric[tau - 1 + hattau, theta, Deltaphi] Exp[
     barc (2 Pi + baralpha) (tau - 1 + hattau)], 
   rhogeneric[tau, theta, Deltaphi] Exp[barc (2 Pi + baralpha) tau], 
   hatell0arc2case5[tau, theta, Deltaphi]]
\end{verbatim}
\smed
Plot of the final value Fig. \ref{Fig:PTarc_case5_1} and computation of the minimal values
\smed
\begin{verbatim}
Plot3D[{Evaluate@
   Table[hatrbarphiarc2case5[(1 - 
          sigma) (2 + (-Deltaphi - theta)/(2 Pi + baralpha)) + 
       sigma (2 + (-Deltaphi - hatheta)/(2 Pi + baralpha)), theta, 
      Deltaphi]/(theta (theta + Deltaphi)), {Deltaphi, 0, 
     Tan[baralpha], .5}]}, {theta, hatheta + .01, Pi/2}, {sigma, 0, 
  1}, AxesLabel -> Automatic]
\end{verbatim}
\begin{verbatim}
Plot[hatrbarphiarc2case5[(1 - 
       sigma) (2 + (-0 - Pi/2)/(2 Pi + baralpha)) + 
    sigma (2 + (-0 - hatheta)/(2 Pi + baralpha)), Pi/2, 
   0]/(Pi/2 (Pi/2 + 0)), {sigma, 0, 1}, AxesLabel -> Automatic, 
 PlotLegends -> 
  Placed[ToExpression[
    "\\frac{e^{-\\bar c \\phi} \\delta \\check r(\\bar \
\\phi)}{\\pi^2}", TeXForm, HoldForm], {Right, Center}]]
\end{verbatim}
\begin{verbatim}
hatrbarphiarc2case5[(1 - 0.001) (2 + (-0 - Pi/2)/(2 Pi + baralpha)) + 
   0.001 (2 + (-0 - hatheta)/(2 Pi + baralpha)), Pi/2, 
  0]/(Pi/2 (Pi/2 + 0))
\end{verbatim}
\begin{verbatim}
{0.559708}
\end{verbatim}

\sbig

\noindent{\bf Case 6}: computation of the relevant quantities for the tent
\smed
\begin{verbatim}
hatell0arc2case6[tau_, theta_, Deltaphi_] := 
 ell0tent[
  rhogeneric[tau - 1, theta, Deltaphi] Exp[
    barc (2 Pi + baralpha) (tau - 1)], -Cos[theta] + 
   Sin[theta] ((tau - 2) (2 Pi + baralpha) + hatheta + 
      Deltaphi), -Cos[(2 Pi + baralpha) tau - 4 Pi - baralpha + 
     hatheta - (baralpha - Deltaphi - theta)]]
\end{verbatim}
\begin{verbatim}
hatr2arc2case6[tau_, theta_, Deltaphi_] := 
 hatr2tent[
  rhogeneric[tau - 1, theta, Deltaphi] Exp[
    barc (2 Pi + baralpha) (tau - 1)], 
  hatell0arc2case6[tau, theta, Deltaphi], 
  Sin[theta] Cos[baralpha] - 
   Cos[(tau - 2) (2 Pi + baralpha) + hatheta + baralpha - Pi/2 + 
     Deltaphi + theta]]
\end{verbatim}
\begin{verbatim}
hatrbarphiarc2case6[tau_, theta_, Deltaphi_] := 
 Exp[-barc (2 Pi + baralpha) tau] hatrbarphi[
   hatr2arc2case6[tau, theta, Deltaphi], 
   rhogeneric[tau - 1 + hattau, theta, Deltaphi] Exp[
     barc (2 Pi + baralpha) (tau - 1 + hattau)], 
   rhogeneric[tau, theta, Deltaphi] Exp[barc (2 Pi + baralpha) tau], 
   hatell0arc2case6[tau, theta, Deltaphi]]
\end{verbatim}
\smed
Plot of the final value Fig. \ref{Fig:PTarc_case6_1} and computation of the minimal values
\smed
\begin{verbatim}
Plot3D[{Evaluate@
   Table[hatrbarphiarc2case6[(1 - 
          sigma) (2 + (-Deltaphi - Min[theta, hatheta])/(2 Pi + 
             baralpha)) + 
       sigma (2 + ( - Max[Deltaphi, hatheta])/(2 Pi + baralpha)), 
      theta, Deltaphi]/theta, {sigma, 0.01, .99, .245}]}, {theta, 
  0.0001, Pi/2}, {Deltaphi, Max[hatheta - theta, 0], Tan[baralpha]}, 
 AxesLabel -> Automatic]
\end{verbatim}
\begin{verbatim}
Plot3D[{hatrbarphiarc2case6[(1 - .001) (2 + (-Deltaphi - 
           Min[theta, hatheta])/(2 Pi + baralpha)) + .001 (2 + ( - 
           Max[Deltaphi, hatheta])/(2 Pi + baralpha)), theta, 
    Deltaphi]/theta}, {theta, 0.0001, Pi/2}, {Deltaphi, 
  Max[hatheta - theta, 0], Tan[baralpha]}, AxesLabel -> Automatic]
\end{verbatim}
\begin{verbatim}
Plot3D[hatrbarphiarc2case6[(1 - 
       sigma) (2 + (-Deltaphi - hatheta)/(2 Pi + baralpha)) + 
    sigma (2 + ( -  Max[Deltaphi, hatheta])/(2 Pi + baralpha)), 
   hatheta, Deltaphi]/hatheta, {Deltaphi, 0, .1}, {sigma, 0, 1}, 
 AxesLabel -> Automatic]
\end{verbatim}
\begin{verbatim}
Plot[{hatrbarphiarc2case6[(1 - .001) (2 + (-0.001 - 
           Min[theta, hatheta])/(2 Pi + baralpha)) + .001 (2 + ( - 
           Max[0.001, hatheta])/(2 Pi + baralpha)), theta, 0.001]/
   theta}, {theta, hatheta - .001, Pi/2}, AxesLabel -> Automatic]
\end{verbatim}
\begin{verbatim}
Plot[hatrbarphiarc2case6[(1 - 
       sigma) (2 + (-0.001 - hatheta)/(2 Pi + baralpha)) + 
    sigma (2 + ( -  Max[0.001, hatheta])/(2 Pi + baralpha)), hatheta, 
   0.001]/hatheta, {sigma, 0, 1}, AxesLabel -> Automatic, 
 PlotLegends -> 
  Placed[ToExpression[
    "\\frac{e^{-\\bar c \\phi} \\delta \\check r(\\bar \\phi)}{\\hat \
\\theta}", TeXForm, HoldForm], {Right, Center}]]
\end{verbatim}

\sbig

\noindent{\bf Case 7}: computation of the relevant quantities for the tent
\smed
\begin{verbatim}
hatell0arc2case7[tau_, theta_, Deltaphi_] := 
 ell0tent[
  rhogeneric[tau - 1, theta, Deltaphi] Exp[
    barc (2 Pi + baralpha) (tau - 1)], -1 + (1 - Cos[theta] + 
      Sin[theta] Deltaphi) Exp[
      Cot[baralpha] (2 Pi + baralpha) (tau - 2 + hattau)]/
     Sin[baralpha]^2, 
  rhogeneric[tau - 2 + hattau, theta, Deltaphi] Cos[baralpha] Exp[
     barc (2 Pi + baralpha) (tau - 2 + hattau)] - 
   Cos[(2 Pi + baralpha) tau - 4 Pi - baralpha + 
     hatheta - (baralpha - Deltaphi - theta)]]
\end{verbatim}
\begin{verbatim}
hatr2arc2case7[tau_, theta_, Deltaphi_] := 
 hatr2tent[
  rhogeneric[tau - 1, theta, Deltaphi] Exp[
    barc (2 Pi + baralpha) (tau - 1)], 
  hatell0arc2case7[tau, theta, Deltaphi], 
  rhogeneric[tau - 2 + hattau, theta, Deltaphi] Sin[2 baralpha] Exp[
     barc (2 Pi + baralpha) (tau - 2 + hattau)] - 
   Cos[(tau - 2) (2 Pi + baralpha) + hatheta + baralpha - Pi/2 + 
     Deltaphi + theta]]
\end{verbatim}
\begin{verbatim}
hatrbarphiarc2case7[tau_, theta_, Deltaphi_] := 
 Exp[-barc (2 Pi + baralpha) tau] hatrbarphi[
   hatr2arc2case7[tau, theta, Deltaphi], 
   rhogeneric[tau - 1 + hattau, theta, Deltaphi] Exp[
     barc (2 Pi + baralpha) (tau - 1 + hattau)], 
   rhogeneric[tau, theta, Deltaphi] Exp[barc (2 Pi + baralpha) tau], 
   hatell0arc2case7[tau, theta, Deltaphi]]
\end{verbatim}
\smed
Plot of the final value Fig. \ref{Fig:PTarc_case7_1} and computation of the minimal values
\smed
\begin{verbatim}
Plot3D[{Evaluate@
   Table[hatrbarphiarc2case7[(1 - 
          sigma) (2 + (-Min[Deltaphi + theta, hatheta])/(2 Pi + 
             baralpha)) + sigma (2 + (-Deltaphi)/(2 Pi + baralpha)), 
      theta, Deltaphi]/(theta (theta + Deltaphi)), {sigma, 0.01, 
     0.99, .245}]}, {theta, 0.01, Pi/2}, {Deltaphi, 0.01, 
  hatheta - .01}, AxesLabel -> Automatic, PlotLegends -> Automatic]
\end{verbatim}
\begin{verbatim}
Plot3D[hatrbarphiarc2case7[(1 - 
       sigma) (2 + (-Min[0 + theta, hatheta])/(2 Pi + baralpha)) + 
    sigma (2 + (-0)/(2 Pi + baralpha)), theta, 
   0]/(theta (theta + 0)), {sigma, 0, 1}, {theta, 0.01, Pi/2}, 
 AxesLabel -> Automatic, PlotLegends -> Automatic]
\end{verbatim}
\begin{verbatim}
Plot[hatrbarphiarc2case7[(1 - 
       sigma) (2 + (-Min[0 + Pi/2, hatheta])/(2 Pi + baralpha)) + 
    sigma (2 + (-0)/(2 Pi + baralpha)), Pi/2, 
   0]/(Pi/2 (Pi/2 + 0)), {sigma, 0, 1}, AxesLabel -> Automatic, 
 PlotLegends -> 
  Placed[ToExpression[
    "\\frac{e^{-\\bar c \\phi} \\delta \\check r(\\bar \
\\phi)}{\\pi^2}", TeXForm, HoldForm], {Right, Center}]]
\end{verbatim}

\sbig

\noindent{\bf Case 8}: computation of the relevant quantities for the tent
\smed
\begin{verbatim}
hatell0arc2case8[tau_, theta_, Deltaphi_] := 
 ell0tent[
  rhogeneric[tau - 1, theta, Deltaphi] Exp[
    barc (2 Pi + baralpha) (tau - 1)], 
  Sin[theta] hatheta, -Sin[theta] Cos[Pi/2 + hatheta]]
\end{verbatim}
\begin{verbatim}
hatr2arc2case8[tau_, theta_, Deltaphi_] := 
 hatr2tent[
  rhogeneric[tau - 1, theta, Deltaphi] Exp[
    barc (2 Pi + baralpha) (tau - 1)], 
  hatell0arc2case8[tau, theta, Deltaphi], 
  Sin[theta] Cos[baralpha] - Sin[theta] Cos[baralpha + hatheta]]
\end{verbatim}
\begin{verbatim}
hatrbarphiarc2case8[tau_, theta_, Deltaphi_] := 
 Exp[-barc (2 Pi + baralpha) tau] hatrbarphi[
   hatr2arc2case8[tau, theta, Deltaphi], 
   rhogeneric[tau - 1 + hattau, theta, Deltaphi] Exp[
     barc (2 Pi + baralpha) (tau - 1 + hattau)], 
   rhogeneric[tau, theta, Deltaphi] Exp[barc (2 Pi + baralpha) tau], 
   hatell0arc2case8[tau, theta, Deltaphi]]
\end{verbatim}
\smed
Plot of the final value Fig. \ref{Fig:PTarc_case8_1} and computation of the minimal values
\smed
\begin{verbatim}
Plot3D[{Evaluate@
   Table[hatrbarphiarc2case8[(1 - 
          sigma) (2 + (-Deltaphi)/(2 Pi + baralpha)) + 
       sigma (2 + (-hatheta)/(2 Pi + baralpha)), theta, Deltaphi]/
     theta, {sigma, 0.01, 0.99, .245}]}, {theta, 0.001, 
  Pi/2}, {Deltaphi, hatheta + .001, Tan[baralpha]}, 
 AxesLabel -> Automatic]
\end{verbatim}
\begin{verbatim}
Plot3D[hatrbarphiarc2case8[(1 - 
       sigma) (2 + (-hatheta - theta)/(2 Pi + baralpha)) + 
    sigma (2 + (-hatheta)/(2 Pi + baralpha)), theta, hatheta + theta]/
  theta, {sigma, 0, 1}, {theta, 0, Pi/2}, AxesLabel -> Automatic]
\end{verbatim}
\begin{verbatim}
Plot[hatrbarphiarc2case8[(1 - .0001) (2 + (-hatheta - theta)/(2 Pi + 
          baralpha)) + .0001 (2 + (-hatheta)/(2 Pi + baralpha)), 
   theta, hatheta + theta]/theta, {theta, 0, Pi/2}, 
 AxesLabel -> Automatic, 
 PlotLegends -> 
  Placed[ToExpression[
    "\\frac{e^{-\\bar c \\phi} \\delta \\check r(\\bar \
\\phi)}{\\theta}", TeXForm, HoldForm], {Right, Center}]]
\end{verbatim}

\sbig

\noindent{\bf Case 9}: computation of the relevant quantities for the tent
\smed
\begin{verbatim}
hatell0arc2case9[tau_, theta_, Deltaphi_] := 
 ell0tent[
  rhogeneric[tau - 1, theta, Deltaphi] Exp[
    barc (2 Pi + baralpha) (tau - 1)], (1 - Cos[theta] + 
      Sin[theta] Deltaphi) (Exp[
        Cot[baralpha] (2 Pi + baralpha) (tau - 2 + hattau)] - 1)/
     Sin[baralpha]^2 + 
   Sin[theta] (4 Pi + 2 baralpha - (2 Pi + baralpha) tau), 
  rhogeneric[tau - 2 + hattau, theta, Deltaphi] Cos[baralpha] Exp[
     barc (2 Pi + baralpha) (tau - 2 + hattau)] - 
   Sin[theta] Cos[Pi/2 + hatheta]]
\end{verbatim}
\begin{verbatim}
hatr2arc2case9[tau_, theta_, Deltaphi_] := 
 hatr2tent[
  rhogeneric[tau - 1, theta, Deltaphi] Exp[
    barc (2 Pi + baralpha) (tau - 1)], 
  hatell0arc2case9[tau, theta, Deltaphi], 
  rhogeneric[tau - 2 + hattau, theta, Deltaphi] Sin[2 baralpha] Exp[
     barc (2 Pi + baralpha) (tau - 2 + hattau)] - 
   Sin[theta] Cos[baralpha + hatheta]]
\end{verbatim}
\begin{verbatim}
hatrbarphiarc2case9[tau_, theta_, Deltaphi_] := 
 Exp[-barc (2 Pi + baralpha) tau] hatrbarphi[
   hatr2arc2case9[tau, theta, Deltaphi], 
   rhogeneric[tau - 1 + hattau, theta, Deltaphi] Exp[
     barc (2 Pi + baralpha) (tau - 1 + hattau)], 
   rhogeneric[tau, theta, Deltaphi] Exp[barc (2 Pi + baralpha) tau], 
   hatell0arc2case9[tau, theta, Deltaphi]]
\end{verbatim}
\smed
Plot of the final value Fig. \ref{Fig:PTarc_case9_1} and computation of the minimal values
\smed
\begin{verbatim}
Plot3D[{Evaluate@
   Table[hatrbarphiarc2case9[(1 - 
          sigma) (2 + (-Min[Deltaphi, hatheta])/(2 Pi + baralpha)) + 
       sigma (2), theta, Deltaphi]/(theta (theta + Deltaphi)), {sigma,
      0.01, 0.99, .245}]}, {theta, 0.01, Pi/2}, {Deltaphi, .01, 
  Tan[baralpha]}, AxesLabel -> Automatic]
\end{verbatim}
\begin{verbatim}
Plot3D[hatrbarphiarc2case9[(1 - 
       sigma) (2 + (-Min[.001, hatheta])/(2 Pi + baralpha)) + 
    sigma (2), theta, .001]/(theta (theta + .001)), {sigma, 0, 
  1}, {theta, 0.01, Pi/2}, AxesLabel -> Automatic]
\end{verbatim}
\begin{verbatim}
Plot[hatrbarphiarc2case9[(1 - 
       sigma) (2 + (-Min[.001, hatheta])/(2 Pi + baralpha)) + 
    sigma (2), Pi/2, .001]/(Pi/2 (Pi/2 + .001)), {sigma, 0, 1}, 
 AxesLabel -> Automatic, 
 PlotLegends -> 
  Placed[ToExpression[
    "\\frac{e^{-\\bar c \\phi} \\delta \\check r(\\bar \
\\phi)}{\\pi^2}", TeXForm, HoldForm], {Right, Center}]]
\end{verbatim}

\sbig

\noindent{\bf Case 10}: both points in the second negativity region, so that we use the simplified functional of Lemma \ref{Lem:final_value_satu}, Fig.\ref{Fig:PTarc_case10_1}
\smed
\begin{verbatim}
Plot3D[Evaluate@
  Table[(rhogeneric[tau, theta, 
       Deltaphi] - (1 - 
         Cos[hatheta]) Exp[-barc (2 Pi + baralpha)] rhogeneric[
         tau - 1, theta, 
         Deltaphi]/(Cos[baralpha - hatheta] - 
          Cos[baralpha]))/(theta (theta + Deltaphi)), {Deltaphi, 0, 
    Tan[baralpha], .5}], {theta, 0.001, Pi/2}, {tau, 2, 5}]
\end{verbatim}
\begin{verbatim}
Plot3D[(rhogeneric[tau, Pi/2, 
     Deltaphi] - (1 - 
       Cos[hatheta]) Exp[-barc (2 Pi + baralpha)] rhogeneric[tau - 1, 
       Pi/2, Deltaphi]/(Cos[baralpha - hatheta] - Cos[baralpha]))/(Pi/
     2 (Pi/2 + Deltaphi)), {Deltaphi, 0, Tan[baralpha]}, {tau, 2, 5}]
\end{verbatim}
\begin{verbatim}
Plot[(rhogeneric[tau, Pi/2, 
     0] - (1 - Cos[hatheta]) Exp[-barc (2 Pi + baralpha)] rhogeneric[
       tau - 1, Pi/2, 
       0]/(Cos[baralpha - hatheta] - Cos[baralpha]))/(Pi/
     2 (Pi/2 + 0)), {tau, 2, 5}, AxesLabel -> Automatic, 
 PlotLegends -> 
  Placed[ToExpression[
    "\\frac{1}{(\\pi/2)^2} \\bigg( \\rho(\\bar \\tau) - \\frac{(1 - \
\\cos(\\hat \\theta)) e^{-\\bar c(2\\pi + \\bar \
\\alpha)}}{\\cos(\\bar \\alpha - \\hat \\theta) - \\cos(\\bar \
\\alpha)} \\rho(\\bar \\tau - 1) \\bigg)", TeXForm, HoldForm], {Left, 
    Top}]]
\end{verbatim}
\smed
Derivative, Fig. \ref{Fig:PTarc_case10bis_1}
\smed
\begin{verbatim}
Plot3D[Evaluate@
  Table[(rhogeneric[tau, theta, 
       Deltaphi] - (1 - 
         Cos[hatheta]) Exp[-barc (2 Pi + baralpha)] rhogeneric[
         tau - 1, theta, 
         Deltaphi]/(Cos[baralpha - hatheta] - Cos[baralpha]) - 
      rhogeneric[tau - 1, theta, 
       Deltaphi] - (1 - 
         Cos[hatheta]) Exp[-barc (2 Pi + baralpha)] rhogeneric[
         tau - 2, theta, 
         Deltaphi]/(Cos[baralpha - hatheta] - 
          Cos[baralpha]))/(theta (theta + Deltaphi)), {Deltaphi, 0, 
    Tan[baralpha], .5}], {theta, 0.001, Pi/2}, {tau, 4, 5}]
\end{verbatim}
\begin{verbatim}
Plot3D[(rhogeneric[tau, theta, 
     0] - (1 - Cos[hatheta]) Exp[-barc (2 Pi + baralpha)] rhogeneric[
       tau - 1, theta, 0]/(Cos[baralpha - hatheta] - Cos[baralpha]) - 
    rhogeneric[tau - 1, theta, 
     0] - (1 - Cos[hatheta]) Exp[-barc (2 Pi + baralpha)] rhogeneric[
       tau - 2, theta, 
       0]/(Cos[baralpha - hatheta] - Cos[baralpha]))/(theta (theta + 
      0)), {theta, 0.001, Pi/2}, {tau, 4, 5}]
\end{verbatim}
\begin{verbatim}
Plot[(rhogeneric[tau, Pi/2, 
     0] - (1 - Cos[hatheta]) Exp[-barc (2 Pi + baralpha)] rhogeneric[
       tau - 1, Pi/2, 0]/(Cos[baralpha - hatheta] - Cos[baralpha]) - 
    rhogeneric[tau - 1, Pi/2, 
     0] - (1 - Cos[hatheta]) Exp[-barc (2 Pi + baralpha)] rhogeneric[
       tau - 2, Pi/2, 
       0]/(Cos[baralpha - hatheta] - Cos[baralpha]))/(Pi/
     2 (Pi/2 + 0)), {tau, 4, 5}, AxesLabel -> Automatic, 
 PlotLegends -> 
  Placed[ToExpression[
    "\\frac{1}{(\\pi/2)^2} \\bigg( \\frac{d\\rho(\\bar \\tau)}{d\\bar \
\\tau} - \\frac{(1 - \\cos(\\hat \\theta)) e^{-\\bar c(2\\pi + \\bar \
\\alpha)}}{\\cos(\\bar \\alpha - \\hat \\theta) - \\cos(\\bar \
\\alpha)} \\frac{d\\rho(\\bar \\tau - 1)}{d\\bar \\tau} \\bigg)", 
    TeXForm, HoldForm], {Right, Top}]]
\end{verbatim}

\newpage

\printglossaries

\newpage

\bibliographystyle{alpha}
\bibliography{fire}


\end{document}